 \documentclass[english]{smfbook}
\usepackage{CJKutf8}
\usepackage{epigraph}

\usepackage{lmodern}
\usepackage{imakeidx}
\makeindex[title=Glossary]
\makeindex[name=n, title=Index]

\usepackage{float}
\usepackage{placeins}
\usepackage{xcolor}
\usepackage{subcaption}

\usepackage{setspace}

\usepackage{tikz}
\usepackage{amsmath}
\usepackage[all]{xy}
\xyoption{pdf}
\usepackage{amscd}
\usepackage{stmaryrd}
\usepackage{amsfonts}
\usepackage{amssymb}
\usepackage{graphicx}
\usepackage{mathrsfs}
\usepackage{amsthm}
\usepackage{smfenum}
\usepackage{leftidx}
\usepackage{hyperref}
\usepackage{extarrows}
\usepackage{wasysym}
\usepackage{mathtools}
\usepackage{url}
\usepackage{epstopdf}
\usepackage {hyperref}

\newtheorem{thm}[subsection]{Theorem}
\newtheorem*{thm*}{Theorem}
\newtheorem{thm0}{Theorem}

\newtheorem{cor}[subsection]{Corollary}
\newtheorem{lem}[subsection]{Lemma}
\newtheorem{hypo}[subsection]{Hypothesis}
\newtheorem{convention}[subsection]{Convention}
\newtheorem{conj}{Conjecture}
\newtheorem{prop}[subsection]{Proposition}
\theoremstyle{definition}
\newtheorem{defn}[subsection]{Definition}
\theoremstyle{remark}
\newtheorem{rem0}{Remark}
\newtheorem{rem}[subsection]{Remark}
\newtheorem{eg}[subsection]{Example}
\numberwithin{equation}{subsection}

\DeclareMathOperator{\AJT}{AJT}
\DeclareMathOperator{\ab}{ab}
\DeclareMathOperator{\ad}{ad}

\DeclareMathOperator{\Aut}{Aut}
\DeclareMathOperator{\AJ}{AJ}
\DeclareMathOperator{\sign}{sgn}

\DeclareMathOperator{\bad}{bad}

\DeclareMathOperator{\Cent}{Cent}

\DeclareMathOperator{\Ell}{ell}
\DeclareMathOperator{\shall}{sh}
\DeclareMathOperator{\Wh}{Wh}
\DeclareMathOperator{\der}{der}
\DeclareMathOperator{\diag}{diag}
\DeclareMathOperator{\disc}{disc}

\DeclareMathOperator{\Endos}{eds}
\DeclareMathOperator{\End}{End}
\DeclareMathOperator{\et}{\acute{e}t }

\DeclareMathOperator{\Gal}{Gal}
\DeclareMathOperator{\coh}{\mathbf {H}}
\DeclareMathOperator{\icoh}{\mathbf{IH}}
\DeclareMathOperator{\IC}{IC}
\DeclareMathOperator{\tatecoh}{\widehat{\mathbf {H}}}
\DeclareMathOperator{\Hom}{Hom}

\DeclareMathOperator{\id}{id}
\DeclareMathOperator{\Int}{Int}
\DeclareMathOperator{\irr}{irr}

\DeclareMathOperator{\im}{im}
\DeclareMathOperator{\Out}{Out}

\DeclareMathOperator{\tasho}{\twonotes}

\DeclareMathOperator{\sat}{\mathcal{S}}
\DeclareMathOperator{\inv}{inv}
\DeclareMathOperator{\N}{N}
\DeclareMathOperator{\ch}{ch}

\DeclareMathOperator{\odd}{odd}
\DeclareMathOperator{\even}{even}
\DeclareMathOperator{\Nor}{Nor}
\DeclareMathOperator{\Res}{Res}

\DeclareMathOperator{\SC}{SC}
\DeclareMathOperator{\Sh}{Sh}
\DeclareMathOperator{\uni}{unit}

\DeclareMathOperator{\spec}{Spec}
\DeclareMathOperator{\sspec}{spec}
\DeclareMathOperator{\st}{st}
\DeclareMathOperator{\Frob}{Frob}

\DeclareMathOperator{\Tr}{Tr}

\DeclareMathOperator{\vol}{vol}
\DeclareMathOperator{\ur}{ur}
\DeclareMathOperator{\Lie}{Lie}
\DeclareMathOperator{\GL}{GL}

\DeclareMathOperator{\cost}{Const.}

\DeclareMathOperator{\SL}{SL}
\DeclareMathOperator{\SO}{SO}
\DeclareMathOperator{\SU}{SU}
\DeclareMathOperator{\Uni}{U}

\DeclareMathOperator{\Gspin}{GSpin}
\DeclareMathOperator{\Sp}{Sp}
\DeclareMathOperator{\std}{std}
\DeclareMathOperator{\stan}{Std}

\DeclareMathOperator{\WD}{WD}
\DeclareMathOperator{\RBC}{\mathcal{RBC}}

\DeclareMathOperator{\Pink}{Pink}
\DeclareMathOperator{\stdlev}{StdLev}
\DeclarePairedDelimiter\ceil{\lceil}{\rceil}
\DeclarePairedDelimiter\floor{\lfloor}{\rfloor}

\DeclareFontFamily{U}{stixbbit}{}
\DeclareFontShape{U}{stixbbit}{m}{it}{
	<-> stix-mathbbit
}{}
\DeclareSymbolFont{stixbbit}{U}{stixbbit}{m}{it}
\DeclareMathSymbol{\bend}{\mathord}{stixbbit}{"F6}
\newcommand*{\bendsymbol}{\ensuremath{\bend\hspace{-2mm}\bend\hspace{-2mm}\bend\hspace{-2mm}\bend}}

\newcommand{\norm}[1]{\left\Vert#1\right\Vert}
\newcommand{\dbp}[1]{[\! [ #1 ]\! ] }

\newcommand{\adele}{\mathbb A}
\newcommand{\DS}{\mathrm{symdiag}}

\newcommand{\admpar}{\mathrm{AdmPar}(G)}

\newcommand{\M}{\mathrm{M}}
\newcommand{\da}{{M,\diamond}}
\newcommand{\swap}{\mathtt{sw}}
\newcommand{\chH}{\check H}
\newcommand{\Irr}{\mathrm{Irr}(G(\mathbb A_f))}

\newcommand{\ED}{\mathrm {ED}}

\newcommand{\derD}{\underline {\mathcal D}}
\newcommand{\cD}{\mathcal D}
\newcommand{\derV}{\underline V}
\newcommand{\derT}{\underline T}
\newcommand{\derB}{\underline B}
\newcommand{\derj}{\underline j}

\newcommand{\EDNice}{\ED(V)^o_{\mathrm{nice}}}
\newcommand{\EDI}{\ED(\derV)^o_{\Wh\text{-}\mathrm{I}}}
\newcommand{\EDII}{\ED(\derV)^o_{\Wh\text{-}\mathrm{II}}}
\newcommand{\wI}{\mathfrak w_{\mathrm{I}}}
\newcommand{\wII}{\mathfrak w_{\mathrm{II}}}
\newcommand{\pintowh}{\mathscr W}
\newcommand{\Pinn}{\mathcal{S}plit}
\newcommand{\Whitt}{\mathcal{W}hitt}
\newcommand{\RI}{\mathcal{(I)}}

\newcommand{\RII}{\mathcal{(II)}}

\newcommand{\RIII}{\mathcal{(III)}}

\newcommand{\RIV}{\mathcal{(IV)}}

\newcommand{\RV}{\mathcal{(V)}}

\newcommand{\RVII}{\mathcal{(VII)}}

\newcommand{\RVIII}{\mathcal{(VIII)}}

\newcommand{\RA}{(\mathscr A)}
\newcommand{\RB}{(\mathscr B)}
\newcommand{\RC}{(\mathscr C)}
\newcommand{\RD}{(\mathscr D)}

\newcommand{\X}{\mathcal {X}}
\newcommand{\Y}{\mathcal {Y}}
\newcommand{\abs}[1]{\left\vert#1\right\vert}

\newcommand{\cW}{\mathcal W}
\newcommand{\fL}{\mathbf L}
\newcommand{\HZ}{\mathcal H_0}
\newcommand{\KS}[1]{\mathscr S_{#1 , p , \mathrm{can}}}
\newcommand{\MP}{\overline{\mathscr S_{\tilde K , p , \mathrm{can}}}}

\newcommand{\lprod}[1]{\langle #1  \rangle}
\newcommand{\set}[1]{\left\{#1\right\}}

\newcommand{\fkp}{\mathfrak p}
\newcommand{\fke}{\mathfrak e}

\newcommand{\To}{\longrightarrow}
\newcommand{\isom}{\overset{\sim}{\To}}
\newcommand{\lix}{\leftidx}

\newcommand{\lang}{\leftidx^L}

\newcommand{\NN}{\mathbb{Z}_{\geq 1}}

\newcommand{\ZZ}{\mathbb{Z}}
\newcommand{\QQ}{\mathbb{Q}}
\newcommand{\RR}{\mathbb{R}}
\newcommand{\FF}{\mathbb{F}}

\newcommand{\CC}{\mathbb{C}}
\newcommand{\GG}{\mathbb{G}}
\newcommand{\oo}{\mathcal{O}}
\newcommand{\D}{\mathfrak{D}}

\newcommand{\E}{\mathfrak{E}}

\newcommand{\Shh}{\mathscr{S}}
\newcommand{\tH}{\widetilde {\mathcal{H}}}
\newcommand{\Gm}{\mathbb G_m}
\newcommand{\GmR}{\mathbb G_{m,\RR}}
\newcommand{\GmC}{\mathbb G_{m,\CC}}

\newcommand{\W}{\mathfrak W}

\newcommand{\spl}{\mathbf{spl}}
\newcommand{\splI}{\mathbf{spl}_{\mathrm {I}}}

\newcommand{\AB}{\mathscr A_{\mathsf{B}}}

\newcommand{\AD}{\mathscr A_{\mathsf{D}}}

\newcommand{\footnotee}[1]{}
\newcommand{\ignore}[1]{}

\title[The stabilization of traces on orthogonal Shimura varieties]{The stabilization of the Frobenius--Hecke traces on the intersection cohomology of orthogonal Shimura varieties}
\author[Yihang Zhu]{Yihang Zhu}
\address{Yau Mathematical Sciences Center\\ Tsinghua University\\ 
Beijing \\ China}
\email{yhzhu@tsinghua.edu.cn}

\date{July, 2023}
\begin{document}
	
	\frontmatter

	\begin{abstract}
We study Shimura varieties associated with special orthogonal groups over the field of rational numbers. We prove a version of Morel's formula for the Frobenius--Hecke traces on the intersection cohomology of the Baily--Borel compactification. Our main result is the stabilization of this formula. As an application, we compute the Hasse--Weil zeta function of the intersection cohomology in some special cases, using the recent work of Arthur and Ta\"ibi on the endoscopic classification of automorphic representations of special orthogonal groups.  
	\end{abstract}
\begin{altabstract}Nous étudions les variétés de Shimura associées à des groupes spéciaux orthogonaux  sur le corps des nombres rationnels. Nous prouvons une version de la formule de Morel pour les traces de Frobenius--Hecke sur la cohomologie d'intersection de la compactification de Baily--Borel. Notre résultat principal est la stabilisation de cette formule. Comme application, nous calculons la fonction zêta de Hasse--Weil de la cohomologie d'intersection dans certains cas particuliers, en utilisant les travaux récents d'Arthur et Ta\"ibi sur la classification endoscopique des représentations automorphes des groupes  spéciaux orthogonaux. \end{altabstract}

	\keywords{orthogonal Shimura varieties, Baily--Borel compactification, intersection cohomology, Hasse--Weil zeta function, Langlands--Kottwitz method, stabilization of the trace formula, endoscopic classification of automorphic representations}
		\altkeywords{Variétés de Shimura orthogonales, compactification de Baily–Borel, cohomologie d'intersection, fonction zêta de Hasse–Weil, méthode de Langlands–Kottwitz, stabilisation de la formule de trace, classification endoscopique des représentations automorphes.}
	\thanks{The author was supported by NSF grant DMS-1802292.}
	
	
	\maketitle



	\renewcommand{\epigraphflush}{center}
 \renewcommand{\epigraphsize}{\Large}
 \setlength{\epigraphwidth}{0.7\textwidth}

\epigraph{ \begin{CJK*}{UTF8}{bsmi} 
	其始也，皆收視反聽，耽思傍訊，精騖八極，心遊萬仞。其致也，情曈曨而彌鮮，物昭晰而互進。
 \end{CJK*}}{ \begin{CJK*}{UTF8}{bsmi} 陸機《文賦》
	\end{CJK*}}

 \epigraph{In the beginning, \\ All external vision and sound are suspended, \\ Perpetual thought itself gropes in time and space; \\ Then, the spirit at full gallop reaches the eight \\ \qquad limits of the cosmos, \\ And the mind, self-buoyant, will ever soar to new\\ \qquad  insurmountable heights. \\ When the search succeeds, \\ Feeling, at first but a glimmer, will gradually\\ \qquad  gather into full luminosity, \\ Whence all objects thus lit up glow as if each the \\ \qquad  other's light reflects.\footnotemark} {Excerpt from \textit{Essay on Literature} \\by LU Ji (261--303 AD)}

\footnotetext{Translated from Chinese by CHEN Shixiang.}

		
	\tableofcontents

	 \mainmatter

	\chapter*{Introduction} \addtocontents{toc}{\protect\setcounter{tocdepth}{-5}}
\numberwithin{equation}{chapter}
Inspired by the early works of Eichler, Shimura, Kuga, Sato, and Ihara, the ongoing study of expressing Hasse–Weil zeta functions of Shimura varieties through automorphic $L$-functions remains a focal point within the Langlands program.  Langlands approached this problem by proposing a comparison of the Frobenius--Hecke traces on the cohomology of Shimura varieties with the stable Arthur--Selberg trace formulas, as detailed in  \cite{langlands1977shimura,langlandsmarchen,langlands1979zeta}. Kottwitz further formalized these ideas into precise conjectures  \cite{kottwitzannarbor,kottwitz1992points}. In this paper, we confirm a version of Kottwitz's conjecture specifically for the intersection cohomology of orthogonal Shimura varieties.  
	
	 \section*{The conjectures}\label{subsec: precise}
	 Let $(G,\X)$ be a Shimura datum with reflex field $E$. For each sufficiently small compact open subgroup $K\subset G(\adele_f)$, we have the Shimura variety $$\Sh_K= \Sh_K(G,\X),$$ which is a smooth quasi-projective algebraic variety over $E$. Let $\overline{\Sh_K}$  be the Baily--Borel compactification of $\Sh_K$. Let $\icoh^*$ be the intersection cohomology of $\overline {\Sh_K} \otimes_E \overline E$ with $\overline \QQ_{\ell}$-coefficients. (More generally, a non-trivial ``automorphic'' coefficient system is allowed, which we ignore in the introduction.) Let $p$ be a hyperspecial prime for $K$, i.e., $K = K_p K^p$ with $K_p \subset G(\QQ_p)$ a hyperspecial subgroup and $K^p \subset G(\adele_f^p)$ a compact open subgroup. (Here $\adele_f^p$ denotes the finite adeles away from $p$.) Assume that $p \neq \ell$. On $\icoh^*$, we have commuting actions of $\Gal(\overline E/E)$ and  the Hecke algebra $\mathcal H (G(\adele_f^p) \sslash K^p)_{\overline \QQ_{\ell}}$ consisting of the $\overline \QQ_{\ell}$-valued smooth compactly supported $K^p$-bi-invariant distributions on $G(\adele_f^p)$. Fix $f^{p,\infty}\in \mathcal H (G(\adele_f^p) \sslash K^p)_{\overline \QQ_{\ell}}$, and let $\Phi = \Phi_{\mathfrak p}$ be a geometric Frobenius at a place $\mathfrak p$ of $E$ above $p$. Let $a \in \ZZ_{\geq 1}$. 
	 
	  \begin{conj}[Kottwitz, see {\cite[\S 10]{kottwitzannarbor} }]\label{intro conj}
	The action of $\Gal(\overline E/E)$ on $\icoh^*$ is unramified at $\mathfrak p$, and under simplifying assumptions of a group-theoretic nature, we have \begin{align}\label{intro id}
	\sum_k (-1)^k \Tr(f^{p,\infty} \times \Phi^a \mid  \icoh^k ) = \sum_H \iota(G,H) ST^H (f^H).
	\end{align}  On the right, $H$ runs through the isomorphism classes of elliptic endoscopic data of $G$. For each $H$, $ST^H(\cdot )$ is the geometric side of the stable trace formula for $H$, and $f^H$ is a function on $H(\adele)$ determined by the Shimura datum, $f^{p,\infty}$, and $a$. 
	  \end{conj}
	  	
	  In addition, Kottwitz also formulated the following conjecture for the compact support cohomology $ \coh^*_c $ of $\Sh_K\otimes_E \overline E$. 
	  \begin{conj}[Kottwitz, see {\cite[\S 7]{kottwitzannarbor}}] \label{baby conj}
	  	 The action of $\Gal(\overline E/E)$ on $\coh^*_c$ is unramified at $\mathfrak p$, and under simplifying assumptions we have \begin{align}\label{baby id}
	  	\sum_k (-1)^k \Tr(f^{p,\infty} \times \Phi^a \mid  \coh^k_c ) = \sum_H \iota(G,H) ST^H_e (f^H).
	  	\end{align} Here $H$ and $f^H$ are the same as in Conjecture \ref{intro conj}, while $ST^H_e (\cdot )$ is the elliptic part of the geometric side of the stable trace formula for $H$. 
	  \end{conj}
  	
  \section*{The main result}
  Let $(V,q)$ be a quadratic space over $\QQ$ of signature $(n,2)$, where $n \geq 3$. We assume that $V$ has a $2$-dimensional totally isotropic subspace, which is automatic if $n \geq 5$. Let $G= \SO(V,q)$. We have a natural Shimura datum $(G,\X)$, where $\X$ can be identified with the set of oriented negative definite planes in $V_{\RR}$. This Shimura datum is of abelian type (but not of Hodge type). The associated Shimura varieties are called \emph{orthogonal Shimura varieties}. They are $n$-dimensional varieties over the reflex field $\QQ$.
  \begin{thm0}[Corollary \ref{Main Main result}]\label{intro thm} Conjecture \ref{intro conj} is true for the orthogonal Shimura varieties associate to $(V,q)$, for almost all primes $p$ and for all sufficiently large $a$. 
  \end{thm0}

We refer the reader to the statements of Theorem \ref{geometric assertion} and Corollary \ref{Main Main result} for the precise meaning of ``almost all primes $p$''. Here we just mention that the set of primes to be excluded should depend on a fixed element $f^{\infty}$ of the ``full'' Hecke algebra $\mathcal H(G(\adele_f)\sslash K )_{\overline \QQ_{\ell}}$, whereas $f^{p,\infty}$ in (\ref{intro id}) should be the component of $f^{\infty}$ away from $p$, after $p$ has been chosen. 
	  \section*{Some remarks}
	  
	  From a group-theoretic point of view, both sides of (\ref{baby id}) are less complicated
compared to (\ref{intro id}). In fact, the RHS of (\ref{baby id}) has an elementary definition in terms of stable orbital integrals. For the LHS of (\ref{baby id}), Kottwitz computed it for PEL Shimura varieties of type $\mathsf A$ or $\mathsf C$ in \cite{kottwitz1992points}  by counting (virtual) abelian varieties with additional structures over finite fields and using the Grothendieck--Lefschetz--Verdier trace formula. He obtained:
	  \begin{align}\label{gp thry expression}
	  	\sum_k (-1)^k \Tr(f^{p,\infty} \times \Phi^a \mid  \coh^k_c ) = 
	  \sum_{(\gamma_0 , \gamma,\delta)} c(\gamma_0, \gamma,\delta)  O_{ \gamma} (f^{p,\infty}) TO_{\delta} (\phi)\Tr (\gamma_0\mid  \mathbb V).
	  \end{align} We do not  explain the terms on the RHS in detail here, but only mention that they are group-theoretic in nature and include orbital integrals $O_{\gamma}(\cdot)$ and twisted orbital integrals $TO_{\delta}(\cdot)$. In \cite{kottwitzannarbor}, Kottwitz conjectured that (\ref{gp thry expression}) should hold for general Shimura varieties (at least under some simplifying assumptions of a group-theoretic nature). In the same paper Kottwitz \emph{stabilized} the RHS of (\ref{gp thry expression}), namely he found\footnote{The construction of $f^{H}$ relies on the Langlands--Shelstad Transfer Conjecture and the Fundamental Lemma,  which were unproven at the time of  \cite{kottwitzannarbor}. They are now theorems thanks to the work of numerous mathematicians, most notably Ng\^o and Waldspurger.} the functions $f^H$ such that the RHS of (\ref{gp thry expression}) is equal to the RHS of (\ref{baby id}). In \cite{KSZ}, both the formula (\ref{gp thry expression}) and the stabilization step are generalized to arbitrary Shimura varieties of abelian type, and Conjecture \ref{baby conj} is proved for these varieties. 
  
	  One should view Conjecture \ref{intro conj} as one step forward from Conjecture \ref{baby conj}. From a spectral perspective, it is $ST^H$ rather than $ST^H_{e}$ that sees the ``whole picture''. More specifically, $ST^H$ has a spectral expansion, from which one can eventually make a link to automorphic representations. By contrast it is unclear how $ST^H_e$ can be directly related to spectral information in general.
	  
	  We also mention that the expectation that the intersection cohomology is the correct cohomology to insert in  (\ref{intro id}) is motivated by Zucker's conjecture and Arthur's work on $L^2$-cohomology, among other things.\ignore{Recall that Zucker's conjecture (proved by Looijenga and Saper-Stern) states that $\icoh^*$ is isomorphic to the $L^2$-cohomology of $\Sh_K(\CC)$ as Hecke modules. Arthur \cite{arthurlef} showed that the trace of a Hecke operator on the $L^2$-cohomology of an arbitrary locally symmetric space can be expressed in terms of his \emph{invariant trace formula}, and in particular an analogue of Conjecture \ref{intro conj} without the Frobenius follows from Arthur's general stabilization of the invariant trace formula \cite{arthursta1} \cite{arthursta2} \cite{arthursta3}. We mention that Goresky--Kottwitz--MacPherson \cite{GKM} proved the same formula as \cite{arthurlef} for the trace of a Hecke operator on the intersection cohomology, which is the same as the $L^2$-cohomology by Zucker's Conjecture. Thus Conjecture \ref{intro conj} can be regarded as a refinement of the work of Arthur and Goresky--Kottwitz--MacPherson mentioned above, with extra information about the Galois action. For more motivations see \cite{morel2010icm}.
	  }
	  We refer the reader to  \cite{morel2010icm} for a more detailed discussion on these motivations.

	  \ignore{
	\begin{rem0}
In stating Conjectures \ref{intro conj} and \ref{baby conj} in the forms above, we need to assume the possibility of embedding each $\lang H$ into $\lang G$, and assume that the maximal $\QQ$-split torus in $Z_G$ is $\RR$-split. While the conjectures can be stated in total generality by considering z-extensions and central data (see \cite{KSZ}), the orthogonal Shimura datum $(G,\X)$ does satisfy these assumptions. However, the assumption that the derived group of $G$ is simply connected, which often appears in the context of Shimura varieties, is neither satisfied by $G$ nor necessary for stating the conjectures for $(G,\X)$. 
	\end{rem0}
	
}

\section*{Application: the Hasse--Weil zeta functions}\label{subsec:intro app}
In \cite{kottwitzannarbor},  Kottwitz showed that one can combine Conjecture \ref{intro conj} with the conjectural framework of Arthur parameters and Arthur's multiplicity conjectures to infer a description of the Galois--Hecke module $\icoh^*$, and in particular a formula for the Hasse--Weil zeta function associated to $\icoh^*$. 

Currently some of these premises related to Arthur's conjectures have been established  in special cases. Most notably, Arthur \cite{arthurbook} has established the multiplicity conjectures for quasi-split classical groups.\footnote{The results in \cite{arthurbook} are contingent on the release of several upcoming papers, including the reference [A25], which have not appeared as of the time of writing. \label{foot:Arthur}} 
In fact, our interest in delving into special orthogonal groups within this paper is driven by a desire to connect with Arthur's work. This intentional decision distinguishes our focus from similar groups such as $\Gspin$, whose Shimura varieties display a relative simplicity in various aspects, for instance, being of Hodge type.

Unfortunately, when the rank is large the special orthogonal groups that have Shimura varieties cannot be quasi-split even over $\RR$, because of the signature $(n,2)$ condition. Arthur's work has been generalized to limited cases of inner forms by Ta\"ibi \cite{taibi} (building on earlier work of Kaletha \cite{kalethaglobal, kalrigid} and Arancibia--Moeglin--Renard \cite{AMR}, among others). 
We combine Theorem \ref{intro thm} with Arthur's and Ta\"ibi's work to obtain the following theorem. Here we state it only for odd $n$ for simplicity. 

\begin{thm0}[Theorem \ref{thm:whole zeta}, Remark \ref{rem:more general K}]\label{intro app} Assume that $n $ is odd, and that $G = \SO(V,q)$ is quasi-split at all finite places. 
	For any finite set $S$ of prime numbers, let $\zeta^S( \icoh^* , s )$ be the $S$-partial Hasse--Weil zeta function associated to $\icoh^*$. When $S$ is sufficiently large, we have \begin{multline*}
 \log \zeta ^S ( \icoh^* ,s)  \\ = \sum _{\psi} \sum_{\pi^\infty} \sum _ {\nu} \dim (\pi^\infty) ^K  m (\pi^{\infty}, \psi,  \nu) (-1) ^{n} \nu (s_{\psi}) \log L^S (\mathcal M (\psi ,\nu) , s).
	\end{multline*}
 Here $\psi$ runs through a certain set of Arthur's substitutes of global Arthur parameters, $\pi^\infty$ runs through the away-from-$\infty$ global packet of $\psi$, and $\nu $ runs through characters of the centralizer group of $\psi$ (which is finite abelian). The three-fold summation is over a finite range. The numbers 
	$m(\pi^\infty ,\psi ,\nu) \in \set{0, 1}$ and $\nu (s_{\psi}) \in \set{\pm 1}$ are defined in terms of constructions in \cite{arthurbook} and \cite{taibi}. The term $L^S (\mathcal M(\psi, \nu) , s)$ is a finite product of $S$-partial standard automorphic $L$-functions for general linear groups  (with some shifting in the variable $s$), and hence has  meromorphic continuation to $\CC$. In particular, the above formula implies that $\zeta^S(\icoh^*, s)$ has meromorphic continuation to $\CC$. 
\end{thm0}
   In the proof of Theorem \ref{intro app}, one crucial ingredient is a relatively simple  formula for $ST^H (f^H)$ when the test function $f^H$ is \emph{stable cuspidal} at the real place; see Hypothesis \ref{hypo}.    This formula follows from Kottwitz's stabilization of the $L^2$ Lefschetz number formula in his unpublished notes, and is also used in Morel's work \cite{morel2010book,morel2011suite}. A self-contained proof of this formula for $ST^H(f^H)$, from a different point of view, is given in a recent paper by Z.~Peng \cite{peng}.

We also prove a refinement of Theorem \ref{intro app} concerning the decomposition of $\icoh^*$ in the Grothendieck group of Galois--Hecke modules, under the same assumption on $G$. When $n$ is odd (as well as in some cases when $n$ is even), we express $\icoh^*$ in terms of the known Galois representations associated to regular algebraic cuspidal  automorphic representations of general linear groups, with multiplicities given in a similar way as the multiplicities in Theorem \ref{intro app}.  See Theorem \ref{thm:decomp c}, Corollary \ref{cor:single degree}, and Corollary \ref{cor:known Galois}. When $n$ is even, both the computation of the partial Hasse--Weil zeta function and the decomposition of $\icoh^*$ proved in this paper are weaker than the conjectures in \cite{kottwitzannarbor}, in that a certain ambiguity up to outer automorphism is constantly present. This is due to the extra ambiguity in the endoscopic classification of representations for even special orthogonal groups in \cite{arthurbook} and \cite{taibi}, which seems intrinsic to the methods therein.

As a byproduct of our refinement of Theorem \ref{intro app}, we prove that if an Arthur parameter $\psi$ contributes to $\icoh^*$, then the cuspidal automorphic representations of general linear groups that constitute $\psi$ all satisfy the Ramanujan--Petersson conjecture at almost all primes. These representations need not be regular algebraic, in which case the conjecture was previously known. See Theorem \ref{thm:decomp c} (3) and Remark \ref{rem:Rama}.

\section*{Reduction to the stabilization of the boundary terms}
  We now discuss the structure of the proof of Theorem \ref{intro thm}. For some period of time, the study of the LHS of (\ref{intro id}) had been restricted to sporadic low dimensional cases; see for instance \cite{montreal}. The essential tools for treating arbitrary dimensions were developed by Morel  \cite{morelthese,morel2008complexes} (cf.~\cite{morel2010icm}), who went on to prove Conjecture \ref{intro conj} for some unitary similitude Shimura varieties and the Siegel modular varieties of arbitrary dimensions in \cite{morel2010book} and \cite{morel2011suite} respectively. We use Morel's work to obtain the following result for the orthogonal Shimura varieties associated to $(V,q)$. We fix a minimal parabolic subgroup of $G = \SO(V,q)$ and fix a Levi component of it. Thus we get a notion of standard parabolic subgroups and standard Levi subgroups of $G$. 
\begin{thm0}[Theorem \ref{geometric assertion}]\label{intro Morel thm}
 For almost all primes $p$, we have	\begin{align}\label{intro Morel}
\sum_k (-1)^k \Tr (f^{p,\infty}\times \Phi^j \mid  \icoh^k) = \sum_{M} \Tr_M,
	\end{align}  where $M$ runs through the standard Levi subgroups of $G$.
\end{thm0} Let us roughly describe the terms $\Tr_M$. For $M =G$, we have  $$\Tr_G = \sum_k (-1)^k \Tr(f^{p,\infty} \times \Phi^j \mid  \coh^k_{c}),$$ where $\coh^k_c$ is the compact support cohomology of $\Sh_{K , \overline \QQ}$. For a proper $M$, the term $\Tr_{M}$ is a more complicated mixture of the following ingredients. 
\begin{itemize}
	\item The analogue of $\sum_k (-1)^k\Tr(f^{p,\infty} \times \Phi^j \mid  \coh^k_{c})$ for a boundary stratum in $\overline{\Sh_K}$. In another words, an enumeration of points on the stratum fixed under certain Frobenius--Hecke operators. 
	\item The topological fixed point formula of Goresky--Kottwitz--MacPherson as in \cite{GKM}, for the trace of a Hecke operator on the compact support cohomology of a certain locally symmetric space.
	\item \emph{Kostant--Weyl terms}. By this we mean characters for certain algebraic sub-representations of $M_P$ inside $$\coh^*(\Lie N_P, \mathbb V) ,$$ where $P$ is a standard parabolic subgroup of $G$ containing $M$, and $P = M_P N_P$ is the standard Levi decomposition. These sub-representations are defined by certain truncations of weights, and can be understood in terms Kostant's theorem \cite{kostant} describing $\coh^*(\Lie N_P, \mathbb V)$. 
\end{itemize}

As we have already mentioned, in \cite{KSZ} the term $\Tr _G$ is computed and stabilized for all Shimura varieties of abelian type. Thus 
$\Tr _G$ is known to be equal to the RHS of (\ref{baby id}). In view of this, Theorem \ref{intro thm} follows from Theorem \ref{intro Morel thm} and the following result, which may be viewed as the ``stabilization of the boundary terms''.
\begin{thm0}[Theorem {\ref{main identity}}]\label{reduced thm} We have 
	\begin{align}\label{intro main id}
\sum_{ M\subsetneqq G} \Tr_M = \sum_H \iota(G,H) [ ST^H (f^H) - ST^H_e(f^H)].
	\end{align}
\end{thm0}

\section*{Stabilization of the boundary terms}
The method for proving Theorem \ref{reduced thm} is by calculating the two sides of (\ref{intro main id}) and matching the explicit expressions. To calculate the RHS, we use Kottwitz's formula in his unpublished notes, as mentioned below Theorem \ref{intro app}. According to this formula (to be recalled in \S \ref{simplified geometric side}), we have an expansion of the form $$ST^H (f^H) - ST^H_e (f^H) = \sum_{ M ' \neq H} ST^H_{M'} (f^H),$$ where $M'$ runs through standard proper Levi subgroups of $H$, and each term $ST^H_{M'} (\cdot)$ has a relatively simple expression.  
	
Roughly speaking, we label the pairs $(H,M')$ appearing in the above summation by either a standard proper Levi subgroup $M$ of $G$ or the symbol $\emptyset$. We write $(H,M') \sim M$, or $(H,M') \sim \emptyset$. In order to prove Theorem \ref{reduced thm}, we need to show 
\begin{align}\label{intro pos}
\Tr_M = \sum_{ (H, M') \sim M } ST^H_{M'} (f^H), 
\end{align} where $M$ is either a standard proper Levi subgroup of $G$ or the symbol $\emptyset$, and we define $\Tr_{\emptyset}$ to be $0$. 
The proof of (\ref{intro pos}) involves the following ingredients. 

\subsection*{(i) Fixed point formula for a boundary stratum}
  We need a formula that enumerates points on a boundary stratum fixed under a Frobenius--Hecke operator, of a form similar to (\ref{gp thry expression}). The boundary stratum in question is (a finite quotient of) either a modular curve or a zero-dimensional Shimura variety, so such a formula is essentially a classical result. However, the zero-dimensional case causes some extra complication. We will come back to this technical point later in the introduction.

\subsection*{(ii) Archimedean comparison}
We need a series of identities between the archimedean contributions to the two sides of (\ref{intro pos}). These are identities between terms of two different natures, namely discrete series character values (which appear on the RHS of (\ref{intro pos})) and Kostant--Weyl terms (which appear on the LHS of (\ref{intro pos}); see the discussion below Theorem \ref{intro Morel thm}). We establish such identities by explicit computation. On the discrete series side, we use formulas due to Harish-Chandra \cite{HCdiscreteseriesI} and Herb \cite{herb}. On the Kostant--Weyl side, we use Kostant's theorem \cite{kostant} and the Weyl character formula. 

We point out that \textit{a priori} it is not clear which identities between the archimedean contributions would eventually lead to the proof of (\ref{intro pos}). Finding the correct forms of the archimedean identities seems to be a harder task than proving them. It would be desirable to have a more conceptual understanding of how the archimedean comparison should be woven into the proof of (\ref{intro pos}) in general.

	\subsection*{(iii) Computation at $p$}
	
	 We need to compute the $p$-adic contributions to the two sides of (\ref{intro pos}) explicitly. \textit{A priori} there are more $p$-adic terms on the RHS than the LHS. We will need to prove, among other things, that the extra terms eventually cancel each other. 

	This finishes our discussion on the structure of the proof of Theorem \ref{intro thm}. Next we highlight three new features in the proof which did not show up in Morel's work \cite{morel2011suite,morel2010book} for symplectic similitude and unitary similitude groups.
	
\section*[Abelian-type Shimura varieties]{Arithmetic feature: Shimura varieties of abelian type}

The orthogonal Shimura varieties are of abelian type and not of PEL type. In this paper we take as a black box the main result of \cite{KSZ} that proves Conjecture \ref{baby conj} for these Shimura varieties. In Morel's work, the Shimura varieties are of PEL type, and for them Conjecture \ref{baby conj} was already proved by Kottwitz. 

The reason that Theorem \ref{intro thm} is proved only for primes outside an unspecified finite set is also due to a certain lack of understanding of Shimura varieties of abelian type. Ideally one would like to prove the theorem for all hyperspecial primes $p$, but a prerequisite for that would be a robust theory of integral models of the Baily--Borel and toroidal compactifications. Such a theory has been established by Madapusi Pera \cite{peratoroidal} in the case of Hodge type. For the Baily--Borel compactifications alone, a ``crude'' construction of the integral models in the case of abelian type has been given by Lan--Stroh \cite{lanstrohII}. However, for the above-mentioned purpose the integral models of toroidal compactifications are equally important, and this is currently unavailable beyond the case of Hodge type. 

All the difficulty about integral models of compactifications can be circumvented at the cost of excluding an unspecified finite set of primes, and this is the point of view taken in this paper. We refer the reader to \S \ref{subsec:background} for a more detailed discussion. 
	
	\section*[Zero-dimensional boundary strata as quotients]{Geometric feature: zero-dimensional boundary strata as quotients of Shimura varieties}
	
	In general, the boundary strata of the Baily--Borel compactification are naturally isomorphic to finite quotients of Shimura varieties at certain natural levels. Often these quotients are isomorphic to genuine Shimura varieties. However this is not true for the zero-dimensional boundary strata in the present case. From a group-theoretic point of view, this issue corresponds to the fact that the orthogonal Shimura datum does not satisfy Morel's axioms in \cite[Chap.~1]{morel2010book}. As a result, in the proof of Theorem \ref{intro Morel thm} we need to modify the axiomatic approach in  \textit{loc.~cit.}, and the terms $\Tr_M$ in (\ref{intro Morel}) are also given by formulas that are slightly different from those in \cite{morel2010book,morel2011suite}.
	
	\section*[Normalizing transfer factors]{Endoscopic-theoretic feature: normalizing transfer factors}	
	
	In the proof of (\ref{intro pos}), signs are utterly important. One source of signs is the difference between the normalizations of transfer factors at the real place. The necessity of computing these signs was not emphasized in \cite{morel2010book,morel2011suite}. For the orthogonal Shimura varieties, these signs form a delicate pattern. 
	
	To understand these signs we need to compare the normalization $\Delta_{j,B}$ introduced in \cite[\S 7]{kottwitzannarbor}, and the Whittaker normalization. Here we explicitly fix $G_{\RR}$ as a pure inner form of its quasi-split inner form $G^*_{\RR}$ and fix a Whittaker datum for $G^*_{\RR}$, so the Whittaker normalizations for the transfer factors between $G_{\RR}$ and its endoscopic groups can be defined. The normalization $\Delta_{j,B}$ naturally shows up in the description of the archimedean component of $f^H$. To compare these two normalizations, we compare the corresponding spectral transfer factors that appear in the endoscopic character relations and compute the sign between them. 
	
Extra complication arises when $G^*_{\RR}$ has more than one equivalence class of Whittaker data. This happens if and only if $\dim V$ is divisible by $4$, when there are precisely two equivalence classes. In this case, we need to study how the two (different) Whittaker normalizations relate to the explicit formulas of Waldspurger \cite{walds10}, the latter having the merit of being easier to keep track of when passing to Levi subgroups. In this direction we prove Theorem \ref{thm:comparing Waldspurger}, which may be of independent interest in representation theory.

\addtocontents{toc}{\protect\setcounter{tocdepth}{\arabic{tocdepth}}}

\chapter*{Acknowledgments} This paper  constitutes a revised and expanded version of my PhD thesis. I extend my utmost gratitude to my PhD advisor, Mark Kisin, for introducing me to the realm of Shimura varieties and proposing this specific project---an endeavor that unfolded as a thrilling and beautiful adventure. During my graduate studies, I benefited significantly from his constant guidance, encouragement, support, as well as the pressure he would give me when progress was slow. Since then, I have continued not only to glean insights into mathematical thinking from him but also to relish his kindness and friendship.

Special appreciation goes to Tasho Kaletha for investing a tremendous amount of time in explaining the intricacies of endoscopy and addressing my questions during a crucial stage of this project. I also extend my thanks to Keerthi Madapusi for steering me through the literature on non-compact Shimura varieties.

Sophie Morel's groundbreaking work is the foundation of this project. In addition to learning from her papers, my visit to her in September 2016 proved valuable. I am genuinely grateful for her giving me the opportunity to present this project at the Princeton/IAS number theory seminar later that fall, where I spoke in front of some of the field's key contributors, including Prof.~Langlands. Additionally, I appreciate her feedback on an earlier draft.

 My gratitude extends to Robert Kottwitz for generously sharing his unpublished notes on the stabilization of the $L^2$ Lefschetz number formula, which serves as a paradigm for the stabilization of the trace formula for the Shimura variety.

 The results presented in the final chapter were not part of the original thesis. I first had the opportunity to share them during my visit to the Morningside Center of Mathematics in January 2018. I express my gratitude to Ye Tian for the invitation, hospitality, and the chance to present these results.

 During the revision stage, Sug Woo Shin provided enormous help, for which I am truly thankful. I also appreciate the invitations from Sug Woo Shin and Sungmun Cho to the Korea Institute for Advanced Study (KIAS) in the summer and winter of 2019, providing an ideal environment for revising this paper.

I extend my thanks to George Boxer, Kai-Wen Lan, Chao Li, Michael Rapoport, Koji Shimizu, Olivier Ta\"ibi, David Vogan, Xinyi Yuan, Zhiwei Yun, Wei Zhang, and Rong Zhou for informative conversations and for answering my questions at various stages of this work. A special acknowledgment to Michael Harris for his sustained interest in the project.

Last but not least, I  express my deep gratitude to the anonymous referee(s) for their meticulous examination of the computations, bringing to light numerous subtleties and corrections, providing various suggestions for improvement, and offering new insights regarding \S \ref{subsec:refined decomp}. 

	\chapter*{Leitfaden}

In \S \ref{Section global}, we recall the setting of orthogonal Shimura varieties and state Morel's formula in Theorem \ref{geometric assertion}. The terms in this formula are defined in \S \ref{Section Geometirc}, and the proof is given in \S \ref{sec:Proof of Morel's formula}. For a more detailed introduction to the structure of the proof see \S \ref{subsec:background}.

In \S \ref{Section infty}, we carry out the archimedean comparison between the Kostant--Weyl terms and the stable discrete series characters. 
The results proved in this chapter to be used later are Propositions \ref{Arch comp M_1}, \ref{Arch comp M2}, \ref{Arch comp odd ++}, \ref{effect of switching to Phi_endos}, and \ref{Arch comp M12 even}.

In \S \ref{sec:endoscopic}, we review the endoscopic data for special orthogonal groups and give explicit presentations which are important for the later computations.

In \S \ref{sec:transfer factors}, we compare different normalizations of archimedean transfer factors for special orthogonal groups. The goal is to explicitly determine certain signs. 

In \S \ref{Section p}, we calculate some Satake transforms at $p$ that will be needed later in the stabilization.

In \S \ref{Section Stabilization}, we prove the stabilization of the boundary terms by assembling the preceding ingredients and explicit manipulation. We deduce the main result Theorem \ref{intro thm} of this paper in Corollary \ref{Main Main result}. 

In \S \ref{sec:app}, we apply our main result to the actual computation of Hasse--Weil zeta functions in some special cases, after reviewing results of Arthur and Ta\"ibi on the endoscopic classification of automorphic representations. The main results in this chapter are Theorems \ref{thm:final}, \ref{thm:whole zeta}, and \ref{thm:decomp c}.

\chapter*{Conventions and notations}
\begin{itemize}
	\item 
	For $x\in \RR$, we denote by $\floor{x}$ the largest integer $\leq x$ and denote by $\ceil{x}$ the smallest integer $\geq x$. If $x \geq 0$, we denote by $x^{1/2}$ the non-negative square root of $x$. 
	\item We denote $i\in \CC$ alternatively by $\sqrt{-1}$. 
	\item For any $n \in \ZZ_{\geq 1}$, we denote by $[n]$\index{$[n]$} the set $\set{1,2,\cdots, n}$. We denote by $\mathfrak S_n$\index{$\mathfrak S_n$} the symmetric group of the set $[n]$. 
	\item Let $A$ be a subset of $\ZZ_{\geq 1}$. For each $i \in \ZZ_{\geq 1}$, we set 
	$\nabla_i (A)=1$ if $i\in A$, and $\nabla_i(A) = -1$ if $i \notin A$. 
	\item When the symbol $\pm$ appears for multiple times in a single expression, it is understood that all possible combinations of the signs are considered. For example, we shall write $\set{\pm x \pm y}$ for the set $\set{x+y, x-y, -x+y, -x-y}$. 
	
	\item A \emph{basis} of a finite-dimensional vector space is always understood as an ordered basis. We often just use the notation for a set such as $\set{e_1,\cdots, e_d}$ to denote a basis, but the ordering is understood.  
	
	\item For $x_1,\cdots, x_n\in \CC^\times$, we write\index{$\mathrm{symdiag}(x_1,\cdots,x_n)$} $\DS(x_1,\cdots,x_n)$ for the $2n \times 2n$ diagonal matrix 
	$\diag(x_1,\cdots, x_n, x_n^{-1}, \cdots, x_1^{-1})$. 
	
	\item For any square matrix $A$, we write $A^{\mathsf T}$ for the transpose. 
	\item If a group $G$ acts on a set $X$, we write $\mathrm{Cent}_G X$ for the action kernel, namely the largest subgroup of $G$ acting trivially on $X$. 
	\item When $x$ is an element of a group, we write $\Int(x)$ for the automorphism $y \mapsto x y x^{-1}$. 
	\item If $\Sigma$ is a finite set of prime numbers, we denote by $\ZZ[1/\Sigma]$\index{$\ZZ[1/\Sigma]$} the ring $\ZZ[ 1/p, p\in \Sigma]$. 
	\item 
	For $a\in \ZZ_{\geq 1}$ and $p$ a prime number, we denote by $\QQ_{p^a}$\index{$\QQ_{p^a}, \ZZ_{p^a}$} the degree $a$ unramified extension of $\QQ_p$, and by $\ZZ_{p^a}$ the valuation ring of $\QQ_{p^a}$. We denote by $\sigma$\index{$\sigma$} the arithmetic $p$-Frobenius acting on $\QQ_{p^a}$. 
	\item If $H$ is either a locally profinite group or a real Lie group, we write $C^{\infty}_c(H)$ for the set of compactly supported smooth $\CC$-valued functions on $H$. 
	\item  
	We use the following abbreviations:\index{$\Gamma_v$}
	\begin{align*}
 \Gamma_{\QQ} & = \Gal (\overline \QQ/ \QQ),&  \Gamma_p & = \Gamma_{\QQ_p} = \Gal(\overline \QQ_p/\QQ_p),& \Gamma_{\infty} &= \Gamma_{\RR} = \Gal(\CC/\RR).&
	\end{align*}
More generally, if $F$ is a field, we write $\Gamma_F$\index{$\Gamma_F$} for the absolute Galois group of $F$.
\item We say that a profinite Galois \'etale covering $Y \to X$ of schemes is a \emph{$G$-torsor}, where $G$ is a profinite group, if $G$ is the limit $\varprojlim_{i\in I} G_i$ of finite groups $G_i$, and $Y\to X$ is the limit of finite Galois \'etale coverings $Y_i \to X$ that is a $G_i$-torsor.  
	\item By a \emph{reductive group}, we always mean a connected reductive group. 
	\item For a reductive group $G$ over $\RR$ and a maximal torus $T$ in $G$ defined over $\RR$, we write $\Omega_{\CC}(G,T)$\index{$\Omega_{\CC}(G,T)$} for the complex Weyl group $\Nor_{G(\CC)}(T)/ T(\CC)$, and write $\Omega_{\RR}(G,T)$\index{$\Omega_{\RR}(G,T)$} for the real Weyl group $\Nor_{G(\RR)}(T)/ T(\RR)$. 
	
	\item For a reductive group $G$ over $\RR$, we denote by $G(\RR)^0$ the identity connected component of the real Lie group $G(\RR)$. 
	
	\item In the structural theory of reductive groups, the words ``pinning'',  ``splitting'', and ``\'epinglage'' are synonyms. We use the word ``splitting''. 
		
	\item If $P$ is a parabolic subgroup of a reductive group over a field, we write $N_P$\index{$N_P$} for the unipotent radical of $P$. We reserve the notation $U_P$ for a different purpose. We write $M_P$\index{$M_P$} for $P/N_P$. When it is clear from the context, $M_P$ also denotes a fixed Levi component of $P$. 	
		
	\item 
	We freely use the language of abelianized Galois cohomology as developed in  \cite{borovoi} and  \cite{labesse1999}. For an overview, cf.~\cite[\S 1]{KSZ}. We also use Kottwitz's more classical formulation \cite{kottwitzelliptic} in terms of centers of Langlands dual groups. 
	\item Let $G$ be a reductive group over $\QQ$. We denote by $\ker^1 (\QQ, G)$\index{$\ker^1 (\QQ, G)$} the kernel set 
	$$ \ker ( \coh^1(\QQ, G) \to \coh^1(\adele, G) ). $$
	It is well known that $\ker^1(\QQ, G)$ has the canonical structure of an abelian group; see for instance \cite{borovoi}. 
	
	\item When normalizing transfer factors, we use the classical normalization of local class field theory as opposed to Deligne's normalization, cf.~\cite[\S \S  4.1--4.2]{KSconvention}.   
	\item Names of Dynkin types are denoted by sans serif letters, e.g., $\mathsf A_n , \mathsf B_n$, etc. 
	
	\item We sometimes use the abbreviations ``LHS'' and ``RHS'' for ``left hand side'' and ``right hand side''. \index[n]{LHS} \index[n]{RHS} 
\end{itemize}

\ignore{

\chapter*{Epigraph}
\renewcommand{\epigraphflush}{center}
\renewcommand{\epigraphsize}{\LARGE}
\setlength{\epigraphwidth}{0.6\textwidth}

\epigraph{ \begin{CJK*}{UTF8}{bsmi}
		其始也，皆收視反聽，耽思傍訊，精騖八極，心遊萬仞。其致也，情曈曨而彌鮮，物昭晰而互進。
\end{CJK*}}{ \begin{CJK*}{UTF8}{bsmi} 陸機《文賦》
	\end{CJK*} \footnotemark}

\footnotetext{``In the beginning, / All external vision and sound are suspended, / Perpetual thought itself gropes in time and space; / Then, the spirit at full gallop reaches the eight limits of the cosmos, / And the mind, self-buoyant, will ever soar to new insurmountable heights. / When the search succeeds, / Feeling, at first but a glimmer, will gradually gather into full luminosity, / Whence all objects thus lit up glow as if each the other's light reflects.''\\	
	\textemdash Excerpt from \textit{Essay on Literature} by Lu Ji (\begin{CJK*}{UTF8}{bsmi}陸機\end{CJK*}) (261-303 AD). Translation by Chen Shixiang.}

}

\numberwithin{equation}{subsection}

\chapter{The orthogonal Shimura varieties}\label{Section global}

\section[General definitions]{General definitions concerning reductive groups}\label{subsec:general defns}
We collect some definitions that will appear repeatedly in the paper. 

\begin{defn}\label{defn:n^G_M}
	Let $G$ be a reductive group over a field $F$. Let $P$ be a parabolic subgroup of $G$, with unipotent radical $N_P$. Let $M$ be a Levi component of $P$.
	\begin{enumerate}
		\item We denote by $A_M$\index{$A_M$} the \emph{split component}\index[n]{split component} of $M$, namely the maximal $F$-split torus in the center of $M$. 
		\item Let $ \Nor_G(M)$ be the normalizer of $M$ in $G$. We denote by $\cW^G_M$\index{$\cW^G_M$} the quotient group
		$ \Nor_G(M) (F) / M (F),$ and denote by $n^G_M$\index{$n^G_M$} the cardinality of $\cW^G_M$.
		\item For any $\gamma \in M(F)$, we define \index{$D^G_M(\cdot)$}
		\begin{align*}
		D^G_M (\gamma) : = \det \big (1-\mathrm{Ad}(\gamma)\mid  \Lie G/\Lie M \big ) \in F.
		\end{align*}
		\item Assume that $F =\QQ_v$ for a place $v$ of $\QQ$. For any $\gamma \in P(\QQ_v)$, we define \index{$\delta_{P(\QQ_v)}(\cdot)$}
		\begin{align*}\delta_{P(\QQ_v)}(\gamma) : = \abs{\det\big(\mathrm{Ad}( \gamma)  \mid  \Lie N_P \otimes \QQ_v \big)}_v \in \RR_{> 0},
		\end{align*} where $\abs{\cdot}_v$ denotes the usual absolute value on $\QQ_v$.
	\end{enumerate} 
\end{defn}

\begin{rem}\label{rem:Weyl group of Levi} 
	In Definition \ref{defn:n^G_M} (2), we in fact have $\Nor_{G} (M)(F) = \Nor_{G}(A_M)(F)$, and $M(F) = \Cent_{G}(A_M)(F)$. Hence $\cW^G_M$ is isomorphic to the image of $\Nor_{G}(A_M) (F)$ in $\Aut (A_M)$. 
\end{rem}

\begin{defn}\label{defn:Weyl denominator}
	Let $G$ be a quasi-split reductive group over a field $F$. By a \emph{Borel pair}\index[n]{Borel pair} in $G$, we mean a pair $(T,B)$ consisting of a maximal torus $T$ in $G$ and a Borel subgroup $B$ of $G$ containing $T$. Given a Borel pair $(T,B)$, we denote the sets of roots, coroots, positive roots, positive coroots by $\Phi(G,T)$, $\Phi(G,T)^{\vee}$, $\Phi(G,T)^+$, $\Phi(G,T)^{\vee, +}$ respectively.\index{$\Phi(G,T)$}\index{$\Phi(G,T)^{\vee}$}\index{$\Phi(G,T)^+$}\index{$\Phi(G,T)^{\vee,+}$} We write $\mathrm{BRD}(T,B)$\index{$\mathrm{BRD}(T,B)$} for the based root datum $(X^*(T), \Phi(G,T), \Phi(G,T)^+, X_*(T), \Phi(G,T)^{\vee}, \Phi(G,T)^{\vee, +})$.  We define the \emph{Weyl denominator}\index[n]{Weyl denominator}\index{$\Delta_B(\cdot)$}
	\begin{align*}
	\Delta_B (\gamma) &: = \prod _{\alpha \in \Phi(G,T)^+} (1- \alpha^{-1} (\gamma)) \in \overline F,& \forall \gamma \in T(\overline F).&
	\end{align*} 
\end{defn}
\begin{defn}\label{defn:q(G)}Let $G$ be a reductive group over $\RR$. We denote by $X_G$\index{$X_G$} the symmetric space associated to $G$, namely $X_G = G(\RR)/K A_G(\RR)^0$, where $K$ is a maximal compact subgroup of $G(\RR)$. Thus $X_G$ is a smooth manifold. We let $q(G)$\index{$q(G)$} be the half of the dimension of $X_G$. 
\end{defn}

\begin{rem}\label{rem:X is conn}
	In Definition \ref{defn:q(G)}, $K$ meets every connected component of $G(\RR)$ by Matsumoto's theorem (see \cite[14.4]{boreltits}). Hence $X_G$ is connected. 
\end{rem}

\begin{defn}\label{defn:cuspidal}
	We call a reductive group $G$ over $\QQ$ \emph{cuspidal}\index[n]{cuspidal (reductive group)} if $G_\RR$ contains elliptic maximal tori and $Z_G^0$ has equal $\QQ$-split and $\RR$-split rank. Equivalently, $(G/A_{G})_{\RR}$ contains  $\RR$-anisotropic maximal tori, where $A_G$ is the split component of $G$ over $\QQ$. 
\end{defn}
\begin{rem}
	\label{usage of cuspidal} In this paper, every reductive group over $\QQ$ that appears will be a direct product of special orthogonal groups and general linear groups. Thus the only case where the center can have different $\QQ$-split and $\RR$-split ranks is when we have a direct factor $\SO_2$ which is non-split over $\QQ$ but split over $\RR$.
\end{rem}
\begin{defn}\label{defn:R-ell}
	Let $G$ be a reductive group over $\QQ$. We say that an element $\gamma \in G(\QQ)$ is \emph{$\RR$-elliptic},\index[n]{$\mathbb R$-elliptic} if there is an elliptic maximal torus $T$ in $G_{\RR}$ such that $\gamma \in T(\RR)$. 
\end{defn}

\section{Generalities on quadratic spaces}\label{subsec:generalities on quad sp} \subsection{}
Let $F$ be a field of characteristic zero, with a fixed algebraic closure $\overline F$. In this paper, all quadratic spaces over $F$ are assumed to be finite-dimensional and non-degenerate. Let $(V,q)$ be a quadratic space over $F$. We denote by $[\cdot ,\cdot]_q : V \otimes V \to F$ the associated bilinear pairing, defined as $[x,y] _q= q(x,y)  - q(x) -q(y)$. When no confusion can arise we simply write $V$ for $(V,q)$, and write $[\cdot, \cdot]$ for $[\cdot,\cdot]_q$. Recall that the \emph{determinant}\index[n]{determinant of a quadratic space} of $q$, denoted by $\det q$, is the image in $F^{\times} / F^{\times, 2}$ of the determinant of the matrix of $q$ under any basis of $V$. We define the \emph{discriminant}\index[n]{discriminant of a quadratic space} of $(V,q)$ to be 
$$
\delta : = (-1)^{\floor{\dim V /2}} \det q \in F^{\times} / F^{\times,2} .$$ 
For $m \in \ZZ_{\geq 1}$, we write $J_m$\index{$J_m$} for the $m\times m$ matrix $$
J_m = \begin{pmatrix}
&& 1 \\
& \reflectbox{$\ddots$} \\
1 
\end{pmatrix}.$$

\begin{defn}\label{defn:qsplitting}
	Let $(V,q)$ be a quadratic space over $F$ of dimension $d$ and discriminant $\delta$. Let $m = \floor{d/2}$. We define the following notions. 
	\begin{enumerate}
		\item A basis $\set{e_1,\cdots, e_d}$ of $V$ is called \emph{hyperbolic}\index[n]{hyperbolic basis}, if the matrix $([e_i,e_j]_q)$ is of the form 
		$$\begin{pmatrix}
		&& J_m \\
		& x \\
		J_m	\end{pmatrix}$$ for some $x\in F^\times$ when $d$ is odd, and is equal to	$$\begin{pmatrix}
		& J_m \\
		J_m	\end{pmatrix}$$ when $d$ is even. Note that when $d$ is even, a hyperbolic basis exists only when $\delta$ is trivial. 
		\item Assume that $d$ is even, and that $\delta$ is non-trivial. In this case a basis $\set{e_1,\cdots, e_d}$ of $V$ is called \emph{near-hyperbolic} \index[n]{near-hyperbolic basis}, if the matrix $([e_i,e_j]_q)$ is equal to		$$\begin{pmatrix}
		&&&& J_{m-1}\\
		&& 1 \\
		&&& - x
		\\
		J_{m-1}	\end{pmatrix}$$ for some $x\in F^{\times}$. Note that $x$ is a lift of $\delta \in F^{\times}/F^{\times, 2}$. We say that $x$ is the \emph{discriminant}\index[n]{discriminant of a near-hyperbolic basis} of  $\set{e_1,\cdots, e_d}$.  
	\end{enumerate}
\end{defn}
\begin{defn}\label{defn:qsplit quad sp} We call $(V,q)$ \emph{quasi-split},\index[n]{quasi-split quadratic space} if there exists a hyperbolic basis or a near-hyperbolic basis of $V$. If there exists a hyperbolic basis we also say that $V$ is \emph{split};\index[n]{split quadratic space} this is equivalent to requiring that $V$ contains a totally isotropic subspace of dimension $\floor{\dim V/2}$.  
\end{defn}

\begin{eg}\label{eg:real quasi-split}
Let $F= \RR$. Then a quadratic space over $\RR$ of signature $(p,q)$ is quasi-split if and only $p-q \in \set{1,-1, 2}$. For any $p\in \ZZ_{\geq 1}$, the quadratic spaces of signature $(p,p)$ and $(p+1,p-1)$ are both quasi-split, and their discriminants are $1$ and $-1 \in \RR^{\times}/\RR^{\times,2}$ respectively. 
\end{eg}
\subsection{}\label{subsubsec:standard RD}
 Let $m \in \ZZ_{\geq 1}$. We denote by $\mathrm{RD}(\mathsf B_m)$\index{$\mathrm{RD}(\mathsf B_m)$} the \emph{standard type $\mathsf B_m$ root datum}\index[n]{standard type $\mathsf B_m$ root datum}, given by $$ (\ZZ^m = \mathrm{span}_{\ZZ} \set{\epsilon_1,\cdots, \epsilon_m}, R, \ZZ^m = \mathrm{span}_{\ZZ} \set{\epsilon_1^{\vee},\cdots, \epsilon_m^{\vee}} , R^{\vee}) ,$$ where $\lprod{\epsilon_i, \epsilon_j^{\vee}} = \delta_{i,j}$, and \begin{align*}
 	R & = \set{\pm \epsilon_i \mid 1\leq i \leq m} \cup \set {\pm \epsilon_i \pm \epsilon_j  \mid  1 \leq i < j \leq m}, \\
 	R^{\vee} & = \set{\pm 2 \epsilon_i^{\vee} \mid 1\leq i \leq m} \cup \set{\pm \epsilon_i^{\vee} \pm \epsilon_j^{\vee}  \mid  1 \leq i < j \leq m}. 
 \end{align*}
(If $m=1$, then $R = \set{\epsilon_1}, R^{\vee} = \set{2\epsilon_1^{\vee}}$.)
By the \emph{standard simple roots}\index[n]{standard simple roots (of type $\mathsf B_m$)} we mean the following choice of simple roots: 
$$ \epsilon_1 - \epsilon_2,\cdots, \epsilon_{m-1} -\epsilon_m , \epsilon_{m}. $$ We denote by $\mathrm{BRD}(\mathsf B_m)$\index{$\mathrm{BRD}(\mathsf B_m)$} the based root datum corresponding to the above choice of simple roots, called the \emph{standard based root datum}.
The Weyl group of $\mathrm{RD}(\mathsf B_m)$ is naturally identified with $\set{\pm 1}^m \rtimes \mathfrak S_m$. 		

Similarly, for $m \in \ZZ_{\geq 1}$ we denote by $\mathrm{RD}(\mathsf D_m)$\index{$\mathrm{RD}(\mathsf D_m)$} the \emph{standard type $\mathsf D_m$ root datum}\index[n]{standard type $\mathsf D_m$ root datum}, given by
$$ (\ZZ^m , R, \ZZ^m , R^{\vee}), $$ where 
\begin{align*}
R &  = \set{\pm \epsilon_i  \pm \epsilon_j \mid 1 \leq i < j \leq m } , \\ 
R^{\vee} & = \set{\pm \epsilon_i^{\vee} \pm \epsilon_j^{\vee} \mid  1 \leq i < j \leq m  }.
\end{align*} (If $m=1$, then $R = R^{\vee} = \emptyset.$) By the \emph{standard simple roots}\index[n]{standard simple roots (of type $\mathsf D_m$)} we mean the following choice of simple roots: 
$$ \epsilon_1 - \epsilon_2,\cdots, \epsilon_{m-1} -\epsilon_m , \epsilon_{m-1} +\epsilon_m . $$
We denote by $\mathrm{BRD}(\mathsf D_m)$\index{$\mathrm{BRD}(\mathsf D_m)$} the corresponding based root datum. The Weyl group of $\mathrm{RD}(\mathsf D_m)$ is naturally identified with $(\set{\pm 1} ^m)' \rtimes \mathfrak S_m$, where $(\set{\pm 1} ^m)'$\index{$(\set{\pm 1}^m)'$} denotes the kernel of the homomorphism $\set{\pm 1}^m \to \set{\pm 1}$ taking the product of the coordinates. 

\begin{defn}\label{defn:orienting F bar}
	Let $\alpha \in \overline F$ be an element such that $\alpha^2 \in F^{\times}$ and $\alpha \notin F$. Let $\Uni(1)_{\alpha}$\index{$\Uni(1)_{\alpha}$} be the norm-one subtorus of $\Res_{F(\alpha)/ F} \Gm$. We have a canonical isomorphism $\Uni(1)_{\alpha, \overline F} \cong \GG_{m,\overline F}$ corresponding to the inclusion $F(\alpha) \hookrightarrow \overline F$. In particular, we canonically identify $X^*(\Uni(1)_{\alpha})$ and  $X_*(\Uni(1)_{\alpha})$ with $\ZZ$. We also have a canonical injective $F$-homomorphism $\iota_{\alpha}: \Uni(1)_{\alpha} \to \GL_2$,\index{$\iota_{\alpha}$} which represents the multiplication action of $\Uni(1)_{\alpha}$ on $F(\alpha)$ under the $F$-basis $\set{1,\alpha}$ of $F(\alpha)$. If $F = \RR, \overline F = \CC, \alpha = \sqrt{-1}$, we simply write $\Uni(1)$ for $\Uni(1)_{\alpha}$.  
	\end{defn}
\subsection{} \label{subsubsec:Borel pairs assoc to hd and nhd}  Let $V=(V,q)$ be a quadratic space over $F$ of dimension $d$ and discriminant $\delta$. Let $G = \SO(V)$. Then $G$ is a reductive algebraic group over $F$, and semi-simple if $d\neq 2$. The absolute rank of $G$ is $m = \floor{d/2}$. 

Assuming that $(V,q)$ is quasi-split, we shall obtain an explicit description of a Borel pair in $G$ and the associated based root datum as follows. There are two cases to consider.

The first case is when $V$ has a hyperbolic basis $\mathbb B  = \set{e_1,\cdots,e_d}$. We then identify $G$ with a subgroup of $\GL_d$ using the basis $\mathbb B$. When $d$ is odd, we obtain an $F$-embedding \index{$\iota_{\mathbb B}$}
\begin{align*}
\iota_{\mathbb B} &: \Gm^{m}  \To G  , &  
(z_1,\cdots, z_m)  & \longmapsto \mathrm{diag}(z_1, \cdots , z_m, 1 , z_m^{-1} ,\cdots, z_1^{-1}).&
\end{align*} 
When $d$ is even, we obtain an $F$-embedding \begin{align*}
	\iota_{\mathbb B} & : \Gm^{m}  \To G  , & 
	(z_1,\cdots, z_m) &  \longmapsto \mathrm{diag}(z_1, \cdots , z_m,  z_m^{-1} ,\cdots, z_1^{-1}).&
\end{align*}  For both parities of $d$, the image $T$ of $\iota_{\mathbb B}$ is a split maximal torus in $G$. Also, the intersection of $G$ with the upper triangular Borel subgroup of $\GL_d$ is a Borel subgroup $B$ of $G$ containing $T$. Under $\iota_{\mathbb B}$, the based root datum $\mathrm{BRD}(T,B)$ is identified with the standard based root datum $\mathrm{BRD}(\mathsf B_m)$ (resp.~$\mathrm{BRD}(\mathsf D_m)$) when $d$ is odd (resp.~even).

The second case is when $d$ is even, $\delta$ is non-trivial, and $V$ has a near-hyperbolic basis $\mathbb B= \set{e_1,\cdots, e_d}$. Let $x\in F^\times$ be the discriminant of $\mathbb B$ (see Definition \ref{defn:qsplitting}), and fix a square root $\alpha \in \overline F$ of $x$. We identify $G$ with a subgroup of $\GL_d$ using the basis $\mathbb B$, and obtain an $F$-embedding \index{$\iota_{\alpha, \mathbb B}$}
\begin{align*}
\iota_{\alpha, \mathbb B} : \Gm^{m-1}\times \Uni(1)_{\alpha}  & \To G  \\  (z_1,\cdots, z_{m-1}, z_m) & \longmapsto \mathrm{diag}(z_1,\cdots , z_{m-1}, \iota_{\alpha}(z_m) , z_{m-1}^{-1},\cdots, z_1^{-1}).
\end{align*} 
 Here $\Uni(1)_{\alpha}$ and $\iota _{\alpha}: \Uni(1)_{\alpha }\to \GL_2$ are as in Definition \ref{defn:orienting F bar}. The image $T$ of $\iota_{\alpha, \mathbb B}$ is a maximal torus in $G$. Recall from Definition \ref{defn:orienting F bar} that $X^*(\Uni(1)_{\alpha})$ and $X_*(\Uni(1)_{\alpha})$ are canonically identified with $\ZZ$, so $X^*(\GG_m^{m-1} \times \Uni(1)_{\alpha})$ and $X_*(\GG_m^{m-1} \times \Uni(1)_{\alpha})$ are canonically identified with $\ZZ^m$. Under $\iota_{\alpha,\mathbb B}$, the root datum of $(T_{\overline F}, G_{\overline F})$ is identified with $\mathrm{RD}(\mathsf D_m)$. The standard based root datum $\mathrm{BRD}(\mathsf D_m)$ thus gives rise to a Borel subgroup $B_{\overline F}$ of $G_{\overline F}$ containing $T_{\overline F}$. The $\Gamma_F$-action on $X^*(\Gm^{m-1} \times \Uni(1)_{\alpha}) \cong \ZZ^m$ factors through $\Gal(F(\alpha)/F)$, and the non-trivial element of $\Gal(F(\alpha)/F)$ acts by $\ZZ^m \to \ZZ^m, (a_1,\cdots, a_m) \mapsto (a_1,\cdots, a_{m-1}, -a_m)$. Hence the $\Gamma_F$-action preserves the set of standard simple roots. It follows that $B_{\overline F}$ comes from a Borel subgroup $B$ of $G$. Thus $(T,B)$ is a Borel pair in $G$, and $\iota_{\alpha, \mathbb B}$ induces an isomorphism between $\mathrm{BRD}(\mathsf D_m)$ and $\mathrm{BRD}(T,B)$.  
 
\begin{prop}\label{prop: even TFAE} 	Let $(V,q)$ be a quadratic space over $F$ of dimension $d$ and discriminant $\delta$. Let $G= \SO(V)$. Assume that $d \geq 3$. The following statements hold. 
\begin{enumerate}
	\item The quadratic space $V$ is split if and only if $G$ is split.
	\item If $d$ is odd, then $G$ is split if and only if $G$ is quasi-split.
	\item If $d$ is even, then $G$ is split if and only if $G$ is quasi-split and $\delta$ is trivial. 
	\item Assume that $d$ is even, $\delta$ is non-trivial, and $V$ is quasi-split. Then $G$ is quasi-split.
	\item Assume that $d$ is even, $\delta$ is non-trivial, and $G$ is quasi-split over $F$. Then $G$ is split over $F(\alpha)$, for any $\alpha \in \overline F$ whose square is a lift of $\delta$. 
	\item Keep the assumptions in (5), and assume that $F$ is a non-archimedean local field of characteristic zero. Then $G$ is unramified if and only if $F(\alpha)$ is unramified over $F$, if and only if $\delta\in F^{\times}/F^{\times,2}$ has a representative in $\oo_F^{\times}/\oo_F^{\times,2}$. 
	\item Suppose $F=\QQ_p$ for an odd prime $p$. Then $(V,q)$ is quasi-split if and only if the Hasse invariant is $(-1)^{\frac{p-1}{2}v_p(\delta) \floor{\frac{d-1}{2}}}$. Here  $v_p(\delta)$ is well defined in $\ZZ/2\ZZ$. 
	\item Suppose $F=\QQ$. Then $(V,q)\otimes_{\QQ} \QQ_p$ is quasi-split for almost all primes $p$.  
\end{enumerate}	
\end{prop} 
\begin{proof} 
	\textbf{(1)} This is well known; see for instance \cite[Prop.~2.14]{PR}. 
	
	\textbf{(2)} This follows from the fact that the Dynkin diagram of type $\mathsf B_{(d-1)/2}$ does not have non-trivial automorphisms.	
	
	\textbf{(3)} If $G$ is split, then $V$ is split by part (1), and so $\delta$ is trivial. Conversely, assume that $G$ is quasi-split and $\delta$ is trivial. By the abstract classification of quasi-split semi-simple groups of type $\mathsf D_m$ (where $m= \frac{d}{2}\geq 2$), we know that the split rank of $G$ is at least $\frac{d}{2}-1$. This implies that $V$ is an orthogonal direct sum of $\frac{d}{2}-1$ hyperbolic planes and a $2$-dimensional quadratic space $V_0$, by \cite[Prop.~2.14]{PR}. The discriminant of $V_0$ is the same as that of $V$, which is trivial. Therefore there is a basis of $V_0$ under which the matrix of the quadratic form on $V_0$ is $\diag(a, -ab^2)$ for some $a, b \in F^{\times}$. Clearly this implies that $V_0$ is a hyperbolic plane. Hence $V$ is split, and therefore $G$ is split by (1).

	\textbf{(4)} By \S \ref{subsubsec:Borel pairs assoc to hd and nhd}, $G$ admits a Borel subgroup over $F$.

	\textbf{(5)} This follows from (3) by base changing both $V$ and $G$ from $F$ to $F(\alpha)$.  
	
	\textbf{(6)} Since $\delta$ is non-trivial, by (1) we know that $G$ is non-split. By the abstract classification of quasi-split non-split semi-simple groups of type $\mathsf D_m$ (with $m \geq 2$), we know that $G$ splits over a unique quadratic extension $E/F$ inside $\overline F$, and that any splitting field of $G$ inside $\overline F$ must contain $E$. Thus $G$ is unramified if and only if $E/F$ is unramified. By (5), we know that $E =  F(\alpha)$. Thus $G$ is unramified if and only if $F(\alpha)$ is unramified over $F$, which is also equivalent to that $\delta$ has a representative in $\oo_F^{\times}/\oo_F^{\times, 2}$.  
	
	\textbf{(7)} If $(V,q)$ is quasi-split, then it has matrix representation $$\begin{pmatrix}
		I_{\frac{d-1}{2}} \\ & x \\ && -I_{\frac{d-1}{2}}
	\end{pmatrix}$$ when $d$ is odd and $$\begin{pmatrix}
	I_{\frac{d}{2}} \\ & -x \\ && -I_{\frac{d}{2}-1}
\end{pmatrix}$$ when $d$ is even, for some $x\in F^{\times}$ representing $\delta$. Hence the Hasse invariant is  $(x,-1)_p^{\floor{\frac{d-1}{2}}} =(-1)^{\frac{p-1}{2}v_p(x) \floor{\frac{d-1}{2}}}$. This proves the ``only if'' direction. The ``if'' direction follows since two quadratic spaces over $\QQ_p$ with the same dimension, discriminant, and Hasse invariant are isomorphic.   
	
	\textbf{(8)} For almost all $p$, $v_p(\delta) = 0 \in \ZZ/2\ZZ$ and the Hasse invariant of $(V,q)$ at $p$ is trivial. By (7) we know that $(V,q)\otimes_{\QQ} \QQ_p$ is quasi-split for such $p$.  
\end{proof}
\begin{rem}From the assumptions that $d$ is even, $\delta$ is non-trivial, and $G = \SO(V)$ is quasi-split over $F$, it does not follow that $V$ is quasi-split. For example, the quadratic spaces over $\RR$ of signatures $(n+2, n)$ and $(n,n+2)$ define isomorphic special orthogonal groups, but only the former quadratic space is quasi-split; cf.~Example \ref{eg:real quasi-split}. 
\end{rem}

\section[Shimura data and rational boundary components]{Generalities on Shimura data and rational boundary components}\label{subsec:boundary} In this section we collect some general facts concerning the formalism of mixed Shimura data and rational boundary components in \cite{pink1989compactification}.
\subsection{}
According to the definition of Pink \cite[Chap.~2]{pink1989compactification}, a \emph{mixed Shimura datum}\index[n]{mixed Shimura datum} is a tuple $(P, U, \Y,h)$, where $P$ is a connected linear algebraic group over $\QQ$, $U$ is a subgroup of the unipotent radical of $P$ that is normal in $P$, $\Y$ is a left homogeneous space under the subgroup $P(\RR)U(\CC)$ of $P(\CC)$, and $h$ is a $P(\RR)U(\CC)$-equivariant map $\Y \to \Hom (\mathbb S_\CC, P)$, satisfying the axioms in \cite[2.1]{pink1989compactification}. (Here $\mathbb S : = \Res_{\CC/\RR} \GG_m$.\index{$\mathbb S$}) If $h$ is clear from the context, we omit it from the notation. If $U$ is trivial, we also omit it from the notation. The mixed Shimura datum is called \emph{pure}\index[n]{pure Shimura datum} if $P$ is reductive. Note that the notion of a pure Shimura datum according to Pink's definition is less restrictive than Deligne's definition in \cite[2.1]{deligne1979varietes}, in that $h$ is allowed to be non-injective, cf.~\cite[2.2 (d)]{pink1989compactification}. In the sequel all pure Shimura data are understood in the sense of Pink. 

Some comments on the homogeneous space $\Y$ are in order. First, note that $P(\RR)U(\CC)$ is the preimage of $(P/U)(\RR)$ along the map $P(\CC) \to (P/U)(\CC)$, since $\coh^1(\RR, U)$ is trivial. It follows that $P(\RR)U(\CC)$ is a closed Lie subgroup of the real Lie group $P(\CC)$. Recall that for any real Lie group $\mathcal G$, a left homogeneous space under $\mathcal G$ is a set $S$ equipped with a transitive left action of $\mathcal G$ such that the stabilizers are closed Lie subgroups of $\mathcal G$. Then $S$ has the unique structure of a smooth manifold such that the $\mathcal G$-action is smooth. In the definition of a mixed Shimura datum, $\Y$ is required to be a left homogeneous space under the real Lie group $P(\RR)U(\CC)$, and so $\Y$ is canonically a smooth manifold. As explained in \cite[2.2]{pink1989compactification}, $\Y$ has finitely many connected components, and the smooth structure on $\Y$ can be upgraded to a canonical complex structure, which is invariant under $P(\RR)U(\CC)$.

\subsection{}
By definition (\cite[2.3]{pink1989compactification}), a \emph{morphism}\index[n]{morphism between mixed Shimura data} between two mixed Shimura data $(P, U, \Y, h)$ and $(P', U', \Y', h')$ is a pair $(\pi, F)$, where $\pi: P \to P'$ is a homomorphism of $\QQ$-algebraic groups, and $F: \Y \to \Y'$ is a map, required to satisfy the following conditions: 
\begin{itemize}
	\item $\pi$ maps $U$ into $U'$.
	\item $F$ is equivariant with respect to the homomorphism $P(\RR)U(\CC) \to P'(\RR) U'(\CC)$ induced by $\pi$. 
	\item For any $y \in \Y$, the homomorphism $h'(F(y)): \mathbb S_{\CC} \to P'_{\CC}$ is equal to the composite homomorphism 
	$$ \mathbb S_{\CC} \xrightarrow{h(y)} P_{\CC} \xrightarrow{\pi} P'_{\CC}. $$
\end{itemize}
As shown in \cite[2.4]{pink1989compactification}, if $(\pi,F)$ is a morphism as above, then $F$ is automatically holomorphic with respect to the canonical complex structures on $\Y$ and $\Y'$. 
\subsection{}\label{subsubsec:pure quotient}
Let $(P, U, \Y,h)$ be a mixed Shimura datum. 
In \cite[2.9]{pink1989compactification}, Pink constructs the \emph{quotient}\index[n]{quotient of a mixed Shimura datum} of $(P, U, \Y,h)$ by a normal subgroup $P_0$ of $P$. This is a mixed Shimura datum for the group $P/P_0$ equipped with a morphism from $(P, U, \Y,h)$ satisfying a universal property. In the following, we give an alternative construction of the quotient in the special case where $P_0$ is the unipotent radical of $P$. 

Let $W$ be the unipotent radical of $P$. We write $G$ for $P/W$, and write $\pi$ for the projection $P \to G$. Since $\coh^1(\RR, W)$ is trivial, and since $W(\RR)$ and $U(\CC)$ are connected, the natural map $P(\RR)U(\CC) \to G(\RR)$ is surjective and induces an isomorphism $\pi_0(P(\RR) U(\CC)) \isom \pi_0(G(\RR))$. In particular, $\pi_0(G(\RR))$ acts on $\pi_0(\Y)$.

Suppose we have a left $\pi_0(G(\RR))$-set $\Gamma$ and a $\pi_0(G(\RR))$-equivariant map $\lambda: \pi_0(\Y) \to \Gamma$. 
We define the map \index{$\mathbb H_{\lambda}$}
\begin{align*}
	\mathbb H_{\lambda} : \Y & \To \Gamma \times \Hom(\mathbb S_{\CC}, G_{\CC}) .\\
	y & \longmapsto (\lambda([y]), \pi \circ h(y)).
\end{align*}
We have a diagonal $G(\RR)$-action on $\Gamma \times \Hom(\mathbb S_{\CC}, G_{\CC})$, where the action on the second factor is by conjugation. The map $\mathbb H_{\lambda}$ is equivariant with respect to the natural homomorphism $P(\RR)U(\CC) \to G(\RR)$. Let $\X_{\lambda} : = \im (\mathbb H_{\lambda})$. Let $h_{\lambda}: \X \to \Hom (\mathbb S_{\CC}, G_{\CC})$ be the projection map to the second factor. It is easy to check that $(G, \X_{\lambda}, h_{\lambda})$ is a pure Shimura datum, and that the pair $(\pi : P \to G, ~\mathbb H_{\lambda}: \Y \to \X_{\lambda})$ is a morphism $(P,U, \Y, h) \to (G, \X_{\lambda}, h_{\lambda})$ between mixed Shimura data. Since $\mathbb H_{\lambda} : \Y \to \X_{\lambda}$ is surjective by the definition of $\X_{\lambda}$, it induces a surjection $\pi_0 (\mathbb H_{\lambda}) : \pi_0(\Y) \to \pi_0(\X_{\lambda})$. 
\begin{lem}\label{lem:constr of pure quotient}
Let $\Gamma$ and $\lambda$ be as above. The following statements hold. 
\begin{enumerate}
	\item The map $\X_{\lambda} \to \Gamma$ given by the projection to the first factor induces an injection $\pi_0(\X_{\lambda}) \to \Gamma$. 
	\item The surjection $\pi_0 (\mathbb H_{\lambda}) : \pi_0(\Y) \to \pi_0(\X_{\lambda})$ is a bijection if and only if $\lambda$ is injective. 
	\item If $\lambda$ is injective, then the morphism $(\pi, \mathbb H_{\lambda}): (P,U, \Y , h) \to (G,\X_{\lambda}, h_{\lambda})$ identifies $(G,\X_{\lambda}, h_{\lambda})$ with the quotient of $(P,U, \Y , h)$ by $W$. 
\end{enumerate}		
\end{lem}
\begin{proof}
\textbf{(1)} A connected component of $\X_{\lambda}$ is the same thing as a $G(\RR)^0$-orbit in $\X_{\lambda}$, but $G(\RR)^0$ acts trivially on $\Gamma$.

\textbf{(2)} The composition of $\pi_0(\mathbb H_{\lambda})$ followed by the injection $\pi_0(\X_{\lambda}) \to \Gamma$ in part (1) is equal to $\lambda$.

\textbf{(3)} Let $(\pi, F): (P, U, \Y, h) \to (G,\X_{\mathrm{abs}}, h_{\mathrm{abs}})$ be the abstract quotient by $W$, which is characterized by a universal property and constructed in \cite[2.9]{pink1989compactification}. By the universal property, there is a unique $G(\RR)$-equivariant map $j: \X_{\mathrm{abs}}\to \X_{\lambda}$ such that $h_{\lambda} \circ j = h_{\mathrm{abs}}$ and $j \circ F = \mathbb H_{\lambda}$. We only need to show that $j$ is a bijection. Since $\mathbb H_{\lambda} : \Y \to \X_{\lambda}$ is surjective, so is $j$. By part (2), $j$ induces an injection $\pi_0(\X_{\mathrm{abs}}) \to \pi_0(\X_{\lambda})$. It remains to show that the restriction of $j$ to each connected component of $\X_{\mathrm{abs}}$ is injective. For this, it is enough to show that the restriction of $h_{\lambda} \circ j = h_{\mathrm{abs}}$ to each connected component of $\X_{\mathrm{abs}}$ is injective. But this is \cite[2.12]{pink1989compactification}.
\end{proof}
 
\subsection{}\label{subsubsec:formalism} We recall the formalism of rational boundary components developed in \cite[Chap.~4]{pink1989compactification}. 
Let $(G,\X) = (G,\X,h)$ be a pure Shimura datum. For simplicity, we assume that $G^{\ad}$ is $\QQ$-simple, which will suffice for our applications. We denote by $\admpar$\index{$\admpar$} the set of \emph{admissible parabolic subgroups of $G$}\index[n]{admissible parabolic subgroup}, namely $G$ itself and the maximal proper parabolic subgroups of $G$ (defined over $\QQ$). For any $P \in \admpar$, Pink \cite[4.7, 4.8]{pink1989compactification} defines a canonical normal subgroup $P^{\Pink}$\index{$P^{\Pink}$} of $P$, and a canonical normal subgroup $U_P$\index{$U_P$} of $P^{\Pink}$ contained in the unipotent radical of $P^{\Pink}$.\footnote{Our $P$, $P^{\Pink}$, and $ U_P$ are denoted respectively by $Q$, $ P_1$, and $U_1$ in \cite[4.7, 4.8]{pink1989compactification}.} 
As $G$ is reductive, the proof of \cite[4.8]{pink1989compactification} shows that the unipotent radical of $P^{\Pink}$ is equal to the unipotent radical $N_P$ of $P$. In particular, the subgroup $P^{\Pink} \subset G$ uniquely determines $P$. We shall write $M_P$\index{$M_P$} for $P/N_P$ and write $G_P$\index{$G_P$} for $P^{\Pink}/N_P$.  

We define \index{$\mathscr Y_P$} 
$$\mathscr Y_P : = \pi_0 (\X) \times \Hom(\mathbb S_{\CC}, P^{\Pink}_{\CC}),$$ equipped with the diagonal action of $P(\RR)U_P(\CC)$. Here the action on the first factor is via  $\pi_0(P(\RR)U_P(\CC)) \cong \pi_0(P(\RR)) \to \pi_0(G(\RR))$, and on the second factor via conjugation. We write $p_1^P$ and $ p_2^P$ for the projection maps from $\mathscr Y_P$ to the two factors. In \cite[4.11]{pink1989compactification}, Pink defines a canonical $P(\RR)$-equivariant map\index{$\omega_P$}
 $$\omega_P: \X \To \mathscr Y_P$$
 such that $p_1^P \circ \omega_P$ is the natural projection $\X \to \pi_0(\X)$.
 
 By definition (\cite[4.11]{pink1989compactification}),  a \emph{rational boundary component}\index[n]{rational boundary component} of $(G,\X)$ is a pair $(P, \Y)$, where $P \in \admpar$, and $\Y$ is any $P^{\Pink}(\RR)U_P(\CC)$-orbit in $\mathscr Y_P$ such that $\Y \cap \im(\omega_P) \neq \emptyset$. We denote by $\RBC(G,\X)$ or simply $\RBC$ \index{$\RBC$} the set of all rational boundary components. For each $P\in \admpar$, we denote by $\RBC_P(G, \X)$ or simply $\RBC_P$\index{$\RBC_P$} the set of all rational boundary components whose first coordinate is $P$. For $(P, \Y) \in \RBC$, we write $\X^\Y$\index{$\X^\Y$} for the subset $\omega_P^{-1}(\Y)$ of $\X$. We have the following facts (see \cite[Chap.~4]{pink1989compactification}):
 \begin{enumerate}
  	\item[(I)] For $(P,\Y) \in \RBC$, the $P^{\Pink}(\RR)U_P(\CC)$-action on $\Y$ and the map $p_2^P |_{\Y} : \Y \to \Hom(\mathbb S_{\CC}, P^{\Pink}_{\CC})$ make the tuple $(P^{\Pink}, U_P, \Y)$ a mixed Shimura datum. 
  	\item[(II)] For $(P, \Y) \in \RBC$, the set $\X^{\Y}$ is the union of some connected components of $\X$. The map $\omega_P$ maps $\X^{\Y}$ injectively and holomorphically into $\Y$, inducing a bijection 	\index{$\gamma_{\Y}$}
  	\begin{align}\label{eq:gamma_Y}
\gamma_{\Y}:  \pi_0(\X^{\Y}) \isom \pi_0 (\Y). 
  	\end{align} Moreover, the map $\pi_0(\Y) \to \pi_0(\X)$ induced by $p_1^P |_{\Y} : \Y \to \pi_0(\X)$ is the inverse of $\gamma_\Y$. 
  	\item[(III)]  For each fixed $P \in \admpar$, $\X$ is the disjoint union \begin{align}\label{eq:disjoint union}
  	\X  = \coprod_{(P,\Y) \in \RBC_P} \X^{\Y}. 	\end{align}		
 \end{enumerate}

\subsection{}\label{subsubsec:taking pure quotient} We keep the setting of \S \ref{subsubsec:formalism}. Let $P \in \admpar$. 
For each $(P,\Y) \in \RBC_P$, we let $(G_P, \X_\Y)$\index{$\X_\Y$} be the quotient of the mixed Shimura datum $(P^{\Pink}, U_P, \Y)$ by the unipotent radical $N_P$ of $P$ (which is also the unipotent radical of $P^{\Pink}$), and let $F_\Y: \Y \to \X_\Y$\index{$F_{\Y}$} be the canonical $P^{\Pink}(\RR)U_P(\CC)$-equivariant map. By Lemma \ref{lem:constr of pure quotient}, we know that $F_\Y$ induces a bijection between the sets of connected components. 

Let $\pi_{\Y}$\index{$\pi_\Y$} be the composition 
$$\pi_{\Y} :  \X^{\Y} \xrightarrow{\omega_P} \Y \xrightarrow{F_{\Y}} \X_\Y. $$ 
Then $\pi_\Y$ is holomorphic and induces a bijection between the sets of connected components, since both $\omega_P|_{\X^{\Y}}$ and $F_\Y$ have these properties. Moreover, $\pi_\Y$ is equivariant with respect to the surjective Lie group homomorphism $P^{\Pink}(\RR)U_P(\CC) \to G_P(\RR)$. In particular, $\pi_\Y$ is a surjective submersion, since  $\X^\Y$ (resp.~$\X_\Y$) is a left homogeneous space under $P^{\Pink}(\RR) U_P(\CC)$ (resp.~$G_P(\RR)$). 

Let $\X_P$\index{$\X_P$} be the disjoint union
$$\X_P: = \coprod_{(P,\Y) \in \RBC_P} \X_\Y, $$ as a complex manifold with a $G_P(\RR)$-action. In view of (\ref{eq:disjoint union}), we have a map \index{$\pi_P$}
$$ \pi_P : = \coprod_{(P,\Y) \in \RBC_P} \pi_\Y: \X \To \X_P. $$
Then $\pi_P$ is holomorphic, surjective, submersive, equivariant with respect to  $P^{\Pink}(\RR)U_P(\CC) \to G_P(\RR)$, and induces a bijection between the sets of connected components, since each $\pi_\Y$ has these properties. When $P= G$, the map $\pi_G$ is an isomorphism.  

Consider the set-theoretic disjoint union \footnote{While we shall only consider $\X^*$ as a set, there is a natural \emph{Satake topology} on $\X^*$; see \cite[6.2]{pink1989compactification}. Under this topology, $\X^*$ contains $\X$ as a dense open subset. } \index{$\X^*$}
\begin{align}\label{eq:X^*}
\X^* = \coprod_{P \in \admpar} \X_P = \coprod_{(P,\Y) \in \RBC} \X_\Y. 
\end{align}
There is a natural $G(\QQ)$-action on $\X^*$, satisfying the following properties (see \cite[4.16, 6.2]{pink1989compactification}): 
\begin{itemize}
	\item The action respects the stratification of $\X^*$ by the subsets $\X_{\Y}$.  
	\item For $g\in G(\QQ)$ and $P \in \admpar$, we have $g(\X_P) = \X_{gPg^{-1}}$. In particular, $\mathrm{Stab}_{G(\QQ)} \X_P = P(\QQ)$. 
	\item For $P \in \admpar$, the map $\pi_P : \X \to \X_P$ is $P(\QQ)$-equivariant. Here $P(\QQ) $ acts on $X_P$ since $\mathrm{Stab}_{G(\QQ)} \X_P = P(\QQ)$. Moreover, the $P(\QQ)$-action on $\X_P$ factors through the quotient map $P(\QQ) \to M_P(\QQ)$. 
\end{itemize}

 Let  $P \in \admpar$. As discussed above we have an $M_P(\QQ)$-action on $\X_P$. Since $M_P(\QQ)$ is dense in $M_P(\RR)$, there is at most one way to extend this action to a continuous $M_P(\RR)$-action. It is shown in \cite[3.6]{pink1989compactification} that such an extension indeed exists. Since we need to explicitly describe this $M_P(\RR)$-action later for the orthogonal Shimura datum, we give its construction in the following proposition. 

\begin{prop}\label{prop:action}Keep the setting of \S \ref{subsubsec:formalism}, and let $P \in \admpar$. The following statements hold. 
	\begin{enumerate}
		\item There is a unique extension of the $G_P(\RR)$-action on $\X_P$ to an $M_P(\RR)$-action such that the map $\pi_P : \X \to \X_P$ is equivariant with respect to the homomorphism $P(\RR) \to M_P(\RR)$.
		\item  The $M_P(\RR)$-action on $\X_P$ in (1) factors through the natural homomorphism
		\begin{align}\label{eq:natural hom from M_P}
	M_P(\RR) & \To \pi_0(M_P(\RR)) \times \Aut(G_{P,\RR}) \\ \nonumber
	m & \longmapsto ([m], \Int m). 
		\end{align}
		\item The $M_P(\RR)$-action on $\X_P$ in (1) is transitive and continuous. Its restriction to $M_P(\QQ)$ coincides with the $M_P(\QQ)$-action discussed in \S \ref{subsubsec:taking pure quotient}. 
	\end{enumerate}
\end{prop}

 \begin{proof}
 	\textbf{(1)} The uniqueness immediately follows from the surjectivity of $\pi_P$. We prove the existence. Using the canonical isomorphism $\pi_0(P(\RR)U_P(\CC)) \cong \pi_0 (M_P(\RR))$, we view $\pi_0(\X)$ as a $\pi_0(M_P(\RR))$-set. In particular, $\pi_0(\X)$ is a  $\pi_0(G_P(\RR))$-set. To simplify notation, we write $\mathbb H_P$ for the set $\pi_0(\X) \times \Hom(\mathbb S_{\CC}, G_{P,\CC})$, which is equipped with the diagonal $G_P(\RR)$-action as in \S \ref{subsubsec:pure quotient} (where we take $\Gamma$ to be $\pi_0(\X)$). The $G_P(\RR)$-action on $\mathbb H_P$ extends to an $M_P(\RR)$-action in the obvious way (using the normality of $G_P$ in $M_P$). We have a natural map
 	\begin{align*}
\mathscr F_P: \mathscr Y_P  = \pi_0 (\X) \times \Hom(\mathbb S_{\CC}, P^{\Pink}_{\CC}) & \To \mathbb H_P =  \pi_0(\X) \times \Hom(\mathbb S_{\CC}, G_{P,\CC}) \\
([x], l ) & \longmapsto ([x], ( \mathbb S_{\CC} \xrightarrow{l} P^{\Pink}_{\CC} \to G_{P,\CC})),
 	\end{align*}
 	which is equivariant with respect to $P(\RR) \to M_P(\RR)$. 
 	 	
 	 	Let $(P,\Y) \in \RBC_P$. We denote by $\lambda_{\Y}$ the injective map $$ \pi_0(\Y) \xrightarrow{\gamma_{\Y}^{-1}} \pi_0(\X^{\Y}) \hookrightarrow \pi_0(\X),$$ where $\gamma_{\Y}$ is as in (\ref{eq:gamma_Y}). As in \S \ref{subsubsec:pure quotient}, $\lambda_\Y$ induces a map $\mathbb H_{\lambda_\Y} : \Y \to \mathbb H_P$, whose image is denoted by $\X_{\lambda_\Y}$. By Lemma \ref{lem:constr of pure quotient}, we may assume that $\X_\Y$ is equal to $\X_{\lambda_{\Y}}$, and that the map $F_\Y : \Y \to \X_\Y$ is equal to the map $\mathbb H_{\lambda_\Y}$. Then we have a commutative diagram 
 	 $$ \xymatrix{ \Y \ar@{^{(}->}[r] \ar[d]^{F_\Y} & \mathscr Y_P \ar[d]^{\mathscr F_P} \\ \X_{\Y} \ar@{^{(}->}[r]   & \mathbb H_P } $$ 
 	 
 For different elements $(P,\Y) \neq (P,\Y') \in \RBC_P$, the subsets $\X_\Y$ and $\X_{\Y'}$ of $\mathbb H_P$ are disjoint, because their projections in $\pi_0(\X)$ are the disjoint subsets $\pi_0(\X^{\Y})$ and $\pi_0(\X^{\Y'})$. Therefore we may identify $\X_P$ with the union of the $\X_\Y$'s inside $\mathbb H_P$. Under this identification, the map $\pi_P : \X \to \X_P$ is given by the composite map 
 $$ \X \xrightarrow{\omega _P} \mathscr Y_P \xrightarrow{\mathscr F_P} \mathbb H_P. $$ Since $\pi_P: \X \to \X_P$ is surjective, and since $\mathscr F_P \circ \omega_P : \X \to \mathbb H_P$ is equivariant with respect to $P(\RR) \to M_P(\RR)$, we see that $\X_P$ is an $M_P(\RR)$-stable subset of $\mathbb H_P$. We define the desired $M_P(\RR)$-action on $\X_P$ to be the one inherited from the $M_P(\RR)$-action on $\mathbb H_P$. Then $\pi_P$ is indeed equivariant with respect to $P(\RR) \to M_P(\RR)$.   
 
  \textbf{(2)} It suffices to observe that the $M_P(\RR)$-action on $\mathbb H_P$ factors through (\ref{eq:natural hom from M_P}), which is obvious.
  
  \textbf{(3)} Firstly, by \cite[4.7]{pink1989compactification}, the $G(\RR)$-action on $\X$ restricts to a transitive $P(\RR)$-action on $\X$. Since $\pi_P: \X \to \X_P$ is surjective, the $M_P(\RR)$-action on $\X_P$ is transitive. Secondly, the continuity of the $M_P(\RR)$-action on $\X_P$ follows from the continuity of the $P(\RR)$-action on $\X$, and the fact that the maps $\pi_P: \X \to \X_P$ and $P(\RR) \to M_P(\RR)$ are surjective submersions. Finally, the last statement in (3) follows from the surjectivity and $P(\QQ)$-equivariance of $\pi_P: \X \to \X_P$, where $P(\QQ)$ acts on $\X_P$ in the way described in \S \ref{subsubsec:taking pure quotient}.  
\end{proof}
 \begin{rem}\label{rem:pink}
 	In the above exposition, we started with the rational boundary components in the sense of \cite{pink1989compactification}, and used them to construct the $M_P(\RR)$-homogeneous space $\X_P$, the $P(\RR)$-equivariant map $\pi_P: \X \to \X_P$, and the $G(\QQ)$-set $\X^*$. This is the approach taken in \cite{pink1989compactification}. Alternatively, one could apply the classical (i.e.~non-adelic) formalism of rational boundary components in \cite{amrt} to each connected component of the Hermitian symmetric domain $\X$ in order to construct each connected component of $\X_P$ and each connected component of $\X^*$. One could then construct the whole $\X_P$ and $\X^*$ by taking suitable disjoint unions, and reconstruct the subsets $\X_{\Y} \subset \X_P$ as the $G_P(\RR)$-orbits in $\X_P$. This alternative approach is the point of view taken in \cite{pink1992ladic}. These two approaches are logically equivalent. Our usage of the notations $\X^*$ and $\X_P$ agrees with \cite[\S 3.6]{pink1992ladic} and \cite[\S 1.1]{morel2010book}.
 \end{rem}

\section{The group-theoretic setting}\label{global groups} In this section we fix the group-theoretic setting for our discussion of orthogonal Shimura varieties. 

\subsection{}
Let $(V,q)$ be a quadratic space over $\QQ$, of signature $(n,2)$. We always assume that $n \geq 3$. Let $d = \dim V = n+2$, and let $m = \floor{d/2}$. Let $G = \SO (V)$. Throughout the paper, we shall refer to ``\emph{the odd case}'' and ``\emph{the even case}'' according to the parity of $d$. 

Since $n \geq 3$, the maximal totally isotropic subspaces of $V_{\RR}$ are of dimension $2$. Throughout the paper we assume that the maximal totally isotropic subspaces of $V$ are also of dimension $2$. If $n \geq 5$, this assumption is automatic by Meyer's theorem (see \cite[\S IV.3.2 Cor.~2]{serrecourse}). We fix a flag \begin{align}
\label{flag} 0 \subset V_1 \subset V_2  \subset V_2^{\perp } \subset V_1^{\perp} \subset V,
\end{align}where $V_i$ is an $i$-dimensional totally isotropic $\QQ$-subspace of $V$.\index{$V_i,~i =1,2$} We set\index{$W_i,~i =1,2$} 
\begin{align*}
W_i   : = V_i^{\perp}/ V_i,
\end{align*}
for $i\in \set{1,2}$. Define \index{$P_1, P_2, P_{12}$}
\begin{align*} 
P_1 & : = \mathrm{Stab}_G (V_2) \subset G ,\\ P_2 & : = \mathrm{Stab}_G (V_1) \subset G, \\ P_{12} & : = P_1 \cap P_2 \subset G.
\end{align*} Then $P_{12}$ is a minimal parabolic subgroup of $G$, and $P_1$ and $P_2$ are the only proper parabolic subgroups of $G$ strictly containing $P_{12}$. If $S$ is a non-empty subset of $\set{1,2}$, we write $P_S$\index{$P_S$} for the one of $P_1$, $P_2$, and $P_{12}$ corresponding to $S$. 

\subsection{}
We fix once and for all a splitting of the flag (\ref{flag}). Then we obtain a Levi component $M_S$\index{$M_S$} of $P_S$ for each non-empty $S\subset \set{1,2}$. We have \index{$M_1, M_2, M_{12}$}
\begin{align} \nonumber
M_1 & \cong \GL(V_2) \times \SO ( W_2 )  , \\ \nonumber   M_2 & \cong \GL(V_1) \times \SO(W_1), \\ \label{eq:decomp M_{12}}  M_{12}  & \cong \GL(V_1) \times \GL(V_2 / V_1) \times \SO(W_2). 
\end{align}

In the sequel we call parabolic subgroups of $G$ containing $P_{12}$ \emph{standard}.\index[n]{standard parabolic subgroups of $G$} For each standard parabolic subgroup $P$, we denote by $M_P$ the Levi component of $P$ containing $M_{12}$, also called \emph{standard},\index[n]{standard Levi subgroups of $G$} and denote by $N_P$ the unipotent radical of $P$. Thus the standard proper parabolic subgroups are $P_1, P_2, P_{12}$, and for $P= P_S$ we have $M_P = M_S$. We also write $N_S$ for $N_{P_S}$. \index{$N_S$} 
\subsection{}\label{para:e_1}
We fix a basis $\set{e_1}$ of $V_1$ and a basis $\set{e_2}$ of $V_2/V_1$. By the fixed splitting of the flag (\ref{flag}), we can view $e_2$ as a vector in $V_2 \subset V$. Let $e_1' \in V / V_1^{\perp}  $ and $e_2' \in V_1^{\perp} / V_2^{\perp}$ be determined by the conditions $[e_i , e_i']  = 1, i =1,2. $\index{$e_1, e_2, e_1', e_2'$} We view $e_1',e_2'$ as vectors in $V$ as well. Under these choices we have identifications 
\begin{align*}
&\GL(V_i) \cong \GL_{i}, ~ i\in \set{1,2},& &\text{and} &  \GL(V_2/V_1) \cong \GL_1,&\end{align*} which we shall use freely in the sequel. In particular, the decomposition (\ref{eq:decomp M_{12}}) becomes $$ M_{12} \cong \Gm \times \Gm \times \SO(W_2).$$
We shall refer to the factor corresponding to $\GL(V_1) $ as \emph{the first $\Gm$}\index[n]{the first and the second $\GG_m$}, and refer to the factor corresponding to $\GL(V_2/V_1)$ as \emph{the second $\GG_m$}. 

\subsection{} Let $M$ be a standard proper Levi subgroup of $G$. We set \index{$M^{\GL}$} \index{$M^{\SO}$} \index{$M_h$} \index{$M_l$}
\begin{flalign*}
M^{\GL}  &  := \begin{cases}
\GL(V_2) ,   \\
\GL(V_1),  \\ 
\GL(V_1) \times \GL(V_2/V_1), 
\end{cases}  &  M^{\SO} &  := \begin{cases}
\SO(W_2), \\
\SO(W_1), \\
\SO(W_2),
\end{cases} \\ 
 M_h  & := \begin{cases}
\GL(V_2) ,  \\
\GL(V_1),   \\ 
\GL(V_1),  
\end{cases}  & 
M_l  &  := \begin{cases}
\SO(W_2),  \\
\SO(W_1), \\
\GL(V_2/V_1)\times \SO(W_2),
\end{cases} 
\end{flalign*}
where the three cases are when $M = M_1,M_2$, and $M_{12}$ respectively. (Here $h$ stands for ``hermitian'' and $l$ stands for ``linear''.) We have $$M = M^{\GL} \times M^{\SO} = M_h \times M_l. $$

\section{The orthogonal Shimura datum} \label{subsec:datum} 

\subsection{}\label{para:Hodge cochar}
Let $(V,q)$ and $G= \SO(V)$ be as in \S \ref{global groups}. In this paper we are concerned with the \emph{orthogonal Shimura datum} on $G$. In the following we recall its definition and some basic facts. More details can be found in \cite{MPspin}.

Consider the set $\X$ of oriented, negative definite, two-dimensional subspaces of $V_{\RR}$. Then $\X$ is a left homogeneous space under the natural action of $G(\RR)$. Moreover, $\X$ has two connected components, and the action of $\pi_0 (G(\RR)) = \ZZ/2\ZZ$ on $\pi_0(\X)$ is the  non-trivial one. 

Let $x \in \X$. For any $r e^{i\theta} \in \CC^{\times}$ (with $r \in \RR_{> 0}, \theta \in \RR$), we let $$\underline h(x) (re^{i\theta}) \in G(\RR)$$ be the element which acts on $V_{\RR} =x \oplus x^{\perp}$ as the rotation on $x$ by angle $-2\theta$ (according to the given orientation on $x$) and as the identity on $x^{\perp}$. The map $\underline h (x): \CC^\times \to G(\RR)$ comes from an $\RR$-algebraic group homomorphism $$h(x) : \mathbb S \To G_{\RR} .$$ Moreover, the association $x\mapsto h(x)$ is $G(\RR)$-equivariant and identifies $\X$ with a $G(\RR)$-conjugacy class of homomorphisms $\mathbb S \to G_{\RR}$. The tuple $(G, \X,h)$ is a pure Shimura datum, called the \emph{orthogonal Shimura datum}\index[n]{the orthogonal Shimura datum}. From now on we also denote this Shimura datum by $\mathbf O(V)$\index{$\mathbf O(V)$}. 
It is known that $\mathbf O(V)$ is of abelian type.\index[n]{abelian type} In fact, the pair $(\Gspin(V), \X)$ can be upgraded to  a Shimura datum of Hodge type\index[n]{$\mathrm{GSpin}$ Shimura datum}\index[n]{Hodge type}, and $\mathbf O(V)$ is the quotient of that by the central $\GG_m$ in $\Gspin(V)$. 

The Hodge cocharacter\index[n]{Hodge cocharacter} $\mu: \GG_m \to G$  of $\mathbf O(V)$ (well-defined up to $G(\overline \QQ)$-conjugacy) is given as follows. Choose an arbitrary hyperbolic basis $\mathbb B$ of $V_{\overline \QQ}$, and let $\iota_{\mathbb B}: \GG_m^m \hookrightarrow G_{\overline \QQ}$ be the embedding constructed in \S \ref{subsubsec:Borel pairs assoc to hd and nhd}. Let $\set{\epsilon_1^{\vee} , \cdots, \epsilon_m^{\vee}}$ be the standard basis of $X_*(\GG_m^m)$. Then $\mu$ is conjugate to $\iota_{\mathbb B} \circ \epsilon_1^{\vee}$. Moreover, it is possible to find a representative $\mu: \GG_m\to G$ defined over $\QQ$. In fact, we may assume that the first and the last vectors in $\mathbb B$ are respectively $e_1$ and $e_1'$. Then $\iota_{\mathbb B} \circ \epsilon_1^{\vee}$ is defined over $\QQ$. Consequently, the reflex field of $\mathbf O(V)$ is $\QQ$. 

Next we determine some explicit information about the rational boundary components of $\mathbf O(V)$. We follow the notation in \S \ref{subsec:boundary}. In the present case the set $\admpar$ consists of $G$ and the $G(\QQ)$-conjugates of $P_1$ and $P_2$.  
\begin{prop}\label{prop:rational bdry comp} The following statements hold. 
	\begin{enumerate}
		\item For each $P \in \admpar$, the set $\RBC_P(\mathbf O(V))$ is a singleton. 
		\item For $i =1,2$, we have $P_{i}^{\Pink} = M_{i,h} N_i$. In particular, $G_{P_i} = P_i^{\Pink}/N_i$ is naturally identified with $M_{i,h}$. 
		\item For $i =1,2$, under the identification $M_{i,h} \cong \GL_{3-i}$, the Shimura datum $(M_{i,h}, \X_{P_i})$ is identified with the Siegel Shimura datum\index[n]{Siegel Shimura datum} $(\GL_{3-i}, \mathcal H_{2(2-i)})$\index{$\mathcal H_{2g}$} (see \cite[2.7, 2.8]{pink1989compactification}).
		\item The action of the subgroup $M_{1,l}(\RR) \subset M_1(\RR)$ on $\X_{P_1}$ is trivial.
		\item The groups  $\pi_0(M_{2,h}(\RR))$, $\pi_0(M_{2,l}(\RR))$, and $\pi_0(G(\RR))$ are all isomorphic to $\ZZ/2\ZZ$. The map $$\pi_0(M_2(\RR)) \cong \pi_0(M_{2,h}(\RR)) \times \pi_0(M_{2,l}(\RR)) \To \pi_0(G(\RR))$$ induced by the inclusion $M_2(\RR) \hookrightarrow G(\RR)$ is given by 
		\begin{align*}
\ZZ/2\ZZ \times \ZZ/2\ZZ & \To \ZZ/2\ZZ  \\
(a,b) & \longmapsto a+b.
		\end{align*}
		The action of $M_2(\RR)$  on $\X_{P_2}$ as in Proposition \ref{prop:action} is given by the composite map $M_2(\RR) \to \pi_0(M_2(\RR)) \to \pi_0(G(\RR))$ followed by the unique non-trivial action of $\pi_0(G(\RR)) \cong \ZZ/2\ZZ$ on the two-element set $\X_{P_2} = \mathcal H_0$. 
	\end{enumerate}
\end{prop}
\begin{proof} Statements (1) (2) (3) follow from  \cite[Prop.~2.4.5]{hoermannbook}. To show (4), note that $M_{1,l}(\RR) \cong \SO(n-2,0)(\RR)$ is connected, and that it commutes with $G_{P_1} = M_{1,h}$. The statement then follows from Proposition \ref{prop:action} (2). We now show (5). We have 
	\begin{align*}
M_{2,h}(\RR) & \cong \RR^{\times},&   M_{2,l}(\RR) &\cong \SO(n-1,1)(\RR),&  G(\RR) &\cong \SO(n,2)(\RR). 
	\end{align*} The first two statements in (5) follow from the standard description of the connected components of special orthogonal groups; see for instance \cite[I.17]{knappbeyond}. The third statement follows from the fact that the map $\pi_{P_2}: \X \to \X_{P_2}$ is $P_2(\RR)$-equivariant and induces a bijection $\pi_0(\X) \isom \pi_0(\X_{P_2}) =\X_{P_2}$; see \S \ref{subsubsec:taking pure quotient} and Proposition \ref{prop:action}.\end{proof}

 \section{Shimura varieties}\label{subsec:SV}
From now on until the end of \S \ref{Section global}, we let $\mathbf{O}(V) = (G,\X, h)$ be the orthogonal Shimura datum fixed in \S \ref{subsec:datum}. Let $K$ be a neat\index[n]{neat} compact open subgroup of $G(\adele_f)$. (See \cite[0.6]{pink1989compactification} for the meaning of ``neat''.) As usual we define 
  $$ \Sh_K (\mathbf{O}(V))(\CC) :  = G(\QQ) \backslash \X \times G(\adele_f)/K.$$ This is the set of $\CC$-points of the \emph{canonical model}\index[n]{canonical model} $\Sh_K (\mathbf{O}(V)) $\index{$\Sh_K (\mathbf{O}(V))$}, which is a smooth quasi-projective variety of dimension $n =d-2$ over the reflex field $\QQ$. As $\mathbf O(V)$ is of abelian type, the existence of the canonical model follows from \cite{deligne1979varietes}. We write $\Sh_K$ for $\Sh_K (\mathbf{O}(V))$. 
  
  Let $K_1$ and $K_2$ be neat compact open subgroups of $G(\adele_f)$, and let $g$ be an element of $G(\adele_f)$ such that $K_1 \subset g K_2 g^{-1}$. We have a
  finite \'etale $\QQ$-morphism \index{$[\cdot g]_{K_1, K_2}$} $$[\cdot g]_{K_1, K_2}: \Sh_{K_1} \To \Sh_{K_2} $$ called a \emph{Hecke operator}\index[n]{Hecke operator}. On $\CC$-points, it is induced by \begin{align*}
  	 \X \times G(\adele_f) & \To \X \times G(\adele_f) \\  (x,k) &\longmapsto (x, kg). &
   \end{align*}
   When the context is clear we will simply write $[\cdot g]$ for $[\cdot g]_{K_1, K_2}$.
   
We recall the following facts proved in \cite[12.3]{pink1989compactification}. For any neat compact open subgroup $K \subset G(\adele_f)$, the Shimura variety $\Sh_K$ has the canonical Baily--Borel compactification \index[n]{Baily--Borel compactification} 
 $$ j : \Sh_{K} \To \overline{\Sh_K},$$
 where $\overline{\Sh_K}$\index{$\overline{\Sh_K}$} is a normal projective variety over $\QQ$, and $j$ is a dense open embedding defined over $\QQ$. At the level of $\CC$-points, we have $$\overline{\Sh_K} (\CC) = G(\QQ) \backslash \X^* \times G(\adele_f) /K, $$ where $\X^*$ is the $G(\QQ)$-set defined in (\ref{eq:X^*}), and $j$ is induced by $\omega_G: \X \isom \X_G \hookrightarrow \X^*$. For $K_1, K_2,$ and $g$ as in the last paragraph, the morphism $[\cdot g]: \Sh_{K_1} \to \Sh_{K_2}$ uniquely extends to a finite $\QQ$-morphism $\overline{[\cdot g]} = \overline{[\cdot g]}_{K_1,K_2}: \overline{\Sh_{K_1}} \to \overline {\Sh_{K_2}}$.
 \section{Automorphic \texorpdfstring{$\lambda$}{lambda}-adic sheaves}\label{subsec:automorphic sheaves}
  \subsection{}\label{subsubsec:V and lambda}
   Let $\mathbb V$\index{$\mathbb V$} be a finite-dimensional vector space over a number field $\mathbb E$\index{$\mathbb E$} equipped with a $G$-representation, i.e., an $\mathbb E$-algebraic group homomorphism $G_{\mathbb E} \to \GL (\mathbb V)$. Let $\lambda$ be a finite place of $\mathbb E$.\index{$\lambda$ (a finite place of $\mathbb E$)} Then by a well-known construction, for any neat compact open subgroup $K \subset G(\adele_f)$ there is an $\mathbb E_{\lambda}$-sheaf on $\Sh_{K}$ associated to $\mathbb V$, which we denote by $\mathcal F^{K} \mathbb V$.\index{$\mathcal F^{K} \mathbb V$} Moreover, for each Hecke operator $[\cdot g]: \Sh_{K_1} \to \Sh_{K_2}$ (with $K_1, K_2$ neat), there is a canonical isomorphism \index{$\mathcal F_{[\cdot g]}$}
   \begin{align}\label{eq:canonical coh corr}
\mathcal F_{[\cdot g] }: \mathcal F^{K_1} \mathbb V \isom [\cdot g]^* \mathcal F ^{K_2} \mathbb V.
   \end{align}
We refer the reader to \cite[\S 5.1]{pink1992ladic} and \cite[\S 1.5]{KSZ} for more details. 

Let $\ell$ be the rational prime below $\lambda$, and fix a $\QQ_{\ell}$-algebra embedding $\mathbb E_{\lambda} \hookrightarrow \overline \QQ_{\ell}$. Let $K$ be as above. We view the $\mathbb E_{\lambda}$-sheaf $\mathcal F^K \mathbb V$ as a $\overline \QQ_{\ell}$-sheaf and keep the same notation. We have the intersection complex \index{$\IC^{ K}  \mathbb V$}
  	$$ \IC ^{ K} \mathbb V: =  \bigg( j_{!*} \big((\mathcal F^{ K} \mathbb V) [n] \big ) \bigg)  [-n] \in D^b_c(\overline {\Sh_K}, \overline \QQ_{\ell}). $$ Here $j$ is the open embedding $\Sh _K \hookrightarrow \overline{ \Sh_K}$, and remember that $n = \dim \Sh_K$. 
  	
  \subsection{} We have analogues of the canonical isomorphisms (\ref{eq:canonical coh corr}) for the intersection complexes, which we now explain. Consider a Hecke operator $[\cdot g]: \Sh_{K_1} \to \Sh_{K_2}$ and its  extension\index{$\overline{[\cdot g]} $} $\overline{[\cdot g]} : \overline{\Sh_{K_1}} \to \overline{\Sh _{K_2}}.$
   To ease notation we write $g$ for $[\cdot g]$ and write $\bar g$ for $\overline{[\cdot g]}$. For $i =1,2$, we write $\mathcal F_i$ and $\IC_i$ for $\mathcal F^{K_i} \mathbb V$ and $\IC^{K_i}\mathbb V$ respectively, and write $j_i$ for the open embedding $\Sh_{K_i} \to \overline{\Sh_{K_i}}.$ 
   
    For any $\mathscr F \in D^b_c({\Sh_{K_2}}, \overline \QQ_{\ell})$, we have the commutative diagram 
  $$ \xymatrix{ \bar g^* Rj_{2,!} \mathscr F  \ar[d]\ar[r] & Rj_{1,!}  g^* \mathscr F \ar[d]  \\ \bar g^* Rj_{2,*} \mathscr F \ar[r] &  Rj_{1,*}  g^*  \mathscr F } $$
  where the horizontal maps are the base change maps, and the vertical maps are induced by the natural maps $Rj_{i, !} (\cdot )\to Rj _{i,*} (\cdot) , ~ i = 1,2$. Since $\bar g$ is finite (see \S \ref{subsec:SV}), $\bar g^*$ is exact with respect to the (middle-perversity) perverse t-structures. Therefore the above commutative diagram induces  
  a natural map
\begin{align}\label{eq:first direction between middle extensions}
\bar g^* j_{2, !*}  \mathscr F \To j_{1,!*} g^* \mathscr F .
\end{align}
Taking $\mathscr F$ to be $\mathcal F_2[n]$, we obtain a map 
  $$ \bar g^* j_{2, !*}  (\mathcal F_2[n]) \To j_{1,!*} g^*(\mathcal F_2[n]). $$ The composition of the above map followed by $j_{1, !*} (\mathcal F_{[\cdot g]}^{-1})$ gives a map
$$ \bar g^* j_{2, !*}  (\mathcal F_2[n]) \To  j_{1, !*} (\mathcal F_1 [n]) .$$ Shifting by $[-n]$ we obtain a map 
\begin{align}\label{map1}
\bar g ^* \IC_2 \To  \IC _1.
\end{align}

Similarly, using the base co-change maps (see \cite[XVIII]{SGA4-t3}) 
\begin{align*}
Rj _{1,!} g^! \mathscr F & \To\bar g^! Rj _{2,!}  \mathscr F, \\  Rj_{1,*} g^! \mathscr F ,& \To \bar g^! Rj_{2, *} \mathscr F ,
\end{align*}
we obtain a map
\begin{align}\label{eq:second direction of middle extension}
j_{1,!*} g^! \mathscr F  \To  \bar g^! j_{2, !*} \mathscr F 
\end{align} as a 
counterpart of (\ref{eq:first direction between middle extensions}). Note that because $g$ is finite \'etale (see \S \ref{subsec:SV}), we have $ g^! = g^*$. Again, taking $\mathscr F$ to be $ \mathcal F_2 [n]$ in (\ref{eq:second direction of middle extension}), pre-composing with $j_{1, !*} (\mathcal F_{[\cdot g]}) $, and shifting by $[-n] $, we obtain a map 
\begin{align}\label{map2}
\IC_1 \To \bar g^!  \IC _2. 
\end{align}
  
 Now for Hecke operators $ [\cdot g_1] _{K' , K_1}$ and $ [\cdot g_2] _{K', K_2}$, we obtain a canonical cohomological correspondence \index{$\mathscr H_{g_1, g_2 , K_1, K_2, K'}$}
  \begin{align}\label{cohomological Hecke correspondence}
\mathscr H_{g_1, g_2, K_1, K_2, K'} : ~ \bar g_1 ^* \IC^{K_1} \mathbb V \To  \bar g_2 ^! \IC^{K_2} \mathbb V
  \end{align} by composing (\ref{map1}) for $g= g_1$ with (\ref{map2}) for $ g= g_2$. 
\section{Intersection cohomology and Morel's formula}\label{subsec:intersection coh}
\subsection{} Keep the setting of \S \ref{subsubsec:V and lambda}. Let $K$ be a neat compact open subgroup of $G(\adele_f)$. Define  \index{$\mathbf{IH}^*(\overline{\Sh_K},\mathbb V)$}\index{$\mathbf{H}^*_c(\Sh_K ,\mathbb V)$}
 \begin{align*}
\icoh^* (\overline{\Sh_K} ,\mathbb V) : = \coh ^*_{\et} (\overline{\Sh_K} \otimes_{\QQ} \overline \QQ, ~ \IC^K  \mathbb V),\\
\coh^*_c (\Sh_K ,\mathbb V) : = \coh ^*_{\et, c} ({\Sh_K} \otimes_{\QQ} \overline \QQ, ~ \mathcal F^K  \mathbb V),
 \end{align*} which we view as graded $\overline \QQ_{\ell}$-vector spaces. 
  We denote by $\mathcal H (G(\adele_f)\sslash K)_{\QQ}$\index{$\mathcal H(G(\adele_f)\sslash K)_{\QQ}$} the Hecke algebra\index[n]{Hecke algebra (adelic)} of $\QQ$-valued smooth compactly supported $K$-bi-invariant distributions on $G(\adele_f)$. If we choose a Haar measure $dg^{\infty}$ on $G(\adele_f)$ that gives rational volumes to compact open subgroups, then each element of $\mathcal H (G(\adele_f)\sslash K)_{\QQ}$ can be uniquely written as $f^{\infty} dg^{\infty}$, where $f^{\infty}$ is a smooth compactly supported $K$-bi-invariant function $G(\adele_f) \to \QQ$. We have commuting actions of $\Gal(\overline \QQ/ \QQ)$ and $\mathcal H(G(\adele_f) \sslash K)_{\QQ}$ on $\icoh^* ( \overline{\Sh_K } , \mathbb V)$ and $\coh^*_c(\Sh_K, \mathbb V)$. Here the $\mathcal H(G(\adele_f) \sslash K)_{\QQ}$-action on $\icoh^*(\overline{\Sh_K}, \mathbb V)$ is characterized as follows. For any $g\in G(\adele_f)$, the element $$1_{KgK}\cdot \vol_{dg^{\infty}} (K)^{-1} dg^{\infty} \in \mathcal H(G(\adele_f) \sslash K)_{\QQ} $$  depends only on $g$ and not on the choice of $dg^{\infty}$. We require that this element acts on $\icoh^* (\overline{\Sh_K},\mathbb V)$ via the endomorphism induced by the cohomological correspondence 
  	$$\mathscr H_{g, 1, K, K , gKg^{-1} \cap K} : ~ \bar g ^* \IC^{K} \mathbb V \To \bar 1 ^{!} \IC^{K} \mathbb V, $$ where the notation is as in (\ref{cohomological Hecke correspondence}). By linearity, this determines the $\mathcal H(G(\adele_f) \sslash K)_{\QQ}$-action on $\icoh^*(\overline{\Sh_K}, \mathbb V)$. The $\mathcal H(G(\adele_f) \sslash K)_{\QQ}$-action on $\coh^*_c(\Sh_K, \mathbb V)$ is characterized similarly.
  	
  	If $p$ is a prime and $K^p$ is a compact open subgroup of $G(\adele_f^p)$, we denote by $\mathcal H (G(\adele_f^p)\sslash K^p)_{\QQ}$\index{$\mathcal H(G(\adele_f^p)\sslash K^p)_{\QQ}$} the Hecke algebra of $\QQ$-valued smooth compactly supported $K^p$-bi-invariant distributions on $G(\adele_f^p)$. Similarly as before, its elements can be written as $f^{p,\infty} dg^{p,\infty}$, where $f^{p,\infty}$ is a function $G(\adele_f^p) \to \QQ$ and $dg^{p,\infty}$ is a Haar measure on $G(\adele_f^p)$ giving rational volumes to compact open subgroups.  
 \begin{defn}\label{defn:hyperspecial pr}
 	Let $K$ be a compact open subgroup of $G(\adele_f)$ and let $p$ be a prime number. 
 	\begin{enumerate}
 		\item We say that $p$ is a \emph{hyperspecial prime}\index[n]{hyperspecial prime} for $K$, if we have $K= K^pK_p$, with $K_p$ a hyperspecial subgroup of $G(\QQ_p)$, and $K^p $ a compact open subgroup of $ G(\adele_f^p)$. 
 		\item Let $f^{\infty}dg^\infty \in \mathcal H(G(\adele_f)\sslash K)_{\QQ}$. We say that $p$ is an \emph{unramified prime} \index[n]{unramified prime (for an element of the Hecke algebra)} for $f^{\infty}dg^{\infty}$, if $p$ is hyperspecial for $K$, and we have $f^{\infty} dg^{\infty}= f^{p,\infty} dg^{p,\infty} 1_{K_p} dg_p$, where $f^{p,\infty} dg ^{p,\infty}$ is an element of $\mathcal H (G(\adele_f ^p) \sslash K^p)_{\QQ}$, $1_{K_p}: G(\QQ_p) \to \QQ$ is the characteristic function of $K_p$, and $dg_p$ is a Haar measure on $G(\QQ_p)$ giving rational volumes to compact open subgroups. 
 	\end{enumerate}
 \end{defn}

\subsection{}\label{subsubsec:setting for Morel's formula}
 Fix a neat compact open subgroup $K$ of $G(\adele_f)$, and fix $f^{\infty} dg^{\infty} \in \mathcal H(G(\adele_f)\sslash K)_{\QQ}$. Let $\Sigma_0$\index{$\Sigma_0$} be the finite set consisting of the prime $\ell$, the primes not hyperspecial for $K$, and the primes not unramified for $f^{\infty}dg^{\infty}$. For each prime $p\notin \Sigma_0$, we write $K = K^pK_p$ and $f^{\infty} dg^{\infty} = f^{p,\infty}dg^{p,\infty}1_{K_p} dg_p$ as in Definition \ref{defn:hyperspecial pr}. Without loss of generality, we may and shall assume that $\vol_{dg_p } (K_p) = 1$ by rescaling $f^{p,\infty}dg^{p,\infty}$. 
 
 Recall from \S \ref{subsubsec:V and lambda} that we have fixed an embedding $\mathbb E_{\lambda} \hookrightarrow \overline \QQ_{\ell}$. We now also fix a field embedding $\mathbb E \hookrightarrow \CC$. For any endomorphism $u$ of the graded $\overline \QQ_{\ell}$-vector space $\icoh^*(\overline{\Sh_K}, \mathbb V)$, we write $\Tr(u \mid  \icoh^*(\overline{\Sh_K}, \mathbb V) )$ for the alternating sum $$\sum_k (-1)^k  \Tr(u \mid  \icoh^k(\overline{\Sh_K}, \mathbb V) ) \in \overline \QQ_{\ell}. $$ (The sum is finite, since the terms are zero unless $0 \leq k\leq 2\dim \Sh_K$.) The same convention is applied to $\coh^*_c({\Sh_K}, \mathbb V)$.
 \begin{thm}[Morel's formula]\label{geometric assertion}
 In the setting of \S \ref{subsubsec:setting for Morel's formula}, there exists a finite set of prime numbers $\Sigma = \Sigma(\mathbf{O} (V), \mathbb V, \lambda, K,f^{\infty})$\index{$\Sigma(\mathbf{O} (V), \mathbb V, \lambda, K,f^{\infty})$} containing $\Sigma_0$ such that  the following statements hold for all primes $p \notin \Sigma$.
 	  \begin{enumerate}
 	  	\item The actions of $\Gal(\overline \QQ/ \QQ) $ on $\icoh^*(\overline{\Sh_K}, \mathbb V)$ and on $\coh^*_c (\Sh_K,\mathbb V)$ are both unramified at $p$.
 	  	\item Let $\Frob_p \in \Gal(\overline \QQ/ \QQ)$ be a geometric Frobenius at $p$.\index{$\Frob_p$} There exists a positive integer $a_0 = a_0 (\mathbf{O} (V), \mathbb V, \lambda, K, f^{\infty},p)$ such that for all integers $a \geq a_0$ we have
 	    \begin{multline}\label{eq:in geometric assertion}
\Tr(\Frob_p^a \times f^{\infty}dg^{\infty} \mid  \icoh^* (\overline{ \Sh _K}, \mathbb V))  \\ = \Tr(\Frob_p^a \times f^{\infty}  dg ^{\infty} \mid   \coh_c^* ( \Sh _K,  \mathbb V))  +  \sum_M \Tr _{M} (f^ {p, \infty} dg^{p,\infty} , K, a ).
 	  \end{multline} Here in the summation $M$ runs through the standard proper Levi subgroups of $G$, and $\Tr _{M} (f^ {p, \infty} dg^{p,\infty} , K, a )$\index{$\Tr_M(f^{p,\infty} dg^{p,\infty}, K, a)$} will be given in Definition \ref{Defn Tr12} below (which depends on the embedding $\mathbb E \hookrightarrow \CC$). The two sides of (\ref{eq:in geometric assertion}) are \emph{a priori} numbers in $\overline \QQ_{\ell}$ and $\CC$ respectively, but they actually both lie in $\mathbb E$.  
 	  \end{enumerate}
 \end{thm} 
 
 The proof of Theorem \ref{geometric assertion} will be given in \S \ref{sec:Proof of Morel's formula}.
 
 \begin{rem} We expect that Theorem \ref{geometric assertion} should in fact hold for $\Sigma=\Sigma_0$. The proof of this would require a robust theory of integral models of the Baily--Borel compactification and the toroidal compactifications of $\Sh_K$ at all hyperspecial primes, which is currently unavailable. See \S \ref{subsec:background} below for a more detailed discussion.
 \end{rem}

\chapter{Definition of the terms in Morel's formula}
\label{Section Geometirc}
In this chapter we define the terms $\Tr _{M} (f^ {p, \infty} dg^{p,\infty} , K, a )$ in Theorem \ref{geometric assertion}. We keep the setting in \S \ref{global groups}--\S \ref{subsec:intersection coh}. In particular, we fix $\mathbb E \hookrightarrow \CC$ as in \S \ref{subsubsec:setting for Morel's formula}.

\section{Truncated Lie algebra cohomology}
\begin{defn}\label{defn:varpi_i}
	For $i\in \set{1,2}$, let $\varpi_i: \GG_m \to M_{i,h}$\index{$\varpi_i$} be the weight cocharacter of the Shimura datum $(M_{i,h} , \X_{P_i})$, and let $t_i = \dim \X_{P_i} - \dim \X$.\index{$t_i$} (Here $\dim$ means the complex dimension.)
\end{defn}
\begin{lem}\label{lem:determining varpi and t}
The following statements hold.\begin{enumerate}
	\item The cocharacter $\varpi_1$ of $M_{1,h} = \GL(V_2) \cong \GL_2$ is given by $z \mapsto \diag(z,z)$.
	\item The cocharacter $\varpi_2$ of $M_{2,h} = \GL(V_1) \cong \GG_m$ is given by $z \mapsto z^2$.
	\item We have $t_1 = 3-d$, and $t_2 = 2-d$. 
\end{enumerate}
\end{lem}
\begin{proof}
	By Proposition \ref{prop:rational bdry comp}, we have $(M_{i,h} , \X_{P_i}) \cong (\GL_{2-i}, \mathcal H_{2(2-i)})$. The statements about $\varpi_1$ and $\varpi_2$ are clear. To determine $t_1$ and $t_2$, we use that $\dim \X = n = d-2$, $\dim \X_{P_1}=1$, and $\dim \X_{P_2}= 0$.
\end{proof}
\subsection{}
Let $S$ be a non-empty subset of $\set{1,2}$.  By Kostant's theorem\index[n]{Kostant's theorem} \cite{kostant} (cf.~ \cite[\S 11]{GHM} or \S \ref{subsec:Kostant's thm} below), the Lie algebra cohomology \index[n]{Lie algebra cohomology}
$$ \coh^k (\Lie (N_S)_{\CC} , \mathbb V \otimes_{\mathbb E} \CC)$$ is a finite-dimensional algebraic representation of $M_S(\CC)$, and is non-zero only for finitely many non-negative integers $k$. For $i \in S$, since we have $M_i= M_{i,h} \times M_{i,l}$ and since $\varpi_i$ is a central cocharacter of $M_{i,h}$ defined over $\QQ$, we know that $\varpi_i$ is a cocharacter of the split component $A_{M_i}$ of $M_i$, and \textit{a fortiori} a cocharacter of the split component $A_{M_S}$ of $M_S$.
\begin{defn}\label{defn:RGamma}
Let $S$ be a non-empty subset of $\set{1,2}$. We write\index{$\mathbf{H}^k (\Lie (N_S)_{\CC} , \mathbb V \otimes_{\mathbb E} \CC)_{>t_S}$} $$ \coh^k (\Lie (N_S)_{\CC} , \mathbb V \otimes_{\mathbb E} \CC)_{>t_S}$$ for the maximal $M_S(\CC)$-sub-representation of $ \coh^k (\Lie (N_S)_{\CC} , \mathbb V \otimes_{\mathbb E} \CC)$ on which $\varpi_i$ has weights strictly greater than $t_i$ for each $i \in S$. (Here we say that a $\GG_m$-representation has weights greater than a number $t$ if all the  appearing characters $z \mapsto z^k$ satisfy $k>t$.) We define the virtual $M_{S}(\CC)$-representation: \index{$R\Gamma (\Lie N_S, \mathbb V) _{>t_S}$}
$$ R\Gamma (\Lie N_S , \mathbb V) _{>t_S} : = \sum_{k\geq 0 } (-1)^k \coh^k (\Lie (N_S)_{\CC}, \mathbb V\otimes_{\mathbb E} \CC)_{>t_S}.$$ When $P = P_S$ is fixed in the context, we also replace the symbol ``$>t_S$'' by ``$>t_P$''.
\end{defn}

\section{The Kostant--Weyl term \texorpdfstring{$L_{M}$}{LM}}
\label{subsec:KW 12}
In this section, let $M$ be a standard proper Levi subgroup of $G$, i.e., $M \in \set{M_1, M_2, M_{12}}$. 
\begin{defn}\label{defn: P(M)}  Let $\mathcal P(M)$\index{$\mathcal P(M)$} be the set of pairs $(P,g)$, where $P$ is a standard proper parabolic subgroup of $G$, and $g$ is an element of $G(\QQ)$, satisfying the following conditions. 
	\begin{enumerate} 
		\item We have $M_h = M_{P,h}$, and $M_l$ is a Levi subgroup of $M_{P,l}$. In particular, $M \subset M_P$. 
		\item The element $g$ centralizes $M_h \subset G$, and normalizes $M_l \subset G$. In particular, $g$ normalizes $M \subset G$.   
	\end{enumerate}
Let $\sim$ be the equivalence relation on $\mathcal P(M)$\index{$\sim$ on $\mathcal P(M)$} such that $(P,g)\sim (P',g') $ if and only if $P= P'$ and $g\in M_{P}(\QQ)g'M(\QQ)$. (Here $M_P$ is the standard Levi component of $P$, which may not be the same as $M$.) For any standard proper parabolic subgroup $Q$ of $G$, let\index{$\mathcal P (M, Q)$} $$ \mathcal P (M, Q): = \set{(P,g) \in \mathcal P(M) \mid  P = Q } \subset \mathcal P(M).$$ 
\end{defn}
\begin{defn}\label{defn:refined L_M}
Set  \index{$\mathtt m_M$} $\mathtt m_M$ to be $1$ if $M=M_1$, and $2$ if $M \in \set{M_2, M_{12}}$. For $\gamma \in M(\RR)$ and $(P,g) \in \mathcal P(M)$, define the complex number \index{$L_{M,P,g} (\gamma)$} 	\begin{multline*}
L_{M,P,g} (\gamma):= \mathtt m_M   (-1)^{\dim A_{M}/ A_{M_P}}  (n^{M_P} _{M})^{-1}  \abs{D^{M_P} _{M} (g \gamma g^{-1})}^{1/2}_{\RR}   \\    \cdot  \delta_{P(\RR)} (g \gamma g^{-1})^{1/2}  \Tr(g\gamma g^{-1} \mid  R\Gamma (\Lie N_P , \mathbb V) _{>t_P} ).
\end{multline*} 
Here the terms $n^{M_P} _{M}$, $D^{M_P} _{M}(\cdot)$, $\delta_{P(\RR)}(\cdot)$ are all defined in \S \ref{subsec:general defns}, and $R\Gamma(\Lie N_P, \mathbb V)_{>t_P}$ is as in Definition \ref{defn:RGamma}. 
\end{defn}

It is easy to see that $L_{M,P,g}(\gamma)$ depends on $(P,g)$ only via the $\sim$-equivalence class of $(P,g)$. We use this fact in the next definition. 

	\begin{defn}\label{Defn L_M 12 odd}  For $\gamma \in M(\RR)$, define the \emph{Kostant--Weyl term}\index{$L_M(\gamma)$}\index[n]{Kostant--Weyl term}
	\begin{align}\label{eq:KW}
	L_{M} (\gamma):= \sum_{ (P,g) \in \mathcal P (M) /{\sim}} \abs{\mathcal P(M,P)/{\sim}}^{-1} L_{M,P,g}(\gamma) \in \CC. 
	\end{align} 
\end{defn}

\begin{prop}\label{prop:defn of L_M for M_1 M_2}
	Let $i =1$ or $2$. Then every element of $\mathcal P(M_i)$ is $\sim$-equivalent to $(P_i,1)$. In particular, for $\gamma \in M_i(\RR)$ we have 
	$$L_{M_i}(\gamma) = \mathtt m_{M_i} \delta_{P_i(\RR)} (\gamma)^{1/2} \Tr(\gamma \mid R\Gamma(\Lie N_i, \mathbb V)_{>t_i}).$$
 \end{prop}
\begin{proof} It is clear that $(P_i,1) \in \mathcal P(M_i)$. Let $(P,g) \in \mathcal P(M_i)$. By condition (1) in Definition \ref{defn: P(M)}, we have $P = P_i$. Since $M_{i,h}$ contains $A_{M_i}$, the centralizer of $M_{i,h}$ in $G$ is contained in $M_i$. Hence by condition (2) in Definition \ref{defn: P(M)}, we have $g \in M_i(\QQ)$. It then follows that $(P,g) \sim (P_i, 1)$. 
\end{proof}

\subsection{}

Next we give an explicit description of the set $\mathcal P(M_{12})/{\sim}$. Recall from \S \ref{global groups} that we have identified $V$ with the orthogonal direct sum of $\mathrm{span}_{\QQ}\set{e_1, e_1'}$ and $W_1$, and identified $W_1$ with the orthogonal direct sum of $\mathrm{span}_{\QQ} \set{e_2, e_2'}$ and $W_2$. Also recall that $M_{2,l} = \SO(W_1) \subset M_2 = \GL(V_1) \times \SO(W_1)$. 
 \begin{defn}\label{defn:n12}
	Let $M_{2,l}(\QQ)^{\sharp}$\index{$M_{2,l}(\QQ)^{\sharp}$} be the set consisting of $g \in M_{2,l}(\QQ)$ satisfying the following conditions:
	\begin{enumerate}
		\item $g(e_2) = e_2',~ g(e_2') = e_2$. 
		\item $g$ stabilizes $W_2$, and $g|_{W_2}$ is an element of $\mathrm{O}(W_2)(\QQ)$ with determinant $-1$. 
	\end{enumerate}
\end{defn}
\begin{rem}
	Since $\dim W_2 = n-2 \geq 1$, the group $\mathrm{O}(W_2)(\QQ)$ indeed contains elements with determinant $-1$. It is then  clear that $M_{2,l}(\QQ)^{\sharp} \neq \emptyset$.  
\end{rem}
\begin{prop}\label{prop:to simplify LM12} The set $\mathcal P(M_{12}, P_1)$ is empty. Every element of  $\mathcal P(M_{12}, P_2)$ is $\sim$-equivalent to $(P_{2}, 1)$. The set $\mathcal P(M_{12},P_{12})$ is the union of exactly two $\sim$-equivalence classes, and they are represented by $(P_{12}, 1)$ and $(P_{12}, g_0)$, where $g_0$ is any element of $M_{2,l}(\QQ)^{\sharp}$. 
\end{prop}	
\begin{proof}
	Since $M_{12,l}$ is not contained in $M_{1, l}$, we have $\mathcal P(M_{12},P_1) = \emptyset$. 
Since $M_{12,h} = \GL(V_1) = A_{M_2}$, by condition (2) in Definition \ref{defn: P(M)} we know that any $(P,g) \in \mathcal P(M_{12})$ must satisfy $g \in \Nor_{M_2} (M_{12,l})(\QQ)$. Conversely, for any $g \in \Nor_{M_2} (M_{12,l})(\QQ)$, we have $(P_2,g),(P_{12},g) \in \mathcal P(M_{12})$. The statement about $\mathcal P(M_{12},P_2) $ immediately follows. 

To show the last statement about $\mathcal P(M_{12},P_{12})$, we know from the above discussion that we have a surjection 
\begin{align*}
\Nor_{M_2} (M_{12,l})(\QQ) & \To \mathcal P(M_{12}, P_{12})/{\sim }\\ 
g & \longmapsto (P_{12}, g). 
\end{align*} This surjection restricts to a surjection 
\begin{align*}
\Nor_{M_{2,l}} (M_{12,l})(\QQ) \To \mathcal P(M_{12}, P_{12})/{\sim} ,
\end{align*}
which induces a bijection (see Definition \ref{defn:n^G_M} for the notation)
  $$ \cW^{M_{2,l}}_{M_{12,l}} \isom \mathcal P(M_{12},P_{12}) / {\sim}. $$
  Now note that $\GL(V_2/V_1) \cong \GG_m$ is the split component of $M_{12,l}$. As in Remark \ref{rem:Weyl group of Levi}, we have an injective homomorphism 
  \begin{align*}
 \cW^{M_{2,l}}_{M_{12,l}} & \To \Aut(\GL(V_2/V_1)) \cong \ZZ/2\ZZ \\
 g & \longmapsto \Int(g)|_{\GL(V_2/V_1)}.
  \end{align*}The desired statement follows from the fact that for all $g_0 \in M_{2,l}(\QQ)^{\sharp}$, we have $g_0 \in \Nor_{M_{2,l}}(M_{12,l})(\QQ)$, and $\Int(g_0)|_{\GL(V_2/V_1)}$ is non-trivial. 
\end{proof}

\section{Definitions related to Kottwitz's fixed point formula}\label{subsec: defn in Kott}
\subsection{}
Let $M_h$ be the reductive group $\GL_i$ over $\QQ$, where $i=1$ or $2$. We equip $M_h$ with the Siegel Shimura datum $\mathcal H_{2(2-i)}$ (see \cite[2.7, 2.8]{pink1989compactification}). We define some group-theoretic terms that appear in Kottwitz's fixed point formula for the Shimura varieties associated to $(M_h, \mathcal H_{2(2-i)})$. The main reference is \cite[Part I]{kottwitzannarbor}; see also \cite[\S 1.6]{morel2010book}. We fix a prime $p$, and an integer $a \geq 1$.

	Define a cocharacter $\mu$ of $M_h$ as follows. When $M_h = \GG_m$, let $\mu$ be the identity cocharacter. When $M_h = \GL_2$, let  $\mu$ be
	$z \mapsto \diag(z,1)$. Thus $\mu$ is a Hodge cocharacter\index[n]{Hodge cocharacter} for the Shimura datum $(M_h, \mathcal H_{2(2-i)})$.

The following definition is equivalent to the standard definition as in \cite[\S 19]{kottwitz1992points} or \cite[\S 1.6]{morel2010book}; it appears simpler since in the group $M_h$ stable conjugacy is the same as conjugacy.  
\begin{defn}\label{defn:Kott triple}
	A \emph{Kottwitz triple}\index[n]{Kottwitz triple} in $M_h$ (of level $p^a$, for the Shimura datum $(M_h,\mathcal H_{2(2-i)})$) is a triple $(\gamma_0,\gamma,\delta)$, with $\gamma_0 \in M_h(\QQ)$, $\gamma \in M_h(\adele_f^p)$, $\delta \in M_h(\QQ_{p^a})$, satisfying the following conditions:
	\begin{enumerate}
\item The element $\gamma_0$ is semi-simple and $\RR$-elliptic (see Definition \ref{defn:R-ell}). 
\item The element $\gamma$ is conjugate to $\gamma_0$ in $M_h(\adele_f^p).$
\item The element $\N\delta: = \delta \sigma(\delta) \cdots \sigma^{a-1} (\delta)  \in M_h (\QQ_{p^a})$ is conjugate to $\gamma_0$ in $M_h(\QQ_{p^a})$. 
\item If $M_h = \Gm$, then the $p$-adic valuation of $\delta \in \QQ_{p^a}^{\times}$ is $-1$. If $M_h = \GL_2$, then the $p$-adic valuation of the determinant of $\delta \in \GL_2(\QQ_{p^a})$ is $-1$. 
	\end{enumerate}
	Two Kottwitz triples $(\gamma_0,\gamma,\delta)$ and $(\gamma_0' ,\gamma',\delta')$ are said to be \emph{equivalent}, if $\gamma_0$ is conjugate to $\gamma_0'$ in $M_h(\QQ)$, and $\delta$ is $\sigma$-conjugate to $\delta'$ inside $M_h(\QQ_{p^a})$. In the sequel, it is understood that whenever Kottwitz triples appear in a summation, they are taken up to equivalence.
\end{defn}
\begin{rem}
	Abstractly, condition (4) in Definition \ref{defn:Kott triple} says that the image of $\delta$ in $\pi_1(M_h) _{\Gamma_p}$ under the Kottwitz map\index[n]{Kottwitz map} is equal to that of $-\mu$.
\end{rem}
\subsection{}
Let $(\gamma_0,\gamma,\delta)$ be a Kottwitz triple. Let $I_0 = M _{h,\gamma_0}$ be the centralizer (which is connected) of $\gamma_0$ in $M_h$. Define $\mathfrak K (I_0 /\QQ)$\index{$\mathfrak K (I_0 /\QQ)$} to be the finite abelian group consisting of those elements of $\pi_0 ([Z(\widehat{ I_0}) / Z(\widehat {M_h})]^{\Gamma_{\QQ}})$ whose images in $\coh^1(\Gamma_{\QQ}, Z(\widehat{M_h}))$ are locally trivial; see \cite[\S 4.6]{kottwitzelliptic}. In \cite[\S 2]{kottwitzannarbor} Kottwitz defines an invariant \index{$\alpha (\gamma_0,\gamma,\delta)$} $$\alpha (\gamma_0,\gamma,\delta) \in \mathfrak K(I_0/\QQ) ^D $$ of the triple $(\gamma_0,\gamma,\delta)$. Here $\mathfrak K(I_0/\QQ) ^D$ is the Pontryagin dual of $\mathfrak K(I_0/\QQ)$. 

\begin{lem}\label{simplification 1 M1}
For $M_h = \GL_1$ or $\GL_2$, we always have $\mathfrak K(I_0/\QQ)=0$.
\end{lem}
\begin{proof} If $I_0 = M_h$ then obviously $\mathfrak K(I_0/\QQ)=0$. Thus we may assume that $M_h = \GL_2$ and that $\gamma_0$ in non-central. 
	Then $I_0$ is a maximal torus $T$ in $\GL_2$ defined over $\QQ$. In this case $Z(\widehat{I_0} )  = \widehat{I_0} = \widehat T$. Since the Galois action on $Z(\widehat {\GL_2})$ is trivial, by Chebotarev's density theorem the only locally trivial element of $\coh^1(\Gamma_{\QQ}, Z(\widehat {\GL_2}))$ is the trivial element. In view of the exact sequence 
	$$ \pi_0 (Z(\widehat {\GL_2}) ^{\Gamma_{\QQ}}) \to \pi _0 ( \widehat T^{\Gamma_{\QQ}}) \to \pi_0 ([\widehat T / Z(\widehat {\GL_2} )]^{\Gamma_{\QQ}}) \to \coh^1(\Gamma_{\QQ}, Z(\widehat {\GL_2})),$$ it suffices to show that $$\widehat T^{\Gamma_{\QQ}} \subset Z(\widehat {\GL_2}).$$ Since $\gamma_0$ is $\RR$-elliptic, $T_{\RR}$ is an elliptic maximal torus in $\GL_{2,\RR}$. Hence there exists an identification $\widehat T \cong \CC^{\times } \times \CC^{\times}$ such that the non-trivial element of $\Gamma_{\infty}$ acts on $\widehat T$ by switching the two coordinates. It follows that $\widehat T ^{\Gamma _{\infty}} \subset Z(\widehat {\GL_2})$, and \textit{a fortiori} $\widehat T^{\Gamma_{\QQ}} \subset Z(\widehat {\GL_2}).$  
\end{proof}

\subsection{}\label{subsubsec:Haar}
Let $(\gamma_0,\gamma,\delta)$ be a Kottwitz triple. By Lemma \ref{simplification 1 M1}, the invariant $\alpha(\gamma_0,\gamma,\delta)$ automatically vanishes. Hence as in \cite[\S 3]{kottwitzannarbor}, there is an inner form $I$
 of $I_0$ over $\QQ$ satisfying the following conditions. \begin{itemize}
 	\item The group $I_{\RR}$ is anisotropic modulo center.
 	\item For any finite place $v$ of $\QQ$ not equal to $p$, $I_{\QQ_v}$ is the trivial inner form of $I_{0,\QQ_v}$. 
 	\item The inner form $I_{\QQ_p}$ of $I_{0,\QQ_p}$ is isomorphic (as an inner from) to the $\sigma$-centralizer $(M_h)_{\delta\sigma}$\index{$(M_h)_{\delta\sigma}$} of $\delta$ in $M_h$ (which is denoted by $I(p)$ in \textit{loc.~cit.}).  
 \end{itemize}
We refer the reader to \textit{loc.~cit.~}for more details.
 
 Fix Haar measures on $I(\QQ_p)$, $I(\adele_f^p)$, and $I(\RR)$ such that the product Haar measure on $I(\adele)$ is the Tamagawa measure. Fix a Haar measure on $M_h(\QQ_{p^a})$ such that $M_h(\ZZ_{p^a})$ has volume $1$. Fix Haar measures on $ M_h(\RR)$ and $M_h(\adele_f^p)$  arbitrarily. 
 
 \begin{defn}\label{defn:c of a Kott trip} In the setting of \S \ref{subsubsec:Haar}, we define \index{$c(\gamma_0 , \gamma,\delta)$} \index{$c_1(\gamma_0,\gamma,\delta)$} \index{$c_2(\gamma_0, \gamma,\delta)$}
$$c (\gamma_0 , \gamma,\delta) : = c_1(\gamma_0,\gamma,\delta)c_2(\gamma_0, \gamma,\delta), $$ where  
 \begin{align*}
c_1 (\gamma_0,\gamma,\delta) & = \vol (I(\QQ) \backslash I(\adele_f)) =\tau(I) \vol(A_{M_h} (\RR) ^0 \backslash I(\RR) )^{-1} , \\ c_2 (\gamma_0,\gamma,\delta) &= \abs{\ker (\ker^1(\QQ, I_0) \to \ker^1(\QQ, M_h))}.
 \end{align*} Here $\tau(I)$ is the Tamagawa number of $I$.
 \end{defn}
  \begin{defn}\label{defn:O and TO} In the setting of \S \ref{subsubsec:Haar}, we define the \emph{orbital integral along $\gamma$}\index[n]{orbital integral} to be the functional \index{$O_{\gamma}$} 
 	\begin{align*}
 	O_{\gamma} : C^{\infty}_c(M_h(\adele_f^p)) & \To \CC \\
 	f & \longmapsto O_{\gamma}(f) = \int_{M_{h,\gamma}(\adele_f^p)\backslash M_h(\adele_f^p)} f(g^{-1} \gamma g), 
 	\end{align*}
 	with respect to the fixed Haar measure on $M_h(\adele_f^p)$ and the Haar measure on $M_{h,\gamma} (\adele_f^p)$ transferred from $I(\adele_f^p)$. We define the \emph{twisted orbital integral along $\delta$}\index[n]{twisted orbital integral} to be the functional \index{$TO_{\delta}$}
 		\begin{align*}
 	TO_{\delta} : C^{\infty}_c(M_h(\QQ_{p^a})) & \To \CC \\
 	f & \longmapsto TO_{\delta}(f) = \int_{(M_h) _{\delta\sigma} (\QQ_p) \backslash M_h(\QQ_{p^a})} f(g^{-1} \delta \sigma(g) ) , 
 	\end{align*} 
with respect to the fixed Haar measure on $M_h(\QQ_{p^a})$ and the Haar measure on $(M_h) _{\delta\sigma} (\QQ_p)$ transferred from $I(\QQ_p)$. For more details see \cite[\S 3]{kottwitzannarbor}.
 \end{defn}
 
 \begin{defn}\label{defn:phi_a} Let $\phi_{a}^{M_h} : M_h(\QQ_{p^a})  \to \QQ $ be the characteristic function of $ M_h (\ZZ_{p^a}) \mu(p)^{-1}  M_h (\ZZ_{p^a}).$ \index{$\phi_{a}^{M_h}$}
 \end{defn}
\section{Definition of \texorpdfstring{$\Tr_{M}$}{TrM}} 
In this section, let $P$ be a standard parabolic subgroup of $G$, and let $M= M_P$ be the standard Levi component of $P$. We define the term $ \Tr_M(f^{p,\infty} d g^{p,\infty}, K, a)$ in (\ref{eq:in geometric assertion}). 
\begin{defn}\label{defn:simR}
	For $\gamma_0 \in M_h(\RR)$ and $\gamma_L \in M_l(\RR)$, we write $\gamma_0 \sim_{\RR} \gamma_L$\index{$\sim_{\RR}$}, if one of the following conditions holds.
	\begin{enumerate}
		\item We have $M_h \cong \GL_2$. 
		\item We have $M_h \cong \GG_m$, $\gamma_0 \in M_h(\RR)^0$, and $\gamma_L \in M_l(\RR)^0$.
		\item  We have $M_h \cong \GG_m$, $\gamma_0 \notin M_h(\RR)^0$, and $\gamma_L \notin M_l(\RR)^0$.	\end{enumerate}  
\end{defn}
\begin{rem}\label{rem:simR}
	When $M= M_1$, we have $M_h = \GL_2$, and so the condition $\gamma_0 \sim_{\RR} \gamma_L$ is by definition automatic. When $M= M_{12} $ or $M_2$, we have $\pi_0 (M_h(\RR) )  \cong \pi_0 (M_l (\RR) ) \cong \ZZ/2\ZZ$. Thus the condition $\gamma_0 \sim_{\RR} \gamma_L$ depends only on the $M_h(\RR)$-conjugacy class of $\gamma_0$ and the $M_l(\RR)$-conjugacy class of $\gamma_L$.
\end{rem}
\begin{defn}\label{Defn Tr12} Let $K$ be a compact open subgroup of $G(\adele_f)$. Let $p$ be a hyperspecial prime for $K$, and let $K_p, K^p$ be as in Definition \ref{defn:hyperspecial pr}. Let $f^{p,\infty}  dg ^{p,\infty}\in  \mathcal H (G(\adele_f ^p) \sslash K^p)_{\QQ}$, and let $a\in \ZZ_{\geq 1}$. We define the complex number \index{$\Tr_M(f^{p,\infty} dg^{p,\infty}, K, a)$}
	\begin{multline}\label{eq:Tr_M}
 \Tr_M(f^{p,\infty} d g^{p,\infty}, K, a) : = \sum_{\gamma_L} \iota^{M_l}  (\gamma_L) ^{-1} \chi ((M_{l,\gamma_L }) ^0  )  \sum_{(\gamma_0 , \gamma,\delta)} c(\gamma_0, \gamma,\delta) \\   \cdot \delta_{P(\QQ_p)}  (\gamma_0) ^{1/2}    O_{\gamma_L \gamma} (f^{p,\infty} _M) O_{\gamma_L} (1_{M_l (\ZZ_p)}) TO_{\delta} (\phi_a^{M_h})L_M(\gamma_L \gamma_0), 
	\end{multline}
	where $\gamma_L$ runs through the semi-simple conjugacy classes in $M_l(\QQ)$ that are $\RR$-elliptic (see Definition \ref{defn:R-ell}; if no such $\gamma_L$ exists, then the sum is empty), and $(\gamma_0, \gamma,\delta )$ runs through the equivalence classes of Kottwitz triples in $M_h$ of level $p^a$ (see Definition \ref{defn:Kott triple}) such that $\gamma_0 \sim_{\RR} \gamma_L$ (see Definition \ref{defn:simR} and Remark \ref{rem:simR}). The other terms are defined as follows: \begin{enumerate}
\item We write $\iota^{M_l} (\gamma_L)$ for $\abs{M_{l, \gamma_L} (\QQ ) / (M_{l,\gamma_L}) ^0 (\QQ)}$.\index{$\iota^{M_l} (\cdot)$}

\item We write $\chi((M_{l,\gamma_L }) ^0 )$ for the \emph{Euler characteristic}\index[n]{Euler characteristic (of a reductive group)} of the reductive group $(M_{l,\gamma_L }) ^0 $ over $\QQ$, as defined in \cite[\S 7.10]{GKM}.\index{$\chi (\cdot)$}
\item The term $c (\gamma_0 , \gamma,\delta)$ is as in Definition \ref{defn:c of a Kott trip}.
\item We let $f^{p,\infty}_M \in C^{\infty}_c(M(\adele_f^p))$\index{$f^{p,\infty}_M$} be the \emph{constant term}\index[n]{constant term} of $f^{p,\infty}$ as defined in \cite[\S 7.13]{GKM}. This function depends on auxiliary choices, but its orbital integrals are well defined once all the relevant Haar measures are fixed. 
\item We have a canonical identification $$C^{\infty}_c(M(\adele_f^p)) \cong  C^{\infty}_c(M_h(\adele_f^p)) \otimes_{\CC} C^{\infty}_c(M_l(\adele_f^p)). $$ In view of this, we define the functional $O_{\gamma_L \gamma} : C^{\infty}_c(M(\adele_f^p)) \to \CC$\index{$O_{\gamma_L \gamma}$} to be the tensor product of the functional $O_{\gamma} : C^{\infty}_c(M_h(\adele_f^p))  \to \CC$ in Definition \ref{defn:O and TO} and the functional \index{$O_{\gamma_L}$ (away from $p$)} 
\begin{align} \label{eq:O gamma L}
O_{\gamma_L} :  C^{\infty}_c(M_l(\adele_f^p)) &  \To \CC \\  \nonumber
f & \longmapsto O_{\gamma_L}(f) = \int_{M_{l,\gamma_L}(\adele_f^p) \backslash M_l(\adele_f^p)} f(g^{-1} \gamma_L g) dg, 
\end{align}
where the relevant Haar measures are to be specified in Remark \ref{rem:Haar measures} below. 
\item We let $M_l(\ZZ_p)$\index{$M_l(\ZZ_p)$} be the hyperspecial subgroup of $M_l(\QQ_p)$ given by 
\begin{align}\label{eq:M_l(Z_p)}
M_l(\ZZ_p) : = 
[K_p \cap (M_l(\QQ_p) N_P(\QQ_p) ) ] / (K_p \cap N_P (\QQ_p)) .
\end{align} See Remark \ref{rem:why hyp} below for more explanations. 
\item We define\index{$O_{\gamma_L} (1_{M_l(\ZZ_p)})$} \begin{align}\label{eq:O gamma L at p}
O_{\gamma_L} (1_{M_l(\ZZ_p)}) : =  \int_{M_{l,\gamma_L}(\QQ_p) \backslash M_l(\QQ_p)} 1_{M_l(\ZZ_p)} (g^{-1} \gamma_L g) dg,
\end{align} 
where the relevant Haar measures are to be specified in Remark \ref{rem:Haar measures} below. 
\item The term $TO_{\delta}(\phi_a^{M_h})$ is as in Definitions \ref{defn:O and TO} and \ref{defn:phi_a}.
\item The term $L_M(\cdot)$ is as in Definition \ref{Defn L_M 12 odd}.
	\end{enumerate}  \end{defn}
\begin{rem}\label{rem:Haar measures} We make precise the choices of various Haar measures in Definition \ref{Defn Tr12}. We choose an arbitrary Haar measure on $M_l (\adele_f^p)$, and choose arbitrary Haar measures on $M_{l,\gamma_L} (\adele_f^p)$ and $M_{l,\gamma_L}(\QQ_p)$ for each $\gamma_L$. We then define the Haar measure on $M (\adele_f^p) = M_h (\adele_f^p) \times M_l (\adele_f^p)$ to be the product of the Haar measure on $M_l(\adele_f^p)$ chosen above and the Haar measure on $M_h(\adele_f^p)$ chosen in \S \ref{subsubsec:Haar}. We then specify various normalizations: 
	\begin{enumerate}
		\item Use the Haar measure on $M(\adele_f^p)$ as above and the Haar measure $d g^{p,\infty}$ on $G(\adele_f^p)$ to define the constant term $f^{p,\infty} _M$. 
		\item Use the Haar measures on $M_l(\adele_f^p)$ and $M_{l,\gamma_L} (\adele_f^p)$ chosen above to define (\ref{eq:O gamma L}).
		\item Use the Haar measure on $M_l(\QQ_p)$ giving volume $1$ to $M_l(\ZZ_p)$, and the Haar measure on $M_{l,\gamma_L} (\QQ_p)$ chosen above, to define (\ref{eq:O gamma L at p}).
		\item Use the Haar measures on $M _{l,\gamma_L} (\adele_f^p)$ and $M _{l,\gamma_L} (\QQ_p)$ chosen above to define the product measure on $(M_{l,\gamma_L} )^0 (\adele_f)$, and use the latter to define $\chi ((M_{l,\gamma_L} )^0)$ as in \cite[\S 7.10]{GKM}.
	\end{enumerate}
\end{rem}
\begin{rem}\label{rem:why hyp}
	We explain why $M_l(\ZZ_p)$ defined by (\ref{eq:M_l(Z_p)}) is a hyperspecial subgroup of $M_l(\QQ_p)$ by collecting standard facts about reductive group schemes from \cite[XXVI]{SGA3III}. Since $K_p$ is a hyperspecial subgroup of $G(\QQ_p)$, there is a reductive group scheme $\mathcal G$ over $\ZZ_p$ with generic fiber $G_{\QQ_p}$ such that $K_p = \mathcal G(\ZZ_p) \subset G(\QQ_p)$. By \cite[XXVI, Cor.~3.5]{SGA3III}, the parabolic subgroup $P_{\QQ_p}$ of $G_{\QQ_p}$ extends to a unique parabolic subgroup $\mathcal P$ of $\mathcal G$. Since parabolic subgroups are closed (see \cite[XXVI, Prop.~1.2]{SGA3III}), we have $\mathcal P(\ZZ_p)=P(\QQ_p)\cap K_p$. Now the reductive quotient $\mathcal M$ of $\mathcal P$ (see \cite[XXVI, Cor.~1.5, Prop.~1.6]{SGA3III}) is a reductive group scheme over $\ZZ_p$ whose generic fiber is $M$. Since $\spec \ZZ_p$ is affine, by \cite[XXVI, Cor.~2.3]{SGA3III} we know that $\mathcal P$ admits a Levi component. It follows that the natural map $\mathcal P(\ZZ_p) \to \mathcal M(\ZZ_p)$ is surjective. Therefore, the subgroup $\mathcal M(\ZZ_p)$ of $M(\QQ_p)$ is equal to the image of $P(\QQ_p) \cap K_p$ under $P(\QQ_p) \to M(\QQ_p)$. Now since $M = M_h \times M_l$, any hyperspecial subgroup of $M(\QQ_p)$ (such as $\mathcal M(\ZZ_p)$) must be the direct product of a hyperspecial subgroup of $M_h(\QQ_p)$ and a hyperspecial subgroup of $M_l(\QQ_p)$. Hence the kernel of $\mathcal M(\ZZ_p) \hookrightarrow M(\QQ_p) \to M_h(\QQ_p)$, which is $M_l(\ZZ_p)$, must be a hyperspecial subgroup of $M_l(\QQ_p)$.
\end{rem}
\begin{rem}\label{rem:Tr_M = 0 for even M_2}
When $M = M_1$ or $M_{12}$, every element of $M_l (\QQ)$ is semi-simple $\RR$-elliptic, because $M_{l,\RR}$ is isomorphic to either $\SO(n-2, 0)$ or $\GG_{m,\RR} \times \SO(n-2, 0)$, and $\SO(n-2, 0) (\RR)$ is compact. When $d$ is even and $M= M_2$, we know that $M_{l,\RR}\cong \SO(n-1, 1)$ does not have elliptic maximal tori (as $n$ is even and at least $4$), so there are no $\RR$-elliptic elements of $M_l(\QQ)$ in the sense of Definition \ref{defn:R-ell}. In this case it is understood that $\Tr_M(f^{p,\infty} dg^{p,\infty} , K, a) =  0$.
\end{rem}

\ignore{
\begin{rem}
	The factor $\iota^{M_l} (\gamma_L)$ is omitted in \cite{morel2011suite} but recorded in \cite{morel2010book}. It arises when one deduces \cite[Thm.~1.6.6]{morel2010book} from \cite{GKM}. 
\end{rem}
}
\section{An equivalent form of Morel's formula} 

At this point we have defined the terms in (\ref{eq:in geometric assertion}). In this section we give an equivalent form of (\ref{eq:in geometric assertion}). It is this equivalent form that we shall prove in \S \ref{sec:Proof of Morel's formula}. In the following, we fix $K$, $p$, $f^{p,\infty}dg^{p,\infty}$, and $a$ as in Definition \ref{Defn Tr12}, and we shall omit them from the notation when convenient. For instance, we shall write $\Tr_M$ for $\Tr_M(f^{p,\infty} d g^{p,\infty} , K, a)$. 

\subsection{}
Let $M\in \set{M_1, M_2, M_{12}}$. We set \index{$\stdlev(M_l)$}
$$\stdlev(M_l) : = \begin{cases}
\set{M_{1,l}}, & M= M_1, \\
\set{M_{2,l}, M_{12,l}} , & M= M_2 ,\\
\set{M_{12,l}}, & M = M_{12}. 
\end{cases}$$
	Thus in each case $\stdlev(M_l)$ is a set of representatives of the $M_l(\QQ)$-conjugacy classes of Levi subgroups of $M_l$. 

\begin{defn}\label{defn:refined Tr_M}
	Let $M \in \set{M_1, M_2, M_{12}}$ and let $(Q,g) \in \mathcal P(M)$ (see Definition \ref{defn: P(M)}). We define $\Tr_{M, Q , g}$\index{$\Tr_{M,Q,g}$} by the same formula (\ref{eq:Tr_M}) used to define $\Tr_M$, but with $L_M(\cdot)$ replaced by $L_{M,Q,g}(\cdot)$ (see Definition \ref{defn:refined L_M}). Thus 
		\begin{multline}\label{eq:Tr_MQg}
		\Tr_{M,Q,g} : = \sum_{\gamma_L} \iota^{M_l}  (\gamma_L) ^{-1} \chi ((M_{l,\gamma_L }) ^0  )  \sum_{(\gamma_0 , \gamma,\delta)} c(\gamma_0, \gamma,\delta) \\   \cdot \delta_{P(\QQ_p)}  (\gamma_0) ^{1/2}    O_{\gamma_L \gamma} (f^{p,\infty} _M) O_{\gamma_L} (1_{M_l (\ZZ_p)}) TO_{\delta} (\phi_a^{M_h})L_{M,Q,g}(\gamma_L \gamma_0), 
	\end{multline}
\end{defn}

\begin{defn}\label{defn:T_Q flat}
	For $Q \in \set{P_1, P_2, P_{12}}$, we define \index{$\mathrm T_Q'$}
	$$\mathrm T_Q' : = \sum_{M} \Tr_{M , Q , 1},$$
	 where the sum is over $M \in \set{M_1, M_2, M_{12}}$ such that $M_l \in \stdlev(M_{Q,l})$. Indeed, for each such $M$, we have $(Q,1) \in \mathcal P(M)$, and so $\Tr_{M , Q , 1}$ is defined as in Definition \ref{defn:refined Tr_M}. 
\end{defn}
\begin{lem}	\label{lem:elementary} 
	We have \begin{align*}
	\Tr_{M_1} + \Tr_{M_2} + \Tr_{M_{12}} = \mathrm T_{P_1}' + \mathrm T_{P_2}' + \mathrm T_{P_{12}}'. 
	\end{align*}
\end{lem}
\begin{proof}
	By (\ref{eq:KW}), for each $M\in \set{M_1, M_2, M_{12}}$ we have 
	$$ \Tr_M = \sum_{(Q,g) \in \mathcal P(M)/{\sim}} \abs{\mathcal P(M,Q)/{\sim}}^{-1} \Tr_{M,Q,g}. $$ By Propositions \ref{prop:defn of L_M for M_1 M_2} and \ref{prop:to simplify LM12}, if $(Q,g) \in \mathcal P(M)$, then $(Q,1) \in \mathcal P(M)$. We claim that in this case $\Tr_{M,Q,g}= \Tr_{M,Q,1}$.  Indeed, by definition $L_{M,Q,g}(\gamma_L \gamma_0) = L_{M,Q,1}(g\gamma_Lg^{-1} \gamma_0)$, so it suffices to show that the expression $$\iota^{M_l} (\gamma_L) ^{-1} \chi ((M_{l,\gamma_L }) ^0  )     O_{\gamma_L \gamma} (f^{p,\infty} _M) O_{\gamma_L} (1_{M_l (\ZZ_p)})$$ on the RHS of (\ref{eq:Tr_MQg}) is invariant under the replacement $\gamma_L\mapsto g \gamma_L g^{-1}$. The invariance of  $\iota^{M_l} (\gamma_L)$ and  $ \chi ((M_{l,\gamma_L }) ^0  ) $ follows from the fact that $g$ normalizes $M_l$. To show the invariance of $O_{\gamma_L \gamma} (f^{p,\infty} _M)$, it suffices to show that $f^{p,\infty}_M$ and its composition with the automorphism $\Int(g)$ of $M(\adele_f^p)$ have equal orbital integrals at all elements. By Kazhdan's density result \cite{KazhdanCusp} it suffices to check this only at regular semi-simple elements. Since orbital integrals are locally constant on the regular semi-simple locus, we further reduce to $(G,M)$-regular semi-simple elements. That is, we only need to show the invariance of $O_{\gamma_L \gamma}(f^{p,\infty}_M)$ under $\gamma_L \mapsto g \gamma_L g^{-1}$ under the assumption that $\gamma_L \gamma$ is $(G,M)$-regular. This follows from the descent formula (see \cite[Lem.~6.1]{shintemplier}  or \cite{vandijk}) 
$$ O_{\gamma_L \gamma}	(f^{p,\infty} _M)  = \abs{D^{G} _{M} (\gamma_L \gamma)}_{\adele_f^p}^{1/2}  O_{\gamma_L \gamma} (f^{p,\infty})
	$$ and the fact that $g$ normalizes $M$.\footnote{The above argument of reducing to the $(G,M)$-regular case and then applying the descent formula is quite standard. In fact, one uses a similar argument to show, in the first place, that the choices made in the definition of the constant term do not affect its orbital integrals; cf.~\cite[\S 6.1]{shintemplier}.} Finally, to show the invariance of $O_{\gamma_L} (1_{M_l (\ZZ_p)})$, it suffices to show that $g M_l(\ZZ_p) g^{-1}$ is conjugate to $M_ l (\ZZ_p)$ in $M_l(\QQ_p)$. By the descriptions in Propositions \ref{prop:defn of L_M for M_1 M_2} and \ref{prop:to simplify LM12}, we are left to show that for any hyperspecial subgroup $U \subset \SO(W_2)(\QQ_p)$ and any $x \in \mathrm{O}(W_2)(\QQ_p) - \SO(W_2)(\QQ_p)$, we have $x U x^{-1}$ is conjugate to $U$ in $\SO(W_2)(\QQ_p)$. For this it suffices to exhibit one element of $\mathrm{O}(W_2)(\QQ_p) - \SO(W_2)(\QQ_p)$ normalizing $U$. But $U$ is the stabilizer of a $\ZZ_p$-lattice $\Lambda$ in $W_2$ (cf.~\cite[\S 2]{LustStevens}). Since $p>2$, if we take $v\in \Lambda$ such that $v_p(\lprod{v,v})$ is minimal, then the projection $w \mapsto w - \frac{\lprod{v,w}}{\lprod{v,v} } v$ preserves $\Lambda$, and hence $\Lambda$ is the orthogonal direct sum of $\ZZ_p v$ and its orthogonal complement in $\Lambda$. We can therefore take the desired element of $\mathrm{O}(W_2)(\QQ_p) - \SO(W_2)(\QQ_p)$ to be the reflection along $v$, which stabilizes $\Lambda$. This finishes the proof of the claim.
	
	By the claim we have 
	$$\Tr_M = \sum_{Q} \Tr_{M , Q , 1},$$ where the sum is over $Q \in \set{P_1, P_2, P_{12}}$ such that $\mathcal P(M,Q) \neq \emptyset$. We finish the proof by noting that for $M \in \set{M_1, M_2, M_{12}}$ and $Q \in \set{P_1, P_2, P_{12}}$, we have $ \mathcal P(M,Q) \neq \emptyset $ if and only if $M_l \in \stdlev(M_{Q,l}).$
\end{proof}

\begin{defn}\label{defn:T_Q}
	For $Q \in \set{P_1, P_2, P_{12}}$, we define \index{$\mathrm T_Q$}
	\begin{multline}\label{eq:T_Q}
	\mathrm T_Q = \mathtt m_{M_Q} \sum_{\fL \in \stdlev(M_{Q,l})}   (-1)^{\dim A_{\fL}/ A_{M_{Q,l}}}  (n^{M_{Q,l}} _{\fL})^{-1}  \\ \cdot \sum_{\gamma_L} \iota^{\fL}  (\gamma_L) ^{-1}  \chi (\fL_{\gamma_L} ^0)  \abs{D^{M_{Q,l}} _{\fL} (\gamma_L)}^{1/2}_{\RR}  \\  \cdot  \sum_{(\gamma_0 , \gamma,\delta)} c(\gamma_0, \gamma,\delta) \delta_{Q(\QQ_p)}  (\gamma_0) ^{1/2} O_{\gamma_L \gamma} (f^{p,\infty}_{M_\fL}) O_{\gamma_L} (1_{\fL (\ZZ_p)}) TO_{\delta} (\phi_a^{M_{Q,h}})     \\  \cdot    \delta_{Q(\RR)} (\gamma_L\gamma_0)^{1/2}  \Tr(\gamma_L\gamma_0 \mid  R\Gamma (\Lie N_Q , \mathbb V) _{>t_Q} ). \end{multline}
	Here, for each $\fL \in \stdlev(M_{Q,l})$, we let
	$M_\fL$ be the unique element of $\set{M_1, M_2, M_{12}}$ such that $M_{\fL, l} = \fL$. In other words, $M_{\fL} = M_{Q, h} \times \fL$. The second sum is over all semi-simple conjugacy classes $\gamma_L$ in $\fL(\QQ)$ which are $\RR$-elliptic in the sense of Definition \ref{defn:R-ell}. (If no such element exists, then the summand labeled by $\fL$ is zero.) The third sum is over equivalence classes of Kottwitz triples $(\gamma_0,\gamma, \delta)$ in $M_{Q,h}$ with $\gamma_0 \sim_{\RR} \gamma_L$. The definition of $\fL(\ZZ_p)$ is given by (\ref{eq:M_l(Z_p)}) applied to $M : = M_{\fL}$. All the other terms are defined in the same way as in Definition \ref{Defn Tr12}. 
\end{defn}
\begin{lem}\label{lem:no flat}	For $Q \in \set{P_1, P_2, P_{12}}$, we have $\mathrm T_Q'= \mathrm T_Q$. 
\end{lem}
\begin{proof}For each $\fL \in \stdlev (M_{Q,l})$, let $P_\fL$ be the unique element of $\set{P_1, P_2, P_{12}}$ such that $M_{P_{\fL}} = M_{\fL}$. 
Combining Definitions \ref{defn:refined L_M}, \ref{defn:refined Tr_M}, \ref{defn:T_Q flat}, and using the fact that for $\fL \in \stdlev (M_{Q,l})$ we have $M_{\fL, h} = M_{Q,h}$, we obtain
	\begin{multline*}
\mathrm T_Q' = \mathtt m_{M_Q} \sum_{\fL \in \stdlev(M_{Q,l})}   (-1)^{\dim A_{M_\fL}/ A_{M_Q}}  (n^{M_Q} _{M_\fL})^{-1} \\\cdot  \sum_{\gamma_L} \iota^{\fL}  (\gamma_L) ^{-1}  \chi (\fL_{\gamma_L} ^0)  \abs{D^{M_{Q}} _{M_\fL} (\gamma_L\gamma_0)}^{1/2}_{\RR}  \\   \cdot 
\sum_{(\gamma_0 , \gamma,\delta)} c(\gamma_0, \gamma,\delta) \delta_{P_\fL(\QQ_p)}  (\gamma_0) ^{1/2} O_{\gamma_L \gamma} (f^{p,\infty}_{M_\fL}) O_{\gamma_L} (1_{\fL (\ZZ_p)}) TO_{\delta} (\phi_a^{M_{Q,h}})   \\  \cdot    \delta_{Q(\RR)} (\gamma_L\gamma_0)^{1/2}  \Tr(\gamma_L\gamma_0 \mid  R\Gamma (\Lie N_Q , \mathbb V) _{>t_Q} ). \end{multline*}
Here the three summations are the same as on the RHS of (\ref{eq:T_Q}). To finish the proof, we only need to check the following four identities for each $\fL \in \stdlev(M_{Q,l})$: 
\begin{enumerate}
	\item $\dim A_{M_\fL}/A_{M_Q}  = \dim A_\fL / A_{M_{Q,l}}.$ 
	\item $	n^{M_Q} _{M_\fL}  = n^{M_{Q,l}}_{\fL}$.
	\item $D^{M_Q} _{M_\fL} (\cdot)  = D^{M_{Q,l}}_{\fL} (\cdot).$
	\item $\delta_{P_{\fL}(\QQ_p)} (\gamma_0)  = \delta_{Q(\QQ_p)} (\gamma_0)  . $
\end{enumerate}
The first three identities follow from the fact that $M_\fL = M_{Q,h} \times \fL$ and $M_Q = M_{Q,h} \times M_{Q,l}$. For the fourth identity, we have $P_{\fL} \subset Q$, and the subgroup $N_{P_{\fL}} / N_Q$ of $Q/N_Q = M_Q$ is contained inside $M_{Q,l} \subset M_Q$. Hence $\gamma_0\in M_{Q,h}(\QQ)$ acts trivially on $\Lie N_{P_{\fL}}/ \Lie N_Q$, and the desired identity follows. 
\end{proof}

\begin{prop}\label{prop:equivalent form} The formula (\ref{eq:in geometric assertion}) in Theorem \ref{geometric assertion} is equivalent to the following formula.  
	\begin{multline}\label{eq:equivalent form of Morel's formula}
	\Tr(\Frob_p^a \times f^{\infty} dg^{\infty} \mid  \icoh^* (\overline{ \Sh _K}, \mathbb V)) - \Tr(\Frob_p^a \times f^{\infty}  dg^{\infty} \mid   \coh_c^* ( \Sh _K,  \mathbb V))    \\ =    \mathrm T_{P_1} + \mathrm T_{P_2} + \mathrm T_{P_{12}} .
	\end{multline}\end{prop}
\begin{proof}
	This follows from Lemmas \ref{lem:elementary} and \ref{lem:no flat}. 
\end{proof}

\chapter{Proof of Morel's formula}\label{sec:Proof of Morel's formula}
In this chapter we prove Theorem \ref{geometric assertion}. 
\section{Introduction to the proof}\label{subsec:background}  
	
\subsection{}
Our goal is to prove the formula (\ref{eq:in geometric assertion}). In Proposition \ref{prop:equivalent form}, we have shown that (\ref{eq:in geometric assertion}) is equivalent to (\ref{eq:equivalent form of Morel's formula}). This last formula is a variant of  \cite[Thm.~1.7.1]{morel2010book}, and our proof will be a modification of the proof in \textit{loc.~cit.}.

First we review some key ingredients in \cite[Thm.~1.7.1]{morel2010book}. The proof is axiomatic in nature, building on the earlier work of Morel \cite{morelthese,morel2008complexes}, and the work of Pink \cite{pink1992ladic}. Other ingredients needed in this axiomatic approach include:
\begin{enumerate}
	\item Deligne's conjecture on local terms in the  Grothendieck--Lefschetz--Verdier trace formula, which was proved in special cases that are already enough for Shimura varieties by Pink \cite{pinkcalc}, and in general by Fujiwara \cite{fujiwara} and Varshavsky \cite{varshavsky}. 
	\item The fixed point formula of Goresky--Kottwitz--MacPherson \cite{GKM}.
	\item The fixed point formula of Kottwitz \cite{kottwitz1992points}.
\end{enumerate} 
The ingredient (1) is of course still valid in our case. As regards (2), we will need the original formula as well as a variant of it (see Proposition \ref{prop:generalization ofGKM} below). As regards (3), we will apply this formula to the boundary pure Shimura data $(\GG_m, \mathcal H_0)$ and $(\GL_2, \mathcal H_2)$. The Shimura datum $(\GL_2, \mathcal H_2)$ gives rise to the usual modular curves, and Kottwitz's formula is valid. For $(\GG_m, \mathcal H_0)$, we need a version of Kottwitz's fixed point formula for certain variants of the usual zero-dimensional Shimura varieties associated to the datum  (see Proposition \ref{prop:generalization of0-dim Shim} below). Finally, note that in Theorem \ref{geometric assertion} we have not provided a formula for the term $\Tr(\cdot \mid   \coh_c^* ( \Sh _K,  \mathbb V))$. Such a formula is eventually needed in order to fully understand the LHS of (\ref{eq:in geometric assertion}). This ingredient is provided in \cite{KSZ} (for all Shimura varieties of abelian type), and is treated as a black box in the present paper when we prove Corollary \ref{Main Main result} below. 
\subsection{}\label{discussion:2}
Let $P$ be a standard proper parabolic subgroup of $G$. There are
the following differences between our $\mathrm T_P$
in Definition \ref{defn:T_Q} and Morel's definition \cite[p.~23]{morel2010book}. We do not explicitly assume that the Kottwitz triples should have trivial Kottwitz invariant, but this is automatic by Lemma \ref{simplification 1 M1}. Also, in the first summation in (\ref{eq:T_Q}) we do not explicitly assume that $\fL$ is cuspidal (see Definition \ref{defn:cuspidal}), but in our case if $\fL$ is non-cuspidal then the sum over $\gamma_L$ is empty. (Indeed, the possible choices of $\fL$ are $M_{1,l}, M_{2,l}, M_{12,l}$. In the odd case all of them are cuspidal. In the even case, $M_{1,l}$ and $M_{12,l}$ are cuspidal, whereas $(M_{2,l})_{\RR}$ does not contain elliptic maximal tori, as noted in Remark \ref{rem:Tr_M = 0 for even M_2}.) The sole essential difference is that we impose the condition $\gamma_0 \sim_{\RR} \gamma_L$, which is not imposed by Morel, and this is due to the fact that our orthogonal Shimura datum $\mathbf O(V)$ does not satisfy the axioms in \cite[\S 1.1]{morel2010book}. 

Recall that Morel's axioms require that for each $P \in \admpar$, the Levi quotient $M_P$ of $P$ should admit a decomposition $M_P = G_P \times L_P$ such that $G_P(\RR)$ acts transitively on $\X_P$ and $L_P(\RR)$ acts trivially on $\X_P$, among other things.
In our case, by Proposition \ref{prop:rational bdry comp} (5), such a decomposition is clearly impossible for $P = P_2$. This is in fact related to the following geometric phenomenon. In general, each boundary stratum of the Baily--Borel compactification of a Shimura variety can be identified with the quotient of a smaller Shimura variety by the action of a finite group. If Morel's axioms are satisfied, then this finite quotient can be ``absorbed'' by a change of level. By contrast, in our case, the zero-dimensional boundary strata corresponding to $P_2$ cannot be identified as Shimura varieties without taking quotients.   
	
To resolve this problem, we need to systematically modify the arguments in \cite[Chap.~1]{morel2010book} whenever they concern zero-dimensional boundary strata. Roughly speaking, Morel's formula for $\mathrm T_P$ is a mixture of two formulas: the fixed point formula of Kottwitz for a Shimura variety associated to $G_P$, and the fixed point formula of Goresky--Kottwitz--MacPherson for a locally symmetric space associated to $L_P$. In our case, we need to replace the ``Shimura variety associated to $ G_P$'' by a finite quotient of it, and meanwhile replace the ``locally symmetric space associated to $L_P$'' by a finite covering of it. Fortunately, we only need these generalizations in very simple situations, and the extra complication is mainly of a combinatorial nature.
	
	\subsection{}\label{discussion:3}
We now discuss another ingredient in Morel's proof of \cite[Thm.~1.7.1]{morel2010book}, namely the construction of suitable integral models. In \cite[\S 1.3]{morel2010book} Morel provides two approaches to the construction of the integral model of the Baily--Borel compactification, for which Pink's formula (see \cite[Thm.~1.2.3]{morel2010book} and \cite[p.8 item (6)]{morel2010book}) holds, among other things. The first approach, \cite[Prop.~1.3.1]{morel2010book}, applies Lan's work \cite{lanbook} to construct the integral model away from a controlled finite set of bad primes. This approach is valid in the PEL-type case. The second approach, \cite[Prop.~1.3.4]{morel2010book}, is applicable in much more general situations, but it only constructs the integral model away from an \emph{uncontrolled} finite set of primes. Although Lan's work has been generalized by Madapusi Pera \cite{peratoroidal} to the case of Hodge type, our Shimura datum $\mathbf O(V)$ of abelian type is still beyond the applicability.\footnote{In \cite{lanstrohII}, Lan--Stroh have given a ``crude'' construction of the integral models of the Baily--Borel compactifications in the case of abelian type. However, since good integral models of \emph{toroidal compactifications }are also implicitly needed in order to verify Pink's formula, their construction does not seem sufficient for our purpose.} 
Hence we have to follow Morel's second approach, losing control of the set of bad primes. 
This explains why in Theorem \ref{geometric assertion} the set $\Sigma$ is not made specific and may also depend on $\lambda$ and $f^{\infty}$. 

Nevertheless, we shall show (see Lemma \ref{lem:enlarge} below) that the localizations to almost all primes of the abstract integral models constructed by Morel's second approach can be compared with other known integral models of Shimura varieties in the expected way. In particular, for sufficiently large primes we are in a position to apply the result of Lan--Stroh \cite[Thm.~4.19]{lanstrohII}, which relates the intersection cohomology and compact support cohomology of the special fiber of the integral model to those of the generic fiber respectively. 
\subsection*{Outline of the proof}
In \S \ref{subsec:gen2}, we prove an analogue of the fixed point formula of Goresky--Kottwitz--MacPherson for certain double coverings of locally symmetric spaces. The main result is Proposition \ref{prop:generalization ofGKM}. In \S \ref{subsec:gen1}, we study certain finite quotients of zero-dimensional Shimura varieties that will appear on the boundary of $\overline {\Sh_K}$. We develop the analogues of various constructions in \cite[Chap.~1]{morel2010book} for these quotients. The main results are Propositions \ref{prop:generalization of0-dim Shim} and  \ref{prop:variant of coh corr}. In \S \ref{subsec:axioms}, we explain how Morel's axioms in \cite[\S 1.1]{morel2010book} should be modified to suit our situation. In \S \ref{subsec:int mod}, we construct the integral models away from an uncontrolled set of bad primes, and compare the localizations of these models at almost all primes with other known integral models. In \S \ref{subsec:finish of proof}, we assemble all the ingredients and explain how to modify the proof of \cite[Thm.~1.7.1]{morel2010book} to prove our Theorem \ref{geometric assertion}. 

\section[A fixed point formula for double coverings]{A fixed point formula for some double coverings of locally symmetric spaces}\label{subsec:gen2}
\subsection{}\label{subsubsec:group setting for GKM variant}
Let $L$ be a reductive group over $\QQ$. We assume that $\pi_0(L(\RR)) \cong \ZZ/2\ZZ$. By the real approximation theorem, $L(\QQ) ^+ : = L(\QQ) \cap L(\RR) ^0$ is of index $2$ in $L(\QQ)$. We also assume that a minimal Levi subgroup $L_0$ of $L_{\RR}$ satisfies $\pi_0(L_0(\RR))\cong \ZZ/2\ZZ$. Then by Matsumoto's theorem (see \cite[14.4]{boreltits}), for any Levi subgroup $L'$ of $L_{\RR}$, the inclusion $L'(\RR) \hookrightarrow L(\RR)$ induces an isomorphism $\pi_0(L'(\RR)) \isom \pi_0(L(\RR))$. Now for each Levi subgroup $L'$ of $L$ defined over $\QQ$, we set \index{$L'(\QQ)^+$}
$$L'(\QQ)^+ : = L'(\QQ)\cap L'(\RR)^+,$$ which is of index $2$ in $L'(\QQ)$. 

\subsection{}\label{subsubsec:setting for GKM variant}
Let $L$ be as in \S \ref{subsubsec:group setting for GKM variant}. Let $K$ be a neat compact open subgroup of $L(\adele_f)$. Let $X_L$ be the symmetric space associated to $L_{\RR}$ as in Definition \ref{defn:q(G)}. We have the usual \emph{locally symmetric space} \index[n]{locally symmetric space} \index{$\M^K$}
$$\M^K : = L(\QQ) \backslash X_L \times L(\adele_f)/K, $$
as considered in \cite[\S 7]{GKM} and \cite[Chap.~1]{morel2010book}.
We shall consider the following variant of $\M^K$: \index{$\M^K_{\shall}$}
\begin{align*}
\M^K_{\shall}  : = L(\QQ) ^+ \backslash X_L \times L(\adele_f) / K.
\end{align*} We call $\M^K_{\shall}$ a \emph{shallower locally symmetric space}\index[n]{shallower locally symmetric space}. Both $\M^K$ and $\M^K_{\shall}$ are smooth real manifolds, and the natural map $\M^K_{\shall} \to \M^K$ is easily seen to be a double covering. 

Let $\mathbb W$ be an algebraic representation of $L_{\CC}$. Denote by $\mathcal F^K \mathbb W$\index{$\mathcal F^K \mathbb W$} the sheaf on $\M^K$ of local sections of the map
$$ L(\QQ) \backslash (\mathbb W \times X_L \times L(\adele_f)/ K) \To \M^K. $$ 
Denote by $R\Gamma_c (K, \mathbb W)$ \index{$R\Gamma_c (K, \mathbb W)$} the virtual alternating sum of the compact support cohomology $\coh^*_c (\M^K, \mathcal F^K \mathbb W)$. Similarly, we let $\mathcal F^K_{\shall} \mathbb W$\index{$\mathcal F^K_{\shall} \mathbb W$} be the sheaf on $\M^K_{\shall}$ of local sections of the map
$$ L^+(\QQ) \backslash (\mathbb W \times X_L \times L(\adele_f) / K) \To \M^K_{\shall}, $$ and denote by $R\Gamma_{c,\shall} (K, \mathbb W)$\index{$R\Gamma_c^{\shall} (K, \mathbb W)$} the virtual alternating sum of the compact support cohomology $\coh^*_c (\M^K_{\shall}, \mathcal F^K_{\shall} \mathbb W)$, cf.~\cite[\S 1.2]{morel2010book}.

Fix $g\in L(\adele_f)$, and let $K'\subset L(\adele_f)$ be another compact open subgroup such that $K' \subset K\cap gKg^{-1}$. Analogous to \cite[p.~22]{morel2010book}, we have finite \'etale Hecke operators 
$$T_1, T_g: \M^{K'}_{\shall} \To  \M^K_{\shall}.$$  As in \cite[Thm.~1.6.6]{morel2010book}, the natural cohomological correspondence $$T_g^* \mathcal F^K_{\shall} \mathbb W \To T_1^! \mathcal F^K_{\shall} \mathbb W$$ gives rise to an endomorphism $u_g$ of $R\Gamma_{c,\shall} (K, \mathbb W)$. \footnote{Note that Morel \cite[Thm.~1.6.6]{morel2010book} and Goresky--Kottwitz--MacPherson \cite{GKM} follow different conventions concerning the definition of $u_g$; see \cite[Rmk.~1.6.7]{morel2010book}. We follow Morel's convention here.}

Let $l_0$ denote the non-trivial element of $L(\QQ)/L(\QQ)^+$. We have a natural action of $L(\QQ)/L(\QQ)^+$ on $\M^K_{\shall}$, induced by the diagonal left action of $L(\QQ)$ on $X_L\times L(\adele_f)$. Under this action the covering  $\M^K_{\shall} \to \M^K $ is a $L(\QQ)/L(\QQ)^+$-torsor. The sheaf $\mathcal F^K_{\shall} \mathbb W$ has a natural $L(\QQ)/L(\QQ)^+$-equivariant structure, and so $l_0$ induces an endomorphism, still denoted by $l_0$, of $R\Gamma_{c,\shall} (K,\mathbb W)$. This endomorphism commutes with $u_g$. The following result is a variant of \cite[Thm.~1.6.6]{morel2010book}, the latter being a special case of \cite[Thm.~7.14 B]{GKM}.
\begin{prop}\label{prop:generalization ofGKM}
	In the setting of \S \ref{subsubsec:setting for GKM variant}, we have\begin{multline}\label{eq: gen GKM 1}
	\Tr(u_g \mid  R\Gamma_{c,\shall} (K, \mathbb W))  = 2\sum_{L' }  (-1)^{\dim (A_{L'}/A_L)}  (n^L_{L'}) ^{-1}  \sum_{\gamma} \iota^{L'}(\gamma) ^{-1} \chi ((L'_{\gamma} )^0)   \\ \cdot O_{\gamma} (f_{L'}^{\infty}) \abs{D ^L_{L'} (\gamma)} ^{1/2} \Tr(\gamma\mid  \mathbb W) , \end{multline} and 
\begin{multline}
	 \label{eq: gen GKM 2} 
	\Tr(u_g l_0\mid  R\Gamma_{c,\shall} (K, \mathbb W))  = 2\sum_{L' } (-1)^{\dim (A_{L'}/A_L)}  (n^L_{L'}) ^{-1} \sum_{\gamma} \iota^{L'}(\gamma) ^{-1} \chi ((L'_{\gamma} )^0)  \\  \cdot O_{\gamma} (f_{L'}^{\infty}) \abs{D ^L_{L'} (\gamma)} ^{1/2} \Tr(\gamma\mid  \mathbb W).
\end{multline}
	Here:
	\begin{itemize}
		\item  In both (\ref{eq: gen GKM 1}) and (\ref{eq: gen GKM 2}), the first sum is over $L(\QQ)$-conjugacy classes of Levi subgroups $L'$ of $L$.
		\item In (\ref{eq: gen GKM 1}) (resp.~(\ref{eq: gen GKM 2})), the second sum is over $L'(\QQ)$-conjugacy classes $\gamma$ in $L'(\QQ)^+$ (resp.~$L'(\QQ) - L'(\QQ) ^+$) that are $\RR$-elliptic in $L'(\RR)$ in the sense of Definition \ref{defn:R-ell}.  	
		\item We denote by $f^{\infty}$ the function $1_{gK}/\vol(K') \in C^{\infty}_c(L(\adele_f))$, and let $f^{\infty}_{L'}$ be the constant term of $f^{\infty}$ along $L'$, cf.~Definition \ref{Defn Tr12}. 	
		\item All the other terms on the right hand sides of (\ref{eq: gen GKM 1}) and (\ref{eq: gen GKM 2}) are defined in the same way as in \cite[Thm.~1.6.6]{morel2010book}, cf.~Definition \ref{Defn Tr12}. 
	\end{itemize} 
\end{prop}
\begin{proof}
	The formula	(\ref{eq: gen GKM 1}) follows from similar arguments as in \cite[\S 7]{GKM}. The key point is that the main tools used in \textit{loc.~cit.}, namely the reductive Borel--Serre compactification\index[n]{reductive Borel--Serre compactification} and the weighted complexes on it, are still available in the current setting. In fact, these objects were studied in \cite{GHM} in the non-adelic setting, where one is allowed to replace any given arithmetic subgroup by an arbitrary finite-index subgroup. Hence by the standard translation between the adelic and the non-adelic languages, we can consider the reductive Borel--Serre compactification of $\M^K_{\shall}$, as well as weighted complexes on it. The arguments in \cite[\S 7]{GKM} can be easily transported to this new setting. 
	
	We explain some more details. Fix a minimal parabolic subgroup $P_0$ of $L$, and fix a Levi component $L_0$ of $P_0$. For any standard parabolic subgroup $P$ of $L$ (i.e.~one that contains $P_0$), we denote by $L_P$ the Levi component of $P$ containing $L_0$, and denote by $N_P$ the unipotent radical of $P$. As in \cite[\S 7]{GKM}, the reductive Borel--Serre compactification of the usual locally symmetric space $\M^K$ has a stratification indexed by the standard parabolic subgroups $P$ of $L$. The stratum indexed by $P$ is of the form \begin{align}\label{form of stratum}
	L_P(\QQ)\backslash [ (N_P(\adele_f) \backslash L(\adele_f) /K) \times X_{L_P} ].
	\end{align} In \cite[\S 7]{GKM}, one considers the spaces $\mathrm{Fix} (P, x_0,\gamma)$, where $P$ runs through the standard parabolic subgroups of $L$, $x_0$ runs through representatives of the double cosets in $P(\adele_f) \backslash L(\adele_f)/ K'$, and $\gamma$ runs through conjugacy classes in $L_P(\QQ)$. Each space $\mathrm{Fix} (P,x_0, \gamma)$ is of the form
	$$ \mathrm{Fix} (P, x_0, \gamma) =  L_{P,\gamma} (\QQ) \backslash (Y^{\infty} \times Y_{\infty}).$$
	We refer the reader to \cite[p.~523]{GKM} for the definition of $Y^{\infty} $ and $Y_{\infty}$.
	
	For us, the reductive Borel--Serre compactification of $\M^K_{\shall}$ still has a stratification indexed by the standard parabolic subgroups $P$ of $L$, and the stratum indexed by $P$ is of the form 
	\begin{align}\label{form of stratum'}
	L_P(\QQ)^+ \backslash [ (N_P(\adele_f) \backslash L(\adele_f) /K) \times X_{L_P} ].
	\end{align} Comparing (\ref{form of stratum}) and (\ref{form of stratum'}), it is clear that if one is to count the fixed points of the cohomological correspondence in the same way as in \cite[\S 7]{GKM}, one should consider 
	\begin{align}\label{eq:to consider instead of GKM}
	\coprod_{P, x_0, \gamma} \mathrm{Fix}' (P,x_0,\gamma),
	\end{align}
	where $P$ runs through the standard parabolic subgroups of $L$, $x_0$ runs through representatives of the double cosets in $P(\adele_f) \backslash L(\adele_f)/ K'$, $\gamma$ runs through conjugacy classes in $L_P(\QQ)^+$, and 
	$$\mathrm{Fix}' (P,x_0,\gamma): = L_P(\QQ)^+_{\gamma} \backslash (Y^{\infty} \times Y_{\infty}). $$
	Here $L_P(\QQ)^+_{\gamma}$ denotes the centralizer of $\gamma$ in $L_P(\QQ)^+$
	
	Let $P$ be a standard parabolic subgroup of $L$. For $\gamma\in L_P(\QQ)^+$, we say that $\gamma$ is \emph{of first kind} if $L_{P,\gamma}(\QQ) \subset L_P(\QQ)^+$, and \emph{of second kind} if otherwise. When $\gamma$ is of first kind, the $L_P(\QQ)$-conjugacy class of $\gamma$ is the disjoint union of two $L_P(\QQ)^+$-conjugacy classes, and we have $\mathrm{Fix} ' (P, x_0, \gamma) = \mathrm{Fix} (P,x_0, \gamma)$. When $\gamma$ is of second kind, the $L_P(\QQ)$-conjugacy class of $\gamma$ is the same as the $L_P(\QQ)^+$-conjugacy class of $\gamma$, and $\mathrm {Fix}' (P,x_0 ,\gamma)$ is a double covering of $\mathrm{Fix} (P,x_0,\gamma)$. 
	From this discussion, we see that the space (\ref{eq:to consider instead of GKM}) is the same as \begin{align}\label{eq:to consider instead of GKM'}
	\coprod_{P, x_0, \gamma} \mathrm{Fix}'' (P,x_0,\gamma),
	\end{align}
	where $P$ and $x_0$ run through the same indexing sets as before, $\gamma$ runs through $L_P(\QQ)$-conjugacy classes in $L_P(\QQ)^+$, and 
	$\mathrm{Fix''} (P,x_0, \gamma) $ is the disjoint union of two copies of  $\mathrm{Fix} (P, x_0 ,\gamma)$ if $\gamma$ is of first kind, and is equal to $	\mathrm{Fix}' (P, x_0 ,\gamma) $ if $\gamma$ is of second kind.
	
	From the above discussion, the compact support Euler characteristic (see \cite[\S 7.10, 
\S 7.11]{GKM}) of $\mathrm{Fix''} (P,x_0, \gamma)$ is equal to twice that of $\mathrm{Fix} (P,x_0,\gamma)$. 
	
	In the qualitative discussion in \cite[\S 7.12]{GKM}, the contribution from $\mathrm{Fix} (P, x_0,\gamma)$ to the Lefschetz formula is a product of three factors (1), (2), and (3), where factor (1) is the compact support Euler characteristic of $\mathrm  {Fix} (P, x_0 ,\gamma)$. For us the contribution from $\mathrm{Fix} ''(P, x_0,\gamma)$ is also a product of three analogous factors, where our factors (2) and (3) are identical to those in \textit{loc.~cit.}, and our factor (1) is two times the factor (1) in \textit{loc.~cit.}~as we have already seen. Therefore, analogous to \cite[(7.12.1)]{GKM}, we have the following expression for the Lefschetz formula:
	\begin{align}\label{expression}
	\sum_{ P} \sum_{\gamma} 2 (-1) ^{\dim A_I /\dim A_{L_P}} \abs{L_{P,\gamma} (\QQ) / L_{\gamma}^0 (\QQ)} ^{-1} \chi (L_{\gamma}^0) \mathrm L_P^{\mathrm{GKM}} (\gamma) O_{\gamma} (f_P) ,
	\end{align}
	where $\gamma$ runs through the $L_P(\QQ)$-conjugacy classes in $L_P(\QQ)^+$ (instead of $L_P(\QQ)$), and the other notations are the same as in \textit{loc.~cit.~}except that we write $\mathrm L_P^{\mathrm{GKM}}(\cdot)$ for the function denoted by $L_P(\cdot)$ in \textit{loc.~cit.}. 
	
	Now the rest of the arguments in \cite[\S 7]{GKM} that deduce \cite[Thm.~7.14 B]{GKM} from \cite[(7.12.1)]{GKM} can be applied to (\ref{expression}). Also the elementary translation from \cite[Thm.~7.14 B, \S 7.17]{GKM} to the formula of \cite[Thm.~1.6.6]{morel2010book} carry over to imply (\ref{eq: gen GKM 1}). 
	
	We have proved (\ref{eq: gen GKM 1}). We now prove (\ref{eq: gen GKM 2}). We claim that \begin{align}
	\label{eq: sum of Lefschets}
	\Tr(u_g \mid  R\Gamma_{c,\shall} (K, \mathbb W))+ \Tr(u_gl_0 \mid  R\Gamma_{c,\shall} (K, \mathbb W))= 2\Tr (u_g \mid  R\Gamma _c (K,\mathbb W)).
	\end{align}
	Here we  abuse  notation and write $u_g$ also for the endomorphism of $R\Gamma_c (K,\mathbb W)$ induced by $g$. 
	Once (\ref{eq: sum of Lefschets}) is proved, the desired identity (\ref{eq: gen GKM 2}) follows from (\ref{eq: gen GKM 1}), (\ref{eq: sum of Lefschets}), and the formula for $\Tr (u_g \mid  R\Gamma _c (K,\mathbb W))$ given in \cite[Thm.~1.6.6]{morel2010book}. 
	
	We now prove (\ref{eq: sum of Lefschets}). Let $\pi$ denote the double covering map $\M^K_{\shall } \to \M^K$. We write $\mathscr F_{\shall}$ (resp.~$\mathscr F$) for the sheaf $\mathcal F^K_{\shall} \mathbb W$ (resp.~$\mathcal F^K\mathbb W$) on $\M^K_{\shall}$ (resp.~$\M^K$). Since $\mathscr F_{\shall} = \pi^* \mathscr F$, and since $\pi$ is a finite covering, we have
	\begin{align}\label{eq:ln}\coh^*_c(\M^K_{\shall}, \mathscr F_{\shall}) = \coh^*_c(\M^K_{\shall}, \pi^*\mathscr F)  = \coh^*_c (\M^K, \pi_* \pi^* \mathscr F).
	\end{align} For each character $\chi$ of the deck group $\Delta = \ZZ/2\ZZ$ of $\pi$, we let $\mathscr G_{\chi}$ be the local system on $\M^K$ given by the covering $\pi$ and the character $\chi$. 
	Combining (\ref{eq:ln}) and the projection formula $$\pi_* \pi ^* \mathscr F \cong \mathscr F  \otimes \bigoplus_{\chi: \Delta\to\CC^{\times}} \mathscr G_{\chi} ,$$ we obtain a decomposition 
	$$ \coh^*_c (\M^K_{\shall} ,\mathscr F_{\shall}) \cong \bigoplus_{\chi: \Delta \to \CC^{\times}}  \coh^*_c (\M^K, \mathscr F \otimes \mathscr G_{\chi}).$$ 
	This decomposition is equivariant with respect to $u_g$, and the direct summand $\coh^*_c (\M^K, \mathscr F \otimes \mathscr G_{\chi})$ corresponds to the $\chi$-eigenspace for the $\Delta$-action on the left hand side. The desired (\ref{eq: sum of Lefschets}) follows. 
\end{proof}

\section[Correspondences on  zero-dimensional Shimura varieties]{Cohomological correspondences on some zero-dimensional Shimura varieties}\label{subsec:gen1}

\subsection{} \label{subsubsec:zero dim SV}
Let $(\Gm, \HZ)$ be the zero-dimensional Siegel Shimura datum as in  \cite[2.8]{pink1989compactification}. Recall that $\HZ$ consists of two elements, and $\Gm(\RR)$ acts on $\HZ$ via the unique non-trivial action of $\pi_0(\Gm(\RR)) \cong \ZZ/2\ZZ$.
We now recall the construction of the associated zero-dimensional Shimura varieties, following \cite[11.3, 11.4]{pink1989compactification} and \cite[\S 5.5]{pink1992ladic}. 

As usual, we fix a neat compact open subgroup $K$ of $\Gm(\adele_f)$, and define the set of $\CC$-points of the Shimura variety as
$$\Sh_K(\CC) = \Sh_K(\Gm ,\HZ) (\CC): = \Gm(\QQ) \backslash \HZ \times \Gm(\adele_f)/ K. $$ 
There is a natural action of $\pi_0 (\Gm(\adele)/\Gm(\QQ))$ on the finite set $\Sh_K(\CC)$, from which we obtain an action of $\Gal(\overline \QQ / \QQ)$ on $\Sh_K(\CC)$ via the isomorphism 
\begin{align}\label{eq:CFT}
\Gal(\QQ ^{\mathrm{ab}} / \QQ) \isom \pi_0 (\Gm(\adele)/\Gm(\QQ)) 
\end{align} 
from class field theory (normalized such that geometric Frobenius elements correspond to uniformizers). The \emph{canonical model} $$\Sh_K=  \Sh_K(\Gm,\HZ)$$ is by definition the finite \'etale $\QQ$-scheme corresponding to the $\Gal(\overline \QQ/\QQ)$-set $\Sh_K(\CC)$. 

In fact, using the transitivity of the $\pi_0 (\Gm(\adele)/\Gm(\QQ))$-action on $\Sh_K(\CC)$, we can describe $\Sh_K$ more explicitly as follows. The inclusion $\hat \ZZ^\times \subset \Gm(\adele_f)$ induces an isomorphism $\hat \ZZ ^{\times} \isom \pi_0(\Gm(\adele)/\Gm(\QQ))$. We thus identify $\hat \ZZ^\times$ with $\Gal(\QQ^{\ab}/\QQ)$ via (\ref{eq:CFT}). (According to our normalization, this identification is induced by the Gauss isomorphisms $(\ZZ/m\ZZ )^\times \isom \Gal(\QQ(\zeta_m)/\QQ), k+ m\ZZ \mapsto (\zeta_m \mapsto \zeta_m^k)$.) Let $F_K/\QQ$\index{$F_K$} be the finite abelian extension corresponding to the open subgroup $K \subset \hat \ZZ^{\times} \cong \Gal(\QQ^{\ab}/\QQ)$. Then we have a canonical identification $$\Sh_K\cong  \spec F_K. $$ 
From this description, it is clear that $\varprojlim_{K} \Sh_K \cong \spec \QQ^{\ab}$.

Observe that the non-identity bijection $\HZ \to \HZ$ induces a bijection $\Sh_K(\CC) \to \Sh_K(\CC)$ which is $\pi_0(\Gm(\adele)/\Gm(\QQ))$-equivariant. From this we obtain an automorphism of the $\QQ$-scheme $\Sh_K$, denoted by $\sigma_{\infty}$\index{$\sigma_{\infty}$}. If we identify $\Sh_K $ with $\spec F_K$ as above, then $\sigma_{\infty}$ is given by the complex conjugation acting on $F_K$. Moreover, since $K$ is neat, we have $\QQ^{\times} \cap K = \set{1}$, and it follows that $\sigma_{\infty}$ is always a non-trivial automorphism of $\Sh_K$ (or equivalently, $F_K$ is always totally complex).

We denote by $\Sh_K^{\flat}$\index{$\Sh_K^{\flat}$} the quotient of $\Sh_K$ by $\sigma_{\infty}$. Thus $\Sh_K^{\flat} \cong \spec {F_K^{\flat}}$, where $F_K^{\flat}$\index{$F_K^{\flat}$} is the maximal totally real subfield of $F_K$. Alternatively, $\Sh_K^{\flat}$ is the Shimura variety at level $K$ associated to the Shimura datum $(\Gm , \set{\N_{\CC/\RR}: \mathbb S \to \GG_{m,\RR}})$. 

 We shall need a common generalization of the $\QQ$-schemes $\Sh_K$ and $\Sh_K^\flat$. First we define the generalization of a level subgroup.  
\ignore{\begin{rem}
	For any neat compact open subgroup $K$ of $\Gm(\adele_f)$, the $\QQ$-scheme $\Sh_K^{\flat}$ could be identified with the canonical model of the Shimura variety at level $K$ associated to the Shimura datum $(\Gm, \HZ')$, where $\HZ'$ is the image of $\HZ$ in $\Hom(\mathbb S, \GmR)$. For example, if $K$ is the principal congruence subgroup $\set{g \in \hat \ZZ^\times| g \equiv 1 \mod m}$ with $m\in \ZZ_{\geq 3}$, then $K$ is neat, and we have $\Sh_K = \spec \QQ(\zeta_m)$, $\Sh_K^{\flat} = \spec \QQ(\zeta_m + \zeta_m^{-1})$. \end{rem}  }
 
  \begin{defn}\label{defn:neat in the broad sense} We say that a subgroup $U$ of $\Gm(\adele_f)\times \ZZ/2\ZZ$ is an \emph{admissible level}\index[n]{admissible level}, if there are neat compact open subgroups $K_1$ and $K_2$ of $\Gm(\adele_f)$ such that $$K_1 \times \set{0} \subset U \subset K_2 \times \ZZ/2\ZZ. $$ 
 \end{defn}
 \subsection{}\label{subsubsec:generalized SV} Note that for any neat compact open subgroup $K \subset \Gm(\adele_f)$, we have $K\subset \hat \ZZ^{\times}$, and the element $-1 \in \hat \ZZ^{\times}$ is not in $K$. Thus $K \times \ZZ/2\ZZ$ can be identified with a subgroup of $\hat \ZZ^{\times}$, where the non-trivial element of $\ZZ/2\ZZ$ corresponds to $-1 \in \hat \ZZ^{\times}$. It follows that every admissible level $U$ as in Definition \ref{defn:neat in the broad sense} can be canonically identified with an open subgroup of $\hat \ZZ^{\times} \cong \Gal(\QQ^{\ab}/\QQ)$, and thus determines a finite abelian extension $F_U/\QQ$. We define
 \begin{align*}
\Sh_U = \Sh_U(\Gm,\HZ)&: = \spec F_U.
 \end{align*}
When $U\subset \Gm(\adele_f)$, the current definition of $\Sh_U$ agrees with the one in \S \ref{subsubsec:zero dim SV}. Also, if $K$ is a neat compact open subgroup of $\Gm(\adele_f)$, then $K\times \ZZ/2\ZZ$ is an admissible level and we have $\Sh_{K\times \ZZ/2\ZZ} = \Sh_K^{\flat}$. 
 
 The usual Hecke operators can be generalized to this new setting as follows. Let $U$ be an admissible level, and let $g \in \Gm(\adele_f)\times \ZZ/2\ZZ$. We shall define an automorphism 
 $$[\cdot g]_U : \Sh_U \To \Sh_U. $$
For this we identify $g$ with an element of $\Gm(\adele)$ by identifying $\ZZ/2\ZZ$ with $\set{\pm 1} \subset \RR^{\times}$. Then $g$ determines an element $\rho(g)\in \Gal(\QQ^{\ab}/\QQ)$ via the inverse of (\ref{eq:CFT}). We define $[\cdot g]_U$ to be the automorphism of $\Sh_U = \spec F_U$ corresponding to the restriction of $\rho(g)$ to $F_U$.  

If $U'$ is another admissible level contained in $U$, then we have a natural map $\Sh_{U'} \to \Sh_U$, and the two compositions 
$$
\xymatrix{
\Sh_{U'}   \ar[r]  & \Sh_U  \ar[r]^{[\cdot g]_U} & \Sh_ U , \\
\Sh_{U'} \ar[r]^{[\cdot g]_{U'}}&  \Sh_ {U'} \ar[r]&  \Sh_U } 
$$ are equal. We denote them by $[\cdot g]_{U', U}$. 

If $K$ is a neat compact open subgroup of $\Gm(\adele_f)$ and $g\in \Gm(\adele_f)$, then the above definition of $[\cdot g]_K$ recovers the usual Hecke operator on $\Sh_K$. If $\epsilon$ denotes the non-trivial element of $\ZZ/2\ZZ$, then $[\cdot \epsilon]_K$ is the automorphism $\sigma_{\infty}$ of $\Sh_K$ as in \S \ref{subsubsec:zero dim SV}. 
     
For an admissible level $U$, we define \index{$\Shh_U(\Gm,\HZ)$}
$$\Shh_U = \Shh_U(\Gm,\HZ) : = \spec \oo_{F_U},$$ and call it the \emph{canonical integral model}\index[n]{canonical integral model} of $\Sh_U$. The Hecke operators $[\cdot g]_U$ and $[\cdot g]_{U',U}$ as above uniquely extend to the canonical integral models. 
\begin{lem}\label{lem:Galois groups}
	Let $U_1$ and $U_2$ be two admissible levels with $U_1 \subset U_2$. Then the following statements hold.
	\begin{enumerate}
		\item The natural map $\Sh_{U_1} \to \Sh_{U_2}$ is a Galois covering and a $U_2/U_1$-torsor.
		\item Let $p$ be a prime number such that $\ZZ_p^{\times} \subset U_1$. (Here $\ZZ_p^\times$ is viewed as a subgroup of $\Gm(\adele_f) \subset \Gm(\adele_f)\times \ZZ/2\ZZ$.) Then $\Shh_{U_i}\otimes_{\ZZ} \ZZ_{(p)}$ are finite \'etale over $\ZZ_{(p)}$ for $i = 1,2$. Moreover, the natural map $\Shh_{U_1}\otimes_{\ZZ} \ZZ_{(p)} \to\Shh_{U_2}\otimes_{\ZZ} \ZZ_{(p)}$ is a Galois covering and a $U_2/U_1$-torsor. 
	\end{enumerate}
\end{lem} 
\begin{proof} Statement (1) is just Galois theory. To show (2), we observe that $p$ is unramified in $F_{U_1}$ and $F_{U_2}$ by class field theory. 
\end{proof}

\subsection{}\label{subsubsec:setting for zero dim SV}
Let $L$ be a reductive group over $\QQ$, and fix a continuous action of $L(\RR)$ on the set $\HZ$. We write $L(\QQ)^\natural$\index{$L(\QQ)^\natural$} for $\mathrm{Cent}_{L(\QQ)} \HZ$. Thus $L(\QQ)^{\natural}$ is a normal subgroup of $L(\QQ)$ of index at most $2$. We have a canonical injection 
\begin{align}\label{eq:can inj}
L(\QQ) /L(\QQ)^{\natural} \hookrightarrow \Aut(\HZ) = \ZZ/2\ZZ.
\end{align}

Let $M = \Gm \times L.$ Thus the group $M(\RR)$ acts on $\HZ$, where we let $\Gm(\RR)$ act as in \S \ref{subsubsec:zero dim SV}. Let $K_M$ be a neat compact open subgroup of $M(\adele_f)$. Define \index{$K_{\da}$}
\begin{align}\nonumber
K_{\da} : = K_M/(K_M\cap L(\adele_f)).
\end{align} We identify $K_{\da}$ with the image of $K_M$ under the projection $M(\adele_f) \to \Gm(\adele_f)$. Since $K_M$ is a neat compact open subgroup of $M(\adele_f)$, we know that $K_{\da}$ is a neat compact open subgroup of $\Gm(\adele_f)$. Define the following subgroups\footnote{In the application, typically $M$ will be the Levi quotient of a parabolic subgroup $P$ of a reductive group $G$, and we reserve the notations $H, H_L^{\natural}$ for certain subgroups of $P(\adele_f)$ whose images in $M(\adele_f)$ are the subgroups $\bar H, \bar H_L^{\natural}$ defined here, cf.~\S \ref{subsubsec:relation}.} of $M(\adele_f)$:\index{$\bar H$} \index{$\bar H_L^{\natural}$} 
\begin{align} \label{eq:bar H}
\bar H  & : = K_M \cap( \Gm(\adele_f) L(\QQ)),\\ \label{eq:bar H_L^natural}
\bar H_L^{\natural} & : = K_M \cap L(\QQ)^{\natural}.
\end{align}
Note that $\bar H_L^{\natural}$ is a normal subgroup of $\bar H$. We define \index{$\check H$} $$\chH : = \bar H/\bar H_L^{\natural}.$$ We have a natural homomorphism $\chH \to \Gm(\adele_f)$ induced by the projection map $K_M \to \Gm(\adele_f)$, and a natural homomorphism $\chH \to \ZZ/2\ZZ$ induced by the composition 
$$ \Gm(\adele_f) L(\QQ) \to L(\QQ) \to L(\QQ)/ L(\QQ)^{\natural} \xrightarrow{(\ref{eq:can inj})} \ZZ/2\ZZ,$$ where the first map is the projection to the second factor. Taking the product, we obtain a homomorphism $\chH \to \Gm(\adele_f) \times \ZZ/2\ZZ $ which is injective. We use it to view $\chH$ as a subgroup of $\Gm(\adele_f) \times \ZZ/2\ZZ$.

\begin{lem}\label{lem:explaining the H groups}
In the setting of \S \ref{subsubsec:setting for zero dim SV}, the following statements hold.\begin{enumerate}
	\item We have $\bar H_L^{\natural} = K_M \cap (\mathrm{Cent}_{M(\QQ)} \HZ)$. 
	\item The subgroup $\chH$ of $\Gm(\adele_f)\times \ZZ/2\ZZ$ is an admissible level. 
	\end{enumerate}
\end{lem}
\begin{proof}
	For (1), the containment $\bar H_L^{\natural} \subset  K_M \cap (\mathrm{Cent}_{M(\QQ)} \HZ)$ is clear. For the reverse containment, let $g\in \Gm(\QQ)$ and $l\in L(\QQ)$ be such that $gl \in K_M \cap (\mathrm{Cent}_{M(\QQ)} \HZ)$. Then $g \in K_{\da} \cap \Gm(\QQ)$, which is the trivial group by the neatness of $K_{\da}$. Hence $gl = l$, and $l \in L(\QQ)^{\natural}$. This shows (1). For (2), we let $K_1=K_M\cap \Gm(\adele_f)$ and $K_2= K_{\da}$. Then $K_1$ and $K_2$ are neat compact open subgroups of $\Gm(\adele_f)$, and we have 
	$K_1 \times \set{0} \subset \chH \subset K_2 \times \ZZ/2\ZZ$.  
\end{proof}
\subsection{}\label{subsubsec:defn of Sh_H}
We keep the setting of \S \ref{subsubsec:setting for zero dim SV}. By Lemma \ref{lem:explaining the H groups} (2), $\chH$ is an admissible level. Applying the construction in \S \ref{subsubsec:generalized SV}, we obtain a $\QQ$-scheme $\Sh_{\chH}$ and a $\ZZ$-scheme $\Shh_{\chH}$. 

By definition, the profinite Galois covering $\spec \QQ^{\ab} \to \Sh_{\chH}$ is a $\chH$-torsor. We may thus construct \'etale sheaves on $\Sh_{\chH}$ associated to suitable $\chH$-modules. More precisely, let $\mathrm{Rep}_{M}$\index{$\mathrm{Rep}_{M}$} be the category of finite-dimensional algebraic representations of $M$ on $\mathbb E_{\lambda}$-vector spaces (where $\mathbb E$ and $\lambda$ are as in \S \ref{subsubsec:V and lambda}). Let $D^b(\mathrm{Rep}_{M})$\index{$D^b(\mathrm{Rep}_{M})$} be the bounded derived category of $\mathrm{Rep}_{M}$ (i.e., the category of graded objects of $\mathrm{Rep}_{M}$ of finite length, as $\mathrm{Rep}_{M}$ is semi-simple). As explained in \cite[\S 1.2]{morel2010book} and \cite[\S 2.1.4]{morelthese}, we have an additive triangulated functor \index{$\mathcal F^{\chH} R\Gamma(\bar H_L^{\natural}, - )$}
\begin{align}\label{eq:complicated functor}
\mathcal F^{\chH} R\Gamma(\bar H_L^{\natural}, - ) : D^b(\mathrm{Rep}_{M}) \To  D^b_c(\Sh_{\chH}, \mathbb E_{\lambda}).
\end{align} Roughly speaking, to compute this functor at $\mathbb W \in D^b(\mathrm{Rep_{M}})$, one first applies the right derived functor of $\coh^0(\bar H_L^{\natural},-)$ to $\mathbb W$ to get a complex of $\chH$-modules, and then uses this complex and the $\chH$-tower $\spec \QQ^{\ab} \to \Sh_{\chH}$ to construct a complex of $\mathbb E_{\lambda}$-sheaves on $\Sh_{\chH}$. We refer the reader to \cite[\S 2.1.4, G\'en\'eralisation]{morelthese} for the precise construction.\footnote{It is assumed in \textit{loc.~cit.~}that $L(\QQ) = L(\QQ)^{\natural}$, but this assumption can be removed without affecting any of the arguments.}

Using (\ref{eq:complicated functor}), we define the following functor, which can be viewed as a compact support analogue: \index{$\mathcal F^{\chH} R\Gamma_c(\bar H_L^{\natural},-) $}
\begin{align}\label{eq:compact supp 1}
\mathcal F^{\chH} R\Gamma_c(\bar H_L^{\natural},- ) : D^b(\mathrm{Rep}_{M}) & \To D^b_c(\Sh_{\chH'} ,\mathbb  E_{\lambda}),  \\ \nonumber
\mathbb W & \longmapsto D\bigg( \mathcal F^{\chH} R\Gamma(\bar H_L^{\natural}, \mathbb W^*)[2q(L_{\RR})] \bigg),
\end{align}
where $D(\cdot)$ denotes the Verdier dual, $\mathbb W^*$ denotes the contragredient of $\mathbb W$, and $q(L_{\RR})$ is as in Definition \ref{defn:q(G)}. 

Similarly, let $\mathrm{Rep}_{\Gm\times \ZZ/2\ZZ}$\index{$\mathrm{Rep}_{\Gm\times \ZZ/2\ZZ}$} be the category of finite-dimensional algebraic representations of $\Gm\times \ZZ/2\ZZ$ on $\mathbb E_{\lambda}$-vector spaces, and let $D^b(\mathrm{Rep}_{\Gm\times \ZZ/2\ZZ})$\index{$D^b(\mathrm{Rep}_{\Gm\times \ZZ/2\ZZ})$} be the bounded derived category. (Here we view $\ZZ/2\ZZ$ as a constant group scheme.) We have an additive triangulated functor \index{$\mathcal F^{K_{\da}\times \ZZ/2\ZZ}(-)$}
\begin{align}\label{eq:functor 1 for K flat}
\mathcal F^{K_{\da}\times \ZZ/2\ZZ}(-) : D^b(\mathrm{Rep}_{\Gm\times \ZZ/2\ZZ}) \To D^b_c(\Sh_{K_{\da}\times \ZZ/2\ZZ}, \mathbb E_{\lambda}) =D^b_c(\Sh_{K_{\da}}^{\flat}, \mathbb E_{\lambda}) 
\end{align} given as follows. Let $\mathbb W \in \mathrm{Rep}_{\Gm\times \ZZ/2\ZZ} $. First viewing $\mathbb W$ as an algebraic representation of $\Gm$, we obtain the associated automorphic $\mathbb E_{\lambda}$-sheaf on $\Sh_{K_{\da}}$ as usual (see \S \ref{subsubsec:V and lambda}). We then use the $\ZZ/2\ZZ$-action on $\mathbb W$ to define the descent datum with respect to the double covering $\Sh_{K_{\da}} \to \Sh_{K_{\da}} ^{\flat}$, and obtain an $\mathbb E_{\lambda}$-sheaf on $\Sh_{K_{\da}} ^{\flat}$. Equivalently, we let the Galois group $\Gamma $ of $\spec \QQ^{\ab} \to \Sh_{K_{\da}}^{\flat}$, namely $\Gamma = K_{\da} \times \ZZ/2\ZZ \subset \hat \ZZ^{\times}$, act on $\mathbb W$ via the projection $\Gamma \to \Gm(\QQ_{\ell}) \times \ZZ/2\ZZ$ followed by the canonical action of $\Gm(\QQ_{\ell}) \times \ZZ/2\ZZ$ on $\mathbb W$. We then obtain an $\mathbb E_{\lambda}$-sheaf on $\Sh_{K_{\da}}^{\flat}$ via the $\Gamma$-torsor  $\spec \QQ^{\ab} \to \Sh_{K_{\da}}^{\flat}$ and the $\Gamma$-representation $\mathbb W$.
\subsection{} \label{subsubsec:RGammaU} We keep the setting of \S \ref{subsubsec:setting for zero dim SV}. Let $D^b(\mathrm{Rep}_{M})$ and $D^b(\mathrm{Rep}_{\Gm\times \ZZ/2\ZZ})$ be as in \S \ref{subsubsec:defn of Sh_H}. For any neat compact open subgroup $U \subset L(\adele_f)$, we shall construct a functor \index{$R\Gamma_{\natural}(U, -)$}
\begin{align}\label{eq:functor 2 for K flat}
R\Gamma_{\natural}(U,-) : D^b(\mathrm{Rep}_{M}) \To D^b(\mathrm{Rep}_{\Gm\times \ZZ/2\ZZ}).
\end{align}
The construction is similar to the one described in \cite[Rmk.~1.5.2 (1)]{morel2010book}. Consider the space \index{$\M^{U}_{\natural}$}
$$\M^{U}_{\natural} : = L(\QQ)^{\natural}\backslash X_L \times L(\adele_f)/U, $$ where $X_L$ is as in Definition \ref{defn:q(G)}. Thus $\M^{U}_{\natural}$ is a variant of the usual locally symmetric space 
$\M^U$, cf.~\S \ref{subsubsec:setting for GKM variant}. We know that $\M^U_{\natural}$ is a smooth manifold, and the natural map $\M^U_{\natural} \to \M^U$ is a covering map of degree $[L(\QQ): L(\QQ)^{\natural}]$. (In our later application, $L$ will satisfy the assumptions in \S \ref{subsubsec:group setting for GKM variant} and we will have $L(\QQ)^{\natural} = L(\QQ)^+$, so $\M^U_{\natural}$ is the same as $\M^U_{\shall}$ discussed in \S \ref{subsubsec:setting for GKM variant}.)

Fix a system of representatives $(l_i)_{i\in I}$ of the double cosets in $L(\QQ)^{\natural}\backslash L(\adele_f)/U$. Here the indexing set $I$ is finite, since the set $L(\QQ)\backslash L(\adele_f)/U$ is finite and $[L(\QQ): L(\QQ)^{\natural}] \leq 2$. Then we have 
$$\M^U_{\natural} \cong \coprod_{i\in I} \Gamma_i \backslash X_L, $$ where each $\Gamma_i : = l_i U l_i ^{-1} \cap L(\QQ)^{\natural}$ is a neat arithmetic subgroup of $L(\QQ)$. For $\mathbb W \in D^b(\mathrm{Rep}_{M})$, we define 
\begin{align}\label{eq:RGamma U}
R\Gamma_{\natural}(U, \mathbb W) : = \bigoplus_{ i \in I} R \Gamma (\Gamma_i, \mathbb W),
\end{align}
where each $R \Gamma (\Gamma_i, -)$ is the functor $D^b(\mathrm{Rep}_{M}) \to D^b(\mathrm{Rep}_{\Gm})$ as in \cite[Rmk.~1.2.2]{morel2010book} such that the cohomology of $R \Gamma (\Gamma_i, \mathbb W)$ computes the group cohomology $\coh^*(\Gamma_i, \mathbb W)$.

 We further equip $R\Gamma_{\natural} (U,\mathbb W)$ with a $\ZZ/2\ZZ$-action as follows: If $L(\QQ)/L(\QQ)^{\natural}$ is trivial, we define this action to be trivial. Assume that $L(\QQ)/L(\QQ)^{\natural}\cong \ZZ/2\ZZ$. Then left-multiplication by the non-trivial element of $L(\QQ)/L(\QQ)^{\natural}$ induces an involution on the set $L(\QQ)^{\natural} \backslash L(\adele_f)/U$, and hence an involution on $I$. If $\set{i,j}$ is a size-two orbit in $I$ under the involution, then there is a canonical coset in $\Gamma_j \backslash L(\QQ)$ consisting of $l\in L(\QQ)$ satisfying $ll_i \in l_j U$. For any such $l$, the isomorphism $\mathbb W \to \mathbb W$ given by the action of $l$ intertwines with the isomorphism $\Gamma_i \isom \Gamma_j$ given by $\Int(l)$, and we obtain an isomorphism $\tau_{i,j}:R\Gamma(\Gamma_i , \mathbb W) \isom R\Gamma(\Gamma_j , \mathbb W)$, which is independent of the choice of $l$. Moreover, the isomorphism $\tau_{j,i}: R\Gamma(\Gamma_j , \mathbb W)\isom R\Gamma(\Gamma_i , \mathbb W)$ obtained in the similar way is inverse to $\tau_{i,j}$. Now consider a size-one orbit $\set{i}$ in $I$ under the involution. Then $\Gamma_i$ is a subgroup of $l_i U l_i^{-1}\cap L(\QQ)$ of index $2$. For any $l \in (l_i U l_i^{-1}\cap L(\QQ) ) - \Gamma_i$, the isomorphism $\mathbb W \to \mathbb W$ given by the action of $l$ intertwines with the isomorphism $\Gamma_i \isom \Gamma_i$ given by $\Int(l)$, and we obtain an automorphism $\tau_{i}$ of $R\Gamma(\Gamma_i , \mathbb W)$, which is independent of the choice of $l$ and has order at most $2$. The collection of $\tau_{i,j}$ and $\tau_{i}$ as above thus gives a canonical $\ZZ/2\ZZ$-action on $R\Gamma_{\natural}(U,\mathbb W)$, and we thereby view $R\Gamma_{\natural}(U,\mathbb W)$ as an object in $D^b(\mathrm{Rep}_{\Gm\times \ZZ/2\ZZ})$. 
 
 At this point, we have constructed the desired functor (\ref{eq:functor 2 for K flat}), after fixing the choice of a system of representatives $(l_i)_{i\in I}$. It can be checked that changing the system of representatives does not change the functor up to natural isomorphism. 
 
 Using (\ref{eq:functor 2 for K flat}), we define the following functor as a compact support analogue:  \index{$R\Gamma_{c,\natural}(U, -)$}
 \begin{align}\label{eq:compact supp 2}
 R\Gamma_{c,\natural}(U, -) : D^b(\mathrm{Rep}_{M}) & \To D^b(\mathrm{Rep}_{\Gm\times \ZZ/2\ZZ}) \\
 \nonumber \mathbb W &\longmapsto  \bigg(R\Gamma_{\natural}(U, \mathbb W^*)[2q(L_{\RR})] \bigg)^*,
 \end{align}
 where $*$ denotes taking contragredient, and $q(L_{\RR})$ is as in Definition \ref{defn:q(G)}. 
  
\begin{rem}\label{rem:incarnation}For $\mathbb W \in \mathrm{Rep}_M$, the object $R\Gamma_{\natural}(U,\mathbb W)$ (resp.~$R\Gamma_{c,\natural}(U,\mathbb W)$) is an incarnation of the cohomology (resp.~cohomology with compact support) of the space $\M^U_{\natural}$ ``with coefficients in $\mathbb W$''. To explain this, fix a field embedding $\mathbb E_{\lambda} \hookrightarrow \CC$. Then $\mathbb W$ determines an algebraic representation $\mathbb W_{\CC}$ of $L_{\CC}$ over $\CC$. Consider the sheaf $\mathcal F^U_{\natural} (\mathbb W_{\CC})$ of local sections of $$ L(\QQ)^{\natural}\backslash \mathbb W_{\CC} \times X_L \times L(\adele_f)/U \To \M^U_{\natural},$$ cf.~\cite[\S 1.2]{morel2010book} and \S \ref{subsubsec:setting for GKM variant}. Then for each $k\in \ZZ$, the base change to $\CC$ of the $k$-th cohomology of $R\Gamma_{\natural}(U,\mathbb W)$ (resp.~$R\Gamma_{c,\natural}(U,\mathbb W)$) is isomorphic to $\coh^k(\M^U_{\natural}, \mathcal F^U_{\natural} (\mathbb W_{\CC}))$ (resp.~$\coh^k_c(\M^U_{\natural}, \mathcal F^U_{\natural} (\mathbb W_{\CC}))$). 
\end{rem}
\subsection{}\label{subsubsec:setting for coh corr} We keep the setting of \S \ref{subsubsec:setting for zero dim SV}. Consider the following situation, which is a special case of the situation described below \cite[Notation 1.5.1]{morel2010book}. Fix $m\in \Gm(\adele_f)L(\QQ) \subset M(\adele_f)$. Let $K_M'$ be a compact open subgroup of $M(\adele_f)$ such that $$K_M' \subset K_M \cap mK_Mm^{-1}. $$ Let $\bar H'$ and $(\bar H_{L}^{\natural})'$ be defined by the formulas  (\ref{eq:bar H}) and (\ref{eq:bar H_L^natural}, but with $K_M$ replaced by $K_M'$. Note that we have $\bar H' \subset \bar H \cap m\bar H m^{-1}$. 

Let $\chH' : = \bar H'/(\bar H_L^{\natural})'$. Let $\theta_1: \chH' \to \chH$ be the homomorphism induced by $\Int(m^{-1}): \bar H' \to \bar H$, and let $\theta_2: \chH' \to \chH$ be the homomorphism induced by the inclusion $\bar H' \subset \bar H$. As a generalization of the functor (\ref{eq:complicated functor}), for $i\in \set{1,2}$ we have a functor\index{$\mathcal F^{\chH'} \theta_i^* R\Gamma(\bar H_L^{\natural},- )$} 
\begin{align}\label{eq:more complicated functor}
\mathcal F^{\chH'} \theta_i^* R\Gamma(\bar H_L^{\natural},- ) : D^b(\mathrm{Rep}_{M}) \To D^b_c(\Sh_{\chH'} ,\mathbb  E_{\lambda}). \end{align}
 To compute this functor at $\mathbb W \in D^b(\mathrm{Rep}_M)$, roughly speaking one first applies the right derived functor of $\coh^0(\bar H_L^{\natural},-)$ to $\mathbb W$ to get a complex of $\chH$-modules, then pulls this complex back via $\theta_i^*$ to obtain a complex of $\chH'$-modules, and then uses the last complex and the $\chH'$-tower $\spec \QQ^{\ab} \to \Sh_{\chH'}$ to construct a complex of $\mathbb E_{\lambda}$-sheaves on $\Sh_{\chH'}$. The precise construction of (\ref{eq:more complicated functor}) is along similar lines as the construction of (\ref{eq:complicated functor}), for which we refer to \cite[\S 1.5]{morel2010book}. 

Let $\bar m$ be the image of $m$ in $\Gm(\adele_f) \times (L(\QQ)/L(\QQ)^{\natural}) \subset \Gm(\adele_f) \times \ZZ/2\ZZ$. As in \S \ref{subsubsec:generalized SV}, we have Hecke operators 
\begin{align*}
[\cdot \bar m]_{\chH', \chH} : \Sh_{\chH'} & \To \Sh_{\chH} ,\\ 
 [\cdot 1]_{\chH', \chH} : \Sh_{\chH'} & \To \Sh_{\chH} .
\end{align*}
In the sequel we denote them simply by $[\cdot m]$ and $[\cdot 1]$. 

Let $\mathbb W \in D^b(\mathrm{Rep}_{M})$. Applying the functor (\ref{eq:complicated functor}) to $W$, we obtain 
$$\mathscr L : = \mathcal F^{\chH} R\Gamma(\bar H_L^{\natural}, \mathbb W) \in D^b_c(\Sh_{\chH}, \mathbb E_\lambda).$$  As explained in \cite[\S 1.5]{morel2010book}, it follows from \cite[Prop.~1.11.5]{pink1992ladic} that there are canonical isomorphisms 
\begin{align*}
\mathcal F^{\chH'} \theta_1^* R\Gamma(\bar H_L^{\natural}, \mathbb W) & \cong [\cdot m]^* \mathscr L, \\
 \mathcal F^{\chH'} \theta_2^* R\Gamma(\bar H_L^{\natural}, \mathbb W) & \cong [\cdot 1]^* \mathscr L .
\end{align*}
Using these isomorphisms, as in \cite[\S 1.5]{morel2010book} one constructs a canonical cohomological correspondence \index{$c_{m,1}$}
\begin{align}\label{eq:coh corr 1}
c_{m,1} : [\cdot m]^* \mathscr L \To [\cdot 1]^! \mathscr L  =  [\cdot 1]^* \mathscr L .
\end{align}
(Both sides are complexes of sheaves on $\Sh_{\chH'}$.) Similarly, applying the functor (\ref{eq:compact supp 1}) we obtain 
$$\mathscr L_c: = \mathcal F^{\chH} R\Gamma_c(\bar H_L^{\natural}, \mathbb W) \in D^b_c(\Sh_{\chH}, \mathbb E_\lambda),$$ and there is a canonical cohomological correspondence 
\begin{align}\label{eq:coh corr 1 c}
c_{m,1} : [\cdot m]^* \mathscr L_c \To [\cdot 1]^! \mathscr L_c  =  [\cdot 1]^* \mathscr L_c.
\end{align}

Now let $p$ be a prime number which is coprime to $\lambda$ and hyperspecial for $K_M$ (see Definition \ref{defn:hyperspecial pr}). Assume in addition that $m\in \Gm(\adele_f^p)L(\QQ)$. Then there exists $K_M'$ as in the above discussion such that $p$ is also hyperspecial for $K_M'$. For such $K_{M}'$, it is clear from Lemma \ref{lem:Galois groups} (2) that the Hecke operators $[\cdot m]: \Sh_{\chH'} \to \Sh_{\chH}$ and  $[\cdot 1]: \Sh_{\chH'} \to \Sh_{\chH}$ extend to finite \'etale morphisms $\Shh_{\chH'} \otimes_{\ZZ} \ZZ_{(p)} \to \Shh_{\chH} \otimes_{\ZZ} \ZZ_{(p)}$ (still denoted by $[\cdot m]$ and $[\cdot 1]$), that $\mathscr L$ and $\mathscr L_c$ extend to complexes of lisse $\mathbb E_{\lambda}$-sheaves on $\Shh_{\chH} \otimes_{\ZZ} \ZZ_{(p)}$, and that the cohomological correspondences (\ref{eq:coh corr 1}) and (\ref{eq:coh corr 1 c}) also extend. We denote by $\overline {\mathscr L}$ (resp.~$\overline {\mathscr L_c}$) the pull-back to $\Shh_{\chH}\otimes_{\ZZ} \FF_p$ of the extension of $\mathscr L$ (resp.~$\mathscr L_c$) over  $\Shh_{\chH} \otimes_{\ZZ} \ZZ_{(p)}$. As in \cite[Notation 1.5.1]{morel2010book}, for any $a\in \ZZ_{\geq 1}$ we can twist the reductions of (\ref{eq:coh corr 1}) and (\ref{eq:coh corr 1 c}) over $\FF_p$ by the $a$-th power of the absolute Frobenius, and obtain cohomological correspondences \index{$\Phi^a c_{m,1}$} 
\begin{align*}
\Phi^a c_{m,1} : [\cdot m]^* \overline {\mathscr L } &  \To [\cdot 1]^! \overline {\mathscr L} , \\
\Phi^a c_{m,1} : [\cdot m]^* \overline {\mathscr L_c} &  \To [\cdot 1]^! \overline {\mathscr L_c}. 
\end{align*} 
\begin{defn}\label{defn:Tr_mathcal H} In the setting of \S \ref{subsubsec:setting for coh corr}, we define 
\index{$\Tr_{\mathcal H} (a, m, K_M, K_M' ,\mathbb W)$} \index{$\Tr_{\mathcal H, c} (a, m, K_M, K_M' ,\mathbb W)$}
\begin{align*}
\Tr_{\mathcal H} (a, m, K_M, K_M' ,\mathbb W) & := \sum_k (-1)^k\Tr(\Phi^ac_{m,1} \mid \coh^k(\Shh_{\chH}\otimes_{\ZZ} \overline \FF_p, \overline {\mathscr L})), \\
\Tr_{\mathcal H, c} (a, m, K_M, K_M' ,\mathbb W) & := \sum_k (-1)^k \Tr(\Phi^ac_{m,1} \mid \coh^k(\Shh_{\chH}\otimes_{\ZZ} \overline \FF_p, \overline {\mathscr L_c})).
\end{align*}
\end{defn}
\subsection{}\label{subsubsec:setting for coh corr for K flat} We keep the setting of \S \ref{subsubsec:setting for zero dim SV}. Let $\mathbb U \in D^b(\mathrm{Rep}_{\Gm\times \ZZ/2\ZZ})$. Applying the functor (\ref{eq:functor 1 for K flat}) to $\mathbb U$ we obtain
$$\mathscr M : = \mathcal F^{K_{\da}\times \ZZ/2\ZZ} (\mathbb U) \in D^b_c(\Sh_{K_{\da}\times \ZZ/2\ZZ}, \mathbb E_\lambda).$$  Let $p$ be a prime number coprime to $\lambda$ such that $\ZZ_p^{\times} \subset K_{\da}$. (For instance, if $p$ is hyperspecial for $K_M$, then $\ZZ_p^{\times} \subset K_{\da}$.) Let $g \in \Gm(\adele_f^p)$.  As in \S \ref{subsubsec:generalized SV}, we have the Hecke operator
\begin{align*}[\cdot g]_{K_{\da}\times \ZZ/2\ZZ} : \Sh_{K_{\da} \times \ZZ/2\ZZ} \To  \Sh_{K_{\da} \times \ZZ/2\ZZ}, 
\end{align*} which we denote simply by $[\cdot g]$. Similarly as in \S \ref{subsubsec:setting for coh corr}, we have a canonical cohomological correspondence $$ u(0,g): [\cdot g]^* \mathscr M \To \mathscr M, $$
and we can pass to the special fiber of the canonical integral model mod $p$, twist by the $a$-th power of the absolute Frobenius (where $a \in \ZZ_{\geq 1}$), and obtain a cohomological correspondence 
\index{$u(a,g) $}
\begin{align}\label{eq:u(a,g)}
u(a,g) = \Phi^a u(0,g)  : [\cdot g]^* \overline{\mathscr M} \To \overline {\mathscr M}. 
\end{align} 
\begin{defn}\label{defn:Tr on flat} In the setting of \S \ref{subsubsec:setting for coh corr for K flat}, we define \index{$\Tr(a,g, K_{\da}\times \ZZ/2\ZZ, \mathbb U)$}
\begin{align*}
\Tr(a,g, K_{\da}\times \ZZ/2\ZZ, \mathbb U) : =  \sum _k (-1)^k \Tr(u(a,g) \mid \coh^k(\Shh_{K_{\da} \times \ZZ/2\ZZ} \otimes_{\ZZ} \overline \FF_p, \overline {\mathscr M})). 
\end{align*}
\end{defn}

The following result is a variant of \cite[Rmk.~1.6.5]{morel2010book}.
\begin{prop}\label{prop:generalization of0-dim Shim} Keep the setting of \S \ref{subsubsec:setting for coh corr for K flat}. We have 
	\begin{align}\label{eq: gen 0-dim}
\Tr(a,g, K_{\da}\times \ZZ/2\ZZ, \mathbb U) = \sum _{(\gamma_0,\gamma,\delta)}   c(\gamma_0,\gamma,\delta) O_{\gamma} (f^p) TO_{\delta} (\phi_a^{\GG_m}) \widetilde \Tr(\gamma_0 \mid \mathbb U).
	\end{align}
	Here the terms on the right are as follows.  
	\begin{enumerate}
		\item The summation is over Kottwitz triples $(\gamma_0,\gamma,\delta)$ in $\Gm$ of level $p^a$, as in \S \ref{subsec: defn in Kott}. 
		\item The terms $c(\gamma_0,\gamma,\delta)$, $ O_{\gamma} (\cdot)$, and  $TO_{\delta} (\cdot) $ are defined as in \S \ref{subsec: defn in Kott}. 
		\item We define $f^p := 1_{gK_{\da}^p} / \vol(K_{\da}^p) \in C^{\infty}_c(\Gm(\adele_f^p))$, where $K_{\da}^p$ is the subgroup of $\Gm(\adele_f^p)$ such that $K_{\da} = \ZZ_p^{\times}K_{\da}^p$. The function $\phi_a^{\Gm}$ is as in Definition \ref{defn:phi_a}. 
		\item For any $\gamma_0 \in \Gm(\QQ) = \QQ^\times$, we set \index{$\widetilde \Tr$}
		$$ \widetilde \Tr(\gamma_0 \mid \mathbb U) : =  \begin{cases}
		\Tr(\gamma_0 \mid  \mathbb U) , &  \text{if } \gamma_0>0,\\ 
		\Tr (\gamma_0 \times \epsilon \mid \mathbb U), & \text{if }\gamma_0<0,
		\end{cases} $$  
		where $\epsilon$ denotes the non-trivial element of $\ZZ/2\ZZ$.
	\end{enumerate}
\end{prop}
\begin{proof} We write $K$ for $K_{\da}$, and write $S$ for the set $\Shh_{K_{\da}\times \ZZ/2\ZZ}(\overline \FF_p)$. We identify the three sets $S$, $\Sh_{K}^{\flat} (\CC)$, and $\Gm(\QQ)\backslash \Gm(\adele_f)/K = \QQ^{\times} \backslash \adele_f^{\times}/K$. Let $\Phi$ be the endomorphism of $S$ induced by the absolute Frobenius on the $\FF_p$-scheme $\Shh_{K_{\da}\times \ZZ/2\ZZ}$.  We denote by $p_p$ the image of $p$ under the embedding $\QQ_p^{\times}\hookrightarrow \adele_f^\times$. Then the endomorphism $\Phi^a \circ [\cdot g]$ of $S$ is given by the multiplication by $p_p^a g$ on $\QQ^{\times} \backslash \adele_f^{\times}/K$. Similarly, we write $\tilde S$ for $\Shh_K(\overline \FF_p)$, and identify it with $\Sh_K(\CC) = \QQ^{\times} \backslash \HZ \times \adele_f^{\times}/K \cong \QQ_{>0} \backslash \adele_f^{\times} / K$. 

Since we are in the zero-dimensional case, we can compute $\Tr(a,g, K_{\da}\times \ZZ/2\ZZ, \mathbb U) $ by summing the naive local terms over the fixed points of $S$ under $\Phi^a \circ [\cdot g]$. 

	Let $x\in S$ be a fixed point under $\Phi^a\circ [\cdot g]$. Then $x$ has a representative $\tilde x\in\adele_f^{\times}$ for which there exists $f_0 \in \QQ^{\times}$ satisfying $f_0\tilde x \in p_p^ag \tilde x K$, or equivalently $f_0\in p_p^a g K$. Hence the set of fixed points is non-empty if and only if $\QQ^{\times} \cap p_p^a g K \neq \emptyset$, and when it is non-empty it is equal to $S$.

	If $\QQ^{\times} \cap p_p^a g K  = \emptyset$, then $\Tr(a,g, K_{\da}\times \ZZ/2\ZZ, \mathbb U) = 0 $ since there are no fixed points. In this case it is straightforward to check that the RHS of (\ref{eq: gen 0-dim}) is also zero.
	
	Assume that $\QQ^{\times} \cap p_p^a g K \neq \emptyset$. In this case, this set has a unique element $f_0$, since we have $\QQ^{\times} \cap K = \set{1}$ by the neatness of $K$. We have seen that in this case every point in $S$ is a fixed point. There are two cases to consider.
	
	First suppose that $f_0>0$. Then every point in $\tilde S$ is fixed by $\Phi^a \circ [\cdot g]$. Write $g_{\ell}\in \Gm(\QQ_{\ell})$ for the $\ell$-adic component of $g$. In this case, the naive local term at each point in $S$ is equal to the naive local term at any one of the two lifts of that point in $\tilde S$, and the latter is equal to the trace on the algebraic $\Gm(\QQ_{\ell})$-representation $\mathbb U$ of the product of $g_{\ell}^{-1}$ and the $\ell$-adic component of $f_0^{-1}p_p^a g \in K$ (cf.~the argument on \cite[p.~433]{kottwitz1992points}). Hence the native local term is equal to $\Tr(f_0^{-1} \mid \mathbb U)$, which is equal to $\widetilde \Tr(f_0^{-1} \mid \mathbb U)$ since $f_0 > 0$. 
	
	Now suppose that $f_0<0$. Then for every point in $S$, the two lifts of it in $\tilde S$ are permuted non-trivially by $\Phi^a \circ [\cdot g]$. In this case, the naive local term at each point in $S$ is equal to the trace on the algebraic $\Gm(\QQ_{\ell}) \times \ZZ/2\ZZ$-representation $\mathbb U$ of the product of $g_{\ell}^{-1}$ and the projection to $\QQ_{\ell}^\times \times \ZZ/2\ZZ$ of $f_0^{-1}p_p^a g  \times \epsilon \in K \times \ZZ/2\ZZ$, which is $\Tr(f_0^{-1} \times \epsilon \mid \mathbb U) = \widetilde{\Tr}(f_0^{-1} \mid \mathbb U)$.
	
	We conclude that in both cases the naive local term at each point in $S$ is equal to $\widetilde{\Tr}(f_0^{-1} \mid \mathbb U)$. Hence 
	$$ \Tr(a,g, K_{\da}\times \ZZ/2\ZZ, \mathbb U) = \widetilde \Tr(f_0^{-1} \mid \mathbb U) \abs{S}. $$ To compute the RHS of (\ref{eq: gen 0-dim}), we note that every Kottwitz triple $(\gamma_0,\gamma,\delta)$ that makes a non-zero contribution must satisfy $\gamma_0 = f_0^{-1}$. (In fact all Kottwitz triples satisfying this condition are in one equivalence class.) On the other hand, by \cite[Rmk.~1.6.5]{morel2010book} we know that 
	$$ \sum _{(\gamma_0,\gamma,\delta)}   c(\gamma_0,\gamma,\delta) O_{\gamma} (f^p) TO_{\delta} (\phi_a^{\GG_m}) = \frac{1}{2} \abs{\tilde S},$$ which is nothing but $\abs{S}$. Hence the RHS of (\ref{eq: gen 0-dim}) is equal to $\widetilde \Tr(f_0^{-1} \mid \mathbb U) \abs{S}$ as well. The proof is complete. 
\end{proof} 

\subsection{} \label{subsubsec:setting for variant of coh corr} 
We now state a variant of \cite[Prop.~1.7.2]{morel2010book}.
We keep the setting of \S \ref{subsubsec:setting for zero dim SV}. Let $p$ be a prime number which is coprime to $\lambda$ and hyperspecial for $K_M$. Fix $m \in M(\adele_f^p)$ (not necessarily in $\Gm(\adele_f^p)L(\QQ)$). Let $K_M'$ be a compact open subgroup of $M(\adele_f)$ such that $p$ is hyperspecial for $K_M'$ and such that $$K_M'\subset K_M \cap m K_M m^{-1}. $$ Fix a system of representatives $(m_i)_{i\in I}$ of those double cosets $e$ in $$\Gm(\adele_f)L(\QQ) \backslash M(\adele_f)/K_M'
$$
satisfying $$em K_M = e K_M. $$ For every $i\in I$, let $g_i\in \Gm(\adele_f)$ and $l_i\in L(\QQ)$ be such that $$g_i l_i m_i \in m_i m K_M. $$ We may and shall assume that $m_i \in M(\adele_f^p)$ and $g_i\in \Gm(\adele_f^p)$ for each $i\in I$. 

Let $\mathbb W \in D^b(\mathrm{Rep}_{M})$, and let 
\begin{align*}
	\mathbb U  & := R\Gamma_{\natural}(K_M\cap L(\adele_f), \mathbb W) \in D^b(\mathrm{Rep}_{\Gm\times \ZZ/2\ZZ}),\\   
	\mathbb U_c  &:= R\Gamma_{c,\natural}(K_M\cap L(\adele_f), \mathbb W) \in D^b(\mathrm{Rep}_{\Gm\times \ZZ/2\ZZ}),
\end{align*} where the notations are as in (\ref{eq:functor 2 for K flat}) and (\ref{eq:compact supp 2}). The following result is a variant of \cite[Prop.~1.7.2]{morel2010book}.
\begin{prop}\label{prop:variant of coh corr} Keep the setting and notation of \S \ref{subsubsec:setting for variant of coh corr}.   
	Write $g$ for the projection of $m$ in $\Gm(\adele_f^p)$, and write $K_{\da}'$ for $ K_M'/(K_M'\cap L(\adele_f))$.
Assume that $[L(\QQ) : L(\QQ)^{\natural}] = 2$. Then for each $a\in \ZZ_{\geq 1}$ we have 
	\begin{multline}\label{eq:I to 1}
	 \sum_{i\in I} \Tr_{\mathcal H}(a, g_il_i, m_iK_Mm_i^{-1}, m_i K_M' m_i^{-1}, \mathbb W) \\ = \Tr(a, g,  K_{\da} \times \ZZ/2\ZZ, \mathbb U) \cdot [K_{\da}: K_{\da}'], \end{multline} and \begin{multline}
	  \label{eq:I to 1'}
	  \sum_{i\in I} \Tr_{\mathcal H,c}(a, g_il_i, m_iK_Mm_i^{-1}, m_i K_M' m_i^{-1}, \mathbb W) \\  = \Tr(a, g,  K_{\da} \times \ZZ/2\ZZ, \mathbb U_c) \cdot [K_{\da}: K_{\da}'].
	\end{multline}
Here the terms $\Tr_{\mathcal H}(\cdots)$, $\Tr_{\mathcal H,c}(\cdots)$,  and $\Tr(\cdots)$ are as in Definitions \ref{defn:Tr_mathcal H} and \ref{defn:Tr on flat}.  
 \end{prop} 
\begin{rem}The RHS of (\ref{eq:I to 1}) is indeed the analogue of the RHS of \cite[Prop.~1.7.2 (1)]{morel2010book}. We have the seemingly extra factor $[K_{\da}: K_{\da}']$, but this is due to the fact that in our definition of the cohomological correspondence (\ref{eq:u(a,g)}) we used the Hecke operator $[\cdot g]$ as an endomorphism of $\Sh_{K_{\da} \times \ZZ/2\ZZ}$, as opposed to using the correspondence $\Sh_{K_{\da} \times \ZZ/2\ZZ} \xleftarrow{[\cdot g] } \Sh_{K_{\da}' \times \ZZ/2\ZZ} \xrightarrow{[\cdot 1]} \Sh_{K_{\da} \times \ZZ/2\ZZ}$. 
	
	Similarly, the RHS of (\ref{eq:I to 1'}) is the analogue of the RHS of \cite[Prop.~1.7.2 (2)]{morel2010book}.
\end{rem}
\begin{proof} By duality, (\ref{eq:I to 1}) implies (\ref{eq:I to 1'}). The proof of (\ref{eq:I to 1}) is essentially the same as that of \cite[Prop.~1.7.2(1)]{morel2010book}, the only difference being that we need to modify Morel's functor $R\Gamma(K_M , -)$ (and its analogue for $K_M'$). Below we explain this modification.

	Consider the space \index{$\M^{K_M}_{\natural}$}
$$\M^{K_M}_{\natural} : =  M(\QQ)\backslash (\HZ \times X_L \times M(\adele_f))/K_M \cong (\mathrm{Cent}_{M(\QQ)}\HZ) \backslash X_L\times M(\adele_f)/K_M, $$ where $M(\QQ)$ acts on $\HZ \times X_L \times M(\adele_f)$ diagonally, and for the action of $M(\QQ)$ on $\HZ$ both the factors $\Gm(\QQ)$ and $L(\QQ)$ act non-trivially. (The action of $L(\QQ)$ on $\HZ$ is via the unique non-trivial action of $L(\QQ)/L(\QQ)^{\natural}\cong \ZZ/2\ZZ$.) Thus $\M^{K_M}_{\natural}$ is a  double covering  of the usual locally symmetric space $$\M^{K_M} = M(\QQ) \backslash X_M \times M(\adele_f)/K_M, $$ where $X_M = X_L$ since $X_{\Gm}$ is a point. Let $R\Gamma_{\natural} (K_M, \mathbb W)$\index{$R\Gamma_{\natural} (K_M, \mathbb W)$} be the ``cohomology of $\M^{K_M}_{\natural}$ with coefficients in $\mathbb W$'' (cf.~Remark \ref{rem:incarnation}). Namely, we write $\M_{\natural}^{K_M}$ as 
$$ \coprod_{j \in J} (n_j K_M n_j^{-1} \cap \mathrm{Cent}_{M(\QQ)} \HZ) \backslash X_L,  $$
where $(n_j)_{j\in J}$ is a system of representatives of the double cosets in $$(\mathrm{Cent}_{M(\QQ)} \HZ) \backslash M(\adele_f)/ K_M,$$ and define  
$$R\Gamma_{\natural} (K_M, \mathbb W): = \bigoplus_{j\in J} R\Gamma(n_j K_M n_j^{-1} \cap \mathrm{Cent}_{M(\QQ)} \HZ , \mathbb W) $$ inside the derived category of finite-dimensional $\mathbb E_{\lambda}$-vector spaces. 

Observe that we have a fibration 
\begin{align}\label{eq:fibration 1}
M^{K_M}_{\natural} \To \Gm(\QQ) \backslash \Gm(\adele_f) / K_{\da}
\end{align}
 induced by the projection $\HZ \times X_L \times M(\adele_f) \to \Gm(\adele_f)$. The fibers of (\ref{eq:fibration 1}) are naturally identified with $$ L(\QQ) \backslash \HZ \times X_L \times L(\adele_f)/ (K_M \cap L(\adele_f)),$$ which we observe is the same as $\M^{K_M \cap L(\adele_f)}_{\natural}$ defined in \S \ref{subsubsec:RGammaU}, since $[L(\QQ) : L(\QQ)^{\natural}] =2$. The base of the fibration (\ref{eq:fibration 1}) is identified with $\Sh_{K_{\da} \times \ZZ/2\ZZ} (\CC)$. Hence we have identifications 
 \begin{align}\label{eq:idn 1}
R\Gamma_{\natural}(K_M, \mathbb W)  \cong R\Gamma(\Sh_{K_{\da} \times \ZZ/2\ZZ} \otimes_{\QQ} \CC ,  {\mathscr M}) \cong  R\Gamma(\Shh_{K_{\da} \times \ZZ/2\ZZ} \otimes_{\ZZ} \overline \FF_p, \overline {\mathscr M}),
 \end{align}
 where $\mathscr M = \mathcal F^{K_{\da} \times \ZZ/2\ZZ} (\mathbb U)$ and $\overline{\mathscr M}$ is the reduction of $\mathscr M$ (cf.~\S \ref{subsubsec:setting for coh corr for K flat}). 
 
 On the other hand, we have a fibration 
 \begin{align}\label{eq:fibration 2}
 	M^{K_M}_{\natural} \To \Gm(\adele_f) L(\QQ) \backslash M(\adele_f)/K_M
 \end{align} induced by the projection $\HZ \times X_L \times M(\adele_f) \to M(\adele_f)$. For each $e \in M(\adele_f)$, we denote by $\M^{K_M}_{\natural} (e)$ the fiber of (\ref{eq:fibration 2}) over the double coset of $e$. Then $\M^{K_M}_{\natural} (e)$ is identified with $$ \Gm(\QQ) \backslash \HZ \times  \Gm(\adele_f) \times X_L  / \bar H_e,$$ where $\bar H_e : = eK_M e^{-1} \cap (\Gm(\adele_f) L(\QQ))$ is the analogue of $\bar H$ in \S \ref{subsubsec:setting for zero dim SV} with $e K_M e^{-1}$ replacing the role of $K_M$, and the right action of $\bar H_e$ on $\HZ \times \Gm(\adele_f) \times X_L  $ is given as follows. The action of $\bar H_e$ on $\HZ \times \Gm(\adele_f)$ is the restriction of the $\Gm(\adele_f) L(\QQ)$-action, where $\Gm(\adele_f)$ acts on $\Gm(\adele_f)$ by multiplication and $L(\QQ)$ acts on $\HZ$ via the non-trivial action of $L(\QQ)/L(\QQ)^{\natural}$. The action of $\bar H_e$ on $X_L$ is given by the restriction of the projection map $\Gm(\adele_f)L(\QQ) \to L(\QQ)$ followed by the inversion on $L(\QQ)$ and followed by the natural left $L(\QQ)$-action on $X_L$. 

Let $\bar H_{L,e}^{\natural}: = eK_M e^{-1} \cap L(\QQ)^{\natural}$ and $\chH_e : = \bar H_e / \bar H_{L,e}^{\natural}$, which are the analogues of $\bar H_L^{\natural}$ and $\chH$ in \S \ref{subsubsec:setting for zero dim SV} with $eK_Me^{-1}$ replacing the role of $K_M$. Then we have a fibration 
\begin{align}\label{eq:fibration 3}
 \M^{K_M}_{\natural} (e) \To \Gm(\QQ) \backslash \HZ \times \Gm(\adele_f) / \chH_e,
\end{align}
 where the $\chH_e$-action on $\HZ \times \Gm(\adele_f)$ is induced by the $\bar H_e$-action on $\HZ \times \Gm(\adele_f) \times X_L$ in the above. The fibers of (\ref{eq:fibration 3}) are identified with $X_L/ \bar H_{L,e}^{\natural}$, while the base is identified with $\Sh_{\chH_e}(\CC)$. Hence we have an identification 
 \begin{align}\label{eq:idn 2}
R\Gamma_{\natural}(K_M, \mathbb W)  \cong \bigoplus_e R\Gamma(\Sh_{\chH_e} \otimes_{\QQ} \CC , \mathscr L(e) ) \cong  \bigoplus_e  R\Gamma(\Shh_{\chH_e} \otimes_{\ZZ} \overline \FF_p, \overline {\mathscr L(e)}), 
 \end{align}
 where $e$ runs through a system of representatives of the double cosets in $$\Gm(\adele_f) L(\QQ) \backslash M(\adele_f)/K_M, $$ and for each $e$ we define $\mathscr L(e) :  =  \mathcal F^{\chH_e} R\Gamma(\bar H_{L,e}^{\natural}, \mathbb W)$ and define $\overline {\mathscr L(e)}$ to be its reduction (cf.~\S \ref{subsubsec:setting for coh corr}). 

In view of (\ref{eq:idn 1}) and (\ref{eq:idn 2}), we can replace Morel's functor $R\Gamma(K_M, -)$ by $R\Gamma_{\natural}(K_M, - )$ (and also for $K_M'$) and proceed in exactly the same way as in \cite[Prop.~1.7.2]{morel2010book} to conclude the proof. 
\end{proof}

\section{Modifying Morel's axioms}\label{subsec:axioms}
\subsection{}\label{subsubsec:axioms}
Let $(G,\X)$ be a pure Shimura datum. We keep the notation in \S \ref{subsec:boundary}. We replace the axioms on p.~2 of \cite{morel2010book} by the following axioms: 
\subsubsection*{\textbf{A0}}For each $P \in \admpar$, the Levi quotient $M_P$ of $P$ admits a decomposition $M_P = G_P \times L_P$, where $G_P$ is the image of $P^{\Pink} \subset P$ as in \S \ref{subsec:boundary}. 
\subsubsection*{\textbf{A1}} For each $P \in \admpar$, the set $\RBC_P(G,\X)$ is a singleton. In particular, $\X_P $ is equal to $\X_\Y$ for the unique element $(P,\Y) \in \RBC_P (G,\X)$, and we have a Shimura datum $(G_P, \X_P)$; cf.~\S \ref{subsubsec:taking pure quotient}.  
\subsubsection*{\textbf{A2}}  For each $P\in \admpar$, the action of $L_P(\RR)$ on $\X_P$ (see Proposition \ref{prop:action}) is trivial unless $\X_P$ is zero-dimensional. 
\subsubsection*{\textbf{A3}} For each $P\in \admpar$, let $L_P(\QQ)^{\natural}: = \mathrm{Cent}_{L_P(\QQ)} \X_P$. For each neat compact open subgroup $K_M$ of $M_P(\adele_f)$, we have $K_M \cap \mathrm{Cent}_{M_P(\QQ)} \X_P = K_M \cap L_P(\QQ)^{\natural}$. 

\begin{rem} Our axiom \textbf{A0} is slightly more restrictive than the first two conditions on p.~2 of \cite{morel2010book}, where $G_P$ is allowed to be different from the image of $P^{\Pink}$. Assuming \textbf{A0}, our axiom \textbf{A1} is equivalent to the first part of the fourth condition in \textit{loc.~cit.}, and our axiom \textbf{A2} is weaker than the second part of that condition. Our axiom \textbf{A3} is identical to the fifth condition in  \textit{loc.~cit.}. We have deleted the third condition in \textit{loc.~cit.~}from the axioms as a general correction. Indeed, this condition is neither used in \cite{morel2010book} nor satisfied by any of the Shimura data considered in \cite{morel2010book}, \cite{morel2011suite}, or the present paper. 
\end{rem}

\subsection{} From now on we assume the axioms in \S \ref{subsubsec:axioms}. Let $P\in \admpar$ and $g\in G(\adele_f)$. On p.~2 of \cite{morel2010book}, Morel defines the groups $H_P, H_L, K_Q, K_N$ associated to the pair $(P,g)$. We define: \index{$H_P$} \index{$H_L^{\natural}$} \index{$K_Q$} \index{$K_N$}
\begin{align*}
H_P &: = gKg^{-1} \cap P(\QQ) P^{\Pink}(\adele_f), \\
H_L^\natural & : = gKg^{-1} \cap L_P(\QQ)^\natural N_P(\adele_f),  \\ 
K_Q &: =  gKg^{-1} \cap P^{\Pink} (\adele_f), \\
K_N &: = gKg^{-1} \cap N_P(\adele_f). 
\end{align*}
Our $H_P, K_Q, K_N$ are the same as Morel's definitions, and our $H_L^{\natural}$ is equal to Morel's $H_L$\index{$H_L$} (defined to be $gKg^{-1} \cap L_P(\QQ) N_P(\adele_f)$) when $L_P(\QQ) = L_P(\QQ)^{\natural}$ (which is always true under Morel's axioms). In general, $H_L^{\natural}$ may be different from $H_L$, and $H_L^{\natural}$ is the correct replacement of $H_L$ in the discussion on the structure of the boundary strata in \cite[Chap.~1]{morel2010book}. The point is that under the axioms in \S \ref{subsubsec:axioms}, the group $H_L^\natural$ is always equal to Pink's group $H_C$ in \cite[\S 3.7]{pink1992ladic}, which has a canonical definition. More precisely, as on p.~2 of \cite{morel2010book}, the boundary stratum in the Baily--Borel compactification corresponding to $(P,g)$ is of the form\footnote{We systematically replace Morel's notation $M^{K_Q/K_N} (G_P, \X_P)$ for the Shimura variety by the notation $\Sh_{K_Q/K_N} (G_P, \X_P)$.}
\begin{align}\label{eq:form of bdry strata}
\Sh_{K_Q/K_N} (G_P,\X_P) / H_P,
\end{align}
and the action of $H_P$ factors through the finite quotient group $H_P / H_L^\natural K_{Q}$ (instead of $H_P/H_L K_Q$).

In Table \ref{table:1} below we compare Pink's notation in \cite[\S 3.7]{pink1992ladic}, Morel's notation in \cite[p.~2]{morel2010book}, and our notation. The symbols in the first column all have canonical definitions, independent of the axioms in \cite{morel2010book} or \S \ref{subsubsec:axioms}. Under the axioms in \S \ref{subsubsec:axioms}, the three symbols in every row denote the same object, with the only exception that Morel's $H_L$ is not equal to Pink's $H_C$ in general.
\begin{table}[H]   \renewcommand\thetable{1}
	\centering\caption{Comparison of notations}\label{table:1}
	\begin{tabular}{|c|c|c|}
		\hline
		\text{Pink's notation} & \text{Morel's notation}  & \text{Our notation}
				\\
				\hline
				$Q$ & $P$ & $P$ \\
				\hline 
				$P_1$ & $Q_P$ & $P^{\Pink}$ \\
				\hline
				$W_1$ & $N_P$ & $N_P$ \\
				\hline
				$G_1$ & $G_P$ & $G_P$ \\
				\hline 
				$\X_1$ & $\X_P$ & $\X_P$ \\ 
				\hline 
				$\X_Q$ & $\X_P$ & $\X_P$ \\
				\hline 
				$\mathrm{Stab}_{Q(\QQ)}\X_1$ & $P(\QQ)$ & $P(\QQ)$ \\
				\hline
				$H_Q$ & $H_P$ & $H_P$ \\
				\hline 
				$H_C$ & $H_L \quad \bendsymbol$    & $H_L^{\natural}$ \\
				\hline 
				$K_W$ & $K_N$ & $K_N$ \\
				\hline 
				$K_P$ & $K_Q$ & $K_Q$ \\
				\hline
				$\pi_1(K_P) = K_P/K_W$ & $K_Q/K_N$ & $K_Q/K_N$ \\
				\hline
			\end{tabular}
	
	\end{table}
{ $\bendsymbol:  \quad H_L \neq H_L^{\natural}$ unless $L_P(\QQ) = L_P(\QQ)^{\natural}$.}
\FloatBarrier
\subsection{} \label{subsubsec:simplifying assumptions}
We make the following crucial assumption \textbf{CA}, in addition to the axioms in \S \ref{subsubsec:axioms}. 
\subsubsection*{\textbf{CA}} If $P \in \admpar$ is such that $\X_P$ is zero-dimensional, then the Shimura datum $(G_P,\X_P)$ is the Siegel Shimura datum $(\Gm, \HZ)$. For such $P$ we also assume that $L_P$ satisfies the assumptions in \S \ref{subsubsec:group setting for GKM variant}. Namely, we assume that $\pi_0(L_P(\RR)) \cong \pi_0(L_0(\RR)) \cong \ZZ/2\ZZ$, where $L_0$ is any minimal Levi subgroup of $L_{P,\RR}$. Moreover, for such $P$ we assume that $\pi_0(L_P(\RR))$ acts non-trivially on $\HZ$. In particular, we have $L_P(\QQ)^{\natural} = L_P(\QQ)^+$. 
\subsection{}\label{subsubsec:relation} Under \textbf{CA}, we know that for any $P\in\admpar$ such that $\X_P$ is zero-dimensional, and for any $g\in G(\adele_f)$, the boundary stratum (\ref{eq:form of bdry strata}) corresponding to $(P,g)$ is related to the generalized Shimura varieties in \S \ref{subsubsec:generalized SV} in the following way. In \S \ref{subsubsec:setting for zero dim SV}, we identify $\Gm$ with $G_P$, and take $L = L_P$, $M = M_P$. Let $K_M$ be the image of $gKg^{-1}\cap P(\adele_f)$ under the projection $P(\adele_f)\to M(\adele_f)$, and define $\bar H, \bar H_L^{\natural}, \chH$ as in \S \ref{subsubsec:setting for zero dim SV}. Then $\bar H$ (resp.~$\bar H_L^\natural$) is the image of $H_P$ (resp.~$H_L^{\natural}$) under $P(\adele_f)\to M(\adele_f)$, and (\ref{eq:form of bdry strata}) is the same as $\Sh_{\chH}$ defined in \S \ref{subsubsec:defn of Sh_H}.
	 
	 \subsection{}\label{subsubsec:verify axioms}Our orthogonal Shimura datum 	 $\mathbf O(V)$ satisfies \textbf{A0}--\textbf{A3} in \S \ref{subsubsec:axioms}, and \textbf{CA} in \S \ref{subsubsec:simplifying assumptions}. Indeed, it suffices to verify these conditions for the standard maximal proper parabolic subgroups $P_i, ~ i =1,2$. We take $L_{P_i}$ to be $M_{i,l}$. Then the desired conditions follow from Proposition \ref{prop:rational bdry comp} and Lemma \ref{lem:explaining the H groups} (1).
	 
\section{Integral models}\label{subsec:int mod}
\subsection{}\label{subsubsec:int mod 1}
	We now turn to construct the integral models of the Baily--Borel compactification $\overline{\Sh_K}$ and its strata. For this let us specialize to the orthogonal Shimura datum $(G,\X) = \mathbf O(V)$. Recall that the standard maximal proper parabolic subgroups of $G$ are $P_1$ and $P_2$. We write $(G_i,\X_i)$ for the Shimura datum $(G_{P_i} = M_{i,h}, \X_{P_i})$ for $i\in \set{1,2}$, and write $(G_0,\X_0)$ for $(G,\X)$. (Our numbering of the $P_i$ and $G_i$ is the same as the abstract numbering in \cite[\S 1.1]{morel2010book}.) For $i\in \set{1,2}$, we set $L_{P_i}$ to be $M_{i,l}$. In accordance with \textit{loc.~cit.}, we define $L_{P_{12}}$ to be $ M_{12,l}$, so that $M_{12}$ is the direct product of $G_2$ and $L_{P_{12}}$. 
	
	Without loss of generality, we assume that the function $f^{\infty}$ in Theorem \ref{geometric assertion} is of the form $1_{KgK}/\vol(K\cap gKg^{-1})$ for some fixed $g\in G(\adele_f)$. Since $\mathbf O(V)$ is of abelian type, we  can apply \cite[Prop.~1.3.4]{morel2010book} to construct the following objects:
	\begin{itemize}
		\item a finite set $\Sigma$ of prime numbers  containing $\Sigma_0$ (where $\Sigma_0$ is as in \S \ref{subsubsec:setting for Morel's formula}).
		\item a set $\mathcal K_i$\index{$\mathcal K_i$ for $i\in \set{0,1,2}$} of neat compact open subgroups of $G_i(\adele_f)$ for $i\in \set{0,1}$ such that $K$ and $ K\cap gKg^{-1}$ are elements of $\mathcal K_0$.
		\item a set $\mathcal K_2$ of admissible levels, in the sense of Definition \ref{defn:neat in the broad sense}.
		\item a subset $A_i$ of $G_i(\adele_f)$ for $i\in \set{0,1,2}$ such that $1$ and $g$ are elements of $A_0$. 
		\item a smooth quasi-projective scheme $\Shh_{U}(G_i, \X_i)$ over $\ZZ[1/\Sigma]$ with generic fiber $\Sh_U(G_i,\X_i)$, for each $i\in \set{0,1,2}$ and each $U\in \mathcal K_i$. Here when $i=2$, the Shimura variety $\Sh_U(G_2,\X_2)$ at the admissible level $U$ is understood as in \S \ref{subsubsec:generalized SV}. 
		\item a normal projective scheme $\overline {\Shh_{U}}(G_i, \X_i)$ over $\ZZ[1/\Sigma]$ containing $\Shh_{U}(G_i, \X_i)$ as a dense open subscheme, whose generic fiber is the Baily--Borel compactification $\overline{\Sh_U}(G_i,\X_i)$ of $\Sh_U(G_i,\X_i)$, for each $i\in \set{0,1}$ and each $U\in \mathcal K_i$.
	\end{itemize}
These objects should satisfy all the requirements in \cite[Prop.~1.3.4]{morel2010book} and the paragraph following it. To be more precise, the formulations of these requirements need to be suitably modified when they concern zero-dimensional boundary strata. In the above, we have already modified the formulation of \cite[Prop.~1.3.4]{morel2010book} when it concerns $\mathcal K_2$, i.e., our $\mathcal K_2$ is a set of admissible levels, which are more general than neat compact open subgroups of $G_2(\adele_f) = \Gm(\adele_f)$. The conditions (a), (b), and (1)--(7) in \cite[\S 1.3]{morel2010book} also need to be modified as follows. 
\begin{itemize}
	\item  In condition (a), if $j=2$, we need to replace $L_{P'}(\QQ)$ with $L_{P'}(\QQ)^{+}$. (Here $P'$ is either $P_2$ or $P_{12}$, and $L_{P'}(\QQ)^+$ is the same as $L_{P'}(\QQ) \cap L_{P_2}(\QQ)^{\natural}$.) After this replacement, the quotient group in question is naturally a subgroup of $G_2(\adele_f)\times \ZZ/2\ZZ = \Gm(\adele_f)\times \ZZ/2\ZZ$, and the requirement is that this subgroup should be a member of $\mathcal K_2$.
	\item As in the paragraph following \cite[Prop.~1.3.4]{morel2010book}, we may and shall assume that the $\mathcal K_i$ are minimal in the following sense. We assume that $\mathcal K_0$ is the union of the $G(\adele_f)$-conjugacy class of $K$ and that of $K \cap gK g^{-1}$. Then we determine $\mathcal K_1$ as the minimal set that is stable under $G_1(\adele_f)$-conjugacy and such that condition (a) is satisfied for $(i,j) = (0,1)$. Having determined $\mathcal K_0$ and $\mathcal K_1$, we determine $\mathcal K_2$ as the minimal set such that the modified version of condition (a) as above is satisfied for $(i,j) \in \set{(0,2), (1,2)}$. In particular, $\mathcal K_1$ is finite modulo $G_1(\adele_f)$-conjugacy, and $\mathcal K_2$ is finite.
	\item In condition (b), if $j=2$, we still keep $L_P(\QQ)$, and do \emph{not} replace it with $L_P(\QQ)^{\natural}$.
	\item In conditions (3) and (4), if $i<2$, then the relevant requirements about zero-dimensional boundary strata should be reinterpreted in the obvious way, taking into account that in the generic fiber these strata are given by the generalized Shimura varieties $\Sh_U(\Gm, \HZ)$ at admissible levels $U$; cf.~\S \ref{subsubsec:relation}.   
	\item In conditions (5)--(7), for $i=2$ and $U\in \mathcal K_2$, the sheaves on the integral model $\Shh_U(G_2,\X_2)$ in question should be extensions of those sheaves on the generic fiber $\Sh_U(G_2,\X_2)$ that are constructed by the functors (\ref{eq:complicated functor}), (\ref{eq:compact supp 1}), and (\ref{eq:functor 1 for K flat}). (Indeed, by the minimality of $\mathcal K_2$ assumed above, each  $U\in \mathcal K_2$ is of the form either $\bar H/\bar H_L^{\natural}$ or $K_{\da}$, for a suitable choice of $L$ and $K_M$ as in \S\ref{subsubsec:setting for zero dim SV}; cf.~\S \ref{subsubsec:relation}.)
\end{itemize}
With the above modifications, the same proof of \cite[Prop.~1.3.4]{morel2010book} still goes through.

\begin{rem}
	The construction in \S \ref{subsubsec:int mod 1} can be easily generalized to an arbitrary abelian-type Shimura datum satisfying \textbf{A0}--\textbf{A3} in \S \ref{subsubsec:axioms} and \textbf{CA} in \S \ref{subsubsec:simplifying assumptions}. 
\end{rem}

Next we would like to to compare the localizations of the integral models constructed in \S \ref{subsubsec:int mod 1} with other known integral models, at least at almost all primes. We need some preparations. 
\begin{defn}\label{defn:language of integral models}
	Let $S$ be a scheme of finite type over $\QQ$. \begin{enumerate}
			\item By a \emph{family of local integral models}\index[n]{family of local integral models} of $S$, we mean the choice of an integral model $\mathcal S_p$ of $S$ over $\ZZ_p$ (i.e.~a $\ZZ_p$-scheme with generic fiber $S\otimes_{\QQ} \QQ_p$) for almost all primes $p$. Two such families $(\mathcal S_p)_{p\gg 0}$ and $(\mathcal S_p')_{p\gg 0}$ are called \emph{equivalent}, if for almost all $p$ there exists a $\ZZ_p$-isomorphism $\mathcal S_p \isom \mathcal S_p'$ extending the identity on the generic fiber. 
		\item Given a finite-type $\ZZ$-scheme $\mathcal S$ with generic fiber $S$, we obtain a family of local integral models $(\mathcal S \otimes_{\ZZ} \ZZ_p)_{p \gg 0}$ of $S$. Any family of local integral models equivalent to such a family is called \emph{eventually globalizable}.\index[n]{eventually globalizable}
	\end{enumerate}
\end{defn}
\begin{rem}\label{rem:spread out}
	By the ``spreading out'' property of isomorphisms (see \cite[Thm.~(8.10.5) (i)]{EGAIV3} or \cite[Thm.~3.2.1]{poonen}), the eventually globalizable condition characterizes the family of local integral models up to equivalence. 
\end{rem} 
\begin{lem}\label{lem:descent} Let $R$ be an integral domain, with fraction field $F$. Let $\Y$ be a scheme flat and locally of finite presentation over $R$. Let $X$ be a scheme over $F$, and let $\pi: \Y \otimes_R F \to X$ be an $F$-morphism. Then there exists at most one separated $R$-scheme $\X$ with generic fiber $X$ such that $\pi$ extends to an fppf $R$-morphism $\pi_0: \Y \to \X$. 
\end{lem}
\begin{proof}
	Let $\X$ and $\X'$ be two separated $R$-schemes with generic fiber $X$, together with fppf $R$-morphisms $\pi_0 : \Y \to \X$ and $\pi_0': \Y \to \X'$ extending $\pi$. 
	We claim that $\pi_0'$ factors uniquely through $\pi_0$. The lemma follows from the claim by symmetry.
	
	To prove the claim, form the fiber product $\Y\times_{\X} \Y$ with respect to $\pi_0: \Y \to \X$. Since $\pi_0$ is an fpqc covering and therefore a universal effective epimorphism, it suffices to check the equality of the two morphisms $$g_i: \Y\times_{\X} \Y \xrightarrow{\mathrm{pr}_i} \Y \xrightarrow{\pi_0'} \X', \qquad i= 1,2 .$$ Since both $\pi_0$ and the structure morphism $\Y\to \spec R$ are flat and locally of finite presentation, the same holds for the structure morphism $\Y\times_{\X} \Y \to \spec R$, which implies that it is open. Hence the generic fiber of $\Y \times_{\X} \Y$ is dense in $\Y \times_{\X} \Y$. Since the $R$-morphisms $g_1$ and $g_2$ agree on the dense generic fiber, and since the target $\X'$ is separated over $R$ (which implies that the locus where $g_1 = g_2$ is closed), we conclude that $g_1 = g_2$ on a closed subscheme of $\mathcal Y \times _{\mathcal X} \mathcal Y$ whose underlying topological space is that of $\mathcal Y \times _{\mathcal X} \mathcal Y$. In particular $g_1$ and $g_2$ induce the same map at the level of topolgical spaces. To finish the proof, we can reduce to the affine case, namely we can replace $\mathcal X'$ by an affine $R$-scheme $\spec A$, and replace $\mathcal Y \times _{\mathcal X} \mathcal Y$ by an affine $R$-scheme $\spec B$ flat over $R$. We know that $g_1, g_2: A \to B$ induce the same map $A \otimes_R F \to B \otimes _R F$. Hence we can conclude that $g_1 =g_2$ since $B \to B\otimes_R F$ is injective. 
\end{proof}
\subsection{}\label{subsubsec:preparation for enlarging Sigma}  We keep the notation in \S \ref{subsubsec:int mod 1}. In the following, by ``enlarging $\Sigma$'' we always mean replacing $\Sigma$ by a finite set of primes containing $\Sigma$. Also, when we write $p\notin \Sigma$ it is understood that $p$ is a prime. 

Let $(\Gspin(V),\X')$ be the GSpin Shimura datum\index[n]{GSpin Shimura datum} associated to the quadratic space $V$, which is of Hodge type and has reflex field $\QQ$. The natural homomorphism $\Gspin(V) \to G$ extends to a morphism $(\Gspin(V) , \X') \to (G,\X)$ of Shimura data, inducing an isomorphism between the adjoint Shimura data. For more details see \cite[\S 3]{MPspin}.

 We fix a neat compact open subgroup $\tilde K \subset \Gspin(V)(\adele_f)$\index{$\tilde K$} such that its image in $G(\adele_f)$ is contained in $K$. We denote by $\Sh_{\tilde K}$\index{$\Sh_{\tilde K}$} the canonical model over $\QQ$ of the Shimura variety associated to $(\Gspin(V),\X')$ at level $\tilde K$, and denote by $\overline {\Sh_{\tilde K}}$\index{$\overline {\Sh_{\tilde K}}$} the
 Baily--Borel compactification 
  over $\QQ$. Thus $\Sh_{\tilde K}$ is smooth quasi-projective over $\QQ$, and $\overline{\Sh_{\tilde K}}$ is normal projective over $\QQ$. There are natural $\QQ$-morphisms $\pi: \Sh_{\tilde K} \to \Sh_K$ and $\bar \pi: \overline {\Sh_{\tilde K}} \to \overline {\Sh_K}$. 

 Note that $\pi$ is finite \'etale surjective. Indeed, by fpqc descent, it suffices to check these properties for the base change of $\pi$ to $\CC$, which is clear from the adelic description of the Shimura varieties over $\CC$ and Hilbert 90 applied to $\ker(\Gspin(V) \to G) = \Gm$; cf.~\cite[\S 3.2]{MPspin}.
 
 Recall that $ K \in  \mathcal K_0$. We let $\Shh_K = \Shh_K(G, \X)$ be the smooth quasi-projective scheme over $\ZZ[1/\Sigma]$ with generic fiber $\Sh_K$ as given in \S \ref{subsubsec:int mod 1}. 
By standard ``spreading out'' (see \cite[Thm.~3.2.1]{poonen}), we may and shall assume that the following objects exist after enlarging $\Sigma$: 
\begin{itemize}
	\item a smooth quasi-projective scheme $\Shh_{\tilde K}$ over $\ZZ[1/\Sigma]$ with generic fiber $\Sh_{\tilde K}$. 
	\item a normal projective scheme $\overline{\Shh_{\tilde K}}$ over $\ZZ[1/\Sigma]$ with generic fiber $\overline {\Sh_{\tilde K}}$. 
	\item a dense open embedding $\Shh_{\tilde K} \hookrightarrow \overline{\Shh_{\tilde K}}$ extending the embedding $\Sh_{\tilde K} \hookrightarrow \overline{\Sh_{\tilde K}}$. 
	\item a finite \'etale surjective morphism $\pi_0 : \Shh_{\tilde K} \to \Shh_K$ extending $\pi$.
\end{itemize}
We also enlarge $\Sigma$ so that the following condition holds: 
\begin{itemize}
	\item For each $p\notin \Sigma$, there are reductive group schemes $\tilde {\mathcal {G}}_p$ and $\mathcal G_p$ over $\ZZ_p$ with generic fibers $\Gspin(V)_{\QQ_p}$ and $G_{\QQ_p}$ respectively such that the homomorphism $\Gspin(V)_{\QQ_p} \to G_{\QQ_p}$ extends to a homomorphism $\tilde {\mathcal {G}}_p \to \mathcal G_p$. Moreover, we have $\tilde K = \tilde {\mathcal {G}}_p(\ZZ_p) \tilde K^p$ and $K = \mathcal G_p (\ZZ_p) K^p$ for some compact open subgroups $\tilde K^p \subset \Gspin(V)(\adele_f^p)$ and $K^p \subset G(\adele_f^p)$.
\end{itemize}
\begin{lem}\label{lem:enlarge}
	In the setting of \S \ref{subsubsec:preparation for enlarging Sigma}, it is possible to further enlarge $\Sigma$ and find a number field $F$ unramified outside $\Sigma$ such that the following conditions hold for all $p \notin \Sigma$. Here all isomorphisms between integral models are required to extend the identity on the generic fiber. 
	\begin{enumerate}
		\item For each $U \in \mathcal K_2$, $\Shh_U(G_2,\X_2)\otimes{\ZZ_p}$ is isomorphic to the base change to $\ZZ_p$ of the canonical integral model of $\Sh_U(G_2,\X_2)$ in \S \ref{subsubsec:defn of Sh_H}.
		\item For each  $U \in \mathcal K_1$,  $\Shh_U(G_1, \X_1)\otimes {\ZZ_p}$ is isomorphic to the canonical hyperspecial integral model over $\ZZ_p$ of the modular curve $\Sh_U(G_1, \X_1)$. 
		
		\item The integral model $\Shh_{\tilde K} \otimes \ZZ_p$ (resp.~$\overline {\Shh_{\tilde K}} \otimes \ZZ_p$) is isomorphic to the canonical hyperspecial integral model $\KS{\tilde K}$ (resp.~$\MP$) over $\ZZ_p$ constructed in \cite{kisin2010integral} (resp.~in \cite{peratoroidal}).
		\item The integral model $\Shh_{K}\otimes \ZZ_p$ is isomorphic to the canonical hyperspecial integral model $\KS{K}$ over $\ZZ_p$ constructed in \cite{kisin2010integral}.
		\item For each place $v$ of $F$ above $p$,  $\overline{\Shh_K} \otimes _{\ZZ} \oo_{F,v}$ is isomorphic to the base change to $\oo_{F,v}$ of the integral model over $\ZZ_p$ of $\overline{\Sh_K}$ constructed in \cite[Prop.~2.4]{lanstrohII}.
		\end{enumerate}
\end{lem}
\begin{proof}
	First note that for (1) and (2) it suffices to show that we can enlarge $\Sigma$ for each $U$ separately, since $\mathcal K_1$ is finite modulo $G_1(\adele_f)$-conjugacy and $\mathcal K_2$ is finite. (Indeed, as is implicit in the proof of \cite[Prop.~1.3.4]{morel2010book}, the integral models at conjugate levels are by construction isomorphic to each other.)
	
	For (1)--(3), we know that the canonical integral models in each case form an eventually globalizable family of local integral models (Definition \ref{defn:language of integral models}) as $p$ varies. We are done by Remark \ref{rem:spread out}.   
	\ignore{The case with $\Sh_{\tilde K}$ is obvious from the construction, and we explain the case with $\overline{\Sh_{\tilde K}}$.
		In \cite{peratoroidal}, integral models over $\ZZ_{(p)}$ of toroidal compactifications of $\Sh_{\tilde K}$ are constructed via taking normalizations of the closures of $\Sh_{\tilde K}$ inside the integral models over $\ZZ_{(p)}$ of suitable toroidal compactifications of Siegel moduli spaces. Since the last integral models form a globalizable family when $p$ varies (see \cite{faltingschai} or \cite{lanbook}), we know that the toroidal compactifications over $\ZZ_{(p)}$ in \cite{peratoroidal} form a globalizable family when $p$ varies. Now hyperspecial canonical integral models over $\ZZ_{(p)}$ of $\overline{\Sh_{\tilde K}}$ are obtained in \cite{peratoroidal} via a Proj construction applied to the Hodge bundle on the above-mentioned integral model of a toroidal compactification. For different $p$, the Hodge bundles are integral models of the same vector bundle on the generic fiber of the toroidal compactification. Hence the hyperspecial canonical integral models of $\overline{\Sh_{\tilde K}}$ indeed form a globalizable co-finite set of models.} 
	
	For (4), we would like to apply Lemma \ref{lem:descent} to characterize $\KS{K}$ in terms of $\KS{\tilde K}$. Let $p\notin \Sigma$. By the construction in \cite{kisin2010integral} (cf.~\cite[Prop.~2.4, Remark 2.6]{lanstrohII}) and by the surjectivity of $\pi$,  the morphism $\pi_{\QQ_p}: \Sh_{\tilde K}\otimes_{\QQ} \QQ_p \to \Sh_K \otimes_{\QQ} \QQ_p$ extends to a finite \'etale surjective (hence fppf) morphism $\KS{\tilde K} \to \KS{K}$. We also know that $\KS{\tilde K}$ is flat of finite presentation over $\ZZ_p$. By \cite[Prop.~2.4]{lanstrohII}, $\KS{K}$ is quasi-projective and hence separated over $\ZZ_p$. By part (3), we may assume that $\KS{\tilde K} = \Shh_{\tilde K} \otimes_{\ZZ} \ZZ_p$. Since $\Shh_{\tilde K} \otimes_{\ZZ} \ZZ_p \to \Shh_{K}\otimes_{\ZZ} \ZZ_p$ is also finite \'etale surjective and since $\Shh_K\otimes_{\ZZ} \ZZ_p$ is also separated over $\ZZ_p$ (as it is quasi-projective), we know from Lemma \ref{lem:descent} that $\Shh_K \otimes_{\ZZ} \ZZ_p$ is isomorphic to $\KS{K}$ as integral models of $\Sh_K$. 
		
For (5), we let $(C_{i,\mathrm{geom}})_{i\in I}$ be the connected components of $\overline {\Sh_{\tilde K}} \otimes_{\QQ} \overline \QQ$, and let $(D_{j,\mathrm{geom}})_{j\in J}$ be the connected components of $\overline {\Sh_{K}} \otimes_{\QQ} \overline \QQ$. For each $i\in I$, let $C^0_{i,\mathrm{geom}}$ be the intersection of $C_{i,\mathrm{geom}}$ with $\Sh_{\tilde K} \otimes_{\QQ} \overline \QQ$, and similarly define $D^0_{j,\mathrm{geom}}$. The morphism $\bar \pi: \overline {\Sh_{\tilde K}} \to \overline {\Sh_K}$ induces a surjection $I \to J$, which we still denote by $\bar \pi$. As in the proof of \cite[Prop.~2.4]{lanstrohII}, we know that for each $i\in I$, the morphism $C_{i,\mathrm{geom}} \to D_{\bar \pi(i), \mathrm{geom}}$ induced by $\pi$ is the quotient by a finite group $\Delta_i$ acting on $C_{i,\mathrm{geom}}$, in the sense of \cite[Rmk.~2.6]{lanstrohII}. Moreover, $\Delta_i$ acts freely on $C^0_{i,\mathrm{geom}}$ and the Galois \'etale cover $C^0_{i,\mathrm{geom}} \to D^0_{\bar \pi(i), \mathrm{geom}}$ is a $\Delta_i$-torsor. We pick a number field $F$ such that each $C_{i,\mathrm{geom}}$ is the base change of a connected component $C_i$ of $\overline{\Sh_{\tilde K}}\otimes_{\QQ} F$, and such that the action of $\Delta_i$ on $C_{i,\mathrm{geom}}$ descends to $C_i$. For each $i\in I$, define $D_i$ to be the quotient $C_i/\Delta_i$, in the sense of \cite[Rmk.~2.6]{lanstrohII}. We fix a section $\iota:J \to I$ of the surjection $ I \to J$. Then it is clear that $\overline{\Sh_K}\otimes_{\QQ} F$ can be identified with $\coprod_{j\in J}  D_{\iota(j)}$.  

Since our choice of $F$ is independent of $\Sigma$, we can enlarge $\Sigma$ such that $F$ is unramified outside $\Sigma$. After further enlarging $\Sigma$, we may and shall assume that each $C_i$ is contained in a unique connected component $\mathscr C_i$ of $\overline{\Shh_{\tilde K}}\otimes_{\ZZ} \oo_F$, and that the action of $\Delta_i$ on $C_i$ extends to $\mathscr C_i$. Since the formation of the quotient of a quasi-projective scheme by the action of a finite group commutes with flat base change, we know that the generic fiber of $\coprod_{j \in J} \mathscr D_{\iota(j)}$ is the same as $\coprod_{j\in J} D_{\iota(j)}$, which we have already identified with $\overline{\Sh_K}\otimes_{\QQ} F$. Thus $\coprod_{j \in J} \mathscr D_{\iota(j)}$ and $\overline {\Shh_K}\otimes_{\ZZ} \oo_F$ are two finite-type $\oo_F\otimes_{\ZZ} \ZZ[1/\Sigma]$-schemes with the common generic fiber, and we can hence enlarge $\Sigma$ to assume that they are $\oo_F$-isomorphic. It is then clear from parts (3) and (4) above, and the construction in the proof of \cite[Prop.~2.4]{lanstrohII}, that the condition in (5) holds for all $p\notin \Sigma$ and all places $v$ of $F$ above $p$. \end{proof}

 \section{Finish of the proof}\label{subsec:finish of proof}
 Essentially all arguments in \cite[Chap.~1]{morel2010book} can be easily modified to suit our new axiomatic setting (i.e.~\textbf{A0}--\textbf{A3} in \S \ref{subsubsec:axioms} plus \textbf{CA} in \S \ref{subsubsec:simplifying assumptions}). With each appearance of $H_L$ replaced by $H_L^{\natural}$, the results of \cite[\S 1.4, \S 1.5]{morel2010book} all carry over. More precisely, in \cite[Prop.~1.4.5]{morel2010book}, if the index $n_r$ corresponds to zero-dimensional boundary data, then we replace the functor $\mathcal F^{H/H_L} R\Gamma(H_L/K_N, - )$ with the functor (\ref{eq:complicated functor}) (applied to $M = M_P$, $\bar H=$ the image of $H$ under $P(\adele_f) \to M_P(\adele_f)$, and $\bar H_L^{\natural}=$ the intersection of $\bar H$ with $L_{n_r}(\QQ)^{+}$). We then modify \cite[Cor.~1.4.6]{morel2010book} correspondingly (by replacing the functor $\mathcal F^{H/H_L} R\Gamma_c(H_L/K_N,- )$ with the functor (\ref{eq:compact supp 1})), and modify the definitions of $L_{C_1}$ and $L_{C_2}$ on pp.~17--18 of \cite[\S 1.5]{morel2010book} correspondingly.

Let $\overline{\Shh_K}$ be the integral model constructed in \S \ref{subsubsec:int mod 1}. We now explain the modification of the proof of \cite[Thm.~1.7.1]{morel2010book}, applied to the special fiber of $\overline{\Shh_K}$ modulo a prime $p\notin \Sigma$, where $\Sigma$ is as in Lemma \ref{lem:enlarge}. We follow the notation in \textit{loc.~cit.}. Modifications are only needed when $n_r$ corresponds to zero-dimensional boundary data. In this case, the definitions of $v_h$ and $u_{h'}$ need to be modified in accordance with the modifications in \cite[Cor.~1.4.6, \S 1.5]{morel2010book} mentioned above. To get the relation between $\Tr(v_h)$ and $\Tr(u_{h'})$, we need to apply Proposition \ref{prop:variant of coh corr} in place of \cite[Prop.~1.7.2]{morel2010book}. Finally, in the calculation of $\Tr(v_h)$ on the bottom of \cite[p.~25]{morel2010book}, we apply Proposition \ref{prop:generalization of0-dim Shim} in place of \cite[Rmk.~1.6.4, Rmk.~1.6.5]{morel2010book}, and apply Proposition \ref{prop:generalization ofGKM} in place of \cite[Thm.~1.6.6]{morel2010book}. Note that Proposition \ref{prop:generalization of0-dim Shim} is applicable thanks to condition (1) in Lemma \ref{lem:enlarge}. Also, the fixed point formula of Kottwitz for the one-dimensional boundary strata is applicable thanks to condition (2) in Lemma \ref{lem:enlarge}.

	After calculating $\Tr(v_h)$, the same arguments as those on pp.~26--27 of \cite{morel2010book} lead to a modified version of \cite[Thm.~1.7.1]{morel2010book}, where the right hand side of the equality in that theorem is replaced by 
	\begin{align*}
		\Tr(\Frob_p^a \times f^{\infty}  dg^\infty \mid   \coh_c^* (\Shh_K\otimes_\ZZ \overline \FF_p,  \mathcal F^K \mathbb V)) +  \mathrm T_{P_1} + \mathrm T_{P_2} + \mathrm T_{P_{12}},
	\end{align*}
with the  terms $\mathrm T_{P_1}$, $\mathrm T_{P_2}$, and $\mathrm T_{P_{12}}$  as in Definition \ref{defn:T_Q}.\footnote{Note that the factor $\mathtt m_M$ in Definition \ref{defn:T_Q} comes from the factor $2$ in Proposition \ref{prop:generalization ofGKM}, which is an analogue of \cite[Thm.~1.6.6]{morel2010book}. By contrast, in Morel's case the extra factor $2$ comes from \cite[Rmk.~1.6.5]{morel2010book}. } 
	From this, we deduce the analogue of the identity (\ref{eq:equivalent form of Morel's formula}) for the special fiber. Namely we have proved (\ref{eq:equivalent form of Morel's formula}) for $a$ sufficiently large, but with $\overline{\Sh_K}$ and $\Sh_K$ replaced by the mod $p$ reductions of the integral models. 
	
	To prove (\ref{eq:equivalent form of Morel's formula}) itself, we apply \cite[Thm.~4.19]{lanstrohII}. This result confirms Theorem \ref{geometric assertion} (1) and asserts that the terms 	$\Tr(\cdots\mid  \icoh^* (\overline{ \Sh _K}, \mathbb V)) $ and $\Tr(\cdots \mid   \coh_c^* ( \Sh _K,  \mathbb V)) $ in (\ref{eq:equivalent form of Morel's formula}) are unchanged if we replace $\overline{\Sh_K}$ and $\Sh_K$ by the mod $p$ reductions of the integral models.\footnote{In \cite[\S 3]{lanstrohII}, an extra assumption is made on the relation between the level $K$ and the prime $\ell$. This assumption can be easily removed if we consider the system of levels in \cite[\S 4.9]{pink1992ladic} instead of the system $\mathcal H(\ell^r), r>0$ in the notation of \cite[\S 3]{lanstrohII}.} Indeed, by \cite[Thm.~4.19]{lanstrohII} and conditions (4), (5) in Lemma \ref{lem:enlarge}, we know that the compact support cohomology and the intersection cohomology (with coefficients in $\mathbb V$) of $\Sh_{K,\overline \QQ_p}$ are respectively isomorphic to those of $\Shh_{K,\overline \FF_p}$ under the canonical adjunction morphisms (which are Hecke-equivariant and $\Gal(\overline \QQ_p/\QQ_p)$-equivariant). Note that in Lemma \ref{lem:enlarge} (5) we only compare the integral models over an extension of $\ZZ_p$, but this already suffices for the current purpose since whether the canonical adjunction morphisms are isomorphisms is insensitive to finite base change. This finishes the proof of Theorem \ref{geometric assertion} (1) and (\ref{eq:equivalent form of Morel's formula}). In Proposition \ref{prop:equivalent form}, we have already proved that  (\ref{eq:equivalent form of Morel's formula}) is equivalent to the identity (\ref{eq:in geometric assertion}) in Theorem \ref{geometric assertion} (2). 
	
	Finally, we explain why the two sides of (\ref{eq:in geometric assertion}) lie in $\mathbb E$ for all sufficiently large $a$. In the above proof of (\ref{eq:equivalent form of Morel's formula}), it is already implicit that the LHS of (\ref{eq:equivalent form of Morel's formula}) lies in the algebraic closure $\overline \QQ$ of $\QQ$ inside $\overline \QQ_{\ell}$, and that the equality holds when we view the LHS as a number in $\CC$ by choosing an arbitrary $\mathbb E$-algebra embedding $\overline \QQ \hookrightarrow \CC$. (Remember that at the outset we fixed field embeddings $\mathbb E_{\lambda}\hookrightarrow \overline \QQ_{\ell}$ and $\mathbb E \hookrightarrow \CC$, and that the RHS of (\ref{eq:equivalent form of Morel's formula}) is a number in $\CC$.) Since the definition of the RHS of (\ref{eq:equivalent form of Morel's formula}) depends only on the embedding $\mathbb E \hookrightarrow \CC$ but not on the choice of $\overline \QQ \hookrightarrow \CC$, we see that both sides of (\ref{eq:equivalent form of Morel's formula}) are in $\mathbb E$ since they must be fixed by $\Gal(\overline \QQ/\mathbb E)$. 
Thus it remains to check 	$$\Tr(\Frob_p^a \times f^{\infty}  dg^{\infty} \mid   \coh_c^* ( \Sh _K,  \mathbb V))  \in \mathbb E. $$ But this follows from the point counting formula in \cite{KSZ}.

 The proof of Theorem \ref{geometric assertion} is complete.

\chapter{Comparison with discrete series characters}\label{Section infty}
\section{Elliptic maximal tori in Levi subgroups}\label{R groups}
\subsection{} \label{para:R groups}
We now pass to a local setting over $\RR$. The symbols $V$, $V_i$, $W_i$, $G$, $P_S$, and $M_S$ will now denote the base change to $\RR$ of the corresponding objects in \S \ref{global groups}. We note that over $\RR$, $P_{12}$ is still a minimal parabolic subgroup of $G$, and $P_{12}, P_1, P_2$ are still the only proper parabolic subgroups of $G$ containing $P_{12}$. Also note that the split component of $M_S$ over $\RR$ is just the base change to $\RR$ of the split component over $\QQ$. For this reason we still use the notation $A_{M_S}$ for the split component over $\RR$. 

Note that $W_2$ and $W_1$ are quadratic spaces of signatures $(n-2,0)$ and $(n-1,1)$ respectively. We have \begin{align*}
M_1 & \cong \GL_2 \times \SO(W_2),& M_2 & \cong \GL_1 \times \SO(W_1),& M_{12} & \cong \Gm^2 \times \SO(W_2).&
\end{align*} Hence $M_{1}$ and $M_{12}$ always contain elliptic maximal tori (over $\RR$), whereas $M_2$ contains elliptic maximal tori if and only if $d$ is odd (recall that when $d$ is even we assume that $n = d-2 \geq 4$). We fix an elliptic maximal torus $T_{W_2}$\index{$T_{W_2}$} in $ \SO(W_2)$. We then obtain elliptic maximal tori:\index{$T_1$}\index{$T_{12}$}
\begin{align*}
T_1  & : = T_{\GL_2}^{\std} \times T_{W_2} \subset M_1 = \GL_2 \times \SO(W_2), \\ T_{12} & := \GG_m^2 \times T_{W_2} \subset M_{12} = \GG_m^2 \times \SO(W_2),
\end{align*}where \index{$T_{\GL_2}^{\std}$} 
$$T_{\GL_2}^{\std} = \set{\begin{pmatrix}
	a & b\\ c & d
	\end{pmatrix} \in \GL_2 \mid  a= d, b = -c}.$$ 

When $d$ is odd, we also fix an elliptic maximal torus $T_{W_1}$\index{$T_{W_1}$} in $\SO(W_1)$, and obtain an elliptic maximal torus $T_2 = \GG_m \times T_{W_1}$ in $M_2 = \GG_m \times \SO(W_1)$.\index{$T_2$}

 \subsection{}
We define a maximal torus $T'$\index{$T'$} in $G_{\CC}$ as follows. Remember that $V$ is the orthogonal direct sum of $\mathrm{span} \set{e_1, e_1'}$, $\mathrm{span} \set{e_2, e_2'}$, and $W_2$. We choose a hyperbolic basis  (see Definition \ref{defn:qsplitting})  $\mathbb B = \set{f_1,\cdots, f_d}$ of the quadratic space $V_{\CC}$ over $\CC$ such that \begin{align*} 
f_1 &= e_1,& f_2 & = e_2, & f_d & = e_1' , & f_{d-1} & = e_2'. &  
\end{align*} As in \S \ref{subsubsec:Borel pairs assoc to hd and nhd}, from $\mathbb B$ we obtain an embedding $$\iota_{\mathbb B} : \GmC^m \isom  T' \subset G_\CC, $$ and a Borel subgroup $B$ of $G_{\CC}$ containing $T'$. By construction, $T'$ is contained in $M_{\CC}$ for each $M \in \set{M_1, M_2,M_{12}}$, and $B$ is contained in $P_{\CC}$ for each $P \in \set{P_1,P_2,P_{12}}$. Moreover, $\iota_{\mathbb B}$ identifies the first two copies of $\GG_m$ with the split component $\Gm^2 = \GL(V_1) \times \GL(V_2/V_1)$ of $M_{12}$. 

Let $S$ be a non-empty subset of $\set{1,2}$ and assume that $S \neq \set{2}$ if $d$ is even. We fix an element $g_S \in M_S(\CC)$ such that $\Int(g_S)(T_{S,\CC})= T'$. Denote the standard characters of $\GG_{m,\CC}^m \cong T'$ by $\epsilon_1, \cdots, \epsilon_m$, and the standard cocharacters by $\epsilon_1^{\vee}, \cdots, \epsilon_m^{\vee}.$ We transport them to $T_{S,\CC}$ using $\Int(g_S)$, and retain the same notation. 

For $S$ as above, we let $R_{S}$ be the subset of $\Phi(G_{\CC},T_{S,\CC})$ consisting of real elements, and similarly we define $R_S^{\vee} \subset \Phi(G_{\CC}, T_{S,\CC}) ^{\vee}$.\index{$R_S, R_S^{\vee}$} We view $R_{S}$ and $R_S^{\vee}$ as subsets of $X^*(A_{M_S})$ and $X_*(A_{M_S})$ respectively. In Tables \ref{table:odd real root syst} and  \ref{table:even real root syst} below, we determine $R_S$ and $R_S^{\vee}$ explicitly in the odd and even cases respectively. In the last rows of the two tables, we record the type of the root datum $ (X^*(A_{M_S}) , R_S, X_*(A_{M_{S}} ), R_S^{\vee})$. 
\begin{table}[H]
	  \renewcommand\thetable{2}
	\centering\caption{Real root systems in the odd case}\label{table:odd real root syst}
	\begin{tabular}{|c|c|c|c|}
		\hline $S$ & $\set{1}$ & $\set{2}$ & $\set{1,2}$ \\ \hline
		$R_S$ &  $\set{\pm (\epsilon_1 +\epsilon_2)}$ & $\set{\epsilon_1}$ & $\set{\pm \epsilon_1, \pm \epsilon_2, \pm \epsilon_1 \pm \epsilon _2}$ \\
		\hline 
		$R_S^{\vee}$ & $\set{\pm (\epsilon_1^{\vee} + \epsilon_2 ^{\vee})}$ &  $\set{2\epsilon_1^{\vee}}$  &   $\set{\pm 2\epsilon_1^{\vee}, \pm 2\epsilon_2^{\vee}, \pm \epsilon_1 ^{\vee}\pm \epsilon_2^{\vee}}$ \\ \hline
		$X^*(A_{M_S})$ & $\frac{1}{2} \ZZ (\epsilon_1 +\epsilon_2)$ & $\ZZ \epsilon_1$ & $\ZZ \epsilon_1 \oplus \ZZ \epsilon_2$ \\ 
		\hline 
			$X_*(A_{M_S})$ & $\ZZ (\epsilon_1^{\vee} +\epsilon_2^{\vee})$ & $\ZZ \epsilon_1^{\vee}$ & $\ZZ \epsilon_1^{\vee} \oplus \ZZ \epsilon_2^{\vee}$ \\ \hline 
			type &  $\mathsf{A}_1$ & $\mathsf A_1$ & $\mathsf B_2$ \\ \hline
		\end{tabular}\end{table} 
	\begin{table}[H]   \renewcommand\thetable{3}
		\centering\caption{Real root systems in the even case}\label{table:even real root syst}
		\begin{tabular}{|c|c|c|}
			\hline $S$ & $\set{1}$ & $\set{1,2}$ \\ \hline
			$R_S$ &  $\set{\pm (\epsilon_1 +\epsilon_2)}$  & $\set{\pm \epsilon_1 \pm \epsilon _2}$ \\
			\hline 
			$R_S^{\vee}$ & $\set{\pm (\epsilon_1^{\vee} + \epsilon_2 ^{\vee})}$ &  $\set{\pm \epsilon_1 ^{\vee}\pm \epsilon_2^{\vee}}$ \\ \hline
			$X^*(A_{M_S})$ & $\frac{1}{2} \ZZ (\epsilon_1 +\epsilon_2)$ & $\ZZ \epsilon_1 \oplus \ZZ \epsilon_2$ \\ 
			\hline 
			$X_*(A_{M_S})$ & $\ZZ (\epsilon_1^{\vee} +\epsilon_2^{\vee})$ & $\ZZ \epsilon_1^{\vee} \oplus \ZZ \epsilon_2^{\vee}$ \\ \hline 
			type &  $\mathsf{A}_1$ & $\mathsf A_1 \times \mathsf A_1$ \\ \hline
	\end{tabular}\end{table}

\section{Stable discrete series characters}\label{subsec:stable discrete series chars}
\subsection{}\label{subsubsec:mathbb V}
We keep the setting of \S \ref{R groups}. Fix an irreducible algebraic representation $\mathbb V$\index{$\mathbb V$} of $G_{\CC}$. This gives rise to an L-packet $\Pi(\mathbb V)$\index{$\Pi(\mathbb V)$} of discrete series representations of $G(\RR)$. Let $\Theta = \Theta_{\mathbb V}$\index{$\Theta$}\index{$\Theta_{\mathbb V}$}
 be the stable character\index[n]{stable character} associated to $\Pi(\mathbb V)$, i.e., the sum\footnote{Our definition of the stable character is the same as \cite{morel2011suite}, whereas in \cite{GKM} a sign $(-1)^{q(G)}$ is included.} of the characters of the members of $\Pi(\mathbb V)$.

Let $S$ be a non-empty subset of $\set{1,2}$, and assume that $S \neq \set{2}$ if $d$ is even. Let $M : = M_S$. In \S \ref{R groups} we fixed an elliptic maximal torus $T_S$ in $M$. In the sequel, unless otherwise stated, we call an element $\gamma \in T_S(\RR)$ \emph{regular} if it is regular in $G$, i.e., if $\alpha (\gamma) \neq 1$ for all $\alpha \in \Phi(G_\CC, T_{S,\CC})$. 

The \emph{normalized stable discrete series character}\index[n]{normalized stable discrete series character} $\Phi^G_M (\cdot, \Theta)$\index{$\Phi^G_M(\cdot,\Theta)$} is defined and studied in \cite{arthurlef} and \cite{GKM}; see also \cite[\S 3.2]{morel2010book}. It is a continuous function $T_S(\RR) \to \CC$ such that $$\Phi^G_M(\gamma, \Theta) = \abs{D^G_M(\gamma)}^{1/2}_{\RR} \Theta(\gamma) $$ for all regular $\gamma\in T_S(\RR)$. In the following we recall a formula for $\Phi^G_M(\gamma, \Theta)$, for regular $\gamma \in T_S(\RR)$. 

\subsection{}\label{subsubsec:notation for roots and Weyl}
In \S \ref{R groups} we fixed a Borel pair $(T', B')$ in $G_{\CC}$, an elliptic maximal torus $T_S$ in $M$, and an element $g_S \in M(\CC)$ such that $\Int(g_S) (T_{S,\CC})  = T'$. We now denote by $B$ the Borel subgroup $\Int(g_S)^{-1} (B')$ of $G_{\CC}$ containing $T_{S,\CC}$. Remember that $P_{S,\CC} \supset B$. We let $B_M : = M_{\CC} \cap B$, which is a Borel subgroup of $M_{\CC}$. We make the following definitions: \begin{itemize}
	\item Denote by $\Phi^+$\index{$\Phi^+$} the set of $B$-positive roots in $\Phi(G_{\CC}, T_{S,\CC})$. 
	\item Denote by $\Phi_M^+$\index{$\Phi_M^+$ } the set of $B_M$-positive roots in $\Phi(M_{\CC}, T_{S,\CC})$.
	\item Denote by $\rho \in X^* (T_S)\otimes \frac{1}{2}\ZZ$\index{$\rho \in X^* (T_S)$} the half sum of the elements of $\Phi^+$. 
	\item Denote by $\lambda \in X^*(T_S)$\index{$\lambda \in X^*(T_S)$} the highest weight of the $G_{\CC}$-representation $\mathbb V$ with respect to the Borel pair $(T_{S,\CC}, B)$ in $G_{\CC}$.
	\item Denote by $\Omega$ the complex Weyl group $\Omega_{\CC} (G, T_{S})$.\index{$\Omega$}
	\item For $\omega \in \Omega$, denote by $\omega B$\index{$\omega B$} the Borel subgroup $\dot \omega B \dot\omega ^{-1}$ of $G_{\CC}$, where $\dot \omega \in \Nor _{G(\CC)} (T_S)$ is any representative of $\omega$. 
	\item Denote by $\Delta_M$\index{$\Delta_M$} the Weyl denominator of $M_{\CC}$ with respect to the Borel pair $(T_{S,\CC} , B_M)$ in $M_{\CC}$; see Definition \ref{defn:Weyl denominator}. Thus $\Delta_M = \prod_{\alpha \in \Phi_M^+}(1 - \alpha^{-1})$. 
	\item 
	For $\omega \in \Omega$, define\index{$\Phi(\Omega)$}\index{$l(\omega)$}\index{$\epsilon(\omega)$}
	\begin{align*}
	\Phi(\omega) & : =  \Phi^+ \cap ( - \omega \Phi^+ ), \\ l(\omega)  & := \abs{\Phi(\omega)} , \\ \epsilon(\omega) & : = (-1)^{l(\omega)}. 
	\end{align*}
Thus $l(\omega)$ and $\epsilon(\omega)$ are the length and sign of $\omega$ respectively. 
\end{itemize} 

Recall that in \S \ref{R groups} we explicitly identified the set $R_S$ of real roots in $\Phi(G_{\CC}, T_{S,\CC})$. Since $S$ is currently fixed, we simply write $R$ for $R_S$. For $\gamma\in T_S(\RR)$, we define \index{$R_{\gamma}$} \index{$R_{\gamma}^+$} \index{$\epsilon_R$}
\begin{align*}
R_{\gamma} & : = \set{\alpha \in R \mid  \alpha (\gamma) > 0} ,\\
R_{\gamma}^+ & : = \set{\alpha \in R \mid  \alpha(\gamma) > 1}, \\ \epsilon_R(\gamma) & : = (-1)^{\abs{\Phi^+ \cap ( - R_{\gamma}^+ ) }}  = (-1)^{\# \set{\alpha \in \Phi^+ \cap R \mid  0 < \alpha(\gamma)<1 } }.
\end{align*}
Then by the work of Harish-Chandra \cite{HCdiscreteseriesI} and Herb \cite{herb}, we have the following formula for $\Phi^G_M(\gamma, \Theta)$, for regular $\gamma\in T_S(\RR)$:
\begin{multline}\label{eq:to explain notation}
\Phi^G_M(\gamma, \Theta) = (-1)^{q(G)} \epsilon_{R} (\gamma) \delta_{P_S(\RR)} (\gamma)^{1/2}\Delta_M(\gamma)^{-1} \\ \cdot  \sum_{\omega \in \Omega } \epsilon(\omega) n(\gamma, \omega B) (\omega \lambda)(\gamma) \prod _{\alpha \in \Phi(\omega)} \alpha^{-1} (\gamma).  
\end{multline}
See also \cite[\S 4]{GKM} and \cite[Fait 3.1.6]{morel2011suite}. Here $q(G)$ and $\delta_{P_S(\RR)}$ are defined in \S \ref{subsec:general defns}, and $n(\gamma,\omega B)$ are certain integers, whose definition we now explain following \cite[\S 4]{GKM}. \index{$n(\gamma,\omega B)$}

Let $G^{\SC}$\index{$G^{\SC}$} be the simply connected cover of $G$, and write $\im(G^{\SC} (\RR))$ for the image of $G^{\SC}(\RR) \to G(\RR)$.  Firstly, if $\gamma \notin Z_G(\RR) \im(G^{\SC} (\RR))$, then  $n(\gamma,\omega B) = 0$ for all $\omega \in \Omega$.

\begin{rem}\label{rem:G(R)^0}
	In our case $Z_G(\RR) \im(G^{\SC} (\RR)) = G(\RR)^0$. In fact, since $G$ is semi-simple, we have $\im(G^{\SC} (\RR)) = G(\RR)^0$ by the connectedness of $G^{\SC } (\RR)$. Now in the odd case $Z_G$ is trivial, and in the even case $Z_G(\RR) = \set{\pm \id_V}$ is contained in $G(\RR)^0$ (see \cite[I.17]{knappbeyond}). 
\end{rem}
\subsection{}\label{subsubsec:to explain n}
We now assume that $\gamma \in T_S(\RR)$ is regular and lies in $Z_G(\RR) \im (G^{\SC}(\RR))$, and  explain the definition of $n(\gamma,\omega B)$ in this case. First we need some preparations. 

Let $E^*$ be a finite-dimensional $\RR$-vector space, and $U \subset E^*$ a root system. Let $E_*$ denote the dual vector space of $E^*$, and let $U^{\vee} \subset E_*$ be the set of coroots. Assume that $U$ spans $E^*$, and that the Weyl group of $U$ contains $-1 \in \GL(E_*)$. Let $E_{*,\mathrm{reg}} \subset E_*$ and  $E^*_{\mathrm{reg}} \subset E^*$ be the regular loci with respect to $U$ and $U^{\vee}$ respectively. One associates to the datum $(E^*, U)$ a function \index{$\bar c_U$}
\begin{align}\label{eq:bar c}
\bar c_{U} : E_{*,\mathrm{reg}} \times E^*_{\mathrm{reg}}  \To \ZZ. 
\end{align}
This function appeared in the work of Herb \cite{herb}, and can be inductively characterized by the properties (1)--(5) listed in \cite[\S 3]{GKM}. We will give explicit formulas for $\bar c_U$ in some special cases in Lemmas \ref{lem:rk1} and \ref{lem:rk2} below. Later in the paper (\S\S \ref{subsec:odd case vanishing} and \ref{subsec:even case vanishing}), we will recall Herb's close formula for $\bar c_U$ in more complicated situations. 

Now we write $X_*(A_M)_{\RR}$ and $X^*(A_M)_{\RR}$\index{$X_*(A_M)_{\RR}, X^*(A_M)_{\RR}$} for $X_*(A_M) \otimes_{\ZZ} \RR$ and $X^*(A_M) \otimes _{\ZZ} \RR$ respectively, and identify  $X_*(A_M)_{\RR}$ with $\Lie (A_M)$. We view the Weyl group of the root system $R_{\gamma}$ as a subgroup of $\GL(X_*(A_M)_{\RR})$. Let $X_*(A_M) _{\RR, \mathrm{reg}} \subset X_*(A_M)_{\RR}$ and $X^*(A_M) _{\RR, \mathrm{reg}} \subset X^*(A_M)_{\RR}$\index{$X_*(A_M) _{\RR, \mathrm{reg}}$, $X^*(A_M) _{\RR, \mathrm{reg}}$} be the regular loci, with respect to the root systems $R_{\gamma}$ and $R_{\gamma}^{\vee}$, respectively. 

\begin{lem}[{\cite[p.~499]{GKM}}]\label{lem:GKM}For regular $\gamma \in T_S(\RR)$ which lies in  $Z_G(\RR) \im (G^{\SC}(\RR))$, the Weyl group of $R_{\gamma}$ contains $-1 \in \GL(X_*(A_M)_{\RR})$. \qed
\end{lem} 
\subsection{} Keep the setting of \S \ref{subsubsec:to explain n}. In view of Lemma \ref{lem:GKM} and the general construction (\ref{eq:bar c}), we obtain a function\index{$\bar c_{R_\gamma} $}
\begin{align*}
\bar c_{R_\gamma} : X_*(A_M)_{\RR,\mathrm{reg}} \times X^*(A_M)_{\RR,\mathrm{reg}} \To \ZZ. 
\end{align*}
We can now define the integers $n(\gamma, \omega B)$ in terms of the function $\bar c_{R_\gamma}$. Let $T_S(\RR)_1$\index{$T_S(\RR)_1$} be the maximal compact subgroup of $T_S(\RR)$. We have a canonical decomposition $$T_S(\RR) = A_M(\RR) ^0 \times T_S(\RR)_1.$$ We write the projection of $\gamma\in T_S(\RR)$ in $A_M(\RR) ^0$ as $\exp(x_{\gamma})$, with $x_{\gamma}\in \Lie (A_M) = X_*(A_M)_{\RR}$.\index{$x_{\gamma}$} Our assumption that $\gamma$ is regular ensures that $$x_{\gamma} \in X_*(A_M)_{\RR, \mathrm{reg}}. $$ Let $\wp: X^*(T_S)_{\RR}\to X^*(A_M)_{\RR}$\index{$\wp$} be the natural restriction map. Then for any $\omega \in \Omega$ we have $$\wp(\omega \lambda +\omega \rho) \in X^*(A_M)_{\RR,\mathrm{reg}}. $$ Define 
\begin{align}\label{eq:defn of n(gamma B)}
n (\gamma ,\omega B) : = \bar c _{R_\gamma} (x_{\gamma}, \wp(\omega \lambda +\omega \rho)).
\end{align}
This finishes our explanation of (\ref{eq:to explain notation}).
\begin{cor}\label{cor:GKM}
	Let $\gamma \in T_S(\RR)$ be a regular element such that the Weyl group of $R_{\gamma}$ does not contain $-1 \in \GL(X_*(A_M)_{\RR})$. Then $\Phi^G_M(\gamma, \Theta) = 0$. 
\end{cor}
\begin{proof}
	By Lemma \ref{lem:GKM}, we have $\gamma \notin Z_G(\RR) \im(G^{\SC}) (\RR)$. Hence $n(\gamma, \omega B) = 0$ for all $\omega \in \Omega$, and we have $\Phi^G_M(\gamma,\Theta) =0$ by (\ref{eq:to explain notation}).
\end{proof}
In the sequel we will need explicit descriptions of the function $\bar c_{U}$ for certain root systems $U$ in $\RR^1$ and $\RR^2$. For $i\in \set{1,2}$, we use the standard inner product on $\RR^i$ to identify $\RR^i$ with its own dual space. 
\begin{lem}\label{lem:rk1}
	Let $\epsilon$ be the basis vector $1$ of $\RR^1$. The Weyl group of the root system $U = \set{\pm \epsilon}$ contains $-1$. The regular loci in $\RR^1$ with respect to $U$ and with respect to $U^{\vee}$ are both $\RR^1 -\set{0}$. The function $\bar c _U : (\RR^1 - \set{0}) \times( \RR^1 - \set{0}) \to \ZZ$ is given by: 
	$$ \bar c _U (x \epsilon,y \epsilon) = \begin{cases}
	2, & \text{if } xy < 0 , \\
	0, & \text{if } xy > 0.
	\end{cases}$$ 
\end{lem}
\begin{proof}
	This follows from a direct computation based on properties (1)--(5) listed in \cite[\S 3]{GKM}.
\end{proof}
\subsection{}\label{subsubsec:three cases of c}
We now consider certain root systems in $\RR^2$.  
Let $\set{\epsilon_1 = (1,0), \epsilon_2= (0,1)}$ be the natural basis of $\RR^2$, and let $x_1$ and $x_2$ be the two coordinate functions on $\RR^2$. Let\footnote{Here the subscript $\Endos$ stands for ``endoscopic''.} \index{$U_{\odd},U_{\Endos},U_{\even}$}
\begin{align*}
U_{\odd} & : = \set{\pm \epsilon_1 , \pm \epsilon_2, \pm \epsilon_1 \pm \epsilon_2}, \\ 
U_{\Endos} & := \set{\pm \epsilon_1, \pm \epsilon_2}\\
U_{\even} & := \set{\pm \epsilon_1 \pm \epsilon_2}.
\end{align*} For each subscript $? \in \set{\odd, \Endos, \even}$, $U_?$ is a root system in $\RR^2$. The regular locus in $\RR^2$ with respect to $U_?$ is equal to the regular locus with respect to $U_?^{\vee}$. We denote this locus by $\RR^2_?$.\index{$\RR^2_{\odd}, \RR^2_{\Endos}, \RR^2_{\even}$}

Explicitly, $\RR^2_{\odd}$ is the complement of the two coordinate axes and the two diagonal lines. Thus it is the disjoint union of eight open cones. We label the cone $\set{(x_1,x_2) \mid  0< x_2 <x _1}$ by the symbol $\RI$, and label the other cones counterclockwise, by $\RII$, $\RIII$, $\cdots$, $\RVIII$. See Figure \ref{fig1}.\index{$(\mathcal {I}), (\mathcal {II}), \cdots, (\mathcal {VIII})$}

Similarly, $\RR^2_{\Endos}$ is the complement of the two coordinate axes,  and $\RR^2_{\even}$ is the complement of the two diagonal lines $x_1= \pm x_2$. We label the four open cones constituting $\RR^2_{\even}$ counterclockwise, starting with the cone $\set{(x_1,x_2) \mid x_1 > \abs{x_2}}$, by the symbols $\RA, \RB, \RC, \RD$. See Figure \ref{fig2}.\index{$(\mathscr A), (\mathscr B), (\mathscr C), (\mathscr D)$}  

\begin{figure}[H] \centering
	\beginpgfgraphicnamed{eight-cones}
	\includegraphics{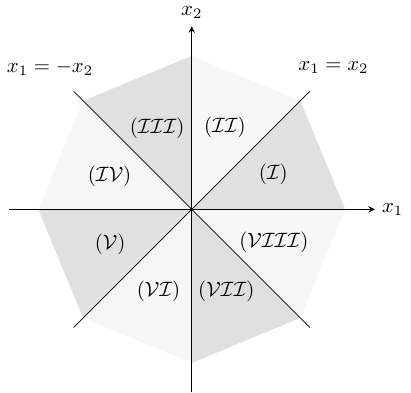}
	
	\caption{Labeling of the eight open cones complement to the two coordinate axes and the two diagonal lines in the $x_1$-$x_2$-plane. The union of the cones is denoted by $\RR^2_{\odd}$.}
	\label{fig1}
\end{figure}

\begin{figure}[H] \centering
	\beginpgfgraphicnamed{four-cones}
	\includegraphics{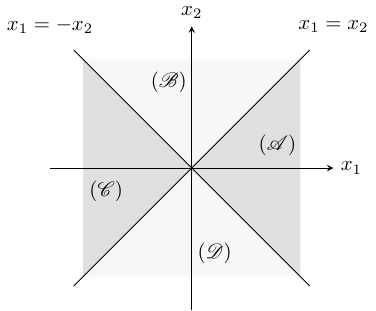}
	\caption{Labeling of the four open cones complement to the two diagonal lines in the $x_1$-$x_2$-plane. The union of the cones is denoted by $\RR^2_{\even}$.}
	\label{fig2}
\end{figure} 

We shall use the same symbols $\RI, \RII, \cdots, \RA, \RB, \cdots$, to denote the characteristic functions of the corresponding open cones. For each subscript $? \in \set{\odd, \Endos, \even}$, the Weyl group of $U_?$ contains $-1\in \GL(\RR^2)$. Hence we have the associated function $$\bar c_{U_?} : \RR^2 _? \times \RR^2 _? \To \ZZ.$$ The following lemma describes this function. For each fixed $x \in \RR^2_?$, we let $\mathbf f_{?,x} : \RR^2_? \to \ZZ$ be the function that sends $x'\in \RR^2_?$ to $\bar c_{U_?} (x,x')$. \index{$\mathbf f_{\odd,x}, \mathbf f_{\Endos,x}, \mathbf f_{\even,x}$}

\FloatBarrier

\begin{lem}\label{lem:rk2}The following statements hold. 
	\begin{enumerate}
		\item If $x\in \RV$, then 
		\begin{align}\label{eq:odd V}
		\frac{1}{4}	\mathbf f_{\odd,x}  = \RII + \RVIII. 
		\end{align}
		If $x\in \RIV$, then 
		\begin{align}\label{eq:odd IV}
		\frac{1}{4}	\mathbf f_{\odd,x} = \RI + \RVII. 
		\end{align}
		\item The function $\bar c_{U_{\Endos}}: \RR^2_{\Endos} \times \RR^2_{\Endos} \to \ZZ$ is given by 
		\begin{align}\label{eq:c_eds}
		\bar c_{U_{\Endos}} (x,x') = \begin{cases}
		4, & \text{if $x$ and $x'$ lie in opposite quadrants}, \\
		0, & \text{otherwise}. 
		\end{cases} \end{align}
		In particular, if $x\in \RV$, then 
		\begin{align}\label{eq:eds V}
		\frac{1}{4}	\mathbf f_{\Endos,x}|_{\RR^2_{\odd}}  = \RI + \RII  .
		\end{align}
		If $x\in \RIV$, then 
		\begin{align}\label{eq:eds IV}
		\frac{1}{4}	\mathbf f_{\Endos,x} |_{\RR^2_{\odd}}  = \RVII + \RVIII .
		\end{align}
		\item If $x \in \RC$, then 
		\begin{align}\label{eq:even C}
		\frac{1}{4} \mathbf f_{\even,x} = \RA. 
		\end{align}
	\end{enumerate} 
\end{lem}
\begin{proof}
	This follows from a direct computation based on properties (1)--(5) listed in \cite[\S 3]{GKM}.
\end{proof}
\begin{rem}
 The complete descriptions of $\bar c_{U_{\odd}}$ and $\bar c_{U_{\even}}$ follow immediately from Lemma \ref{lem:rk2} and the Weyl invariance of these functions (see property (5) in \cite[\S 3]{GKM}). 
\end{rem}
\section{Kostant's theorem}\label{subsec:Kostant's thm}
We apply Kostant's theorem\index[n]{Kostant's theorem} \cite{kostant} to compute the character of the virtual representation in Definition \ref{defn:RGamma}. 

 Let $S$ be a non-empty subset of $\set{1,2}$, and let $M : = M_S$. Assume that $S \neq \set{2}$ in the even case. Let $T_S$ be as in \S \ref{R groups}. We fix $\mathbb V$ as in \S \ref{subsubsec:mathbb V}, and continue to use the notations introduced in \S \ref{subsubsec:notation for roots and Weyl}. Let $R\Gamma (\Lie N_S , \mathbb V) _{> t_S}$ be as in Definition \ref{defn:RGamma}. Let $\varpi_1$ and $\varpi_2$ be as in Definition \ref{defn:varpi_i}.  
\begin{lem}\label{lem:Kostant} For $\gamma \in T_S(\CC)$ regular in $G$ (or more generally, regular in $M$), we have 
	$$	\Tr (\gamma \mid  R\Gamma ( \Lie N_S , \mathbb V) _{> t_S})  =  \Delta_{M}(\gamma)^{-1}  \sum_{\substack{\omega \in \Omega\\ \langle\omega  (\lambda+\rho) , \varpi_i\rangle > 0 ,~\forall i\in S}} \epsilon(\omega)  (\omega\lambda)(\gamma) \prod_{\alpha \in \Phi(\omega)}\alpha^{-1} (\gamma). $$
\end{lem} 
\begin{proof}The proof is the same as a computation  in the proof of \cite[Prop.~3.3.1]{morel2011suite}.	Let $\Omega_S: = \Omega_{\CC}(M, T_S)$, which is naturally a subgroup of $\Omega$. For $\omega_1 \in \Omega_S$ we define $l(\omega_1)$ and $\epsilon(\omega_1) = (-1)^{l(w)}$ by viewing $\omega_1$ as in $\Omega$; as a  standard fact $l(w_1)$ is also the length of $w_1$ in $\Omega_S$ with respect to the simple roots in $\Phi_M^+$.  Consider 
\begin{align*}
\Omega_S' & : = \bigg \{\omega \in \Omega \mid  \Phi(\omega) \subset \set{\text{roots of $T_{S,\CC}$ on $\Lie (N_S)_{\CC}$}} \bigg \} &  \\ & = \set{\omega \in \Omega\mid  \Phi(\omega) \cap \Phi_M^+ = \emptyset }.
\end{align*}
	Then $\Omega_S'$ is the set of minimal length representatives\index[n]{minimal length representatives} of the cosets in $\Omega_S \backslash \Omega$; see \cite[p.~361]{kostant} or \cite[p.~165]{GHM}. In particular, multiplication induces a bijection 
	\begin{align}\label{eq:Kostant bij}
\Omega_S \times \Omega_S' \isom \Omega.
	\end{align}
	
	We have fixed the positive system $\Phi_M^+$ inside $\Phi(M_{\CC}, T_{S,\CC})$. As usual, we say that an element $\lambda ' \in X^*(T_{S})$ is \emph{$M$-dominant}, if the pairing of $\lambda'$ with any positive coroot in $\Phi(M_\CC, T_{S,\CC})^{\vee}$ is non-negative. 
	For such $\lambda'$, we let $V_{M,\lambda'}$ be the irreducible algebraic representation of $M(\CC)$ of highest weight $\lambda'$. 
	
	As recalled on p.~1700 of \cite{morel2011suite}, Kostant's theorem states that as an algebraic representation of $M(\CC)$, we have $$ \coh^k (\Lie (N_S)_{\CC}, \mathbb V)\cong \bigoplus_{\substack{\omega' \in \Omega'_S \\ l(\omega') = k } } V_{M, \omega'(\lambda+\rho) -\rho}. $$ Consequently, $$ \coh^k (\Lie (N_S)_{\CC}, \mathbb V)_{> t_S}\cong \bigoplus_{\substack{\omega'\in \Omega'_S \\ l(\omega') = k \\ \langle \omega'(\lambda+\rho) -\rho , \varpi_i \rangle > t_i ,~\forall i\in S} } V_{M, \omega'(\lambda+\rho) -\rho}. $$ 
	By a simple computation, we have
	$t_i = \langle -\rho, \varpi_i\rangle$ for $i = 1,2$.  Hence we have 
	\begin{align}\label{eq:Kostant}
	\coh^k (\Lie (N_S)_{\CC}, \mathbb V)_{> t_S}\cong \bigoplus_{\substack{\omega'\in \Omega'_S \\ l(\omega') = k \\\langle \omega' (\lambda +\rho ) , \varpi_i \rangle > 0,~\forall i\in S} } V_{M, \omega'(\lambda+\rho) -\rho}. 
	\end{align}
	
	By the Weyl character formula (see for instance \cite[Fait 3.1.6 ]{morel2011suite}), for any $M$-dominant $\lambda' \in X^*(T_S)$ we have
	\begin{align}\label{eq:Weyl}
	\Tr(\gamma \mid  V_{M, \lambda'}) = \Delta_M( \gamma)^{-1} \sum_{\omega_1 \in \Omega_S} \epsilon(\omega_1) (\omega_1 \lambda') (\gamma) \prod_{\alpha \in \Phi(\omega_1)} \alpha^{-1} (\gamma).
	\end{align}
	(Here we have used the fact that for each $\omega_1 \in \Omega_S$, the set $\Phi(\omega_1)=\Phi^+ \cap (-\omega_1\Phi^+)$ is also equal to $\Phi_M^+ \cap (-\omega _1 \Phi_M^+)$.)
	
	Combining (\ref{eq:Kostant}) and (\ref{eq:Weyl}), we obtain \begin{multline*} 
	\Tr (\gamma\mid  R\Gamma ( \Lie (N_S) , \mathbb V) _{> t_S}) =  \sum_{\substack{\omega' \in \Omega_S'\\ \langle\omega ' (\lambda+\rho) , \varpi_i\rangle > 0,~\forall i\in S }} (-1)^{l(\omega')} \Tr(\gamma\mid  V_{M, \omega'(\lambda +\rho) -\rho})  \\  =\sum_{\substack{\omega' \in \Omega_S'\\ \langle\omega ' (\lambda+\rho) , \varpi_i\rangle > 0,~\forall i\in S }}  \epsilon(\omega') \Delta_M( \gamma)^{-1} \\ \cdot \sum_{\omega_1 \in \Omega_S} \epsilon(\omega_1) \bigg(\omega_1 (\omega'(\lambda +\rho) -\rho)\bigg) (\gamma) \prod_{\alpha \in \Phi(\omega_1)} \alpha^{-1} (\gamma).
	\end{multline*}
	Since $\varpi_i$ is invariant under $\Omega_S$ for every $i\in S$, and since we have the bijection (\ref{eq:Kostant bij}), the above is equal to $$ \sum_{\substack{\omega \in \Omega\\ \langle\omega  (\lambda+\rho) , \varpi_i\rangle > 0 , ~\forall i \in S }} \epsilon(\omega) \Delta_M(\gamma) ^{-1}  (\omega\lambda  )(\gamma)\cdot \bigg(\omega \rho - p_1(\omega) \rho \bigg) (\gamma) \prod_{\alpha \in \Phi(p_1(\omega))}\alpha^{-1} (\gamma),$$ where for each $\omega \in \Omega $ we set $p_1(\omega)$ to be the unique element of $\Omega_S$ such that $\omega \in  p_1(\omega  )  \Omega_S'$. To finish the proof, we just need to check that for all $\omega \in \Omega$, we have $$ \omega(\rho)  - p_1(\omega)(\rho) - \sum _{ \alpha \in \Phi(p_1(\omega_1))} \alpha = - \sum_{\alpha \in \Phi(\omega)} \alpha.$$ But this follows from the identity 
	$$ \rho - \theta (\rho) = \sum_{\alpha \in \Phi (\theta)} \alpha $$ which holds for arbitrary $ \theta \in \Omega.$
\end{proof}

\section[Case $M_1$]{Kostant--Weyl terms and discrete series characters, case \texorpdfstring{$M_1$}{M1}}\label{KH M1} 

\subsection{}  We keep the notations in \S \ref{subsec:stable discrete series chars} and \S \ref{subsec:Kostant's thm}. We take $S = \set{1}$ and $M = M_1$. Recall from \S \ref{R groups} that we have fixed an elliptic maximal torus
$T_1 = T_{\GL_2}^{\std} \times T_{W_2}$  in $M$. Consider a regular element $\gamma \in T_1(\RR)$. We write\index{$\gamma_{W_2}$} $$\gamma = (\begin{pmatrix}
a & b \\ -b & a 
\end{pmatrix} , \gamma_{W_2}) \in  T_{\GL_2}^{\std}(\RR) \times T_{W_2} (\RR) ,$$ with $a, b \in \RR$ and $a^2+b^2 \neq 0$. Note that $(\epsilon_1 + \epsilon_2) (\gamma) = a^2+ b^2$. Hence we have 
$$R_{\gamma} = \set{\pm (\epsilon_1 +\epsilon_2)}.$$
Let $L_M(\gamma)$ be as in Definition \ref{Defn L_M 12 odd}. 
\begin{prop}\label{Arch comp M_1} 
	Suppose $a^2 + b ^2 < 1 $. Then we have \begin{align*}
	\Phi^G_M ( \gamma , \Theta) = 2 (-1)^{q(G)+1} L_M(\gamma).
	\end{align*} \end{prop}

\begin{proof}
	We first compute $\Phi^G_M(\gamma,\Theta)$ using (\ref{eq:to explain notation}). Clearly $T_1(\RR)$ is connected. Hence $\gamma \in G(\RR) ^0$, and so the integers $n(\gamma,\omega B)$ in (\ref{eq:to explain notation}) are defined by (\ref{eq:defn of n(gamma B)}). 
	
	The subgroup $A_M(\RR)^0 \subset T_1(\RR)$ consists of $\begin{pmatrix}
	z & 0 \\ 0 & z
	\end{pmatrix} \in T_{\GL_2}^{\std} (\RR), z \in \RR_{>0}$. The subgroup $T_1(\RR)_1 \subset T_1(\RR)$ is $\Uni (1 ) (\RR)\times T_{W_2}(\RR)$, where $\Uni(1)(\RR)$ consists of $$\begin{pmatrix}
	z_1 & z_2 \\ -z_2 & z_1
\end{pmatrix} \in T_{\GL_2}^{\std} (\RR), \qquad z_1, z_2 \in \RR, z_1^2 + z_2^2 = 1. $$ Hence the projection of $\gamma$ in $A_M(\RR)^0 = \RR_{> 0}$ is $\sqrt{a^2 +b^2}$, and  	$$x_{\gamma}  = \log \sqrt{a^2 + b^2} \in \RR\cong \Lie (A_M) = X_*(A_M) _{\RR} .$$ 
	Since $a^2 + b^2 <1$, we have $$x_{\gamma} \in \RR_{< 0} \cong \RR_{>0}(-\epsilon_1^{\vee} - \epsilon_2^{\vee}) .$$

	Since $R_{\gamma} = \set{\pm (\epsilon_1 + \epsilon_2)}$, by Lemma \ref{lem:rk1} we have 
	$$\bar c _{R_{\gamma}} (x_{\gamma}, \chi) = \begin{cases}
	2 , & \text{if } \chi \in \RR_{>0} (\epsilon_1 +\epsilon_2), \\ 0,  & \text{if } \chi  \in \RR_{>0} (-\epsilon_1 -\epsilon_2).
	\end{cases}  $$	
	Hence by the definition (\ref{eq:defn of n(gamma B)}), for $\omega \in \Omega$ we have $$ n(\gamma , \omega B) = \begin{cases}
	2 , & \text{if } \wp (\omega(\lambda +\rho )  )\in \RR_{>0} (\epsilon_1 +\epsilon_2), \\ 0,  &\text{if }  \wp (\omega(\lambda +\rho )  ) \in \RR_{>0} (-\epsilon_1 -\epsilon_2) .
	\end{cases} $$ Now the term $\epsilon_R(\gamma)$ in (\ref{eq:to explain notation}) is $-1$. By the above computation and by (\ref{eq:to explain notation}), we obtain \begin{multline}\label{LHS for M_1}
	\Phi_M^G (\gamma,\Theta) = 2 (-1)^{q(G)+1} \delta_{P_1(\RR)}(\gamma)^{1/2} \Delta_M(\gamma)^{-1} \\ \cdot  \sum_{\substack{\omega\in \Omega \\ \wp(\omega(\lambda +\rho   ) )  \in \RR_{>0} (\epsilon_1 +\epsilon_2) }} \epsilon(\omega) (\omega \lambda) (\gamma) \prod_{\alpha \in \Phi(\omega)} \alpha^{-1} (\gamma).
	\end{multline}
	
Next we compute $2(-1)^{q(G)+1} L_M(\gamma)$. By Proposition \ref{prop:defn of L_M for M_1 M_2} and Lemma \ref{lem:Kostant}, we have
	\begin{multline}\label{RHS for M_1}
	2(-1)^{q(G)+1} L_M(\gamma) = 2(-1)^{q(G)+1} \delta_{P_1(\RR)} (\gamma)^{1/2} \Delta_M(\gamma) ^{-1} \\ \cdot \sum_{\substack{\omega \in \Omega\\  \langle\omega  (\lambda+\rho) , \varpi_1\rangle > 0 }} \epsilon(\omega)  (\omega\lambda )(\gamma) \prod_{\alpha \in \Phi(\omega)}\alpha^{-1} (\gamma).
	\end{multline}
	Comparing (\ref{LHS for M_1}) and (\ref{RHS for M_1}), we see that the proof reduces to checking that for all $\omega \in \Omega$, we have $$ \langle\omega  (\lambda+\rho) , \varpi_1\rangle > 0  \Longleftrightarrow \wp(\omega (\lambda +\rho)  ) \in \RR_{>0} (\epsilon_1 + \epsilon_2) .$$ This is obvious.	
\end{proof}

\section[Odd case $M_2$]{Kostant--Weyl terms and discrete series characters, odd case \texorpdfstring{$M_2$}{M2}} \label{KH M2}
\subsection{} We keep the notations in \S \ref{subsec:stable discrete series chars} and \S \ref{subsec:Kostant's thm}. We take $S = \set{2}$ and $M = M_2$. Assume that $d$ is odd. Recall from \S \ref{R groups} that we have fixed an elliptic maximal torus $T_2 = \GG_m \times T_{W_1}$ in $M$. Consider a regular element $\gamma \in T_2(\RR)$. We write \index{$\gamma_{W_1}$}
\begin{align*}
\gamma  = (a, \gamma_{W_1}),
\end{align*} with $a \in \RR^{\times}$. If $a>0$, then $R_{\gamma} = \set{\pm \epsilon_1}$. Otherwise $R_{\gamma} = \emptyset$. Let $L_M(\gamma)$ be as in Definition \ref{Defn L_M 12 odd}. 

\begin{prop}\label{Arch comp M2} When $a<0$, we have $\Phi^G_M(\gamma,\Theta) = 0$. When $0<a< 1 $, we have \begin{align*}
	\Phi^G_M ( \gamma , \Theta) =  (-1)^{q(G)+1} L_M(\gamma).
	\end{align*}
\end{prop}

\begin{proof}
	When $a < 0$, we have $R_{\gamma} =\emptyset$. It follows from Corollary \ref{cor:GKM} that $\Phi^G_M(\gamma,\Theta) = 0$, as desired. 
	
	Now assume that $0 < a < 1$. We first compute $\Phi^G_M(\gamma,\Theta)$ using (\ref{eq:to explain notation}). We have $T_2 \cong \GG_m \times \Uni(1)^{m-1}$, and $A_M \cong \GG_m$, $T_2(\RR)_1 = \set{\pm 1 }\times \Uni(1)(\RR)^{m-1}$. Hence the projection of $\gamma$ in $A_M(\RR)^0 = \RR_{>0}$ is $a$, and  	$$x_{\gamma}  = \log a \in \RR\cong \Lie (A_M) = X_*(A_M) _{\RR} .$$ 
	Since $0 < a <1$, we have $$x_{\gamma} \in \RR_{< 0} \cong \RR_{>0}(-\epsilon_1^{\vee}) .$$

	Since $R_{\gamma} = \set{\pm \epsilon_1}$, by Lemma \ref{lem:rk1} we have 
	$$\bar c _{R_{\gamma}} (x_{\gamma}, \chi) = \begin{cases}
	2 , & \chi \in \RR_{>0} (\epsilon_1), \\ 0,  & \chi  \in \RR_{>0} (-\epsilon_1).
	\end{cases}  $$	
	Hence by the definition (\ref{eq:defn of n(gamma B)}), for $\omega \in \Omega$ we have $$ n(\gamma , \omega B) = \begin{cases}
	2 , &  \text{if } \wp (\omega(\lambda +\rho )  )\in \RR_{>0} (\epsilon_1), \\ 0,  &  \text{if } \wp (\omega(\lambda +\rho )  ) \in \RR_{>0} (-\epsilon_1 ) .
	\end{cases} $$ Now the term $\epsilon_R(\gamma)$ in (\ref{eq:to explain notation}) is $-1$. By the above computation and by (\ref{eq:to explain notation}), we obtain \begin{multline}\label{LHS for M_2}
	\Phi_M^G (\gamma,\Theta) = 2 (-1)^{q(G)+1} \delta_{P_2(\RR)}(\gamma)^{1/2} \Delta_M(\gamma)^{-1} \\ \cdot \sum_{\substack{\omega\in \Omega \\ \wp(\omega(\lambda +\rho   ) )  \in \RR_{>0} (\epsilon_1) }} \epsilon(\omega) (\omega \lambda) (\gamma) \prod_{\alpha \in \Phi(\omega)} \alpha^{-1} (\gamma).
	\end{multline}
	
	Next we compute $(-1)^{q(G)+1} L_M(\gamma)$. By Proposition \ref{prop:defn of L_M for M_1 M_2} and Lemma \ref{lem:Kostant}, we have
	\begin{multline}\label{RHS for M_2}
	(-1)^{q(G)+1} L_M(\gamma) = 2(-1)^{q(G)+1} \delta_{P_2(\RR)} (\gamma)^{1/2} \Delta_M(\gamma) ^{-1} \\ \cdot \sum_{\substack{\omega \in \Omega\\  \langle\omega  (\lambda+\rho) , \varpi_2\rangle > 0 }} \epsilon(\omega)  (\omega\lambda )(\gamma) \prod_{\alpha \in \Phi(\omega)}\alpha^{-1} (\gamma).
	\end{multline}
	Comparing (\ref{LHS for M_2}) and (\ref{RHS for M_2}), we see that the proof reduces to checking that for all $\omega \in \Omega$, we have $$ \langle\omega  (\lambda+\rho) , \varpi_2\rangle > 0  \Longleftrightarrow \wp(\omega (\lambda +\rho)  ) \in \RR_{>0} (\epsilon_1) .$$ This is obvious.
\end{proof}

\section[Case $M_{12}$]{Kostant--Weyl terms and discrete series characters, case \texorpdfstring{$M_{12}$}{M12}} \label{KH M12}
\subsection{}\label{subsubsec:KH M12}
We keep the notations in \S \ref{subsec:stable discrete series chars} and \S \ref{subsec:Kostant's thm}. We take $S = \set{1,2}$ and $M = M_{12}$. (We drop the assumption that $d$ is odd made in \S \ref{KH M2}.) Recall from \S \ref{R groups} that we have fixed an elliptic maximal torus $T_{12}  = \GG_m \times \GG_m \times T_{W_2} $ in $M$. Consider a regular element $\gamma \in T_{12}(\RR)$. We write 
\begin{align*}
\gamma  = (a,b, \gamma_{W_2}), 
\end{align*} with $a,b\in \RR^{\times}$. Let $L_M(\gamma)$ be as in Definition \ref{Defn L_M 12 odd}. We fix an element $g_0 \in M_{2,l}(\QQ)^{\sharp}$, as in Definition \ref{defn:n12}.  
\begin{lem}\label{lem:first comp of LM12} We have \begin{align}\label{first calc for L_M}
	L_{M} (\gamma)=  & ~ \delta_{P_{12}(\RR)}  (\gamma )^{1/2} \Tr(\gamma  \mid  R\Gamma (\Lie N_{12} , \mathbb V) _{>t_{12}} ) \\\nonumber  &  + \delta_{P_{12}(\RR)} (g_0\gamma g_0^{-1} ) ^{1/2}  \Tr(g_0\gamma g_0^{-1}  \mid  R\Gamma (\Lie N_{12} , \mathbb V) _{>t_{12}} )
	\\ \nonumber &  - \abs{D^{M_2} _{ M} ( \gamma )}^{1/2}_{\RR} \delta_{P_2(\RR)} ( \gamma )  ^{1/2} \Tr(\gamma \mid  R\Gamma (\Lie N_{2} , \mathbb V) _{>t_2} ).
	\end{align}
\end{lem}
\begin{proof}The lemma follows from Proposition \ref{prop:to simplify LM12}, the fact that $\dim A_{M}/ A_{M_2}=1$, and the fact that $n^{M_2}_M =2$. Here $n^{M_2}_M$ is clearly equal to the cardinality of $\cW^{M_{2,l}}_{M_{l}}$, and we already showed in the proof of Proposition \ref{prop:to simplify LM12} that this group is $\ZZ/2\ZZ$.
\end{proof}

\begin{defn}\label{defn:omega_0}When $d$ is odd, let \index{$\omega_0, \omega_1, \omega_2$}
	\begin{align*}
	\omega_0 &:= s_{\epsilon_2} \in \Omega , \\
	\omega_1 & := s_{\epsilon_1-\epsilon_2} \in \Omega ,\\
	\omega_2 & : = s_{\epsilon_1} \in \Omega. 
	\end{align*}When $d$ is even, let $$\omega_0: = s_{\epsilon_2 +\epsilon_3} s_{\epsilon_2 - \epsilon_3} \in \Omega. $$ Here $s_{\alpha}$ denotes the reflection in $\Omega$ corresponding to $\alpha \in \Phi(G_\CC, T_{12,\CC})$. 
\end{defn}
The following lemma is similar to an argument on p.~1702 of \cite{morel2011suite}.
\begin{lem}\label{lem:quot}
	Let $s \in \set{\omega_0, \omega_1, \omega_2}$ if $d$ is odd, and let $s = \omega_0$ if $d$ is even. 
	\begin{enumerate}
		\item The automorphism of $T_{12, \CC}$ induced by $s$ is defined over $\RR$.
		\item Let $\gamma \in T_{12}(\RR)$ be regular, and let $\gamma':= s (\gamma) \in T_{12} (\RR)$. For any $\omega \in \Omega$ we have 
		\begin{align}\label{eq:quot}
		\frac{
			\delta_{P_{12} (\RR)}  (\gamma' ) ^{1/2} \prod_{\alpha \in \Phi(\omega)} \alpha^{-1} (\gamma') }{ \delta_{P_{12} (\RR)}  (\gamma)^{1/2} \prod _{\alpha\in \Phi (s \omega) } \alpha^{-1} (\gamma)} = \prod_{\alpha \in \Phi_M^+} \frac{\abs{\alpha(\gamma')}^{-1/2} } { \abs{\alpha(\gamma)}^{-1/2} }\cdot  \prod_{\alpha \in \Phi(s)}\abs{\alpha(\gamma)}^{-1} \alpha(\gamma) .
		\end{align}
		Here $\abs{\cdot}$ denotes the usual absolute value on $\CC$, and as usual $\Phi(s)$ denotes $\Phi^+ \cap (-s\Phi^+)$. 
	\end{enumerate}

\end{lem}
\begin{proof}
\textbf{(1)} The automorphism in question is given by 
	\begin{align*}
	(x ,y,z ) &\longmapsto (y,x, z), \quad \forall (x,y)\in \GG_m^2,~ z\in T_{W_2, \CC},
	\end{align*} when $s = \omega_1$, and is given by
	\begin{align*}
	(x ,y,z ) &\longmapsto (x^{-1},y, z), \quad \forall (x,y)\in \GG_m^2,~ z\in T_{W_2, \CC},
	\end{align*} when $s = \omega_2$. In these cases the claim is obvious. When $s = \omega_0$, the automorphism in question is of the form
	\begin{align*}
	(x ,y,z ) &\longmapsto (x, y^{-1}, f(z)), \quad (x,y)\in \GG_m^2,~ z\in T_{W_2, \CC},
	\end{align*}
	for some automorphism $f$ of $T_{W_2, \CC}$. Since $T_{W_2}\cong \Uni(1)^{m-2}$, every automorphism of $T_{W_2, \CC}$ is defined over $\RR$. This proves the claim. 
	
	\textbf{(2)} Note that $\Phi^+ $ is the disjoint union of $\Phi_M^+$ and the set of roots of $T_{12,\CC}$ acting on $\Lie (N_{12})_{\CC}$. Hence 
	$$\delta_{P_{12} (\RR)} (\nu) = \prod_{\alpha \in \Phi^+} \abs{\alpha(\nu)} \prod_{\alpha \in \Phi_M^+} \abs{\alpha(\nu)}^{-1}, \quad \forall \nu \in T_{12}(\RR).$$
	For any $\omega \in \Omega$, we have 
	\begin{align*}
	\delta_{P_{12} (\RR)}  (\gamma' )^{1/2}  & \prod_{\alpha \in \Phi(\omega)} \alpha^{-1} (\gamma')  \\ &  = \prod_{\alpha \in \Phi^+} \abs{\alpha(\gamma')}^{1/2} \prod_{\alpha\in \Phi_M^+} \abs{\alpha(\gamma')} ^{-1/2} \prod_{\alpha \in \Phi^+ \cap (-\omega\Phi^+)}\alpha^{-1}(\gamma')
	\\ &  = \prod_{\alpha \in s \Phi^+} \abs{\alpha(\gamma)}^{1/2} \prod_{\alpha\in \Phi_M^+} \abs{\alpha(\gamma')} ^{-1/2} \prod_{\alpha \in s\Phi^+ \cap (-s\omega\Phi^+)}\alpha^{-1}(\gamma).
	\end{align*}
	Also we have 
	$$ \delta_{P_{12} (\RR)}  (\gamma) ^{1/2} \prod _{\alpha\in \Phi (s \omega) } \alpha^{-1} (\gamma) =   \prod_{\alpha \in \Phi^+} \abs{\alpha(\gamma)}^{1/2} \prod_{\alpha\in \Phi_M^+} \abs{\alpha(\gamma)} ^{-1/2} \prod_{\alpha \in \Phi^+ \cap (-s\omega\Phi^+)}\alpha^{-1}(\gamma).$$
	Hence
	\begin{multline*}\frac{
		\delta_{P_{12} (\RR)}  (\gamma' ) ^{1/2}\prod_{\alpha \in \Phi(\omega)} \alpha^{-1} (\gamma') }{ \delta_{P_{12} (\RR)}  (\gamma) ^{1/2} \prod _{\alpha\in \Phi (s \omega) } \alpha^{-1} (\gamma)} \\  = \prod_{\alpha \in \Phi_M^+} \frac{\abs{\alpha(\gamma')}^{-1/2} } { \abs{\alpha(\gamma)}^{-1/2} } \cdot   \frac{\prod_{\alpha \in s \Phi^+} \abs{\alpha(\gamma)}^{1/2} }{\prod_{\alpha \in  \Phi^+} \abs{\alpha(\gamma)}^{1/2} }  
	\cdot \frac{\prod_{\alpha \in s \Phi^+ \cap (-s \omega \Phi^+)} \alpha^{-1}(\gamma) }{\prod_{\alpha \in  \Phi^+ \cap (-s \omega \Phi^+)} \alpha^{-1}(\gamma) }.
	\end{multline*} To finish the proof, we note that 
	$$\frac{\prod_{\alpha \in s \Phi^+} \abs{\alpha(\gamma)}^{1/2} }{\prod_{\alpha \in  \Phi^+} \abs{\alpha(\gamma)}^{1/2} } =  \frac{\prod_{\alpha \in \Phi^+ \cap s \Phi^+} \abs{\alpha(\gamma)}^{1/2} \prod_{\alpha \in - \Phi (s)} \abs{\alpha(\gamma)}^{1/2} }{\prod_{\alpha \in  \Phi^+ \cap s \Phi^+} \abs{\alpha(\gamma)}^{1/2} \prod_{\alpha \in  \Phi(s)} \abs{\alpha(\gamma)}^{1/2} } = \prod_{\alpha \in \Phi(s) } \abs{\alpha(\gamma) }^{-1},$$
	and that \begin{align*}
	\frac{\prod_{\alpha \in s \Phi^+ \cap (-s \omega \Phi^+)} \alpha^{-1}(\gamma) }{\prod_{\alpha \in  \Phi^+ \cap (-s \omega \Phi^+)} \alpha^{-1}(\gamma) } & = \frac{\prod_{\alpha \in (-\Phi^+) \cap s \Phi^+ \cap (-s \omega \Phi^+)} \alpha^{-1}(\gamma) }{\prod_{\alpha \in  \Phi^+ \cap (-s \Phi^+) \cap (-s \omega \Phi^+)} \alpha^{-1}(\gamma) } \\ & = \frac{\prod _{\alpha \in \Phi(s) \cap (s \omega \Phi^+)} \alpha (\gamma)}{ \prod _{\alpha \in \Phi(s) \cap (-s \omega \Phi^+)} \alpha ^{-1} (\gamma) } \\ &
	= \prod _{\alpha \in \Phi(s)} \alpha (\gamma).
	\end{align*}
	The desired (\ref{eq:quot}) follows.
\end{proof}

\begin{lem}\label{lem:g_0 and omega_0}
	For any $g_0 \in M_{2,l}(\QQ)^{\sharp}$ (see Definition \ref{defn:n12}), there exists $g\in \SO(W_2)(\RR) \subset G(\RR)$ such that $gg_0$ normalizes $T_{12}$ and the image of $gg_0$ in $\Omega$ is $\omega_0$ as in Definition \ref{defn:omega_0}. 
\end{lem}
\begin{proof}
Recall that $T_{12} = \GG_m^2 \times T_{W_2}$, where $\GG_m^2 = \GL(V_1) \times \GL(V_2/V_1)$, and $T_{W_2}$ is an elliptic maximal torus in $\SO(W_2)$. From the definition of $M_{2,l}(\QQ)^{\sharp}$, we know that $g_0$ normalizes $\GG_m^2$, stabilizes $W_2 \subset V$, and restricts to an element $g_0 |_{W_2} \in \mathrm{O}(W_2)(\RR) - \SO(W_2)(\RR)$. Since all elliptic maximal tori in $\SO(W_2)$ over $\RR$ are conjugate under $\SO(W_2)(\RR)$, there exists $g\in \SO(W_2)(\RR)$ such that $gg_0$ normalizes $T_{12}$. We let $h$ denote $(gg_0 )|_{W_2}$, which is an element of $\mathrm{O}(W_2)(\RR) - \SO(W_2)(\RR)$ normalizing $T_{W_2}$. 

If $d$ is odd, we can take $g$ to be $-\id_{W_2} \cdot (g_0|_{W_2})^{-1}$. Then $gg_0$ permutes $\set{e_2, e_2'}$ non-trivially, fixes $e_1$ and $e_1'$, and acts as $-\id_{W_2}$ on $W_2$. It follows that the image of $gg_0$ in $\Omega$ is $\omega_0$, as desired.  

Assume that $d$ is even. Then $m =d/2 \geq 3$. By our definition of the $\ZZ$-basis $\set{\epsilon_1,\cdots, \epsilon_m}$ of $X^*(T_{12})$, we know that $\set{\epsilon_3,\cdots, \epsilon_m}$ is a $\ZZ$-basis of $X^*(T_{W_2})$. Moreover, $$\Phi(\SO(W_2)_{\CC}, T_{W_2,\CC}) = \set{\pm \epsilon_i \pm \epsilon_j \mid 3 \leq i< j \leq m}.$$ It is easy to check that there exists an element $h' \in \mathrm{O}(W_2)(\CC) - \SO(W_2)(\CC)$ normalizing $T_{W_2,\CC}$ such that the automorphism $\sigma'$ of $X^*(T_{W_2})$ induced by $h'$ satisfies $\sigma '(\epsilon_3) = -\epsilon_3$ and $\sigma'(\epsilon_i) = \epsilon_i$ for $4\leq i \leq m$. Denote by $\sigma$ the automorphism of $X^*(T_{W_2})$ induced by $h$. It suffices to show that $$\sigma \in \Omega_{\RR}(\SO(W_2), T_{W_2}) \sigma' \subset \Aut(X^*(T_{W_2})). $$  Here $\Omega_{\RR} (\SO(W_2), T_{W_2})$ is the real Weyl group $\Nor_{\SO(W_2)(\RR)} (T_{W_2}) / T_{W_2}(\RR)$, viewed as a subgroup of $\Aut(X^*(T_{W_2}))$. Since $h$ and $h'$ differ by left-multiplication by an element of $\SO(W_2)(\CC)$ normalizing $T_{W_2,\CC}$, we have $\sigma \in \Omega_{\CC}(\SO(W_2), T_{W_2}) \sigma'$. We finish the proof by noting that $\Omega_{\CC}(\SO(W_2), T_{W_2})  = \Omega_{\RR}(\SO(W_2), T_{W_2})$, since $\SO(W_2)$ is anisotropic over $\RR$. 
 \end{proof}
\begin{defn}\label{defn:N_i}
	For $\omega\in \Omega$, define
	\begin{align*}
	N_1(\omega)  & : = \begin{cases}
	1,  & \text{if } \langle\omega  (\lambda+\rho) , \varpi_i\rangle > 0  \text{ for } i =1,2, \\ 0,  & \mbox{otherwise}.
	\end{cases} \\ N_2(\omega) & : = \begin{cases}
	1,    & \text{if } \langle\omega  (\lambda+\rho) , \omega_0 \varpi_i\rangle > 0 \text{ for }  i=1,2, \\ 0, & \mbox{otherwise}.
	\end{cases} \\ N_3(\omega) & : = \begin{cases}
	1,  &  \text{if } \langle\omega  (\lambda+\rho) , \varpi_2\rangle > 0 ,\\ 0, & \mbox{otherwise}.
	\end{cases}
	\end{align*} Here $\omega_0$ is as in Definition \ref{defn:omega_0}. 
\end{defn}

\begin{lem}\label{Computing L_M for M_12} Let $\gamma = (a,b, \gamma_{W_2})$ be a regular element of $T_{12}(\RR)$. The quantity \index{$\tilde L_M(\gamma)$} $\tilde L_M(\gamma) : = L_M(\gamma) \cdot  ( \delta_{P_{12}(\RR)} (\gamma) ^{1/2} \Delta_{M} (\gamma) ^{-1} )^{-1}$  can be computed as follows. \begin{enumerate} 
		\item If $d$ is odd, then 
		$$\tilde L_M(\gamma)  =\sum_{\omega \in \Omega} \bigg[  N_1(\omega)  -  \sign(b ) N_2(\omega )  - \sign (1-b^{-1}) N_3(\omega)\bigg]  \epsilon(\omega) (\omega\lambda )(\gamma) \prod_{\alpha \in \Phi(\omega)}\alpha^{-1} (\gamma).	$$
		\item If $d$ is even, then $$
		\tilde  L_M(\gamma) = \sum_{\omega \in \Omega} \bigg[ N_1(\omega) +N_2(\omega) -  N_3(\omega)\bigg]\epsilon(\omega)  (\omega\lambda) (\gamma) \prod_{\alpha \in \Phi(\omega)} \alpha^{-1} (\gamma) .$$
	\end{enumerate}
	
\end{lem}
\begin{proof}
	Our starting point is (\ref{first calc for L_M}). Let $\gamma'= \omega_0 (\gamma)$. By Lemma \ref{lem:g_0 and omega_0}, we may replace $g_0 \gamma g_0^{-1}$ in the second summand on the RHS of (\ref{first calc for L_M}) by $\gamma'$. Now we would like to rewrite the third summand. Define $$ \eta_2 (\gamma) : = \prod_{\alpha \in \Phi_{M_2} ^+ - \Phi_M^+} \frac{\abs{1- \alpha^{-1} (\gamma)}}{ 1- \alpha^{-1} (\gamma) }.$$ Then arguing as on p.~1701 of \cite{morel2011suite}, we have 
	\begin{align*}
	\abs{D_M^{M_{2}} (\gamma)}^{1/2} \delta_{P_{2} (\RR) } (\gamma) ^{1/2} \Delta_{M_{2}} (\gamma) ^{-1} = \eta_{2} (\gamma) \delta_{P_{12}(\RR)} (\gamma) ^{1/2} \Delta_{M} (\gamma) ^{-1}.
	\end{align*} 
	Hence we can rewrite (\ref{first calc for L_M}) as follows:
	\begin{align}\label{second calc for L_M}
	L_{M} (\gamma)=  & ~ \delta_{P_{12}(\RR)}  (\gamma )^{1/2} \Tr(\gamma  \mid  R\Gamma (\Lie N_{12} , \mathbb V) _{>t_{12}} ) \\\nonumber  &  + \delta_{P_{12}(\RR)} (\gamma' ) ^{1/2}  \Tr(\gamma'\mid  R\Gamma (\Lie N_{12} , \mathbb V) _{>t_{12}} )
	\\ \nonumber &  - \delta_{P_{12}(\RR)} (\gamma) ^{1/2} \Delta_{M} (\gamma) ^{-1} \Delta_{M_2}(\gamma) \eta_2 (\gamma) \Tr(\gamma \mid  R\Gamma (\Lie N_{2} , \mathbb V) _{>t_2} ).
	\end{align}
	Using Lemma \ref{lem:Kostant} to compute the $\Tr(\cdots)$ terms in (\ref{second calc for L_M}), we get 
	\begin{align*}
	L_{M} (\gamma)=  & ~ \delta_{P_{12}(\RR)}  (\gamma )^{1/2} \Delta_{M}(\gamma)^{-1}  \sum_{\substack{\omega \in \Omega\\ \langle\omega  (\lambda+\rho) , \varpi_i\rangle > 0 ,~\forall i\in \set{1,2}}} \epsilon(\omega)  (\omega\lambda)(\gamma) \prod_{\alpha \in \Phi(\omega)}\alpha^{-1} (\gamma) \\\nonumber  &  + \delta_{P_{12}(\RR)} (\gamma' ) ^{1/2}  \Delta_{M}(\gamma')^{-1}  \sum_{\substack{\omega \in \Omega\\ \langle\omega  (\lambda+\rho) , \varpi_i\rangle > 0 ,~\forall i\in \set{1,2}}} \epsilon(\omega)  (\omega\lambda)(\gamma') \prod_{\alpha \in \Phi(\omega)}\alpha^{-1} (\gamma')
	\\ \nonumber &  - \delta_{P_{12}(\RR)} (\gamma) ^{1/2} \Delta_{M} (\gamma) ^{-1}  \eta_2 (\gamma)  \sum_{\substack{\omega \in \Omega\\ \langle\omega  (\lambda+\rho) , \varpi_2\rangle > 0}} \epsilon(\omega)  (\omega\lambda)(\gamma) \prod_{\alpha \in \Phi(\omega)}\alpha^{-1} (\gamma).
	\end{align*}
	By Lemma \ref{lem:quot}, the second summand in the above is equal to 
	\begin{align*} \delta_{P_{12}(\RR)} (\gamma) ^{1/2}  \Delta_{M}(\gamma)^{-1} A(\gamma,\gamma')  \sum_{\substack{\omega \in \Omega\\ \langle\omega  (\lambda+\rho) , \varpi_i\rangle > 0 ,~\forall i\in \set{1,2}}} \epsilon(\omega)  (\omega\lambda)(\gamma')  \prod_{\alpha \in \Phi(\omega_0\omega)}\alpha^{-1} (\gamma) ,
	\end{align*}
	where $$A(\gamma,\gamma'): =   \frac{\Delta_M(\gamma)}{\Delta_M (\gamma')}  \prod_{\alpha \in \Phi_M^+} \frac{\abs{\alpha(\gamma')}^{-1/2} } { \abs{\alpha(\gamma)}^{-1/2} }\prod_{\alpha \in \Phi(\omega_0)}\frac{\alpha (\gamma)} {\abs{\alpha(\gamma)}} . $$
	Therefore we have 
	\begin{align*}
\tilde	L_M(\gamma) =   &   \sum_{\substack{\omega \in \Omega \\ \langle\omega  (\lambda+\rho) , \varpi_i\rangle > 0, ~ \forall  i \in \set{1,2}}}  \epsilon(\omega) (\omega\lambda )(\gamma) \prod_{\alpha \in \Phi(\omega)}\alpha^{-1} (\gamma)  \\ \nonumber &  + A(\gamma,\gamma')   \sum_{\substack{\omega \in \Omega\\ \langle\omega  (\lambda+\rho) , \varpi_i\rangle > 0, \forall i \in \set{1,2} } } \epsilon(\omega) (\omega\lambda )(\gamma') \prod_{\alpha \in \Phi(\omega_0\omega)}\alpha^{-1} (\gamma)  \\ \nonumber &  -\eta_2(\gamma)\sum_{\substack{\omega \in \Omega\\ \langle\omega  (\lambda+\rho) , \varpi_2\rangle > 0 }} \epsilon(\omega) (\omega\lambda )(\gamma) \prod_{\alpha \in \Phi(\omega)}\alpha^{-1} (\gamma) . \end{align*}
	Making the substitution $\omega \mapsto \omega_0 \omega$ in the second summation and using the following obvious relations:
	$$ \omega_0^2  =1 ,$$
	$$ 	(\omega_0 \omega \lambda) (\gamma')   = (\omega \lambda) (\gamma),$$ 
	$$ \epsilon(\omega_0\omega)  = \epsilon(\omega_0) \epsilon(\omega) , $$
	$$ \lprod{\omega_0\omega(\lambda+\rho) , \varpi_i}  = \lprod{\omega(\lambda+\rho) , \omega_0 \varpi_i}, $$
 we obtain 
	\begin{align} \nonumber
	\tilde L_M(\gamma) = &       \sum_{\substack{\omega \in \Omega \\ \langle\omega  (\lambda+\rho) , \varpi_i\rangle > 0, ~ \forall  i \in \set{1,2}}}  \epsilon(\omega) (\omega\lambda )(\gamma) \prod_{\alpha \in \Phi(\omega)}\alpha^{-1} (\gamma)  \\ \nonumber &   + \epsilon(\omega_0)A(\gamma,\gamma')   \sum_{\substack{\omega \in \Omega\\ \langle\omega  (\lambda+\rho) , \omega_0 \varpi_i\rangle > 0, \forall i \in \set{1,2} } } \epsilon(\omega) (\omega\lambda )(\gamma) \prod_{\alpha \in \Phi(\omega)}\alpha^{-1} (\gamma)  \\ \nonumber &  -\eta_2(\gamma)\sum_{\substack{\omega \in \Omega\\ \langle\omega  (\lambda+\rho) , \varpi_2\rangle > 0 }} \epsilon(\omega) (\omega\lambda )(\gamma) \prod_{\alpha \in \Phi(\omega)}\alpha^{-1} (\gamma)\\ \label{third calc} = &    \sum_{\omega \in \Omega}    \bigg[  N_1(\omega)  + \epsilon(\omega_0) A(\gamma,\gamma') N_2(\omega) - \eta_2(\gamma) N_3(\omega)\bigg]  \\ \nonumber & \cdot \bigg[ \epsilon(\omega) (\omega\lambda )(\gamma) \prod_{\alpha \in \Phi(\omega)}\alpha^{-1} (\gamma)\bigg]. \end{align} 
	To finish the proof it remains to compute the quantities $\epsilon(\omega_0)$, $A(\gamma,\gamma')$, and $\eta_2(\gamma)$, which we carry out separately in the odd and even cases. 
	
	First assume that $d$ is odd. Then 
	\begin{align}\label{eq:epsilon odd}
	\epsilon(\omega_0) = -1.
	\end{align}
	To compute $A(\gamma,\gamma')$, first note that $\Delta_M(\gamma)/\Delta_M (\gamma')$ and $$\prod_{\alpha \in \Phi_M^+} \abs{\alpha(\gamma')}^{\frac{-1}{2}}  \abs{\alpha(\gamma)}^{\frac{1}{2}}$$ are both $1$, since $\gamma^{-1} \gamma'$ lies in the center of $M$. To compute $$\prod_{\alpha \in \Phi(\omega_0)} \abs{\alpha (\gamma)} ^{-1} \alpha (\gamma), $$ we have $\Phi(\omega_0) = \set{\epsilon_2} \cup \set{\epsilon_2 \pm \epsilon_j \mid j\geq 3}$, and we know that $\epsilon_2 +\epsilon_j$ is the complex conjugate of $\epsilon_2 - \epsilon_j$ for $j\geq 3$, with respect to the real structure of $T_{12}$. (In fact, the complex conjugation acts on $X^*(T_{W_2}) = \mathrm{span}_{\ZZ} \set{\epsilon_3,\cdots, \epsilon_m}$ as $-1$.) Hence we have 
	\begin{align}\label{eq:part in partial odd} A(\gamma,\gamma') = 
	\prod_{\alpha \in \Phi(\omega_0)} \abs{\alpha (\gamma)} ^{-1} \alpha (\gamma) = \abs{\epsilon_2(\gamma) } ^{-1} \epsilon_2(\gamma) = \sign(b). 
	\end{align}
	
	We are left to compute $\eta_2(\gamma)$. We have $\Phi_{M_2} ^+ - \Phi_M^+ = \set{\epsilon_2} \cup \set{\epsilon_2 \pm \epsilon_j \mid j\geq 3}$. Since $\epsilon_2 + \epsilon_j$ is the complex conjugate of $\epsilon_2 - \epsilon_j$ for $j\geq 3$, we have  
	\begin{align}\label{eq:computing eta odd}
	\eta_2 (\gamma) = \frac{\abs{1- \epsilon_2^{-1} (\gamma)}}{ 1- \epsilon_2^{-1} (\gamma) } = \sign(1-b^{-1}).
	\end{align}
	The proof is finished by combining (\ref{third calc}), (\ref{eq:epsilon odd}), (\ref{eq:part in partial odd}), and (\ref{eq:computing eta odd}).
	
	Now assume that $d$ is even. Then $\epsilon(\omega_0) = 1$. To finish the proof it suffices to check that
	$A(\gamma,\gamma') = \eta_2(\gamma) = 1.$
	
	We compute $A(\gamma,\gamma')$. Let $x_j : = \epsilon_j (\gamma), 1\leq j \leq m$. We have
	\begin{align}\label{eq:A1}
	\frac{\Delta_M(\gamma)}{\Delta_{M} (\gamma')} &  = \prod_{\alpha \in \Phi_M^+} \frac{1-\alpha^{-1} (\gamma)}{1-\alpha^{-1} (\gamma')}  = \prod_{\alpha \in \set{\epsilon_3 \pm \epsilon_j \mid j\geq 4}} \frac{1-\alpha^{-1} (\gamma)}{1-\alpha^{-1} (\gamma')} \\ \nonumber  & =\prod_{j \geq 4} \frac{1- x_3^{-1} x_j^{-1} }{1- x_3 x_j^{-1}} \frac{1- x_3^{-1} x_j}{1- x_3 x_j} = \prod_{j\geq 4} x_3^{-2}. 
	\end{align}
	Also
	\begin{align}\label{eq:A_2}
	\prod_{\alpha\in \Phi_M^+} \frac{\abs{\alpha(\gamma')}^{-1/2 }}{\abs{\alpha(\gamma)}^{-1/2}} = \prod_{\alpha \in \set{\epsilon_3 \pm \epsilon_j\mid j \geq 4}}\frac{\abs{\alpha(\gamma')}^{-1/2 }}{\abs{\alpha(\gamma)}^{-1/2}}  = \prod_{j\geq 4} \abs{\frac{x_3^{-1} x_j}{x_3 x_j} \frac{x_3^{-1} x_j^{-1} }{x_3 x_j^{-1}}} ^{-1/2} = \prod_{j\geq 4} \abs{x_3} ^2. 
	\end{align} 
	To compute $$\prod_{\alpha \in \Phi(\omega_0)}\alpha (\gamma) \abs{\alpha(\gamma)}^{-1}, $$ we have  $$\Phi(\omega_0) = \set{\epsilon_2 \pm \epsilon_j \mid j\geq 3} \cup  \set{\epsilon_3 \pm \epsilon_j \mid j\geq 4}.$$ Note that $\epsilon_2 +\epsilon_j$ is the complex conjugate of $\epsilon_2 - \epsilon_j$, for $j\geq 3$. Hence we have 
	\begin{align}\label{eq:A3}
	\prod_{\alpha \in \Phi(\omega_0)}\frac{\alpha (\gamma)} {\abs{\alpha(\gamma)}}  = \prod_{\alpha \in \set{\epsilon_3 \pm \epsilon_j \mid j \geq 4}}\frac{\alpha (\gamma)} {\abs{\alpha(\gamma)}}  =\prod_{ j \geq 4} \frac{x_3 x_j x_3 x_j^{-1}}{\abs{x_3 x_j x_3 x_j^{-1} }} = \prod_{ j \geq 4} \frac{ x_3^2}{\abs{x_3} ^2}. 
	\end{align}
	Combining (\ref{eq:A1}) (\ref{eq:A_2})  (\ref{eq:A3}), we conclude that $A(\gamma, \gamma') = 1$, as desired. 
	
	We are left to check that $\eta_2(\gamma) = 1$. 
	We have $\Phi_{M_2} ^+ - \Phi_M^+ = \set{\epsilon_2 \pm \epsilon_j\mid j\geq 3}$. As we observed before, $\epsilon_2 + \epsilon_j$ is the complex conjugate of $\epsilon_2 - \epsilon_j$ for all $j\geq 3$. Hence 
	$\eta_2 (\gamma)= 1$ as desired.    
\end{proof}

\subsection{}\label{subsubsec:double bracket} Keep the setting of \S \ref{subsubsec:KH M12}. In the following we compare $L_M(\gamma)$ with $\Phi^G_M(\gamma,\Theta)$. We will also introduce and study a variant of $\Phi^G_M(\gamma,\Theta)$, denoted by $\Phi^G_{M}(\gamma,\Theta)_{\Endos}$.

We have $A_M = M^{\GL} = \GG_m \times \GG_m$, and $T_{12}( \RR)_1 = \set{\pm 1} \times \set{\pm 1} \times T_{W_2} (\RR)$. The projection of $\gamma$ in $A_M (\RR) ^0 \cong \RR_{>0} \times \RR_{>0}$ is $(\abs{a} , \abs{b})$, and $$x_{\gamma} = (\log \abs{a}, \log \abs{b})\in \RR^2  \cong \Lie(A_{M}) = X_*(A_{M})_{\RR}. $$ 

Let $\wp $ be the natural restriction map $ X^*(T_{12})_{\RR}\to X^*(A_M)_{\RR}$. We identify $X^*(A_M)_{\RR}$ with $\RR^2$, and let $\RR^2_{\odd}, 
\RR^2_{\Endos}, \RR^2_{\even}$ be the subsets of $\RR^2$ defined in \S \ref{subsubsec:three cases of c}. Note that when $d$ is odd (resp.~even), we have $\wp ( \omega (\lambda +\rho)) \in \RR^2_{\odd}$ (resp.~$\in \RR^2_{\even}$) for all $\omega \in \Omega$. Suppose $f$ is a function $\RR^2_{\odd} \to \CC$ (resp.~$\RR^2_{\even} \to \CC$) when $d$ is odd (resp.~even). We write $\dbp{f}$ for the function \index{$\dbp{f}$}
\begin{align*}\dbp {f}: 
\Omega & \To \CC \\
\omega & \longmapsto  f (\wp ( \omega (\lambda +\rho))).
\end{align*}

Recall from \S \ref{subsubsec:three cases of c} that $\RI, \RII, \cdots, \RVIII,\RA$ denote the characteristic functions of some open cones in $\RR^2$.  
\begin{lem}\label{lem:N(omega)} When $d$ is odd, we have the following identities between functions on $\Omega$: 
	\begin{align*}
	N_1(\cdot) & = \dbp{\RI + \RII + \RVIII},\\
	N_2(\cdot) & = \dbp{\RI + \RVII + \RVIII}, \\
	N_3(\cdot) & = \dbp{\RI + \RII + \RVII + \RVIII}.
	\end{align*} When $d$ is even, we have the following identities between functions on $\Omega$: \begin{align*}
	N_1(\cdot) & = \dbp{\RA + \RII},\\
	N_2(\cdot) & = \dbp{\RA + \RVII}, \\
	N_3(\cdot) & = \dbp{\RA + \RII + \RVII}. 
	\end{align*}
\end{lem}
\begin{proof}
	This follows immediately from Definition \ref{defn:N_i}.
\end{proof}
\subsection{}
Recall from \S \ref{subsec:stable discrete series chars} that $\Phi^G_M(\gamma,\Theta)$ can be computed by (\ref{eq:to explain notation}). Using the notation $\dbp{f}$ introduced in \S \ref{subsubsec:double bracket}, we recall the definition of $n(\gamma,\omega B)$ appearing in (\ref{eq:to explain notation}) as follows:
$$n(\gamma,\omega B) : = \begin{cases}
\dbp{\bar c_{R_{\gamma}} (x_{\gamma}, \cdot)} (\omega), &  \text{if }
\gamma \in  G(\RR) ^0 ,\\
0, & \text{if } \gamma \notin G(\RR)^0 .
\end{cases} $$ 	

Let $R_{\Endos} := \set{ \pm \epsilon_1, \pm \epsilon_2 } \subset X^*(A_{M}) _{\RR}$. Under the identification $X^*(A_{M}) _{\RR} \cong \RR^2$, the subset $R_{\Endos}$ is identified with the root system $U_{\Endos}$ considered in \S \ref{subsubsec:three cases of c}. In particular, the Weyl group of $R_{\Endos}$ contains $-1$, and the function $\bar c_{R_{\Endos}}$ associated to $R_{\Endos}$ is identified with the function $\bar c_{U_{\Endos}} : \RR^2_{\Endos} \times \RR^2_{\Endos}\to \ZZ$ considered in \S \ref{subsubsec:three cases of c}.

 \emph{When $d$ is odd}, we define 
 \begin{align}\label{eq:n_eds}
n_{\Endos}(\gamma,\omega B) : = \begin{cases} \dbp{\bar c_{R_{\Endos}} (x_{\gamma}, \cdot)} (\omega), & \text{if } a, b >0 , \\
	0, &  \mbox{otherwise},
\end{cases} 
 \end{align}
for $\omega \in \Omega$. Here $\dbp{\bar c_{R_{\Endos}} (x_{\gamma}, \cdot) }$ is well defined because $\RR^2_{\odd} \subset \RR^2_{\Endos}$. 

Analogous to (\ref{eq:to explain notation}), we define, \emph{when $d$ is odd},  \index{$\Phi^G_M(\cdot,\Theta) _{\Endos}$}
\begin{multline}\label{eq:Phi^G_M endos}
\Phi^G_M (\gamma ,\Theta) _{\Endos} : =  (-1)^{q(G)} \epsilon_{R} (\gamma) \delta_{P_{12}(\RR)} (\gamma)^{1/2}\Delta_M(\gamma)^{-1}  \\  \cdot \sum_{\omega \in \Omega } \epsilon(\omega) n_{\Endos}(\gamma, \omega B) (\omega \lambda)(\gamma) \prod _{\alpha \in \Phi(\omega)} \alpha^{-1} (\gamma). \end{multline}

\begin{lem}\label{lem:in conn}
	For both parity of $d$, let $\nu = (s,t,u) \in T_{12}(\RR) = \RR^{\times} \times \RR^{\times} \times T_{W_2}(\RR)$ be an element with $s,t <0$. Then $\nu \in G(\RR)^0$. 
\end{lem}
\begin{proof}
	Since $T_{W_2} (\RR)$ is connected (being a product of copies of $\Uni(1)$), we know that $\nu$ is in the same connected component of $G(\RR)$ as $$\nu_1 : = (-1,  -1 ,  1) \in T_{12}(\RR).$$ It remains to show that $\nu_1 \in G(\RR)^0$. We know that $\nu_1$ acts as $-1$ on $\RR X_1 + \RR X_2$ and on $\RR Y_1 + \RR Y_2$, where $X_i = e_i + e_i'$ and $Y_i = e_i- e_i '$. Now $\RR X_1 + \RR X_2$ is a positive definite plane and $\RR Y_1 + \RR Y_2$ is a negative definite plane, and $\nu_1$ acts on both of them with determinant $1$. Also $\nu_1$ acts as the identity on the orthogonal complement of these two planes. This implies that $\nu_1 \in G(\RR) ^0$, by the standard description of the connected components of indefinite special orthogonal groups (see \cite[I.17]{knappbeyond}).
\end{proof}
\begin{prop}\label{Arch comp odd ++}
	Assume that $d$ is odd. Let $\gamma = (a,b, \gamma_{W_2})\in T_{12}(\RR)$ be a regular element. Let $(x_1,x_2) =  (\log \abs{a}, \log \abs{b}).$ When $ab<0$, we have $$\Phi ^G_M(\gamma, \Theta) = \Phi ^G_M (\gamma, \Theta) _{\Endos} =0.$$ When $ab > 0$, assume that $x_1 < -\abs{x_2}.$ Then we have 
	$$ 4 (-1) ^{q(G)} L_M(\gamma)  = 
	\Phi_M^G (\gamma ,\Theta)  + \Phi_M^G (\gamma ,\Theta)_{\Endos}.$$
\end{prop}
\begin{proof} 
	We first treat the case $ab<0$. Then $\Phi ^G_M(\gamma,\Theta) _{\Endos} = 0 $ since all $n_{\Endos}(\gamma,\omega B)$ vanish by definition. To show $\Phi ^G_M(\gamma ,\Theta) =0$, note that $R_{\gamma} = \set{\pm \epsilon_1}$ or $\set{\pm \epsilon_2}$. Thus the Weyl group of $R_{\gamma}$ (as a root system in  $X^*( A_{M}) _{\RR}= \RR^2$) does not contain $-1$. By Corollary \ref{cor:GKM}, we have $\Phi^G_M(\gamma,\Theta) =0$.
	
	We now treat the case $ab>0$. First assume that $a$ and $b$ are both positive. Under our assumption that $x_1 < -\abs{x_2}$, there are two cases to consider, namely $0< a < b <1$ or $0< ab <1 < b$. (Here $b \neq 1$ since $\gamma$ is regular.) We have 
	$$\epsilon_R(\gamma) = \begin{cases}
	1,  & \text{ if }0<a<b<1 ,\\
	-1, & \text{ if }0 < ab < 1<b.
	\end{cases}$$  
	Comparing Lemma \ref{Computing L_M for M_12} with (\ref{eq:to explain notation} and (\ref{eq:Phi^G_M endos}), we see that the current proposition reduces to the following two statements:  
	\begin{itemize}
		\item 	When $0< a < b < 1,$ we have 
		\begin{align}\label{eq:statement 1}
		\frac{1}{4} (n(\gamma ,\omega B) + n_{\Endos} (\gamma,\omega B) ) = N_1 (\omega) - N_2(\omega) +N_3(\omega), \quad \forall \omega \in \Omega .
		\end{align}
		\item 	When $ 0<ab < 1 <b,$ we have 
		\begin{align}\label{eq:statement 2}
		\frac{1}{4} (n(\gamma ,\omega B) + n_{\Endos} (\gamma,\omega B) ) = -N_1 (\omega) + N_2(\omega) +N_3(\omega),  \quad \forall \omega \in \Omega .
		\end{align}
	\end{itemize}
	
	Since obviously $\gamma \in T_{12} (\RR) ^0 \subset G(\RR) ^0$, we have $$n(\gamma , \omega B) = \dbp{ \bar c_{R_{\gamma}} (x_{\gamma}, \cdot)}(\omega), \quad \forall \omega \in \Omega ,$$ by definition. Since $R_{\gamma} = \set{\pm \epsilon_1, \pm \epsilon_2, \pm \epsilon_1 \pm \epsilon_2} = U_{\odd}$, we have $\bar c_{R_{\gamma}} (x_{\gamma}, \cdot) =\mathbf f_{\odd,x_{\gamma}}$. (See \S \ref{subsubsec:three cases of c} for the notation.) In other words we have 
	\begin{align}\label{eq:n and f, odd}
n(\gamma, \omega B) = \dbp{\mathbf f_{\odd,x_\gamma}}(\omega), \quad \forall \omega \in \Omega.
	\end{align}
Similarly we have 
\begin{align}\label{eq:n and f, eds}
n_{\Endos}(\gamma, \omega B) = \dbp{\mathbf f_{\Endos,x_\gamma}}(\omega), \quad \forall \omega \in \Omega.
\end{align}
	When $0<a<b <1$, we have $x_\gamma \in \RV$. By (\ref{eq:odd V}), (\ref{eq:eds V}), (\ref{eq:n and f, odd}), and (\ref{eq:n and f, eds}), we have \begin{align*}
	\frac{1}{4} n(\gamma,\omega B) &= \dbp{ \RII + \RVIII}(\omega ), \\
	\frac{1}{4} n_{\Endos}(\gamma, \omega B ) &  =\dbp{\RI + \RII}(\omega).
	\end{align*} Thus the LHS of (\ref{eq:statement 1}) is equal to $\dbp{\RI +2\RII +\RVIII}(\omega)$. On the other hand, by Lemma \ref{lem:N(omega)}, the RHS of (\ref{eq:statement 1}) is also equal to $\dbp{\RI +2\RII +\RVIII}(\omega)$. Hence (\ref{eq:statement 1}) holds, as desired.
	
	When $0<ab<1<b$, we have $x_\gamma \in \RIV$. By (\ref{eq:odd IV}) and (\ref{eq:eds IV}), we have \begin{align*}
	\frac{1}{4} n(\gamma, \omega B) &= \dbp{\RI + \RVII}(\omega), \\
	\frac{1}{4} n_{\Endos}(\gamma, \omega B ) &  =\dbp{\RVII + \RVIII}(\omega).
	\end{align*}Thus the LHS of (\ref{eq:statement 2}) is equal to $\dbp{\RI +2\RVII +\RVIII}(\omega)$. By Lemma \ref{lem:N(omega)}, the RHS of (\ref{eq:statement 2}) is also equal to $\dbp{\RI +2\RVII +\RVIII}(\omega)$. Hence (\ref{eq:statement 2}) holds, as desired. 
	
	We now assume that $a$ and $b$ are both negative. In this case $\Phi ^G_M(\gamma,\Theta)_{\Endos} =0$ by definition. We have $\epsilon_R(\gamma) = 1$. Comparing Lemma \ref{Computing L_M for M_12} with (\ref{eq:to explain notation}), we see that the current proposition reduces to the following identity:
	\begin{align}\label{eq:statement 3}
	\frac{1}{4} n(\gamma ,\omega B) = N_1 (\omega) + N_2(\omega) -N_3(\omega) , \quad \forall \omega \in \Omega. 
	\end{align}
	
	By Lemma \ref{lem:in conn}, we have $\gamma \in G(\RR)^0$, and so $$n(\gamma , \omega B) =\dbp{\bar c_{R_{\gamma}} (x_{\gamma}, \cdot)}(\omega), \quad \forall \omega \in \Omega,$$ by definition. Since $R_{\gamma}  = \set{\pm \epsilon_1 \pm\epsilon_2} = U_{\even}$, we have $\bar c_{R_{\gamma}} (x_{\gamma},\cdot) = \mathbf f_{\even,x_{\gamma}}.$ (See \S \ref{subsubsec:three cases of c} for the notation). Thus 
	\begin{align}\label{eq:n and f, quasi-even}
n(\gamma , \omega B) = \dbp{\mathbf f_{\even ,x_{\gamma}}}(\omega), \quad \forall \omega \in \Omega.
	\end{align}
	Since $x_1 < -\abs{x_2}< 0 $, we have $x_ \gamma \in \RIV \cup \RV \subset \RC$. Hence by (\ref{eq:even C}) and (\ref{eq:n and f, quasi-even}), we have $$\frac{1}{4} n(\gamma, \omega B) = \dbp{\RA} (\omega ) = \dbp{\RI + \RVIII}(\omega).$$ By Lemma \ref{lem:N(omega)}, the RHS of (\ref{eq:statement 3}) is also equal to $\dbp{\RI + \RVIII}(\omega)$. Hence (\ref{eq:statement 3}) holds, as desired.\end{proof}

The following proposition will also be needed in \S \ref{pf:break 1} below.
\begin{prop}\label{effect of switching to Phi_endos}
	Assume that $d$ is odd. Let $\gamma = (a,b , \gamma_{W_2}) \in T_{12} (\RR)$ be a regular element, with $ab > 0$. Let $\omega_1, \omega_2$ be as in Definition \ref{defn:omega_0}, and let  \begin{align*}
	\gamma' & : =  \omega_1 (\gamma) = (b,a, \gamma_{W_2}) \in T_{12} (\RR) ,\\
	\gamma'' & : = \omega_2 (\gamma) = (a^{-1},b, \gamma_{W_2}) \in T_{12} (\RR).
	\end{align*} Then we have 
	\begin{align}\label{eq:invariance1}
	\Phi ^G_M ( \gamma ,\Theta) & = \Phi ^G_M (\gamma' ,\Theta)  = \Phi ^G_M (\gamma'' ,\Theta) , \\\label{eq:invariance2}
	\epsilon_R (\gamma) \epsilon_{R_{\Endos}} (\gamma)\Phi^G_M (\gamma, \Theta) _{\Endos} & = -  \epsilon_R (\gamma') \epsilon_{R_{\Endos}} (\gamma')\Phi^G_M (\gamma' , \Theta )_{\Endos}, \\ \label{eq:invariance3}
	\epsilon_R (\gamma) \epsilon_{R_{\Endos}} (\gamma)\Phi^G_M (\gamma, \Theta) _{\Endos} & =   \epsilon_R (\gamma'') \epsilon_{R_{\Endos}} (\gamma'')\Phi^G_M (\gamma'' , \Theta )_{\Endos}. 
	\end{align}
	Here $\epsilon_{R_{\Endos}}(\gamma)$ is defined to be
	$$ (-1)^{\# \set{\alpha \in \Phi^+ \cap R_{\Endos} \mid  0 < \alpha(\gamma)<1 } }, $$
	and similarly for $\epsilon_{R_{\Endos}}(\gamma')$ and $\epsilon_{R_{\Endos}}(\gamma'')$.
\end{prop}
\begin{proof}The equalities in (\ref{eq:invariance1}) hold because $\omega_1$ and $\omega_2$ can be represented by elements of $(\Nor_G M)(\RR)$, and $\Phi^G_M(\cdot, \Theta)$ is invariant under $(\Nor_G M)(\RR)$. 
	
	We now prove (\ref{eq:invariance2}). We have 
	$\Delta_M(\gamma)  = \Delta_M(\gamma') $ because $\gamma^{-1} \gamma'$ lies in the center of $M$. Also
	$
	\epsilon_{R_{\Endos}} (\gamma)  = \epsilon_{R_{\Endos}} (\gamma') .$
	Hence we have reduced the proof to showing that  
	\begin{align}\label{eq:exchange a b}
	\delta_{P_{12} (\RR)} (\gamma) ^{1/2} \sum_{ \omega } &  \epsilon(\omega) n_{\Endos} (\gamma,\omega B) (\omega \lambda) (\gamma) \prod_{ \alpha \in \Phi(\omega)} \alpha^{-1} (\gamma) \\ \nonumber  & = - \delta_{P_{12} (\RR)} (\gamma') ^{1/2} \sum_{ \omega } \epsilon(\omega) n_{\Endos} (\gamma',\omega B) (\omega \lambda) (\gamma') \prod_{ \alpha \in \Phi(\omega)} \alpha^{-1} (\gamma') .\end{align}
	
	We claim that for all $\omega \in \Omega$, we have $n_{\Endos} (\gamma', \omega B) = n_{\Endos} (\gamma, \omega_1\omega B)$. Indeed, if $a$ and $b$ are both negative, then both sides are by definition zero. If $a$ and $b $ are both positive, then our claim follows from the following property: 
	\begin{align*}
	\bar c_{R_{\Endos}} (y,y') = \bar c_{R_{\Endos}}(\omega_1 y, \omega_1 y'), \quad \forall y, y ' \in \RR^2_{\Endos},
	\end{align*} which is a direct consequence of (\ref{eq:c_eds}). 
	
By the claim and Lemma \ref{lem:quot}, the RHS of (\ref{eq:exchange a b}) is equal to
	\begin{multline*}
 - \delta_{P_{12} (\RR)} (\gamma) ^{1/2} \sum_{ \omega } \epsilon(\omega) n_{\Endos} (\gamma,\omega_1 \omega B) (\omega \lambda) (\gamma') \\ \cdot \prod_{ \alpha \in \Phi(\omega_1\omega)} \alpha^{-1} (\gamma)  \prod_{\alpha \in \Phi_M^+} \frac{\abs{\alpha(\gamma')}^{-\frac{1}{2}} } { \abs{\alpha(\gamma)}^{-\frac{1}{2}} }  \prod_{\alpha \in \Phi(\omega_1)} \frac{\alpha(\gamma)} {\abs{\alpha(\gamma)}} .
	\end{multline*} 
Under the substitution $\omega \mapsto \omega_1\omega$ in the summation, the above becomes
\begin{multline*}
	\delta_{P_{12} (\RR)} (\gamma) ^{1/2} \sum_{ \omega } \epsilon(\omega) n_{\Endos} (\gamma,\omega B) (\omega \lambda) (\gamma)\\ \cdot  \prod_{ \alpha \in \Phi(\omega)} \alpha^{-1} (\gamma)  \prod_{\alpha \in \Phi_M^+} \frac{\abs{\alpha(\gamma')}^{-\frac{1}{2}} } { \abs{\alpha(\gamma)}^{-\frac{1}{2}} } \prod_{\alpha \in \Phi(\omega_1)} \frac{\alpha(\gamma)} {\abs{\alpha(\gamma)}}.
\end{multline*}
	To finish the proof of (\ref{eq:exchange a b}) it suffices to check
	$$\prod_{\alpha \in \Phi_M^+} \frac{\abs{\alpha(\gamma')}^{-\frac{1}{2}} } { \abs{\alpha(\gamma)}^{-\frac{1}{2}} }  \prod_{\alpha \in \Phi(\omega_1)} \frac{\alpha(\gamma)} {\abs{\alpha(\gamma)}} = 1.$$ Since $\gamma^{-1} \gamma'$ lies in the center of $M$, the first product in the above is equal to $1$. The second product is also equal to $1$, because $\Phi(\omega_1) =  \set{\epsilon_1 - \epsilon_2}$, and $(\epsilon_1- \epsilon_2) (\gamma) = a/b > 0$. We have thus proved (\ref{eq:exchange a b}). As we have already seen, this implies (\ref{eq:invariance2}).
	
	We now prove (\ref{eq:invariance3}) in a completely analogous way. We have 
	$ \Delta_M(\gamma)  = \Delta_M(\gamma')$, and $
	\epsilon_{R_{\Endos}} (\gamma)  = - \sign(a) \epsilon_{R_{\Endos}} (\gamma')$, so we need to check 
	\begin{multline}\label{eq:exchange a a^-1}
	\delta_{P_{12} (\RR)} (\gamma) ^{1/2} \sum_{ \omega }   \epsilon(\omega) n_{\Endos} (\gamma,\omega B) (\omega \lambda) (\gamma) \prod_{ \alpha \in \Phi(\omega)} \alpha^{-1} (\gamma) \\  = -\sign(a) \delta_{P_{12} (\RR)} (\gamma'') ^{1/2} \sum_{ \omega } \epsilon(\omega) n_{\Endos} (\gamma'',\omega B) (\omega \lambda) (\gamma'') \prod_{ \alpha \in \Phi(\omega)} \alpha^{-1} (\gamma'') .\end{multline}
	Again it easily follows from the definition of $n_{\Endos}$ and (\ref{eq:c_eds}) that $n_{\Endos} (\gamma'', \omega B) = n_{\Endos} (\gamma, \omega_2\omega B)$, for all $\omega \in \Omega$. By this fact and Lemma \ref{lem:quot}, the RHS of (\ref{eq:exchange a a^-1}) is equal to
	\begin{multline*}
- \sign(a) \delta_{P_{12} (\RR)} (\gamma) ^{1/2} \sum_{ \omega } \epsilon(\omega) n_{\Endos} (\gamma,\omega_2 \omega B) (\omega \lambda) (\gamma'') \\ \cdot  \prod_{ \alpha \in \Phi(\omega_2\omega)} \alpha^{-1} (\gamma)  \prod_{\alpha \in \Phi_M^+} \frac{\abs{\alpha(\gamma'')}^{-\frac{1}{2}} } { \abs{\alpha(\gamma)}^{-\frac{1}{2}} }  \prod_{\alpha \in \Phi(\omega_2)} \frac{\alpha(\gamma)} {\abs{\alpha(\gamma)}} .
	\end{multline*}
Under the substitution $\omega \mapsto \omega_2\omega$ in the summation the above becomes 
	\begin{multline*}
\sign(a)\delta_{P_{12} (\RR)} (\gamma) ^{1/2} \sum_{ \omega } \epsilon(\omega) n_{\Endos} (\gamma,\omega B) (\omega \lambda) (\gamma) \\ \cdot  \prod_{ \alpha \in \Phi(\omega)} \alpha^{-1} (\gamma)  \prod_{\alpha \in \Phi_M^+} \frac{\abs{\alpha(\gamma'')}^{-\frac{1}{2}} } { \abs{\alpha(\gamma)}^{-\frac{1}{2}} } \prod_{\alpha \in \Phi(\omega_2)} \frac{\alpha(\gamma)} {\abs{\alpha(\gamma)}}. 
	\end{multline*}
	To finish the proof of (\ref{eq:exchange a a^-1}), it suffices to check 
	$$\prod_{\alpha \in \Phi_M^+} \frac{\abs{\alpha(\gamma'')}^{-\frac{1}{2}} } { \abs{\alpha(\gamma)}^{-\frac{1}{2}} }  \prod_{\alpha \in \Phi(\omega_2)} \frac{\alpha(\gamma)} {\abs{\alpha(\gamma)}} = \sign(a).$$ Again the first product in the above is equal to $1$, so we need to check 
	that the second product is equal to $\sign (a).$ For this, we may replace the product over all $\alpha \in \Phi(\omega_2)$ by the product over those $\alpha \in \Phi(\omega_2)$ that are real. This is because $\Phi(\omega_2)$ is stable under complex conjugation, and we obviously have $$ \frac{\alpha (\gamma)} {\abs{\alpha (\gamma)}} \frac{\bar \alpha (\gamma)} {\abs{\bar \alpha (\gamma)}} =1$$ for any $\alpha, \bar \alpha \in \Phi(\omega_2)$ that are complex conjugate to each other. Now the real roots in $\Phi(\omega_2)$ are $\epsilon_1, \epsilon_1 + \epsilon_2, \epsilon_1 - \epsilon_2$. Hence
	$$	\prod_{\alpha \in \Phi(\omega_2)} \frac{\alpha (\gamma)} {\abs{\alpha (\gamma)}}  = \prod_{\alpha \in \set{ \epsilon_1, \epsilon_1 + \epsilon_2, \epsilon_1 - \epsilon_2}} \frac{\alpha (\gamma)} {\abs{\alpha (\gamma)}} = \frac{a}{\abs{a}} \frac{ab}{\abs{ab}} \frac{ab^{-1}}{\abs{ab^{-1}}} =  \frac{a}{\abs{a}} = \sign(a) ,$$ as desired. We have thus proved (\ref{eq:exchange a a^-1}). As we have already seen, this implies (\ref{eq:invariance3}).
\end{proof}

The following proposition is the counterpart of Proposition \ref{Arch comp odd ++} in the even case.
\begin{prop}\label{Arch comp M12 even}
	Assume that $d$ is even. Let $\gamma = (a,b, \gamma_{W_2})\in T_{12}(\RR)$ be a regular element. Let $(x_1,x_2) =  (\log \abs{a}, \log \abs{b}).$  When $ab < 0$, we have $$ \Phi ^G_M(\gamma,\Theta) = 0.$$
	When $ab>0$, assume that $x_1 < -\abs{x_2}$. Then we have $$ 4 (-1) ^{q(G)} L_M(\gamma) = 
	\Phi_M^G (\gamma ,\Theta) .$$
\end{prop}
\begin{proof} 
	When $ab<0 $, we have $R_{\gamma} = \emptyset$. Thus $\Phi^G_M(\gamma,\Theta) =0$ by Corollary \ref{cor:GKM}. 
	
	Assume that $ab >0$. Under our assumption that $x_1 <  -\abs{x_2}$, we have $\epsilon_R(\gamma) = 1$. 
	In view of Lemma \ref{Computing L_M for M_12}, to prove the current proposition it suffices to prove 
	\begin{align}\label{eq:statement even}
	\frac{1}{4} n(\gamma ,\omega B) = N_1 (\omega) + N_2(\omega) -N_3 (\omega) , \quad \forall \omega \in \Omega.  
	\end{align}
	
	By Lemma \ref{lem:in conn}, we have $\gamma \in G(\RR)^0$, and so $$n(\gamma , \omega B) =\dbp{\bar c_{R_{\gamma}} (x_{\gamma}, \cdot)}(\omega), \quad \forall \omega \in \Omega,$$ by definition. Since $R_{\gamma}  = \set{\pm \epsilon_1 \pm\epsilon_2} = U_{\even}$, we have $\bar c_{R_{\gamma}} (x_{\gamma},\cdot) = \mathbf f_{\even,x_{\gamma}}$. (See \S \ref{subsubsec:three cases of c} for the notation). Thus 
	\begin{align}\label{eq:n and f, even}
n(\gamma , \omega B) =\dbp{\mathbf f_{\even ,x_{\gamma}}}(\omega), \quad \forall \omega \in \Omega.
	\end{align}
	Since $x_1 < -\abs{x_2}$, we have $x_\gamma \in \RC$. By (\ref{eq:even C}) and (\ref{eq:n and f, even}), we have $$\frac{1}{4} n(\gamma, \omega B) =\dbp{\RA}(\omega ).$$ Now by Lemma \ref{lem:N(omega)}, the RHS of (\ref{eq:statement even}) is also equal to $ \dbp{\RA}(\omega )$. Hence (\ref{eq:statement even}) holds, as desired. 	
\end{proof}

\chapter{Endoscopic data for special orthogonal groups}\label{sec:endoscopic}
In this chapter, let $F$ be a local or global field of characteristic zero. Let $V = (V,q)$ be a quadratic space over $F$ of dimension $d$ and discriminant $\delta$ (see \S \ref{subsec:generalities on quad sp}). Let $G = \SO(V)$. Let $m = \floor{d/2}$, which is the absolute rank of $G$. As usual, we refer to ``the odd case'' and ``the even case'' according to the parity of $d$. 
\section{The quasi-split inner form} \label{Fixing inner twist}

We need to explicitly fix an inner twisting between $G$ and a quasi-split inner form. For this, let $\underline V = (\underline V, \underline q)$ be the unique (up to isomorphism) quasi-split quadratic space over $F$ of dimension $d$ and discriminant $\delta$. 

 \begin{defn}\label{defn:fixing isom between quad spaces}
 	We fix an isomorphism of quadratic spaces over $\overline F$:\index{$\phi_V$} $$\phi_V: (V,q)\otimes_F \overline F \isom (\underline{V},\underline{q})\otimes_F \overline F.$$ If $F= \RR$, we may and shall assume that $\phi_V$ satisfies the following condition: Let $(a,b)$ be the signature of $(V,q)$. If $a > b$ (resp.~$a\leq b$), then there exists an orthogonal basis $\set{v_1, \cdots , v_d}$ of $V$, and an orthogonal basis $\set{\underline{v}_1, \cdots, \underline v_d}$ of $\underline{ V}$, such that for each $1\leq j \leq d$, we have $q(v_j), \underline q (\underline{ v}_j) \in \set{\pm 1}$, and $\phi_V(v_j) =  \underline v_j \otimes \lambda_j $ for some $\lambda_j \in \set{1, \sqrt{-1}}$, with $\lambda_j = \sqrt{-1}$ only if $q(v_j) = 1$ (resp.~only if $q(v_j) = -1$).     
 \end{defn}	
 \subsection{}\label{subsubsec:defn of the cocycle u}
Let $G^*:= \SO(\underline{ V} , \underline q)$, which is quasi-split over $F$ by Proposition \ref{prop: even TFAE}. Using $\phi_V$ as in Definition \ref{defn:fixing isom between quad spaces}, we define the isomorphism\index{$\psi_V$} 
\begin{align*}
\psi_V: G_{\overline F} &  \isom G^*_{\overline F} \\
g & \longmapsto \phi_V g \phi_V^{-1}.
\end{align*}
 Define the function \index{$u_V$}
 \begin{align*}
u_V: \Gamma_F & \To \GL(\underline V \otimes_F \overline F) \\ \rho & \longmapsto \leftidx^\rho\phi_V \phi_V^{-1}.
 \end{align*}
 Clearly the image of $u_V$ is contained in $\mathrm O(\underline V)(\overline F)$. If we fix $F$-bases of $V$ and $\underline V$, then since $q$ and $\underline q$ have the same discriminant, the square of the determinant of the matrix of $\phi_V$ lies in $F^{\times,2}$, which implies that the determinant of the matrix of $\phi_V$ lies in $F$. Hence $u_V(\rho)$ has determinant $1$ for each $\rho \in \Gamma_F$. Thus the image of $u_V$ is contained in $G^*(\overline F)$. Note that we have 
 \begin{align}\label{eq:pure inner twist}
\leftidx^{\rho}\psi_V \psi_V^{-1} = \Int( u_V(\rho)) \in \Aut(G^*_{\overline F}), \quad \forall \rho \in \Gamma_F.
 \end{align}
  It follows that $\psi_V$ is an inner twisting. 
 \begin{rem}\label{rem:remember phi_V}
 If we view $\SO(V)$ and $\SO(\underline V)$ as abstract reductive groups over $F$, then in the odd case there is a unique $\SO(\underline V)(\overline F)$-conjugacy class of inner twistings $\SO(V)_{\overline F}\isom\SO(\underline V)_{\overline F}$, whereas in the even case there are two such conjugacy classes, interchanged under the conjugation by any element of $\mathrm{O}(\underline V)(\overline F) - \SO(\underline V)(\overline F)$. If we change the choice of $\phi_V$ to some $\phi_V'$, then $\phi_V' = g \circ \phi_V$ for some $g\in \mathrm{O}(\underline V)(\overline F)$. The inner twisting $\psi_V'$ arising from $\phi_V'$ stays in the same $\SO(\underline V)(\overline F)$-conjugacy class as $\psi_V$ if and only if $g\in \SO(\underline V)(\overline F)$. Thus for the purpose of realizing $G^*=\SO(\underline V)$ as an inner form of $G$, it suffices to remember $\phi_V$ up to replacing it by $g \circ \phi_V$ for $g\in \SO(\underline V)(\overline F)$. 
 \end{rem}
 \begin{rem}\label{rem:pure inner form}
 	The pair $(\psi_V,u_V )$ realizes $G$ as a \emph{pure inner form}\index[n]{pure inner form} of $G^*$ in the sense of Vogan \cite{voganllc}; cf.~the introduction of \cite{kalrigid}. The pair  $(\psi_V,u_V )$ itself is called a \emph{pure inner twist}\index[n]{pure inner twist}; cf.~\cite[\S 2]{kaldepth0}. Fixing such a pure inner twist (or rather its $G^*(\overline F)$-conjugacy class, see below) is more refined than just fixing $G^*$ as an inner form of $G$, and it plays an essential role in normalizing transfer factors when $F$ is a local field. Specifically, suppose $(H,\lang H, s, \eta)$ is an elliptic endoscopic datum for $G$, and suppose we have fixed a normalization of transfer factors between $H$ and $G^*$. Then the datum $(\psi_V, u_V)$ allows one to ``transport'' that normalization to a  normalization of transfer factors between $H$ and $G$, as observed by Kottwitz and explained in \cite[\S 2.2]{kaldepth0}.  For this purpose, it actually suffices to just remember $\phi_V$ up to replacing it by $g\circ \phi_V$ for $g\in G^*(\overline F) = \SO(\underline V)(\overline F)$, which will result in $(\psi_V, u_V)$ being replaced by $ (\Int(g)\circ \psi_V , \rho \mapsto \lix ^{\rho} g u_V(\rho) g^{-1}) $ and will not change the transported normalization between $H$ and $G$. By contrast, if one abstractly modifies $(\psi_V,u_V)$ by keeping $\psi_V$ unchanged and replacing $u_V$ by $ \rho \mapsto \lix ^{\rho} g u_V(\rho) g^{-1}$ for some $g\in G^*(\overline F)$, the resulting normalization of transfer factors between $H$ and $G$ can change, as observed in \cite[\S 1.11 (4)]{walds10}. 
 \end{rem}

\begin{defn}\label{defn:fixing alpha}
	When $d$ is even and $\delta$ is trivial, we fix an $\SO(\underline V)(F)$-orbit of hyperbolic bases (Definition \ref{defn:qsplitting}) of $\underline V$ once and for all, denoted by $[\mathbb B_{\underline V}]$.\index{$[\mathbb B_{\underline V}]$} When $d$ is even and $\delta$ is non-trivial, we fix $\alpha \in  \overline F$ such that $x = \alpha^2 \in F^{\times}$ is a lift of $\delta$, and we fix an $\SO(\underline V)(F)$-orbit $[\mathbb B_{\underline V}]$ of near-hyperbolic bases of $\underline V$ such that all members of this orbit have discriminant $x$ (see Definition \ref{defn:qsplitting}). If $F = \RR$, we identify $\overline F$ with $\CC$ and take $\alpha = \sqrt{-1}$.
\end{defn}
 \section{Some matrix groups over \texorpdfstring{$\CC$}{C}}\label{subsec:matrix groups}
 \subsection{}\label{subsubsec:Join notation}
 We define some algebraic groups over $\CC$, which we also identify with their $\CC$-points. Let $N$ be a positive even integer. Let $\set{\hat e_k \mid 1\leq k \leq N}$\index{$\hat e_k$} be the standard basis of $\CC^N$. Define two $N\times N$ matrices \index{$I_N^-, I_N^+$}
 $$ I_N^-  : =  \begin{pmatrix}
 &&&& - 1 \\
 &&& 1 \\
 && -1 \\
 & \reflectbox{$\ddots$} \\
 1 
 \end{pmatrix},  \qquad I_N^+ : = \begin{pmatrix}
  I_{N/2} & \\ & - I_{N/2}
  \end{pmatrix} I_N^-.$$
Thus $I_N^+$ and $I_N^-$ define a quadratic form and a symplectic form on $\CC^N$ respectively. We use these forms to define the groups $\mathrm{O}_N(\CC)$, $\SO_N(\CC)$, and $\Sp_N(\CC)$, as subgroups of $\GL_N(\CC)$. By convention, $\SO_0 (\CC) = \Sp_0(\CC) = \GL_0 (\CC)  = \set{1}$.
 
 We introduce a short-hand notation in order to conveniently denote certain diagonal matrices. For $x_1,\cdots, x_n \in \CC^{\times}$, we write $\DS(x_1,\cdots, x_n)$ \index{$\mathrm{symdiag}(x_1,\cdots,x_n)$} for the $2n \times 2n$ diagonal matrix $\diag(x_1,\cdots, x_n , x_n^{-1},\cdots, x_1^{-1})$.  
 \begin{defn}\label{defn:fixing borel pair in dual} Let $m =d/2$. In the reductive group $\Sp_N (\CC)$ (resp.~$\SO_N(\CC)$), we fix once and for all a Borel pair $(\mathcal T, \mathcal B)$\index{$(\mathcal T, \mathcal B)$}, together with an isomorphism $(\CC^{\times})^m \isom \mathcal T$, as follows. Let $\mathcal T$ be the intersection of $\Sp_N (\CC)$ (resp.~$\SO_N(\CC)$) with the diagonal torus in $\GL_N(\CC)$, and define the isomorphism $(\CC^{\times})^m  \isom \mathcal T$ by 
 	$$(t_1,\cdots, t_m) \longmapsto \DS(t_1,\cdots, t_m). $$ Using this isomorphism we identify $X^*(\mathcal T)$ and $X_*(\mathcal T)$ with $\ZZ^m$. The root datum of $\Sp_N (\CC)$ (resp.~$\SO_N(\CC)$) on $(X^*(\mathcal T), X_*(\mathcal T))$ is dual to the standard root datum $\mathrm{RD}(\mathsf B_m)$ (resp.~$\mathrm{RD}(\mathsf D_m)$) as in \S \ref{subsubsec:standard RD}. We define $\mathcal B$ by the condition that the based root datum $\mathrm{BRD}(\mathcal T, \mathcal B)$ is dual to the standard based root datum $\mathrm{BRD}(\mathsf B_m)$ (resp.~$\mathrm{BRD}(\mathsf D_m)$) as in \S \ref{subsubsec:standard RD}. We call $(\mathcal T , \mathcal B)$ the \emph{standard Borel pair}\index[n]{standard Borel pair (in the dual group)}.
 \end{defn}
\section{Fixing the \texorpdfstring{$L$}{L}-group}\label{fixing L group}
\subsection{}\label{subsubsec:L-group data}  Let $\mathrm{BRD}(G)$\index{$\mathrm{BRD}(G)$} be the \emph{canonical based root datum}\index[n]{canonical based root datum} of $G_{\overline F}$, namely the projective limit 
$$ \mathrm{BRD}(G)  = \varprojlim_{(T,B)} \mathrm{BRD}(T,B),$$ where $(T,B)$ runs through the Borel pairs in $G_{\overline F}$, and the transition maps are the canonical isomorphisms induced by inner automorphisms of $G_{\overline F}$. Since $G$ is defined over $F$, there is a canonical action of $\Gamma_F$ on $\mathrm{BRD}(G)$; see \cite[\S 1.3]{borelcorvallisarticle}. Recall that the \emph{$L$-group of $G$}\index[n]{$L$-group} consists of the following data (cf.~\cite[\S 2]{borelcorvallisarticle}, \cite[\S 1.2]{KS99}): 
\begin{enumerate}
	\item a reductive group $\widehat G$ over $\CC$.\index{$\widehat G$} 
	\item a Borel pair $(\mathcal T, \mathcal B)$ in $\widehat G$. 
	\item an action of $\Gamma_F$ on $\widehat G$ via algebraic automorphisms such that there exists a $\Gamma_F$-stable  splitting extending $(\mathcal T, \mathcal B)$. In particular, $\Gamma_F$ acts on the based root datum $\mathrm{BRD}(\mathcal T, \mathcal B)$.
	\item a $\Gamma_F$-equivariant isomorphism 
	\begin{align}\label{eq:frak v}
	\mathfrak v: \mathrm{BRD}(G) \isom \mathrm{BRD}(\mathcal T, \mathcal B) ^{\vee},
	\end{align} where $\mathrm{BRD}(\mathcal T, \mathcal B) ^{\vee}$ denotes the dual of $\mathrm{BRD}(\mathcal T, \mathcal B)$. 
\end{enumerate}
Given the above data, one defines\index{$\lang G$} \index{$\Gamma'$} $$\lang G: = \widehat G \rtimes \Gamma', $$  where $\Gamma'$ is taken to be one of the following groups depending on the context: If $F$ is a number field, we typically take $\Gamma'$ to be  $\Gamma_F$ or a sufficiently large finite quotient of it. When $F = \RR$, we typically take $\Gamma'$ to be the Weil group\index[n]{Weil group} $W_{\RR}$, which acts on $\widehat G$ through the map $W_{\RR} \to \Gamma_{\RR}$. When $F=\CC$ we take $\Gamma'$ to be trivial. (This case will never be considered in the paper.) When $F$ is a non-archimedean local field of characteristic zero, we typically take $\Gamma'$ to be the Weil group $W_{F}$\index{$W_{F}$} acting on $\widehat G$ through $W_F \hookrightarrow \Gamma_F$, or a sufficiently large (finite or infinite) quotient of $W_F$. Here ``sufficiently large'' always means that $\Gamma'$ should admit a quotient $\Gal(E/F)$, where $E/F$ is a Galois extension sufficiently large such that the $\Gamma_F$-action on $\widehat G$ in (3) above factors through $\Gal(E/F)$. As a result, $\Gamma'$ acts on $\widehat G$. For our specific $G$, this means that when $d$ is even and $\delta$ is non-trivial, $\Gamma'$ should admit $\Gal(F(\alpha)/F)$ as a natural quotient, where $\alpha$ is as in Definition \ref{defn:fixing alpha}. 

We have a canonical $\Gamma_F$-equivariant isomorphism between $\mathrm{BRD}(G)$ and $\mathrm{BRD}(G^*)$ (coming from the fixed $G^*(\overline F)$-conjugacy class of inner twistings $G_{\overline F} \isom G^*_{\overline F}$ represented by $\psi_V$). Thus if $\widehat G$ and $(\mathcal T, \mathcal B)$ are as in (1), (2), (3) above, then specifying $\mathfrak v$ as in (4) is equivalent to specifying a $\Gamma_F$-equivariant isomorphism 
\begin{align}\label{eq:frak v^*}
\mathfrak v^*: \mathrm{BRD}(G^*) \isom \mathrm{BRD}(\mathcal T, \mathcal B) ^{\vee}.
\end{align}
In other words, fixing an $L$-group of $G$ is equivalent to fixing an $L$-group of $G^*$. 

\subsection{}\label{subsubsec:L group}
We now explicitly present the $L$-group of $G$. 
We take $\widehat G$\index{$\widehat G$} to be $\Sp_{d-1}(\CC)$ (resp.~$\SO_d(\CC)$) as in \S \ref{subsec:matrix groups} if $d$ is odd (resp.~even). Define the action of $\Gamma_F$ on $\widehat G$ as follows. The action is trivial unless $d$ is even and $\delta$ is non-trivial. In the latter case, we define the action to factor through $\Gamma_F \to \Gal(F(\alpha)/F)$ (see Definition \ref{defn:fixing alpha} for $\alpha$), and let the non-trivial element of $\Gal(F(\alpha)/F)$ act on $\widehat G = \SO_d(\CC)$ by conjugation by the permutation matrix on $\CC^d$ that switches $\hat e_{m}$ and $\hat e_{m +1}$ and fixes all the other $\hat e_i$'s. 

 We take $(\mathcal T, \mathcal B)$ to be the standard Borel pair fixed in Definition \ref{defn:fixing borel pair in dual}. Then it is easy to check that the condition in (3) in \S \ref{subsubsec:L-group data} is indeed satisfied. To complete the presentation of the $L$-group, we have yet to specify (\ref{eq:frak v}). As we have already noted, this is equivalent to specifying (\ref{eq:frak v^*}). 

Under the isomorphism $(\CC^{\times})^m \isom \mathcal T$ specified in Definition \ref{defn:fixing borel pair in dual}, the based root datum $\mathrm{BRD}(\mathcal T,\mathcal B)^{\vee}$ is identified with the standard based root datum $\mathrm{BRD}(\mathsf B_m)$ (resp.~$\mathrm{BRD}(\mathsf D_m)$) in the odd (resp.~even) case. Moreover the $\Gamma_F$-action on $\mathrm{BRD}(\mathsf B_m)$ or $\mathrm{BRD}(\mathsf D_m)$ induced by the $\Gamma_F$-action on $\widehat G$ fixed above is the trivial action unless $d$ is even and $\delta$ is non-trivial, in which case it is given by the unique non-trivial action of $\Gal(F(\alpha)/F) = \ZZ/2\ZZ$ on $\mathrm{BRD}(\mathsf D_m)$. Hence to specify (\ref{eq:frak v^*}), it suffices to specify a $\Gamma_F$-equivariant isomorphism
${\mathfrak v^*}' : \mathrm{BRD}(G^*) \isom \mathrm{BRD}(\mathsf B_m)$ or ${\mathfrak v^*}' : \mathrm{BRD}(G^*) \isom \mathrm{BRD}(\mathsf D_m)$, where $\Gamma_F$ acts on the right hand sides in the way just described. 

In the odd case, there is a unique choice of ${\mathfrak v^*}'$. In the even case, remember that when $\delta$ is trivial (resp.~non-trivial), we have fixed $[\mathbb B_{\underline V}]$ (resp.~$\alpha$ and $[\mathbb B_{\underline V}]$) in Definition \ref{defn:fixing alpha}. Any member $\mathbb B_{\underline V}$ of $[\mathbb B_{\underline V}]$ gives rise to a Borel pair $(T,B)$ in $G^*$, and it together with $\alpha$ gives rise to an isomorphism $\mathrm{BRD}(T,B)\isom \mathrm{BRD}(\mathsf D_m)$, as in \S \ref{subsubsec:Borel pairs assoc to hd and nhd}. We thus obtain an isomorphism $\mathrm{BRD}(G^*) \isom \mathrm{BRD}(\mathsf D_m)$, which we easily check is $\Gamma_F$-equivariant, and depends on $\mathbb B_{\underline V}$ only via $[\mathbb B_{\underline V}]$. This specifies ${\mathfrak v^*}'$.

The presentation of the $L$-group of $G$ is complete.

\subsection{}\label{para:local global comp}
	Suppose $F= \QQ$, and let $v$ be a place of $\QQ$. Fix a field embedding $\overline {\QQ} \to \overline {\QQ}_v$. Then our above presentation of the $L$-group of $G$ naturally gives rise to a presentation of the $L$-group of $G_{\QQ_v}$. On the other hand, if $(\underline V, \underline q)$ is the quasi-split quadratic space over $\QQ$ fixed in \S \ref{Fixing inner twist}, then $\underline V_{\QQ_v} = (\underline V, \underline q)\otimes_{\QQ} \QQ_v$ is up to isomorphism the unique quasi-split quadratic space over $\QQ_v$ of dimension $d$ and discriminant $\delta$. Thus one could choose the data as in Definitions  \ref{defn:fixing isom between quad spaces}, and \ref{defn:fixing alpha} with respect to the base field $\QQ_v$ and with $V$ and $\underline V$ replaced by $V_{\QQ_v}$ and $\underline V_{\QQ_v}$, say $\phi_{V_{\QQ_v}}$, $[\mathbb B_{\underline V_{\QQ_v}}]$, and $\alpha_v$, and obtain from these data a presentation of the $L$-group of $G_{\QQ_v} = \SO(V_{\QQ_v})$ by going through the above constructions again. These two presentations of the $L$-group of $G_{\QQ_v}$ are  identical in the odd case, and in the even case they are identical as long as the following conditions are satisfied:
	\begin{enumerate}
		\item We have $\phi_{V_{\QQ_v}} = g_v \circ (\phi_V\otimes_{\overline \QQ} \id_{\overline \QQ_v})$ for some $g_v \in G^*(\overline {\QQ}_v)$. 
		\item The isomorphism $\mathrm{BRD}(G^*) \isom \mathrm{BRD}(\mathsf D_m)$ arising from $[\mathbb B_{\underline V}]$ and $\alpha$ is compatible with the isomorphism $\mathrm{BRD}(G^*_{\QQ_v}) = \mathrm{BRD}(\SO(\underline V_{\QQ_v})) \isom \mathrm{BRD}(\mathsf D_m)$ arising from $[\mathbb B_{\underline V_{\QQ_v}}]$ and $\alpha_v$. 
	\end{enumerate}  In the rest of the paper these compatibility conditions are always implicitly assumed when we simultaneously deal with $\QQ$ and its localizations; note that when $\phi_{V}$ is given, there indeed exists $\phi_{V_\RR}$ satisfying simultaneously (1) in the above and the extra condition in Definition \ref{defn:fixing isom between quad spaces}. By contrast, we will not assume that the local data $\phi_{V_{\QQ_v}}, [\mathbb B_{\underline V_{\QQ_v}}], \alpha_v$ are induced by the global data $\phi_V, [\mathbb B_{\underline V}], \alpha$ on the nose.  Under condition (1) we also know that the inner class of the inner twisting $\psi_V: G_{\overline \QQ} \isom G^*_{\overline \QQ}$ (arising from $\phi_V$) induces the inner class of the inner twisting $\psi_{V_{\QQ_v}} : G_{\overline \QQ_v} \isom G^*_{\overline \QQ_v}$ (arising from $\phi_{V_{\QQ_v}}$) via base change.  
\section{The elliptic endoscopic data}
\label{subsec:elliptic end data} 
\subsection{}\label{subsubsec:extended endo data}
Keep the setting of \S \ref{fixing L group}. Denote by $\mathscr E(G)$\index{$\mathscr E(G)$} the set of isomorphism classes of elliptic endoscopic data\index[n]{elliptic endoscopic data} for $G$, in the sense of \cite[\S 2.1]{KS99}. 
In the following we construct explicit representatives of $\mathscr E(G)$, following \cite{walds10}. Recall from \cite[\S 2.1]{KS99} that in general, the category of elliptic endoscopic data for $G$ is a full subcategory of the category of endoscopic data for $G$, and the latter is a full subcategory of the groupoid category described as follows:
\begin{itemize}
	\item The objects are tuples $(H, \mathcal H, s, \eta)$, where $H$ is a quasi-split reductive group over $F$, $\mathcal H$ is a group containing $\widehat {H}$ as a subgroup, $s$ is an element of $\widehat G$, and $\eta$ is an injective group homomorphism $\mathcal H \to \lang G$.
	\item An isomorphism from $(H,\mathcal H, s,\eta)$ to $(H',\mathcal H' , s' ,\eta')$ is an element $g\in \widehat G$ such that $g \im(\eta) g^{-1} = \im (\eta')$ and   $g s g^{-1} \equiv s' \mod Z(\widehat G)$.
\end{itemize}
We do not recall here the conditions characterizing the subcategories of endoscopic data and elliptic endoscopic data. 

In the following, all our explicit representatives of $\mathscr E(G)$ will be of the form $(H,\lang H, s, \eta)$. Thus in the terminology of \cite{kaldepth0}, we  represent  each isomorphism class of elliptic endoscopic data by an \emph{extended endoscopic triple}\index[n]{extended endoscopic triple}. The advantage of doing so is that we could avoid introducing $z$-extensions, which is in general a necessity for the theory of endoscopy when $G^{\der}$ is not simply connected; cf.~\cite[\S 2.3]{taibi}.

 We first define a set of numerical parameters that will be used. 
\begin{defn}\label{defn:mathscr P_V}
Let $V$ be a quadratic space over $F$ of dimension $d$ and discriminant $\delta$. Define a set $\mathscr P_V$ as follows.\index{$\mathscr P_V$} 
\begin{enumerate}
	\item When $d$ is odd, we let $\mathscr P_V$\index{$\mathscr P_V$} be the set of pairs $(d^+, d^-)$ of positive odd integers such that $d^+ + d^- = d+1$. We define an involution $\swap$\index{$\swap$} on $\mathscr P_V$ (called \emph{swapping}) by sending $(d^+, d^-)$ to $(d^-, d^+)$. 
	\item When $d$ is even, we let $\mathscr P_V$ be the set of quadruples $(d^+, \delta^+, d^- ,\delta^-)$, where: 
	\begin{itemize}
		\item  $d^+$ and $d^-$ are non-negative even integers such that $d^+ + d^- =d$. 
		\item $\delta^+ $ and $ \delta^- $ are elements of $F^\times / F^{\times ,2}$ such that $\delta^+ \delta^- = \delta$. 
		\item Neither of $(d^+, \delta^+)$ and $(d^- ,\delta^-)$ is equal to $(0, x)$ for any non-trivial $x\in F^{\times}/F^{\times,2}$. If $d \geq 4$, then neither of $(d^+, \delta^+)$ and $(d^- ,\delta^-)$  is equal to $(2,1)$. 
	\end{itemize}
 We define an involution $\swap$\index{$\swap$} on $\mathscr P_V$ by sending $(d^+, \delta^+, d^-, \delta^-)$ to $(d^-,\delta^-, d^+, \delta^+)$. 
\end{enumerate} When $d$ is odd,
 we sometimes write elements of $\mathscr P_V$ also as $(d^+,\delta^+,d^-,\delta^-)$, understanding that $\delta^+ = \delta^- =1$. 
\end{defn}

\subsection{} \label{endoscopic data} Fix an element $(d^+,\delta^+, d^-, \delta^-) \in \mathscr P_V$. We shall construct an elliptic endoscopic datum for $G = \SO(V)$ associated to this parameter. The endoscopic datum will be of the form $(H, \mathcal H, s, \eta)$, where
\begin{itemize}
	\item $H$ is given as $H = H^+ \times H^-$, with $H^{\pm} = \SO(V^{\pm})$, where $V^\pm$ are the unique (up to isomorphism) quasi-split quadratic spaces over $F$ of dimension $d^\pm$ and discriminant $\delta^\pm$;\index{$H^+, H^-$} remember that $\delta^{\pm}$ are understood to be trivial in the odd case.
	\item $\mathcal H = \lang H$ is the $L$-group of $H$; cf.~the discussion in \S \ref{subsubsec:extended endo data}.
	\item $s$ is a semi-simple element of $\widehat G$.
	\item $\eta$ is an $L$-embedding $\lang H \to \lang G$.\index{$\eta:\lang H \to \lang G$}
\end{itemize} 

To be more precise, in the even case, we fix similar choices as in Definition \ref{defn:fixing alpha} for $V^{\pm}$, which we shall denote by $\alpha^\pm$ and $[\mathbb B_{V^\pm}]$. (Here $\alpha^+$ is only needed when $\delta^+$ is non-trivial, and similarly for $\alpha^-$.) We then use these choices to specify the analogues of (\ref{eq:frak v^*}) for $H^{\pm}$ in both the odd and even cases, and present the $L$-groups $\lang H^{\pm}$ as $\widehat{H^{\pm}} \rtimes \Gamma'$ as in \S\ref{fixing L group}, where $\widehat{H^{\pm}}$ are the matrix groups $\Sp_{d^{\pm}-1}(\CC)$ (resp.~$\SO_{d^{\pm}}(\CC)$) in the odd (resp.~even) case. 
In the even case, $\Gamma'$ needs to be large enough so as to 
admit $\Gal(F(\alpha^+)/F)$ (resp.~$\Gal(F(\alpha^-)/F)$, resp.~$\Gal(F(\alpha^+, \alpha^-)/F)$) as a quotient when $\delta^+ $ is non-trivial (resp.~$\delta^-$ is nontrivial, resp.~$\delta^+$ and $\delta^-$ are both non-trivial).

We present the $L$-group $\lang H$ of $H$ as the fiber product of $\lang H^+$ and $\lang H^-$ over $\Gamma'$. Thus $\lang H$ is a semi-direct product $$(\widehat { H^+} \times \widehat {H^- } ) \rtimes \Gamma', $$ and $\widehat H = \widehat { H^+} \times \widehat {H^- } $ is equipped with the standard Borel pair \index{$(\mathcal T_{\widehat H},  \mathcal B_{\widehat H} ) = (\mathcal T_{V^+} \times \mathcal T_{V^-} , \mathcal B_{V^+} \times \mathcal B_{V^-})$} $$(\mathcal T_{\widehat H},  \mathcal B_{\widehat H} ) = (\mathcal T_{V^+} \times \mathcal T_{V^-} , \mathcal B_{V^+} \times \mathcal B_{V^-}). $$ Here $(\mathcal T_{V^\pm}, \mathcal B_{V^\pm})$ are as in Definition \ref{defn:fixing borel pair in dual} for the matrix groups $\widehat{H^{\pm}}$.

We now specify the components $s$ and $\eta$. The element $s\in \widehat G$ will always be a diagonal matrix, with $\pm 1$'s on the diagonal. We write $s=\diag (s_1,\cdots , s_{d-1})$ or $\diag(s_1, \cdots, s_{d})$, when $d$ is odd or even respectively.

For $w\in \Gamma'$, we write\footnote{In practice, it can happen that $\lang H$ is presented as $\widehat{H} \rtimes \Gamma'_H$ whereas $\lang G$ is presented as $\widehat{G} \rtimes \Gamma'_G$, for different quotient groups $\Gamma'_H$ and $\Gamma'_G$ of the absolute Galois (or Weil) group of $F$.  For instance, in the even case, when $\delta$ is trivial and $\delta^+$ and $\delta^-$ are both non-trivial, we can take $\Gamma'_H$ to be $\Gal(F(\alpha^+, \alpha^-)/F)$ and take $\Gamma'_G$ to be trivial. In all cases, we may and shall assume that $\Gamma'_G$ is always a quotient of $\Gamma'_H$. Then the formula $\eta(w) = (\rho(w), w)$ is understood as $\eta(w) = (\rho(w), \pi(w))$, where $\pi$ is the quotient map $\Gamma'_H \to \Gamma'_G$. In the text we slightly abuse notation to write $\Gamma'$ for both $\Gamma'_H$ and $\Gamma'_G$.} $$\eta(w) = (\rho(w), w) \in \lang G = \widehat G \rtimes \Gamma'.$$ To specify the map $\eta: \lang H \to \lang G$ it suffices to specify the map $\eta|_{\widehat H}: \widehat H \to \widehat G$ and the map $\rho: \Gamma' \to \widehat G$. 

We now specify the numbers $s_i \in \CC^{\times}$,  and the maps $\eta|_{\widehat H}$ and $\rho$.

\subsubsection{The odd case} 
Write $m^{\pm} : = \floor{d^{\pm}/2}$. Define $$s_k : = \begin{cases}
1,  & \text{if } m^- +1 \leq k \leq d-m^- - 1, \\ -1,  & \text{otherwise}.
\end{cases} $$
 Define the map $$\eta|_{\widehat H} :  \widehat H  = \widehat {H^+} \times \widehat {H^- } = \Sp_{d^+ -1} (\CC) \times \Sp_{d^- -1} (\CC) \To \widehat G = \Sp_{d-1} (\CC)$$ to be the restriction of the map $\GL_{d^+-1} (\CC) \times \GL_{d^- -1} (\CC) \to \GL_{d-1}(\CC) $ given by the identification 
\begin{align*}
\CC^{d^+ -1} \times \CC^{d^- -1}  & \overset {\sim} {\longrightarrow} \CC^{d-1}   \\ (\hat e_k,0)  & \longmapsto \hat e_{k+ m^-} \\ (0, \hat e_k) & \longmapsto \begin{cases}
\hat e_{k},  &\text{if } k\leq m^- ,\\ \hat e_{k+d^+ -1}, &\text{if } m^- +1 \leq k \leq d^- -1
.\end{cases} 
\end{align*}
Finally, define $\rho$ to be trivial.

\subsubsection{The even case}
 Write $m^{\pm} : =  d^{\pm} /2$. Define 
$$s_k : = \begin{cases}
1, &\text{if } m^- +1 \leq k \leq d-m^- ,  \\ -1, &  \text{otherwise}.
\end{cases}$$
Define the map $$\eta|_{\widehat H} : \widehat H  =  \widehat{ H^+} \times \widehat {H^-} = \SO_{d^+ } (\CC) \times \SO_{d^- } (\CC) \To \widehat G = \SO_{d} (\CC)$$  to be the restriction of the map $\GL_{d^+} (\CC) \times \GL_{d^-} (\CC) \to \GL_{d}(\CC)$ given by the identification
\begin{align}\label{eq:eta in even case}
\CC^{d^+ } \times \CC^{d^- }   & \overset{\sim}{\longrightarrow} \CC^{d}   \\ \nonumber (\hat e_k,0) & \longmapsto \hat e_{k+ m^-}, \\ \nonumber (0, \hat e_k) & \longmapsto \begin{cases}
\hat e_{k}, & \text{if } k\leq m^- ,\\  \hat e_{k+d^+}, &\text{if }  m^- +1 \leq k \leq d^- .
\end{cases}
\end{align}

We define $\rho: \Gamma' \to \widehat G$ as follows. First we define a matrix $S \in \GL_d(\CC)$.\index{$S \in \GL_d(\CC)$} If $d^+ \neq 0$, we take $S$ to be the permutation matrix that switches $\hat e_{m^-}$ and $\hat e_{d-m^- +1}$, switches $\hat e_{m}$ and $\hat e_{m+1}$, and leaves all the other $\hat e_i$'s fixed. If $d^+ = 0$, we take $S$ to be $I_d$.
Thus in all cases we have $S \in \widehat G$.
We then let $\rho: \Gamma' \to \widehat G$ be the map 
\begin{align} \label{eq:eta on Gamma'}
w \longmapsto 
\begin{cases}
1, &  \text{if } w |_{F(\alpha^-)} = \id , \\
S, &  \mbox{otherwise}.
\end{cases}
\end{align} 
Here remember that $\alpha^-$ is a fixed square root in $\overline F$ of a fixed lift of $\delta^-$ in $F^\times$ when $\delta^-$ is non-trivial. If $\delta^-$ is trivial, we understand $F(\alpha^-)$ as $F$. The above formula (\ref{eq:eta on Gamma'}) makes sense because when $\delta^-$ is non-trivial we have assumed that $\Gamma'$ admits $\Gal(F(\alpha^-)/F)$ as a quotient.

\subsection{}\label{para:e_p}
In both the odd and even cases, the construction in \S \ref{endoscopic data} associates to each parameter $\mathfrak p \in \mathscr P_V$ an elliptic endoscopic datum \index{$\mathfrak e_{\mathfrak p}$} $$ \mathfrak e_{\mathfrak p} = \mathfrak e_{\mathfrak p} (V) =  (H, \lang H , s,\eta)$$ for $G$. Moreover, the construction $\mathfrak p \mapsto \mathfrak e_{\mathfrak p}$ induces a bijection 
$$ \mathscr P_V /\swap \isom \mathscr E(G). $$
These facts are well known (see \cite[\S 1.8]{walds10} or \cite[\S 2.3]{taibi}) and can be proved similarly as  \cite[Prop.~2.1.1]{morel2011suite}.

\ignore{
\begin{rem}\label{rem:s to s inverse}
	In the literature there are two different conventions for transfer factors (apart from the normalization of local class field theory), which are related by $s \mapsto s^{-1}$. See \cite[p.~184]{kottwitzannarbor} and \cite[\S 5.1]{KSconvention} for more details. In our case, since $s$ always satisfies $s^ 2= 1$, the two conventions are the same. 
\end{rem}}

\subsection{}\label{subsubsec:out of endos}
Let $\mathfrak p \in \mathscr P_V$. The \emph{outer automorphism group}\index[n]{outer automorphism group (of an endoscopic datum)} $\Out (\mathfrak e_{\mathfrak p})$\index{$\Out (\mathfrak e_{\mathfrak p})$} of the endoscopic datum $\mathfrak e_{\mathfrak p} = (H,\lang H, s,\eta)$ is defined in \cite[\S 2.1]{KS99}.
Note that the group $\bar Z$ discussed on p.~19 of \cite{KS99} is trivial, as $Z(\widehat G)$ is contained in $\eta(\widehat H)$. Hence $\Out (\mathfrak e_{\mathfrak p})$ is isomorphic to $\Aut(\mathfrak e_{\mathfrak p})/ \widehat H$ (where $\Aut(\mathfrak e_{\mathfrak p})$ denotes the automorphism group of $\mathfrak e_{\mathfrak p}$ in the category of endoscopic data), and can be naturally viewed as a subgroup of $\Out_F(H) : = \Aut_F(H)/H^{\ad}(F)$; see \textit{loc.~cit.} for details. 

In the odd case, $\Out (\mathfrak e_{\mathfrak p})$ is trivial unless $\mathfrak p = \swap (\mathfrak p)$, in which case we have $\Out (\mathfrak e_{\mathfrak p}) \cong \ZZ/2\ZZ$, with the non-trivial element acting by swapping $H^+$ and $H^-$.

In the even case, write $\mathfrak p = (d^+,\delta^+,d^-,\delta^-)$. Then $\Out (\mathfrak e_{\mathfrak p})$ is trivial when $d^+d^- = 0$. When $d^+ = d^- =d/2$ and $\delta =1$, we have  $\Out (\mathfrak e_{\mathfrak p}) \cong \ZZ/2\ZZ \times \ZZ/2\ZZ$, where the non-trivial element of the first $\ZZ/2\ZZ$ induces simultaneously non-trivial outer automorphisms on $H^{+}$ and $H^-$, and the non-trivial element of the second $\ZZ/2\ZZ$ acts by swapping $H^+$ and $H^-$. In the remaining cases, we have $\Out(\mathfrak e_{\mathfrak p})\cong \ZZ/2\ZZ$, with the non-trivial element acting by the simultaneously non-trivial outer automorphisms on $H^+$ and $H^-$.

\section{The endoscopic \texorpdfstring{$G$}{G}-data for Levi subgroups} \label{subsec:G-endosc arch}
\subsection{}\label{subsubsec:defn of G-endosc datum} Let $M$ be a Levi subgroup of $G$. 
The notion of an \emph{endoscopic $G$-triple for $M$}\index[n]{endoscopic $G$-triple for $M$} is introduced by Kottwitz in his unpublished notes, and recalled in \cite[\S 2.4]{morel2010book}. (For $G=M$, this is the usual notion of an endoscopic triple for $M$, as in \cite[\S 7.4]{kottwitzcuspidal}.) Given an endoscopic datum $(M', \mathcal M', s_M, \eta_M)$ for $M$, we shall say that it is an \emph{endoscopic $G$-datum for $M$},\index[n]{endoscopic $G$-datum for $M$} if $(M', s_M, \eta_{M}|_{\widehat M'})$ is an endoscopic $G$-triple for $M$ in the sense of \cite[Def.~2.4.1]{morel2010book}. By an \emph{isomorphism} between two endoscopic $G$-data $(M'_1, \mathcal M'_1, s_{M,1}, \eta_{M,1})$ and $(M'_2, \mathcal M'_2, s_{M,2}, \eta_{M,2})$ for $M$, we mean an element $g\in \widehat M$ such that $g \im (\eta_{M,1}) g^{-1} = \im (\eta_{M,2})$ and $g s_{M,1} g^{-1} \equiv s_{M,2} \mod Z(\widehat G)$. Here $Z(\widehat G)$ is canonically embedded in $Z(\widehat M)$.

We mention that the category of endoscopic $G$-data for $M$ (where the morphisms are the isomorphisms) is in fact equivalent to the category of \emph{endoscopic $G$-pairs for $M$} in Kottwitz's unpublished notes.

It is easy to see that the association $(M',\mathcal M', s_M, \eta_M) \mapsto (M', s_M, \eta_M|_{\widehat M'})$ defines a bijection 
$$\set{\text{endoscopic $G$-data for $M$}}/\text{isom} \isom \set{\text{endoscopic $G$-triples for $M$}}/\text{isom} .$$
We also have a similar bijection
$$\set{\text{endoscopic data for $G$}}/\text{isom} \isom \set{\text{endoscopic triples for $G$}}/\text{isom}.$$
As recalled in \cite[\S 2.4]{morel2010book}, Kottwitz constructs a map 
$$ \set{\text{endoscopic $G$-triples for $M$}}/\text{isom}  \To \set{\text{endoscopic triples for $G$}}/\text{isom} .$$
We thus obtain a map
$$ \set{\text{endoscopic $G$-data for $M$}}/\text{isom}  \To \set{\text{endoscopic data for $G$}}/\text{isom} . $$
We say that an endoscopic $G$-datum for $M$ is  \emph{bi-elliptic},\index[n]{bi-elliptic} if both the underlying endoscopic datum for $M$ and the associated endoscopic datum for $G$ (well-defined up to isomorphism) are elliptic. We denote by $\mathscr E_G(M)$\index{$\mathscr E_G(M)$} the set of isomorphism classes of   bi-elliptic endoscopic $G$-data for $M$. Thus we have natural maps 
$\mathscr E_G(M) \to \mathscr E(G)$ and $\mathscr E_G(M) \to \mathscr E(M).$ 

In the following we construct explicit representatives of $\mathscr E_G(M)$. For later purposes, it suffices to consider only certain Levi subgroups $M$ specified as follows. 
\subsection{}\label{subsubsec:setting for G-endosc}
Consider a subspace $W$ of $V$ such that the quadratic form on $V$ is non-degenerate on $W$ and such that the orthogonal complement $W^{\perp}$ of $W$ in $V$ is even-dimensional and split as a quadratic space. We write $d_W$ for the dimension of $W$, and let $n = \floor{d_W/2}$. Recall that $V$ has dimension $d$ and discriminant $\delta$, and as always $m$ denotes $\floor{d/2}$. Clearly the discriminant of $W$ equals $\delta$, and $d_W$ has the same parity as $d$. 

Fix $r,t \in \ZZ_{\geq 0}$ such that $ m = n  + r+2t$. Thus $\dim W^{\perp} = 2(r+2t)$. We fix a hyperbolic basis  (Definition \ref{defn:qsplitting}) $$\mathbb B_{W^{\perp}} = \set{f_1,\cdots, f_{2(r+2t)}}$$ of $W^{\perp}$, which exists since $W^{\perp}$ is split. Using this basis, we identify $\SO(W^{\perp})$ as a subgroup of $\GL_{2(r+2t)}$, and define an embedding 
\begin{align}\label{eq:M}
\Gm^r \times \GL_2^t  \To \SO(W^{\perp}) 
\end{align}
by sending $(z_1,\cdots, z_r, w_1, \cdots, w_t)$
to the block diagonal matrix $$\diag(z_1,\cdots, z_r, w_1,\cdots, w_t, J_2(w_t^{\mathsf T})^{-1}J_2,\cdots, J_2(w_1^{\mathsf T})^{-1}J_2, z_r^{-1} ,\cdots, z_1^{-1} ) ,$$ where 
$$J_2 = \begin{pmatrix}
 0 & 1 \\ 1 & 0 
\end{pmatrix} . $$ We denote the image of (\ref{eq:M}) by $M^{\GL}$, and define $M$ to be $M^{\GL} \times \SO(W)$, viewed as a subgroup of $G$. Then $M$ is a Levi subgroup of $G$. We also write $M^{\SO}$ for $\SO(W)$. 

\subsection{}\label{subsubsec:choices for W}
We proceed similarly as before to fix the quasi-split inner form of $\SO(W)$, present the $L$-group of $\SO(W)$, and fix explicit representatives of the isomorphism classes of the elliptic endoscopic data for $\SO(W)$. We need to fix notation and impose some compatibility conditions. Since $d_W$ has the same parity as $d$, in the following we shall still refer to the ``odd case'' and the ``even case'' unambiguously. As in \S \ref{Fixing inner twist}, we fix the unique (up to isomorphism) quasi-split quadratic space $\underline W$ over $F$ of dimension $d_W$ and discriminant $\delta$ (which is the common discriminant of $V$ and $W$) and fix an isomorphism\index{$\phi_W$} $$\phi_W: W \otimes_F \overline F \isom \underline W \otimes_F \overline F$$ of quadratic spaces over $\overline F$, from which we get the inner twisting \index{$\psi_W$}
\begin{align*}
\psi_W: \SO(W)_{\overline F} & \isom \SO(\underline{W})_{\overline F} \\
g & \longmapsto \phi_W g\phi_W^{-1}.
\end{align*}

Note that as quadratic spaces over $F$, $\underline V$ is isomorphic to the orthogonal direct sum of  $W^{\perp}$ and $\underline W$. We fix such an isomorphism\index{$\phi_W^V$}
$$ \phi_W^V : W^{\perp} \oplus \underline W \isom \underline V ,$$ and use it to obtain an embedding 
\begin{align}\label{eq:Levi resulting from phi_W^V}
M^{\GL} \times \SO(\underline W) \hookrightarrow G^*
\end{align}whose image is a Levi subgroup.

We remind the reader that when $F= \RR$ we require both $\phi_V$ and $\phi_W$ to satisfy the extra condition in Definition \ref{defn:fixing isom between quad spaces}. In general, we assume the following compatibility condition, which can obviously 
be arranged by adjusting $\phi_W$.
\begin{enumerate}
	\item The diagram  $$\xymatrixcolsep{5pc}  \xymatrix{ (W^\perp \oplus W) \otimes \overline F \ar@{=}[r]  \ar[d] ^{\id \oplus \phi_W} & V \otimes \overline F  \ar[d]^{\phi_V} \\ (W^{\perp} \oplus \underline W) \otimes \overline F  \ar[r]^-{\phi^V_W} & \underline V \otimes \overline F } $$ commutes up to an element of $G^*(\overline F) = \SO(\underline V)(\overline F)$. 
\end{enumerate}

As a consequence of this condition, we know that the diagram 
		\begin{equation*}\xymatrixcolsep{5pc}
		\xymatrix{ M _{\overline F}= (M^{\GL}\times \SO(W))_{\overline F} \ar[r]^-{\text{inclusion}} \ar[d] ^{(\id, \psi_W)} &  G_{\overline F}  \ar[d]^{\psi_V} \\ (M^{\GL} \times \SO(\underline W) )_{\overline F} \ar[r]^-{(\ref{eq:Levi resulting from phi_W^V})}  & G^*_{\overline F} }
		\end{equation*}  commutes up to an inner automorphism of $G^*_{\overline F}$. 
	\ignore{
		\item The function $u_V: \Gamma_F \to G^*(\overline F)$ is equal to the composition of $u_W: \Gamma_F \to \SO(\underline W)(\overline F)$ followed by (\ref{eq:Levi resulting from phi_W^V}). 
	\end{enumerate}}

In the odd case, we present the $L$-group $\lang \SO(W)$ as in \S \ref{fixing L group}. 
In the even case, we make similar choices as in Definition \ref{defn:fixing alpha}  for $\underline W$, to be denoted by $\alpha_{\underline W}$ (needed only when $\delta$ is non-trivial) and $[\mathbb B_{\underline W}]$\index{$[\mathbb B_{\underline W}]$}, and use them to present the $L$-group $\lang \SO(W)$ as in \S \ref{fixing L group}.
We may and shall assume the following compatibility conditions: 
\begin{enumerate}
	\setcounter{enumi}{1}
	\item There is a member $\mathbb B_{\underline W} \in [ \mathbb B_{\underline W}]$ such that  $\phi^V_W$ sends the ordered basis $(\mathbb B_{W^\perp} , \mathbb B_{\underline W})$ to a member of $[\mathbb B_{\underline V}]$. 
	\item When $\delta$ is non-trivial, the choices $\alpha_{\underline W}$ and $\alpha$ are equal. 
\end{enumerate}Note that 
the above two conditions are consistent: if (2) is already arranged then we have $\alpha^2 = \alpha_{\underline W}^2$ when $\delta$ is non-trivial, and so we can arrange (3). 

In both the odd and even cases, we canonically identify $M^{\GL}$ with $\Gm^r \times \GL_2^t$ via (\ref{eq:M}), and canonically present $\widehat{M^{\GL}}$ as $(\CC^\times) ^r \times \GL_2(\CC) ^{t}$. We now present the $L$-group of $M$ as $$\lang M =\widehat{ M^{\GL}} \times \lang \SO(W).$$ The above compatibility conditions (1)--(3) ensure that the canonical $\widehat G$-conjugacy class of maps  $
		\lang M \to \lang G$ arising from the fact that $M$ is a Levi subgroup of $G$ is represented by the following map: 
\begin{multline}\label{eq:LM to LG}
	\lang M = (\CC^\times) ^r\times \GL_2(\CC)^t \times \widehat{ \SO(W)} \rtimes \Gamma'  \ni (g_1,\cdots, g_r, h_1,\cdots, h_t, k) \rtimes \tau \\ 	\longmapsto 
 \diag(g_1,\cdots, g_r, h_1,\cdots, h_t, k ,h_t^{\dagger},\cdots,h_1^\dagger, g_r^{-1} ,\cdots, g_1^{-1} ) \rtimes \tau \\ \in \lang G = \widehat G \rtimes \Gamma',
\end{multline}
where we define \index{$h^\dagger,~ \forall  h\in \GL_2(\CC)$} 
\begin{align}\label{eq:dagger}
h^\dagger := \begin{pmatrix}
	&- 1 \\  1 & 
\end{pmatrix}(h^{\mathsf T})^{-1} \begin{pmatrix}
	&  1 \\ -1 & 
\end{pmatrix} , \quad \forall h \in \GL_2(\CC),
\end{align}
 (i.e., $h^\dagger$ is the adjoint of $h^{-1}$ with respect to the symplectic form defined by $\begin{pmatrix}
 	& -1 \\ 1 & 
 \end{pmatrix}$).

We now construct explicit representatives of $\mathscr E_G(M)$. 
\begin{defn} \label{defn:P_W}
	Let $\mathscr P_W$ be as in Definition \ref{defn:mathscr P_V} with respect to the quadratic space $W$, and for each positive integer $x$ we write $[x]$ for the set $\set{1,2,\cdots, x}$. Also set $[0] = \emptyset$. We define the following objects.  
	\begin{enumerate}
		\item Let $\mathscr P_{r,t}$\index{$\mathscr P_{r,t}$} be the set of pairs $(A,B)$, where $A$ is a subset of $[r]$ and $B$ is a subset of $[t]$. For $(A,B) \in \mathscr P_{r,t}$, we write $A^c$ for the complement of $A$ in $[r]$ and write $B^c$ for the complement of $B$ in $[t]$.  
		\item Let $\mathscr P_{r,t} \times' \mathscr P_W$\index{$\mathscr P_{r,t} \times' \mathscr P_W$} be the subset of  $\mathscr P_{r,t} \times \mathscr P_W$ consisting of those $$(A, B,  d^+,\delta^+,d^-,\delta^-) \in \mathscr P_{r,t} \times \mathscr P_W$$ such that the quadruple $$(d^+ + 2\abs{A} + 4\abs{B}, \delta^+, d^- + 2 \abs{A^c} + 4\abs{B^c}, \delta^-)$$ belongs to $\mathscr P_V$. (In the odd case, we understand that $\delta^+ = \delta^- =1$, and note that $\mathscr P_{r,t} \times' \mathscr P_W = \mathscr P_{r,t} \times \mathscr P_W$.) 
		\item Note that $(A, B, \mathfrak p ) \mapsto (A^c, B^c, \swap(\mathfrak p))$ is an involution on the set $\mathscr P_{r,t} \times' \mathscr P_W$. We denote this involution still by $\swap$.\index{$\swap$}	
	\end{enumerate}
\end{defn}
\begin{defn}\label{defn:nabla}
	Let $A$ be a subset of $\ZZ_{\geq 1}$. For each $i \in \ZZ_{\geq 1}$, we define \index{$\nabla_i(\cdot)$} 
	$$\nabla_i (A) : = \begin{cases}
		1, & \text{if } i\in A, \\
		-1, & \text{if } i \notin A.
	\end{cases}$$
\end{defn}
\subsection{}\label{para:presentation of endoscopic G-data}
Fix an element $(A, B, \mathfrak p) \in \mathscr P_{r,t} \times' \mathscr P_W$. In the following we construct an endoscopic $G$-datum for $M$ associated to this parameter, denoted by $\mathfrak e_{A,B,\mathfrak p}$\index{$\mathfrak e_{A,B,\mathfrak p}$}.  From $\mathfrak p \in \mathscr P_W$ we obtain the endoscopic datum $\mathfrak e_{\mathfrak p} (W)$ for $\SO(W)$ as in \S \ref{subsec:elliptic end data}, which we write as $$(M^{\prime,\SO}, \lang M^{\prime, \SO}, s^{\SO}, \eta^{\SO}:\lang M^{\prime, \SO} \to \lang \SO(W)).$$ 
We then set $$\mathfrak e_{A,B,\mathfrak p} : = (M', \lang M', s_M, \eta_M: \lang M' \to \lang M) $$ with components given as follows. Let $$M' := M^{\GL} \times M^{\prime, \SO}. $$ Let $s_M$ be the element of $\widehat M = \widehat{M^{\GL}} \times \widehat{ M^{\SO} } $ whose component in $\widehat{ M^{\SO} }$ is $s^{\SO}$ and whose component in $\widehat {M^{\GL}}= (\CC^\times) ^{r} \times \GL_2(\CC) ^{t}$ is 
\begin{align}\label{eq:s^GL}
s^{\GL} = (\nabla_1(A), \cdots, \nabla_r(A), \nabla_1(B) I_2, \cdots, \nabla_t(B) I_2).
\end{align}  
We present the $L$-group $\lang M'$\index{$\lang M'$} of $M'$ as $$\lang M' = \widehat{M^{\GL} } \times \lang M^{\prime, \SO}, $$ and define $\eta_M$ to be the map $$\eta_M = (\id, \eta^{\SO}): \lang M' = \widehat{M^{\GL} } \times \lang M^{\prime, \SO} \to \lang M = \widehat{M^{\GL}} \times \lang M^{\SO}.$$

For each $\mathfrak p \in \mathscr P_W$, we also set  \index{$\mathfrak e_{\mathfrak p}(M)$} 
$$ \mathfrak e_{\mathfrak p}(M) : = (M', \lang M', s_M', \eta_M),$$ where $M', \lang M', \eta_M$ are as above, and $s_M'$ is the element of $\widehat M = \widehat{M^{\GL}} \times \widehat{ M^{\SO} } $ whose component in $\widehat{ M^{\SO} }$ is $s^{\SO}$ and whose component in $\widehat {M^{\GL}}$ is the trivial element. Then $\mathfrak e_{\mathfrak p}(M)$ is an elliptic endoscopic datum for $M$. 
\begin{prop}\label{prop:odd classification of G-endoscopy}
	For each $(A,B, \mathfrak p) \in \mathscr P_{r,t} \times' \mathscr P_W$, the tuple $\mathfrak e_{A,B, \mathfrak p}$ is a bi-elliptic endoscopic $G$-datum for $M$ whose underlying endoscopic datum for $M$ is isomorphic to $\mathfrak e_{\mathfrak p}(M)$. The construction $(A, B, \mathfrak p) \mapsto \mathfrak e_{A,B, \mathfrak p}$ induces a bijection \begin{align*} \bigg (\mathscr P_{r,t}\times' \mathscr P_W  \bigg) / \swap \isom \mathscr E_G(M).
	\end{align*}
 Moreover, for $	(A, B, d^+,\delta^+, d^-,\delta^- )  \in 	\mathscr P_{r,t} \times' \mathscr P_W$, the image of $\mathfrak e_{A,B, d^+, \delta^+, d^- ,\delta^-}$ under the map $\mathscr E_G(M) \to \mathscr E(G)$ is represented by  \begin{align*}
 \mathfrak e_{d^+ + 2 \abs{A} + 4\abs{B},\delta^+, d^-+ 2 \abs{A^c} + 4\abs{B^c},\delta^-}.
\end{align*}
(Remember that if the common parity of $d_W$ and $d$ is odd, then we keep the convention that $\delta^{\pm} = 1$ as in Definition \ref{defn:mathscr P_V}.)
\end{prop}
\begin{proof}
	This can be checked in a similar way as the proof of \cite[Lem.~2.3.3]{morel2011suite}. The key point is that $M^{\GL}$ is a product of copies of $\GG_m$ and $\GL_2$, and these groups do not have any non-trivial elliptic endoscopic data.
\end{proof} 
\subsection{}\label{para:Out_G}
Let $(A,B,\mathfrak p) \in \mathscr P_{r,t} \times' \mathscr P_W$. Write $\mathfrak e_{A,B,\mathfrak p}$ as $(M', \lang M', s_M, \eta_M)$. We define the \emph{outer $G$-automorphism group}\index[n]{outer $G$-automorphism group (of an endoscopic $G$-datum)} of $\mathfrak e_{A,B,\mathfrak p}$ to be\index{$\Out_G(\mathfrak e_{A,B,\mathfrak p})$} $$\Out_G(\mathfrak e_{A,B,\mathfrak p}) : = \Aut_G(\mathfrak e_{A,B,\mathfrak p}) / \widehat M', $$ where $\Aut_G (\mathfrak e_{A,B,\mathfrak p})$\index{$\Aut_G (\mathfrak e_{A,B,\mathfrak p})$} denotes the automorphism group of $\mathfrak e_{A,B,\mathfrak p}$ in the category of endoscopic $G$-data for $M$ (see \S \ref{subsubsec:defn of G-endosc datum}). We make two remarks on this definition. Firstly, $\Out_G(\mathfrak e_{A,B,\mathfrak p})$ is naturally isomorphic to the outer automorphism group of the endoscopic $G$-triple $(M', s_M, \eta_M |_{\widehat M'})$ defined in \cite[\S 2.4]{morel2010book}. (This is explained in Kottwitz's unpublished notes.) Secondly, $\Out_G(\mathfrak e_{A,B,\mathfrak p})$ is naturally isomorphic to a subgroup of the outer automorphism group $\Out(\mathfrak e_{A,B,\mathfrak p})$ of the underlying endoscopic datum for $M$. (See \S \ref{subsubsec:out of endos} for the latter.)

We now explicitly determine $\Out_G(\mathfrak e_{A,B,\mathfrak p})$. In the odd case, we always have $\Out_G(\mathfrak e_{A,B,\mathfrak p}) = \set{1}$. In the even case, write $\mathfrak p = (d^+,\delta^+,d^-,\delta^-)$. Then $\Out_G(\mathfrak e_{A,B,\mathfrak p})$ is trivial if $d^+d^- = 0$. In the remaining cases, we have $\Out_G(\mathfrak e_{A,B,\mathfrak p})\cong \ZZ/2\ZZ$, where the non-trivial element acts via the non-trivial outer automorphism on $M^{\GL}$, and via the simultaneously non-trivial outer automorphisms on the two special orthogonal groups constituting $M^{\prime,\SO}$. 

\subsection{}\label{para:two maps from endoscopic G data}Let $(A, B, d^+,\delta^+, d^-,\delta^- ) \in 	\mathscr P_{r,t} \times' \mathscr P_W$. We have the endoscopic $G$-datum $$\mathfrak e_{A,B, d^+,\delta^+, d^-, \delta^-} = (M', \lang M', s_M, \eta_M : \lang M' \to \lang M)$$ for $M$, and the endoscopic datum $$ \mathfrak e_{d^+ + 2\abs{A} +4 \abs{B} , \delta^+,  d^- + 2\abs{A^c} + 4 \abs{B^c}, \delta^-} = (H, \lang H, s, \eta: \lang H\to \lang G)$$ for $G$. Thus the isomorphism class of $(H, \lang H, s, \eta)$ in $\mathscr E(G)$ is equal to the image of the isomorphism class of $(M', \lang M', s_M, \eta_M)$ in $\mathscr E_G(M)$, by Proposition \ref{prop:odd classification of G-endoscopy}. As explained on p.~43 of \cite{morel2010book}, the endoscopic $G$-datum $(M', \lang M', s_M, \eta_M)$ for $M$ determines an  $H(F)$-conjugacy class of Levi subgroups of $H$, all of which are isomorphic to $M'$. In the following we upgrade this construction to an $H(F)$-conjugacy class of $F$-embeddings $M' \hookrightarrow H$ with images Levi subgroups. (This depends on our explicit presentation of the groups.) As a result we obtain a $\widehat H$-conjugacy class of embeddings $\lang M' \to \lang H$. Our construction will be such that the following diagram commutes up to $\widehat G$-conjugation:
\begin{align}\label{eq:diag with M'}
	\xymatrix{ \lang H \ar[r]^{\eta} & \lang G \\
		\lang M' \ar[u] \ar[r]^{\eta_M}  &  \lang M \ar[u]}
\end{align}
Here the vertical arrow on the right is canonical up to $\widehat G$-conjugation, arising from the fact that $M$ is a Levi subgroup of $G$ (cf. (\ref{eq:LM to LG})).

 Recall from \S \ref{subsubsec:setting for G-endosc} that $M^{\GL}$ is a subgroup of $\SO(W^{\perp})$, and that $W^{\perp}$ is equipped with a hyperbolic basis $\set{f_1,\cdots, f_{2(r+2t)}}.$ Let \index{$(W^{\perp})_{A,B}, (W^{\perp})_{A^c,B^c} $}
\begin{align*}
(W^{\perp})_{A,B} & : = \mathrm{span}  \{f_i , f_{2(r+2t)+1-i} \mid i \in A \text{ or } \ceil{\frac{i-r}{2}} \in B  \}   ,\\ (W^{\perp})_{A^c,B^c} & 
: = \mathrm{span}  \{f_i , f_{2(r+2t)+1-i} \mid i \in A^c \text{ or } \ceil{\frac{i-r}{2}} \in B^c  \}  .
\end{align*}The natural action of $M^{\GL}$ on $W^{\perp}$ stabilizes $(W^{\perp})_{A,B}$ and $(W^{\perp})_{A^c, B^c}$. Let $M^{\GL}_{A,B}$ (resp.~$M^{\GL}_{A^c,B^c}$)\index{$M^{\GL}_{A,B}$, $M^{\GL}_{A^c,B^c}$} be the maximal quotient of $M^{\GL}$ acting faithfully on $(W^{\perp})_{A,B}$ (resp.~$(W^{\perp})_{A^c, B^c}$). Concretely, if we write $A = \set{i_1,\cdots, i_u}$ and $B = \set{j_1,\cdots, j_v}$ where $i_1< i_2 < \cdots< i_u$ and $j_1<j_2<\cdots< j_v$, then $M^{\GL}_{A,B}$ is identified with $\GG_m^{u} \times \GL_2^{v}$, and the quotient map $$M^{\GL} \cong \GG_m^r \times \GL_2^t \To M^{\GL}_{A,B} \cong \GG_m^{u} \times \GL_2^{v}$$ is given by 
$$ (z_1,\cdots, z_r, w_1,\cdots, w_t) \longmapsto (z_{i_1},\cdots, z_{i_u}, w_{j_1},\cdots, w_{j_v}).$$
Similarly we have a concrete description of the quotient map $M^{\GL} \to M^{\GL}_{A^c,B^c}$. 

We now specify the $H(F)$-conjugacy class of embeddings $M' \hookrightarrow H$. First consider the odd case. Choose an isometric isomorphism $\mathbf  f^+$ from the orthogonal direct sum of $(W^{\perp})_{A,B}$ and $W^+$ to $V^+$. (Such $\mathbf  f^+$ indeed exists since both quadratic spaces are split and have dimension $d^+ + 2\abs{A}+4 \abs{B}$.) This choice, together with the natural action of $M^{\GL}_{A,B}$ on $(W^\perp)_{A,B}$ and the natural action of $\SO(W^+)$ on $W^+$, determines an embedding $$\mathbf  f^+_*: M^{\GL}_{A,B} \times \SO(W^+) \To \SO(V^+). $$ We claim that the $\SO(V^+)(F)$-conjugacy class of $\mathbf  f^+_*$ is independent of the choice of $\mathbf  f^+$. Indeed, the $\mathrm{O}(V^+)(F)$-conjugacy class of $\mathbf  f^+_*$ is clearly independent of the choice of $\mathbf  f^+$. The element of $\mathrm{O}(V^+)(F)$ acting as $1$ on $\mathbf  f^+((W^\perp)_{A,B})$ and as $-1$ on $\mathbf  f^+(W^+)$ has determinant $-1$ and centralizes $\mathbf  f^+_*$. Hence the $\mathrm{O}(V^+)(F)$-conjugacy class of $\mathbf  f^+_*$ is in fact equal to the $\SO(V^+)(F)$-conjugacy class of $\mathbf  f^+_*$. Our claim follows.

Similarly, we choose an isometric isomorphism $\mathbf  f^-$ from the orthogonal direct sum of $(W^{\perp})_{A^c, B^c}$ and $W^-$ to $V^-$. We then obtain an embedding $$\mathbf  f^-_*: M^{\GL}_{A^c,B^c} \times \SO(W^-) \To \SO(V^-),$$ whose $\SO(V^-)(F)$-conjugacy class is independent of the choice of $\mathbf  f^-$. Taking the direct product of $\mathbf  f^+_*$ and $\mathbf  f^-_*$, we obtain the desired embedding $M' \to H$ which is canonical up to $H(F)$-conjugacy.

We now consider the even case. Since the orthogonal direct sum of $(W^\perp)_{A,B}$ and $W^+$ is a quasi-split quadratic space of the same dimension and discriminant as $V^+$, we can choose an isometric isomorphism $\mathbf f^+$ between them just as in the odd case. We then obtain the embedding $$\mathbf f^+_*: M^{\GL}_{A,B} \times \SO(W^+) \To \SO(V^+). $$ At this point, only the $\mathrm{O}(V^+)(F)$-conjugacy of $\mathbf f^+_*$ is well defined. We explain how to narrow this down to an $\SO(V^+)(F)$-conjugacy class. As before we canonically identify $M^{\GL}_{A,B}$ with $\GG_m^u \times \GL_2^v$ (where $u =\abs{A}$ and $v = \abs{B}$). Consider the canonical embedding $\iota^{\GL}_{A,B} : \GG_m^{u+2v} \to M^{\GL}_{A,B}$ given by 
$$ (z_1,\cdots, z_{u+2v}) \longmapsto  (z_1, \cdots, z_u, \begin{pmatrix}
		z_{u+1} \\ &  z_{u+2}
	\end{pmatrix}, \cdots, \begin{pmatrix}
		z_{u+2v-1} \\ &  z_{u+2v}
	\end{pmatrix} ) . 
$$  We divide our discussion into the cases where $\delta^+$ is trivial and non-trivial. 

Suppose $\delta^+$ is trivial. As in \S \ref{endoscopic data}, $W^+$ is equipped with an $\SO(W^+)(F)$-orbit $[\mathbb B_{W^+}]$ of hyperbolic bases, and $V^+$ is equipped with an $\SO(V^+)(F)$-orbit $[\mathbb B_{V^+}]$ of hyperbolic bases.
They determine an $\SO(W^+)(F)$-conjugacy class of embeddings $$\iota_{W^+}= \iota_{\mathbb B_{W^+}} : \GG_m^{d^+/2} \To \SO(W^+)$$ and an $\SO(V^+)(F)$-conjugacy class of embeddings $$\iota_{V^+} = \iota_{\mathbb B_{V^+}}: \GG_m^{d^+/2 + u+2v}  \To \SO(V^+)$$ (cf.~\S\ref{subsubsec:Borel pairs assoc to hd and nhd}). We impose the condition that the embedding $$\mathbf f^+_* \circ (\iota^{\GL}_{A,B} \times \iota_{W^+}) : \GG_m^{u+2v} \times \GG_m^{d^+/2} \To \SO(V^+) $$ should be $\SO(V^+)(F)$-conjugate to $\iota_{V^+}$ under the obvious identification 
\begin{align*}
\GG_m^{u+2v} \times \GG_m^{d^+/2} & \isom \GG_m^{d^+/2 +u+2v} \\  ((z_1,z_2, \cdots), (w_1,w_2,\cdots)) & \longmapsto (z_1, z_2, \cdots, w_1, w_2, \cdots).
\end{align*}
This extra condition narrows the $\mathrm{O}(V^+)(F)$-conjugacy class of $\mathbf f^+_*$ to an $\SO(V^+)(F)$-conjugacy class. 

Now suppose $\delta^+$ is non-trivial. In this case, $W^+$ is equipped with an $\SO(W^+)(F)$-orbit $[\mathbb B_{W^+}]$ of near-hyperbolic bases, and we have fixed a square root $\alpha^{+,\prime}\in \overline F$ of the common discriminant of members of $[\mathbb B_{W^+}]$. Similarly, $V^+$ is equipped with an $\SO(V^+)(F)$-orbit $[\mathbb B_{V^+}]$ of near-hyperbolic bases, and we have fixed a square root $\alpha^{+}\in \overline F$ of the common discriminant of members of $[\mathbb B_{V^+}]$. (Here $\alpha^+$ may not be equal to $\alpha^{+,\prime }$.) These extra data determine an $\SO(W^+)(F)$-conjugacy class of embeddings $$\iota_{W^+} = \iota_{\alpha^{+,\prime}, \mathbb B_{W^+}}: \GG_m^{d^+/2 -1} \times \Uni(1)_{\alpha^{+, \prime}} \To \SO(W^+), $$ and an $\SO(V^+)(F)$-conjugacy class of embeddings $$\iota_{V^+} = \iota_{ \alpha^+, \mathbb B_{V^+}}: \GG_m^{d^+/2  + u+2v -1} \times \Uni(1)_{\alpha^{+}} \To \SO(V^+)$$ (cf.~\S\ref{subsubsec:Borel pairs assoc to hd and nhd}). Note that $\Uni(1)_{\alpha^{+,\prime}}$ is canonically identified with $\Uni(1)_{\alpha^+}$, since the fields $F(\alpha^{+,\prime})$ and $F(\alpha^+)$ are the same. We impose the condition that $$\mathbf f^+_* \circ (\iota^{\GL}_{A,B} \times \iota_{W^+}) : \GG_m^{u+2v} \times \GG_m^{d^+/2-1} \times \Uni(1)_{\alpha^{+,\prime}} \To \SO(V^+)$$ should be $\SO(V^+)(F)$-conjugate to $\iota_{V^+}$ under the obvious identification 
\begin{align*}
\GG_m^{u+2v} \times \GG_m^{d^+/2-1} \times \Uni(1)_{\alpha^{+,\prime}} & \isom \GG_m^{d^+/2 +u+2v-1} \times \Uni(1)_{\alpha^+} \\ ((z_1,z_2,\cdots), (w_1,w_2,\cdots), y) & \longmapsto (z_1,z_2,\cdots, w_1, w_2, \cdots, y).
\end{align*} This extra condition narrows the $\mathrm{O}(V^+)(F)$-conjugacy class of $\mathbf f^+_*$ to an $\SO(V^+)(F)$-conjugacy class.

We have specified an $\SO(V^+)(F)$-conjugacy class of embeddings $M^{\GL}_{A,B} \times \SO(W^+) \to \SO(V^+)$. Similarly, we specify an  $\SO(V^-)(F)$-conjugacy class of embeddings $M^{\GL}_{A^c, B^c} \times \SO(W^-) \to \SO(V^-)$. Taking the direct product we obtain the desired embedding $M' \to H$ which is canonical up to $H(F)$-conjugacy.

Write $A = \set{i_1,\cdots, i_u}$, $B = \set{j_1,\cdots, j_v}$, $A^c=\set{p_1,\cdots, p_{r-u}}, $ and $B^c= \set{q_1,\cdots , q_{t-v}}$ with increasing ordering (i.e., $i_1<\cdots<i_u$ etc.). 
In both the odd and even cases, the $\widehat H$-conjugacy class of embeddings $\lang M' \to \lang H$ arising from our construction is represented by the map
\begin{multline}\label{eq:LM' to LH}
	\lang M' = (\CC^\times) ^r\times \GL_2(\CC)^t \times \widehat{ \SO(W^+)} \times \widehat{\SO(W^-)} \rtimes \Gamma' \\ \ni (g_1,\cdots, g_r, h_1,\cdots, h_t, k^+, k^-) \rtimes \tau 	\longmapsto  \\ 	\diag(g_{i_1},\cdots, g_{i_u}, h_{j_1},\cdots, h_{j_v}, k^+ ,h_{j_v}^{\dagger}, \cdots, h_{j_1}^{\dagger}, g_{i_u}^{-1} ,\cdots, g_{i_1}^{-1} )    \\ \times 
	\diag(g_{p_1},\cdots, g_{p_{r-u}}, h_{q_1},\cdots, h_{q_{t-v}}, k^- ,h_{q_{t-v}}^{\dagger}, \cdots, h_{q_1}^{\dagger}, g_{p_{r-u}}^{-1} ,\cdots, g_{p_1}^{-1} )  \\\rtimes \tau \\ \in \lang H = \widehat {H^+}\times \widehat {H^-}  \rtimes \Gamma',
\end{multline}
where the notation $\dagger$ is as in (\ref{eq:dagger}). Using the formulas (\ref{eq:LM to LG}) and (\ref{eq:LM' to LH}), one sees that the diagram (\ref{eq:diag with M'}) indeed commutes up to $\widehat G $-conjugation.
\ignore{
\begin{rem}
	In applications, only the formulas (\ref{eq:LM to LG}), (\ref{eq:LM' to LH}), and the commutativity of (\ref{eq:diag with M'}) will matter to us, while the concrete description of the $H(F)$-conjugacy class of embeddings $M' \hookrightarrow H$ will not matter. 
\end{rem}
}
\section{Admissible isomorphisms and embeddings}\label{subsec:admissible isom} 
\subsection{}
Keep the setting of \S \ref{subsec:elliptic end data}. 
For any torus $T$ over $F$, we denote by $\widehat T$\index{$\widehat T$} the dual torus\index[n]{dual torus} over $\CC$, whose group of characters is canonically identified with $X_*(T)$. If $f: T_1 \to T_2$ is a homomorphism of tori over $F$, we denote by $\hat f$\index{$\hat f$} the dual homomorphism $\widehat {T_2} \to \widehat {T_1}$. 

For any Borel pair $(T,B)$ in $G_{\overline F}$ and any Borel pair $(\derT, \derB)$ in $G^*_{\overline F}$, the fixed isomorphisms (\ref{eq:frak v}) and (\ref{eq:frak v^*}) give rise to isomorphisms $\widehat{T} \isom \mathcal T $ and $ \widehat{\derT} \isom \mathcal T$ of tori over $\CC$. We denote these isomorphisms by $\mathfrak d_{B, \mathcal B}$ and $\mathfrak d_{\derB, \mathcal B}$ respectively.\index{$\mathfrak d_{B, \mathcal B}$}\index{$\mathfrak d_{\derB, \mathcal B}$}

Now consider an elliptic endoscopic datum $(H, \lang H, s, \eta)$ for $G$ as in \S \ref{endoscopic data}. Given a Borel pair $(T_H, B_H)$ in $H_{\overline F}$, there is a similar isomorphism \index{$\mathfrak d_{B_H, \mathcal B_{\widehat H}}$} 
$$
\mathfrak d_{B_H, \mathcal B_{\widehat H}}: \widehat {T_H}   \isom \mathcal T_{\widehat H}. $$ Here $(\mathcal T_{\widehat H} , \mathcal B_{\widehat H})$ is the standard Borel pair in $\widehat H$ as in \S \ref{endoscopic data}. Note that $\eta : \lang H \to \lang G$ maps $\mathcal T_{\widehat H}$ isomorphically onto $\mathcal T$. Hence we obtain isomorphisms 
\begin{align*}
\mathfrak d_{B, \mathcal B} ^{-1} \circ \eta \circ \mathfrak d_{B_H, \mathcal B_{\widehat H}}  & : \widehat {T_H}  \isom \widehat {T},  \\
\mathfrak d_{\derB, \mathcal B} ^{-1} \circ \eta \circ \mathfrak d_{B_H, \mathcal B_{\widehat H}} & : \widehat {T_H}  \isom \widehat {\derT}, 
\end{align*}
or equivalently, isomorphisms 
\begin{align*}
j & : T_H \isom T \subset G_{\overline F} , \\
\derj & : T_H \isom \derT \subset G^*_{\overline F}. 
\end{align*}
We call $j$ an \emph{admissible isomorphism}\index[n]{admissible isomorphism} between $T_H$ and $T$, and an \emph{admissible embedding}\index[n]{admissible embedding} of $T_H$ into $G_{\overline F}$; cf.~\cite[\S 1.3]{LS87}. Similar terminology applies to $\derj$. We shall also say that $j$ is associated to the Borel pairs $(T_H,B_H)$ and $(T, B)$, and say that $\derj$ is associated to the Borel pairs $(T_H,B_H)$ and $(\derT, \derB)$. 

The following facts are well known and straightforward to verify. 

\begin{lem}\label{lem:facts on adm iso}
	Fix maximal tori $T_H \subset H_{\overline F}$, $T \subset G_{\overline F}$,  and $\derT \subset G^*_{\overline F}$.\begin{enumerate}
		\item The set of admissible isomorphisms between $T_H$ and $T$ (resp.~between $T_H$ and $\derT$) is a torsor under the Weyl group of $(G_{\overline F}, T)$ (resp.~the Weyl group of $(G^*_{\overline F}, \derT)$). 
		\item The set of admissible embeddings of $T_H$ into $G_{\overline F}$ (resp.~into $G^*_{\overline F}$) is a single orbit under $G(\overline F)$-conjugation (resp.~$G^*(\overline F)$-conjugation). 
		\item Let $j:T_H \to G_{\overline F}$ and $\derj: T_H \to G^*_{\overline F}$ be arbitrary embeddings such that $$j = \Int (g) \circ \psi_V^{-1} \circ \derj$$ for some $g \in G(\overline F)$. (Here $\psi_V$ is the fixed inner twisting between $G$ and $G^*$; see \S \ref{Fixing inner twist}.) Then $\derj$ is admissible if and only if $j$ is admissible.  
	\end{enumerate} \qed
\end{lem}

\chapter[Transfer factors for real $\SO$ groups]{Transfer factors for real special orthogonal groups} \label{sec:transfer factors}
\section{Cuspidality and transfer of elliptic tori}\label{subsec:cuspidality and transfer of tori}
\subsection{} \label{para:cusp crit}
We keep the notation in \S \ref{sec:endoscopic}, specialized to $F = \RR$. Thus $(V,q)$ is a quadratic space over $\RR$ of dimension $d$ and discriminant $\delta$, and $G = \SO(V,q)$ is a reductive group over $\RR$. We are interested in the case where $G$ contains anisotropic maximal tori. When $d$ is odd, this is always the case. When $d$ is even, this is the case if and only if $\delta = (-1) ^{d/2} \in \RR^{\times}/\RR^{\times,2}$. (Note that if $(d,\delta) = (2,1)$, then $G \cong \GG_m$ contains elliptic maximal tori but not anisotropic maximal tori.) In the following we assume that $G$ contains anisotropic maximal tori. We discuss a systematic way of parameterizing anisotropic maximal tori in $G$. As usual, we let $m : = \floor{d/2}$. Our assumption clearly implies that $V$ admits an \emph{elliptic decomposition} defined as follows. 

\begin{defn}\label{defn:ED}
	By an \emph{elliptic decomposition}\index[n]{elliptic decomposition} of $V$, we mean an ordered tuple $(V_j, o_j)_{1\leq j \leq m}$, where $V_1,\cdots, V_m$ are mutually orthogonal definite planes in $V$, and $o_j$ is an orientation on $V_j$. Thus the orthogonal direct sum of $V_1,\cdots, V_m$ is a hyperplane in $V$ (resp.~equal to $V$) when $d$ is odd (resp.~even). We denote by $\ED(V)$\index{$\ED(V)$} the set of all elliptic decompositions of $V$. By abuse of notation, we often write $(V_j)_j$ for an element of $\ED(V)$, understanding that each $V_j$ is equipped with an orientation.
\end{defn} 
\begin{defn}\label{defn:PEMT}
	By a \emph{parameterized anisotropic maximal torus}\index[n]{parameterized anisotropic maximal torus} in $G$, we mean an anisotropic maximal torus $T_G$ in $G$ together with an isomorphism $\Uni(1)^m \isom T_G$. 
\end{defn}
\begin{defn}\label{defn:fund pair}
By a \emph{fundamental pair}\index[n]{fundamental pair} in $G$, we mean a pair $(T_{G}, B)$, where $T_G$ is an anisotropic maximal torus in $G$, and $B$ is a Borel subgroup of $G_{\CC}$ containing $T_{G,\CC}$. 
	\end{defn}
\begin{rem}\label{rem:card of fund pairs}
	Since any two anisotropic maximal tori in $G$ are conjugate under $G(\RR)$, the number of $G(\RR)$-orbits of fundamental pairs in $G$ is equal to the cardinality of $\Omega_{\CC} (G, T_G) / \Omega_{\RR}(G, T_G)$, where $T_G$ is an arbitrary anisotropic maximal torus in $G$. \end{rem}	

\subsection{}\label{subsubsec:constrn with ED}
Given $\cD = (V_j)_j \in \ED(V)$, we obtain a parameterized anisotropic maximal torus from the embedding \index{$f_{\cD}$} \index{$T_{\cD}$} $$f_{\cD}: \Uni(1)^{m} \isom T_{\cD} \subset  G, $$ where the $j$-th copy of $\Uni(1)$ acts by rotation on the oriented definite plane $V_j$. The (absolute) root datum of $G$ on $$(X^*(T_{\cD}), X_*(T_{\cD})) \underset{f_{\cD}}{\isom} (X^*(\Uni(1) ^m) , X_*(\Uni(1) ^m)) = (\ZZ^m, \ZZ^m)$$ is the standard root datum $\mathrm{RD}(\mathsf B_m)$ or $\mathrm{RD}(\mathsf D_m)$ when $d$ is odd or even. Hence the standard based root datum $\mathrm{BRD}(\mathsf B_m)$ or $\mathrm{BRD}(\mathsf D_m)$ gives rise to a Borel subgroup $B_{\cD}$ of $G_{\CC}$ containing $T_{\cD,\CC}$. Thus we obtain a fundamental pair $(T_{\cD}, B_{\cD})$ from $\cD \in \ED(V)$.  
\ignore{
In summary, we have a canonical map 
\begin{align*} \ED(V) &\To \set{\text{parameterized elliptic maximal tori in $G$}}\times \set{\text{fundamental pairs in } G} \\
\cD & \longmapsto (f_{\cD}: \Uni(1)^m \isom T_{\cD} \subset G, (T_{\cD},B_{\cD})).  
\end{align*} Clearly this map is equivariant with respect to the natural $G(\RR)$-actions on both sides.}
\ignore{ In the odd case, two elliptic decompositions $$(V_1,\cdots, V_m ,l), \quad (V_1', \cdots, V_m', l')$$ are in the same $G(\RR)$-orbit if and only if $V_i$ has the same definiteness as $V_i'$ for each $i$. It follows that the number of $G(\RR)$-orbits on the LHS of (\ref{e.d. to pair}) is ${d}\choose p$. 
}

\subsection{}\label{subsubsec:defn of EDo} Recall from \S \ref{Fixing inner twist} that we have fixed a quasi-split quadratic space $(\underline V , \underline q)$ and fixed an isomorphism $\phi_V: V\otimes_{\RR} \CC \isom \underline V \otimes_{\RR} \CC$ of quadratic spaces over $\CC$. By definition we have $G^* = \SO(\underline V)$. We have the obvious analogues of Definitions \ref{defn:ED}, \ref{defn:PEMT}, \ref{defn:fund pair}, and the constructions in \S \ref{subsubsec:constrn with ED}, with $V$ and $G$ replaced by $\underline V$ and $G^*$. Note that our assumption that $G$ contains anisotropic maximal tori implies that $G^*$ also contains anisotropic maximal tori, since these conditions both boil down to the numerical condition that either $d$ is odd or $d$ is even and $\delta = (-1)^{d/2}$. In particular  $\ED(\underline V) \neq \emptyset$. 

Recall from Definition \ref{defn:fixing alpha} that when $d$ is even and when $\delta$ is trivial (resp.~non-trivial), we have fixed a $G^*(\RR)$-orbit $[\mathbb B_{\underline V}]$ of hyperbolic bases (resp.~near-hyperbolic bases) of $\underline V$. Note that all members of $[\mathbb B_{\underline V}]$ induce the same orientation on $\underline V$. We denote this orientation by $o_{\underline V}$.\index{$o_{\underline V}$} Still under the assumption that $d$ is even, we define an orientation $o_V$\index{$o_V$} on $V$ as follows. Let $(a,b)$ be the signature of $V$. Since $\delta = (-1)^{d/2}$, both $a$ and $b$ are even. Also $\underline V$ has signature $$(a^*,b^*) = (2 \ceil{d/4}, 2 \floor{d/4}). $$ We define $o_V$ to be $(-1)^{(b-b^*)/2}$ times the pull-back of $o_{\underline V}$ along the $\RR$-linear isomorphism 
$$ \wedge^d \phi_V : \wedge^d V \isom \wedge^d \underline V. $$
Here $\wedge^d \phi_V$ is indeed defined over $\RR$ because $V$ and $\underline V$ have the same discriminant.

 When $d$ is even, every elliptic decomposition of $\underline V$ (resp.~$V$) gives rise to an orientation on $\underline V$ (resp.~$V$). We define $\ED(\underline V, o_{\underline V})$ to be the set of elliptic decompositions of $\underline V$ that induce the orientation $o_{\underline V}$, and similarly define $\ED(V,o_V)$.\index{$\ED(\underline V, o_{\underline V})$} \index{$\ED(V,o_V)$} In order to have uniform notation in the odd and even cases, we set \index{$\ED(\underline V)^o$} \index{$\ED(V)^o$}
 \begin{align*}
\ED(\underline V)^o & : = \begin{cases}
\ED(\underline V), & \text{if $d$ is odd}, \\
\ED(\underline V,o_{\underline V}), & \text{if $d$ is even}, \end{cases} \end{align*} and \begin{align*}
  \ED(V)^o & : = \begin{cases}
\ED(V), & \text{if $d$ is odd}, \\
\ED(V,o_V), & \text{if $d$ is even}.  
\end{cases}
 \end{align*} 
\begin{lem}\label{lem:motivation for o_V}
	 Assume that $d$ is even.
	Let $\set{v_1,\cdots, v_d}$ be an orthogonal basis of $V$ and $\set{\underline v_1,\cdots, \underline v_d}$ an orthogonal basis of $\underline V$ satisfying the condition in Definition \ref{defn:fixing isom between quad spaces}. Then $\set{v_1,\cdots, v_d}$ induces the orientation $o_V$ if and only if $\set{\underline v_1,\cdots, \underline v_d}$ induces the orientation $o_{\underline V}$. 
\end{lem}
\begin{proof}
	Comparing the signatures we see that the cardinality of the set $$\set{j\mid 1\leq j \leq d, \phi(v_j) = \underline v_j \otimes \sqrt{-1}}$$ is congruent to $b-b^* \pmod 2$. Hence the determinant of the matrix of $\phi_V$ under the given bases is equal to $ (-1)^{(b-b^*)/2}$. The lemma follows. 
\end{proof}
\subsection{}\label{subsubsec:setting for para tori}
Consider an elliptic endoscopic datum $(H,\lang H, s,\eta)$ for $G$. We assume that it is one of the explicit representatives constructed in \S \ref{subsec:elliptic end data}. Recall that $H = \SO(V^+) \times \SO(V^-)$, where $V^\pm$ are quasi-split quadratic spaces over $\RR$. In the even case, we denote by $o_{V^\pm}$\index{$o_{V^\pm}$} the orientation on $V^\pm$ determined by $[\mathbb B_{V^\pm}]$. (See \S \ref{endoscopic data} for $[\mathbb B_{V^\pm}]$.) We assume that $H$ contains anisotropic maximal tori, or equivalently, that both $\SO(V^+)$ and $\SO(V^-)$ contain anisotropic maximal tori. In particular, $\ED(V^\pm)$ are non-empty. Similarly as in \S \ref{subsubsec:defn of EDo}, we set\index{$\ED(V^\pm)^o $}
$$\ED(V^\pm)^o  : = \begin{cases}
\ED(V^\pm), & \text{if $d^\pm$ is odd}, \\
\ED(V^\pm,o_{V^\pm}), & \text{if $d^\pm$ is even}. \end{cases} $$

Let $m^{\pm} : = \floor{ d^{\pm} /2}$. We fix an element $$\mathcal D_H = (\mathcal D_{H}^+, \mathcal D_H^-) \in \ED(V^+)^o \times \ED(V^-)^o. $$ Then we get a parameterized anisotropic maximal torus
\begin{align*}
f_{\mathcal D_H} : \Uni(1)^{m^+ } \times \Uni(1) ^{m^-} \isom T_{\cD_H^+}\times T_{\cD_H^-} = T_{\cD_H} \subset \SO(V^+)\times \SO(V^-) = H, 
\end{align*}
and a fundamental pair $(T_{\mathcal D_H}, B_{\mathcal D_H})$ in $H$, by the obvious generalization of Definitions \ref{defn:PEMT} and \ref{defn:fund pair}. We also fix $\derD \in \ED(\underline V)^o$ and $\mathcal D \in \ED(V)^o$. Let 
\begin{align*}
f_{\derD}  & : \Uni(1)^m \isom T_{\derD} \subset G^*, \\
 f_{\mathcal D}  & : \Uni(1)^m \isom T_{\cD} \subset G 
\end{align*}
be the associated parameterized anisotropic maximal tori, and let $(T_{\derD}, B_{\derD}), (T_{\mathcal D}, B_{\mathcal D})$ be the associated fundamental pairs in $G^*$ and in $G$.  We define the following composite maps with Convention \ref{convention:identifying U(1)} below in force:
\begin{align*}
	j_{\mathcal D_H, \mathcal D} & : T_{\cD_H} \xrightarrow{f_{\mathcal D_H}^{-1}} \Uni(1) ^{m^+} \times \Uni(1) ^{m^-} \cong \Uni(1) ^m \xrightarrow{f_{\mathcal D}} T_{\cD} , \\
	j_{\mathcal D_H, \derD} & : T_{\cD_H} \xrightarrow{f_{\mathcal D_H}^{-1}} \Uni(1) ^{m^+} \times \Uni(1) ^{m^-} \cong \Uni(1) ^m \xrightarrow{f_{\derD}} T_{\derD}.
\end{align*}

\begin{convention}\label{convention:identifying U(1)}
	We identify $\Uni(1) ^{m^+} \times \Uni(1) ^{m^-}$ with $\Uni(1) ^{m}$ by the isomorphism
	$$ ((g_1, \cdots, g_{m^+}) ,(h_1, \cdots , h_{m^-})) \longmapsto (h_1, \cdots , h_{m^-}, g_1, \cdots, g_{m^+}).$$
\end{convention}

Our next goal is to show that $j_{\mathcal D_H, \mathcal D}$ and $j_{\mathcal D_H, \derD}$ are admissible isomorphisms, in the sense of \S \ref{subsec:admissible isom}. 
\begin{lem}\label{e.d. dual} In the setting of \S \ref{subsubsec:setting for para tori}, the following diagram commutes:
	$$\xymatrix{\widehat{\Uni(1)} ^m \ar@{=}[d] &   \widehat {T_{\derD}} \ar[l]_-{\widehat {f_{\derD}}}  \ar[d] ^{\mathfrak d_{B_{\derD} ,\mathcal B}}\\ (\CC^{\times})^m \ar[r] & \mathcal T   }$$ Here the bottom horizontal map is the isomorphism fixed in Definition \ref{defn:fixing borel pair in dual}, and  $\mathfrak d_{B_{\derD} ,\mathcal B}$ is as in \S \ref{subsec:admissible isom}. 
\end{lem} 
\begin{proof}In the odd case, $\underline V$ is split, so we can fix a hyperbolic basis $\mathbb B$ of $\underline V$. In the even case, we fix a member $\mathbb B$ of the $G^*(\RR)$-orbit $[\mathbb B_{\underline V}]$ of bases of $\underline V$ in Definition \ref{defn:fixing alpha}. When $\underline V$ is split (i.e., when either $d$ is odd or $d$ is even and $\delta$ is trivial), $\mathbb B$ is a hyperbolic basis, and we let $\iota_{\mathbb B}: \GG_m ^m \hookrightarrow G^*$ be the associated embedding as in \S \ref{subsubsec:Borel pairs assoc to hd and nhd}. When $\underline V$ is not split (i.e., when $d$ is even and $\delta$ is non-trivial), $\mathbb B$ is a near-hyperbolic basis, and we let $\iota_{\mathbb B}: \GG_m^{m-1} \times \Uni(1) \hookrightarrow G^*$ be the associated embedding as in \S \ref{subsubsec:Borel pairs assoc to hd and nhd}. In all cases we write $T'_{\mathbb B}$ for the image of $\iota_{\mathbb B}$. We view the base change of $\iota_{\mathbb B}$ to $\CC$ as an isomorphism $\iota_{\mathbb B, \CC}:  \GG_{m, \CC} ^m \isom T'_{\mathbb B, \CC}$ (as we canonically identify $\Uni(1)_{\CC}$ with $\GG_{m,\CC}$). 
	Now we claim that there exists $g\in G^*(\CC)$ such that the diagram
	\begin{align}\label{two para tori}
	\xymatrix{ \Uni(1) ^m_{\CC} \ar[r]^-{f_{\derD, \CC}}  \ar[d] & T_{\derD, \CC }  \\ 
		\GG_{m,\CC}^m \ar[r]^-{\iota_{\mathbb B, \CC}} & T'_{\mathbb B,\CC}  \ar[u]_-{\Int (g)} } 
	\end{align} commutes. Here the left vertical arrow is the canonical isomorphism.
	
To prove the claim, first we observe that the truth of the claim does not depend on the choices of $\mathbb B$ and $\derD$ (as long as they both induce the correct orientation $o_{\underline V}$ in the even case). Using this observation, we easily reduce the claim for both the odd and even cases to the even case where $\underline V$ has signature $(2,2)$. In this case, take a basis $\set{u_1, u_2, u_3, u_4}$ of $\underline V$ under which the quadratic form has matrix $\diag (1,1,-1,-1)$. Without loss of generality we assume that this basis induces the orientation $o_{\underline V}$. Let $V_1$ be the oriented plane spanned by $\{u_1, u_2\}$, and let $V_2$ be the oriented plane spanned by $\{u_3, u_4\}$. Then $(V_1, V_2) \in \ED(\underline V, o_{\underline V})$. Define 
	\begin{align*}
	x_1 & = \frac{1}{2}(u_1 +u_3) , & y_1 & = u_1 - u_3 , & x_2  & = \frac{1}{2}(u_2 -u_4), & 	y_2 & = u_2 + u_4.
	\end{align*}
	Then $ \set{x_1,x_2,y_2,y_1}$ is a hyperbolic basis of $\underline V$, and it induces the orientation $o_{\underline V}$. We may and shall assume that $\derD = (V_1,V_2)$ and $\mathbb B = \set{x_1,x_2,y_2,y_1}$. Let $g \in \End(\underline V \otimes \CC)$ be given by 
	\begin{align*}
x_1 & \longmapsto \frac{1}{2} (u_1 - i u_2), &
y_1 & \longmapsto u_1 + i u_2 , &
x_2 & \longmapsto - \frac{1}{2} (u_3 - i u_4) , &
y_2 & \longmapsto u_3 + i u_4 .&
	\end{align*}
Then $g\in \mathrm{O}(\underline V)(\CC)$, and the diagram (\ref{two para tori}) commutes. We have $\det g = 1$ by direct computation, which proves the claim.

Now by the definition of $B_{\derD}$, we know that $f_{\derD}$ pulls back the based root datum $\mathrm{BRD}(T_{\derD,\CC}, B_{\derD})$ on $(X^*(T_{\derD}), X_*(T_{\derD}))$ to the standard based root datum $\mathrm{BRD}(\mathsf B_m)$ or $\mathrm{BRD}(\mathsf D_m)$ on $(\ZZ^m, \ZZ^m)$. By the commutative diagram (\ref{two para tori}), we know that the isomorphism $\mathrm{BRD}(T_{\derD,\CC}, B_{\derD}) \isom \mathrm{BRD}(\mathsf B_m \text{ or }\mathsf D_m)$ induced by $f_{\derD}$ is equal to the isomorphism ${\mathfrak v^*}'$ fixed in \S \ref{fixing L group}. The lemma then follows from the definition of $\mathfrak d_{B_{\derD}, \mathcal B}$. 
\end{proof}

\begin{lem}\label{e.d. and j}In the setting of \S\ref{subsubsec:setting for para tori}, $j_{\cD_H, \derD}$ is admissible, and it is associated to the Borel pairs $(T_{\cD_H,\CC}, B_{\cD_H})$ and $(T_{\derD, \CC}, B_{\derD})$.
\end{lem}
\begin{proof}The map $\eta: \lang H \to \lang G$ restricts to an isomorphism $\mathcal T_{\widehat H} = \mathcal T_{V^+} \times \mathcal T_{V^-} \isom \mathcal T$. This isomorphism is given by 
 \begin{align*}
(\CC^{\times})^{m^+} \times (\CC^{\times})^{m^-} & \To (\CC^{\times})^{m} \\  ((g_1, \cdots, g_{m^+}) ,(h_1, \cdots , h_{m^-})) & \longmapsto (h_1, \cdots , h_{m^-}, g_1, \cdots, g_{m^+}),
 \end{align*} under the identifications $\mathcal T_{V^+} \cong (\CC^\times)^{m^+}, \mathcal T_{V^-} \cong (\CC^\times)^{m^-}, \mathcal T \cong (\CC^\times)^m$  as in Definition \ref{defn:fixing borel pair in dual}. This fact, together with Lemma \ref{e.d. dual} (applied to $\underline{ V}$ and $V^{\pm}$), implies the current lemma. 
\end{proof}
\begin{lem}\label{lem:j is adm}
	In the setting of \S\ref{subsubsec:setting for para tori}, $j_{\cD_H, \cD}$ is admissible.
\end{lem}
\begin{proof}Since $\ED(V)^o$ is a single $G(\RR)$-orbit, the truth of the lemma does not depend on the choice of $\mathcal D\in \ED(V)^o$. Let $\set{v_1,\cdots, v_d}$ be a basis of $V$ and $\set{\underline v_1,\cdots, \underline v_d}$ a basis of $\underline V$, satisfying the condition in Definition \ref{defn:fixing isom between quad spaces}. Let $m = \floor{d/2}$. Up to reordering, we may assume that $q(v_j) = q(v_{j+1})$ for all $j\in \set{1,3,\cdots, 2m-1}$. When $d$ is even, we may further assume that $\set{v_1,\cdots, v_d}$ induces the orientation $o_V$ (because we may switch the order of $v_1$ and $v_2$ without changing the other conditions). For each $1\leq j \leq m$, let $V_j$ be the oriented plane spanned by $\{v_{2j-1}, v_{2j}\}$, and let $\underline V_j$ be the oriented plane spanned by $\{\underline v_{2j-1}, \underline v_{2j}\}$. Then $(V_j)_j \in \ED(V)^o$ and $(\underline V_j)_j \in \ED(\underline V)$. By Lemma \ref{lem:motivation for o_V}, we have $(\underline V_j)_j \in \ED(\underline V)^o$. In \S \ref{subsubsec:setting for para tori}, we can take $\mathcal D$ to be $(V_j)_j$, and take $\derD$ to be $(\underline V_j)_j$. By Lemma \ref{e.d. and j}, we know that $j_{\cD_H, \derD}$ is admissible. In view of Lemma \ref{lem:facts on adm iso} (3), we complete the proof by noting that $j_{\cD_H, \cD} = \psi_V^{-1} \circ j_{\cD_H,\derD}$. 
\end{proof}

\section{Transfer factors, when \texorpdfstring{$d$}{d} is not divisible by \texorpdfstring{$4$}{4}}\label{transf factors odd}
\subsection{}\label{subsubsec:setting for transf factor}
We keep the setting of \S \ref{subsec:cuspidality and transfer of tori}, and in particular keep the assumption that $G$ and $G^*$ contain anisotropic maximal tori. By an \emph{equivalence class of Whittaker data for $G^*$},\index[n]{equivalence class of Whittaker data} we mean a $G^*(\RR)$-conjugacy class of pairs $(B,\lambda)$ consisting of a Borel subgroup $B$ of $G^*$ defined over $\RR$ and a generic character $\lambda: N_B(\RR) \to \CC^\times$, where $N_B$ denotes the unipotent radical of $B$. See \cite[\S 5.3]{KS99} for more details. It is a standard result that the set of equivalence classes of Whittaker data for $G^*$ is a torsor under the finite abelian group $G^{*, \ad}(\RR)/G^*(\RR)$. 

Assume that $d$ is not divisible by $4$. Then the map $G^*(\RR) \to G ^{*,\ad} (\RR)$ is surjective, which can be seen by noting that $\ker (\coh^1 (\RR, Z_{G^*}) \to \coh^1(\RR, G^*))$ is trivial. Hence $G^*$ has a unique equivalence class of Whittaker data. As in \S \ref{subsubsec:setting for para tori}, we fix an elliptic endoscopic datum $(H,\lang H, s,\eta)$ for $G$, assumed to be one of the explicit representatives constructed in \S \ref{subsec:elliptic end data}. Thus we have $H = \SO(V^+) \times \SO(V^-)$, where $V^\pm$ is quasi-split and has dimension $d^\pm$ and discriminant $\delta^\pm$. As usual, both $V^\pm$ are split in the odd case. We write $m = \floor{d/2}, m^\pm = \floor{d^\pm/2}$. We assume that $H$ contains anisotropic tori.

In this paper, unless otherwise stated, ``transfer factor'' always means ``absolute geometric  transfer factor''.
 
The \emph{Whittaker normalization} of the transfer factors between $H$ and $G^*$ was defined by Kottwitz--Shelstad in \cite[\S 5.3]{KS99} (in the general setting of twisted endoscopy), and a correction was later made in \cite{KSconvention}. In this paper, we always use the classical normalization of local class field theory as opposed to Deligne's normalization; see \cite[\S\S 4.1, 4.2]{KSconvention}. Thus among the four $\Delta, \Delta', \Delta _D, \Delta_D'$ discussed at the end of \cite[\S 5.1]{KSconvention}, we only consider $\Delta$ and $\Delta'$. Moreover, since we always have $s^2=1$, we have $\Delta = \Delta'$. We shall call the transfer factors $\Delta_{\lambda}'(\cdot, \cdot)$ given in \cite[(5.5.2)]{KSconvention} the \emph{Whittaker-normalized   transfer factors}\index[n]{Whittaker normalization of the   transfer factors (between $H$ and $G^*$)}. By the discussion at the end of \cite[\S 5.5]{KSconvention} and by $s^2=1$, we have $$\Delta_{\lambda}' = \epsilon_L (V,\psi)\Delta_0' =  \epsilon_L (V,\psi)\Delta_0,$$ where $\Delta_0$ is the Langlands--Shelstad normalization defined on p.~248 of \cite{LS87}. 

In the following we denote the Whittaker-normalized   transfer factors between $H$ and $G^*$ by $\underline \Delta_{\Wh}(\cdot, \cdot )$.\index{$\underline \Delta_{\Wh}$} Also, having fixed $\psi_V: G \to G^*$ and $u_V: \Gamma_\infty \to G^*(\CC)$ as in \S \ref{Fixing inner twist}, we can derive from $\underline \Delta_{\Wh} (\cdot, \cdot)$ a normalization of the   transfer factors between $H$ and $G$ as in Remark \ref{rem:pure inner form}, to be denoted by $\Delta_{\Wh}(\cdot, \cdot)$. (See \S \ref{subsubsec:Kaletha's transf factor} below for more details.)

In \cite[\S 7]{kottwitzannarbor}, another normalization $\Delta_{j, B}(\cdot, \cdot)$\index{$\Delta_{j,B}$} of the   transfer factors between $H$ and $G$ is considered, which is associated to a certain datum $(j,B)$. The goal of this section is to compare the two normalizations $\Delta_{\Wh}$ and $\Delta_{j,B}$. 

\emph{In the following we assume that $V$ is of signature $(p,q)$ with $p>q$, and that $d=p+q$ is not divisible by $4$.}

\section*{Transfer factors between $H$ and $G^*$}
  
\begin{defn}\label{defn:Whitt ED} We define a subset $\ED(\derV)^o_{\Wh}$\index{$\ED(\derV)^o_{\Wh}$} of $\ED(\derV)^o$ (see \S \ref{subsubsec:defn of EDo}) as follows. When $d$ is odd, let $\ED(\derV)^o_{\Wh}$ consist of those $(\derV_j)_{1\leq j \leq m}\in \ED(\derV)^o= \ED(\derV)$ such that $\derV_j$ is $(-1) ^{j+1} \sign(\delta)$-definite for each $1\leq j\leq m$. When $d$ is even (but not divisible by $4$), let $\ED(\derV)^o_{\Wh}$ consist of those $(\derV_j)_{1\leq j \leq m}\in \ED(\derV)^o$ such that $\derV_j$ is $(-1) ^{j+1}$-definite for each $1\leq j\leq m$.
\end{defn}
\begin{rem}\label{rem:existence of Whitt ED}
	Let $(\derV_j)_j $ be an arbitrary element of $ \ED(\derV)^o$. Recall that $\derV$ has discriminant $\delta$ and determinant $(-1)^m\delta$. In the odd case, $\derV$ has signature $(m+1,m)$ when $\delta>0$, and signature $(m,m+1)$ when $\delta<0$. Therefore there are precisely $\ceil{m/2}$ (resp.~$\floor{m/2}$) positive definite planes among the $\derV_j$'s when $\delta >0$ (resp.~when $\delta<0$). It follows that there exists $\sigma \in \mathfrak S_m$ such that $(\derV_{\sigma(j)})_j \in \ED(\derV)^o_{\Wh}$. In the even case, there are precisely $\ceil{m/2}$ positive definite planes among the $\derV_j$'s no matter what $\delta$ is, so again there exists $\sigma \in \mathfrak S_m$ such that $(\derV_{\sigma(j)})_j \in \ED(\derV)^o_{\Wh}$. (Here $(\derV_{\sigma(j)})_j$ automatically induces the same orientation on $\derV$ as $(\derV_j)_j$ does.) Moreover, in both cases $\ED(\derV)^o_{\Wh}$ is a single $G^*(\RR)$-orbit with respect to the natural $G^*(\RR)$-action on $\ED(\derV)^o$. 
\end{rem}
\begin{lem}\label{lem:meaning of Whitt}
Let $\derD \in \ED(\derV)^o_{\Wh}$. Let $(T_{\derD}, B_{\derD})$ be the associated fundamental pair in $G^*$, as in \S \ref{subsubsec:constrn with ED}. Then every $B_{\derD}$-simple root in $X^*(T_{\derD})$ is (imaginary) non-compact.\index[n]{non-compact root} In other words, $(T_{\derD}, B_{\derD})$ is a \emph{fundamental pair of Whittaker type} in the terminology of \cite{sheellipticone}. 
\end{lem}
\begin{proof}	Let $\set{\epsilon_1^\vee,\cdots ,\epsilon_m^\vee}$ be the standard basis of $X_*(\Uni(1)^m)$, and let $\set{\epsilon_1,\cdots, \epsilon_m}$ be the standard basis of $X^*(\Uni(1)^m)$. Let $f_{\derD}: \Uni(1)^m \isom T_{\derD}$ be the parameterized anisotropic maximal torus associated to $\derD$, as in \S \ref{subsubsec:constrn with ED}. We identify $X^*(T_{\derD})$ with $X^*(\Uni(1)^m)$ via $f_{\derD}$. Then the $B_{\derD}$-simple roots are $\alpha_j = \epsilon_j - \epsilon_{j+1}, 1\leq j \leq m-1 $, and $\alpha_m =  \epsilon_m$ (resp.~$\alpha_m = \epsilon_{m-1}+\epsilon_m$) in the odd (resp.~even) case. Denote the complex conjugation by $\tau$. It suffices to check that for each $1\leq j \leq m $ and for one (and hence any) root vector $E_{j}$ of $\alpha_j$, we have 
	$$ [E_{j}, \tau E_{j}] = C(E_j) H_{j} \in \Lie G^*$$ for some $C(E_j) \in \RR_{> 0}$. Here $H_{j}$ is the coroot $\alpha_j^{\vee}$ viewed as an element of $\Lie G^*$.
	
	Write $\derD = (\derV_j)_j$. Since $\derD \in \ED(\derV)^o_{\Wh}$, there exists an integer $r$ such that $\derV_j$ is $(-1)^{r+j}$-definite for each $1\leq j \leq m$. Moreover, we have $(-1)^r= - \sign(\delta)$ when $d$ is odd, and $(-1)^r = -1$ when $d$ is even. For each $1\leq j \leq m$, let $\{f_j, f_j'\}$ be an orthogonal basis of $\derV_j$ inducing the given orientation on $\derV_j$ such that $$\underline{q} (f_j) = \underline{q} (f_j') = (-1) ^{r+j}. $$ Let 
	\begin{align*}
 e_j & := f_j \otimes 1  - f_j' \otimes i \in \underline{V} \otimes \CC,  \\  e_j' & : = (-1)^{r+j}\frac{1}{2}\tau ( e_j) \in \underline{V} \otimes \CC.
	\end{align*} 
	 In the odd case we also fix a non-zero vector $l\in \derV$ which is orthogonal to each $\derV_j$, and satisfies $\underline q(l) \in \set{\pm 1}.$ Thus $\underline q(l)$ is the sign of the determinant of the quadratic space $\derV$, which is $(-1)^{m}\sign(\delta)= (-1)^{r+m+1}$. 
	 
	 Now $\set{e_1,\cdots, e_m , e_1' ,\cdots, e_m', l}$ (resp.~$\set{e_1,\cdots, e_m , e_1',\cdots, e_m'}$) is a $ \CC$-basis of $\underline{V} \otimes \CC$ in the odd (resp.~even) case, and we have  
	 \begin{align*}
[e_j , e_k]&  = [e_j', e_k'] = [e_j,l] = [e_j', l] = 0,  &  [ e_j , e_k'] & = \delta_{j,k}. &
	 \end{align*}
Note that for each $1\leq j \leq m$, the cocharacter $f_{\derD} \circ \epsilon_j^{\vee}$ of $G^*$ acts on $\derV$ with weight $1$ on $e_j$, weight $-1$ on $e_j'$, and weight $0$ on $e_k, e_k'$ for all $k\neq j$. In the odd case, it also acts with weight $0$ on $l$. 
	
	For $1\leq j \leq m-1$, we define $E_j\in \End(\derV\otimes \CC)$ by 
	\begin{align*} e_j & \longmapsto 0, &  e_{j+1} & \longmapsto e_j, & e_j' & \longmapsto -e_{j+1}', & \\ e_{j+1}'& \longmapsto 0 ,&  e_k , e_{k}' & \longmapsto 0  ~ \text{for }k \notin \set{j, j+1} , \\ l & \longmapsto 0  ~\text{(if $d$ is odd)} . 
	\end{align*} It is easy to see that $E_j\in \Lie G^*_{\CC}$ and that it is indeed a root vector of $\alpha_j$. We compute that $\tau E_j$ is given by 
		\begin{align*} e_j & \longmapsto e_{j+1} ,&  e_{j+1} & \longmapsto 0, &  
		 e_j' & \longmapsto 0, \\ 
		 e_{j+1}'& \longmapsto -e_j',  & 
		 e_k , e_{k}' & \longmapsto 0  ~ \text{for }k \notin \set{j, j+1}, &  \\  l & \longmapsto 0 ~ \text{(if $d$ is odd)}  . 
	\end{align*}
Then $[E_{j}, \tau E_{j}]$ is given by 
	\begin{align*} e_j & \longmapsto e_j, &  e_{j+1} & \longmapsto -e_{j+1} , &  e_j' & \longmapsto -e_{j}', &  \\ e_{j+1}'& \longmapsto e_{j+1}' & e_k , e_{k}', & \longmapsto 0  ~ \text{for }k \notin \set{j, j+1}, &  \\ l & \longmapsto 0 ~\text{(if $d$ is odd)} . 
\end{align*}  Thus $[E_j,\tau E_j] = H_j$, as desired.

In the odd case, we define $E_m \in \End(\derV\otimes \CC)$ by 
\begin{align*}
l & \longmapsto e_m, &  e_m' & \longmapsto  -\underline q(l)^{-1} l  = (-1)^{r+m} l , \\  e_k & \longmapsto 0 ~ \text{for }1\leq k \leq m , &  e_k ' & \longmapsto 0 ~ \text{for }1\leq k \leq m-1.
	\end{align*} Then $E_m\in \Lie G^*_{\CC}$ and it is a root vector of $\alpha_m$. We compute that $\tau E_m$ is given by 
	\begin{align*}
	l & \longmapsto (-1)^{m+r} 2 e_m', &  e_m & \longmapsto   2l , \\  e_k & \longmapsto 0 ~ \text{for }1\leq k \leq m-1, & e_k ' & \longmapsto 0 ~ \text{for }1\leq k \leq m.
	\end{align*} 
Then $ [E_{m}, \tau E_{m}] $ is given by
		\begin{align*}
	l & \longmapsto 0 , & e_m & \longmapsto   2 e_m,  \\ e_m' &\longmapsto -2e_m' ,&  e_k, e_k' & \longmapsto 0 ~ \text{for }1\leq k \leq m-1.
	\end{align*} Thus $[E_m,\tau E_m] = H_m$, as desired. 
	
	In the even case, we define $E_m \in \End (\derV\otimes \CC)$ by 
	\begin{align*}
	e_m' & \longmapsto e_{m-1} , &  e_{m-1}' & \longmapsto  -e_m , \\  e_k & \longmapsto 0 ~ \text{for }1\leq k \leq m , & e_k ' & \longmapsto 0 ~ \text{for }1\leq k \leq m-2.
	\end{align*} Then $E_m\in \Lie G^*_{\CC}$ and it is a root vector of $\alpha_m$. We compute that $\tau E_m$ is given by 
	\begin{align*}
e_m & \longmapsto -e_{m-1}' , &  e_{m-1} & \longmapsto   e_m', \\   e_k & \longmapsto 0 ~ \text{for }1\leq k \leq m-2 , &  e_k' & \longmapsto 0 ~ \text{for }1\leq k \leq m.
	\end{align*} 
Then $ [E_{m}, \tau E_{m}] $ is given by
	\begin{align*}
	e_m & \longmapsto e_m , & e_{m-1} & \longmapsto  e_{m-1} , \\ e_m' &\longmapsto - e_m', &  e_{m-1}' & \longmapsto  -e_{m-1}',  \\  e_k, e_k' & \longmapsto 0 ~ \text{for }1\leq k \leq m-2.
	\end{align*} Thus $[E_m,\tau E_m] = H_m$, as desired. 
\end{proof}
\subsection{}\label{subsubsec:derj and derB}
As before, we fix the standard Borel pair $(\mathcal T, \mathcal B)$ in $\widehat G$, and the standard Borel pair $(\mathcal T_{V^+} \times \mathcal T_{V^-}, \mathcal B_{V^+} \times \mathcal B_{V^-}) = (\mathcal T_{\widehat H}, \mathcal B_{\widehat H})$ in $\widehat H = \widehat{\SO(V^+) } \times \widehat{\SO(V^-)}$; see Definition \ref{defn:fixing borel pair in dual} and \S \ref{endoscopic data}. We extend $(\mathcal T,\mathcal B)$ to a $\Gamma_{\infty}$-stable  splitting $\spl_{\widehat G}$, and extend $(\mathcal T_{\widehat H}, \mathcal B_{\widehat{H}})$ to a $\Gamma_{\infty}$-stable  splitting $\spl_{\widehat H}$.

Note that $\eta: \lang H \to \lang G$ maps $(\mathcal T_{\widehat H},\mathcal B_{\widehat H})$ into $(\mathcal T, \mathcal B)$. Given this property, and given the choices $\spl_{\widehat G}$ and $\spl_{\widehat H}$, we have the following constructions (see \cite[\S\S 7.3, 8.1]{shestructure}, \cite[\S 7]{she2}, \cite[\S 3]{shenote}):
\begin{itemize}
	\item Inside each equivalence class $\boldsymbol{\varphi}$ of discrete Langlands parameters for $G^*$, there is a canonical $\mathcal T$-conjugacy class of parameters, whose elements we shall call \emph{almost canonical representatives}.\index[n]{almost canonical representatives} Similarly, inside each equivalence class of discrete Langlands parameters for $H$, there are almost canonical representatives. 
	\item Let $\boldsymbol{\varphi}$ be as above. Consider the set 
	of equivalence classes of discrete Langlands parameters for $H$ that induce $\boldsymbol{\varphi}$ via $\eta: \lang H \to \lang G$. Then this set is non-empty (because $H$ contains anisotropic maximal tori), and it contains a canonical element $\boldsymbol{\varphi_H}$, called \emph{well-positioned}\index[n]{well-positioned (discrete Langlands parameter)}, which is uniquely characterized by the following property: For one (and hence any) almost canonical representative $\varphi_H$ of $\boldsymbol{\varphi_H}$, the composition $\eta \circ \varphi_H$ is an almost canonical representative of $\boldsymbol{\varphi}$.
\end{itemize} 

We now choose an arbitrary equivalence class $\boldsymbol{\varphi}$ of discrete Langlands parameters for $G^*$, and obtain  $\boldsymbol{\varphi_H} $ from $\boldsymbol{\varphi}$ as above. Choose an almost canonical representative $\varphi_H$ of $\boldsymbol{\varphi_H}$, and let $\varphi:= \eta\circ \varphi_H$. Thus $\varphi$ is an almost canonical representative of $\boldsymbol{\varphi}$. By construction, the Borel pair in $\widehat G$ (resp.~$\widehat H$) determined by $\varphi$ (resp.~$\varphi_H$) as on p.~182 of \cite{kottwitzannarbor} is $(\mathcal T, \mathcal B)$ (resp.~$(\mathcal T_{\widehat H}, \mathcal B_{\widehat H})$).

Let $\pi_{0}$ be the unique generic member (with respect to the unique equivalence class of Whittaker data) of the L-packet $\Pi_{\boldsymbol{\varphi }}$; see \cite{kostantwhittaker} and \cite{vogankirillov}. As proved by Shelstad in \cite{she3} (see \cite[Thm.~3.6]{shenote}), we have \begin{align}
\label{Wh=1}
\underline \Delta_{\Wh}^{\sspec} (\boldsymbol{\varphi_H}, \pi_0) =1.
\end{align}
Here $\underline \Delta_{\Wh}^{\sspec}(\cdot, \cdot)$\index{$\underline \Delta_{\Wh}^{\sspec}$} are the (absolute) spectral transfer factors between $H$ and $G^*$, under the Whittaker normalization. (In fact, (\ref{Wh=1}) holds for all discrete $\boldsymbol{\varphi_H}$ inducing $\boldsymbol{\varphi}$, not just the well-positioned one; cf.~\cite[\S 5.6]{kalrigid}. We will not need this.) By \cite[Lem.~12.3]{she2}, the transfer factors $\underline \Delta_{\Wh}^{\sspec}(\cdot, \cdot)$ are \emph{compatible} with the Whittaker-normalized   transfer factors $\underline \Delta_{\Wh}(\cdot,\cdot)$, in the sense that the endoscopic character relations defined by the former are satisfied when the test functions satisfy orbital integral relations with respect to the latter.

We now fix $\derD$ and $\cD_H$ as in \S \ref{subsubsec:setting for para tori}, and assume that $\derD \in \ED(\derV)^o_{\Wh}$. We keep the notation in \S \ref{subsubsec:setting for para tori}. By Lemma {\ref{e.d. and j}}, the map $j_{\cD_H,\derD}$ constructed in \S \ref{subsubsec:setting for para tori} is an admissible isomorphism. We note that $(j_{\cD_H,\derD} , B_{\derD} ,  B_{\cD_H})$ is \emph{aligned}\index[n]{aligned} with $\boldsymbol{\varphi_H}$ in the sense of \cite[p.~184]{kottwitzannarbor}, which follows from our assumption that $\boldsymbol{\varphi_H} $ is well-positioned. In the following, we abbreviate $j_{\cD_H,\derD}$ as $\underline j$, and abbreviate $B_{\derD}$ as $\underline B$.

In \cite[\S 7]{kottwitzannarbor}, a normalization \index{$\Delta_{\underline j, \underline B}$}
$$\Delta_{\underline j, \underline B}(\cdot, \cdot)$$ of the   transfer factors between $H$ and $G^*$ is defined. Write $\Delta_{\underline j, \underline B} ^{\sspec}(\cdot,\cdot)$\index{$\Delta_{\underline j, \underline B} ^{\sspec}$} for the spectral transfer factors normalized compatibly with $\Delta_{\underline j, \underline B}(\cdot,\cdot)$. Then since $(\underline j, \underline B)$ is aligned with $\boldsymbol{\varphi_H}$, we have (see \cite[p.~185]{kottwitzannarbor})
\begin{align}\label{eq:spec id in Ann Arbor}
 \Delta_{\underline j, \underline B}^{\sspec} ( \boldsymbol{\varphi_H} ,  \pi ( \varphi, \omega^{-1} \underline B)) = \lprod{a_{\omega}, s},
\end{align}
for all $\omega \in \Omega_{\CC} (G^*, T_{\derD})$. Here $a_{\omega}$\index{$a_{\omega}$} is defined in \cite[\S 5]{kottwitzannarbor}, and we shall not need the definition of $\lprod{a_{\omega},s}$ except the fact that $$\lprod{a_1, s} = 1. $$
 Now by Vogan's classification theorem for generic representations \cite[Thm.~6.2]{vogankirillov} and by Lemma \ref{lem:meaning of Whitt}, we know that $\pi_0 = \pi (\varphi, \underline B)$. Hence by setting $\omega =1$ in (\ref{eq:spec id in Ann Arbor}) we obtain 
\begin{align}
\label{Kott=1}
\Delta_{\underline j, \underline B}^{\sspec} ( \boldsymbol{\varphi_H} ,  \pi_0 ) = 1.
\end{align}
Comparing (\ref{Wh=1}) and (\ref{Kott=1}), we see that 
\begin{align}\label{eq:comp spec trans}
\underline \Delta_{\Wh}^{\sspec} = \Delta_{\underline j, \underline B} ^{\sspec}.
\end{align}
Now as we recalled above,  $\underline \Delta^{\sspec}_{\Wh}$ is compatible with $\underline \Delta_{\Wh}$. Hence it follows from (\ref{eq:comp spec trans}) that 
$$\Delta_{\underline j, \underline B} = \underline \Delta_{\Wh}.$$ We record this in the following lemma. 
\begin{lem}\label{answer for H and G^*}
	Let $\derD \in \ED(\derV)^o_{\Wh}$ and  $\cD_H\in \ED(V^+)^o \times \ED(V^-)^o$. Let $\underline j = j_{\cD_H ,\derD}$ and let $\underline B = B_{\derD}$. Then $\underline \Delta_{\Wh} = \Delta_{\underline j , \underline B}.$ \qed 
\end{lem}

\section*{Transfer factors between $H$ and $G$}

\subsection{} \label{subsubsec:Kaletha's transf factor}
Recall from \S \ref{Fixing inner twist} that we have fixed an isomorphism $\phi_V: V\otimes \CC \isom \derV \otimes \CC$ between quadratic spaces over $\CC$, and used $\phi_V$ to define the inner twisting $\psi_V: G_{\CC} \isom G^*_{\CC}$ and the cocycle $u_V: \Gamma_{\infty} \to G^*(\CC)$, satisfying (\ref{eq:pure inner twist}). As we have explained in Remark \ref{rem:pure inner form}, these extra data allow us to derive from $\underline \Delta_{\Wh}(\cdot,\cdot)$ a normalization of the  transfer factors between $H$ and $G$, which we denote by $\Delta_{\Wh}(\cdot,\cdot)$.\index{$\Delta_{\Wh}$} 

We now recall the characterization of $\Delta_{\Wh}$ in terms of $\underline \Delta_{\Wh}$ following \cite[\S 2.2]{kaldepth0}. Let $T_H$, $T$, and $\derT$ be anisotropic maximal tori in $H$, $G$, and $G^*$, respectively. (Recall that $H$, $G$, and $G^*$ all contain anisotropic maximal tori.) Assume that $u_V$ takes values in $\derT(\CC)$. (We shall see in \S \ref{subsubsec:describe u} below that this can indeed be arranged.) Let $j: T_{H,\CC} \to T_{\CC}$ and $\derj: T_{H,\CC} \to \derT_{\CC}$ be arbitrary admissible isomorphisms; see \S \ref{subsec:admissible isom}. Note that $T_{H}$, $T$, and $\derT$ are all isomorphic to $\Uni(1)^m$, and so $j$ and $\derj$ are necessarily defined over $\RR$. Let $\gamma^H \in T_H(\RR)$, and let $$\gamma : = j(\gamma^H), \quad \underline \gamma : = \derj(\gamma^H). $$ Assume that $\gamma$ and $\underline \gamma$ are strongly regular. Then $\Delta_{\Wh}$ is characterized by the following formula:   
\begin{align}\label{characterization of Wh for G}
\Delta_{\Wh} (\gamma^H,\gamma) = \underline \Delta_{\Wh} (\gamma^H,\underline \gamma) \lprod{\inv(\gamma, \underline \gamma), s_{\gamma^H,\underline \gamma}} ^{-1},
\end{align}
 where $\inv(\gamma, \underline \gamma)$ and $s_{\gamma^H,\underline \gamma}$ are defined as follows. 
\begin{itemize} 
\item Define  $\inv(\gamma, \underline \gamma)$\index{$\inv(\gamma, \underline \gamma)$} to be the image of the cocycle $(\rho \mapsto u_V(\rho))$ under the Tate--Nakayama isomorphism $ \coh ^1(\RR, \derT) \isom \tatecoh ^{-1} (\Gamma_{\infty}, X_*(\derT))$. In our case, since $\derT \cong \Uni(1)^m$, the norm map on $X_*(\derT)$ is zero, and so $\tatecoh ^{-1} (\Gamma_{\infty}, X_*(\derT))$ is simply $X_*(\derT) _{\Gamma_{\infty}}$. 
\item Define $s_{\gamma^H,\underline \gamma}$ to be  the image of $s\in Z(\widehat H)$ (which is part of the endoscopic datum) under the composite map
\begin{align*}
	\xymatrix{
Z(\widehat H) \ar@{^{(}->}[r] &  \widehat {T_H} \ar[r]^{\widehat {\derj}} & \widehat{\derT}.}
\end{align*}
Here the first map is the common restriction to $Z(\widehat H)$ of any isomorphism $\mathcal T_{\widehat H} \isom \widehat{T_H}$ of the form $\mathfrak d_{B_H, \mathcal B_{\widehat H}}^{-1}$ for any Borel subgroup $B_H$ of $H_{\CC}$ containing $T_{H, \CC}$; see \S \ref{subsec:admissible isom}. We know that $s_{\gamma^H, \underline \gamma}$ is invariant under $\Gamma_{\infty}$ since, in our case, it is of order at most $2$ and the non-trivial element of $\Gamma_{\infty}$ acts on $ \widehat{\derT}$ by inversion. Thus $s_{\gamma^H, \underline \gamma}$ can be paired with $\inv(\gamma,\underline \gamma)$. 
\end{itemize}
\begin{defn}\label{defn:Whitt normalization for G}
	We call $\Delta_{\Wh}(\cdot,\cdot)$ as in (\ref{characterization of Wh for G}) the \emph{Whittaker-normalized   transfer factors between $H$ and $G$}.\index[n]{Whittaker normalization of the  transfer factors (between $H$ and $G$)}
\end{defn}

\begin{defn} \label{defn of suitable} 
	Let $\ED(\derV)^o_{\Wh, \phi_V}$\index{$\ED(\derV)^o_{\Wh, \phi_V}$} be the set of tuples $(\underline V_j, \lambda_j)_{1\leq j \leq m}$, where $(\underline V_j)_j \in \ED(\derV)^o_{\Wh}$ (see Definition \ref{defn:Whitt ED}), and $\lambda_1,\cdots, \lambda_m \in \set{1,\sqrt{-1}}$, satisfying the following conditions.  
	\begin{enumerate}
		\item For each $1\leq j \leq m$, we have $\phi_V^{-1 } (\derV_j) \subset V \otimes \lambda_j^{-1}$.
		\item There exists $j_0\in \ZZ$ such that for each $1\leq j\leq m$, we have $\lambda_j = \sqrt{-1}$ if and only if $\derV_j$ is negative definite and $j\leq j_0$.
		\item If $d$ is odd, then the restriction of $\phi_V^{-1} : \derV\otimes \CC \to V\otimes \CC$ to the orthogonal complement of $\bigoplus_{j=1}^m \derV_j$ in $\derV$ is defined over $\RR$.
	\end{enumerate}
\end{defn}
\begin{rem} The set $\ED(\derV)^o_{\Wh, \phi_V}$ is non-empty. This follows from the condition in Definition \ref{defn:fixing isom between quad spaces}, the fact that $V$ and $\derV$ have the same discriminant, and Remark \ref{rem:existence of Whitt ED}.  
\end{rem}
\subsection{}
Let $(\derV_j , \lambda_j)_j \in \ED(\derV)^o_{\Wh,\phi_V}$ as in Definition \ref{defn of suitable}. We construct an element $(V_j)_j \in \ED(V)^o$ as follows.  
For each $j$, let $\{f_j, f_j'\}$ be a basis of $\derV_j$ inducing the given orientation on $\derV_j$. Then the vectors $\lambda_j \phi^{-1}(f_j), \lambda_j\phi^{-1}(f_j') \in V\otimes \CC$ lie in $V\otimes 1$. We identify $V\otimes 1$ with $V$, and let $V_j$ be the oriented plane spanned by $\{\lambda_j \phi^{-1}(f_j),  \lambda_j\phi^{-1}(f_j') \}$. Then $(V_j)_j$ is an element of $\ED(V)$. By Lemma \ref{lem:motivation for o_V}, we have $(V_j)_j \in \ED(V)^o$. The construction $(\derV_j, \lambda_j)_j \mapsto (V_j)_j$ gives a map 
\begin{align}\label{eq:transfer of ED} \ED(\derV)^o_{\Wh,\phi_V} \To \ED(V)^o.
\end{align}

\begin{defn}\label{defn:EDNice} We define a subset $\EDNice$\index{$\EDNice$} of $\ED(V)^o$ as follows. When $d$ is odd, we let $\EDNice$ consist of those $(V_j)_j\in \ED(V)^o = \ED(V)$ for which there exists $j_0\in \ZZ$ such that 
	\begin{multline*}
\set{j \mid 1\leq j \leq m, V_j \text{ is negative definite}}  \\ = \set{j \mid 1\leq j \leq m,  j>j_0, \text{ and }(-1)^j = \sign(\delta)} . 
	\end{multline*}
When $d$ is even (but not divisible by $4$), we let $\EDNice$ consist of those $(V_j)_j\in \ED(V)^o$ for which there exists $j_0\in \ZZ$ such that 
\begin{multline*}
\set{j \mid 1\leq j \leq m, V_j \text{ is negative definite}} \\ = \set{j \mid 1\leq j \leq m, j>j_0, \text{ and }(-1)^j = 1} .
\end{multline*}
\end{defn}
\begin{eg} \label{eg:nice ED}
	Let $\cD = (V_j)_j$ be an arbitrary element of $ \ED(V)^o$. Recall that $V$ has signature $(p,q)$. If $d$ is odd and $q=2$, then $\cD$ is in $\EDNice$ if and only if $V_m$ is negative definite. If $d$ is odd and $q \leq 1$, then $\cD$ is automatically in $\EDNice$. If $d$ is even (but not divisible by $4$) and $q=2$, then $\cD$ is in $\EDNice$ if and only if $V_{m-1}$ is negative definite. If $d$ is even (but not divisible by $4$) and $q=0$, then $\cD$ is automatically in $\EDNice$. 
\end{eg}
\begin{lem}
	The image of the map (\ref{eq:transfer of ED}) is contained in $\EDNice$. 
\end{lem}
\begin{proof}This is clear from the definitions. 
\end{proof}
\subsection{}\label{subsubsec:describe u}
Now let $(\derV_j, \lambda_j)_j $ be an element of $\ED(\derV)^o_{\Wh,\phi_V}$, with image $(V_j)_j \in \EDNice$ under the map (\ref{eq:transfer of ED}). Write $\derD$ for the element $(\derV_j)_j\in\ED(\derV)^o_{\Wh}$, and write $\cD$ for the element $(V_j)_j \in \EDNice$. Write $\vec{\lambda}$ for the tuple $(\lambda_j)_j$. Let $$f_{\derD}: \Uni(1)^m \To T_{\derD}$$ be the parameterized anisotropic maximal torus in $G^*$ associated to $\derD$, and let $$f_{\cD}: \Uni(1)^m \To T_{\cD}$$ be the parameterized anisotropic maximal torus in $G$ associated to $\cD$. Also, let $(T_{\derD}, B_{\derD})$ be the fundamental pair in $G^*$ associated to $\derD$, and let $(T_{\cD}, B_{\cD})$ be the fundamental pair in $G$ associated to $\cD$.
We abbreviate $(T_{\derD}, B_{\derD})$ as $(\derT, \derB)$, and abbreviate $(T_{\cD}, B_{\cD})$ as $(T,B)$. 

 Note that we have 
 \begin{align}\label{eq:psi and f}
f_{\derD} = \psi_V \circ f_{\cD} ,
 \end{align}
 which is clear from the definition of $\psi_V$ in \S \ref{subsubsec:defn of the cocycle u}. In particular, the cocycle $u_V$ takes values in $\derT(\CC)$. More precisely, for $\rho=\tau$ the complex conjugation, $u_V(\tau)$ acts as $-1$ on $\derV_j$ for those $j$ such that  $\lambda_j = \sqrt{-1}$, and acts as the identity on the orthogonal complement of these $\derV_j$'s. It follows that $u_V(\tau) \in \derT(\RR)$. Another consequence of the relation (\ref{eq:psi and f}) is that $\psi_V$ sends the Borel pair $(T_{\CC}, B)$ in $G_{\CC}$ to the Borel pair $(\derT_{\CC},\derB)$ in $G^*_{\CC}$. 

Take any $\cD_H\in \ED(V^+)^o \times \ED(V^-)^o$, and define 
\begin{align*}
j_{\cD_H, \derD}: T_{\cD_H} &\isom T_{\derD} \\ j_{\cD_H, \cD}: T_{\cD_H} & \isom T_{\cD}
\end{align*}
as in \S \ref{subsubsec:setting for para tori} (where $\derD$ and $\cD$ are fixed in the last paragraph.) We abbreviate $j_{\cD_H, \derD}$ as $\underline j$, and abbreviate $j_{\cD_H,\cD}$ as $j$. Let $(T_{\cD_H}, B_{\cD_H})$ be the fundamental pair in $H$ associated to $\cD_H$. We abbreviate $(T_{\cD_H},B_{\cD_H})$ as $(T_H,B_H)$. Take a test element $\gamma^H \in T_{\cD_H}(\RR)$, and let 
\begin{align*}
\gamma &: = j(\gamma^H),  & \underline \gamma & : = \derj(\gamma^H).&
\end{align*}
 Assume that $\gamma$ and $\underline \gamma$ are strongly regular.

\begin{lem} \label{inv and s} Keep the setting of \S \ref{subsubsec:describe u}. Let $\lprod{\inv(\gamma, \underline \gamma), s_{\gamma^H,\underline \gamma}}$ be the pairing defined in \S \ref{subsubsec:Kaletha's transf factor}. Then we have 
$$ \lprod{\inv(\gamma, \underline \gamma), s_{\gamma^H,\underline \gamma}} = (-1) ^{k(m^-, \vec{\lambda})}, $$ where\index{$k(m^-, \vec{\lambda})$} $$ k(m^-, \vec{\lambda}) := \#\set{j \mid 1\leq j \leq m^- , \lambda_j = \sqrt{-1}}.$$
\end{lem}
\begin{proof}
	By \cite[Lem.~2.3.3]{kaldepth0}, the element $\inv(\gamma, \underline \gamma) \in X_*(\derT) _{\Gamma_{\infty}}$ is equal to the image of any element $\mu \in X_*(\derT)$ such that $\mu(-1) = u_V(\tau)$, where $\tau$ is the complex conjugation. We identify $X_*(\derT)$ with $\ZZ^m$ via $f_{\derD}: \Uni(1)^m \isom \derT$, and let $\set{\epsilon_1^\vee, \cdots, \epsilon_m^{\vee}}$ be the natural basis. By the description of $u_V(\tau)$ in \S \ref{subsubsec:describe u}, we can take $\mu$ to be 
	$$\mu = \sum_{1\leq j \leq m, \lambda_j = \sqrt{-1}} \epsilon_j^{\vee}.$$
 	On the other hand, if we identify $\widehat{ \derT} $ with $(\CC^{\times})^m$ under $\widehat {f_{\derD}}$, then the element $s_{\gamma^H,\underline\gamma} \in \widehat{ \derT}$ is given by 
 	$$ (\underbrace{-1,\cdots, -1}_{m^-}, \underbrace{1,\cdots, 1}_{m^+}) \in (\CC^\times)^m. $$
 	(Remember that Convention \ref{convention:identifying U(1)} is in force in the definition of $j_{\cD_H, \derD}$ in \S \ref{subsubsec:setting for para tori}.) The lemma follows by evaluating $\mu$ at the above element.
\end{proof}

 \begin{lem} \label{easy lemma}Keep the setting of \S \ref{subsubsec:describe u}. We have 
$$\Delta _{\underline j, \underline B} (\gamma^H, \underline \gamma) = (-1)^{q(G)+q(G^*)} \Delta_{j ,B} (\gamma^H, \gamma). $$ Here $\Delta_{\underline j, \underline B}$ (resp.~$\Delta_{j,B}$) is the normalization of the   transfer factors between $H$ and $G^*$ (between $H$ and $G$), associated to $(\underline j, \underline B)$ (resp.~$(j,B)$), as defined in \cite[\S 7]{kottwitzannarbor}. The numbers $q(G)$ and $q(G^*)$ are as in Definition \ref{defn:q(G)}.
 \end{lem}
 \begin{proof}
 	By the formula for $\Delta_{j,B}$ on p.~184 of \cite{kottwitzannarbor}, we have 
 	\begin{align*}
\Delta _{\underline j, \underline B} (\gamma^H, \underline \gamma) &= (-1) ^{q(G^*) + q(H)} \chi_{G^*, H} (\underline \gamma) \Delta_{\underline B} (\underline \gamma^{-1}) \Delta_{B_H} ((\gamma^H)^{-1}) ^{-1}  \\ \Delta _{j, B} (\gamma^H,\gamma) &= (-1) ^{q(G) + q(H)} \chi_{G, H} (\gamma) \Delta_{B} (\gamma^{-1}) \Delta_{B_H} ((\gamma^H)^{-1}) ^{-1} .
 	\end{align*} Here $\Delta_{\underline B}$, $ \Delta_{B}$, and $\Delta _{B_H}$ are as in Definition \ref{defn:Weyl denominator}, and we do not explain the definitions of $\chi_{G^*, H}$ and $\chi_{G,H}$. Since $\psi$ sends the Borel pair $(T_{\CC}, B)$ to $(\derT_{\CC}, \derB)$, we know that 
 	$$\Delta_{\underline B} (\underline \gamma^{-1})  = \Delta_{B} (\gamma^{-1}).$$ It remains to show that $$\chi_{G^*, H} (\underline \gamma) = \chi_{G,H} (\gamma).$$ Unraveling the definitions of these terms on p.~184 of \cite{kottwitzannarbor}, we are reduced to checking that the following diagram commutes up to $\widehat G$-conjugation: 
 	$$ \xymatrix{ \lang T \ar[d] \ar[r]^{\eta_{B}} & \lang G \ar@{=}[d] \\ \lang {\derT} \ar[r]^{\eta_{\underline B } } & \lang G   }$$ where the left vertical arrow is induced by $\psi_V|_{T}: T \isom \derT$ (defined over $\RR$). This is true by the characterizations (a) (b) on p.~183 of \cite{kottwitzannarbor}, in view of the fact that $\psi_V(B) = \derB$.   
 \end{proof}
\begin{cor}\label{cor:summary}
Keep the setting of \S \ref{subsubsec:describe u}, and keep the notation in Lemmas \ref{inv and s} and \ref{easy lemma}. We have $$ \Delta_{j,B} = (-1)^{q(G)+q(G^*) + k(m^-, \vec{\lambda})} \Delta_{\Wh}.$$
\end{cor}
\begin{proof} Comparing (\ref{characterization of Wh for G}) with Lemmas \ref{inv and s} and  \ref{easy lemma}, we have 
	$$\frac{\underline \Delta_{\Wh}}{\Delta_{\derj, \derB}} = (-1)^{q(G) + q(G^*) + k(m^-, \vec{\lambda})} \cdot  \frac{\Delta_{\Wh}}{\Delta_{j,B}}.$$ By Lemma \ref{answer for H and G^*} we have $\underline \Delta_{\Wh} = \Delta_{\derj, \derB}$. The corollary follows. 
\end{proof}
Recall that $V$ has signature $(p,q)$, with $d=p+q$ not divisible by $4$. 
\begin{lem}\label{lem:computing q(G)+q(G^*)}  
We have 
$$ (-1)^{q(G)+q(G^*)} = \begin{cases}
(-1)^{\ceil{\frac{m-p}{2}}} , & \text{if $d$ is odd}, \\
1 , & \text{if $d$ is even}.
\end{cases}$$
\end{lem}\begin{proof}
For any signature $(a,b)$, we have $q(\SO(a,b)) = ab/2.$ In the odd case, $V$ has signature $(p,q) = (p, 2m+1-p)$, and $\underline V$ has signature $(m+1,m)$ or $(m,m+1)$. Hence 
\begin{multline*}
	q(G^*) - q(G) \equiv \frac{(m+1) m}{2} - \frac{p(2m+1-p)}{2} \\ = \frac{(m-p) (m+1-p)}{2} \equiv \ceil{\frac{m-p}{2}} \mod 2.
\end{multline*}
In the even case, our assumption that $G$ and $G^*$ contain anisotropic maximal tori implies that the signatures of $V$ and $\underline V$ are pairs of even numbers. Hence $q(G)$ and $q(G^*)$ are both even. 
\end{proof}

 \begin{prop} \label{answer for H and G} 
 Keep the running assumption that $V$ has signature $(p,q)$, with $p>q$ and $d= p+q$ not divisible by $4$. Let $\cD$ be an arbitrary element of $\EDNice$ (see Definition \ref{defn:EDNice}), and let $\cD_H \in \ED(V^+)^o \times \ED(V^-)^o$. Define $j_{\cD_H, \cD}$ and $(T_{\cD_H}, B_{\cD_H})$ as in \S \ref{subsubsec:setting for para tori}. We abbreviate $(j_{\cD_H, \cD}, B_{\cD_H})$ as $(j,B)$. Let $\Delta_{j,B}$ be the normalization of the   transfer factors between $H$ and $G$ associated to $(j,B)$, as defined in \cite[\S 7]{kottwitzannarbor}. 
\begin{enumerate}
	\item Assume that $d$ is odd. In this case, either assume that $q$ is even and $q/2 \leq \ceil{m^+/2}$, or assume that $q$ is odd and $(q-1)/2 \leq \floor {m^+/2}$. Then 
	$$\Delta_{j,B}  = \begin{cases}
	(-1)^{\ceil{\frac{m}{2}} + \ceil{\frac{m^+}{2}}  + \ceil{\frac{m-p}{2}}} \Delta _{\Wh} ,  & \text{if $q$ is even}, 
	\\ (-1)^{\floor{\frac{m}{2}} + \floor{\frac{m^+}{2}}  + \ceil{\frac{m-p}{2}}} \Delta _{\Wh}, & \text{if $q$ is odd}. 
	\end{cases}$$ In particular, we have 
	$$\Delta_{j,B}  = \begin{cases}
	(-1)^{ \ceil{\frac{m^+}{2}}} \Delta _{\Wh},  & \text{when $q=0$ and $m^+$ is arbitrary}, \\
	(-1)^{ \floor{\frac{m^+}{2}}} \Delta _{\Wh}, &  \text{when $q=1$ and $m^+$ is arbitrary}, \\ 
		(-1)^{1 + \ceil{\frac{m^+}{2}}} \Delta _{\Wh},   &  \text{when $q=2$ and $m^+>0$}.
	\end{cases}$$
\item Assume that $d$ is even. Thus $q$ is even since $G$ contains anisotropic maximal tori. We have $$ \Delta_{j, B} = \begin{cases}
(-1) ^{\floor {\frac{m^-}{2}}} \Delta_{\Wh}, & \text{if }q/2 \leq \floor {m^+ /2}, \\
(-1) ^{\frac{m+1}{2}} \Delta_{\Wh}, & \text{if } m^+ =1 \text{ and } q=2. 
\end{cases}$$  In particular, we have 	$$\Delta_{j,B}  = \begin{cases}
(-1)^{ \floor{\frac{m^-}{2}}} \Delta _{\Wh},  & \text{when $q=0$ and $m^+$ is arbitrary}, \\
(-1)^{ \floor{\frac{m^-}{2}}} \Delta _{\Wh}, &  \text{when $q=2$ and $m^+ \geq 2$}, \\ 
(-1)^{1 + \floor{\frac{m^-}{2}} } \Delta _{\Wh},   &  \text{when $q=2$ and $m^+= 1$}.
\end{cases}$$
\end{enumerate} 	
 
 \end{prop}
    \begin{proof} First note that under the natural action of $G(\RR)$ on $\ED(V)^o$, the subset $\EDNice$ of $\ED(V)^o$ is a single orbit. Thus $\Delta_{j,B}$ is in fact independent of the choice of $\cD \in \EDNice$. Hence we may assume that $\cD$ is the same as the element introduced in \S \ref{subsubsec:describe u}. In view of Corollary \ref{cor:summary} and Lemma \ref{lem:computing q(G)+q(G^*)}, to prove the proposition it suffices to compute the sign $(-1)^{k (m^-,\vec{\lambda})}$ in each case. We recall that 
    	$$ k (m^-, \vec{\lambda}) : =  \#\set{j \mid 1\leq j \leq m^-, \lambda_j = \sqrt{-1}}. $$
    	
    \textbf{(1)} Let $N$ be the number of negative definite planes among the $m^+$ planes $$\derV_{m^- +1}, \derV_{m^- +2}, \cdots , \derV_m. $$ By Definition \ref{defn:Whitt ED}, $\derV_m$ is negative definite if and only if $V$ has positive determinant, which happens if and only if $q$ is even. Hence we have $N = \ceil{m^+/2}$ when $q$ is even, and $N = \floor{m^+ /2}$ when $q$ is odd. Thus our assumption on $q$ can be rewritten as $\floor{q/2} \leq N$. 
    
    If there exists $1\leq j_1 \leq m^-$ such that $\derV_{j_1}$ is negative definite and $\lambda_{j_1} = 1$, then the integer $j_0$ in condition (2) in Definition \ref{defn of suitable} would be strictly less than $j_1$, from which it easily follows that the number of negative definite planes among $V_1,\cdots, V_m$ is at least $N+1$. Thus $q \geq 2 (N+1)$, a contradiction. Hence such $j_1$ does not exist. Then by condition (2) in Definition \ref{defn of suitable}, we have 
    \begin{align*}
k (m^-, \vec{\lambda})= \# \set{j \mid 1\leq j \leq m^-, V_j \mbox{ is negative definite}}.
    \end{align*}
 When $q$ is even, we have  $$ k(m^-, \vec{\lambda}) = \begin{cases}
    	 \ceil{m^-/2}, & \text{if $m$ is odd} \\
    	 \floor{m^-/2}, &\text{if $m$ is even} 	 \end{cases} \equiv \ceil{m/2} + \ceil { m^+/2} \mod 2.$$
    	 When $q$ is odd, we have $$ k(m^-,\vec{\lambda}) = \begin{cases}
    	 \floor{m^-/2}, & \text{if $m$ is odd} \\
    	 \ceil{m^-/2}, &\text{if $m$ is even} 	 \end{cases}  \equiv \floor {m/2} + \floor { m^+/2} \mod 2.$$
 We conclude the proof by combining the above computation of $k(m^-, \vec{\lambda})$ with Corollary \ref{cor:summary} and Lemma \ref{lem:computing q(G)+q(G^*)}. 
   
   \textbf{(2)} Since $d=2m$ is not divisible by $4$, we know that $m$ is odd, and by Definition \ref{defn:Whitt ED} we know that $\derV_m$ is positive definite. Hence among the $m^+$ planes $$\derV_{m^- +1}, \derV_{m^- +2}, \cdots , \derV_m, $$ the number of negative definite planes is $\floor{m^+/2}$. When $q/2 \leq \floor{m^+/2}$, by the same argument as in part (1) we have 
   $$k(m^-, \vec{\lambda}) = \# \set{j \mid 1\leq j \leq m^-, V_j \mbox{ is negative definite}},$$ and this is equal to $\floor{m^-/2}$. When $m^+=1$ and $q=2$, we easily see that $$k(m^-, \vec{\lambda}) = \frac{m-1}{2} -1. $$ In both cases we conclude the proof by combining the computation of $k(m^-,\vec{\lambda})$ with Corollary \ref{cor:summary} and Lemma \ref{lem:computing q(G)+q(G^*)}. 
     \end{proof}

\section{Transfer factors, when \texorpdfstring{$d$}{d} is divisible by \texorpdfstring{$4$}{4}}\label{transf factors even div by 4}
\subsection{}
     We keep the same setting as in \S \ref{subsubsec:setting for transf factor}, except that now we assume that $d$ is divisible by $4$. We keep the assumption that $G$ and $G^*$ contain anisotropic maximal tori, which forces the signature of $V$ to be a pair of even numbers. In particular, $\delta$ is trivial, and so $\derV$ and $G^*$ are split. We would like to establish analogues of the results in \S \ref{transf factors odd} in the current case. The new feature is that there are now two different equivalence classes of Whittaker data for $G^*$. As in \S \ref{subsubsec:setting for transf factor}, we fix $(H,\lang H, s, \eta)$, with $H$ containing anisotropic maximal tori. 
     
     \emph{In the following we assume that $V$ is of signature $(p,q)$ with $p>q$, and that $d= p+q$ is divisible by $4$.}
     
  \section*{Transfer factors between $H$ and $G^*$}    
  \begin{defn}\label{defn:EDI,EDII}
  	We define two subsets $\EDI$ and $\EDII$\index{$\EDI, \EDII$} of $\ED(\derV)^o$ (see \S \ref{subsubsec:defn of EDo}) as follows. Let $\EDI$ consist of those $(\derV_j)_j \in \ED(\derV)^o$ such that $\derV_j$ is $(-1) ^{j+1}$-definite for each $j$. Let $\EDII$ consist of those $(\derV_j)_j \in \ED(\derV)^o$ such that $\derV_j$ is $(-1) ^j$-definite for each $j$. 
  \end{defn}
  \subsection{}\label{subsubsec:type I Whitt data}
  Let $(T_1, B_1)$ (resp.~$(T_{2}, B_{2})$) be the fundamental pair associated to an element of $\EDI$ (resp.~an element of $\EDII$). Then $(T_1,B_2)$ and $(T_2,B_2)$ both satisfy the condition that every simple root is non-compact, which can be proved in the same way as Lemma \ref{lem:meaning of Whitt}. As in \cite[\S 4.2.1]{taibidim}, the two pairs $(T_1, B_1)$ and $(T_2, B_2)$ correspond to two different equivalence classes of Whittaker data $\wI$ and $\wII$\index{$\wI, \wII$} of $G^*$ respectively, characterized by the condition that in any L-packet of discrete series representations of $G^*(\RR)$, the element corresponding to $(T_1, B_2)$ (resp.~$(T_2, B_2) $) is generic with respect to $\wI$ (resp.~$\wII$). Then $\wI$ and $\wII$ exhaust the equivalence classes of Whittaker data. We call $\wI$ the equivalence class of \emph{type-I Whittaker data}, and call $\wII$ the equivalence class of \emph{type-II Whittaker data}.\index[n]{type-I and type-II Whittaker data} See \textit{loc.~cit.~}for more details. 
  
 \begin{defn}\label{defn:two Whitt norm}
We denote by $\underline \Delta_{\Wh}(\cdot, \cdot)$\index{$\underline \Delta_{\Wh}$} the Whittaker-normalized   transfer factors between $H$ and $G^*$ with respect to $\wI$, called the \emph{type-I Whittaker normalization}.\index[n]{type-I Whittaker normalization of the transfer factors (between $H$ and $G^*$)} Denote by $\tilde {\underline\Delta}_{\Wh}(\cdot,\cdot)$\index{$\tilde {\underline\Delta}_{\Wh}$} the analogous objects with respect to $\wII$.  
  \end{defn}
  
  \begin{lem}\label{answer for H and G^* even div by 4}
  	Let $\derD \in \EDI$, and let $\cD_H \in \ED(V^+)^o \times \ED(V^-)^o$. Let $\derj$, $(T_H,B_H)$, and $(\derT, \derB)$ be the objects associated to $\derD$ and $\cD_H$ as in \S\ref{subsubsec:derj and derB}. We have
  	\begin{align}\label{eq:WhI}
\underline \Delta_{\Wh} & = \Delta_{\derj, \derB} , \\ \label{eq:WhII} \tilde {\underline\Delta}_{\Wh} & =  (-1) ^{m^-} \Delta_{\derj, \derB}.
  	\end{align} In particular, $$\tilde {\underline {\Delta}}_{\Wh}  = (-1)^{m^-} \underline \Delta_{\Wh}. $$ 
  \end{lem}
  \begin{proof}
  	The proof of (\ref{eq:WhI}) is the same as the argument in \S \ref{subsubsec:derj and derB} leading to Lemma \ref{answer for H and G^*}. For (\ref{eq:WhII}), by the same argument we are reduced to checking that
  	\begin{align}\label{eq:a_omega and s}
\lprod{a_{\omega}, s} = (-1) ^{m^-} ,
  	\end{align} where $\omega\in \Omega_{\CC} (G^*, \derT)$ is an element such that $(\derT, \omega \derB)$ is the fundamental pair associated to an element of $\EDII$. (Such $\omega$ is unique up to right multiplication by $\Omega_{\RR} (G^*,\derT)$.) We can take $$\omega  = (12) (34) \cdots (m-1, m) \in \mathfrak S_m \subset \Omega_{\CC} (G^*, \derT),$$ and then the class $a_{\omega} \in \coh ^1(\RR, \derT)$ (defined in \cite[\S 5]{kottwitzannarbor}) is represented by the cocycle sending the complex conjugation to $-1 \in \derT(\RR)$. This implies (\ref{eq:a_omega and s}).\footnotee{  	
  	Alternatively, (\ref{eq:WhII}) also follows from (\ref{eq:WhI}) and the observation that if one changes the sign of $\eta \in \RR^{\times}/ \RR^{\times, 2}$ in Waldspurger's explicit formula in \cite[\S 1.10]{walds10} specialized to $G^*$ and $H$, the resulting sign change is $(-1) ^{m^- }$, cf.~(\ref{eq:Walds 4}) below. This indeed implies our statement, because Waldspurger's formula equals the Langlands--Shelstad normalization $\Delta_0$ (defined for a quasi-split group with a fixed Galois-stable  splitting), and changing the  splitting of $G^*$ has the same effect on $\Delta_0$ as changing the Whittaker datum has on the Whittaker normalization.  Waldspurger excludes the factor $\Delta_{IV}$, which is a positive real number and is hence harmless for this discussion.}
  \end{proof}
  
   \section*{Transfer factors between $H$ and $G$.}   
   \begin{defn}\label{defn:Whitt norm for G div by 4}
   	As in \S \ref{subsubsec:Kaletha's transf factor}, having fixed $\psi_V$ and $u_V$, and having fixed the Whittaker datum $\wI$, we obtain a normalization of the   transfer factors between $H$ and $G$, called the \emph{type-I Whittaker normalization}.\index[n]{type-I Whittaker normalization of the   transfer factors (between $H$ and $G$)} We denote this normalization by $\Delta_{\Wh}$.\index{$\Delta_{\Wh}$}  
   \end{defn}
   \begin{rem}
   	Analogously we also have the type-II Whittaker normalization between $H$ and $G$. By (\ref{eq:WhII}), it is equal to $(-1) ^{m^-} \Delta _{\Wh}$. 
   \end{rem}

\begin{defn}\label{defn:EDNice, div by 4} We let $\EDNice$\index{$\EDNice$} be the subset of $\ED(V)^o$ consisting of those $(V_j)_j$ for which there exists $j_0\in \ZZ$ such that 
	$$\set{j \mid 1\leq j \leq m, V_j \text{ is negative definite}} = \set{j \mid 1\leq j \leq m,  j>j_0, \text{ and }(-1)^j = 1} .$$
\end{defn}
Recall our running assumption that $V$ has signature $(p,q)$, with $p>q$ and $d= p+q$ divisible by $4$. Recall that $p$ and $q$ are even since $G$ contains anisotropic maximal tori.
\begin{prop}\label{answer div by 4}
	  Let $\cD\in\EDNice$ and $\cD_H \in \ED(V^+)^o \times \ED(V^-)^o$. Define $j_{\cD_H, \cD}$ and $(T_{\cD_H}, B_{\cD_H})$ as in \S \ref{subsubsec:setting for para tori}. We abbreviate $(j_{\cD_H, \cD}, B_{\cD_H})$ as $(j,B)$. Let $\Delta_{j,B}$ be the normalization of the   transfer factors between $H$ and $G$ associated to $(j,B)$, as defined in \cite[\S 7]{kottwitzannarbor}. When $q/2 \leq \ceil {m^+ /2}$, we have $$\Delta_{j, B} = (-1) ^{\floor {\frac{m^-}{2}}} \Delta_{\Wh}. $$
	 In particular, we have 	$$\Delta_{j,B}  = \begin{cases}
		(-1)^{ \floor{\frac{m^-}{2}}} \Delta _{\Wh},  & \text{when $q=0$ and $m^+$ is arbitrary}, \\
		(-1)^{ \floor{\frac{m^-}{2}}} \Delta _{\Wh}, &  \text{when $q=2$ and $m^+ \geq 1$}.
		\end{cases}$$
\end{prop}
  \begin{proof}
  	The proof is the same as Proposition \ref{answer for H and G}. Note that the bound $q/2 \leq \floor{m^+/2}$ in Proposition \ref{answer for H and G} (2) is replaced by $q/2 \leq \ceil {m^+ /2}$ here. This is because in the current case, for any $(\derV_j)_j \in \EDI$, $\derV_m$ is always negative definite.
  \end{proof}

  \section*{Comparison with Waldspurger's explicit formula}
\subsection{}   \label{para:preparation for Walds}
 We fix the additive character $\psi: \RR \to \CC^{\times}, x\mapsto e^{2\pi i x}$\index{$\psi_{\mathrm{Leb}}$} in all the discussion below. Given any Borel subgroup $B_0$ of $G^*$ defined over $\RR$, by the general construction in \cite[\S 5.3]{KS99} we have a canonical map (depending only on $\psi$)
 \begin{align}\label{eq:KS construction}
\set{\RR \text{- splittings of $G^*$ relative to $B_0$}} \To \set{\text{generic characters } N_{B_0}(\RR) \to \CC^\times},
 \end{align}
where the left hand side is the set of $\RR$-splittings of $G^*$ of the form $(T_0, B_0, \set{X_{\alpha}})$. In our particular situation, since $G^*$ is split, $\RR$-splittings of $G^*$ are the same as  splittings. 

We denote by $\Pinn(G^*)$\index{$\mathcal{S}plit(G^*)$} the set of $G^*(\RR)$-conjugacy classes of ($\RR$-) splittings of $G^*$, and denote by $\Whitt(G^*)$\index{$\Whitt(G^*)$} the set of equivalence classes (i.e.~$G^*(\RR)$-conjugacy classes) of Whittaker data for $G^*$. The map (\ref{eq:KS construction}) induces a canonical bijection (depending only on $\psi$):\index{$\pintowh^{G^*}$}
$$\pintowh^{G^*} :  \Pinn(G^*) \isom \Whitt(G^*).$$
Here both sides are torsors under the  abelian group $G^{*,\ad}(\RR)/G^*(\RR) \cong \ZZ/2\ZZ$.

The two elements of $\Whitt(G^*)$ are of course $\wI$ and $\wII$; see \S \ref{subsubsec:type I Whitt data}. On the other hand, there is an independent way to label the two elements of $\Pinn(G^*)$. Recall that in \cite[\S 1.6]{walds10}, Waldspurger associates an element $\eta\in \RR^{\times} / \RR^{\times ,2} \cong \set{\pm 1}$ to the quintuple $(G^*, \spl, \underline V, \underline q,  \rho_{\std})$, where $\spl$ is an arbitrary element of $ \Pinn(G^*)$ and $\rho_{\std}$ is the standard representation $G^* \to \GL(\underline V)$. This gives rise to a map\index{$\eta_{\derV}$}
 \begin{align}\label{eq:Walds eta}
	\eta_{\derV}: \Pinn(G^*) & \To \set{\pm 1} \\  \nonumber \spl & \longmapsto \eta(G^*,\spl, \underline V, \underline q, \rho_{\std}).
\end{align}
This map is easily seen to be surjective, and hence bijective. Thus we can use it to label the two elements of $\Pinn(G^*)$. 

  The following result will be used in the proof of Proposition \ref{prop:computing tasho} below, and it may be of independent interest in representation theory. 
 \begin{thm}\label{thm:comparing Waldspurger}
 Let $\splI  = \eta_{\derV} ^{-1}(-1)\in \Pinn(G^*)$. Then $\pintowh^{G^*}(\splI) = \wI$. 
 \end{thm} 
 
\begin{proof}Write $\mathfrak w'$ for $ \pintowh^{G^*}(\splI)$. Consider an elliptic endoscopic datum $$\mathfrak e_{d^+ , \delta^+, d^-, \delta^-} = (H , \lang H, s ,\eta)$$ such that $H$ contains anisotropic maximal tori. As in \S \ref{para:cusp crit} we have  $\delta^{\pm}  = (-1) ^{d^{\pm} /2}$. Let $m^{\pm} :  = d^{\pm} /2$. Let $\underline \Delta_{\Wh}$ and $\tilde {\underline \Delta}_{\Wh}$ be the transfer factors between $H$ and $G^*$ as in Definition \ref{defn:two Whitt norm}. By  Lemma \ref{answer for H and G^* even div by 4} we have 
	$$ \underline \Delta_{\Wh} =  (-1)^{m^-}\tilde {\underline \Delta}_{\Wh}. $$ Hence it suffices to show that $\underline \Delta_{\Wh} $ is equal to the Whittaker normalization $\underline \Delta_{\mathfrak w'}$ defined by the Whittaker datum $\mathfrak w'$, for one single choice of $(d^+, d^-)$ with $m^-$ odd. In the following we show that 
	\begin{align}\label{eq:to show in Thm Walds}\underline \Delta _{\Wh}  = \underline \Delta_{\mathfrak w'}
	\end{align}without assuming that $m^-$ is odd. 
	 
Let $\derD= (\derV_j)_j$ and $\cD_H$ be as in Lemma \ref{answer for H and G^* even div by 4}, and keep the other notations in that lemma. As usual, we use the isomorphism $f_{\derD}: \Uni(1)^m\isom \derT$ associated to $\derD$ to identify $X^*(\derT)$ with $\ZZ^m$. By Lemma \ref{answer for H and G^* even div by 4}, we have 
\begin{align}\label{eq:not div eq 1}
 \underline \Delta_{\Wh} = \Delta _{\derj, \derB}. 
\end{align}

  We now recall the explicit formula for $\Delta_{\derj, \derB}$ given in \cite[\S 7]{kottwitzannarbor}, cf.~also \cite[\S 3.2]{morel2011suite}.\footnote{Note the following typo in \cite[\S 3.2]{morel2011suite}: The term $(1-\alpha (\gamma^{-1}))$ there should be $(1-\alpha(\gamma))$.} Let $\Lambda$ be the set of $\derB$-positive roots for $(G^*_{\CC}, \derT_{\CC})$ which do not come from $H$ via $\derj$. Namely, 
$$\Lambda = \set{ \epsilon_i + \epsilon_k, \epsilon_i - \epsilon_k \mid  1\leq i \leq m^-, ~  m^- +1 \leq k \leq m }. $$
Fix a strongly regular element $\underline \gamma \in \derT(\RR)$, and let $\gamma^H : = \derj^{-1} (\underline \gamma) \in T_H(\RR)$.  
Then  
\begin{align}\label{eq:not div eq 2}
\Delta_{\derj,\derB} (\gamma^H, \underline \gamma ) = (-1) ^{q(G^*) + q(H)} \chi_{\derB}(\underline \gamma) \prod_{\alpha \in \Lambda} (1 - \alpha (\underline\gamma)) =  \chi_{\derB}(\underline \gamma) \prod_{\alpha \in \Lambda} (1 - \alpha (\underline \gamma)),
\end{align}
	where $\chi _{\derB} $ is a quasi-character on $\derT(\RR)$ whose definition is recalled in \cite[Def.~3.2.1]{morel2011suite}. In \cite[Ex.~3.2.4]{morel2011suite} Morel proves, in a special case, the following formula:
	\begin{align}\label{eq:formula for chi_B}
	\chi_{\derB} = (\rho_{B_H} \circ \derj^{-1}) \rho_{\derB}^{-1},
	\end{align}
	where $\rho_{\derB}$ and $\rho_{B_H}$ are defined to be the half sums of the $\derB$-positive roots and the $B_H$-positive roots respectively, and they are actual (as opposed to square roots of) quasi-characters in the special case considered in \textit{loc.~cit}. In our case, $\rho_{\derB}$ and $\rho_{B_H}$ are again actual quasi-characters. We explain why (\ref{eq:formula for chi_B}) still holds in our case. In fact, in the proof of (\ref{eq:formula for chi_B}) in \textit{loc.~cit.}, the only special property being used is that the cocycle $a\in Z^1 (W_{\RR}  , \widehat {\derT})$ used to define $\chi_{\derB}$ could be arranged so that it sends the element $\tau \in W_{\RR}$ (see the beginning of \cite[\S 3.1]{morel2011suite}) to $1 \in \widehat T$. In our case, this condition is not even needed. This is because $\derT\cong\Uni(1)^m$, and so the image of $a$ in $\coh^1(W_{\RR} , \widehat {\derT})$, which determines $\chi_{\derT}$ via the local Langlands correspondence for $\derT$, only depends on $a|_{W_{\CC}} : W_{\CC} \to \widehat {\derT}$. Hence Morel's proof of (\ref{eq:formula for chi_B}) remains valid in our case.
	
	By (\ref{eq:formula for chi_B}), we have 
	\begin{align}\label{eq:not div eq 3}
	\chi_{\derB} = - m^+ \epsilon_1 - m^+ \epsilon_2 - \cdots - m^+ \epsilon_{m^-}.
	\end{align}
	Having identified both $\derT$ and $T_H$ with $\Uni(1)^m$ (via $f_{\derD}$ and $f_{\cD_H}$ respectively), we write $$\gamma^H = \underline \gamma = (y_1,y_2,\cdots, y_m)$$ with each $y_i \in \Uni(1) (\RR) \subset \CC^{\times}$. 
In conclusion, by (\ref{eq:not div eq 1}), (\ref{eq:not div eq 2}), and (\ref{eq:not div eq 3}), we have 
\begin{align}\label{eq:not div eq 4}\underline 
\Delta_{\Wh} (\gamma^H, \underline \gamma)= \prod_{\substack{1\leq i \leq m^- \\ m^-+ 1\leq k \leq m}} y_i^{-1} (1- y_i y_k^{-1}) (1- y_i y_k) =  \prod_{\substack{1\leq i \leq m^- \\ m^-+ 1\leq k \leq m}} 2 (\Re y_i - \Re y_k).
\end{align}

We now compute $\Delta_{\mathfrak w'}$. Let $\Delta_0$\index{$\Delta_0$} be the Langlands--Shelstad normalization associated to the  splitting $\splI$. In \cite{walds10} Waldspurger gives an explicit formula for $\Delta_0$ excluding the factor $\Delta_{IV}$.\index{$\Delta_{IV}$} Let us denote the value of Waldspurger's formula by $\Delta_{\mathrm{Wal}}$\index{$\Delta_{\mathrm{Wal}}$}, so that $\Delta_0 =  \Delta_{\mathrm{Wal}} \Delta_{IV}$. Thus we have (see \cite[\S 5.3]{KS99}, \cite[\S 5.5]{KSconvention})\begin{align}\label{eq:Walds 1}
\Delta_{ \mathfrak w'} = \epsilon_L (U,\psi)  \Delta_0 = \epsilon_L (U, \psi) \Delta_{\mathrm{Wal}} \Delta_{IV},
\end{align}
where $U$ is the virtual $\Gamma_{\infty}$-representation $X^* (T_0) \otimes \CC  - X^*(T_{H,0}) \otimes \CC$, with $T_0$ a maximal split torus in $G^*$ and $T_{H,0}$ a maximal split torus in $H$, and $\epsilon_L(\cdot, \psi)$\index{$\epsilon_L(\cdot, \psi)$} is the local epsilon factor\index[n]{local epsilon factor} (according to the ``Langlands normalization''; see \cite[\S 5.3]{KS99}) defined using the additive character $\psi: \RR \to \CC^{\times}, x \mapsto e^{2\pi i x}$ and the usual Lebesgue measure on $\RR$ (which is self-dual with respect to $\psi$). Since $G^*$ is split, $T_0$ is necessarily split, so $X^*(T_0)$ is a direct sum of trivial representations of $\Gamma_{\infty}$. As for $X^*(T_{H,0})$, it is a direct sum of trivial representations when $ m^-$ is even, and a direct sum of trivial representations and two copies of $X^*(\Uni (1)) $ when $m^-$ is odd. Therefore, by \cite[(3.2.4), (3.4.1)]{tatebackground} we have 
\begin{align}\label{eq:Walds 2}
\epsilon_L(U, \psi) = (-1) ^{m^-}.
\end{align}  
By definition we have
\begin{align}\label{eq:Walds 3}
\Delta_{IV} (\gamma^H, \underline \gamma) = \prod_{
\alpha \in \Lambda} \abs{  \alpha (\underline \gamma)} ^{-1/2} \abs{1-\alpha (\underline \gamma)}  = \prod_{\substack{1\leq i \leq m^- \\ m^-+ 1\leq k \leq m}} 2 \abs{\Re y_i - \Re y_k } .  
\end{align}
Waldspurger's explicit formula reads (see \cite[\S 1.10]{walds10})
\begin{align}\label{eq:Walds 4}
\Delta_{\mathrm{Wal}} (\gamma^H , \underline \gamma) = \prod_{ i =1}^{m^-} \sign \bigg(\eta_{\derV}(\splI) c_i (1+ \Re y_i) \prod_{\substack{1\leq k \leq m  \\ k\neq i}} (\Re y _i - \Re y_k) \bigg ),
\end{align}
where $c_i \in \set{\pm 1}$ is such that $\underline V_i$ is $c_i$-definite. Recall that $(\derV_i)_i \in \EDI$, which implies $c_i = (-1) ^{i+1}$. Note that $1+ \Re y_i > 0$, and we have 
\begin{align*}
\prod _{ \substack{1\leq i,k \leq m^-\\ i \neq k  } }  \sign (\Re y_i - \Re y_k) & = (-1) ^{m^- (m^- - 1)/2} = (-1) ^{\floor{m^- /2}},\\ \prod_{i=1} ^{m^-} \sign c_i  & = \prod _{i=1} ^{m^-} (-1) ^{i+1} = (-1) ^{\floor{m^- /2}}.
\end{align*}
 Therefore (\ref{eq:Walds 4}) can be rewritten as follows (remember that $\eta_{\derV}(\splI) = -1$)
\begin{align}\label{eq:Wal answer}
\Delta_{\mathrm{Wal}} (\gamma^H, \gamma) =  (-1)^{m^-}\prod_{\substack{1\leq i \leq m^- \\ m^-+ 1\leq k \leq m}} \sign (\Re y_i - \Re y_k).
\end{align} 
Combining (\ref{eq:not div eq 4}) (\ref{eq:Walds 1}), (\ref{eq:Walds 2}), (\ref{eq:Walds 3}), and (\ref{eq:Wal answer}), we obtain (\ref{eq:to show in Thm Walds}), as desired. 
\end{proof}

\chapter{Transfer maps defined by the Satake isomorphism}\label{Section p}
 
 In this chapter, we fix an odd prime $p$.

\section{Recall of the Satake isomorphism}\label{subsec:Satake}

We recall the Satake isomorphism, following \cite{cartiersurvey,borelcorvallisarticle, haines-rostami,shintemplier}. Let $F$ be a finite extension of $\QQ_p$. Let $q$ be the residue cardinality of $F$ and let $\varpi_F$ be a uniformizer of $F$. In this section we let $G$ be an arbitrary unramified reductive group over $F$.
\subsection{}\label{subsubsec:canonical Satake} Let $K$ be the hyperspecial subgroup of $G(F)$ determined by a hyperspecial point $v_0$ in the building of $G$. Let $S$ be a maximal split torus in $G$ whose apartment contains $v_0$, and let $T$ be the centralizer of $S$ in $G$. Let $\Omega$ (resp.~$\Omega(F)$) be the absolute (resp.~relative) Weyl group\index[n]{relative Weyl group} of $G$ defined using $T$ (resp.~$S$). In other words, \index{$\Omega(F)$}
\begin{align*}
\Omega & := \Nor_G(T) /T ,\\
\Omega(F) & : = \Nor _G(S)/T.  
\end{align*}
There is a natural $\Gamma_F$-action on $\Omega$, and $\Omega^{\Gamma_F} = \Omega(F)$. See \cite[\S 6.1]{borelcorvallisarticle} for more details. 

We equip $G(F)$ with the Haar measure giving volume $1$ to $K$. Let $\mathcal H (G(F)\sslash K)$\index{$\mathcal H (G(F)\sslash K)$} be the Hecke algebra\index[n]{Hecke algebra (local, $K$-bi-invariant)} of $\CC$-valued compactly supported locally constant $K$-bi-invariant distributions on $G(F)$. Using the fixed Haar measure, we identify $\mathcal H (G(F)\sslash K)$ with the set of $\CC$-valued compactly supported locally constant $K$-bi-invariant functions on $G(F)$. In the same way we define $\mathcal H(T(F) \sslash T(F) \cap K)$, and we simply write it as  $\mathcal H(T(F) /T(F) \cap K)$\index{$\mathcal H(T(F)/ T(F) \cap K)$} since $T(F)$ is abelian. For any choice of a Borel subgroup $B$ of $G$ containing $T$, the \emph{Satake isomorphism}\index[n]{Satake isomorphism} is the following $\CC$-algebra isomorphism:\index{$\sat^{G}_{K, S}$}
\begin{align}
\label{eq: Satake iso} \sat^{G}_{K, S} : 
\mathcal H(G(F)\sslash K) &   \isom \mathcal H(T(F)/ T(F)\cap K)^{\Omega(F)} \\ \nonumber  f& \longmapsto f_T, ~ f_T(t) = \delta_{B(F)} (t ) ^{-1/2} \int _{N_B(F)} f(nt) dn ,~ \forall t\in T(F),
\end{align}
where $N_B$ is the unipotent radical of $B$, and we normalize the Haar measure $dn$ on $N_B(F)$ such that $N_B(F) \cap K$ has volume $1$. It is known that $\mathcal S^G_{K,S}$ depends only on $K$ and $S$, not on $B$ (see for instance \cite[\S 6.1]{shintemplier}).

\subsection{}\label{para:transition maps for Hecke} We explain how to make both sides of the Satake isomorphism  more canonical, that is, independent of the choices of $K$ and $S$. First note that we have canonical isomorphisms 
$$ \mathcal H(T(F) /  T(F) \cap K) \cong \mathcal H(S(F)/S(F)\cap K) \cong \CC [ X_*(S)] ;$$ see \cite[\S 9.5]{borelcorvallisarticle} and  cf.~\cite[\S 7.2]{cartiersurvey}. 
Moreover, if $S'$ is another maximal split torus in $G$, then there is a canonical isomorphism  $$\CC[X_*(S)]^{\Omega(F)} \isom \CC[ X_*(S')] ^{\Omega'(F)}$$ induced by   conjugation by any $g\in G(F)$ such that $gSg^{-1} = S'$. (Here $\Omega'(F)$ denotes the analogue of $\Omega(F)$ with $S $ replaced by $S'$.) Let \index{$\mathscr A_G$} $$\mathscr A_{G} : = \varprojlim_S \CC[X_*(S)]^{\Omega(F)}, $$ where the projective limit is over all maximal split tori $S$ in $G$, and the transition maps are the above-mentioned canonical isomorphisms. 
For our fixed $v_0$ and $K$, the Satake isomorphisms (\ref{eq: Satake iso}) for various choices of $S$ whose apartments contain $v_0$ induce the same isomorphism\index{$\sat^G_K$}
\begin{align}\label{eq: satake canonical}
\sat^{G}_{K}: \mathcal H(G(F)\sslash K) \isom \mathscr A_{G}.
\end{align}
This is because any such $S$ extends to a maximal split torus in the reductive model of $G$ over $\oo_F$ corresponding to $v_0$, and hence any two such choices of $S$ must be conjugate by an element of $K$; cf.~\cite[XXVI, Prop.~6.16]{SGA3III}.

If $K$ and $K_1$ are two different hyperspecial subgroups of $G(F)$, we have a canonical isomorphism $$(\mathcal S^G_{K_1})^{-1} \circ \mathcal S^G_{K}: \mathcal{H} (G(F)\sslash K)\isom \mathcal H (G(F) \sslash K_1),$$ where $\sat^G_K$ and $\sat^G_{K_1}$ are as in (\ref{eq: satake canonical}). In fact, this isomorphism can be described more concretely as follows. Recall that all hyperspecial subgroups of $G(F)$ are conjugate under $G^{\ad}(F)$. For any $g\in G^{\ad}(F)$ such that $\Int(g)(K_1) =K$, we have an isomorphism $\mathcal H(G(F)\sslash K) \isom \mathcal H(G(F)\sslash K_1)$ sending each $f$ to $f \circ \Int(g)$. We claim that this isomorphism is equal to $(\mathcal S^G_{K_1})^{-1} \circ \mathcal S^G_{K}$, and is in particular independent of the choice of $g$. To verify this, choose $S$ with respect to $K$ as in \S \ref{subsubsec:canonical Satake}, and let $S_1 : = \Int(g^{-1})(S)$. Then $S_1$ is a maximal split torus in $G$ whose apartment contains a hyperspecial point defining $K_1$. Let $T$ (resp.~$T_1$) be the centralizer of $S$ (resp.~$S_1$). By the functoriality of the definition (\ref{eq: Satake iso}), we only need to check that the map 
\begin{align*}
\mathcal H(T(F)/T(F)\cap K)^{\Omega(F)} & \To \mathcal H(T_1 (F)/ T_1(F) \cap K_1)^{\Omega(F)} \\  f & \longmapsto f \circ \Int(g)
\end{align*} is compatible with the canonical isomorphisms $$\mathcal H(T(F)/T(F)\cap K)^{\Omega(F)} \cong \mathscr A_G \cong  \mathcal H(T_1(F)/T_1(F)\cap K_1)^{\Omega(F)}. $$ For this, it suffices to check that the isomorphism $\CC[X_*(S)]^{\Omega(F)} \isom \CC[X_*(S_1)]^{\Omega(F)} $ induced by $\Int(g) : S \isom S_1$ is the same as that induced by $\Int(g_0): S \isom S_1$ for any $g_0 \in G(F)$ with $\Int(g_0)(S) = S_1$. We can further reduce to the case where $S = S_1$, and then it suffices to check that $\gamma = \Int(g) |_S \in \Aut(S)$ comes from $\Omega(F)$. This is true because $\gamma$ lies in $\Omega$ and it stabilizes $S$. The claim is proved. We let \index{$\mathcal H^{\ur} (G)$} $$ \mathcal H^{\ur} (G) : = \varprojlim_K \mathcal H(G(F)\sslash K) ,$$ where the projective limit is over all hyperspecial subgroups $K$ and the transition maps are the canonical isomorphisms. 

In conclusion, the Satake isomorphism can be viewed as a canonical $\CC$-algebra isomorphism\index{$\sat^G$} 
\begin{align}\label{canonical Satake}
\sat^{G} : \mathcal H^{\ur} (G) \isom \mathscr A_{G},
\end{align}where both sides are canonically associated to $G$, not depending on any extra choices. 

\subsection{}\label{subsubsec:Borel isom}
As in \cite[\S 6]{borelcorvallisarticle}, the $\CC$-algebra $\mathscr A_{G}$ has an alternative interpretation in terms of the $L$-group of $G$. To explain this, fix a finite unramified extension $F'/F$ splitting $G$, and let $\sigma_F$\index{$\sigma_F$} be the arithmetic Frobenius generator of $\Gal(F'/F)$. Since $F'$ splits $G$, we may form the $L$-group of $G$ using $\Gal(F'/F)$. We use the symbol $\lang G^{\ur}$\index{$\lang G^{\ur}$} to denote this version of the $L$-group, i.e.,  
$$\lang G ^{\ur}: = \widehat G  \rtimes \Gal(F'/F) = \widehat G  \rtimes \langle \sigma _F \rangle .$$ Inside the $\CC$-algebra of $\CC$-valued functions on the set of semi-simple $\widehat G$-conjugacy classes in $\widehat G \rtimes \sigma_F$, we let\index{$\CC[\ch(\lang G^{\ur})]$} $$\CC[\ch(\lang G^{\ur})]$$be the sub-algebra generated by the restrictions of characters of finite-dimensional representations of $\lang G^{\ur}$. Then there is a canonical isomorphism 
\begin{align}\label{eq:Borel isom}
\mathscr A_{G} \cong \CC[\ch(\lang G^{\ur})]
\end{align} characterized as follows. Let $f\in \mathscr A_G$. Fix a maximal split torus $S$ in $G$, and let $T$ be the centralizer of $S$. Then $\mathscr A_G \cong \CC[X_*(S)]^{\Omega(F)} \subset \CC[X_*(T)]$, so we can view $f$ as a function on the $\CC$-torus $\widehat T$. As usual (cf.~\S \ref{subsubsec:L-group data}), $\widehat G$ is equipped with a Borel pair $(\mathcal T, \mathcal B)$ and an isomorphism $\mathrm
{BRD}(G) \isom \mathrm{BRD}(\mathcal T,\mathcal B)^{\vee}$. In particular, if we choose a Borel subgroup $B$ of $G $ containing $T$, then we get an isomorphism of $\CC$-tori $\widehat T \isom \mathcal T$. In this way we obtain from $f$ a function $f_{\mathcal T} : \mathcal T \to \CC$. The construction $f\mapsto f_{\mathcal T}$ is independent of the choices of $S$ and $B$. The image of $f$ under (\ref{eq:Borel isom}) is characterized by the condition that its value at the $\widehat G$-conjugacy class of $t \rtimes \sigma_F$ is equal to $f_{\mathcal T} (t)$, for all $t \in \mathcal T$.

In the sequel, we shall often make the identification (\ref{eq:Borel isom}) without explicitly mentioning it. Thus we can evaluate an element of $\mathscr A_G$ at a semi-simple $\widehat G$-conjugacy class in $\widehat G \rtimes \sigma_F$ to get a complex number. 

In view of (\ref{eq:Borel isom}),  we can also view the Satake isomorphism as a canonical isomorphism 
\begin{align}
\label{canonical Satake with L group}
\mathcal S^{G}: \mathcal H^{\ur} (G) \isom \CC[\ch(\lang G^{\ur})].
\end{align}
 \subsection{}
 Next we recall a result of Kottwitz. Let $\lambda$ be a cocharacter of $G$ defined over $F$. Assume that $\lambda$ is \emph{minuscule}\index[n]{minuscule}, in the sense that the representation $\mathrm{Ad}\circ \lambda$ of $\GG_m$ on $\Lie G_{\overline F}$ has no weights other than $\set{-1,0,1}$. Let $K$ and $S$ be as in \S \ref{subsubsec:canonical Satake}, and assume that $\lambda$ factors through $S$. Denote by $\Omega(F)\cdot \lambda$ the $\Omega(F)$-orbit of $\lambda$ in $X_*(S)$. Let $f_{K,\lambda} \in \mathcal H(G(F)\sslash K)$\index{$f_{K,\lambda}$} be the characteristic function of $K \lambda(\varpi_F) K$ inside $G(F)$. By the Cartan decomposition,  the dependence of $f_{K,\lambda}$ on $\lambda$ is only through the set $\Omega(F)\cdot \lambda$.   
\begin{thm}[{\cite[Lem.~1.1.3, \S 2]{kottwitztwisted}}]\label{thm:Kottwitz on Satake} 
	We have $$\sat^{G} _{K,S} (f_{K, \lambda}) = q^{\lprod{\rho, \lambda_{\mathrm{dom}}}}\sum_{\lambda'\in \Omega(F) \cdot \lambda} [ \lambda'] ~ \in \CC[X_*(S)]^{\Omega(F)}, $$ 
		where $\rho$ is the half sum of a fixed set of positive (absolute) roots in $X^*(Z_G(S))$, and $\lambda_{\mathrm{dom}}$ is any element of $\Omega(F) \cdot \lambda$ which is dominant with respect to the same choice of positive roots. Moreover, the element of $\mathscr A_G$ corresponding to $\sat^{G} _{K,S} (f_{K, \lambda})\in \CC[X_*(S)]^{\Omega(F)}$ depends only on the $G(F)$-conjugacy class of $\lambda$, not on $K$ or $S$. \qed
\end{thm}

\begin{defn}\label{defn:canonical fn a la Kottwitz}
	Let $\lambda $ be a minuscule cocharacter of $G$ defined over $F$. We write \index{$f_{\lambda}$}
	$$f_{\lambda} \in \mathcal H^{\ur} (G)$$ for the element corresponding to $f_{K,\lambda} \in \mathcal H(G(F)\sslash K)$, for some choice of $K$ and $S$  as in \S \ref{subsubsec:canonical Satake} such that $\lambda$ factors through $S$. By Theorem \ref{thm:Kottwitz on Satake}, $f_{\lambda}$ depends only on the $G(F)$-orbit of $\lambda$, not on any extra choices. 
\end{defn}

\subsection{} \label{subsubsec:setting for const term}
We now discuss the compatibility between the Satake isomorphisms and the constant term maps. Let $K$, $S$, and $T$ be as in \S \ref{subsubsec:canonical Satake}. Let $M$ be a Levi component of a parabolic subgroup $P$ of $G$. Assume that $M \supset T$. Let $N_P$ be the unipotent radical of $P$. Then $M(F) \cap K$ is a hyperspecial subgroup of $M(F)$. We define the \emph{constant term map}\index[n]{constant term map} \index{$(\cdot)_{M}$}
 \begin{align}\label{defn: partial satake}
(\cdot)_{M}: \mathcal H (G(F)\sslash K)  & \To \mathcal H(M(F)\sslash M(F)\cap K) \\ \nonumber
f & \longmapsto f_M, ~  f_M(m) = \delta _{P(F)} (m)^{-1/2} \int_{ N_P(F)} f(nm) dn, ~ m\in M(F),
 \end{align}
 where the Haar measure $dn$ on $N_P(F)$ is normalized by the condition that $N_P(F)\cap K$ has volume $1$. 
 
 \begin{rem}
 	The constant term map can be defined more generally for $C^{\infty}_c$ functions; see for instance \cite[\S 7.13]{GKM} or \cite[\S 6.1]{shintemplier}. In \cite{shintemplier} the map (\ref{defn: partial satake}) is called the \emph{partial Satake transform}. When $M=T$, the map (\ref{defn: partial satake}) is the same as $\mathcal S^G_{K,S}$ in (\ref{eq: Satake iso}). 
 \end{rem} 

\begin{lem}\label{lem:partial Satake} In the setting of \S \ref{subsubsec:setting for const term}, let $\Omega_M(F)$ be the relative Weyl group of $M$ defined using $S$. Then $\Omega_M(F)$ is a subgroup of $\Omega(F)$ when both groups are viewed as subgroups of $\GL(X_*(S))$. Moreover, we have a commutative diagram:
$$ \xymatrixcolsep{5pc} \xymatrix{  \mathcal H (G(F) \sslash K) \ar[d]^{(\cdot)_M} \ar[r]^{\mathcal S ^{G} _{K,S}} & \CC[X_*(S)] ^{\Omega(F)} \ar[d] \\ \mathcal H (M(F)\sslash M(F)\cap K)  \ar[r]^----{\mathcal S^{M}_{M(F)\cap K, S}} & \CC[X_*(S)] ^{\Omega_M(F)} }$$ where the right vertical arrow is the inclusion. 
\end{lem} 
 \begin{proof}
 	This is well known. See for instance \cite[\S 12.3]{haines-rostami} or \cite[\S 2, \S 6]{shintemplier}.
 \end{proof}
\begin{prop}\label{canonical constant term} In the setting of \S \ref{subsubsec:setting for const term}, the constant term map (\ref{defn: partial satake}) induces a canonical map 
	\begin{align}\label{eq:canonical const term}
	 (\cdot )_M : \mathcal{H}^{\ur}  (G) \To \mathcal H^{\ur} (M)\end{align}
	 which depends only on $M$, not on  $K,S,P$. 
\end{prop} 
\begin{proof}
	This follows from Lemma \ref{lem:partial Satake}, and the fact that for all maximal split tori $S$ in $M$, the inclusion maps $\CC[X_*(S)]^{\Omega(F)}\to \CC[X_*(S)]^{\Omega_M(F)} $ induce the same map $\mathscr A_G \to \mathscr A_M$. 
\end{proof}
\begin{rem} There is a canonical $\widehat G$-conjugacy class of embeddings $\lang M ^{\ur}\hookrightarrow \lang G ^{\ur}$, and these embeddings induce via pull-back a common canonical map 
	\begin{align}\label{eq:canonical map between ch}
	\CC[\ch(\lang G^{\ur})] \To \CC[\ch(\lang M^{\ur})]. 
	\end{align}
	Under the canonical Satake isomorphism (\ref{canonical Satake with L group}) and its analogue for $M$, the canonical constant term map (\ref{eq:canonical const term}) corresponds to (\ref{eq:canonical map between ch}); cf.~\cite[Rmk.~2.8]{shintemplier}. From this description, one sees that (\ref{eq:canonical const term}) depends on the embedding $M \hookrightarrow G$ only up to $G(F)$-conjugacy.  
\end{rem}

\section{The twisted transfer map}\label{subsec:twisted transfer}

We recall the formalism of the twisted transfer map. We keep the notation and setting of \S \ref{subsec:Satake}. We still let $G$ be an arbitrary unramified reductive group over $F$. Fix a positive integer $a$ and let $F_a$ be the degree $a$ unramified extension of $F$.

\subsection{}\label{subsubsec:Weil restriction} We first recall some facts concerning Weil restriction of scalars.

Let $R : = \Res_{F_a/F} G$. Then $\widehat R $ together with the $\Gal(F^{\ur}/F)$-action on it can be identified with $\prod _{i=1}^a \widehat G$, on which the arithmetic Frobenius generator $\sigma_F$ of  $\Gal(F^{\ur}/F)$ acts by  
$$ \sigma _F(x_1,\cdots, x_a) =  (\sigma_F (x_2), \cdots, \sigma_F (x_{a-1}), \sigma_F (x_1) ). $$

We have a canonical isomorphism $\mathscr A_R \cong \mathscr A_{G_{F_a}}$, where $\mathscr A_{G_{F_a}}$ is formed with respect to $G_{F_a}$ over the base field $F_a$ instead of $F$. This isomorphism is  characterized as follows. Let $S'$ be a maximal $F_a$-split torus in $G_{F_a}$. Then $\Res_{F_a/F} S'$ is an $F$-rational torus in $R$, and its maximal $F$-split subtorus $U$ is a maximal $F$-split torus in $R$.   We have $$(\Res_{F_a/F} S')\otimes_F F_a \cong \prod_{\iota\in \Gal(F_a/F)} S'. $$ Let $\pi : (\Res_{F_a/F} S')\otimes_F F_a  \to S'$ be the projection to the factor corresponding to $\id \in \Gal(F_a/F)$. Composing the inclusion map $U \hookrightarrow \Res_{F_a/F} S'$   (or more precisely, its base change to $F_a$) with $\pi$, we obtain a map $U_{F_a} \to S'$, which is in fact an $F_a$-isomorphism. The resulting isomorphism $X_*(U) \isom X_*(S')$ then induces the canonical isomorphism $\mathscr A_R \cong \mathscr A_{G_{F_a}}$.

Under the isomorphism  $\mathscr A_R \cong \mathscr A_{G_{F_a}}$, suppose an element $f' \in \mathscr A_R $ corresponds to $f \in \mathscr A_{G_{F_a}}$. We would like to have a formula, in terms of $f$, for the evaluation of $f'$ at an element $$ (g_1, \cdots, g_a) \rtimes \sigma _F \in \lang R ^{\ur} = ( \prod_{i=1}^a \widehat G) \rtimes \langle \sigma_F \rangle, $$
where $g_1,\cdots, g_a$ are arbitrary semi-simple elements of $\widehat{G}$. (Here $\langle \sigma_F \rangle$ is understood as either the unramified Weil group $W_F^{\ur}$ or a sufficiently large finite quotient of it. In all cases $\sigma_F$ is a generator.) Working through the definitions, we obtain the desired formula as follows:
\begin{align}\label{eq:f and f'}
	f' ((g_1,\cdots, g_a) \rtimes \sigma_F ) = f ( g_1 \sigma(g_2) \cdots \sigma^{a-1} (g_a) \rtimes \sigma_F^a). 
\end{align}
Here, $g_1 \sigma(g_2) \cdots \sigma^{a-1} (g_a) \rtimes \sigma_F^a$ is an element of $\lang (G_{F_a})^{\ur} = \widehat{G} \rtimes \langle\sigma_F^a \rangle$, the unramified Langlands dual group of $G_{F_a}$ formed with respect to the base field $F_a$ (so the Galois part is generated by $\sigma_F^a$), and hence we can evaluate $f$ at this element.

\subsection{}\label{subsubsec:twisted transfer} Consider an endoscopic datum $(H,\mathcal H, s,\eta)$ for $G$. 
For simplicity, assume that $\mathcal H  = \lang H$ and  $s \in \eta(Z(\widehat H)^{\Gamma_F})$; these assumptions will be met in our applications. We assume that $(H, \lang H, s,\eta)$ is \emph{unramified}\index[n]{unramified (endoscopic datum)}, meaning that the following two conditions are satisfied: 
\begin{enumerate}
	\item The group $H$ is unramified over $F$. In particular, the action of $\Gamma_F$ on $\widehat H$ factors through $\Gal(F^{\ur}/F)$. 
	\item The map $\eta: \lang H \to \lang G$ is induced by an $L$-embedding $\lang H^{\ur} \to \lang G^{\ur}$. Here $\lang H^{\ur}$ and $\lang G^{\ur}$ denote the $L$-groups formed with $\Gamma'$, where $\Gamma'$ is either the unramified Weil group $W_F^{\ur}$ or a sufficiently large finite quotient of it. In all cases we denote by $\sigma_F$ the arithmetic Frobenius generator of $\Gamma'$. 
\end{enumerate}

Let $R = \Res_{F_a/F}G$. Define a homomorphism \index{$\tilde \eta$}
$$  \tilde \eta: \lang H^{\ur} = \widehat{ H} \rtimes \langle \sigma_F \rangle \To \lang R^{\ur}  = \left( \prod_{i=1}^a \widehat G \right) \rtimes \langle \sigma_F \rangle,  $$ by 
$$ \widehat H \ni x \longmapsto (\eta (x), \cdots, \eta(x) ) \rtimes 1, $$ and 
\begin{align}\label{convenient choice of t}
1\rtimes \sigma_F \longmapsto (s^{-1} \eta(\sigma_F) \sigma_F^{-1},\eta(\sigma_F) \sigma_F^{-1},\cdots, \eta(\sigma_F) \sigma_F^{-1} ) \rtimes \sigma_F.
\end{align}
Let \index{$\tilde \eta^*$} $$
\tilde \eta^*: \CC[\ch(\lang R^{\ur})] \To \CC[\ch(\lang H ^{\ur})] $$ be the map induced by the pull-back along $\tilde \eta$. 
As we have explained in \S \ref{subsubsec:Borel isom}, the source and target of $\tilde \eta^*$ are canonically identified with $\mathscr A_R$ and $\mathscr A_H$ respectively. Also, as in \S \ref{subsubsec:Weil restriction} we have $\mathscr A_R \cong \mathscr A_{G_{F_a}}$. We can thus view $\tilde \eta^*$ as a map 
$$ \tilde \eta^*: \mathscr A_{G_{F_a}} \To \mathscr A_H. $$ 
We call this map the \emph{twisted transfer map}. \index[n]{twisted transfer map} If we identify the two sides with $\mathcal H^{\ur}(G_{F_a})$ and $\mathcal H^{\ur}(H)$ respectively using the canonical Satake isomorphisms, we obtain a map $\mathcal H^{\ur}(G_{F_a}) \to   \mathcal H^{\ur}(H)$ which is also called the twisted transfer map. 
\begin{lem}\label{lem:essential tb}
	Let $f \in \mathscr A_{G_{F_a}}$, and let $x$ be a semi-simple element of $\widehat H$. Write $\eta(x\rtimes \sigma_F)^a = z \rtimes \sigma_F^a$, with $z \in \widehat G$. Then 
the evaluation of $\tilde \eta^*(f) \in \mathscr A_H$ at $x\rtimes \sigma_F \in \lang H^{\ur}$ is equal to 
$$ f(s^{-1} z  \rtimes \sigma_F^a ).
 $$ Here we have $s^{-1}z \in \widehat G$, and $s^{-1} z  \rtimes \sigma_F^a $ is an element of $\lang (G_{F_a})^{\ur} = \widehat{G} \rtimes \langle \sigma_F^a\rangle$, so we can evaluate $f$ at $s^{-1} z  \rtimes \sigma_F^a $. 
\end{lem}
\begin{proof}
Write $y$ for $\eta(x\rtimes \sigma_F) \sigma_F^{-1} \in \widehat G$. Let $f' \in \mathscr A_R$ be the element corresponding to $f$ under $\mathscr A_R \cong \mathscr A_{G_{F_a}}$. We compute 
\begin{align*}
 \tilde \eta^*(f) (x\rtimes \sigma_F)&  =  f' (\tilde \eta(x \rtimes \sigma_F)) = f' (( s^{-1}y , y ,\cdots, y ) \rtimes \sigma_F ) \\ & = f(s^{-1}y \sigma(y) \cdots \sigma^{a-1}(y)  \rtimes \sigma_F^a ) \\ & = f(s^{-1} z  \rtimes \sigma_F^a ). 
\end{align*}
Here the third equality follows from (\ref{eq:f and f'}). 
\end{proof}

\begin{rem} In the above definition of $\tilde \eta$ we have taken advantage of the simplifying assumptions $\mathcal H  = \lang H$ and $s \in \eta(Z(\widehat H)^{\Gamma_F})$. For the definition in more general situations, see \cite[\S 7]{kottwitzannarbor} or \cite[\S 7.4]{KSZ}. Under our simplifying assumptions, the formula (\ref{convenient choice of t}) can also be replaced by $$1\rtimes \sigma_F \longmapsto (t_1 \eta(\sigma_F) \sigma_F^{-1} , t_2 \eta(\sigma_F) \sigma_F^{-1}, \cdots, t_a \eta(\sigma_F) \sigma_F^{-1} ) \rtimes \sigma_F$$ for any choices of $t_1,\cdots, t_a \in \eta(Z(\widehat H)^{\Gamma_F})$ such that $t_1 t_2\cdots t_a= s^{-1}$. In fact, such a replacement does not change the conclusion of Lemma \ref{lem:essential tb}. We have chosen $t_1= s^{-1}$ and $t_2= \cdots = t_a =1$ for definiteness. 
\end{rem}
\subsection{} \label{para:BC}
As a special case of the twisted transfer map, consider the trivial endoscopic datum $(G, \lang G, 1, \id)$ for $G$, which makes sense since $G$ is quasi-split. Then we obtain the so-called \emph{base change map} 
$$ \mathscr A_{G_{F_a}} \To \mathscr A_G,$$ also viewed as a map $$\mathcal H^{\ur}(G_{F_a}) \to   \mathcal H^{\ur}(G).$$

\section{Explicit description of the twisted transfer map}
 We now make the construction in \S \ref{subsec:twisted transfer} explicit for unramified special orthogonal groups. 
 \subsection{}\label{subsubsec:A_G explicit} We first make explicit the group $\mathscr A_G$ and the evaluation of its elements at semi-simple $\widehat G$-conjugacy classes in $\widehat G \rtimes \sigma_F$.

 We now keep the setting and notation of \S \ref{sec:endoscopic}, specialized to the case where $F$ is a finite extension of $\QQ_p$. In particular, $G$ denotes $\SO(V)$ where $V$ is a quadratic space over $F$ of dimension $d$ and discriminant $\delta$. As always we write $m$ for $\floor{d/2}$. Assume that $G$ is unramified over $F$. By Proposition \ref{prop: even TFAE}, if $d$ is odd, or if $d$ is even and $\delta$ is trivial, our assumption implies that $G$ is split. If $d$ is even and $\delta$ is non-trivial, our assumption implies that $\delta$ has a representative in $\oo_F^{\times}/\oo_F^{\times, 2}$, and that $G$ is split over $F(\alpha)$; here recall that $\alpha \in \overline F$ is a fixed square root of a fixed lift of $\delta$ in $F^\times$.  
 
 To simplify notation, for each positive integer $n$ we define \index{$\AB[X_1,\cdots, X_n]$} \index{$\AD[X_1,\cdots, X_n]$}
 \begin{align*}
 	\AB[X_1,\cdots, X_n] & : = \CC[X_1^{\pm 1}, \cdots, X_n^{\pm 1}]^{\set{\pm 1} ^n \rtimes \mathfrak S_n}, \\ 
 	\AD [X_1,\cdots, X_n] & : = \CC[X_1^{\pm 1}, \cdots, X_n^{\pm 1}]^{(\set{\pm 1} ^n)' \rtimes \mathfrak S_n}. 
 \end{align*}
 Here the group $\set{\pm 1} ^n \rtimes \mathfrak S_n$ acts on $\CC[X_1^{\pm 1}, \cdots , X_n^{\pm 1}]$ as follows. The non-trivial element of the $i$-th copy of $\set{\pm 1}$ acts by swapping $X_i$ and $X_i^{-1}$, and $\mathfrak S_n$ acts by permuting the $n$ variables $X_1,\cdots, X_n$ (and simultaneously permuting $X_1^{-1},\cdots, X_n^{-1}$). As usual,  $(\set{\pm 1}^n)'$ is the kernel of the multiplication map $\set{\pm 1}^n \to \set{\pm 1}$. When $n=1$, by definition we have 
 $\mathscr A_{\mathsf D}[X_1] = \CC[X_1^{\pm 1}]$.

First assume that $d$ is odd. Then $G$ is split. Fix a Borel pair $(T,B)$ in $G$. We then get an isomorphism $\mathrm{BRD}(T,B) \isom \mathrm{BRD}(\mathcal T,\mathcal B)^{\vee}$ from the $L$-group datum fixed in \S \ref{fixing L group}. The right hand side is canonically identified with $\mathrm{BRD}(\mathsf{B}_m)$. Thus we get an isomorphism $X_*(T) \isom \ZZ^m$, and an isomorphism $$\mathscr A_G \cong \CC[X_*(T)]^{\Omega} \isom \CC[\ZZ^m]^{\set{\pm 1} ^m \rtimes \mathfrak S_m} \cong  \AB[X_1,\cdots, X_m] ,$$  which is independent of the choice of $(T,B)$. If an element of $\mathscr A_G$ corresponds to 
 $F(X_1,\cdots, X_m)  \in \AB[X_1,\cdots, X_m],$ then the evaluation of this element at $\DS(t_1,\cdots, t_m) \rtimes \sigma_F  \in \mathcal T \rtimes \sigma_F$ (see \S \ref{subsubsec:Join notation} for the notation) is given by 
 $ F(t_1,\cdots, t_m) \in \CC$.
 
 If $d$ is even and $\delta$ is trivial, then $G$ is still split, and similarly as in the odd case we have a canonical identification 
 $$ \mathscr A_G \cong \AD[X_1,\cdots, X_m].$$ (This is true for $m=1$ as well.) As in the odd case, the evaluation of an element of $\mathscr A_G$ corresponding to $F(X_1,\cdots, X_m) \in \AD[X_1,\cdots, X_m]$ at $\DS(t_1,\cdots, t_m) \rtimes \sigma_F \in \mathcal T \rtimes 
 \sigma_F$ is given by $F(t_1,\cdots, t_m)$. 
 
 Now consider the case where $d$ is even and $\delta$ is non-trivial. Let $S$ be a maximal split torus in $G$, let $T$ be the centralizer of $S$, and let $B$ be a Borel subgroup of $G$ containing $T$. We then get an isomorphism $\mathrm{BRD}(T,B) \isom \mathrm{BRD}(\mathcal T, \mathcal B)^{\vee}$ from the $L$-group datum fixed in \S \ref{fixing L group}. The right hand side is canonically identified with $\mathrm{BRD}(\mathsf{D}_m)$. We thus get an isomorphism $X_*(T) \isom \ZZ^m$. Under this isomorphism, $X_*(S) = X_*(T)^{\Gamma_F}$ corresponds to the subgroup $\ZZ^{m-1} \times \set{0} = \set{(x_1,\cdots,x_{m-1},0)\mid x_i \in \ZZ}$ of $\ZZ^m$, and the $\Omega(F)$-action on $X_*(T)$ corresponds to the natural action of $(\set{\pm 1}^{m})' \rtimes \mathfrak S_{m-1}$ on $\ZZ^m$, that is, the non-trivial element of the $i$-th copy of $\set{\pm 1}$ acts by multiplication by $-1$ on the $i$-th coordinate, and $\mathfrak S_{m-1}$ acts by permuting the first $m-1$ coordinates. We have natural identifications $$\CC[ \ZZ^{m-1} \times \set{0}]^{(\set{\pm 1}^{m})' \rtimes \mathfrak S_{m-1}} \cong \CC[\ZZ^{m-1}]^{\set{\pm 1}^{m-1} \rtimes \mathfrak S_{m-1}}  \cong  \AB[X_1,\cdots, X_{m-1}]. $$ Hence we obtain an identification 
 $$ \mathscr A_G \cong  \AB[X_1,\cdots, X_{m-1}]. $$ As in the previous cases, this identification is independent of the choices of $S$ and $B$. If an element of $\mathscr A_G$ corresponds to 
 $ F(X_1,\cdots, X_{m-1})  \in \AB[X_1,\cdots, X_{m-1}],$ then the evaluation of this element at $\DS(t_1,\cdots, t_m) \rtimes \sigma_F \in \mathcal T \rtimes \sigma_F $ is given by 
 $F(t_1,\cdots, t_{m-1})$.

  \subsection{}\label{subsubsec:explicit tb}  Let $G$ be as in \S\ref{subsubsec:A_G explicit}. In \S \ref{subsec:elliptic end data}, we constructed representatives $\fke_{d^+,\delta^+,d^-,\delta^-}$ of the isomorphism classes of elliptic endoscopic data for $G$, where $(d^+,\delta^+,d^-,\delta^-)$ belongs to a set $\mathscr P_V$ as in Definition \ref{defn:mathscr P_V}. In order to ensure ellipticity, in the definition of $\mathscr P_V$ we have the condition that if $d$ is even and at least $4$ then neither of $(d^+,\delta^+)$ and $(d^-,\delta^-)$ is equal to $(2,1)$. We now take a quadruple $(d^+,\delta^+,d^-,\delta^-)$ satisfying all the conditions in the definition of $\mathscr P_V$ except the condition just mentioned. The construction in \S \ref{subsec:elliptic end data} still applies to $(d^+,\delta^+,d^-,\delta^-)$ and yields an endoscopic datum $$\mathfrak e_{d^+ , \delta^+, d^-, \delta^-} = (H, \lang H, s, \eta)$$ for $G$, which may no longer be elliptic. In fact, the non-elliptic endoscopic data for $G$ arising in this way account for all the non-elliptic endoscopic data (up to isomorphism) that can possibly appear as the localization of  global elliptic endoscopic data, in the case where $G$ is the localization of a  special orthogonal group over a number field.  
  
   Throughout we assume that 
  $ d^+ \neq 0.$ We now assume that $F = \QQ_p$, and write $\sigma$ for $\sigma_F$. We keep assuming that $G$ is unramified. As in \S \ref{subsubsec:twisted transfer}, we assume that the endoscopic datum $\mathfrak e_{d^+ , \delta^+, d^-, \delta^-}$ is unramified. In the odd case the last assumption is automatic, and in the even case it implies that $\delta^+$ and $\delta^-$ both have (unique) representatives in $ \ZZ_p^{\times}/\ZZ_p^{\times,2},$ in view of Proposition \ref{prop: even TFAE}. It is easy to check that the converse is also true. Note that $\ZZ_p^{\times}/\ZZ_p^{\times,2} \cong \ZZ/2\ZZ$ as $p$ is odd. Hence each of   $\delta, \delta^+, \delta^-$ can take only two values: the trivial or the non-trivial element of $\ZZ_p^{\times}/\ZZ_p^{\times,2}$. 
  
 Fix a positive integer $a$. We still write $F_a$ for the degree $a$ unramified extension of $F= \QQ_p$. We now make explicit the twisted transfer map $\tilde \eta^* : \mathscr A_{G_{F_a}} \to \mathscr A_H$ defined in \S \ref{subsubsec:twisted transfer}. As always we write $m$ for $\floor{d/2}$, and write $m^{\pm}$ for $\floor{d^\pm/2}$. 
 \subsubsection{The odd case}\label{para:odd tb} In this case, $\mathscr A_{G_{F_a}}$ is identified with $\AB[X_1,\cdots, X_m]$, and $\mathscr A_{H} = \mathscr A_{H^+} \otimes_{\CC} \mathscr A_{H^-}$ is identified with $$ \AB[Z_1,\cdots, Z_{m^+}] \otimes_{\CC} \AB[Y_1,\cdots, Y_{m^-}],  $$ which we identify with a $\CC$-subalgebra of $\CC[Z_1^{\pm 1},\cdots, Z_{m^+} ^{\pm 1}, Y_1^{\pm 1},\cdots, Y_{m^-}^{\pm 1}].$ Consider an element $$t _{\widehat H} = (\DS (t_1,\cdots, t_{m^+}) , \DS(u_1,\cdots, u_{m^-}) ) $$ of the maximal torus $\mathcal T_{\widehat H} = \mathcal T_{V^+} \times \mathcal T_{V^-}$ in $\widehat H$. We have $$\eta(t_{\widehat H} \rtimes \sigma) = \DS (u_1,\cdots, u_{m^-}, t_1,\cdots, t_{m^+}) \rtimes \sigma \in \mathcal T \rtimes \sigma. $$ Since $\sigma$ acts trivially on $\mathcal T$, we have
 \begin{align*}\eta(t_{\widehat H} \rtimes \sigma)^a & = \DS (u_1^a,\cdots, u_{m^-}^a, t_1^a,\cdots, t_{m^+}^a ) \rtimes \sigma^a, \\
s^{-1} 	 \eta(t_{\widehat H} \rtimes \sigma)^a & = \DS (-u_1^a,\cdots, -u_{m^-}^a, t_1^a,\cdots, t_{m^+}^a ) \rtimes \sigma^a. 
 \end{align*} 
 Suppose $f \in \mathscr A_{G_{F_a}}$ corresponds to 
 $F(X_1,\cdots, X_m) \in \AB[X_1,\cdots, X_m]. $
 By Lemma \ref{lem:essential tb}, the evaluation of $\tilde \eta^*(f)$ at $t_{\widehat{H}} \rtimes \sigma$ is equal to
 $$ F(-u_1^a, \cdots, -u_{m^-}^a, t_1^a, \cdots, t_{m^+} ^ a).  $$ Thus the map $\tilde \eta^* $ is explicitly given by 
 \begin{align*}
 	\AB[X_1,\cdots, X_m] & \To \AB[Z_1,\cdots, Z_{m^+}] \otimes_{\CC} \AB[Y_1,\cdots, Y_{m^-}] \\ 
 F(X_1,\cdots, X_m) & \longmapsto F(-Y_1^a,\cdots, -Y_{m^-}^a, Z_1^a,\cdots, Z_{m^+}^a).
 \end{align*} 
  
  \subsubsection{The even case, with trivial $\delta^+$ and trivial $\delta^-$} 
  In this case, $\delta$ is also trivial since $\delta = \delta^+ \delta^-$. We have $$\mathscr A_{G_{F_a}} \cong \AD[X_1,\cdots, X_m], $$ and $$\mathscr A_{H} = \mathscr A_{H^+} \otimes_{\CC} \mathscr A_{H^-} \cong  \AD[Z_1,\cdots, Z_{m^+}] \otimes_{\CC} \AD[Y_1,\cdots, Y_{m^-}].  $$ By similar computation as in \S \ref{para:odd tb}, we find that $\tilde \eta^*$ is explicitly given by 
  \begin{align*}
  	\AD[X_1,\cdots, X_m] & \To \AD[Z_1,\cdots, Z_{m^+}] \otimes_{\CC} \AD[Y_1,\cdots, Y_{m^-}] \\ 
  	F(X_1,\cdots, X_m) & \longmapsto F(-Y_1^a,\cdots, -Y_{m^-}^a, Z_1^a,\cdots, Z_{m^+}^a).
  \end{align*}

\subsubsection{The even case, with non-trivial $\delta^+$ and trivial $\delta^-$}
In this case, $\delta$ is non-trivial in $\ZZ_p^{\times}/\ZZ_p^{\times,2}$. It is a square in $F_a^\times$ if and only if $a$ is even. Thus we have
$$\mathscr A_{G_{F_a}} \cong
\begin{cases} \AD[X_1,\cdots, X_m], & \text{if }a \text{ is even},\\ 
	\AB[X_1,\cdots, X_{m-1}], & \text{if } a \text{ is odd},
	\end{cases} 
 $$ and $$\mathscr A_{H} = \mathscr A_{H^+} \otimes_{\CC} \mathscr A_{H^-} \cong  \AB[Z_1,\cdots, Z_{m^+-1}] \otimes_{\CC} \AD[Y_1,\cdots, Y_{m^-}].  $$  
 Consider an element $$t _{\widehat H} = (\DS (t_1,\cdots, t_{m^+}) , \DS(u_1,\cdots, u_{m^-}) ) \in \mathcal T_{\widehat H} = \mathcal T_{V^+} \times \mathcal T_{V^-}.  $$ Since $\delta^-$ is trivial, $\sigma$ belongs to the first case in (\ref{eq:eta on Gamma'}). Hence $$ \eta(t_{\widehat{H}} \rtimes \sigma ) =   \DS (u_1,\cdots, u_{m^-}, t_1,\cdots, t_{m^+}) \rtimes \sigma \in \mathcal T \rtimes \sigma. $$ Now the action of $\sigma$ on $\mathcal T$ sends 
 $\DS (x_1,\cdots, x_m)$ to $\DS (x_1,\cdots, x_{m-1}, x_m^{-1})$; cf.~\S \ref{subsubsec:L group}. We introduce the notation \index{$\nu_a$}
 \begin{align}\label{eq:nu_a}
 	\nu_a : = \frac{(-1)^{a+1} + 1}{2}.
 \end{align} Hence 
 \begin{align*}
  \eta(t_{\widehat{H}} \rtimes \sigma )^a & =  
  \DS (u_1^a,\cdots, u_{m^-}^a, t_1^a,\cdots, t_{m^+-1}^a, t_{m^+}^{\nu_a} ) \rtimes \sigma^a  ,   \\
 s^{-1}  \eta(t_{\widehat{H}} \rtimes \sigma )^a & =  
 \DS (-u_1^a,\cdots, -u_{m^-}^a, t_1^a,\cdots, t_{m^+-1}^a, t_{m^+}^{\nu_a} ) \rtimes \sigma^a  ,
\end{align*}
 \paragraph{Suppose $a$ is even} 
  Suppose $f \in \mathscr A_{G_{F_a}}$ corresponds to 
 $ F( X_1, \cdots, X_m)  \in \AD[X_1,\cdots, X_m]. $
 By Lemma \ref{lem:essential tb}, the evaluation of $\tilde \eta^*(f)$ at $t_{\widehat{H}} \rtimes \sigma$ is equal to
 $$F(-u_1^a,\cdots, -u_{m^-}^a, t_1^a, \cdots, t_{m^+ -1}^a, 1).$$ 
 Thus the map $\tilde \eta^* $ is explicitly given by 
 \begin{align*}
 	\AD[X_1,\cdots, X_m] & \To \AB[Z_1,\cdots, Z_{m^+-1}] \otimes_{\CC} \AD[Y_1,\cdots, Y_{m^-}] \\ 
 F(X_1,\cdots, X_m) & \longmapsto F (-Y_1^a,  \cdots ,-Y_{m^-}^a, Z_1^a,  \cdots , Z_{m^+-1} ^ a, 1). 
 \end{align*}
 \paragraph{Suppose $a$ is odd} 
 Suppose $f \in \mathscr A_{G_{F_a}}$ corresponds to 
 $ F(X_1,\cdots, X_{m-1})\in \AB[X_1,\cdots, X_{m-1}]. $
 By Lemma \ref{lem:essential tb}, the evaluation of $\tilde \eta^*(f)$ at $t_{\widehat{H}} \rtimes \sigma$ is equal to
 $$ F(-u_1^a,\cdots, -u_{m^-} ^a, t_1^a,\cdots, t_{m^+ -1}^a).  $$ Thus the map $\tilde \eta^* $ is explicitly given by 
 \begin{align*}
 	\AB[X_1,\cdots, X_{m-1}] & \To \AB[Z_1,\cdots, Z_{m^+-1}] \otimes_{\CC} \AD[Y_1,\cdots, Y_{m^-}] \\ 
 	F(X_1,\cdots, X_{m-1})& \longmapsto F(-Y_1^a,\cdots, -Y_{m^-}^a, Z_1^a,\cdots, Z_{m^+-1}^a).
 \end{align*}

 \subsubsection{The even case, with trivial $\delta^+$ and non-trivial $\delta^-$}\label{para:tb complicated 1}
 In this case, $\delta$ is non-trivial. We have
 $$\mathscr A_{G_{F_a}} \cong
 \begin{cases} \AD[X_1,\cdots, X_m], & \text{if }a \text{ is even},\\ 
 	\AB[X_1,\cdots, X_{m-1}], & \text{if } a \text{ is odd},
 \end{cases} 
 $$ and $$\mathscr A_{H} = \mathscr A_{H^+} \otimes_{\CC} \mathscr A_{H^-} \cong  \AD[Z_1,\cdots, Z_{m^+}] \otimes_{\CC} \AB[Y_1,\cdots, Y_{m^- -1}].  $$  
 Consider an element $$t _{\widehat H} = (\DS (t_1,\cdots, t_{m^+}) , \DS(u_1,\cdots, u_{m^-}) ) \in \mathcal T_{\widehat H} = \mathcal T_{V^+} \times \mathcal T_{V^-}.  $$ Since $\delta^-$ is non-trivial, we are in the second case in (\ref{eq:eta on Gamma'}). Hence $$ \eta(t_{\widehat{H}} \rtimes \sigma ) = \DS(u_1,\cdots, u_{m^-}, t_1,\cdots, t_{m^+})\cdot S  \rtimes \sigma \in \lang G^{\ur},  $$ where $S$ is the permutation matrix switching $\hat e_{m^-}$ and $\hat e_{d-m^- +1}$, and switching $\hat e_{m}$ and $\hat e_{m+1}$.  The conjugation action of $S\rtimes \sigma$ on $\mathcal T$ is given by  $$\DS(x_1,\cdots, x_m)\longmapsto\DS(x_1,\cdots, x_{m^--1}, x_{m^-}^{-1}, x_{m^-+1}, \cdots, x_m). $$ Moreover, $(S \rtimes \sigma)^a = S^a \rtimes \sigma^a$, and $S$ is of order $2$. Therefore, with the notation (\ref{eq:nu_a}), we have 
 \begin{align*}
s^{-1} \eta(t_{\widehat H} \rtimes \sigma)^a  = \DS( -u_1^a,\cdots, -u_{m^- -1} ^a, -u_{m^-}^{\nu_a}, t_1^a,\cdots, t_{m^+}^a) \cdot S^{\nu_a} \rtimes \sigma^a . 
 \end{align*} If $a$ is even, the above element lies in $\mathcal T \rtimes \sigma^a$. If $a$ is odd, the above element is conjugate by some  $g\in \widehat{G}$ to the element 
 \begin{align*}
 	\DS(-u_1^a, \cdots, -u_{m^- -1}^a, t_{m^+}^a, t_1^a,\cdots, t_{m^+ -1}^a, -u_{m^-} ) \rtimes \sigma^a \in  \mathcal T \rtimes \sigma^a. 
 \end{align*}
 For instance, one can take $g$ to be the permutation matrix in $\widehat{G}$ switching $\hat e_{m^-}$ and $\hat e_{m}$ and switching $\hat e_{d-m^- +1}$ and $(-1) ^{m^- + m} \hat e_{m+1}$. Indeed, we have $g^{-1} = g$, $(S \rtimes \sigma ^a) g = g \rtimes \sigma^a$, and \begin{multline*}
 g \cdot \DS( -u_1^a,\cdots, -u_{m^- -1} ^a, -u_{m^-}, t_1^a,\cdots, t_{m^+}^a) \cdot g \\ =  \DS(-u_1^a, \cdots, -u_{m^- -1}^a, t_{m^+}^a, t_1^a,\cdots, t_{m^+ -1}^a, -u_{m^-} ) .
 \end{multline*} 
 \paragraph{Suppose $a$ is even} 
 Suppose $f \in \mathscr A_{G_{F_a}}$ corresponds to 
 $ F(X_1,\cdots, X_m) \in \AD[X_1,\cdots, X_m]. $  By Lemma \ref{lem:essential tb}, the evaluation of $\tilde \eta^*(f)$ at $t_{\widehat{H}} \rtimes \sigma$ is equal to
 $$F(-u_1^a,\cdots, -u_{m^- -1}^a,-1,   t_1^a,\cdots, t_{m^+}^a) . 
 $$
 Thus the map $\tilde \eta^* $ is explicitly given by 
 \begin{align*}
 	\AD[X_1,\cdots, X_m] & \To \AD[Z_1,\cdots, Z_{m^+}] \otimes_{\CC} \AB[Y_1,\cdots, Y_{m^--1}] \\ 
 F(X_1,\cdots, X_m) & \longmapsto   F(-Y_1^a,\cdots, -Y_{m^- -1}^a, -1, Z_1^a,\cdots  , Z_{m^+}^a) .
 \end{align*}
 \paragraph{Suppose $a$ is odd} 
 Suppose $f \in \mathscr A_{G_{F_a}}$ corresponds to 
 $ F(X_1,\cdots, X_{m-1})\in \AB[X_1,\cdots, X_{m-1}]. $
 By Lemma \ref{lem:essential tb}, the evaluation of $\tilde \eta^*(f)$ at $t_{\widehat{H}} \rtimes \sigma$ is equal to
 $$ F (  -u_1^a, \cdots -u_{m^- -1}^a, t_{m^+}^a, t_1^a, \cdots t_{m^+ - 1} ^ a).  $$ Thus the map $\tilde \eta^* $ is explicitly given by 
 \begin{align*}
 	\AB[X_1,\cdots, X_{m-1}] & \To \AD[Z_1,\cdots, Z_{m^+}] \otimes_{\CC} \AB[Y_1,\cdots, Y_{m^--1}] \\  F(X_1,\cdots, X_{m-1}) & \longmapsto   F(-Y_1^a,\cdots, -Y_{m^- -1}^a, Z_{m^+}^a, Z_1^a,\cdots, Z_{m^+ -1} ^a ) .
 \end{align*}

\subsubsection{The even case, with non-trivial $\delta^+$ and non-trivial $\delta^-$}\label{para:two non-trivial}
In this case, $\delta$ is trivial. We have
$$\mathscr A_{G_{F_a}} \cong \AD[X_1,\cdots, X_m]
$$ and $$\mathscr A_{H} = \mathscr A_{H^+} \otimes_{\CC} \mathscr A_{H^-} \cong  \AB[Z_1,\cdots, Z_{m^+-1}] \otimes_{\CC} \AB[Y_1,\cdots, Y_{m^- -1}].  $$  
Consider an element $$t _{\widehat H} = (\DS(t_1,\cdots, t_{m^+}) , \DS(u_1,\cdots, u_{m^-}) ) \in \mathcal T_{\widehat H} = \mathcal T_{V^+} \times \mathcal T_{V^-}.  $$ Since $\delta^-$ is non-trivial, we are in the second case in (\ref{eq:eta on Gamma'}).  Hence $$ \eta(t_{\widehat{H}} \rtimes \sigma ) = \DS(u_1,\cdots, u_{m^-},   t_1,\cdots, t_{m^+} )\cdot S  \rtimes \sigma \in \lang G^{\ur},  $$ where $S$ is the permutation matrix switching $\hat e_{m^-}$ and $\hat e_{d-m^- +1}$, and switching $\hat e_{m}$ and $\hat e_{m+1}$. Since $\delta$ is trivial, the action of $\sigma$ on $\widehat G$ is trivial. We know that $S^2 =1$, and the conjugation action of $S$ on $\mathcal T$ is given by  $$\DS(x_1,\cdots, x_m)\longmapsto\DS(x_1,\cdots, x_{m^--1}, x_{m^-}^{-1}, x_{m^-+1}, \cdots, x_{m-1}, x_m^{-1}). $$ 
Hence  with the notation (\ref{eq:nu_a}) we have  \begin{align*}
	s^{-1} \eta(t_{\widehat H} \rtimes \sigma)^a  = \DS( -u_1^a,\cdots, -u_{m^- -1} ^a, -u_{m^-}^{\nu_a}, t_1^a,\cdots, t_{m^+-1}^a, t_{m^+}^{\nu_a}) \cdot S^{\nu_a} \rtimes \sigma^a . 
\end{align*} If $a$ is even, the above element lies in $\mathcal T \rtimes \sigma^a$.
If $a$ is odd, we claim that the above element is $\widehat G$-conjugate to 
$$ \DS (-u_1^a,\cdots, -u_{m^--1}^a, -1,  t_1^a,\cdots, t_{m^+-1}^a, 1) \rtimes \sigma^a \in \mathcal T \rtimes \sigma^a. $$ 

To show the claim, it suffices to show that $\DS(x_1,\cdots, x_{m^-}, y_1,\cdots, y_{m^+}) \cdot S$ is $\widehat{G}$-conjugate to $\DS(x_1,\cdots, x_{m^- -1}, -1, y_1,\cdots, y_{m^+-1}, 1) $ for arbitrary $x_i,y_i \in \CC^{\times}$. Let $\mathcal J$ be the special orthogonal group of the $4$-dimensional quadratic space $\mathrm{span}\set{\hat e_{m^-}, \hat e_m,    \hat e_{m+1} , \hat e_{d-m^-+1} }$ over $\CC$. We write elements of $\mathcal J$ as $4\times 4$ matrices using the given basis. We identify $\mathcal J$
 as a subgroup of $\widehat G$, by letting elements of $\mathcal J$ act trivially on $\hat e_i$ for all $i \notin \set{m^-, m, m+1, d-m^-+1}$. Then $S \in \mathcal J$, and the $4\times 4$ matrix of $S$ is $$ \begin{pmatrix}
 	&&& 1 \\ && 1 \\ & 1 \\ 1
 \end{pmatrix}.$$ Let $U$ and $K$ be elements of $\mathcal J$ whose $4\times 4$ matrices are $\DS(x_{m^-}, y_{m^+})$ and $\DS(-1,1)$  respectively. Now since $US$ is semi-simple (as can be easily seen in $\GL_4$), it must be conjugate in $\mathcal J$ to some element of the diagonal maximal torus $\set{\DS(a,b)\mid a,b \in \CC^\times}$ in $\mathcal J$, which must be either $K$ or $-K$ by considering the characteristic polynomial. But $K$ and $-K$ are actually conjugate in $\mathcal J$. Hence $US$ is conjugate to $K$ in $\mathcal J$. Now inside $\widehat{G}$ we have 
\begin{multline*}
	\DS (x_1,\cdots, x_{m^-},   y_1,\cdots, y_{m^+})  S  \\ = \DS(x_1,\cdots, x_{m^- -1} , 1 , y_1,\cdots, y_{m^+ -1}, 1) US,
\end{multline*} and $\DS(x_1,\cdots, x_{m^- -1} , 1 , y_1,\cdots, y_{m^+ -1}, 1) $ commutes with $\mathcal J$. Hence the above element is  $\widehat{G}$-conjugate to \begin{multline*}
 \DS(x_1,\cdots, x_{m^- -1} , 1 , y_1,\cdots, y_{m^+ -1}, 1) K  \\ = \DS(x_1,\cdots, x_{m^- -1} , - 1 , y_1,\cdots, y_{m^+ -1}, 1) ,
\end{multline*} as desired. Our claim follows. 

Now suppose $f \in \mathscr A_{G_{F_a}}$ corresponds to 
$F(X_1,\cdots, X_m) \in \AD[X_1,\cdots, X_m]. $ By Lemma \ref{lem:essential tb} and the above claim, the evaluation of $\tilde \eta^*(f)$ at $t_{\widehat{H}} \rtimes \sigma$ is equal to
$$ F(-u_1^a,\cdots, -u_{m^- -1}^a, -1, t_1^a,\cdots, t_{m^+ -1} ^a, 1 ) $$ for both parities of $a$. Thus the map $\tilde \eta^* $ is explicitly given by 
\begin{align*}
	\AD[X_1,\cdots, X_{m}] & \To \AB[Z_1,\cdots, Z_{m^+-1}] \otimes_{\CC} \AB[Y_1,\cdots, Y_{m^--1}] \\ 
F(X_1,\cdots, X_m)  & \longmapsto  F(-Y_1^a,\cdots, -Y_{m^- -1}^a, -1, Z_1^a,\cdots, Z_{m^+ -1} ^a, 1 ).
\end{align*}
\subsection{}\label{para:twisted transfer summary}In the following, we collect the explicit description of $\tilde \eta^*$ in all the cases obtained in \S \ref{subsubsec:explicit tb}. 

\subsubsection{The odd case}

\begin{align*} \AB[X_1,\cdots, X_m] & \To \AB[Z_1,\cdots, Z_{m^+}] \otimes_{\CC} \AB[Y_1,\cdots, Y_{m^-}] \\ F(X_1,\cdots, X_m) & \longmapsto F(-Y_1^a,\cdots, -Y_{m^-}^a, Z_1^a,\cdots, Z_{m^+}^a). 
\end{align*}

\subsubsection{The even case, with trivial $\delta^+$ and trivial $\delta^-$}

\begin{align*}	\AD[X_1,\cdots, X_m] & \To \AD[Z_1,\cdots, Z_{m^+}] \otimes_{\CC} \AD[Y_1,\cdots, Y_{m^-}] \\  F(X_1,\cdots, X_m) & \longmapsto F(-Y_1^a,\cdots, -Y_{m^-}^a, Z_1^a,\cdots, Z_{m^+}^a). 
\end{align*}

\subsubsection{The even case, with non-trivial $\delta^+$ and trivial $\delta^-$}
\paragraph{Suppose $a$ is even}
\begin{align*}
	\AD[X_1,\cdots, X_m] & \To \AB[Z_1,\cdots, Z_{m^+-1}] \otimes_{\CC} \AD[Y_1,\cdots, Y_{m^-}] \\   F(X_1,\cdots, X_m) & \longmapsto F(-Y_1^a,\cdots, -Y_{m^-}^a, Z_1^a,\cdots, Z_{m^+ -1 }^a, 1). 
\end{align*}
\paragraph{Suppose $a$ is odd}  
\begin{align*}
	\AB[X_1,\cdots, X_{m-1}] & \To \AB[Z_1,\cdots, Z_{m^+-1}] \otimes_{\CC} \AD[Y_1,\cdots, Y_{m^-}] \\ 
 F(X_1,\cdots, X_{m-1}) & \longmapsto F(-Y_1^a,\cdots, -Y_{m^-}^a, Z_1^a,\cdots, Z_{m^+ -1 }^a).
\end{align*}

\subsubsection{The even case, with trivial $\delta^+$ and non-trivial $\delta^-$}

\paragraph{Suppose $a$ is even}

\begin{align*}
	\AD[X_1,\cdots, X_m] & \To \AD[Z_1,\cdots, Z_{m^+}] \otimes_{\CC} \AB[Y_1,\cdots, Y_{m^--1}] \\ F(X_1,\cdots, X_m) & \longmapsto   F(-Y_1^a,\cdots, -Y_{m^- -1}^a, -1, Z_1^a,\cdots, Z_{m^+ } ^a ) . 
\end{align*}
\paragraph{Suppose $a$ is odd}

\begin{align*}
	\AB[X_1,\cdots, X_{m-1}] & \To \AD[Z_1,\cdots, Z_{m^+}] \otimes_{\CC} \AB[Y_1,\cdots, Y_{m^--1}] \\ F(X_1,\cdots, X_{m-1}) & \longmapsto   F(-Y_1^a,\cdots, -Y_{m^- -1}^a, Z_{m^+}^a, Z_1^a,\cdots, Z_{m^+ -1} ^a ) . 
\end{align*}

\subsubsection{The even case, with non-trivial $\delta^+$ and non-trivial $\delta^-$}

\begin{align*}
	\AD[X_1,\cdots, X_{m}] & \To \AB[Z_1,\cdots, Z_{m^+-1}] \otimes_{\CC} \AB[Y_1,\cdots, Y_{m^--1}] \\ F(X_1,\cdots, X_m)  & \longmapsto  F(-Y_1^a,\cdots, -Y_{m^- -1}^a, -1, Z_1^a,\cdots, Z_{m^+ -1} ^a, 1 ). 
\end{align*}

 \section{Computation of twisted transfers}\label{subsec:p}
 \subsection{} We keep the setting of \S \ref{subsubsec:A_G explicit}, assume that $F = \QQ_p$, and import the constructions and notations in \S\S \ref{subsubsec:setting for G-endosc}--\ref{subsubsec:choices for W}. In particular, we fix $W,r,t$, and a hyperbolic basis $\mathbb B_{W^{\perp}}$ of $W^{\perp}$, and from these data we obtain a Levi subgroup $M \subset G$ (defined over $\QQ_p$). Since $G$ is by assumption unramified over $\QQ_p$, so is $M$.
   
 Let $\fkp= (d^+,\delta^+,d^-,\delta^-)$ be a quadruple satisfying all the conditions in the definition of the set $\mathscr P_W$, except that even when $\dim W$ is even and at least $4$ we still allow $(d^+, \delta^+) = (2,1)$ or $(d^-, \delta^-) = (2,1)$ (or both); cf.~the discussion at the beginning of \S \ref{subsubsec:explicit tb}. Let $A$ be a subset of $[r]$ and $B$ be a subset of $[t]$. Although $(A,B, d^+,\delta^+,d^-,\delta^-)$ is more general than an element of $\mathscr P_{r,t} \times' \mathscr P_W$ as in Definition \ref{defn:P_W}, the construction in \S \ref{para:presentation of endoscopic G-data} still applies to it and yields an endoscopic $G$-datum for $M$ : $$\fke_{A,B, \fkp}  = (M' , \lang M', s_M, \eta_M),$$ which may no longer be bi-elliptic. Also, we obtain an endoscopic datum for $M$ :
   $$\fke_{\fkp}(M) =  \mathfrak e_{d^+,\delta^+, d^-, \delta^-} (M) = (M', \lang M', s_M', \eta_M)$$ and an  endoscopic datum for $G$: $$\mathfrak e_{d^+ + 2\abs{A} + 4 \abs{B}, \delta^+, d^- + 2 \abs{A^c} + 4 \abs{B^c}, \delta^-} = (H, \lang H,  s, \eta),$$ both of which are possibly non-elliptic (due to the possible appearance of $(2,1)$ in the subscripts). Note that the last two endoscopic data are unramified if and only if both $\delta^+$ and $\delta^-$ have even $p$-adic valuations. (In the odd case this is automatic.) In the following we assume that this is the case.
 
  Fix a positive integer $a$. As in \S \ref{subsec:twisted transfer}, we have the twisted transfer map induced by the unramified endoscopic datum $(H,\lang H, \eta,s)$ for $G$:
 $$ b : \mathcal H^{\ur} (G_{\QQ_{p^a}}) \To  \mathcal H^{\ur} (H). $$
 
 Let $\mu$ be the cocharacter of $G$ such that the $\GG_m$-action on $V$ via $\mu$ has weight $1$ on $f_1 \in \mathbb B_{W^{\perp}}$, weight $-1$ on $f_{2(r+2t)} \in \mathbb B_{W^{\perp}}$, and weight zero on the orthogonal complement of these two vectors. Thus $\mu$ is given by 
 \begin{align*}
 	\GG_m & \To \GG_m^{r} \times \GL_2^t \xrightarrow{(\ref{eq:M})} \SO(W^\perp) \To G \\
 	z & \longmapsto (z,1,\cdots,1, I_2,\cdots, I_2)
  \end{align*}
 if $r>0$, and is given by  
 \begin{align*}
 	\GG_m & \To \GL_2^t \xrightarrow{(\ref{eq:M})} \SO(W^\perp) \To G \\
 	z & \longmapsto (\diag(z,1), I_2, \cdots,I_2)
 \end{align*}
if $r=0$. Let $$f_{- \mu} \in \mathcal H^{\ur} (G_{\QQ_{p^a}})$$ be as in Definition \ref{defn:canonical fn a la Kottwitz}, with $F= \QQ_{p^a}$ and $\lambda = -\mu$. Define $$f^H: = b(f_{-\mu}) \in \mathcal H^{\ur} (H). $$ 

The construction in 
\S \ref{para:two maps from endoscopic G data} still applies to the current slightly more general situation (with the possibly non-elliptic data). Hence $M'$ is identified with a Levi subgroup of $H$ (up to $H(F)$-conjugation). We have the canonical constant term map (see Proposition \ref{canonical constant term}): 
$$(\cdot)_{M'} : \mathcal H^{\ur} (H) \To \mathcal H^{\ur} (M').$$ 
In the following we describe $(f^H)_{M'}$. 

Recall from \S \ref{subsubsec:setting for G-endosc} that $M = M^{\GL} \times M^{\SO}$, where $M^{\GL}$ is identified with $\GG_m^r \times \GL_2^t$ via (\ref{eq:M}), and $M^{\SO} = \SO(W)$. The maximal split torus in $M^{\GL}$ given by the product of $\GG_m^r$ with the diagonal tori in the copies of $\GL_2$ is naturally identified with $\GG_m^{r+2t}$. Correspondingly, the algebra $\mathscr A_{M^{\GL}}$ is naturally identified with \index{$\xi_i$} \index{$\zeta_i$} $$\CC[\xi_1^{\pm 1}] \otimes_{\CC} \cdots \otimes_{\CC} \CC[\xi_r^{\pm 1}] \otimes_{\CC} \CC[\zeta_1^{\pm 1}, \zeta_2^{\pm 1}]^{\mathfrak S_2} \otimes_{\CC} \cdots \otimes_{\CC}  \CC[\zeta_{2t-1}^{\pm 1}, \zeta_{2t}^{\pm 1}]^{\mathfrak S_2} .$$ (Here $\mathfrak S_2$ acts on each $\CC[\zeta_j^{\pm}, \zeta_{j+1}^{\pm}]$ by swapping $\zeta_j$ and $\zeta_{j+1}$.) In the sequel we shall view elements of the above algebra, such as $\zeta_1 + \zeta_2$, as an element of $\mathscr A_{M^{\GL}}$ or $\mathcal H^{\ur} (M^{\GL})$. We have $M' = M^{\GL} \times M^{\prime, \SO}$ (see \S \ref{para:presentation of endoscopic G-data}), and correspondingly we have 
$$\mathcal H^{\ur} (M') = \mathcal H^{\ur} (M^{\GL} ) \otimes_{\CC} \mathcal H^{\ur} ( M^{\prime, \SO} ). $$  We retain the notation $\nabla_i(\cdot)$ as in Definition \ref{defn:nabla}. 
\begin{prop}\label{prop:main computation at p}
	The element $p^{a (2-d)/2}(f^H)_{M'} \in \mathcal H^{\ur} (M')$ is of the form $$k(A,B)\otimes 1 + 1\otimes h , $$ with $k(A,B) \in \mathcal H^{\ur} (M^{\GL })$ and $h \in \mathcal H^{\ur} (M^{\prime,\SO})$. The element $h$ depends only on the parameter  $\fkp = (d^+,\delta^+,d^-,\delta^-)$, not on $(A,B)$. The element $k(A,B)$ is given by $$k(A,B) = \sum_{i = 1}^r \nabla_i (A) (\xi_i^a + \xi_i^{-a}) + \sum_{j=1}^t \nabla_j (B) (\zeta_{2j-1}^a + \zeta_{2j-1}^{-a} + \zeta_{2j}^a + \zeta_{2j}^{-a}) . $$ 
\end{prop} 	
\begin{proof} Write $F_a$ for $\QQ_{p^a}$.
	Fix a maximal $F_a$-split torus $S$ in $G_{F_a}$. In this proof we omit notations for the Satake isomorphisms. We use Theorem \ref{thm:Kottwitz on Satake} to compute (the Satake transform of) $f_{-\mu}$. We have $\lprod{\rho, (-\mu)_{\mathrm{dom}}} = (d-2)/2$, and so 
	\begin{align*} 
		p^{a(2-d)/2} f_{-\mu} =  \sum_{\lambda \in \Omega(F)\cdot (-\mu)} [\lambda] \in  \CC[X_*(S)]^{\Omega(F)} \cong \mathscr A_{G_{F_a}} 
	\end{align*}
	by that theorem. Let $m = \floor{d/2}$ be the absolute rank of $G$. As in \S \ref{subsubsec:A_G explicit}, $\mathscr A_{G_{F_a}}$ is identified with one of the three algebras $$\mathscr A_{\mathsf B}[X_1,\cdots, X_m], \quad   \mathscr A_{\mathsf D}[X_1,\cdots, X_m], \quad \mathscr A_{\mathsf B}[X_1,\cdots, X_{m-1}]. $$  
	Correspondingly, we have 
	$$ 	p^{a(2-d)/2} f_{-\mu} = \begin{cases}
		 X_1 + X_1^{-1} + \cdots + X_m + X_m^{-1}  \in \mathscr A_{\mathsf B}[X_1,\cdots, X_m] , \\
		 X_1 + X_1^{-1} + \cdots + X_m + X_m^{-1}  \in \mathscr A_{\mathsf D}[X_1,\cdots, X_m], \\
		 X_1 + X_1^{-1} + \cdots + X_{m-1} + X_{m-1}^{-1}  \in \mathscr A_{\mathsf B}[X_1,\cdots, X_{m-1}].
	\end{cases} $$ For any positive integer $l$, we introduce short-hand notations
\begin{align*}
	\mathscr A_{\mathsf B}[\mathcal Y_l]  & : = 	\mathscr A_{\mathsf B}[Y_1,\cdots, Y_l], &
	 \mathscr A_{\mathsf D}[\mathcal Y_l] & : = 	\mathscr A_{\mathsf D}[Y_1,\cdots, Y_l], \\
	 \mathscr A_{\mathsf B}[\mathcal Z_l] & : = 	\mathscr A_{\mathsf B}[Z_1,\cdots, Z_l], & 
	 \mathscr A_{\mathsf D}[\mathcal Z_l] & : = 	\mathscr A_{\mathsf D}[Z_1,\cdots, Z_l], \end{align*} and \begin{align*}
\mathcal Y_l^a  & : = \sum _{i =1}^l Y_i^a + Y_i^{-a} , & 
\mathcal Z_l^a  & : = \sum _{i =1}^l Z_i^a + Z_i^{-a} .
\end{align*}   
	We then compute, according to  \S \ref{para:twisted transfer summary}, that (the Satake transform of) $	p^{a(2-d)/2} f^H$ in $\mathscr A_H$ is given by:  
	
	\begin{align}\label{eq:five cases}
	 \begin{cases} -
		\mathcal Y_{m^-}^a + \mathcal Z_{m^+}^a  \in \mathscr A_{\mathsf B}[\mathcal Z_{m^+}] \otimes \mathscr  A_{\mathsf B}[\mathcal Y_{m^-}] , & d  \text{ odd}\\
	-\mathcal Y_{m^-}^a + \mathcal Z_{m^+}^a  \in \mathscr A_{\mathsf D}[\mathcal Z_{m^+}] \otimes \mathscr  A_{\mathsf D}[\mathcal Y_{m^-}] , & d \text{ even}, \delta^+ = \delta^- =1,\\
	-\mathcal Y_{m^-}^a + \mathcal Z_{m^+-1}^a   + 1 + (-1)^a  \in \mathscr A_{\mathsf B}[\mathcal Z_{m^+-1}] \otimes \mathscr  A_{\mathsf D}[\mathcal Y_{m^-}] , & d \text{ even}, \delta^+ \neq 1, \delta^- =1,   \\ 
	-\mathcal Y_{m^--1}^a + \mathcal Z_{m^+}^a  - 1 - (-1)^a  \in  \mathscr A_{\mathsf D}[\mathcal Z_{m^+}] \otimes \mathscr  A_{\mathsf B}[\mathcal Y_{m^--1}], & d \text{ even}, \delta^+ = 1, \delta^- \neq 1 , \\
-	\mathcal Y_{m^--1}^a + \mathcal Z_{m^+-1}^a  \in \mathscr A_{\mathsf B}[\mathcal Z_{m^+-1}] \otimes \mathscr  A_{\mathsf B}[\mathcal Y_{m^--1}] , & d \text{ even}, \delta^+ \neq 1, \delta^- \neq 1 .  
	\end{cases}	
\end{align}

	Recall that $M^{\prime, \SO} = \SO(W^+) \times \SO(W^-)$. Write $n^{\pm}$ for the absolute rank of $\SO(W^{\pm})$. Similar to $\mathscr A_{H^+}$, we identify $\mathscr A _{\SO(W^+)}$ with one of  
	$$\mathscr A_{\mathsf B}[\mathcal Z_{n^+}], \quad   \mathscr A_{\mathsf D}[\mathcal Z_{n^+}], \quad \mathscr A_{\mathsf B}[\mathcal Z_{n^+ -1 }], $$ in the odd case, in the even case with $\delta^+ =1$, and in the even case with $\delta^+ \neq 1$ respectively. Similarly, we identify   $\mathscr A _{\SO(W^-)}$ with one of  
	$$\mathscr A_{\mathsf B}[\mathcal Y_{n^-}], \quad   \mathscr A_{\mathsf D}[\mathcal Y_{n^-}], \quad \mathscr A_{\mathsf B}[\mathcal Y_{n^- -1 }], $$ in the odd case, in the even case with $\delta^- =1$, and in the even case with $\delta^- \neq 1$ respectively. The constant term map $\mathscr A_H \to \mathscr A_{M'}$ is of the form
		$$ \begin{cases} 
	  \mathscr A_{\mathsf B}[\mathcal Z_{m^+}] \otimes \mathscr  A_{\mathsf B}[\mathcal Y_{m^-}] \to    \mathscr A_{M^{\GL}}\otimes \mathscr A_{\mathsf B}[\mathcal Z_{n^+}] \otimes \mathscr A_{\mathsf B}[\mathcal Y_{n^-}] ,  \\
		  \mathscr A_{\mathsf D}[\mathcal Z_{m^+}] \otimes \mathscr  A_{\mathsf D}[\mathcal Y_{m^-}]  \to    \mathscr A_{M^{\GL}}\otimes \mathscr A_{\mathsf D}[\mathcal Z_{n^+}] \otimes \mathscr A_{\mathsf D}[\mathcal Y_{n^-}] , \\
	 \mathscr A_{\mathsf B}[\mathcal Z_{m^+-1}] \otimes \mathscr  A_{\mathsf D}[\mathcal Y_{m^-}]  \to    \mathscr A_{M^{\GL}}\otimes \mathscr A_{\mathsf B}[\mathcal Z_{n^+-1}] \otimes \mathscr A_{\mathsf D}[\mathcal Y_{n^-}] ,  \\ 
	  \mathscr A_{\mathsf D}[\mathcal Z_{m^+}] \otimes \mathscr  A_{\mathsf B}[\mathcal Y_{m^--1}] \to    \mathscr A_{M^{\GL}}\otimes \mathscr A_{\mathsf D}[\mathcal Z_{n^+}] \otimes \mathscr A_{\mathsf B}[\mathcal Y_{n^--1}]  ,   \\
	 \mathscr A_{\mathsf B}[\mathcal Z_{m^+-1}] \otimes \mathscr  A_{\mathsf B}[\mathcal Y_{m^--1}] \to    \mathscr A_{M^{\GL}}\otimes \mathscr A_{\mathsf B}[\mathcal Z_{n^+-1}] \otimes \mathscr A_{\mathsf B}[\mathcal Y_{n^--1}]   ,
	\end{cases}
	$$
	where the division into the five cases is the same as in (\ref{eq:five cases}). In each case, using Lemma \ref{lem:partial Satake}, we see that the map is determined by the following rule: Write \begin{align*}
	A & = \set{i_1,\cdots, i_u}, & A^c & = \set{\tilde i_1,\cdots, \tilde i_{r-u}}, \\ B & = \set{j_1,\cdots, j_v}, &  B^c & = \set{\tilde j_1,\cdots, \tilde j_{t-v}}.
	\end{align*}  We send $Z_1,\cdots, Z_{u}$ to $\xi_{i_1},\cdots, \xi_{i_u}$, send $Y_1,\cdots, Y_{r-u}$ to $\xi_{\tilde i_1}, \cdots, \xi_{\tilde i_{r-u}}$, send $Z_{u+1},\cdots, Z_{u+2v}$ to $$ \zeta_{2j_1 -1}, \zeta_{2j_1}, \zeta_{2j_2 -1}, \zeta_{2j_2},\cdots, \zeta_{2j_v -1}, \zeta_{2j_v}, $$ send $Y_{r-u+1}, \cdots, Y_{r-u + 2t -2v}$ to $$\zeta_{2 \tilde j_1 -1}, \zeta_{2 \tilde j_1}, \zeta_{2 \tilde j_2 -1}, \zeta_{2 \tilde j_2}, \cdots, \zeta_{2 \tilde j_{t-v} - 1 }, \zeta_{2 \tilde j_{t-v}}, $$ send the remaining $Z_i$'s to $Z_1,Z_2,\cdots$, and send the remaining $Y_i$'s to $Y_1, Y_2,\cdots$. From this description of the constant term map and the previous computation (\ref{eq:five cases}) of $p^{a (2-d)/2}f^H$, we see that 
	$p^{a (2-d)/2}(f^H)_{M'} \in \mathcal H^{\ur} (M')$ is of the form $$k(A,B)\otimes 1 + 1\otimes h , $$ where $k(A,B)$ is given as in the statement of the proposition, and $$h \in \mathscr A_{\SO(W^+)} \otimes \mathscr A_{\SO(W^-)}$$ is given by  
	$$ \begin{cases} 
		 -\mathcal Y_{n^-}^a + \mathcal Z_{n^+}^a \in  \mathscr A_{\mathsf B}[\mathcal Z_{n^+}] \otimes \mathscr A_{\mathsf B}[\mathcal Y_{n^-}] ,  \\  -\mathcal Y_{n^-}^a + \mathcal Z_{n^+}^a \in 
  \mathscr A_{\mathsf D}[\mathcal Z_{n^+}] \otimes \mathscr A_{\mathsf D}[\mathcal Y_{n^-}] , \\
	 -\mathcal Y_{n^-}^a + \mathcal Z_{n^+-1}^a + 1 + (-1)^a \in   \mathscr A_{\mathsf B}[\mathcal Z_{n^+-1}] \otimes \mathscr A_{\mathsf D}[\mathcal Y_{n^-}] ,  \\ 
	   -\mathcal Y_{n^--1}^a + \mathcal Z_{n^+}^a - 1 - (-1)^a \in \mathscr A_{\mathsf D}[\mathcal Z_{n^+}] \otimes \mathscr A_{\mathsf B}[\mathcal Y_{n^--1}]  ,   \\
	  -\mathcal Y_{n^--1}^a + \mathcal Z_{n^+-1}^a \in  \mathscr A_{\mathsf B}[\mathcal Z_{n^+-1}] \otimes \mathscr A_{\mathsf B}[\mathcal Y_{n^--1}]   ,
	\end{cases}
	$$
	in the five cases as before. Clearly $h$ depends only on $\fkp$, not on $(A,B)$. 
\end{proof}

\chapter{Stabilization}
\label{Section Stabilization}

\section[Definitions and facts on Langlands--Shelstad transfer]{Standard definitions and facts on Langlands--Shelstad transfer}\label{subsec:review of LS}
\subsection{} For any field $F$ of characteristic zero and homomorphism $I \to J$ of algebraic groups over $F$, we write \index{$\D(I,J;F)$}
$$\D (I, J; F) : =  \ker (\coh ^1(F,I) \to \coh^1(F,J)). $$

Now let $F$ be a non-archimedean local field of characteristic zero, and $G$ a reductive group over $F$. We recall the definition of $\kappa$-orbital integrals in the fashion of \cite[\S 2.7]{labesse1999}. Let $\gamma \in G(F)$ be a  semi-simple element, and write $I_{\gamma}$ for $(G_{\gamma})^0$. Recall from \cite[\S 2.3]{labesse1999} that there is a natural surjection from $\D(I_{\gamma}, G ; F)$ to the set of conjugacy classes in the stable conjugacy class of $\gamma$, which is a bijection if $I_{\gamma} = G _{\gamma}$. We have a short exact sequence of pointed sets
\begin{align}\label{seq for Haar}
1 \To I_{\gamma} (F) \backslash G(F) \To \coh^0 (F, I_{\gamma} \backslash G ) \To \D(I_{\gamma}, G ;F) \To 1,
\end{align} and a natural map (see \cite[\S 1.8]{labesse1999})
$$\coh^0(F, I_{\gamma} \backslash G) \To \coh^0_{\ab} (F, I_{\gamma} \backslash G),$$
where $\coh^0 _{\ab} (F, I_{\gamma } \backslash G
 )$ is a locally compact topological abelian group. Denote by $\mathfrak K (I_{\gamma}, G ; F)$ the Pontryagin dual group\index[n]{Pontryagin dual group} of $\coh^ 0_{\ab} (F, I_{\gamma} \backslash G)$. \footnote{Under the assumption that $F$ is non-archimedean, $\mathfrak K (I_{\gamma}, G ; F)$ is isomorphic to the group $\mathfrak K(I_{\gamma}/F)$ defined   in \cite[\S 4.6]{kottwitzelliptic}.}

Choose Haar measures on $I_{\gamma} (F)$ and on $G(F)$, and equip $\D (I_{\gamma},  G ; F)$ with the counting measure. Then the short exact sequence (\ref{seq for Haar}) defines a measure $dx$ on $\coh^0(F, I_{\gamma } \backslash G)$; see \cite[\S 2.7]{labesse1999}. For $f\in C^{\infty} _c (G(F))$ and $\kappa \in \mathfrak K(I_{\gamma}, G  ; F)$, define the \emph{$\kappa$-orbital integral}\index[n]{$\kappa$-orbital integral} \index{$O_{\gamma}^{\kappa}$}
$$ O_{\gamma} ^{\kappa} (f) :  = \int_{x\in \coh^0(F, I_{\gamma} \backslash G)}  e(I_{x^{-1}\gamma x})  \kappa (x) f (x^{-1} \gamma x) dx , $$ where $e(I_{x^{-1} \gamma x})$ is the Kottwitz sign\index[n]{Kottwitz sign} of $I_{x^{-1} \gamma x}$ (see \cite[Def.~1.7.1]{labesse1999}). Also define the \emph{stable orbital integral}\index[n]{stable orbital integral}\index{$SO_{\gamma}$} $$SO_{\gamma} (f) : = O^1_{\gamma} (f). $$ 

\begin{rem}
We give a more concrete description of $O^{\kappa}_{\gamma}(f)$. For each $[x] \in \D(I_{\gamma}, G ; F)$, fix an element $x \in G(\overline F)$ mapping to $[x]$ under the composite map
	$$ G(\overline F) \To \coh^0 (F , I_{\gamma }\backslash G) \To \D(I_{\gamma} , G  ; F). $$
	Then $\gamma_x : =  x^{-1} \gamma x $ is in $ G(F)$ and $\Int (x)$ induces an inner twisting $I_{\gamma} \to I_{\gamma_x}$. In particular, the Haar measure on $I_{\gamma}(F)$ transfers to a Haar measure on $I_{\gamma_x} (F)$. Using this and the fixed Haar measure on $G(F)$, we define the orbital integral $$ O_{\gamma_x} (f) : = \int_{x \in I_{\gamma_x}(F) \backslash G(F)} f(x^{-1} \gamma x).  $$  Then we have 
	$$ O_{\gamma}^{\kappa} (f)  = \sum_{ [x] \in \D(I_{\gamma}, G ;F)} e(I_{\gamma_x}) \kappa (x)  O_{\gamma_x} (f) .  $$
\end{rem}
\subsection{}\label{para:expl of LS}
 Fix an inner twisting $\psi: G \to G^*$ with $G^*$ quasi-split, and fix an $L$-group datum for $G$, as in \cite{LS87}. Let $(H, \mathcal H, s , \eta)$ be an endoscopic datum for $G$. For simplicity, we assume that $\mathcal H = \lang H$ (cf.~the discussion in \S \ref{subsubsec:extended endo data}). 
The notion of when a semi-simple element $\gamma_H\in H(F)$ (not necessarily $G$-regular) is an \emph{image}\index[n]{image (in the context of endoscopy)} of a semi-simple element $\gamma\in G(F)$ is defined in \cite[\S 1.2]{LS90}. 

Under the additional assumption that $G^{\der}$ is simply connected, Langlands--Shelstad \cite[\S 2.4]{LS90} have defined transfer factors for $(G,H)$-regular elements. Thus after fixing a normalization we have a number
$$\Delta(\gamma_H, \gamma) \in \CC,$$
 for each semi-simple $(G,H)$-regular $\gamma_H\in H(F)$ and each semi-simple $\gamma\in G(F)$. Moreover, $\Delta(\gamma_H,\gamma)$ depends on $\gamma_H$ (resp.~$\gamma$) only via its stable conjugacy class (resp.~conjugacy class) over $F$, and we have $\Delta (\gamma_H, \gamma) =0$ unless $\gamma_H$ is an image of $\gamma$.
 
 Since we have assumed that $\mathcal H = \lang H$, we can in fact define $\Delta(\gamma_H, \gamma)$ for $(G,H)$-regular $\gamma_H$ without assuming that $G^{\der}$ is simply connected. In the more restrictive $G$-regular case, this is done in \cite{LS87}; below we explain the $(G,H)$-regular case. For this, consider a $z$-extension $1\to Z \to G_1 \to G \to 1$. This determines a central extension $1 \to Z \to H_1 \to H \to 1$ as in \cite[\S 4.4]{LS87}. As explained in \textit{loc.~cit.}, we have a homomorphism $\eta_1: \lang H_1 \to \lang G_1$ such that $(H_1, \lang H_1, s, \eta_1)$ is an endoscopic datum for $G_1$. The restriction of $\eta_1$ to $\widehat{H_1}$ is canonical, but $\eta_1$ itself is canonical only up to twisting by a cocycle in the center of $\widehat{H_1}$. In our current situation (with $\mathcal H = \lang H$), we can take $\eta_1$ such that the diagram 
 \begin{align} \label{diag:z-ext}
\xymatrix{\lang H \ar[d] \ar[r]^{\eta} & \lang G \ar[d] \\ \lang H_1 \ar[r]^{\eta_1} & \lang G_1  }
 \end{align}
 commutes, where the vertical arrows are the natural ones associated to $H_1 \to H$ and $G_1 \to G$. This pins down $\eta_1$ canonically. We then define $\Delta(\gamma_H, \gamma)$ to be zero if $\gamma_H$ is not an image of $\gamma$, and otherwise to be $\Delta(\gamma_{H_1}, \gamma_1)$ where $\gamma_{H_1} \in H_1(F)$ (resp.~$ \gamma _1\in G_1(F)$) is a lift of $\gamma_H$ (resp.~$\gamma$) such that $\gamma_{H_1}$ is an image of $\gamma_1$, and $\Delta(\gamma_{H_1}, \gamma_1)$ is defined with respect to the endoscopic datum $(H_1, \lang H_1, s, \eta_1)$ for $G_1$ as in \cite[\S 2.4]{LS90}. In the latter case, the pair $(\gamma_{H_1}, \gamma_1)$ always exists, and is unique up to simultaneous translation by $Z(F)$. To show that this definition of $\Delta(\gamma_H,\gamma)$ is independent of the lifts, it suffices to check that  $\Delta( z \gamma_{H_1}, z \gamma_1) = \Delta( \gamma_{H_1},  \gamma_1)$ for all $z\in Z(F)$. For this it suffices to treat the case where $\gamma_{H_1}$ is strongly $ G_1$-regular. Then the desired statement is proved on p.~254 of \cite{LS87} (with $\lambda =1$). One can also check that the above definition is independent of the choice of the $z$-extension $G_1$. For this, using the standard fact (see \cite[Lem.~1.1]{kottwitzrational}) that any two $z$-extensions of $G$ can be dominated by a third $z$-extension, one is reduced to checking that when $G^{\der}$ is simply connected, for strongly $G$-regular $\gamma_H \in H(F)$, the definition of $\Delta(\gamma_H, \gamma)$ as above (i.e., $\Delta(\gamma_H, \gamma) : = \Delta(\gamma_{H_1}, \gamma_1)$ with a given $z$-extension $G_1$ and with $\eta_1$ pinned down as above) agrees with the original definition of $\Delta(\gamma_H, \gamma)$ in \cite{LS87}. This is a routine exercise which involves checking suitable functorial properties of all the terms $\Delta_I ,\cdots, \Delta_{IV}$ in \textit{loc.~cit.}. 
 
The Langlands--Shelstad Transfer Conjecture and the Fundamental Lemma are now unconditional theorems thanks to the work of Ng\^o \cite{ngo10}, Waldspurger \cite{waldsimplique, waldschangement}, Cluckers--Loeser \cite{cluckersLoeser}, and Hales \cite{halesFL}.  We recall these statements in the following theorem\footnote{We state only the Fundamental Lemma for the unit element of the unramified Hecke algebra. The references \cite{ngo10}, \cite{waldschangement}, and \cite{cluckersLoeser} give this result for characteristic zero local fields with sufficiently large residue characteristic. In \cite{halesFL} it is shown that the Fundamental Lemma for the unit for all sufficiently large residue characteristic is enough to imply the Fundamental Lemma (for the full unramified Hecke algebra) for characteristic zero local fields with arbitrary residue characteristic. See also \cite{LMW}.}, taking into account the extension to $(G,H)$-regular elements in \cite[\S 2.4]{LS90}.
 
\begin{thm}\label{thm:LS} Let $G$ be a reductive group over a non-archimedean local field $F$ of characteristic zero. Let $(H, \lang H, s ,\eta)$ be an endoscopic datum for $G$. 
	\begin{enumerate}
		\item (Langlands--Shelstad Transfer.)\index[n]{Langlands--Shelstad Transfer} Fix a normalization of the transfer factors, and fix Haar measures on $G(F)$ and $H(F)$. For any $f\in C_c^{\infty} (G(F))$, there exists $f ^H\in C_c^{\infty} (H(F)) $, called the \emph{Langlands--Shelstad transfer}\index[n]{Langlands--Shelstad transfer} of $f$, with the following properties: For any semi-simple $(G, H)$-regular  $\gamma_H \in H(F)$, we have  
		\begin{equation}\label{transfer id}
		SO_{\gamma_H}(f^H)	 =\begin{cases}
		 0, & ~ \gamma_H \text{ is not an image from } G ,\\
		  \Delta(\gamma_H, \gamma) O_{\gamma}^{s} (f) ,&  ~ \gamma_H \text{ is an image of } \gamma\in G(F)_{\mathrm{ss}} .
		\end{cases} 
		\end{equation}In the second situation of (\ref{transfer id}) we have the following explanations.   
	\begin{itemize}
\item The component $s$ in $(H, \lang H, s, \eta)$ defines an element of $\mathfrak K(I_{\gamma} , G ; F)$ still denoted by $s$, and we use that to define $O_{\gamma}^s $.
\item We define $SO_{\gamma_H} (f^H)$ and $O^s_{\gamma} (f)$  using the fixed Haar measures on $G(F) $ and $H(F)$ and compatible Haar measures on $G_{\gamma} ^0 (F)$ and $H_{\gamma_H} ^0 (F)$.  
	\end{itemize}  
	
\item (Fundamental Lemma.)\index[n]{Fundamental Lemma} Suppose $G$ and $(H,\lang H, s,\eta)$ are unramified (see \S \ref{subsubsec:twisted transfer}). Normalize the Haar measures on $G(F)$ and $H(F)$ such that all hyperspecial subgroups have volume $1$. Let $K$ (resp.~$K_H$) be an arbitrary hyperspecial subgroup of  $G(F)$ (resp.~$H(F)$). Then $1_{K_H}$ is a Langlands--Shelstad transfer of $1_{K}$ as in part (1), for the unramified normalization canonically associated to $K$ of transfer factors defined in \cite{halesunram}. 

\item (Adelic Transfer.) Let $G_0$ be a reductive group over a number field $F_0$ and let $(H_0, \lang H_0 , s_0, \eta_0)$ be an endoscopic datum for $G_0$ over $F_0$. Suppose there is a finite set $\Sigma$ of finite places of $F_0$ and a reductive model $\mathcal G$ of $G_0$ over $\mathcal O_{F_0}[1/\Sigma]$ such that for all finite places $v$ of $F_0$ outside $\Sigma$ the endoscopic datum $(H_0, \lang H_0, s_0,\eta_0)$ localizes to an unramified endoscopic datum over $F_{0,v}$, and the transfer factors between $H_{F_{0,v}}$ and $G_{F_{0,v}}$ are normalized under the canonical unramified normalization associated to $\mathcal G(\mathcal O_{F_{0,v}})$. Let $S$ be the union of $\Sigma$ and the set of all archimedean places of $F_0$, and let $\adele^S_{F_0}$ denote the adeles over $F_0$ away from $S$. For any $f \in C^{\infty} _c (G_0 (\adele ^S_{F_0}))$, there exists $f^H \in C^{\infty} _c (H_0 (\adele^S_{F_0})) $ such that the $\adele^S_{F_0}$-analogue of (\ref{transfer id}) holds. Here the notion of an adelic $(G_0,H_0)$-regular element is defined in \cite[\S 7, pp.~178--179]{kottwitzannarbor}, and all the orbital integrals are defined with respect to adelic Haar measures. 
	\end{enumerate}
\end{thm}

\begin{rem}\label{rem:on using LS}
Part (1) of Theorem \ref{thm:LS} appears to be stronger than the original form of the Langlands--Shelstad Conjecture in two ways. Firstly, the original conjecture is about transferring functions on $G$ to functions on a central extension $H_1$ of $H$. More precisely, fix a $z$-extension $1 \to Z \to G_1 \to G \to 1$ and obtain $H_1$ as in \S \ref{para:expl of LS}. For a choice of $\eta_1: \lang H_1 \to \lang G_1$ (recall that $\eta_1 |_{\widehat {H_1}}$ is canonical), the conjecture concerns transferring functions in $C^{\infty}_c(G(F))$ to functions in $C^{\infty} _c (H_1(F), \lambda)$. Here  $\lambda$ is a character on $Z(F)$ determined by $\eta_1$, and $C^{\infty} _c (H_1(F), \lambda)$ denotes the set of functions in $C^{\infty} (H_1(F))$ that transform under $Z(F)$ by $\lambda$ and whose supports are compact modulo $Z(F)$. Now under our assumption that $\mathcal H = \lang H$, we may and shall pin down $\eta_1$ as in \S \ref{para:expl of LS}, and then $\lambda = 1$. In view of the definition of the transfer factors discussed in \S \ref{para:expl of LS}, we know that under the natural bijection $ C^{\infty} _c (H_1(F) , 1)\isom C^{\infty} _c (H(F))$, a Langlands--Shelstad transfer of $f \in C^{\infty}_c(G(F))$ to $C^{\infty}_c (H_1(F), 1)$ in the original sense corresponds to a Langlands--Shelstad transfer of $f$ to $C^{\infty}_c (H(F))$ in the sense of Theorem \ref{thm:LS}. 

Secondly, in the original conjecture the identity (\ref{transfer id}) is only required to hold for all $G$-regular $\gamma_H$. In \cite[\S 2.4]{LS90}, Langlands--Shelstad prove that this indeed implies (\ref{transfer id}) for all $(G, H)$-regular $\gamma_H$, under the assumption that $G^{\der}$ is simply connected. In view of the last paragraph, we know that this implication is still valid without assuming that $G^{\der}$ is simply connected (but always under the assumption that $\mathcal H = \lang H$).

Similar remarks also apply to part (2) of Theorem \ref{thm:LS}. 
\end{rem}

\section{Calculation of some invariants}\label{subsec:calc of tau and k}
In this section let $G$ be the special orthogonal group of an arbitrary quadratic space of dimension $d \geq 2$ over $\QQ$. Let $m: = \floor{d/2}$.

\begin{prop}\label{prop:Tamagawa} Assume that $G$ is not the split $\SO_2$. Then the Tamagawa number $\tau (G) =2$. 
\end{prop}
\begin{proof} By \cite[(5.1.1)]{kottwitzcuspidal} and Weil's conjecture on Tamagawa numbers proved in \cite{kottTama}, we have
	\begin{align}\label{eq:formula for Tama}
		\tau(G) = \abs{\pi_0 (Z(\widehat {G}) ^{\Gamma_{\QQ}})} / \abs {\ker^1 (\QQ, Z(\widehat {G}))}.
	\end{align}  
First assume that $d \geq 3$. Then $\widehat {G}$ is a symplectic group of rank at least $1$ or an even orthogonal group of rank at least $2$, so $Z(\widehat {G} ) \cong \mu_2$. In particular, $\ker^1(\QQ, Z(\widehat G)) =0$ by Chebotarev's density theorem. On the other hand $\pi_0 (Z(\widehat G) ^{\Gamma_{\QQ}}) = Z(\widehat G) $ has cardinality $2$. Hence $\tau(G) =2$.
	
	 Now assume that $d=2$. Since $G$ is not split, it is isomorphic to the norm-$1$ subtorus of $\Res_{K/\QQ} \GG_m$ for some quadratic extension $K/\QQ$. We have $Z(\widehat G) = \widehat G = \CC^{\times}$. The action of $\Gamma_{\QQ}$ on $Z(\widehat G)$ factors through $\Gal(K/\QQ)$, and the non-trivial element of $\Gal(K/\QQ)$ acts by $z \mapsto z^{-1}$. Hence $Z(\widehat G) ^{\Gamma_{\QQ}} = \set{\pm 1}$. On the other hand, $\ker^1(\QQ, Z(\widehat G))$ is the dual group of the finite abelian group $\ker^1(\QQ, T)$ by \cite[(3.4.5.1)]{kottwitzcuspidal}, and the latter is trivial by the Hasse norm theorem (cf.~\cite[pp.~307--308]{PR}). Hence $\tau(G) = 2$.   
\end{proof}

\begin{defn} Let $H$ be   reductive group over $\RR$ assumed to contain elliptic maximal tori.  Define \index{$k(\cdot)$} $$ k(H) : = \abs{\im ( \coh^1(\RR, T_e^{\SC})   \to \coh^1 (\RR, T_e) )} , 
	$$
where $T_e$ denotes an elliptic maximal torus in $H$ and $T_e^{\SC}$ denotes the inverse image of $T_e$ in $H^{\SC}$.
Since all elliptic maximal tori in $H$ are conjugate under $H(\RR)$, $k(H)$ is well defined. 
\end{defn}

\begin{prop}\label{prop:comp k}Assume that $G_{\RR}$ contains elliptic maximal tori. Then $k(G) = 2^{m -1}$. 
\end{prop}
\begin{proof} If $d =2$, then $G_{\RR}$ is a torus, so obviously $k(G) =1$. In this case $m=1$, so the proposition is true. Now assume that $d\geq 3$. Let $T_e$ be an elliptic maximal torus in $G_{\RR}$, which is in fact anisotropic. As argued in the proofs of \cite[Lem.~5.4.2]{morel2010book} and \cite[Lem.~5.2.2]{morel2011suite}, we have\footnote{In \textit{loc.~cit.~}it is stated that $\abs{\im ( \coh^1(\RR, T_e \cap G^{\der})   \to \coh^1 (\RR, T_e) )} = \abs{\pi_0 (\widehat {T_e}^{\Gamma_{\infty}})} / \abs{\pi_0 (Z(\widehat{ G}) ^{\Gamma_{\infty}})}$, and in that context $G^{\der}$ is simply connected. For the correct generalization, one replaces the left hand side by $k(G)$.} $$k(G) = \abs{\pi_0 (\widehat {T_e}^{\Gamma_{\infty}})} / \abs{\pi_0 (Z(\widehat{ G}) ^{\Gamma_{\infty}})}.$$
We have  $ \pi_0 (Z(\widehat{ G}) ^{\Gamma_{\infty}}) \cong  Z(\widehat{ G} )\cong \ZZ/2\ZZ$, and since $T_e \cong \Uni(1) ^{m}$ we have $\pi_0 (\widehat{T_e} ^{\Gamma_{\infty}}) \cong (\ZZ/2\ZZ)^m$. Hence $k(G) = 2^{m-1}. $ 
\end{proof}
Recall that $\GL_{j,\RR}$ contains elliptic maximal tori precisely when $j = 1,2$.
\begin{prop}\label{prop:Tamagawa GL}
	For any $j\geq1$, $\tau(\GL_j) = 1$. For $j =1,2$, $  k (\GL_{j,\RR}) = 1.$ 
\end{prop}
\begin{proof}
	For $j\geq 1$, $Z(\widehat \GL_j) = \CC^{\times}$, on which $\Gamma_{\QQ}$ acts trivially. Hence $$ \pi_0 ( Z(\widehat \GL_j) ^{\Gamma_{\QQ}}) =\pi_0(\CC^{\times} ) = 1,$$ and $$  \ker^1(\QQ , Z(\widehat \GL_j)) =1$$ by Chebotarev's density theorem. Thus $\tau(\GL_j) =1$ by (\ref{eq:formula for Tama}). Since $\GL_{1,\RR}$ is a torus, we have $k(\GL_{1,\RR}) =1$. Any elliptic maximal torus $T_e$ in $\GL_{2,\RR}$ is isomorphic to $ \Res_{\CC/\RR} \GG_m$, and $\coh^1(\RR, T_e)$ is trivial by Shapiro's lemma. Hence $k(\GL_{2,\RR}) =1$.  
\end{proof}

\begin{cor}\label{tau k}
	Let $M$ be a Levi subgroup of $G$ defined over $\QQ$. Let $M'$ be the group in a bi-elliptic endoscopic $G$-datum for $M$. Let $H'$ be the induced endoscopic group for $G$. Assume that $M$ is not a direct product of copies of $\GL_1$ and $\GL_2$ over $\QQ$, and assume that all four $\RR$-groups $G_{\RR}, M_{\RR}, M'_{\RR}, H_{\RR}$ contain elliptic maximal tori. Then we have $$ \frac{\tau(G)}{\tau(H)} \frac{\tau (M')}{\tau (M)} = \frac{k(H)}{k(G)} \frac{ k(M)} { k(M')}. $$
\end{cor}
\begin{proof} We have $M \cong M^{\GL } \times M^{\SO}$, where $M^{\GL}$ is a product of copies of $\GL_1$ and $\GL_2$, and $M^{\SO}$ is a special orthogonal group which is not the split $\SO_2$ over $\QQ$. Then $M'$ is either a direct product of $M^{\GL}$ with one special orthogonal group $S_0$ of the same parity and absolute rank as $M^{\SO}$, or a direct product of $M^{\GL}$ with two   special orthogonal groups $S_1, S_2$ of the same parity as $M^{\SO}$ whose absolute ranks add up to that of $M^{\SO}$. In both cases, none of $S_i$ is the split $\SO_2$ over $\QQ$ since $M'$ is an elliptic endoscopic group for $M$. 
	
	In the first case, $H$ is a special orthogonal group of the same parity and absolute rank as $G$. By Proposition \ref{prop:Tamagawa} we have $\tau(G) = \tau(H)$ and $\tau(M) = \tau(M')$. By Proposition \ref{prop:comp k} we have $k(G) = k(H)$ and $k(M) =k(M')$. The desired identity holds.
	
	In the second case, $H$ is a direct product of two  special orthogonal groups $H_1, H_2$ whose absolute ranks $m_1, m_2$ add up to that of $G$, and neither of the two is the split $\SO_2$ over $\QQ$ since $H$ is an elliptic endoscopic group for $G$. 
	 By Proposition \ref{prop:Tamagawa} and the multiplicativity of $\tau(\cdot)$ with respect to direct products, we have $\tau(G) = 2$, $\tau(H) = \tau(H_1) \tau(H_2) = 4$, and $$\tau(M') = \tau(S_1) \tau(S_2) \tau(M^{\GL}) = 4 \tau(M^{\GL}) = 2 \tau(M^{\SO}) \tau(M^{\GL}) = 2 \tau(M). $$ Hence the LHS of the desired identity is $1$. By Proposition \ref{prop:comp k} and the multiplicativity of $k(\cdot)$ with respect to direct products, we have $$k(G) = 2^{m-1} = 2 \cdot 2^{m_1 -1} 2^{m_2-1} = 2 k(H_1) k(H_2) = 2k(H),$$  and similarly $$k(M) = k(M^{\GL}) k(M^{\SO}) = 2k(M^{\GL}) k(S_1) k(S_2)  =2k(M'). $$ Hence the RHS of the desired identity is also $1$. 
	\end{proof}

\section{The simplified geometric side of the stable trace formula}\label{simplified geometric side} We recall the definition of the simplified geometric side of the stable trace formula, applicable to test functions which are stable cuspidal at infinity. This stems from Kottwitz's work in his unpublished notes. Our exposition follows \cite[\S 5.4]{morel2010book}.
More discussion on the relationship between the simplified geometric side given here and the ``usual'' stable trace formula appearing in Arthur's work is given in \S \ref{subsec:intro} below.
\begin{defn}
Let $M$ be a reductive group over $\RR$ containing elliptic maximal tori. Fix a Haar measure on $M(\RR)$. Let $\bar M$ be the inner form of $M$ over $\RR$ that is anisotropic modulo center (which exists by our assumption on $M$).  Define \index{$\bar v(\cdot)$}
$$ \bar v (M) : = e(\bar M) \vol (\bar M (\RR) / A_{M} (\RR) ^0 ) ,$$ where $e(\bar M)$ is the Kottwitz sign of $\bar M$, $\bar M(\RR)$ is equipped with the Haar measure transferred from that on $M(\RR)$, and $A_M(\RR)^0$ is equipped with the canonical Haar measure obtained by choosing an $\RR$-algebraic group isomorphism $\phi: A_M \isom \GG_m^n$ and pulling back the Lebesgue measure along the composite isomorphism $$\log \phi: A_M(\RR)^0 \xrightarrow{\phi} (\RR_{>0})^n \xrightarrow{(x_i)_i \mapsto (\log x_i)_i} \RR^n .$$ (This measure on $A_M(\RR)^0$ is indeed canonical since a different choice of $\phi$ would replace $\log \phi$ by $g\circ \log \phi$ for some $g\in \GL_n(\ZZ)$.) 
\end{defn} 

\begin{defn}\label{defn:SPhi}
Let $G$ be a reductive group over $\RR$. Fix a quasi-character $\nu : A_{G} (\RR) ^0 \to \CC^{\times }$. Let $M$ be a Levi subgroup of $G$ such that $M$ contains elliptic maximal tori (of $M$), and let $f\in C^{\infty} _c (G(\RR) , \nu^{-1})$ be a stable cuspidal function\index[n]{stable cuspidal function} (see \cite[\S 4]{arthurlef}, \cite[\S 5.4]{morel2010book}). For $\gamma \in M(\RR)$ semi-simple elliptic, we define \index{$S\Phi _M ^{G}$}
 $$ S\Phi _M ^{G}(\gamma, f) : = (-1) ^{\dim A_M} k(M ) k (G) ^{-1} \bar v (M_{\gamma} ^0) ^{-1} \sum _{\Pi} \Phi_M (\gamma^{-1} ,\Theta_{\Pi}) \Tr (f\mid  \Pi),$$ where $\Pi$ runs through the discrete series L-packets belonging to $\nu$, $\Theta_{\Pi}$\index{$\Theta_{\Pi}$} denotes the stable character associated to $\Pi$, and $\Phi_M(\cdot, \Theta_{\Pi})$ is the normalized stable discrete series character as in \S \ref{subsubsec:mathbb V}. This definition depends on the choices of a Haar measure on $M_{\gamma} ^0 (\RR) $ (used to define $\bar v(M_{\gamma} ^0)$) and a Haar measure on $G(\RR)$ (used to define $\Tr(f \mid  \Pi)$). 
\end{defn}

\begin{defn}\label{defn:pre geometric side}
Let $G$ be a reductive group over $\QQ$. Assume that $G$ is cuspidal  in the sense of Definition \ref{defn:cuspidal}. For $f = f^{\infty} f_{\infty} \in C^{\infty} _c (G(\adele)) $ with $f_{\infty} \in C^{\infty} _c (G(\RR) , \nu^{-1})$ stable cuspidal (where $\nu$ is a fixed quasi-character $ A_{G} (\RR) ^0 \to \CC^{\times }$), and for $M\subset G$ a Levi subgroup that is cuspidal,  define \index{$ST_M^{G}$}
$$ ST_M^{G}(f) : = 
\tau(M) \sum_{\gamma} \bar \iota ^M (\gamma) ^{-1} SO_{\gamma} (f_M^{\infty}) S\Phi ^{G}_M (\gamma, f_{\infty}) ,$$
where $\gamma$ runs through a set of representatives of the stable conjugacy classes of the $\RR$-elliptic semi-simple elements of $M(\QQ)$, and \index{$\bar \iota$}
$$\bar \iota ^M ({\gamma} ): = \abs{(M_{\gamma} / M_{\gamma}^0) (\QQ) }  .$$ For $M\subset G$ a Levi subgroup that is not cuspidal,  define $$ ST_M^{G}(f) : = 0.$$ We define  \index{$ST^{G}$}
$$ST^{G} (f) : = \sum_M( n_M^{G})^{-1} ST_M^{G} (f), $$ where $M$ runs through the Levi subgroups of $G$  up to $G(\QQ)$-conjugacy, and $n_M^{G}$ is as in Definition \ref{defn:n^G_M}.
\end{defn}
\begin{rem}
	We explain how the Haar measures are normalized in the definitions of $ST^G_M(f)$ and $ST^G(f)$  so that the results are independent of the Haar measures. For each $SO_{\gamma} (f_M^{\infty})$, we need  Haar measures on $M_{\gamma}^0(\adele_f)$ and $M(\adele_f)$ to define the stable orbital integral $SO_{\gamma}(\cdot)$, and need Haar measures on $M(\adele_f)$ and $G(\adele_f)$
	to define the constant term $f^{\infty}_M$. We assume that the two measures on $M(\adele_f)$ are the same. Then $SO_{\gamma} (f_M^{\infty})$ depends only on the Haar measures on $M_{\gamma}^0(\adele_f)$ and $G(\adele_f)$. Now in the definition of $S\Phi ^{G} _M (\gamma, f_{\infty})$, we need Haar measures on $M_{\gamma}^0(\RR)$ and $G(\RR)$ (cf.~Definition \ref{defn:SPhi}). We assume that the measures on $M_{\gamma}^0(\adele_f)$ and $M_{\gamma}^0(\RR)$ multiply to the Tamagawa measure on $M_{\gamma}^0 (\adele)$, and assume that the measures on $G(\adele_f)$ and $G(\RR)$ multiply to  the Tamagawa measure on $G(\adele)$.  Then  $ST^G_M(f)$ and $ST^G(f)$ are independent of the choices of Haar measures.
\end{rem}
\section{Test functions on endoscopic groups}\label{subsec:test functions}
\subsection{}\label{para:test function intro}
We now keep the notation and setting in \S \ref{subsubsec:setting for Morel's formula} and Theorem \ref{geometric assertion}. In particular, $G = \SO(V,q)$, where $(V,q)$ is a quadratic space over $\QQ$ of dimension $d \geq 5$, signature $(d-2,2)$, and discriminant $\delta \in \QQ^{\times}/\QQ^{\times,2}$.  Assume that the $G$-representation  $\mathbb V$ fixed in \S \ref{subsubsec:V and lambda} is absolutely irreducible. Fix a prime $p\notin \Sigma(\mathbf{O} (V), \mathbb V, \lambda, K,f^{\infty})$, and fix an integer $a \geq a_0(\mathbf O(V),\mathbb V, \lambda, K, f^{\infty}, p)$. 
Let $f^{p,\infty} $ be as in \S \ref{subsubsec:setting for Morel's formula}.

Let $\mathfrak e_{d^+, \delta^+, d^- , \delta^-} = (H,\lang H , s , \eta)$ be an elliptic endoscopic datum for $G= \SO(V)$, presented in the explicit form as in \S \ref{subsec:elliptic end data}. In the following we will always assume that $d^+ \geq  2$, or equivalently, that in the decomposition $H = H^+ \times H^- =  \SO(V^+) \times \SO(V^-)$ the factor $H^+$ is non-trivial. By \S \ref{para:e_p}, every isomorphism class in $\mathscr E(G)$ can be represented by such a datum. 

We follow \cite[\S 7]{kottwitzannarbor} to define a test function $f^H \in C^{\infty} _c (H(\adele))$. By definition, $f^H=0$ unless the following condition is satisfied:
\begin{enumerate}
	\item [($\dagger$)]   The $\RR$-group $H_{\RR}$ contains anisotropic maximal tori\footnote{In \cite[\S 7]{kottwitzannarbor}, the more general condition at the archimedean place is that elliptic maximal tori in $G_{\RR}$ should ``come from'' $H_{\RR}$. In our situation, since $G_{\RR}$ contains anisotropic maximal tori, the condition simplifies to the one in the text.}, and the $\QQ_p$-group $H_{\QQ_p}$ is unramified.
\end{enumerate}
Note that for our explicit representative $(H, \lang H, s,\eta)$, the  group $H_{\QQ_p}$ is unramified if and only if the localization of the endoscopic datum $(H,\lang H, s, \eta)$ over $\QQ_p$ is unramified. Also,  if $H_{\RR}$ contains  anisotropic maximal tori, then $H$ is cuspidal as a $\QQ$-group, and neither of $H^{\pm}_{\RR}$ is isomorphic to the split $\SO_2$ over $\RR$. It easily follows from the last condition that the localization of the (globally elliptic) endoscopic datum $(H,\lang H, s, \eta)$ over $\RR$ remains elliptic, as an endoscopic datum over $\RR$. Conversely, if $H$ is cuspidal, then since $A_H$ is trivial by the (global) ellipticity of $(H,\lang H, s, \eta)$, we know that $H_{\RR}$ contains anisotropic maximal tori. In conclusion, ($\dagger$) is equivalent to the following condition:
\begin{enumerate}
	\item [($\ddagger$)] The $\QQ$-group $H$ is cuspidal, and the $\QQ_p$-group $H_{\QQ_p}$ is unramified. 
\end{enumerate} Moreover, as we have seen, these conditions imply that the endoscopic datum $(H,\lang H, s,\eta)$ is elliptic over $\RR$ and unramified over $\QQ_p$.
In the following we assume that ($\dagger$) and ($\ddagger$) hold.

 By definition $f^H$ is of the form\index{$f^H$} $$f^H = f^H_{\infty} f^H _p f^{H, p,\infty} $$ with $f^H_{\infty} \in C^{\infty} _c (H(\RR))$ stable cuspidal, and $f^H_p \in C^{\infty}_c(H(\QQ_p)), f^{H,p,\infty} \in C^{\infty}_c(H(\adele_f^p))$. (As $Z_H^0$ is anisotropic over $\RR$ we do not need to specify central characters for the notion of stable cuspidal functions.)  

We fix a Haar measure on $H(\adele_f^p)$ arbitrarily, and fix the Haar measure on $H(\QQ_p)$ such that hyperspecial subgroups have volume $1$. Then there is a unique Haar measure on $H(\RR)$ such that the product measure on $H(\adele)$ is the Tamagawa measure. We fix this measure on $H(\RR)$ as well.

\subsection{} \label{once and for all}
 The definition of $f^H_{\infty}$ will depend on the choice of an auxiliary datum $(j, B_{G,H})$\index{$B_{G,H}$}, which we now specify. Here $j : T_H \isom T_G$ is an admissible isomorphism between anisotropic maximal tori $T_H \subset H_{\RR}$ and $T_G \subset G_{\RR}$; see \S \ref{subsec:admissible isom} for the notion of admissible isomorphisms over $\CC$, and note that any $\CC$-isomorphism $T_{H,\CC} \isom T_{G,\CC}$ is automatically defined over $\RR$ since both $T_H$ and $T_G$ are anisotropic over $\RR$.  The other part $B_{G,H}$ is a Borel subgroup of $G_{\CC}$ containing $T_{G,\CC}$; in other words, $(T_G, B_{G,H})$ is a fundamental pair in $G_{\RR}$.
 Later we shall also use the choice of $(j,B_{G,H})$ to normalize the archimedean transfer factors between $H$ and $G$. The dependence of $f^H_{\infty}$ on $(j, B_{G,H})$ is analogous to the dependence of a transfer of a function from $G$ to $H$ on the normalization of transfer factors. However this is only an analogy, as $f^H_{\infty}$ is not defined to be the transfer of a function on $G$.

 We now fix $(j, B_{G,H})$ once and for all in the following way. We let the fundamental pair $(T_G, B_{G,H})$ arise, in the way described in \S \ref{subsubsec:constrn with ED}, from an elliptic decomposition (Definition \ref{defn:ED}) $\mathcal D^{H} \in \ED(V_{\RR})$. Moreover, in the even case we assume that $\mathcal D^H$ gives rise to the orientation $o_V$ on $V_{\RR}$ fixed in \S \ref{subsubsec:defn of EDo}. In other words, $\mathcal D^H \in \ED(V_{\RR})^o$ in the notation of \S \ref{subsubsec:defn of EDo}.  As the notation suggests, \emph{we shall make possibly different choices of $\mathcal D^H$ for different $(H,\lang H, s, \eta)$}; a uniform choice is sometimes not possible because of some further conditions to be imposed in the following paragraph. Once $\mathcal D^H$ has been chosen, we choose $j$ as follows. Recall that $H$ is of the form $H= H^+ \times H^- = \SO(V^+) \times \SO(V^-)$. To define $j: T_H \isom T_G$, we choose an elliptic decomposition $\mathcal D_H = (\mathcal D_{H^+}, \mathcal D_{H^-})$ of $(V^+_{\RR}, V^-_{\RR})$ which should induce the fixed orientations on $V^{\pm}$ in the even case; in other words $\mathcal D_H \in \ED(V^+_{\RR})^o \times \ED(V^-_{\RR})^o$ in the notation of \S \ref{subsubsec:setting for para tori}. Then we define $j$ to be $j_{\mathcal D_H, \mathcal D^H}$ in the notation of \S \ref{subsubsec:setting for para tori}. By Lemma \ref{lem:j is adm}, this  $j$ is indeed an admissible isomorphism. 
  
  Now let us specify further conditions on $\mathcal D^H$. Since the signature of $V_{\RR}$ is $(d-2,2)$, we know that $\mathcal D^H$ involves exactly one negative definite plane as its member. In the odd case, we assume that $\mathcal D^H$ lies in $\ED(V_{\RR})^o_{\mathrm{nice}}$ as in Definition \ref{defn:EDNice}. This means that the unique negative definite member of $\mathcal D^H$ is the last member; cf.~Example \ref{eg:nice ED}. In the even case, unless $m= d/2$ is odd and $ d^+ =2$, we assume that $\mathcal D^H$ lies in $\ED(V_{\RR})^o_{\mathrm{nice}}$ as in Definitions \ref{defn:EDNice} and \ref{defn:EDNice, div by 4}, meaning that the unique negative definite member is the last (resp.~second last) member if $m$ is even (resp.~odd). If in the even case $m$ is odd and $d^+ = 2$, we assume that the unique negative definite member of $\mathcal D^H$ is the last member. In this case, $\mathcal D^H$ is not in $\ED(V_{\RR})^o_{\mathrm{nice}}$, but it  differs from an element thereof  by the transposition $(m-1, m) \in \mathfrak S_m$. 
 
 As long as $d$ is not $\equiv 2 \mod 4$, we can clearly choose $\mathcal D^H$ satisfying all the above conditions independently of $(H,\lang H, s, \eta)$. When $d \equiv 2 \mod 4$, we need to adjust the choice of $\mathcal D^H$ according to whether $d^+ =2$ or not. For instance, for all $(H, \lang H,s,\eta)$ with $d^+ \neq 2$ we may choose $\mathcal D^H$ to be some common $\mathcal D$, and then we may choose $\mathcal D^H$ for $d^+ = 2$ to be $(m-1, m) \cdot \mathcal D$, i.e., $\mathcal D$ with the last two members swapped. In particular, we see that in all cases, we may and shall arrange that $T_G$ is independent of $(H, \lang H, s, \eta)$, which justifies our notation.

 Since $\SO(V^+)$ is non-trivial, 
 our assumptions on $\mathcal D^H$ imply that the factor $\Uni(1)$ of $T_G$ corresponding to the unique negative definite member of $\mathcal D^H$ is sent under $j^{-1}$ into $\SO(V^+) \subset H$.   
    
 \subsection{} \label{subsubsec:defn of f_infty}
 The fixed choice of $(j,B_{G,H})$ determines a Borel subgroup $B_H$ of $H_{\CC}$ containing $T_{H,\CC}$, a subset $\Omega_*$\index{$\Omega_*$} of $\Omega = \Omega _{\CC}(G, T_{G})$, and a bijection induced by multiplication $$ \Omega_H \times \Omega_* \To \Omega$$ as follows. Here $\Omega_H : = \Omega_{\CC}(H, T_H)$ is viewed as a subgroup of $\Omega$ via $$\Omega_H \hookrightarrow \Aut (T_{H,\CC}) \underset{j}{\isom} \Aut (T_{G,\CC})  \supset\Omega.$$ 
 The Borel subgroup $B_H$ is characterized by the condition that the $B_H$-positive roots on $T_{H,\CC}$ are transported via $j$ to $B_{G,H}$-positive roots on $T_{G,\CC}$. (Note that $(T_H, B_H)$ is nothing but the fundamental pair in $H_\RR$ determined by $\mathcal D^H$ as in \S \ref{subsubsec:constrn with ED}, where $\mathcal D^H$ is as in \S \ref{once and for all}.)  The subset $\Omega_* \subset \Omega$ consists of those $\omega \in \Omega$ such that the $B_H$-positive roots on $T_{H,\CC}$ are transported via $j$ to $\omega B_{G,H}$-positive roots on $T_{G,\CC}$.

 Let $\mathbb V^*$\index{$\mathbb V^*$} be the contragredient representation of $\mathbb V$. Let $\varphi_{\mathbb V^*}$\index{$\varphi_{\mathbb V^*}$} be the discrete Langlands parameter of $G_{\RR}$ corresponding to $\mathbb V^*$, i.e., the L-packet of $\varphi_{\mathbb V^*}$ consists of discrete series representations of $G(\RR)$ having the same infinitesimal character as the $G(\CC)$-representation $\mathbb V^* \otimes_{\mathbb E} \CC$ (which is irreducible). Let $\Phi_H(\varphi_{\mathbb V^*})$\index{$\Phi_H(\varphi_{\mathbb V^*})$} be the set of equivalence classes of discrete Langlands parameters of $H_{\RR}$ that induce the equivalence class of $\varphi_{\mathbb V^*}$  via $\eta: \lang H \to \lang G$. As on \cite[p.~185]{kottwitzannarbor}, we have a bijection \index{$\omega_*(\cdot)$} $$ \omega_*(\cdot): \Phi_H(\varphi_{\mathbb V^*}) \isom \Omega_* , \quad   \varphi_H \longmapsto \omega_*(\varphi_H),$$ characterized by the condition that $\varphi_H$ is aligned\index[n]{aligned} with $(\omega_*(\varphi_H) ^{-1} \circ j, B_{G, H}, B_H)$ in the sense of \cite[p.~184]{kottwitzannarbor}. 
 
 For any $\varphi_H \in \Phi _H(\varphi_{\mathbb V^*})$, define
 \begin{align}\label{eq:defn of stable pseudo coeff}
f_{\varphi _H} :  =  d(H)^{-1} \sum_{\pi \in \Pi (\varphi_H) } f_{\pi} 
\in C^{\infty}_c(H(\RR)),  
 \end{align}
 where the terms are explained in the following. 
 \begin{itemize}
 	\item  The summation is over the discrete series representations $\pi$ of $H (\RR)$ inside the L-packet $\Pi(\varphi_H)$ of $\varphi_H$.
 	\item For each $\pi$, the function $f_{\pi}\in C^{\infty} _c (H(\RR))$\index{$f_{\pi}$} is a pseudo-coefficient\index[n]{pseudo-coefficient} for $\pi$; see \cite{clozeldelorme}. Note that this notion depends on the choice of a Haar measure on $H(\RR)$. We use the one fixed in \S \ref{para:test function intro}. 
 	\item We define $d(H)$\index{$d(H)$} to be the cardinality of $\Pi (\varphi_H)$. Note that this number is an invariant of $H_{\RR}$, equal to the cardinality of the complex Weyl group divided by the cardinality of the real Weyl group of an elliptic (i.e., anisotropic) maximal torus.
 \end{itemize}
The function $f_{\varphi_H}$ is stable cuspidal. Using this, we build the function $f^H_{\infty}$ in the following definition; cf.~\cite[p.~186]{kottwitzannarbor}, \cite[\S 6.2]{morel2010book}. 

 \begin{defn} \label{defn:test function at infinity}We define \index{$f^H_{\infty}$}
$$f^H_{\infty} : = (-1)^{q(G_{\RR})}\lprod{\mu_{T_G}, s}_j \sum_{ \varphi_H \in \Phi_H(\varphi_{\mathbb V^*})} \det(\omega_*(\varphi_H)) f_{\varphi_H} \in C^{\infty}_c(H(\RR)).$$ Here $\mu_{T_G} \in X_*(T_G)$ is the Hodge cocharacter of any $h$ in the Shimura datum $\X$ that \emph{factors through $T_G$}. The number $\lprod{\mu_{T_G}, s}_j$ is defined to be the image of $(j^{-1} \circ \mu_{T_G} , s)$ under the canonical pairing $$X_*(T_H) \times Z(\widehat H) \to \pi_1(H) \times Z(\widehat H) = X^{*} (Z(\widehat H)) \times Z(\widehat H) \to \CC^{\times}. $$ For each $\omega \in \Omega$, we write $\det(\omega)$\index{$\det(\omega)$} for the sign of $\omega$.\footnote{This is indeed equal to the determinant of $\omega$ acting on the finite free $\ZZ$-module $X_*(T_G)$, which explains the notation. In \S \ref{subsubsec:notation for roots and Weyl} the sign function is denoted by $\epsilon(\cdot)$, but in the current chapter we prefer the notation $\det(\cdot)$.} 
 \end{defn}
\begin{rem}
	By construction $f^H_{\infty}$ is stable cuspidal. 
\end{rem}
  \begin{lem}\label{lem:paring=1}
 We have 
 $\lprod{\mu_{T_G} ,s}_j = 1 $. 
\end{lem}
\begin{proof}
Using the observation made at the end of \S \ref{once and for all}, we compute that the image of $j^{-1} \circ \mu_{T_G} \in X_*(T_H)$ in $\pi_1(H) \cong \pi_1(H^+) \times \pi_1(H^-)$ has non-trivial projection in  $\pi_1(H^+) \cong \ZZ/2\ZZ$ and trivial projection in $\pi_1(H^-)$. We conclude the proof by recalling that $s$ has trivial component in $Z(\widehat{H^+})$. 
\end{proof}
 \subsection{}\label{normalizing the transf factors}
 
 We normalize the transfer factors between $(H, \lang H, s, \eta)$ and $G$ at various places as follows.
 
 We use the canonical unramified normalization associated to $K_p$ of the transfer factors at $p$  (see \cite{halesunram}), denoted by $(\Delta^G_H)_p$. Associated to the datum $(j, B_{G,H})$ fixed in \S \ref{once and for all}, we have Kottwitz's normalization \cite[\S 7]{kottwitzannarbor} for the transfer factors at $\infty$, which we denote by $\Delta_{j, B_{G,H}}$ (cf.~\S\S \ref{transf factors odd}--\ref{transf factors even div by 4}) and also by $(\Delta^G_H)_{\infty}$. We normalize the transfer factors away from $p$ and $\infty$ such that at almost all unramified places we have the canonical unramified normalization (associated to the hyperspecial subgroup determined by some reductive model of $G$ over $\ZZ[1/\Sigma]$ for some finite set $\Sigma$ of primes) and such that the global product formula with $(\Delta^G_H)_p$ and $(\Delta^G_H)_{\infty}$ is satisfied (see \cite[\S 6]{LS87}). For each place $v \notin \set{p,\infty}$, we denote our normalization by $(\Delta^G_H)_v$.\index{$(\Delta^G_H)_v$}
 
 We are now ready to give the definitions of the other two parts $f^{H,p,\infty}$ and $f^H_p$ in $f^H$. 
 
 \begin{defn}
 	Define $f^{H, p,\infty} \in C^{\infty}_c(H(\adele_f^p))$\index{$f^{H, p,\infty}$} to be a Langlands--Shelstad transfer of $f^{p,\infty}$ as in Theorem \ref{thm:LS} with respect to the Haar measure $dg^{p,\infty}$ on $G(\adele_f^p)$ fixed in \S \ref{subsubsec:setting for Morel's formula} and the Haar measure on $H(\adele_f^p)$ fixed in \S \ref{para:test function intro}. Here the transfer factors are normalized as in \S \ref{normalizing the transf factors}. 
 \end{defn}
\begin{defn}\label{defn:f^H_p}
	Let $\mu : \GG_m \to G_{\QQ_p}$ be a Hodge cocharacter of the Shimura datum $\mathbf O(V)$ defined over $\QQ_p$ (see \S \ref{para:Hodge cochar}). Let $f_{-\mu}$ be the element of $\mathcal H^{\ur}(G_{\QQ_{p^a}})$ associated to $- \mu$ as in Definition \ref{defn:canonical fn a la Kottwitz}. Let $f^H_p = b (f_{-\mu})$ be the image of $f_{-\mu}$ under the twisted transfer map $b: \mathcal H^{\ur} (G_{\QQ_{p^a}}) \to \mathcal H^{\ur} (H_{\QQ_p})$ as in \S \ref{subsubsec:twisted transfer}. We identify $f^H_p$ with a realization of it in $C^{\infty}_c(H(\QQ_p))$; see Remark \ref{rem:realization of Hecke} below. 
 \end{defn}

 \begin{rem}\label{rem:realization of Hecke}
 	Once an element $f^H_p \in \mathcal H^{\ur} (H)$ is specified, it still corresponds ambiguously to different functions on $H(\QQ_p)$. Namely, for each choice of a hyperspecial subgroup $K_{H,p}$ of $H(\QQ_p)$ there is a corresponding $K_{H,p}$-bi-invariant function in $\mathcal H(H(\QQ_p)\sslash K_{H,p})$. These functions have the same stable orbital integrals, as noted in \cite[\S 7]{kottwitzannarbor}. Indeed, as we saw in \S \ref{para:transition maps for Hecke}, these functions are related to each other under pull-back by inner automorphisms of $H_{\QQ_p}$, and these automorphisms do not permute the stable conjugacy classes. The same remark applies to the various canonical constant terms (see Proposition \ref{canonical constant term}) $(f^H_p)_{M'} \in \mathcal H^{\ur} (M')$ for Levi subgroups $M'$ of $H$ defined over $\QQ_p$. It follows that the evaluation of $ST^{H}$ (Definition \ref{defn:pre geometric side}) at the test function $f^H = f^H_{\infty} f^H_p f^{H,p,\infty}$ is unaffected by the ambiguity in $f^H_p$. 	
 \end{rem}
 \begin{rem}
 	The function $f^H$ depends on $a$ via the component $f^H_p$. 
 \end{rem}
 
 \subsection{}\label{para:normalizing tf for M}
 Now suppose $M$ is a standard proper Levi subgroup of $G$ (i.e., one of $M_1, M_2, M_{12}$ as in \S \ref{global groups}) and consider a bi-elliptic endoscopic $G$-datum for $M$ 
 $$\mathfrak e_{A,B, \mathfrak p} = \mathfrak e_{A,B, d^+,\delta^+, d^-, \delta^-} = (M', \lang M', s_M, \eta_M)$$ presented in the explicit form as in \S \ref{para:presentation of endoscopic G-data}. More precisely, the construction in \S \ref{para:presentation of endoscopic G-data} depends on the choice of a hyperbolic basis as in \S \ref{subsubsec:setting for G-endosc}. Thus we need to fix a hyperbolic basis of $W_1^{\perp} = V_1 \oplus V /V_1^{\perp}$ (resp.~$W_2^{\perp} = V_2 \oplus  V/V_2^{\perp}$) when $M \in \set{M_2, M_{12} }$ (resp.~$M = M_{1}$). We always take the hyperbolic basis $\set{e_1, e_1'}$ of $W_1^{\perp}$ and the hyperbolic basis $\set{e_1, e_2, e_2', e_1'}$ of $W_2^{\perp}$, where $e_i, e_i'$ are as in \S \ref{para:e_1}.
 
 As in \S \ref{para:presentation of endoscopic G-data} and Proposition \ref{prop:odd classification of G-endoscopy},  $\mathfrak e_{A,B, \mathfrak p}$ induces the endoscopic datum $$e_{\mathfrak p}(M) = (M', \lang M', s_M', \eta_M)$$ for $M$, and the endoscopic datum  $$ \mathfrak e_{d^+ + 2\abs{A} +4 \abs{B} , \delta^+,  d^- + 2\abs{A^c} + 4 \abs{B^c}, \delta^-} = (H, \lang H, s, \eta)$$ for $G$. Moreover, recall that we have fixed in \S \ref{para:two maps from endoscopic G data} an $H(\QQ)$-conjugacy class of embeddings $M' \hookrightarrow H$ with images Levi subgroups, and in particular we have the diagram (\ref{eq:diag with M'}) commuting up to $\widehat{G}$-conjugation. We now fix such an embedding $M' \hookrightarrow H$ on the nose. 
 
 We assume that $H$ satisfies condition ($\dagger$) in \S  \ref{para:test function intro}. It follows that $M'$ is unramified at $p$, and the endoscopic datum $(M', \lang M', s_M', \eta_M)$ for $M$ is unramified at $p$. Also we assume that the parameter $\mathfrak p$ is such that the component of $s_M$ in $\widehat{M^{\SO}}$ is not $-1$, from which it follows that $H^+$ is non-trivial. Thus the preceding discussion in this section can be applied to $(H,\lang H, s, \eta)$. We normalize the transfer factors between $(M', \lang M', s_M', \eta_M)$ and $M$ as follows.
 
  Away from $p$ and $\infty$, we normalize the transfer factors by inheriting the normalization between $(H, \lang H, s, \eta)$ and $G$ fixed in \S \ref{normalizing the transf factors}, with respect to our fixed  embedding  $M' \hookrightarrow H$; see Remark \ref{rem:Levi inherit} below. At $p$, we use the canonical unramified normalization associated to the hyperspecial subgroup  of $M(\QQ_p)$ determined by $K_p$ (i.e., the image of $P(\QQ_p) \cap K_p$ under $P \to M$, where $P\subset G$ is the standard parabolic subgroup such that $M = M_P$; cf.~Remark \ref{rem:why hyp}), which is also the same as the normalization inherited from the canonical unramified normalization between $(H, \lang H, s, \eta)$ and $G$ associated to $K_p$. For later reference, for each finite place $v$, we denote the above-mentioned normalization by $(\Delta^M_{M'})^{A,B}_v$ or simply $(\Delta^M_{M'})^{A,B}$.\index{$(\Delta^M_{M'})^{A,B}_v$} We denote the above-mentioned hyperspecial subgroup of $M(\QQ_p)$ by $\mathcal M(\ZZ_p)$.\index{$\mathcal M(\ZZ_p)$} At $\infty$, we do not yet fix a normalization. In fact, precise knowledge about signs between different normalizations in this case is key to our later computation; this will be investigated in \S \ref{pf:4} below.  
 \begin{rem}\label{rem:Levi inherit}
 	At each place $v$ of $\QQ$, there is a notion of the normalization of the transfer factors between $(M', \lang M', s_M',\eta_M)$ and $M$ \emph{inherited}\index[n]{inherited normalization of transfer factors} from the normalization of the transfer factors between $(H, \lang H, s, \eta)$ and $G$ with respect to our fixed $M' \hookrightarrow H$. It is described via a simple formula as in \cite[\S 5.2]{morel2010book} or \cite[\S 5.1]{morel2011suite}. Roughly speaking, this means that apart from the difference in $\Delta_{IV}$, the transfer factor between $M'$ and $M$ is equal to the transfer factor between $H$ and $G$ for any $G$-regular element of $M'(\QQ_v) \subset H(\QQ_v)$ and any preimage of it in $M(\QQ_v) \subset G(\QQ_v)$. Here it is crucial that the diagram (\ref{eq:diag with M'}) commutes up to $\widehat{G}$-conjugacy.  
 	
 	An important property is that if the normalizations between $H$ and $G$ at all places satisfy the global product formula, then so do the inherited normalizations between $M'$ and $M$ at all places; this is due to the fact that our choice of $M' \hookrightarrow H$ is global. To see this, one simply notes that the term $\Delta_{IV}$ can be ignored from the definition of transfer factor when deciding whether local normalizations satisfy the global product formula.

 	We now say a few words on the proof of the existence of the inherited normalization. The original source is Kottwitz's unpublished notes, where this result is marked as an easy consequence of the definition of transfer factors in \cite{LS87}. Indeed it can be proved similarly as \cite[Lem.~9.2]{halesunram}. Alternatively, in our particular situation, one can prove this without much difficulty using the explicit formulas for the transfer factors in \cite{walds10}. 
  \end{rem}

\begin{prop}\label{prop:LS for Levi}
Keep the setting of \S \ref{para:normalizing tf for M}. The function $(f^{H,p,\infty})_{M'} \in C^{\infty} _c (M'(\adele_f^p)) $ is a Langlands--Shelstad transfer of $ (f^{p,\infty} ) _{M}\in C^{\infty} _c (M(\adele_f^p))$ in the sense of Theorem \ref{thm:LS}, with respect to the normalization of  transfer factors $(\Delta^M_{M'})_{A,B}$ as in \S \ref{para:normalizing tf for M}.
\end{prop}
\begin{proof} In view of the Fundamental Lemma we can pass to a local setting over some $\QQ_v$ (with $v \neq p,\infty$) instead of the adelic setting. 
	The statement can then be proved similarly as \cite[Lem.~6.3.4]{morel2010book}, with the following two modifications.
	
	Firstly, we replace $G_{\gamma}$ and $M_{\gamma}$ by $G_{\gamma}^0$ and $M_{\gamma}^0$ in the proof of part (i) of \textit{loc.~cit.}.

Secondly, in the proof of part (ii) of \textit{loc.~cit.}, Morel cites \cite[Lem.~2.4.A]{LS90} in order to reduce the proof to checking the matching of orbital integrals for those $\gamma' \in M'(\QQ_v)_{\mathrm{ss}}$ that are $M$-regular, or even $G$-regular (meaning that all matching elements of $M(\overline \QQ_v)_{\mathrm{ss}}$ are $G$-regular).\footnote{Note the following typo: In the second line of the second paragraph of the proof of \cite[Lem.~6.3.4]{morel2010book}, ``regular in $\mathbf H$'' should be ``regular in $\mathbf  M$''.} Since $M_{\der}$ is not simply connected in our case, we cannot directly apply \cite[Lem.~2.4.A]{LS90}, but this can be circumvented by the following argument. To simplify notation, we understand that all reductive groups and endoscopic data are over $\QQ_v$. Suppose we have already established that $\phi \in C^{\infty}_c(M(\QQ_v))$ and $\phi' \in C^{\infty}_c(M'(\QQ_v))$ have matching orbital integrals for all $G$-regular $\gamma' \in M'(\QQ_v)_{\mathrm{ss}}$, and want to deduce the same for all $(M, M')$-regular  $\gamma' \in M'(\QQ_v)_{\mathrm{ss}}$. As in \S \ref{para:expl of LS}, we pick a $z$-extension $1 \to Z \to M_1 \to M \to 1$, and obtain from it a central extension $1 \to Z \to M'_1 \to M' \to 1$ as well as an endoscopic datum $(M'_1, \lang M'_1, s'_{M_1}, \eta_{M_1})$ for $M_1$  such that the diagram analogous to (\ref{diag:z-ext})  commutes. As in Remark \ref{rem:on using LS}, we identify $\phi$ with a function $\phi_1 \in C^{\infty}_c(M_1(\QQ_v), 1_Z)$, and identify $\phi'$ with a function $\phi'_1 \in C^{\infty}_c(M'_1(\QQ_v), 1_Z)$, where in both cases $1_Z$ denotes the trivial character on $Z$. We say that an element of $M'_1(\QQ_v)_{\mathrm{ss}}$ is $G$-regular if all the matching elements of $M_1(\overline \QQ_v)_{\mathrm{ss}}$ are preimages of $G$-regular elements of $M(\overline \QQ_v)_{\mathrm{ss}}$. Then $\phi_1$ and $\phi_1'$ have matching orbital integrals for all $G$-regular elements of $M_1'(\QQ_v)_{\mathrm{ss}}$. Now note that for any maximal torus $T \subset M_1'$, there is a dense subset of $T(\QQ_v)$ consisting of $G$-regular elements. By this and the proof of  \cite[Lem.~2.4.A]{LS90}, $\phi_1$ and $\phi_1'$ have matching orbital integrals for all $(M_1, M_1')$-regular elements of $M_1'(\QQ_v)_{\mathrm{ss}}$. It follows that $\phi$ and $\phi'$ have matching orbital integrals for all $(M, M')$-regular elements of $M'(\QQ_v)_{\mathrm{ss}}$, as desired.
\end{proof}

\section{Statement of the main computation}\label{subsec:statement of main}
\subsection{}\label{para:prepare for main}
 	Let $M$ be a standard proper Levi subgroup of $G$. Define \index{$\Tr'_M$}
 	\begin{align}\label{defn:Tr'}
\Tr_M' = (n^G_M) ^{-1} \sum_{\substack{\mathfrak e_{A,B, \mathfrak p} = (M', \lang M' , s_M, \eta_M) \\ \in \dot {\mathscr E}_G(M)} } \abs{\Out _G ( \mathfrak e_{A,B, \mathfrak p} )} ^{-1} \tau(G) \tau(H)^{-1} ST ^H_{M'}  (f^H).
 	\end{align}
 Here the summation is over a subset $\dot {\mathscr E}_G(M)$\index{$\dot {\mathscr E}_G(M)$} of the set of explicitly presented bi-elliptic endoscopic $G$-data for $M$ as in \S \ref{para:presentation of endoscopic G-data} (in other words, $\dot {\mathscr E}_G(M)$ is a subset of the parameter set $\mathscr P_{r,t} \times' \mathscr P_W = \set{(A,B, \mathfrak p)}$ in the notation of \S \ref{para:presentation of endoscopic G-data}) such that  the component of $s_M$ in $\widehat{M^{\SO}}$ is not $-1$  and such that each isomorphism class in $\mathscr E_G(M)$ is represented exactly once. (Clearly the two conditions can be simultaneously met.)  For each $(M', \lang M', s_M, \eta_M) \in \dot {\mathscr E}_G(M)$, we let $(H,\lang H, s, \eta)$ be the induced endoscopic datum for $G$. More precisely, for $(M', \lang M', s_M, \eta_M) = \mathfrak e _{A, B, d^+,\delta^+, d^-,\delta^-}$, we let $$ (H, \lang H, s, \eta) : = \mathfrak e_{d^+ + 2\abs{A} +4 \abs{B} , \delta^+,  d^- + 2\abs{A^c} + 4 \abs{B^c}, \delta^-}$$ as in Proposition \ref{prop:odd classification of G-endoscopy}. Note that $H^+$ is non-trivial by our assumption on $s_M$. The function $f^H$ is defined in \S \ref{subsec:test functions}. We fix $M' \hookrightarrow H$ as in \S \ref{para:normalizing tf for M} so as to view $M'$ as a Levi subgroup of $H$, and define $ST^H_{M'} (f^H)$ as in Definition \ref{defn:pre geometric side}. 
 
Note that our definition of $\Tr_M'$ is independent of the choice of $\dot {\mathscr E}_G(M)$. Indeed, one directly checks that the summand associated to $(A, B ,\mathfrak p)$ is equal to that associated to $(A^c, B^c, \swap(\mathfrak p))$ (in the case where both parameters satisfy the condition on $s_M$ imposed before). Hence such a summand depends only on the isomorphism class of $\mathfrak e_{A,B,\mathfrak p}$ in $\mathscr E_G(M)$.

Recall that the definition of $f^H$ depends on the fixed integer $$a \geq a_0(\mathbf O(V),\mathbb V, \lambda, K, f^{\infty}, p). $$ Clearly the definitions of both $f^H$ and $\Tr'_M$ make sense for all integers $a \geq 1$. We shall henceforth view $\Tr'_M$ as a function in $a\in \ZZ_{\geq 1}$. 
On the other hand, we have $\Tr_M (f^{p,\infty} dg^{p,\infty}, K,  a )$ as in Definition \ref{Defn Tr12}. We abbreviate it as $\Tr_M$\index{$\Tr_M$}, and also view  it as a function in $a\in \ZZ_{\geq 1}$.

\begin{thm}\label{thm:maincomp}
For all large enough $a$ we have $\Tr_M= \Tr_{M}' $. 
\end{thm} 
\subsection{}\label{para:proof in trivial case}
Note that in the even case and for $M= M_2$, we have $\Tr_M =0$ since $(M_l)_{\RR}$ does not contain elliptic maximal tori (see Remark \ref{rem:Tr_M = 0 for even M_2}). In this case, we also know that each $M'$ appearing in (\ref{defn:Tr'}) is non-cuspidal, and hence $ST^H_{M'} \equiv 0$. Indeed, recall that $M' = M^{\GL} \times M^{\prime ,\SO}$, where $M^{\prime , \SO}$ is the group in the elliptic endoscopic datum $\mathfrak e_{d^+, \delta^+, d^-, \delta^-} (W_1)$ for $M_2^{\SO} = \SO(W_1)$. This ellipticity, together with the fact that $M_2^{\SO}$ is not the split $\SO_2$ over $\QQ$, implies that neither of $(d^{\pm},\delta^{\pm})$ is $(2,1)$ in $\ZZ_{\geq 0} \times (\QQ^{\times}/\QQ^{\times,2})$. Hence if $M'$ is cuspidal, then $(M^{\prime , \SO})_{\RR}$ must contain anisotropic maximal tori, and so as in \S \ref{para:cusp crit} we have $\delta^{\pm} = (-1)^{d^\pm/2}$ in $\RR^{\times}/\RR^{\times,2}$, from which $\delta = (-1)^{(d^+ + d^-)/2} = (-1)^{d/2 -1}$ in $\RR^{\times}/\RR^{\times,2}$, contradicting with the fact that $\delta = (-1)^{d/2}$ in $\RR^{\times}/\RR^{\times,2}$. Thus in the even case with $M = M_2$ we have already proved the theorem.
The proof of the theorem in the remaining cases occupies \S\S \ref{pf:1}--\ref{pf:main}.

\section{First simplifications}
\label{pf:1}
\subsection{}\label{para:dot E for M} We keep the setting of \S \ref{para:prepare for main}, and assume that we are not in the even case with $M = M_2$, since in that case Theorem \ref{thm:maincomp}  is already proved. As in \S \ref{para:e_1}, we have $M = \GG_m ^{r} \times \GL_2^{t} \times \SO(W)$ for some $r \in \set{0,1,2}, t \in \set{0,1}, W \in \set{W_1, W_2}$.
Denote by $\mathscr E(M)^{c, \ur}$\index{$\mathscr E(M)^{c,\ur}$} the subset of $\mathscr E(M)$ consisting of isomorphism classes of endoscopic data whose groups $M'$ are cuspidal and unramified over $\QQ_p$. For each isomorphism class in $\mathscr E(M)^{c, \ur}$, we fix a representative of the form $\mathfrak e_{\mathfrak p} (M)$ for some $\mathfrak p \in \mathscr P_W$, where the notation is as in Definitions \ref{defn:mathscr P_V}, \ref{defn:P_W}, and \S \ref{para:presentation of endoscopic G-data}. Thus there are \textit{a priori} up to two choices of $\mathfrak p$ for each isomorphism class, and we fix one choice. We may and shall also assume that each choice $\mathfrak p = (d^+,\delta^+, d^-,\delta^-)$ satisfies  $d^+ \geq 2$. 
In the following, we denote this set of representatives by $\dot {\mathscr E}(M)^{c,\ur}$.\index{$\dot {\mathscr E}(M)^{c,\ur}$}

Note that $M^{\SO}$ is never isomorphic to the split $\SO_2$ over $\QQ$. Hence the same argument as in \S \ref{para:test function intro} shows that every element $\fke_{\fkp}(M)$ of $\dot {\mathscr E}(M)^{c,\ur}$ satisfies the following conditions:
\begin{enumerate}
	\item As in \S \ref{para:presentation of endoscopic G-data}, write the group in $\fke_{\fkp}(M)$ as $M' = M^{\GL} \times M^{\prime,\SO}$. Then the $\RR$-group $(M^{\prime,\SO})_{\RR}$ contains anisotropic maximal tori.
	\item The localization of $\fke _{\fkp}(M)$ over $\RR$ is still elliptic as an endoscopic datum over $\RR$. 
\end{enumerate}

\begin{lem}\label{lem:first simplification} We have
	$$n^G_M \Tr_M' = \sum_{ \mathfrak e_{\mathfrak p}(M) \in \dot {\mathscr E}(M)^{c,\ur}} \abs{\Out_M (\mathfrak e_{\mathfrak p}(M))} ^{-1} \sum_{A, B} \tau(G)\tau(H) ^{-1} ST^H_{M'} (f^H) .  $$
	Here the second summation is over the following ranges:
	\begin{itemize}
		\item	In the odd case for $M = M_{12}$, we have $A \in \set{\emptyset, \set{1}, \set{2}, \set{1,2}}, B \in \set{\emptyset}$.
		\item In the even case for $M=M_{12}$, we have $A \in \set{\emptyset,\set{1,2}}, B \in \set{\emptyset}$. 
		\item For $M = M_1$, we have $A \in \set{\emptyset}, B \in  \set{\emptyset,\set{1}}$.
		\item In the odd case for $M = M_2$, we have $A\in \set{\emptyset, \set{1}}, B \in \set{\emptyset}$.  
	\end{itemize} For each triple $(e_{\mathfrak p}(M) = \mathfrak e_{d^+,\delta^+, d^-,\delta^-} (M), A, B)$ appearing in the summation, we set 
$$   (H, \lang H, s, \eta) : = \mathfrak e_{d^+ + 2\abs{A} +4 \abs{B} , \delta^+,  d^- + 2\abs{A^c} + 4 \abs{B^c}, \delta^-},  $$  write $M'$ for the group in $\mathfrak e_{\mathfrak p}(M)$, and as in \S \ref{para:prepare for main} identify $M'$ with a Levi subgroup of $H$ so as to define $ST^H_{M'}(f^H)$ . 
\end{lem}
\begin{proof} 
	We first note that the formula $\mathfrak e_{d^+ + 2\abs{A} +4 \abs{B} , \delta^+,  d^- + 2\abs{A^c} + 4 \abs{B^c}, \delta^-}$ indeed gives an elliptic endoscopic datum for $G$, i.e., neither of $(d^+ + 2\abs{A} +4 \abs{B} , \delta^+) $ and $ (d^- + 2\abs{A^c} + 4 \abs{B^c}, \delta^-)$ is equal to $(2,1) \in \ZZ_{\geq 0} \times (\QQ^{\times}/\QQ^{\times,2})$. Indeed, since $M^{\SO}$ is not the split $\SO_2$ over $\QQ$, we know that neither of $(d^\pm, \delta^{\pm})$ is equal to $(2,1)$, which immediately implies our assertion. Also, we have $d^+ + 2\abs{A} +4 \abs{B} \geq 2$ since we have already assumed that $d^+ \geq 2$ in \S \ref{para:dot E for M}. Thus $ST^H_{M'}(f^H)$ in the lemma is indeed defined.

	It is clear from the definitions that if a term $ST^H_{M'}(f^H)$ on the RHS of (\ref{defn:Tr'}) is non-zero, then $H$ is cuspidal and unramified over $\QQ_p$ (since otherwise $f^H =0$), and $M'$ is cuspidal (since otherwise $ST^H_{M'} \equiv 0$). 
Clearly the condition that $H$ is unramified over $\QQ_p$ is equivalent to the condition that $M'$ is unramified over $\QQ_p$. In the odd case, the cuspidality conditions are automatic. In the even case, suppose we have $(M', \lang M', s_M, \eta_M) = \mathfrak e _{A, B, d^+,\delta^+, d^-,\delta^-} \in \dot {\mathscr E_G}(M)$ such that $M'$ is cuspidal. Then as we have mentioned in \S \ref{para:dot E for M}, $(M^{\prime ,\SO})_{\RR}$ contains anisotropic maximal tori, so by the same argument as in \S \ref{para:proof in trivial case} we have $\delta^{\pm} = (-1)^{d^{\pm}/2}$ in $\RR^{\times}/\RR^{\times,2}$.  On the other hand the condition that $H$ is cuspidal is equivalent to $H_{\RR}$ having anisotropic maximal tori by the discussion in \S \ref{para:test function intro}, and hence is equivalent to the conditions that $\delta^{+} = (-1)^{d^+/2 + \abs{A} + 2 \abs{B}} $
and that $\delta^- = (-1)^{d^-/2 + \abs{A^c} + 2 \abs{B^c}}$ in $\RR^{\times}/\RR^{\times,2}$. Thus given that $M'$ is cuspidal and given that $d$ is even, $H$ is cuspidal if and only if $\abs{A}$ and $\abs{A^c}$ are both even. 	
	
The above discussion shows that in (\ref{defn:Tr'}),  	we can replace the summation set $\dot {\mathscr E}_G(M)$ by the subset $\dot {\mathscr E}_G(M)^{c,\ur}$ consisting of elements $\fke_{A,B,\fkp} = (M', \lang M', s_M,\eta_M)$ such that $M'$ is cuspidal and unramified over $\QQ_p$, and such that $\abs{A}$ and $\abs{A^c}$ are even in case $d$ is even. Thus up to re-choosing  $\dot {\mathscr E_G}(M)$ (which does not affect the definition of $\Tr'_M$), we may assume that  whenever $\mathfrak e_{A,B,\mathfrak p} \in \dot {\mathscr E_G}(M)^{c,\ur}$ we have $\mathfrak e_{\mathfrak p}(M ) \in \dot {\mathscr E}(M)^{c,\ur}$. We thus have a well-defined map $F: \dot {\mathscr E_G}(M)^{c,\ur} \to \dot {\mathscr E}(M)^{c,\ur}$ sending each $\mathfrak e_{A,B,\mathfrak p}$ to $\mathfrak e_{\mathfrak p}(M)$. For each $\mathfrak e_{\mathfrak p}(M) \in \dot {\mathscr E}(M)^{c,\ur}$, we let $\Gamma(\mathfrak p)$ denote the set of $(A,B)$ as in the summation range in the current lemma. We divide our analysis into two different cases.

	\textbf{Case 1.} Suppose $\mathfrak p  = (d^+,\delta^+,d^-,\delta^-)$ with $(d^+,\delta^+) \neq (d^-,\delta^-)$. Then one checks that $F^{-1} (\mathfrak e_{\mathfrak p}(M)) = \set{\mathfrak e_{A,B,\mathfrak p}\mid (A,B) \in \Gamma(\mathfrak p)}$. Moreover, for each $\mathfrak e_{A,B,\mathfrak p} \in F^{-1} (\mathfrak e_{\mathfrak p}(M))$, we have $\abs{\Out_G(\mathfrak e_{A,B,\mathfrak p})} = \abs{\Out_M(\mathfrak e_{\mathfrak p}(M))}$. (See \S\S \ref{subsubsec:out of endos} and \ref{para:Out_G} for the computation of these two groups.) Thus the summand indexed by $\mathfrak e_{\mathfrak p}(M)$ in the current lemma is equal to the sum over all $\mathfrak e_{A,B,\mathfrak p} \in F^{-1} (\fke_{\fkp}(M) )$ in (\ref{defn:Tr'}). 
	
	\textbf{Case 2.} Suppose $\mathfrak p  = (d^+,\delta^+,d^-,\delta^-)$ with $(d^+,\delta^+)  = (d^-,\delta^-)$. Then 
	\begin{align*}
\set{\fke_{A,B,\fkp} \mid (A,B) \in \Gamma(\fkp)}   =  F^{-1} (\fke_{\fkp}(M)) \sqcup   \set{  \fke_{A^c,B^c,\fkp}   \mid  \fke_{A,B,\fkp} \in F^{-1} (\fke_{\fkp}(M))} .
	\end{align*} (The union is disjoint.)
Moreover, for each $(A,B) \in \Gamma(\fkp)$, we have $\abs{\Out_M(\fke_{\fkp}(M) )} = 2 \abs{\Out_G(\fke_{A,B,\fkp})}$, and we know that the summand $\tau(G) \tau(H)^{-1} ST^H_{M'} (f^H)$  indexed by $(A,B)$ in the current lemma is equal to the term $\tau(G) \tau(H)^{-1} ST^H_{M'} (f^H)$ in (\ref{defn:Tr'}) arising from either $\fke _{A,B,\fkp}$ or $\fke_{A^c,B^c, \fkp}$, whichever lies in $\dot{\mathscr E}_G(M)$. Thus we again see that the summand indexed by $\mathfrak e_{\mathfrak p}(M)$ in the current lemma is equal to the sum over all $\mathfrak e_{A,B,\mathfrak p} \in F^{-1} (\fke_{\fkp}(M) )$ in (\ref{defn:Tr'}). The proof of the lemma is complete.
\end{proof}

 \section[Expanding the geometric side of the stable trace formula]{Expanding the simplified geometric side of the stable trace formula}
\label{pf:2} Let $(\fke_{\fkp}(M), A, B) $ be a summation index as in Lemma  \ref{lem:first simplification}. We study the term $ST^H_{M'} (f^H)$ arising from this index.  

\begin{defn}\label{defn:Sigma}
	Let $\Sigma(M')$\index{$\Sigma(M')$} be a set of representatives in $M'(\QQ)$ of the stable conjugacy classes in $M'(\QQ)$ that are $\RR$-elliptic. 
\end{defn} 
 \begin{lem}\label{lem:only regular}
 We have an expansion
\begin{align}\label{expand by defn}
	ST_{M'}^{H}(f^H) = 
	\tau(M') \sum_{\gamma'\in \Sigma(M')} \bar \iota ^{M'} (\gamma') ^{-1} SO_{\gamma'} (f_{M'}^{H,\infty}) S\Phi ^{H}_{M'} (\gamma', f^H_{\infty}).
\end{align}  
Here $f^{H,\infty} : = f^{H,p,\infty} f^H_p$. Moreover, in (\ref{expand by defn}), only those $\gamma'$ that are $(M,M')$-regular contribute non-trivially.
 \end{lem}
 \begin{proof} The first statement follows from the definitions. To show the second statement, suppose $\gamma' \in \Sigma(M')$ is not $(M,M')$-regular. We show that $S\Phi ^{H}_{M'} (\gamma', f^H_{\infty})$ already vanishes. For this, it suffices to show the vanishing of $$\sum _{\Pi } \Phi^H_{M'} (\gamma'^{-1} ,\Theta_{\Pi}) \Tr (f^H_{\infty} \mid  \Pi),$$ where the summation is over the discrete series L-packets $\Pi$ for $H_{\RR}$. For this it suffices to show the vanishing of $$\sum_{\varphi_H \in \Phi_H (\varphi _{\mathbb V^*}) } \det (\omega_* (\varphi_H)) \Phi ^H_{M'} ({\gamma'} ^{-1}, \Theta_{\Pi(\varphi_H)}).$$ 
 	By \cite[Prop.~3.2.5, Rem.~3.2.6]{morel2011suite}, the above quantity is zero provided that $\gamma'$ is not $(M, M')$-regular. 
 \end{proof}
  \subsection{}\label{para:transfer away from p and infty}
We continue the study of (\ref{expand by defn}). By Lemma \ref{lem:only regular}, we only need to sum over those $\gamma'\in \Sigma(M')$ that are $(M,M')$-regular. By Proposition \ref{prop:LS for Levi}, we may further restrict to those $\gamma'$ that is an image of a semi-simple element $\gamma_M \in M(\adele_f^p)$, and in this case we have 
\begin{align}\label{eq:Levi transfer away from p infty}
SO_{{\gamma'}} (f_{M'} ^{H, p,\infty}) =  ( \Delta^{M}_{M'})^{A,B} (\gamma',\gamma_M) O_{\gamma_M}^{s_M'} (f^{p,\infty} _M),
\end{align}
 where $s_M'$ is given by the endoscopic datum $\fke_{\fkp}(M ) = (M', \lang M', s_M', \eta_M)$ for $M$, and $(\Delta^{M}_{M'})^{A,B} (\gamma',\gamma_M)$ denotes the product of the local transfer factors over finite places $v \neq p$, normalized as in \S \ref{para:normalizing tf for M}. We remind the reader that $s_M'$ is different from $s_M$ as in $\fke_{\fkp, A,B} = (M',\lang M', s_M, \eta_M)$, and $s_M'$ is independent of $(A,B)$. By contrast, the normalization $(\Delta^{M}_{M'})_v^{A,B}$ of transfer factors between $M'$ and $M$ at $v$ depend on $(A,B)$. Nevertheless, for almost all $v$, $(\Delta^{M}_{M'})^{A,B}_v$ is the canonical unramified normalization (associated to the hyperspecial subgroup determined by some reductive model of $M$ over some Zariski open in $\spec \ZZ$). Hence for almost all $v$, $(\Delta^{M}_{M'})^{A,B}_v$ is  independent of $(A,B)$. 
 \begin{defn}\label{defn:epsilon^{p,A}} For each $v \neq p,\infty$, let $\epsilon_v (A,B) \in \CC^{\times}$ be the constant such that $(\Delta^M_{M'})^{A,B}_v = \epsilon_v(A,B) (\Delta^M_{M'})^{\emptyset,\emptyset}_v$.  
 	Let   \index{$\epsilon^{p,\infty}(A,B)$} 
$$ \epsilon^{p,\infty} (A,B) = \prod_{v \neq p,\infty} \epsilon_v (A,B),$$ where almost all terms in the product are $1$.  
 \end{defn}
 \begin{defn}\label{defn:Sigma_1}
 Let $\Sigma(M')_{1}$\index{$\Sigma(M')_{1}$} be the set of $\gamma'\in \Sigma(M')$ such that $\gamma'$ is $(M, M')$-regular and is an image of a semi-simple element of $ M(\adele_f^p)$. For each $\gamma'\in\Sigma(M')_1$, let $\gamma_M \in M(\adele_f^p)$ be a semi-simple element such that $\gamma'$ is an image of $\gamma_M$, and define \index{$I(\mathfrak e_{\fkp}(M),\gamma') $}
\begin{multline*}
 	I(\mathfrak e_{\fkp}(M),\gamma') : = \bar \iota^{M'} (\gamma')^{-1}   ( \Delta^{M}_{M'})^{\emptyset,\emptyset} (\gamma',\gamma_M) O_{\gamma_M}^{s_M'} (f^{p,\infty} _M)  \\  \cdot 
 	\sum_{A,B } \epsilon^{p,\infty} (A,B) \tau(G) \tau(H)^{-1}  \tau(M') SO_{\gamma'} (f^H_{p, M'}) S\Phi^H_{M'} (\gamma', f^H_{\infty}),
\end{multline*}
 where the terms $(\Delta^{M}_{M'})_{\emptyset,\emptyset} (\gamma',\gamma_M)$ and $ O_{\gamma_M}^{s_M'} (f^{p,\infty} _M) $ are the same as in (\ref{eq:Levi transfer away from p infty}) (except that $(A,B)$ is replaced by $(\emptyset,\emptyset)$), and the summation $\sum_{A,B}$ as well as the terms involving $H$ have the same meaning as in Lemma \ref{lem:first simplification}. By (\ref{eq:Levi transfer away from p infty}) we know that this definition is independent of the choice of $\gamma_M$. For each $(A,B)$ as above, we also define \index{$K (\fke_{\fkp}(M), \gamma', A, B)$}
 $$
 K (\fke_{\fkp}(M), \gamma', A, B):= (-1)^{q(G_{\RR})}\sum_{ \varphi_H \in \Phi_H(\varphi_{\mathbb V^*})} \det(\omega_*(\varphi_H)) \Phi^H_{M'} (\gamma'^{-1} ,\Theta_{\varphi_H}) , $$ where $\Theta_{\varphi_H} : = \Theta_{\Pi(\varphi_H)}$ is the sum of the characters of the members of the L-packet $\Pi(\varphi_H)$,\index{$\Theta_{\varphi_H}$} $\Phi^H_{\M'}(\cdot ,\Theta_{\varphi_H})$ is the normalized stable discrete series character as in \S \ref{subsubsec:mathbb V}, and the other notations are as in \S \ref{subsubsec:defn of f_infty}. 
 \end{defn}

\begin{lem}\label{lem:debut of I(e, gamma')} We have 
\begin{align*}
n^G_M \Tr'_M = \sum_{\mathfrak e_{\fkp}(M) = (M',\lang M', s_M', \eta_M)  \in \dot {\mathscr E } (M) ^{c,\ur}} \abs{\Out _M(\mathfrak e_{\fkp} (M)) } ^{-1} \sum_{\gamma' \in \Sigma(M') _1}  I(\mathfrak e_{\fkp}(M),\gamma'),
\end{align*} and  
\begin{multline*}
I(\mathfrak e_{\fkp}(M), \gamma') =   \bar \iota^{M'} (\gamma')^{-1}   ( \Delta^{M}_{M'}) ^{\emptyset,\emptyset} (\gamma',\gamma_M) O_{\gamma_M}^{s_M'} (f^{p,\infty} _M)    \tau(M) k(M)k(G)^{-1} \\ \cdot  (-1) ^{\dim A_{M'}}  \bar v ((M')_{\gamma'} ^0) ^{-1}  \sum_{A,B}  \epsilon^{p,\infty} (A,B) SO_{\gamma'} (f^H_{p, M'})  K (\mathfrak e_{\fkp}(M), \gamma', A,B)  .
\end{multline*}
Here the summation range for $\sum_{A,B}$ is the same as in Lemma \ref{lem:first simplification}. 
\end{lem} 
\begin{proof} The first identity follows from Lemma \ref{lem:first simplification}, Lemma \ref{lem:only regular}, \S \ref{para:transfer away from p and infty}, and the definitions. The second identity follows from Corollary \ref{tau k}, Lemma \ref{lem:paring=1}, and the definitions.
\end{proof}
\section{Computation of $K$}\label{subsec:comp K}
\subsection{}\label{para:B_0}
We keep the notation of Definition \ref{defn:Sigma_1} and study $K (\mathfrak e_{\fkp}(M), \gamma', A, B)$. As usual we write $\fke_{\fkp}(M) = (M',\lang M', s_M',\eta_M)$. We would like to apply \cite[Prop.~3.2.5]{morel2011suite} to compute $K$. First we need some preparations.
 
 By construction $M' = M^{\GL} \times M^{\prime,\SO}$ and $M^{\prime, \SO}$ is a product of two special orthogonal groups $M^{\prime, \SO, +}, M^{\prime, \SO, -} $\index{$M^{\prime, \SO, \pm}$} such that the component of $s_M$ in the dual group of $M^{\prime, \SO, \pm}$  is the scalar matrix $\pm 1$. Fix an elliptic maximal torus $T_{M'}$\index{$T_{M'}$} in $M'_{\RR}$ such that  $\gamma' \in T_{M'} (\RR).$  Then 
 $T_{M'}$ is of the form \index{$T_{M^{\GL}}$} \index{$T_{M^{\prime,\SO,\pm}}$}
 $$T_{M'} = T_{M^{\GL}} \times T_{M^{\prime,\SO,+}} \times T_{M^{\prime, \SO,-}}$$
 where $T_{M^{\GL}}$ (resp.~$T_{M^{\prime,\SO, \pm}}$) is an elliptic maximal torus in $M ^{\GL} _{\RR}$ (resp.~$M^{\prime, \SO, \pm} _{\RR}$). Moreover, as we have already seen in \S \ref{para:dot E for M}, the tori $T_{M^{\prime,\SO, \pm}}$ are in fact anisotropic over $\RR$. 
 Note that when $M= M_{12}$ or $M_2$, we have necessarily $T_{M^{\GL}} = M^{\GL}$. When $M = M_1$, we have $M^{\GL} = \GL_2$, and $T_{M ^{\GL}}$ is $\GL_2(\RR)$-conjugate to $T_{\GL_2}^{\std}$; cf.~\S \ref{para:R groups}. 
 
We then fix an elliptic maximal torus $T_M$\index{$T_M$} in $M_{\RR}$, and an admissible isomorphism \index{$j_M$}$j_M : T_{M'} \isom T_M. $ Recall from \S \ref{global groups} that $M = M^{\GL} \times M^{\SO} = M^{\GL} \times \SO(W)$, where $W = W_2$ if $M = M_1$ or $M_{12}$, and $W= W_1$ if $M = M_2$.   
 We may and shall assume that $T_M$ is of the form $ T_{M^{\GL}} \times T_{M^{\SO}},$ where  $T_{M^{\SO}}$\index{$T_{M^{\SO}}$} is an elliptic (and in fact anisotropic) maximal torus in $M^{\SO}_{\RR}$, and $T_{M^{\GL}}$ is as above. We may and shall also assume that $j_M$ is the product of the identity on $T_{M^{\GL}}$ and an admissible isomorphism
 \index{$j_{M^{\SO}}$}$$j_{M^{\SO}}: T_{M^{\prime,\SO,+}} \times T_{M^{\prime,\SO, -}} \isom T_{M^{\SO}},$$ where the notion of admissibility is with respect to the endoscopic datum $\fke_{\fkp}(W)$ for $M^{\SO}$.

 For any choice of a Borel subgroup $B_0$ of $G_{\CC}$ containing $T_{M,\CC}$, we get a canonical isomorphism 
 $\mathfrak d_{B_0, \mathcal B} : \widehat T_{M} \isom \mathcal T$ as in \S \ref{subsec:admissible isom}, where $(\mathcal T,\mathcal B)$ is the standard Borel pair in $\widehat{G}$ fixed in Definition \ref{defn:fixing borel pair in dual}. Identifying $\mathcal T$ with $(\CC^{\times})^m$ as in Definition \ref{defn:fixing borel pair in dual}, we have $m$ standard characters on $\mathcal T$ forming a basis of $X^*(\mathcal T)$, and they give rise, via $\mathfrak d_{B_0, \mathcal B}$, to $m$ cocharacters of $T_{M,\CC}$. We denote them (in order) by \index{$\tau_{0_1}, \tau_{0_2}, \tau_1, \tau_2,\cdots$} $$\tau_{0_1}, \tau_{0_2}, \tau_1, \tau_2,\cdots, \tau_{m-2}, \quad  \mbox{ if $M = M_{12} $ or $M_1$},$$  and by \index{$\tau_0, \tau_1,\tau_2,\cdots$}
 $$\tau_0, \tau_1,\tau_2,\cdots, \tau_{m-1} , \quad \mbox{ if $M= M_2$ (in the odd case)}. $$ We now fix a choice of $B_0$\index{$B_0$} such that the resulting cocharacters (just mentioned) satisfy the following conditions, the second of which depends on the choice of $j_M$.
   \begin{enumerate}
 	\item When $M= M_{12}$, we require that $\tau_{0_1}$ and $\tau_{0_2}$ are respectively the identity cocharacters of the first $\GG_m$ (i.e.~$\GL(V_1)$) and the second $\GG_m$ (i.e.~$\GL(V_2/V_1)$) in $T_{M^{\GL}} = \GG_m \times \GG_m \subset T_M$.
 	When $M= M_1$, we require that $\tau_{0_1}$ and $ \tau_{0_2}$ are cocharacters of $T_{M^{\GL},\CC} \subset T_{M,\CC}$, and that they are of the form  $$z \longmapsto g \begin{pmatrix}
 	z \\ & 1 
 \end{pmatrix} g ^{-1}  \qquad \text{and} \qquad z \longmapsto g \begin{pmatrix}
1\\ & z 
\end{pmatrix} g^{-1} $$ for some fixed $g\in M^{\GL}(\CC)$ conjugating the diagonal torus in $M^{\GL}_{\CC} = \GL_{2,\CC}$ to $T_{M^{\GL}, \CC}$. (Clearly this pins down $\tau_{0_1}$ and $\tau_{0_2}$ up to swapping the two.) When $M= M_2$, we require that $\tau_0$ is the identity cocharacter of $T_{M^{\GL}} =  \GL(V_1) = \GG_m \subset T_M$. 
 	\item We require that $j_M^{-1} \circ 
 	\tau_i$ is a cocharacter of $T_{M^{\prime, \SO, -} , \CC}$, for each $1 \leq i \leq n^-$. Here $n^-$ is the dimension of $T_{M^{\prime, \SO, -} , \CC}$.    
 \end{enumerate}

Indeed, the above conditions can be arranged because of the following observations:
\begin{itemize}
	\item For an arbitrary choice of $B_0$, the resulting $\tau$'s have the following property: The prescribed cocharacter(s) in (1) which we ask $\tau_{0_1}$ and $\tau_{0_2}$, or $\tau_0$, to equal, are among the $\tau$'s and their inverses. This is because these prescribed cocharacter(s) can be extended to a $\ZZ$-basis of $X_*(T_{M})$ under which the root datum of $(G_{\CC},T_{M,\CC})$ becomes the standard type $\mathsf B$ or type $\mathsf D$ root datum on $(\ZZ^m, \ZZ^m)$.
	
		\item By making different choices of $B_0$, we can arbitrarily permute the order of the $\tau$'s and replace an arbitrary number (resp.~an even number) of them by their inverses in the odd (resp.~even) case. In the even case with $M = M_{12}$ or $M_1$, we can replace either one or two of $\tau_{0_1}, \tau_{0_2}$ by their inverses as we wish, since  $m \geq 3$. Thus we can always arrange (1).  
	\item Once (1) is satisfied, the cocharacters $\tau_1,\tau_2,\cdots$ form a basis of $X_*(T_{M^{\SO}})$ under which the root datum of $(M^{\SO}_{\CC}, T_{M^{\SO}, \CC} )$ becomes the standard type $\mathsf B$ or type $\mathsf D$ root datum. Since $j_{M^{\SO}}$ is admissible, exactly $n^-$ of the $\tau_i$'s are such that $j_M^{-1} \circ \tau_i$ (equal to $j_{M^{\SO}}^{-1} \circ \tau_i$) is a cocharacter of $T_{M^{\prime, \SO, -} , \CC}$. We can rechoose $B_0$ in such a way that $\tau_{0_1}$ and $ \tau_{0_2}$, or $\tau_0$, are unchanged, but the order of $\tau_1,\tau_2,\cdots$ is permuted so that (2) is satisfied. 
	\end{itemize}

 \subsection{}\label{para:rho's}
 Up to now our discussion has not involved $(A,B)$. We now take them into account, so we have an endoscopic datum $(H, \lang H, s,\eta)$ for $G$ that is determined by $(\fke_{\fkp}(M), A, B)$ as in Lemma \ref{lem:first simplification} and Definition \ref{defn:Sigma_1}.  Recall from \S \ref{once and for all} that we have fixed $(T_H, T_G, j, B_{G,H})$. Similarly as in \S \ref{para:B_0}, the pair $(T_G, B_{G,H})$ determines an ordered $m$-tuple of cocharacters of  $T_{G,\CC}$ (via $\mathfrak d_{B_{G,H}, \mathcal B} : \widehat T_G \isom \mathcal T \cong (\CC^\times)^m$). We denote them by \index{$\rho_1,\rho_2,\cdots$} $$\rho_1,\rho_2,\cdots, \rho_m .$$ By the construction of $j$ in \S \ref{once and for all} (which uses \S \ref{subsubsec:setting for para tori} and especially Convention \ref{convention:identifying U(1)}), we know that $\set{	  j^{-1} \circ \rho _i  \mid  1\leq i \leq m^-}$ is a basis of $X_*(T_{H^-})$ (where $T_{H^-} : = T_H \cap H^-$) under which the root datum of $(H^-_{\CC}, T_{H^-,\CC})$ becomes the standard type $\mathsf B$ or $\mathsf D$ root datum. Similarly,   $\set{	  j^{-1} \circ \rho _i  \mid  m^- + 1 \leq i \leq m}$ is a basis of $X_*(T_{H^+})$ under which the root datum of $(H^+_{\CC}, T_{H^+,\CC})$ becomes the standard type $\mathsf B$ or $\mathsf D$ root datum. 
 
 \begin{defn}\label{defn:i_G(A)}
 	Define an isomorphism \index{$i_G(A, B)$} $i_G(A, B ): T_{M,\CC} \isom T_{G,\CC}$ as follows. When $M= M_{12}$ (so $B \equiv \emptyset$), let $i_{G}(A,B)$ map $\tau_{0_1}, \tau_{0_2}, \tau_1,\cdots, \tau_{m-2}$ respectively to 
 	$$ \begin{cases}
 	\rho_1,\rho_2,\cdots, \rho_m , & A=\emptyset ,\\
 	\rho_{m^-+1}, \rho_1,\rho_2,\cdots, \rho_{m^-}, \rho_{m^-+2},\cdots, \rho_m, & A = \set{1} ,\\
 	\rho_1, \rho_{m^-+1}, \rho_2,\cdots, \rho_{m^-}, \rho_{m^-+2},\cdots, \rho_m, & A = \set{2} , \\
 	\rho_{m^- +1} , \rho_{m^-+2}, \rho_1,\cdots, \rho _{m^-}, \rho_{m^- +3}, \cdots, \rho_m, & A= \set{1,2} .
 	\end{cases}$$ (In the even case the parameter $A$ can only assume $\set{1,2}$ and $\emptyset$, cf.~Lemma \ref{lem:first simplification}, and we use only these two cases in the above formula). 
 	When $M = M_1$ (so $A \equiv \emptyset$), let  $i_{G}(A,B)$ map $\tau_{0_1}, \tau_{0_2}, \tau_1,\cdots, \tau_{m-2}$ respectively to 
 	$$ \begin{cases}
 		\rho_1,\rho_2,\cdots, \rho_m , & B=\emptyset ,\\ 
 		\rho_{m^- +1} , \rho_{m^-+2}, \rho_1,\cdots, \rho _{m^-}, \rho_{m^- +3}, \cdots, \rho_m, & B= \set{1} .
 	\end{cases}$$ 
 	When $M = M_2$ (so $B \equiv \emptyset$), let $i_G(A,B)$ map $\tau_0, \tau_1,\cdots, \tau _{m-1}$ respectively to 
 	$$ \begin{cases}
 	\rho_1,\cdots, \rho_m , & A=\emptyset ,\\
 	\rho_{m^- +1}, \rho_1,\cdots, \rho_{m^-}, \rho_{m^- +2},\cdots, \rho_m, & A =\set{1}.
 	\end{cases}$$
 	\end{defn}
 
 In the following lemma, recall from \S \ref{para:prepare for main} that we have identified $M'$ with a Levi subgroup of $H$. 
\begin{lem} \label{lem:i_G and i_H} Let $i_H(A,B)$\index{$i_H(A,B)$} be the unique isomorphism $T_{M',\CC} \isom T_{H,\CC}$ fitting in  the following commutative diagram:
	\begin{align}\label{diag:j and j_M}
 \xymatrix{ T_{H,\CC} \ar[r]^j & T_{G,\CC} \\ 
	T_{M',\CC}  \ar[r]^{j_M} \ar[u] ^{i_H(A,B)} & T_{M,\CC} \ar[u] _{i_G(A,B)}  } 
	\end{align}
 Then 
	$i_G(A,B)$ (resp.~$i_H(A,B)$) is induced by an inner automorphism of $G_{\CC}$ (resp.~$H_{\CC}$).
\end{lem} 
 \begin{proof}Firstly, the isomorphism $i_G(\emptyset,\emptyset): T_{M,\CC} \isom T_{G,\CC} $ is compatible with the two canonical isomorphisms $\mathrm{BRD}(T_{M,\CC}, B_0) \cong \mathrm{BRD}(G)$ and $\mathrm{BRD}(T_{G,\CC}, B_{G,H}) \cong \mathrm{BRD}(G)$, where $\mathrm{BRD}(G)$ is the canonical based root datum of $G_{\CC}$ (see \S \ref{subsubsec:L-group data}). Hence $i_G(\emptyset,\emptyset)$ is induced by an inner automorphism of $G_{\CC}$. For general $(A,B)$, $i_G(A,B)$ differs from $i_G(\emptyset,\emptyset)$ by an automorphism of $T_{G,\CC}$ which permutes the order of the $\rho_i$'s. Such an automorphism is in the Weyl group (because under the basis $\set{\rho_1,\cdots,\rho_m}$ of $X_*(T_G)$ the root datum of $(G_{\CC}, T_{G,\CC})$ becomes the standard type $\mathsf B$ or $\mathsf D$ root datum), and is hence still induced by an inner automorphism of $G_{\CC}$. 
 	
 	We now prove that $i_H(A,B)$ is induced by an inner automorphism of $H_{\CC}$. For brevity, we only illustrate the proof in the special case where $M=M_{12}$ and $(A,B) = (\set{1}, \emptyset)$, the other cases all being similar. Also, we only treat the even case, as the odd case is easier. We freely use the notation of \S \ref{para:two maps from endoscopic G data}; in particular $M^{\prime, \SO , \pm} = \SO(W^{\pm})$, $H^{\pm} = \SO(V^{\pm})$, and $d^{\pm} = \dim W^{\pm}$. As in \S \ref{para:two maps from endoscopic G data}, we have a canonical $\SO(W^+)(\CC)$-conjugacy class of embeddings $$\iota_{W^+} : \GG_m^{d^+/2} \To \SO(W^+)_{\CC}$$ and a canonical $\SO(V^+)(\CC)$-conjugacy class of embeddings $$ \iota_{V^+}: \GG_m^{d^+/2 + 1}  \To \SO(V^+)_{\CC}.$$ (Here, if $\delta^+ = -1$, we identify $\Uni(1)_{\CC}$ with $\GG_m$.) By the two conditions satisfied by $B_0$ in \S \ref{para:B_0} and the admissibility of $j_{M^{\SO}}$, we know that the embedding\footnote{Here we use the following notation: If $\mu_1,\cdots,\mu_k$ are cocharacters of a torus $T$ contained in a reductive group $R$ (everything being over $\CC$), we write $(\mu_1,\cdots, \mu_k)$ for the homomorphism $\GG_m^k \to T \subset R, (z_1,\cdots , z_k) \mapsto \prod_i \mu_i(z_i)$.}
 	$$ (j_M^{-1} \circ \tau_{d^-/2 +1},\cdots, j_M^{-1} \circ \tau_{m-2}) : \GG_m^{d^+/2} \To \SO(W^+)_{\CC} $$ is $\SO(W^+)(\CC)$-conjugate to $\iota_{W^+}$, and that $j_M^{-1} \circ \tau _{0_1}$ is the identity cocharacter of $\GL(V_1)$, namely $\iota^{\GL}_{A,B}$ in the notation of \S \ref{para:two maps from endoscopic G data}. Thus by the construction in \S \ref{para:two maps from endoscopic G data}, the embedding 
 	\begin{align} \label{eq:para T_M'}
  (j_M^{-1} \circ \tau_{0_1}, j_M^{-1} \circ \tau_{d^-/2 +1},\cdots, j_M^{-1} \circ \tau_{m-2}) : \GG_m^{d^+/2 + 1} \To \GL(V_1)_{\CC} \times \SO(W^+)_{\CC}
 	\end{align}
  is $\SO(V^+)(\CC)$-conjugate to $\iota_{V^+}$ when we view $\GL(V_1)  \times \SO(W^+)$  as a subgroup of $\SO(V^+)$ according to the rule in \S \ref{para:two maps from endoscopic G data}. On the other hand, the embedding 
  \begin{align}
  	\label{eq:para T_H} (j^{-1} \circ \rho_{m^- +1} ,\cdots, j^{-1} \circ \rho_m) : \GG_m^{d^+/2 + 1} \To \SO(V^+)_{\CC} 
  \end{align}
is also  $\SO(V^+)(\CC)$-conjugate to $\iota_{V^+}$. Hence   (\ref{eq:para T_M'}) and (\ref{eq:para T_H}) are $\SO(V^+)(\CC)$-conjugate. Similarly, we know that the embeddings 	\begin{align*} 
	(j_M^{-1} \circ \tau_{0_2}, j_M^{-1} \circ \tau_{1},\cdots, j_M^{-1} \circ \tau_{d^-/2}) : \GG_m^{d^-/2 + 1} \To \GL(V_2/V_1)_{\CC} \times \SO(W^-)_{\CC}
\end{align*} and   \begin{align*}
 (j^{-1} \circ \rho_{1} ,\cdots, j^{-1} \circ \rho_{m^-}) : \GG_m^{d^-/2 + 1} \To \SO(V^-)_{\CC} 
\end{align*} are $\SO(V^-)(\CC)$-conjugate. We conclude that the embeddings 
$$ (j_M^{-1} \circ \tau_{0_1}, j_M^{-1} \circ \tau_{0_2}, j_M^{-1} \circ \tau_1 ,\cdots, j_M^{-1} \circ \tau_{m-2}) : \GG_m^m \To H_{\CC}$$ and 
$$ (j^{-1} \circ \rho_{m^- +1}, j^{-1} \circ \rho_1, \cdots, j^{-1} \circ \rho_{m^-}, j^{-1} \circ \rho_{m^-+2}, \cdots, j^{-1} \circ \rho_m) : \GG_m^m \To H_{\CC}$$ are $H(\CC)$-conjugate. But these two embeddings have images $T_{M' , \CC}$ and $T_{H,\CC}$ respectively, and if we invert the first and compose with the second we precisely get the isomorphism $i_H(A,B)$. This finishes the proof. 
 \end{proof}
 
\begin{defn}\label{defn:4 Borels}
	Define the three Borel subgroups:
	\begin{itemize}
		\item $B_M$, a Borel of $M_{\CC}$ containing $T_{M,\CC}$, defined to be $B_0 \cap M$. 
		\item $B_G$, a Borel of $G_{\CC}$ containing $T_{G,\CC}$, defined to be $i_G(A,B) _* B_0$. This can be different from $B_{G,H}$ fixed in \S \ref{once and for all}.
		\ignore{
		\item $B_{M'}$, a Borel of $M'_{\CC}$ containing $T_{M',\CC}$, defined to be the one induced by $( j_M, B_M)$ in the same manner as in \S \ref{subsubsec:defn of f_infty}. In other words, $j_M$ carries the $B_{M'}$-positive roots on $T_{M',\CC}$ to $B_M$-positive roots on $T_{M,\CC}$.  }
		\item $B'_{H}$, a Borel of $H_{\CC}$ containing $T_{H,\CC}$, defined to be the one induced by  $(j,B_G)$. In other words, $j$ carries the $B_H'$-positive roots on $T_{H,\CC}$ to $B_G$-positive roots on $T_{G,\CC}$. 
	\end{itemize}
\end{defn} 
\begin{lem}\label{lem:B_H'}
	We have $B_H' = B_H$, where $B_H$ is defined in \S \ref{subsubsec:defn of f_infty}. 
\end{lem}
\begin{proof}
		We use $j$ to identify $T_{H,\CC}$ and $T_{G,\CC}$. Thus we have an inclusion of root systems $$\Phi_H: = \Phi(H_{\CC}, T_{H,\CC}) \subset \Phi_G: = \Phi(G_{\CC}, T_{G,\CC}). $$
	To prove the lemma, we need to prove that for all $\alpha \in \Phi_H$, it is $B_G$-positive if and only if it is $B_{G,H}$-positive. We denote the permutation of $\rho_i$'s that appears in Definition \ref{defn:i_G(A)} by 
	$$\rho_{\sigma(1)}, \rho_{\sigma(2)} ,\cdots, \rho _{\sigma(m)} , \qquad \sigma \in \mathfrak S_m. $$ (For instance, if $A = \set{1}$, then $\sigma$ sends $1,2,\cdots, m$ respectively to $m^- +1, 1,\cdots, m^-, m^- +2, \cdots, m$.) Let $\set{\rho_1^{\vee}, \cdots, \rho_m ^{\vee}}$ be the basis of $X^*(T_{G,\CC})$ dual to the basis $\set{\rho_1,\cdots, \rho_m}$ of $X_*(T_{G,\CC})$. The $B_{G,H}$-positive roots in $\Phi_G$ are 
	$$\begin{cases}
		\set{\rho_i^{\vee} \pm \rho_j ^{\vee} \mid  i>j } \cup \set{\rho_i^{\vee} \mid i} , & \mbox{odd case} , \\\\
		\set{\rho_i^{\vee} \pm \rho_j ^{\vee} \mid  i>j } , & \mbox{even case}.
	\end{cases}$$   
	The $B_{G}$-positive roots in $\Phi_G$ are 
	$$\begin{cases}
		\set{\rho_{\sigma(i)}^{\vee} \pm \rho_{\sigma(j)} ^{\vee} \mid  i>j } \cup \set{\rho_i^{\vee}\mid i } , & \mbox{odd case} , \\\\
		\set{\rho_{\sigma(i)}^{\vee} \pm \rho_{\sigma(j)} ^{\vee} \mid  i>j } , & \mbox{even case}.
	\end{cases}$$  
	On the other hand, by the last observation in \S \ref{para:rho's}, we have
	$$ \Phi_H =  \begin{cases}
		\set{ \pm \rho_i^{\vee} \pm \rho_j ^{\vee} \mid  i,j \leq m^- , i \neq j } \cup \set{ \pm \rho_i^{\vee} \pm \rho_j ^{\vee} \mid  i,j > m^- , i \neq j } \cup \set{\rho_i^{\vee}\mid i} , \\\\
		\set{ \pm \rho_i^{\vee} \pm \rho_j ^{\vee} \mid  i,j \leq m^- , i \neq j } \cup \set{ \pm \rho_i^{\vee} \pm \rho_j ^{\vee} \mid  i,j > m^- , i \neq j }  ,
	\end{cases} $$  in the odd and even cases respectively.
	It remains to check that $\sigma^{-1}|_{\set{1,2,\cdots, m^-}}$ and $\sigma^{-1}|_{\set{m^- +1, \cdots , m}}$ are increasing, which is true.  
\end{proof}
\subsection{}\label{para:Delta_j_M B_M}
 We now transport  \cite[Prop.~3.2.5]{morel2011suite} to our setting. For any $t \in T_M(\RR)$, let $\epsilon_R(t) \in \set{\pm 1}$\index{$\epsilon_R$} be $-1$ to the number of $B_0$-positive roots $\alpha$ of $(G_{\CC}, T_{M,\CC})$ such that $\alpha$ is real and $0<\alpha(t)< 1$. (Compare with the definition in \S \ref{subsubsec:notation for roots and Weyl}.) Similarly, for $t' \in T_{M'}(\RR)$, we let $\epsilon_{R_H}(t')\in \set{\pm 1}$\index{$\epsilon_{R_H}$} be $-1$ to the number of $i_H(A,B)^{-1}_* (B'_H)$-positive (or equivalently, $i_H(A,B)^{-1}_* (B_H)$-positive, by Lemma \ref{lem:B_H'}) roots $\alpha$ of $(H_{\CC}, T_{M',\CC})$ such that $\alpha$ is real and $0< \alpha (t') < 1$.  We set \footnote{In \cite[Prop.~3.2.5]{morel2011suite}, our $\Delta_{j_M, B_M} ^{A,B}$ is denoted simply by $\Delta _{j_M, B_M}$. However, this object is not intrinsic to $(j_M, B_M)$, since its definition involves the number $q(H_{\RR})$ which depends on $(A,B)$.} \index{$\Delta_{j_M,B_M}^{A,B}$}
 $$ \Delta_{j_M,B_M}^{A,B} :  = (-1) ^{q(G_{\RR} ) + q(H_{\RR}) + q(M_{\RR}) + q(M'_{\RR})} \Delta_{j_M, B_M},$$ where $\Delta_{j_M,B_M}$\index{$\Delta_{j_M,B_M}$} is Kottwitz's normalization of the archimedean transfer factors between $\fke_{\fkp}(M) = (M',\lang M', s_M',\eta_M)$ and $M$ associated to $(j_M, B_M)$  (see \cite[\S 7]{kottwitzannarbor}, cf.~\S\S \ref{transf factors odd}--\ref{transf factors even div by 4}).  
 Let $\Theta_{\mathbb V^*}$\index{$\Theta_{\mathbb V^*}$} denote the analogue of $\Theta_{\mathbb V}$ in \S \ref{subsubsec:mathbb V} with $\mathbb V$ replaced by $\mathbb V^*$. The following result is \cite[Prop.~3.2.5]{morel2011suite}. 
 \begin{prop}\label{prop:morel transfer} We have 
 \begin{multline*} 
 \epsilon_{R}(j_M(\gamma'^{-1})) \epsilon_{R_H} (\gamma'^{-1})   \Delta_{j_M, B_M} ^{A,B}(\gamma', j_M(\gamma')) \Phi^G_M (j_M(\gamma') ^{-1} , \Theta_{\mathbb V^*} ^H)  \\  =  \sum_{\varphi_H \in \Phi_H(\varphi_{\mathbb V^*})} \det(\omega'_*(\varphi_H)) \Phi^H_{M'} (\gamma'^{-1} , \Theta_{\varphi_H}),
 \end{multline*}Here the elements $\omega'_*(\varphi_H) \in \Omega$ are the analogues of the elements $\omega_*(\varphi_H) \in \Omega$ in \S \ref{subsubsec:defn of f_infty} with $(j, B_{G,H})$ replaced by $(j, B_G)$. The term $\Phi^G_M(\cdot, \Theta^H_{\mathbb V^*})$\index{$\Phi^G_M(\cdot,\Theta^H_{\mathbb V^*})$} is given as follows. Only when $M = M_{12}$ and $A = \set{1}$ or $\set{2}$ (which in particular implies that we are in the odd case; see Lemma \ref{lem:first simplification}), it is equal to $\Phi^G_M (\cdot ,\Theta_{\mathbb V^*}) _{\Endos}$ (defined as in (\ref{eq:Phi^G_M endos}), but with $\mathbb V$ replaced by $\mathbb V^*$).
In all the other cases, it is equal to  
$ \Phi^G_M (\cdot , \Theta_{\mathbb V^*}). $
\qed
 \end{prop}

\subsection{}\label{para:investigate diff} 
 For a fixed $\varphi_H$ as in Proposition \ref{prop:morel transfer}, we investigate the relation between $\omega_*(\varphi_H)$ and $\omega_*'(\varphi_H)$. Write $\omega_* : = \omega_*(\varphi_H)$ and $\omega'_* : = \omega'_*(\varphi_H)$. By definition, $\varphi_H$ is aligned\index[n]{aligned} with $(\omega_* ^{-1} \circ j , B_{G, H}, B_H)$ and also aligned with $((\omega_*') ^{-1}\circ j, B_G, B_H')$. Suppose $\omega_0\in \Omega (G_{\CC}, T_{G,\CC})$\index{$\omega_0$} measures the difference between $B_G$ and $B_{G,H}$, so that the map $\widehat T_G \to \widehat G$ determined by $B_{G}$ and $\varphi_{H}$ (namely the first row of the commutative diagram on the bottom of \cite[p.~184]{kottwitz1992lambda}) is equal to the composition of $\widehat{\omega_0}: \widehat T_G \to \widehat T_G$ with the analogous map $\widehat {T_G} \to \widehat G$ determined by $B_{G,H}$ and $\varphi_{H}$. By the definition of ``being aligned'' and by Lemma \ref{lem:B_H'}, we know that the composition 
 $$\widehat{T} \xrightarrow{\widehat{\omega_0}} \widehat{T} \xrightarrow{\widehat{\omega_*^{\prime,-1} \circ j}} \widehat T_H $$ is equal to the map 
 $$ \widehat{\omega_*^{-1} \circ j}:  \widehat{T} \To \widehat T_H  .$$ Hence 
 $$\omega_* ' (\varphi _H) = \omega_*(\varphi_H) \omega_0.  $$ 
 In particular, 
 \begin{align}\label{eq:function of omega_0}
 \det(\omega_*'(\varphi_H)) = \det(\omega_*(\varphi_H)) \det(\omega_0). 
 \end{align}

\begin{lem}\label{lem:B_H=B_H'}
	We have 
	\begin{equation}\label{eq:det(omega_0)}
		\det(\omega_0) =  \begin{cases}
			1, & A =\emptyset, \\
			(-1)^{m^-}, & A = \set{1},\\
			(-1)^{m^- +1}, & A=\set{2} ,\\
			1, & A=\set{1,2}.
		\end{cases}
	\end{equation}
(Here the formula works in all cases considered in Lemma \ref{lem:first simplification}. For instance, $A= \set{1}$ could only happen in the odd case either when $M = M_{12}$ or when $M = M_2$.)
\end{lem}
\begin{proof}
 From the description of the $B_{G,H}$-positive and $B_G$-positive roots in $\Phi_G$ in the proof of Lemma \ref{lem:B_H'}, we see that $\det(\omega_0)$ is equal to the sign of the permutation $\sigma$ in that proof. Thus (\ref{eq:det(omega_0)}) follows from direct calculation of this sign. 
\end{proof}

\begin{prop}\label{pro:K}
	We have 
	\begin{multline*}
K (\mathfrak e_{\fkp}(M), \gamma', A,B) =(-1)^{q(G_{\RR})} \det (\omega_{0})    \epsilon_{R}(j_M(\gamma'^{-1})) \epsilon_{R_H} (\gamma'^{-1}) \\  \cdot \Delta_{j_M, B_M} ^{A,B}(\gamma', j_M(\gamma')) \Phi^G_M (j_M(\gamma') ^{-1} , \Theta_{\mathbb V^*} ^H),
	\end{multline*}
	where $\det(\omega_{0})$ is given in (\ref{eq:det(omega_0)}).  
\end{prop}
\begin{proof}
	This is a consequence of Proposition \ref{prop:morel transfer} and (\ref{eq:function of omega_0}). 
\end{proof}
\section{Computation of some signs}\label{pf:4} We keep the notation of \S \ref{subsec:comp K}. 
\begin{defn}\label{defn:tasho}
 Let $\tasho(A,B)  \in \CC^{\times}$\index{$\tasho(A,B)$} be the constant such that the normalization $\tasho(A,B) \cdot \Delta_{j_M, B_M}^{A,B} $ of transfer factors 
  between $\fke_{\fkp}(M)$ and $M$ at $\infty$ together with the normalizations $(\Delta^M_{M
 '})^{A,B}_v$ at all finite places (fixed in \S \ref{para:normalizing tf for M}) satisfy the global product formula. Here $\Delta_{j_M, B_M}^{A,B} $ is defined in \S \ref{para:Delta_j_M B_M}. 
\end{defn} 
 
\begin{lem}\label{lem:tasho and epsilon} The normalization $\tasho(A,B) \Delta_{j_M, B_M}^{A,B}$ of transfer factors 
	between $\fke_{\fkp}(M)$ and $M$ at $\infty$ is inherited from $\Delta_{j, B_{G,H}}$ in the sense of Remark \ref{rem:Levi inherit}. Let $\epsilon^{p,\infty} (A,B) $ be as in Definition \ref{defn:epsilon^{p,A}}. We have
$$ \Delta^{A,B}_{j_M, B_M} \cdot  \epsilon^{p,\infty} (A,B) = \Delta^{\emptyset,\emptyset}_{j_M, B_M} \cdot  \tasho(A,B)^{-1} \tasho(\emptyset,\emptyset). $$
\end{lem}
 \begin{proof}
 	The first assertion follows from the fact that $(\Delta^G_H)_v$ for all $v$ satisfy the global product formula (see \S \ref{normalizing the transf factors}), and the fact that inheritance of normalizations respects the global product formula (see Remark \ref{rem:Levi inherit}). To prove the second assertion, by the definition of $\tasho(A,B)$ we must have 
 	$$  \tasho(A,B) \Delta_{j_M, B_M}^{A,B} \prod_{v \neq \infty} (\Delta^M_{M'})^{A,B}_v = \tasho(\emptyset,\emptyset) \Delta_{j_M, B_M}^{\emptyset,\emptyset} \prod_{v \neq \infty} (\Delta^M_{M'})^{\emptyset,\emptyset}_v . $$ But $(\Delta^M_{M'})^{A,B}_p = (\Delta^M_{M'})^{\emptyset,\emptyset}_p$ because they are both the canonical unramified normalization associated to $\mathcal M(\ZZ_p)$. (See \S \ref{para:normalizing tf for M} for $\mathcal M(\ZZ_p)$.) Hence 	$$  \tasho(A,B) \Delta_{j_M, B_M}^{A,B} \prod_{v \neq p,\infty} (\Delta^M_{M'})^{A,B}_v = \tasho(\emptyset,\emptyset) \Delta_{j_M, B_M}^{\emptyset,\emptyset} \prod_{v \neq p,\infty} (\Delta^M_{M'})^{\emptyset,\emptyset}_v . $$ Our assertion follows from comparing the above equality with  Definition \ref{defn:epsilon^{p,A}}.
 \end{proof}
\subsection{}\label{para:BMSO}
As usual we denote by $W,  W^{\pm}$ the underlying quadratic spaces for $M^{\SO}, M^{\prime, \SO, \pm}$, i.e., $M^{\SO} = \SO(W), M^{\prime,\SO,\pm} = \SO(W^{\pm})$. Denote by $M^{\SO, *}$ the fixed quasi-split inner form of $M^{\SO}$ as in \S \ref{subsubsec:choices for W}. Namely we have $M^{\SO,*} = \SO(\underline W)$, and as in \S \ref{subsubsec:choices for W} we have fixed isomorphisms $\phi_{W_{\RR}} : W_{\CC} \isom \underline W_{\CC}$ (with respect to $F=\RR$ and satisfying the extra condition in Definition \ref{defn:fixing isom between quad spaces}) and $\psi_{W_{\RR}}: M^{\SO}_{\CC} \isom M^{\SO, *}_{\CC}, g\mapsto \phi_{W_{\RR}} g \phi_{W_{\RR}}^{-1}$.

By the two conditions noted in \S \ref{para:dot E for M}, we know that the localization over $\RR$ of the endoscopic datum  $\fke_{\fkp}(W)$ for $M^{\SO}_{\RR} = \SO(W_{\RR})$ satisfies the hypotheses in \S \ref{sec:transfer factors} (with $V, \underline V, V^{\pm}$ there replaced by $W_{\RR}, \underline W_{\RR}, W^{\pm}_{\RR}$.) In other words, this is an elliptic endoscopic datum over $\RR$, and the group in it contains $\RR$-anisotropic maximal tori. Define $\ED(W_{\RR})^o, \ED(\underline W_{\RR})^o, \ED(W^\pm_{\RR})^o$ as in \S \ref{subsubsec:defn of EDo} and \S \ref{subsubsec:setting for para tori}. Inside $\ED(W_{\RR})^o$ we have the subset $\ED(W_{\RR})^o_{\mathrm{nice}}$ as in Definitions \ref{defn:EDNice} and \ref{defn:EDNice, div by 4}. 	Let $B_{M^{\SO}}$ be the Borel subgroup of $M^{\SO}_{\CC}$ given by $B_M \cap M^{\SO}_{\CC}$, and let $j_{M^{\SO}}: T_{M^{\prime,\SO,+}} \times T_{M^{\prime,\SO, -}} \isom T_{M^{\SO}}$ be as in \S  \ref{para:B_0}. Thus  $(T_{M^{\SO}}, B_{M^{\SO}})$ is a fundamental pair in $M^{\SO}_{\RR} = \SO(W_{\RR})$.
\begin{lem}\label{lem:jBforM}
 There exist $\mathcal D_1 \in \ED(W_{\RR})^o_{\mathrm{nice}}$ and $\mathcal D_2 = (\mathcal D_2^+, \mathcal D_2^-) \in \ED(W^+_{\RR})^o \times \ED(W^-_{\RR})^o$ such that the fundamental pair $(T_{M^{\SO}}, B_{M^{\SO}})$ arises from $\mathcal D_1$ as in \S  \ref{subsubsec:constrn with ED}, and $j_{M^{\SO}} = j_{\mathcal D_2, \mathcal D_1}$ where $j_{\mathcal D_2, \mathcal D_1}$ is as in \S \ref{subsubsec:setting for para tori}.
\end{lem}
\begin{proof}Firstly, since the signature of $W_{\RR}$ is $(d-4,0)$ or $(d-3,1)$, we have $\ED(W_{\RR})^o = \ED(W_{\RR})^o_{\mathrm{nice}}$. Since all anisotropic maximal tori in $M^{\SO}_{\RR}$ are conjugate under $M^{\SO}(\RR)$, we can find $\mathcal D_1 \in \ED(W_{\RR})^o$ such that $T_{M^{\SO}} = T_{\mathcal D_1}$ (notation as in \S \ref{subsubsec:constrn with ED}). By reordering the members of $\mathcal D_1$, and in the odd (resp.~even) case changing the orientations of an arbitrary (resp.~even) number of the members of $\mathcal D_1$, we may and shall assume that the fundamental pair $(T_{M^{\SO}}, B_{M^{\SO}})$ arises from $\mathcal D_1$. Let $m'$ be the absolute rank of $M^{\SO}$. Using Lemma \ref{lem:motivation for o_V} and the same argument as in the proof of Lemma \ref{lem:j is adm}, we see that there exist $g\in M^{\SO}(\RR)$ and $\mathcal D_0 \in \ED(\underline W)^o$ such that $\Int(g)
	\circ f_{\mathcal D_1} = \psi_W^{-1} \circ f_{\mathcal D_0}$. (Here $f_{\mathcal D_0}$ and $f_{\mathcal D_1}$ are as in \S \ref{subsubsec:constrn with ED}.) Then by Lemma \ref{e.d. dual}, the isomorphism $$ (\tau_1,\cdots, \tau_{m'}) :  \GG_{m,\CC}^{m'} \isom T_{M^{\SO}, \CC} $$ (see \S \ref{para:B_0} for the $\tau_i$'s) is equal to the base change to $\CC$ of 
	$ f_{\mathcal D_1} : \Uni(1)^{m'} \isom T_{M^{\SO}}, $ where we identify $\Uni(1)_{\CC}$ with $\GG_{m,\CC}$.\ 
	
	To simplify notation below we write $T^{\pm}$ for $T_{M^{\prime, \SO, \pm}}$. 
	Since $j_{M^{\SO}}$ is admissible, by condition (2) in \S \ref{para:B_0} we know that the isomorphisms
	\begin{align}\label{eq:param1}
	(j_{M^{\SO}}^{-1} \circ \tau_{n^- + 1},\cdots, j_{M^{\SO}}^{-1} \circ \tau_{m'}): \GG_{m,\CC}^{m'-n^-} \isom T^+_{\CC} 
	\end{align}
		and 
		\begin{align}\label{eq:param2}
		(j_{M^{\SO}}^{-1} \circ \tau_1,\cdots, j_{M^{\SO}}^{-1} \circ \tau_{n^-}) : \GG_{m,\CC}^{n^-} \isom T^-_{\CC}
		\end{align}
	are induced by the isomorphisms $$ \mathfrak d _{B^{\pm}, \mathcal B^\pm } : \widehat T^{\pm} \isom \mathcal T^\pm \cong (\CC^{\times})^{m'- n^-} \text{ or }  (\CC^{\times})^{n^-} $$ associated to some Borel subgroups $B^{\pm}$ of $M^{\prime, \SO, \pm}_{\CC}$ containing $T^{\pm}_{\CC}$. Here $(\mathcal T^\pm, \mathcal B^\pm)$ are the standard Borel pairs in the dual groups of $M^{\prime,\SO,\pm}$, and the notation $\mathfrak d_{\cdot, \cdot}$ is as in \S \ref{subsec:admissible isom}. By the same argument as before, we can find $\mathcal D_2 = (\mathcal D_2^+, \mathcal D_2^-) \in \ED(W^+_{\RR})^o \times \ED(W^-_{\RR})^o$ such that $\mathcal D_2^{\pm}$ gives rise to the fundamental pair $(T^{\pm}, B^{\pm})$.  By Lemma \ref{e.d. dual}, the isomorphisms (\ref{eq:param1}) and (\ref{eq:param2}) are equal to $f_{\mathcal D_2^+, \CC}$ and $f_{\mathcal D_2^-, \CC}$ respectively. Combining this with the previously established fact that $(\tau_1,\cdots, \tau_{m'}) = f_{\mathcal D_1, \CC}$, we conclude that $j_{M^{\SO}} = j_{\mathcal D_2, \mathcal D_1}$. 
\end{proof}
\begin{prop}\label{prop:computing tasho}
	For  $(A,B)$ taking values as in Lemma \ref{lem:first simplification}, we have 
	\begin{align*}
\tasho(\emptyset,\emptyset)  & = -1 ,  \\ \tasho(A,B)^{-1} \tasho(\emptyset,\emptyset) &  = \begin{cases}
1,& \text{if } (A,B) = (\emptyset,\emptyset) , \\
-1 ,  & \text{if } A = \set{1,2} \text{ or } B = \set{1}, \\
(-1)^{m^- +1 }, & \text{ in all other cases}. 
\end{cases}
	\end{align*}
	\end{prop} 

\begin{proof}
	 In this proof we pass to the local notation over $\RR$. For instance, we write $M$ for $M_{\RR}$.  We use the phrase ``Whittaker normalization'' when we mean the Whittaker-normalized transfer factors between $H$ and $G^*$ or between $H$ and $G$, associated to the unique (resp.~the type-I) equivalence of Whittaker data for $G^*$ when $d$ is not divisible by $4$ (resp.~$d$ is divisible by $4$); see \S \ref{subsubsec:setting for transf factor}, Definition \ref{defn:Whitt normalization for G}, Definition \ref{defn:two Whitt norm}, and  Definition \ref{defn:Whitt norm for G div by 4}. We shall also apply this notion to the transfer factors between $M^{\prime,\SO}$ and $M^{\SO, *}$, and between $M^{\prime, \SO}$ and $M^{\SO}$.  By extending trivially across $M^{\GL}$, we also obtain the ``Whittaker normalization'' of transfer factors between $M' = M^{\GL} \times M^{\prime,\SO}$ and $M^* = M^{\GL} \times M^{\SO, *}$, and between $M'$ and $M  = M^{\GL} \times M^{\SO}$. As in \S \ref{subsubsec:choices for W}, we view $M^*$ as a Levi subgroup of $G^*$ via (\ref{eq:Levi resulting from phi_W^V}).

		We claim that the Whittaker normalization between $M'$ and $M$ is inherited from the Whittaker normalization between $H$ and $G$ as in Remark \ref{rem:Levi inherit}.

		To prove the claim, first assume $d$ is odd. Then $G^*$ has a unique Whittaker datum (up to equivalence) and a unique $\RR$-splitting (up to $G^*(\RR)$-conjugacy). The same also holds for $M^*$. Thus the unique Langlands--Shelstad normalization of transfer factors between $M'$ and $M^*$ is inherited from the unique Langlands--Shelstad normalization between $H$ and $G^*$. (Indeed, one can see this by going through the definitions in \cite{LS87};  alternatively, one can see this by using Waldspurger's explicit formula  \cite[\S 1.10]{walds10} while noting that the constant $\eta$ in \cite[\S 1.6]{walds10} attached to the unique splitting of $G^* = \SO(\underline V)$ is equal to the discriminant $\delta$, and hence equal to the analogous constant for $M^{\SO, *} = \SO(\underline W)$.) Moreover, the local epsilon factor \index[n]{local epsilon factor} relating the Whittaker normalization and the Langlands--Shelstad normalization (cf.~(\ref{eq:Walds 1})) is $1$ in both the $(H,G^*)$-scenario and the $(M',M^*)$-scenario. This implies that the Whittaker normalization between $M'$ and $M^*$ is inherited from the Whittaker normalization between $H$ and $G^*$. Our claim then follows from the three compatibility conditions in \S \ref{subsubsec:choices for W}.

		Second, assume $d$ is even and not divisible by $4$. Then by assumption $M = M_1$ or $M_{12}$, and so  $M^{\SO} = \SO(W)$ with $\dim W = d -4$ again not divisible by $4$. Hence we still have uniqueness of Whittaker datum and uniqueness of $\RR$-splitting for $G^*$ and $M^*$. As in the odd case, the unique Langlands--Shelstad normalization  between $M'$ and $M^*$ is inherited from the analogous normalization between $H$ and $G^*$. (Again, one can see this by using Waldspurger's explicit formula, noting that this time the constant $\eta$ is equal to $1$ for both $G^* = \SO(\underline V)$ and $M^{\SO,*} = \SO(\underline W)$.) As in the odd case, the epsilon factor is still $1$ in both the $(H,G^*)$-scenario and the  $(M',M^*)$-scenario (since a maximal $\RR$-split torus in each of $G^*, H, M^*, M'$ is of the form a  direct sum of a split torus and one copy of  $\Uni(1)$). Our claim follows, as in the odd case. 
		
		Finally, assume $d$ is divisible by $4$. As in the previous case we have $\dim W = d-4$, and this is divisible by $4$. Using Waldspurger's explicit formula \cite[\S 1.10]{walds10}, we observe that the normalization between $M' = M^{\GL} \times M^{\prime, \SO}$ and $M^* = M^{\GL} \times M^{\SO, *}$   induced by the Langlands--Shelstad normalization between $M^{\prime, \SO}$ and $M^{\SO, *}$  associated to some $\spl_M \in \Pinn(M^{\SO, *})$ is inherited from the Langlands--Shelstad normalization between $H$ and $G^*$ associated to some $\spl \in \Pinn(G^*)$ provided that $\eta_{\underline W} (\spl_M) = \eta_{\underline V} (\spl)$. Here $\eta_{\underline V}(\cdot) : \Pinn(G^*) \to \set{\pm 1}$ and $\eta_{\underline W}(\cdot) : \Pinn(M^{\SO, *}) \to \set{\pm 1}$ are as in \S \ref{para:preparation for Walds}. We now take $\spl_M$ and $\spl$ such that $\eta_{\underline W}(\spl_M) = \eta_{\underline V} (\spl) = -1$. By the above observation and by Theorem \ref{thm:comparing Waldspurger}, we see that the  Whittaker normalization 
		between $M'$ and $M^*$ 
		is inherited from the Whittaker normalization between $H$ and $G^*$ times the following constant. The constant is the ratio between the  two local epsilon factors appearing in (\ref{eq:Walds 1}) and the analogue of (\ref{eq:Walds 1}) for $(M', M)$. By (\ref{eq:Walds 2}) and a similar computation for $(M',M)$, we see that the two epsilon factors are  equal to 
		$ (-1)^{m^-}$ and $ (-1)^{n^-}$ respectively, where $m^-$ is the absolute rank of $H^-$ and $n^-$ is the absolute rank of $M^{\prime, \SO, -}$. Since we are in the even case, we have $m^- \equiv n^- \mod 2$. Thus   the  Whittaker normalization 
		between $M'$ and $M^*$ 
		is inherited from the Whittaker normalization between $H$ and $G^*$, and our claim follows as in the previous two cases.  
		
		It follows from the above claim and Lemma \ref{lem:tasho and epsilon} that $\tasho(A,B)$ is the product of the following three signs: 
		\begin{enumerate}
			\item the sign between $\Delta_{j_M, B_M} ^{A,B}$ and $\Delta_{j_M, B_M}$, namely $(-1) ^{q(G) + q(H) + q(M) + q(M')}$;
			\item the sign between $\Delta_{j_M, B_M}$ and the Whittaker normalization between $M'$ and $M$, which is also equal to the sign between $\Delta_{j_{M^{\SO}}, B_{M^{\SO}}}$ and the Whittaker normalization between $M^{\prime, \SO}$ and $M^{\SO}$;
			\item the sign between $\Delta _{j, B_{G,H}}$ and the Whittaker normalization between $H$ and $G$. 
		\end{enumerate}
		Denote by $m^{\pm}$ (resp.~$n^{\pm}$) the absolute ranks of $H^{\pm}$ (resp.~ $M^{\prime,\SO, \pm}$). Denote by $m$ the absolute rank of $G$. We divide our computation into cases. 
		
\textbf{The odd case with $M= M_{12}$:}
	In this case $B$ is always $\emptyset$, and $A$ is any subset of $\set{1,2}$. 
		We have  
		\begin{align*}
q(G) & = \frac{(2m-1) 2}{2} = 2m-1, &  q(H) &  = \frac{m^+ (m^+ +1) + m^- (m^- +1)}{2}, \\ q(M) & =0 , &  q(M')&=  \frac{n^+ (n^+ +1) + n^- (n^- +1)}{2}.
		\end{align*}
		When $A = \emptyset$, we have $m^+ = n^+$ and $m^- = n^- +2$. Then  
		\begin{align*}
q(G) + q(H)&  + q(M) +q(M') \\  & = 2m -1 + \frac{2m^+ (m^+ +1) + m^-(m^- +1) + (m^- -2) (m^- -1)}{2}  \\ & \equiv 1 + m^+ (m^+ + 1)  + \frac{2(m^-) ^2 -2 m^- +2}{2} \\ & \equiv 1 + m^+ (m^+ +1) + m^- (m^- - 1)  +1  \\ & \equiv 0 \mod 2. 
		\end{align*}
		When $A = \set{1,2}$, we have $m^+ = n^+ +2$ and $ m^- = n^-$. Observing symmetry we again get 
		$$ q(G) + q(H) + q(M) + q(M') \equiv 0 \mod 2. $$
		
		Now assume $A = \set{1}$ or $\set{2}$. Then $m^+ = n^+ +1, m^- = n^- +1$. We have 
		\begin{align*}
q(G) & + q(H)  +q (M) + q(M') \\ & = 2m -1 + \frac{m^+ (m^+ +1) + m^+ (m^+ -1) + m^- (m^- +1) + m^- (m^- - 1)}{2}  \\  & 
 \equiv 1 + \frac{2(m^+) ^2 + 2(m^-) ^{2}}{2 } \\ & \equiv 1 + m^+ + m^-  \\ 
 & \equiv m +1 \mod 2.\end{align*}
We conclude that 
$$ (-1)^{q(G) +q(H) +q(M) +q(M')} = \begin{cases}
1, ~ &A = \set{1,2} \mbox{ or } \emptyset , \\
(-1) ^{m+1}, ~& A = \set{1} \mbox{ or } \set{2}.
\end{cases}$$

The sign between $\Delta_{j_{M^{\SO}}, B_{M^{\SO}}}$ and the Whittaker normalization is 
$ (-1) ^{\ceil{n^+ /2}}$ by Lemma \ref{lem:jBforM} and the $q=0$ case of Proposition \ref{answer for H and G} (1). (We have already noted in \S\ref{para:BMSO} that the results in \S \ref{sec:transfer factors} indeed applies to $M^{\SO}$ together with its endoscopic group $M^{\prime,\SO}$.)  The sign between $\Delta_{j, B_{G,H}}$ and the Whittaker normalization is  
		$ (-1) ^{\ceil {m^+ /2} + 1}$ by the $q=2$ case of Proposition \ref{answer for H and G} (1); here the hypothesis $m^+>0$ (i.e., $H^+$ is non-trivial) is guaranteed in \S \ref{para:prepare for main}, and the hypothesis that $(j, B_{G,H})$ arises from an element of $\EDNice$ and an element of $\ED(V^+)^o \times \ED(V^-)^o$ is guaranteed in \S \ref{once and for all}.  Thus we have \begin{align*}
\tasho(\emptyset,\emptyset) & = (-1) ^{ \ceil{n^+ /2} + \ceil {m^+ /2} + 1} = (-1)^{ \ceil{m^+ /2} + \ceil {m^+ /2} + 1} = - 1 ,  \\ \tasho(\set{1,2},\emptyset) &  = (-1) ^{ \ceil{n^+ /2} + \ceil {m^+ /2} + 1} = (-1)^{ \ceil{(m^+-2) /2} + \ceil {m^+ /2} + 1} = 1 ,   \\ \tasho(\set{1},\emptyset) = \tasho(\set{2},\emptyset) &  =  (-1) ^{ m+1 + \ceil{n^+ /2} + \ceil {m^+ /2} + 1 }  \\ &  =(-1) ^{m+1 +  \ceil{(m^+ -1 ) /2} + \ceil {m^+ /2} + 1 } = (-1) ^{m+ m^+} = (-1) ^{m^-} .
		 \end{align*}
		This finishes the proof in this case.

	\textbf{The even case with $M= M_{12}$:} In this case $B$ is always $\emptyset$, and $A$ is either $\emptyset$ or $\set{1,2}$. 
Note that $q(G), q(H), q(M), q(M')$ are all even. This is because each of $G, H, M, M'$ is a product of a split torus and one or two cuspidal even special orthogonal group(s), namely some $\SO(a,b)$ with $a,b$ even, for which we have  $q(\SO(a,b)) = ab/2 \equiv 0 \mod 2$. It follows that the sign in part (1) is $1$. 
	
	The sign between $\Delta_{j_{M^{\SO}}, B_{M^{\SO}}}$ and the Whittaker normalization is $ (-1) ^{\floor{n^- /2}}$ by Lemma \ref{lem:jBforM} and the $q=0$ case of Proposition \ref{answer for H and G} (2) and Proposition \ref{answer div by 4}.

		Assume it is not the case that $m$ is odd and $m^+ =1$. Then $\mathcal D^H$ which was used to define $(j, B_{G,H})$ in \S \ref{once and for all} lies in $\EDNice$. We have $m^+ \geq 2$ since $m^+ >0 $ (see \S \ref{para:prepare for main}). Applying the $(q=2, m^+ \geq 2)$ case of Proposition \ref{answer for H and G} (2) and Proposition \ref{answer div by 4}, we see that the sign between $\Delta_{j, B_{G,H}}$ and the Whittaker normalization is 
		$ (-1) ^{\floor {m^- /2}}. $ 
		
		Now assume $m$ is odd and $m^+ =1$. In this case $\mathcal D^H$ used to define $(j, B_{G,H})$ differs from an element of $\EDNice$ by the transposition $ (m-1, m) \in \mathfrak S_m$. Let $B_{G,H}'$ be the image of $B_{G,H}$ under $(m-1, m)$, viewed as an element of the complex Weyl group. An argument similar to the proof of the second statement of Lemma \ref{answer for H and G^* even div by 4} shows that 
		$$ \Delta_{j, B_{G,H}} = \lprod{a_{(m-1, m)} ,s} \Delta_{j, B_{G,H} '}   = - \Delta_{j, B_{G, H} ' }.$$ 
		Hence the sign between $\Delta_{j, B_{G,H}}$ and the Whittaker normalization is $-1$ times the sign $(-1) ^{\floor {m^- /2} +1 }$ in the $(q=2, m^+ =1)$ case of  Proposition \ref{answer for H and G} (2). Namely, it is again 
		$ (-1) ^{\floor {m^- /2 }}. $
	
	We conclude that  $\tasho(A,B) = (-1) ^{\floor {n^- /2} + \floor {m^- /2} }.$ Specifically, 
		\begin{align*}
\tasho(\emptyset,\emptyset) & = (-1) ^{\floor {(m^- -2) /2} + \floor {m^- /2}} =-1 ,  &  \tasho(\set{1,2},\emptyset) & = (-1) ^{ \floor{m^- /2} + \floor {m^- /2}  } = 1.
		\end{align*}
		This finishes the proof in this case.
		
\textbf{The odd case with $M= M_1$:} In this case $A$ is always $\emptyset$, and $B$ is any subset of $\set{1}$. We have
\begin{align*}
q(G) & = \frac{(2m-1) 2}{2} = 2m-1, \\  q(H) &  = \frac{m^+ (m^+ +1) + m^- (m^- +1)}{2}, \\ q(M')  & =  q(M^{
	\prime,\SO}) + q (M ^{\GL}) = \frac{n^+ (n^+ +1) + n^- (n^- +1)}{2} +1 , \\ q(M) & =  q(\GL_2) = 1 . 
\end{align*} 
When $B = \emptyset$, we have $m^+ = n^+, m^- = n^- +2$, and so \begin{align*}
q(G) + q(H) & + q(M) +q(M') \\ &  = 2m +1  + \frac{2m^+ (m^+ +1) + m^-(m^- +1) + (m^- -2) (m^- -1)}{2} \\ & \equiv  m^+ (m^+ + 1)  + \frac{2(m^-) ^2 -2 m^- +2}{2} +1  \\ & \equiv m^+ (m^+ +1) + m^- (m^- - 1) \\ &     \equiv 0 \mod 2.
\end{align*} When $B = \set{1}$, we have $m^+ = n^+ +2 , m^- = n^-$, and observing symmetry we again get 	$$ q(G) + q(H) + q(M) + q(M') \equiv 0 \mod 2. $$ Hence the sign in part (1) is $1$. 

The sign between $\Delta_{j_{M^{\SO}}, B_{M^{\SO}}}$ and the Whittaker normalization is $ (-1) ^{\ceil{n^+ /2}} $ by Lemma \ref{lem:jBforM} and the $q=0$ case of Proposition \ref{answer for H and G} (1). The sign between $\Delta_{j, B_{G,H}}$ and the Whittaker normalization is $ (-1) ^{\ceil {m^+ /2} + 1} $ by the $q=2$ case of Proposition \ref{answer for H and G} (1). Thus $\tasho(A,B) = (-1) ^{ \ceil{n^+ /2} + \ceil {m^+ /2} + 1 }$, and specifically
\begin{align*}
\tasho(\emptyset,\emptyset) &= (-1)^{ \ceil{m^+ /2} + \ceil {m^+ /2} +1 } =  - 1 ,   &  \tasho(\emptyset,\set{1}) &= (-1)^{ \ceil{(m^+-2) /2} + \ceil {m^+ /2} + 1} =   1.
\end{align*}
This finishes the proof in this case.

\textbf{The even case with $M=M_1$:} As in the previous case, $A$ is always $\emptyset$, and $B$ is any subset of $\set{1}$. 
	Now $q(G), q(H)$ are even, and $q(M), q(M')$ are odd. Hence the sign in part (1) is $1$. 
	Similarly as in the even case with $M= M_{12}$ treated before, the sign between $\Delta_{j_{M^{\SO}}, B_{M^{\SO}}}$ and the Whittaker normalization is 
	$ (-1) ^{\floor{n^- /2}} , $
	and the sign between $\Delta_{j, B_{G,H}}$ and the Whittaker normalization is 
	$ (-1) ^{\floor {m^- /2}}. $ 	Thus $\tasho(A,B) = (-1) ^{\floor {n^- /2} + \floor {m^- /2} }$, and specifically 
	\begin{align*}
\tasho(\emptyset,\emptyset)  & = (-1) ^{\floor {(m^- -2) /2} + \floor {m^- /2}} =-1 ,  & \tasho(\emptyset, \set{1}) & = (-1) ^{ \floor{m^- /2} + \floor {m^- /2}  } = 1.
	\end{align*}
	This finishes the proof in this case.
	
	\textbf{The odd case with $M= M_2$:}
	In this case $B$ is always $\emptyset$, and $A$ is any subset of $\set{1}$. We have 
		\begin{align*}
q(G) & = \frac{(2m-1) 2}{2} = 2m-1,\\ q(H) & = \frac{m^+ (m^+ +1) + m^- (m^- +1)}{2}, \\ q(M) &  =\frac{d-3}{2} = m-1 , \\ q(M') & =  \frac{n^+ (n^+ +1) + n^- (n^- +1)}{2}. 
		\end{align*}  
		When $A = \emptyset$, we have $
		m^+ = n^+ , m^- = n^- +1$, and so 
		\begin{align*}
q(G) + q(H) & + q(M) +q(M')\\ & = 3m-2 + \frac{2m^+ (m^+ +1) + m^-(m^- +1) + (m^- -1) m^- }{2} \\ 
& \equiv m + m^+ (m^+ + 1)  + \frac{2(m^-) ^2 }{2} \\ 
 & \equiv m + m^+ (m^+ +1) + (m^- )^{2}  \\ & \equiv  m+ m^- \\ & \equiv m^+ \mod 2.
		\end{align*}
	When $A = \set{1}$, we have $m^+ = n^+ +1, m^- = n^-$, and a similar computation yields 
		$$q(G) + q(H) + q(M) + q(M') \equiv m^- \mod 2.$$
		We conclude that 
		$$ (-1)^{q(G) +q(H) +q(M) +q(M')} = \begin{cases}
		(-1)^{m^+}, ~ &A =   \emptyset , \\
			(-1) ^{m^-}, ~& A = \set{1}.
		\end{cases}$$
	
	The sign between $\Delta_{j_{M^{\SO}}, B_{M^{\SO}}}$ is $ (-1) ^{\floor{n^+ /2}} $ by Lemma \ref{lem:jBforM} and the $q=1$ case of Proposition \ref{answer for H and G} (1). The sign between $\Delta_{j, B_{G,H}}$ and the Whittaker normalization is 
	$ (-1) ^{\ceil {m^+ /2} + 1}$ by the $q=2$ case of Proposition \ref{answer for H and G} (1). Thus we have 
		\begin{align*}
\tasho(\emptyset,\emptyset) & = (-1) ^{m^+ +  \floor{n^+ /2} + \ceil {m^+ /2} +1 } \\ & = (-1)^{ m^+ + \floor{m^+ /2} + \ceil {m^+ /2}+1 } = (- 1)^{m^+ + m^++1} =-1 , \\ \tasho(\set{1},\emptyset) &  =  (-1)^{m^- + \floor {n^+ /2}+ \ceil{m^+/2}+1} = (-1)^{m^- + \floor {(m^+ -1) /2}+ \ceil{m^+/2}+1 }= (-1) ^{m^-}.
		\end{align*}
		This finishes the proof in this case.
\end{proof}
 
\begin{defn}\label{defn:J} For $(A,B)$ as in Lemma \ref{lem:first simplification}, define the sign \index{$\sun(A,B)$} $$\sun(A,B) : = \begin{cases}
		-	1, & \text{if } 1 \in A  \text{ or } 1\in B,  \\
		1 ,& \text{otherwise}.
	\end{cases}$$ 
	Suppose $g$ is a function that assigns to each choice of $(A,B)$ an element $g(A,B)\in \mathcal H^{\ur} (M'_{\QQ_p})$. Define \index{$J(\fke_{\fkp}(M), \gamma', g)$}
	\begin{multline*} J(\fke_{\fkp}(M), \gamma', g) = 
J(\fke_{\fkp}(M),\gamma' , (A,B)\mapsto g(A,B)) \\  : = \sum_{A,B}   \sun (A,B) SO_{\gamma'} (g(A,B) ) \epsilon_{R}(j_M(\gamma'^{-1})) \epsilon_{R_H} (\gamma'^{-1})  \Phi^G_M (j_M(\gamma') ^{-1} , \Theta_{\mathbb V^*} ^H) ,
	\end{multline*}
	where the sum is over all choices of $(A,B)$. 
\end{defn}
\begin{defn}\label{defn:Q}
With notation as in Definition \ref{defn:Sigma_1} and \S \ref{para:Delta_j_M B_M}, we define
	\begin{multline*}
Q(\fke_{\fkp}(M), \gamma') : = \bar \iota^{M'} (\gamma')^{-1}   ( \Delta^{M}_{M'})^{\emptyset,\emptyset} (\gamma',\gamma_M) O_{\gamma_M}^{s_M'} (f^{p,\infty} _M) \tau(M) k(M)k(G)^{-1}   \\ \cdot (-1) ^{\dim A_{M'}}  \bar v ({M'}_{\gamma'} ^0) ^{-1} (-1)^{q(G_{\RR}) } \Delta_{j_M, B_M} ^{\emptyset,\emptyset} (\gamma', j_M(\gamma')) .
	\end{multline*} Here, we choose $\gamma_M\in M(\adele_f^p)$ as in Definition \ref{defn:Sigma_1},  which does not affect the definition.
\end{defn}
\begin{cor}\label{cor:formula with sun}
	With $J$ as in Definition \ref{defn:J} and $Q$ as in Definition \ref{defn:Q}, we have 
	$$
	I(\fke_{\fkp}(M), \gamma') =    Q(\fke_{\fkp}(M), \gamma') J(\fke_{\fkp}(M), \gamma', (A,B)\mapsto f^H_{p,M'}) .		$$
	Here the mapping  $ (A,B)\mapsto f^H_{p,M'}$ is defined via the dependence of $H$ on $(A,B)$ as in Lemma \ref{lem:first simplification}. 
\end{cor}
\begin{proof}
	By (\ref{eq:det(omega_0)}) and Proposition \ref{prop:computing tasho}, we have $$\sun (A,B) = \tasho(A,B) ^{-1}\tasho(\emptyset,\emptyset)\det(\omega_{0}).  $$ The corollary then follows from the second equality in Lemma \ref{lem:debut of I(e, gamma')}, Proposition \ref{pro:K}, and Lemma \ref{lem:tasho and epsilon}. 
\end{proof}

\section{Symmetry of order $n^G_M$}\label{subsec:symmetry}
\begin{defn}\label{defn:W}
	We define a subgroup \index{$\mathfrak W$} $\W \subset  \Aut (M^{\GL})$ as follows. When $M= M_{12}$, so that $M^{\GL} = \GG_m^2$, we define $\W$ to be $\set{\pm 1} ^{2} \rtimes \mathfrak S_2,$ where each factor $\set{\pm 1}$ acts on each factor $\GG_m$ non-trivially and $\mathfrak S_2$ acts by swapping the two copies of $\GG_m$. When $M = M_1$, so that $M^{\GL} = \GL_2$, we define $\W$ to be $\ZZ/2\ZZ$ with the non-trivial element acting on $\GL_2$ by transpose inverse. When $M = M_2$ in the odd case, so that $M^{\GL } = \GG_m$, we define $\W$ to be equal to $\Aut (M^{\GL}) = \Aut(\GG_m) = \ZZ/2\ZZ$.  When the context is clear we also view $\W$ as a subgroup of $\Aut (M) $ or $\Aut (M')$, by extending its action on $M^{\GL}$ trivially across $M^{\SO}$ or $M ^{\prime, \SO}$. 	
\end{defn}

\begin{lem}\label{lem:comp n^G_M}
The natural homomorphism $\W \to \Aut (A_M)$ is an injection, and its image is equal to the image of $\Nor_G(M)(\QQ)$ in $\Aut (A_M)$. In particular, $\W$ is naturally isomorphic to $\cW^G_M$ and $\abs{\W} = n^G_M$ (see Definition \ref{defn:n^G_M} and Remark \ref{rem:Weyl group of Levi}). 
\end{lem}
\begin{proof}This is straightforward to check. 
\end{proof}

\subsection{}\label{convention on W}
The action of $\W$ on the set of stable conjugacy classes in $M'(\QQ)_{\mathrm{ss}}$
preserves the following conditions: 
\begin{itemize}
	\item being $\RR$-elliptic,
	\item being $(M,M')$-regular,
	\item being an image of a semi-simple element of $M(\adele_f^p)$. 
\end{itemize} (Indeed, the only non-trivial assertion here is that $\W$ preserves being $\RR$-elliptic in the case $M = M_1$, and this follows from the fact that $M_1^{\GL} = \GL_2$ contains the $\RR$-elliptic maximal torus $T_{\GL_2}^{\std}$ which is $\W$-stable.) Moreover, if $M = M_{12}$ or $M_2$, then two different elements of $M'(\QQ)_{\mathrm{ss}}$ in the same $\W$-orbit are never stably conjugate to each other. Therefore in these cases we may and shall assume that the sets $\Sigma (M')$ and $\Sigma(M')_1$ chosen in Definitions \ref{defn:Sigma} and  \ref{defn:Sigma_1} are stable under $\W$. If $M  = M_1$, then two different $\RR$-elliptic elements of $M'(\QQ)$ in the same $\W$-orbit are either not stably conjugate to each other, or such that their components in $M^{\GL}(\QQ) = \GL_2(\QQ)$ both have determinant $1$. (To see this, note that if $g\in \GL_2(\QQ)_{\mathrm{ss}}$ is stably conjugate to its transpose inverse, then $\det g = \pm 1$, and we have $\det g >0$ if $g$ is $\RR$-elliptic.) Therefore in this case we may and shall assume that $\Sigma(M')_ 1$ contains a subset \index{$\Sigma(M')_2$} $\Sigma(M')_2$ such that $\Sigma(M')_2$ is stable under $\W$ and the component in   $M^{\GL}(\QQ)$ of every element of $\Sigma(M')_1 - \Sigma(M')_2$ has determinant $1$. To unify notation, when $M = M_{12}$ or $M_2$, we set $\Sigma(M')_2$ to be $\Sigma(M')_1$.

\begin{lem}\label{lem:symmetry of Q} For $\gamma' \in \Sigma(M')_2$ and $w \in \W$, we have $Q(\mathfrak e_{\fkp}(M), \gamma') = Q(\mathfrak e_{\fkp}(M), w(\gamma'))$. (See Definition \ref{defn:Q} for $Q$). 
\end{lem}
\begin{proof}  
	By (\ref{eq:Levi transfer away from p infty}), we have 
	$$Q(\mathfrak e_{\fkp}(M), \gamma') = C\times  SO_{\gamma'} (f^{H(\emptyset,\emptyset), p,\infty} _{M'}) \Delta_{j_M, B_M} ^{\emptyset,\emptyset} (\gamma', j_M(\gamma')),$$ where $C$ is an expression that is invariant under $\W$, and $H(\emptyset,\emptyset)$ is the particular choice of $H$ arising from $(A,B) = (\emptyset, \emptyset)$. Note that the subgroup $\W \subset \Aut (M') $ is contained \footnote{Note that this would no longer be not true, for instance, if $H(\emptyset,\emptyset)$ is replaced by the choice of $H$ arising from $(A,B) = (\set{1}, \emptyset)$ when $M = M_{12}$.} in the image  of the natural map $\Nor_{H(\emptyset,\emptyset)} (M')(\QQ) \to \Aut (M'). $ 
	
Since $w$ comes from $\Nor_{H(\emptyset,\emptyset)}(M')(\QQ)$, we have  $$SO_{\gamma'} (f^{H(\emptyset,\emptyset), p,\infty} _{M'}) = SO_{w(\gamma')} (f^{H(\emptyset,\emptyset), p,\infty} _{M'})$$ by exactly the same argument (using Kazhdan density and descent) as in the proof of Lemma \ref{lem:elementary}.
	We are left to check $$ \Delta _{j_M, B_M} ^{\emptyset,\emptyset} (\gamma', j_M(\gamma')) = \Delta _{j_M, B_M} ^{\emptyset,\emptyset} (w(\gamma'), j_M(w(\gamma')) ), $$ or equivalently,
	\begin{align*}
\Delta _{j_M, B_M}  (\gamma', j_M(\gamma')) = \Delta _{j_M, B_M} (w(\gamma'), j_M(w(\gamma')) ).
	\end{align*} The last equality holds because both sides depend only on the common component of $\gamma'$ and $w(\gamma')$ in $M^{\prime,\SO}$. More precisely, if we denote this common component by $\gamma^{\prime,\SO}$ , then both sides are equal to $$ \Delta _{j_{M^{\SO}}, B_{M^{\SO}}}  (\gamma^{\prime,\SO}, j_{M^{\SO}}(\gamma^{\prime, \SO})) ,$$ where $j_{M^{\SO}}$ and $B_{M^{\SO}}$ are as in \S \ref{para:B_0} and \S \ref{para:BMSO}. This finishes the proof. 	
\end{proof}
\begin{prop}\label{cor:breaking symmetry} For each $\mathfrak e = (M', s_M', \eta_M) \in \mathscr E (M)$, choose a set $\Sigma(M')_1$ as in Definition \ref{defn:Sigma_1} and \S \ref{convention on W}. In view of Lemma \ref{lem:symmetry of Q}, for each $\gamma' \in \Sigma(M')_1$ we write $Q(\mathfrak e, \W \gamma')$ for $Q(\mathfrak e, \gamma')$.  
 We have 
		\begin{multline*}
	 n^G_M \Tr'_M = \sum_{\mathfrak e_{\fkp}(M) = (M',\lang M', s_M', \eta_M)  \in \dot {\mathscr E } (M) ^{c,\ur}} \abs{\Out _M(\mathfrak e_{\fkp} (M)) } ^{-1}  \\ \cdot \bigg (  \sum_{\gamma'\in \Sigma(M')_2}  Q( \mathfrak e_{\fkp}(M),  \gamma')\abs{\W}^{-1} \sum_{w\in \W}  J(\mathfrak e_{\fkp}(M), w(\gamma'), (A,B)\mapsto f^H_{p,M'}) \\ + \sum_{\gamma'\in \Sigma(M')_1- \Sigma(M')_2} Q( \mathfrak e_{\fkp}(M),  \gamma') J (\mathfrak e_{\fkp}(M), \gamma', (A,B)\mapsto f^H_{p,M'}) \bigg).
		\end{multline*}
\end{prop}
\begin{proof}
	This is a consequence of Lemma \ref{lem:debut of I(e, gamma')}, Corollary \ref{cor:formula with sun}, Lemma \ref{lem:symmetry of Q}. 
\end{proof}

\section{Computation of $J$}\label{subsec:comp of J}
We compute the term $J(\mathfrak e_{\fkp}(M), \gamma', (A,B)\mapsto f^H_{p,M'})$ in Corollary \ref{cor:formula with sun} using results from \S \ref{subsec:p}. We simply write $\fke$ for $\fke_{\fkp}(M)$. Recall   the functions $\epsilon_R(\cdot): T_M(\RR) \to \set{\pm 1}$ and $\epsilon_{R_H}(\cdot) : T_M(\RR) \to \set{\pm 1}$ from \S \ref{para:Delta_j_M B_M}. The former depends only on $\fke_{\fkp}(M)$, while the latter depends on  $\fke_{\fkp}(M)$ and $(A,B)$.  
\begin{lem}\label{lem:epsilon}
	For $M = M_{12}$ in the odd case, we have 
	\begin{align*}
 \epsilon_{R}(j_M(\gamma'^{-1})) & = \epsilon_{R_H} (\gamma'^{-1}) |_{A= \set{1,2}} =\epsilon_{R_H} (\gamma'^{-1}) |_{A= \emptyset}  , \\  \epsilon_{R_H} (\gamma'^{-1})|_{A= \set{1}} & = \epsilon_{R_H} (\gamma'^{-1})|_{A= \set{2}} .
	\end{align*}	
	In all the other cases of $M$, we have 
	$$  \epsilon_{R}(j_M(\gamma'^{-1})) = \epsilon_{R_H} (\gamma'^{-1}).$$
\end{lem}
\begin{proof}
This follows directly from the definitions.
\end{proof}

\subsection{}\label{para:J notations}
We introduce some notations.
	Write\index{$p^*$} $p^* : = p^{a (d-2)/2}$. 
Write $$f^H_{p,M'} = p^* k (A,B) + p^* h \in \mathcal H^{\ur}(M'_{\QQ_p}),$$ as in Proposition \ref{prop:main computation at p}. When $M= M_{12}$, we further write 
$$k(A,B) = k_1(A) + k_2(A) ,$$
where $k_i(A) \in \mathcal H^{\ur}(M'_{\QQ_p}) $ has Satake transform $\nabla_i(A) (\xi_i^a + \xi_i^{-a})$, as in Proposition \ref{prop:main computation at p}. Thus we have 
\begin{align}\label{eq:two terms}
J(\mathfrak e , \gamma', (A,B)\mapsto f^H_{p,M'}) = p^* J(\mathfrak e , \gamma', k) + p^* J (\mathfrak e , \gamma', h) ,
\end{align}
and when $M= M_{12}$ we further have 
\begin{align}\label{eq:three terms}
J(\mathfrak e , \gamma', ( A,B) \mapsto f^H_{p,M'}) = p^* J(\mathfrak e , \gamma', k_1) + p^* J(\mathfrak e , \gamma', k_2) +p^* J(\mathfrak e , \gamma', h) .
\end{align} 
 Here we use the abbreviated  notation
$  J(\mathfrak e , \gamma', k) :=  J(\mathfrak e , \gamma', (A,B)\mapsto k(A,B)),$ etc. 
In the following computation We write 
\begin{align*}
\Phi ^G_M & := \Phi ^G_M (j_M(\gamma') ^{-1}, \Theta_{\mathbb V^*}), \\  \Phi ^G_{M,\Endos}  & := \Phi ^G_M (j_M(\gamma') ^{-1}, \Theta_{\mathbb V^*})_{\Endos}, \\ \epsilon_R \epsilon_{R_H} & : = \epsilon_R(j_M(\gamma'^{-1})) \epsilon_{R_H} (\gamma'^{-1}).
\end{align*}
\subsection{Odd case $M_{12}$}
\label{para:J odd case M12}With the above notations, it follows from Lemma \ref{lem:epsilon} and the fact that $\epsilon_R$ is independent of $(A,B)$ that we have \begin{align}\label{eq:Jh=0}
J(\mathfrak e,\gamma', h) &  = SO_{\gamma'} (h) \sum_{A,B} \sun(A,B) \epsilon_R\epsilon_{R_H} \Phi^G_M (j_M(\gamma')^{-1}, \Theta_{\mathbb V^*}^H ) \\ \nonumber &  = SO_{\gamma'}(h) \bigg [  \Phi^G_M - \Phi^G_M  - (\epsilon_R \epsilon_{R_H})|_{A=\set{1}} \Phi^G _{M,\Endos}  +  (\epsilon_R \epsilon_{R_H})|_{A=\set{2}} \Phi^G_{M,\Endos}\bigg ]\\ \nonumber &   = 0 . 
\end{align}
Similarly \begin{align} \label{eq:Jk_1}
J(\mathfrak e & ,\gamma', k_1) \\  \nonumber & = SO_{\gamma'}(k_1(\emptyset)) \bigg [  \Phi^G_M + \Phi^G_M  + (\epsilon_R \epsilon_{R_H})|_{A=\set{1}} \Phi^G _{M,\Endos}  +  (\epsilon_R \epsilon_{R_H})|_{A=\set{2}} \Phi^G_{M,\Endos}\bigg ]\\ \nonumber & = 2SO_{\gamma'} (k_1(\emptyset)) \bigg [  \Phi^G_M +  (\epsilon_R \epsilon_{R_H})|_{A=\set{1}} \Phi^G_{M, \Endos} \bigg ],
\end{align}
and 
\begin{align}\label{eq:Jk2}
	 J(\mathfrak e,\gamma', k_2) =   2SO_{\gamma'} (k_2(\emptyset)) \bigg [  \Phi^G_M-  (\epsilon_R \epsilon_{R_H})|_{A=\set{1}} \Phi^G_{M,\Endos} \bigg ] .
\end{align} 
\subsection{Even case $M_{12}$}\label{para:J even M12}
With similar computations as above, we get 
\begin{align*}
J(\mathfrak e, \gamma ' , h) & =0 , \\ J(\mathfrak e, \gamma', k_1) & = 2 SO_{\gamma'} (k_1(\emptyset)) \Phi^G_M, \\  J(\mathfrak e, \gamma', k_2) & = 2 SO_{\gamma'} (k_2(\emptyset)) \Phi^G_M.
\end{align*}
\subsection{Case $M_1$ and odd case $M_2$} Similar computations give 
\begin{align}\label{eq:Jh vanish M1 M2}
J(\mathfrak e, \gamma' , h) & = 0 \\  \label{eq:Jk M1 M2} J(\mathfrak e, \gamma', k) & = 2 SO_{\gamma'} (k(\emptyset,\emptyset)) \Phi^G_M .
\end{align}
\section{Breaking symmetry, case $M_{12}$}\label{pf:break 1}
We keep the notation in Proposition \ref{cor:breaking symmetry} and \S \ref{subsec:comp of J}. 
\begin{defn} Suppose $M = M_{12}$. 
	We say that an element of $ M'(\RR)$ is \emph{good at $\infty$}\index[n]{good at $\infty$} if its component in $M^{\GL}(\RR) = \RR^{\times } \times \RR^{\times}$ lies in $(\RR_{>0}\times \RR_{>0}) \cup (\RR_{<0} \times \RR_{<0})$. We say that an element of $ M'(\QQ_p)$ is \emph{good at $p$}\index[n]{good at $p$} if its component in $M^{\GL}(\QQ_p) = \QQ_p^{\times } \times \QQ_p^{\times}$ has $p$-adic valuations $(-a, 0)$. Here the first $\GG_m$-factor is $\GL(V_1)$ and the second is $\GL(V_2/V_1)$. 
\end{defn}

\begin{prop}\label{cor:symmetry broken M12}
Let $M= M_{12}$. We have 
 \begin{multline*}  n^G_M \Tr'_M = \sum_{\mathfrak e_{\fkp}(M) = (M',\lang M', s_M', \eta_M)  \in \dot {\mathscr E } (M) ^{c,\ur}} \abs{\Out _M(\mathfrak e_{\fkp} (M)) } ^{-1} \\ \cdot    \sum_{\gamma'}  Q(\mathfrak e_{\fkp}(M), \gamma') 4 p^*  J(\mathfrak e_{\fkp}(M), \gamma', k_1),
 \end{multline*}
 where $\gamma'$ runs through the elements of $\Sigma(M')_{1}$ that are good at $\infty$ and good at $p$.
\end{prop}
\begin{proof}We start with the formula for $n^G_M \Tr'_M$ in Proposition \ref{cor:breaking symmetry}, and recall that in that formula $\Sigma(M')_1 = \Sigma(M')_2$ for our $M = M_{12}$.
Fix  $\fke = \fke_{\fkp}(M) = (M', \lang M', s_M', \eta_M) \in \dot {\mathscr E}(M)^{c,\ur}$. 

We first treat the odd case. Let $w_1: = (-1, 1) \in \set{\pm 1}^2 \subset \W = \set{\pm 1} ^2 \rtimes \mathfrak S_2$, and let $w_{12}$ be the non-trivial element of $\mathfrak S_2 \subset \W$. For $\gamma' \in \Sigma(M')_1$, combining the computation of $J(\mathfrak e , \gamma', k_1)$ and $J(\mathfrak e, \gamma', k_2)$ in \S \ref{para:J odd case M12} and the vanishing statement in Proposition \ref{Arch comp odd ++}, we know that $$J(\mathfrak e, \gamma', k_1) = J(\mathfrak e, \gamma', k_2) =0 $$ unless $\gamma'$ is good at $\infty$.  We also note that being good at $\infty$ is a property invariant under $\W$. Now by (\ref{eq:three terms}) and (\ref{eq:Jh=0}) we have 
\begin{align}\label{eq:formula for J}
J(\mathfrak e, \gamma', (A,B) \mapsto f^H_{p,M'}) = p^*  J(\mathfrak e , \gamma', k_1) + p^*  J(\mathfrak e , \gamma', k_2).
\end{align}
Therefore, if $\gamma' \in \Sigma(M')_1$ is such that  
\begin{align}\label{eq:supposing nonvanishing M12 odd}
\sum_{w\in \W} J(\mathfrak e, w(\gamma'), (A,B) \mapsto f^H_{p,M'}) \neq 0 ,
\end{align}
then $\gamma'$ is good at $\infty$,

Suppose $\gamma' \in \Sigma(M')_1$ is good at $\infty$. Then by Proposition \ref{effect of switching to Phi_endos}, (\ref{eq:Jk_1}), and (\ref{eq:Jk2}),   we have
\begin{align}\label{inv:J w12}
J(\mathfrak e , \gamma', k_1) & = J(\mathfrak e, w_{12} (\gamma'), k_2) \\ \label{inv:J w1}
J(\mathfrak e , \gamma', k_1) & = J(\mathfrak e , w_1(\gamma'), k_1)
\end{align}
  because the functions  $k_1(\emptyset)$ and $k_2(\emptyset)$ are pull-backs of each other under $w_{12}$, and the function $k_1(\emptyset)$ is invariant under $w_1$.
  Combining (\ref{eq:formula for J}) and (\ref{inv:J w12}), we obtain 
  \begin{align}\label{eq:twice the sum}
\sum_{w\in \W} J(\mathfrak e, w(\gamma'), (A,B) \mapsto f^H_{p,M'}) & = \sum _{w\in \W} p^* J(\mathfrak e, w(\gamma'), k_1) + p^* J(\mathfrak e, w_{12} w (\gamma'),  k_1) \\
\nonumber & = 2 \sum_{w\in \W} p^* J(\mathfrak e, w(\gamma'), k_1).
  \end{align}
  
   Assume (\ref{eq:supposing nonvanishing M12 odd}) holds. Then by (\ref{eq:twice the sum}), there exists $\gamma'' \in \W \gamma' $ such that $J(\mathfrak e, \gamma'', k_1) \neq 0.$ 
By (\ref{eq:Jk_1}), the last condition implies that $SO_{\gamma''} (k_1(\emptyset)) \neq 0$, from which it easily follows that either $\gamma''$ or $w_1 (\gamma'')$ (but not both) is good at $p$. Note that in $\W$, there are either zero or two elements $w$ such that $w(\gamma')$ is good at $p$. In the latter case, the two elements differ by left multiplication by $w_2 := (1,-1) \in \set{\pm 1}^2 \subset  \W$. Combining this analysis with (\ref{inv:J w1}) and (\ref{eq:twice the sum}),  we have 
\begin{align}\label{eq:computing sum over W orbit}
 \sum_{w\in \W} &  J(\mathfrak e, w(\gamma'), ( A,B) \mapsto f^H_{p,M'})  \\ \nonumber & =  2  p ^* \sum _{w\in \W, w (\gamma') \text{ good at }p} J(\mathfrak e, w(\gamma'), k_1) +  J(\mathfrak e, w_1w(\gamma'), k_1) \\ \nonumber &  = 4 p ^*  \sum _{w\in \W, w (\gamma') \text{ good at }p} J(\mathfrak e, w(\gamma'), k_1) \\ \nonumber &   = \begin{cases}
 0 , & \text{if } \nexists \gamma'' \in \W \gamma' \text{ good at }p , \\
 4p^* \big ( J(\mathfrak e , \gamma'', k_1) + J(\mathfrak e, w_2 (\gamma''), k_1) \big ), & \text{if } \gamma'' \in \W \gamma' \text{ is good at }p.
 \end{cases}
\end{align}
Moreover, if $\gamma '' \in \W \gamma'$ is good at $p$, then we have 
\begin{align}\label{eq:card of W orbit}
\abs{\W \gamma'} = \begin{cases}
\abs{\W} , & \text{if }w_{2} (\gamma'') \neq \gamma'', \\
\abs{\W}/2, &\text{if } w_{2}(\gamma'') = \gamma''.
\end{cases}
\end{align}

Combining the discussion about being good at $\infty$ at the beginning of the proof,  the formulas (\ref{eq:computing sum over W orbit}) and (\ref{eq:card of W orbit}), and Lemma \ref{lem:symmetry of Q}, we obtain: 
\begin{align*} 
\sum_{\gamma'\in \Sigma(M')_1} &  Q( \mathfrak e,  \gamma')\abs{\W}^{-1} \sum_{w\in \W}  J(\mathfrak e, w(\gamma'), ( A, B) \mapsto f^H_{p,M'}) \\ &  = \sum_{\substack{\gamma''\in \Sigma(M')_1 \\ \gamma'' \text{ good at }p,\infty \\ \gamma'' \neq w_2 \gamma''}}  Q( \mathfrak e,  \gamma'')\abs{\W}^{-1}  \abs{\W \gamma''} 4p^*  J(\mathfrak e, \gamma'', k_1)  \\   & + \sum_{\substack{\gamma''\in \Sigma(M')_1 \\ \gamma'' \text{ good at }p,\infty \\ \gamma'' = w_2 \gamma''}}  Q( \mathfrak e,  \gamma'')\abs{\W}^{-1}  \abs{\W \gamma''} 8p^*  J(\mathfrak e, \gamma'', k_1)  \\  &  = \sum_{\substack{\gamma''\in \Sigma(M')_1 \\ \gamma'' \text{ good at }p,\infty}}  Q( \mathfrak e,  \gamma'') 4p^*  J(\mathfrak e, \gamma'', k_1).
\end{align*}This together with Proposition \ref{cor:breaking symmetry} implies the current proposition in the odd case. 

The even case is proved in a similar way. The only differences   are that we now use the vanishing statement in Proposition \ref{Arch comp M12 even} rather than Proposition \ref{Arch comp odd ++}, and that we simply use the invariance of $\Phi^G_M(\cdot, \Theta_{\mathbb V^*})$ under $\Nor_G (M) (\RR)$ to deduce (\ref{inv:J w12}) and (\ref{inv:J w1}) .
\end{proof}

\section{Breaking symmetry, case $M_{1}$ and odd case $M_2$}\label{pf:break 2}
We keep the notation in Proposition \ref{cor:breaking symmetry} and \S \ref{subsec:comp of J}. 
\begin{defn}
	Suppose $M = M_1$. 
	We say that an element of $ M'(\QQ_p)$ is \emph{good at $p$}\index[n]{good at $p$} if its component in $M^{\GL}(\QQ_p) = \GL_2(\QQ_p)$ has determinant of $p$-adic valuation $-a$. We say that all elements of $M'(\RR)$ are \emph{good at $\infty$}.\index[n]{good at $\infty$}
	
	Suppose $M= M_2$ in the odd case. We say that an element of $ M'(\QQ_p)$ is \emph{good at $p$}\index[n]{good at $p$} if its component in $M^{\GL}(\QQ_p) =\QQ_p^{\times} $ has valuation $-a$. We say that an element of $M'(\RR)$ is \emph{good at $\infty$}\index[n]{good at $\infty$} if its component in $M^{\GL} (\RR) = \RR^{\times}$ is positive.
\end{defn}

\begin{prop}\label{cor:symmetry broken M1 and M2}
Suppose $M = M_1$, or $M= M_2$ in the odd case. We have 
	\begin{multline*}
 n^G_M \Tr'_M = \sum_{\mathfrak e_{\fkp}(M) = (M',\lang M', s_M', \eta_M)  \in \dot {\mathscr E } (M) ^{c,\ur}} \abs{\Out _M(\mathfrak e_{\fkp} (M)) } ^{-1} \\ \cdot  \sum_{\gamma'}  Q(\mathfrak e_{\fkp}(M), \gamma') 2 p^*  J(\mathfrak e_{\fkp}(M), \gamma', k),
	\end{multline*}
	where $\gamma'$ runs through the elements of $\Sigma(M')_{1}$ that are good at $\infty$ and good at $p$.
\end{prop}
\begin{proof}We start with the formula for $n^G_M \Tr'_M$ in Proposition \ref{cor:breaking symmetry}. Fix  $\fke = \fke_{\fkp}(M) = (M', \lang M', s_M', \eta_M) \in \dot {\mathscr E}(M)^{c,\ur}$. Let $\gamma' \in \Sigma(M')_1$.
	Let $w_1\in \W$ be the non-trivial element.
 In view of (\ref{eq:Jk M1 M2}), it follows from the obvious invariance of $k(\emptyset,\emptyset)$ under $w_1$ and the invariance of $\Phi^G_M(\cdot, \Theta_{\mathbb V^*})$ under $\Nor_G (M)(\RR)$ that we have
	\begin{align}\label{inv:J w, case M_1}
	J(\mathfrak e , \gamma', k) = J(\mathfrak e, w_1(\gamma'), k).
	\end{align} By (\ref{eq:two terms}) and (\ref{eq:Jh vanish M1 M2}) we have 
\begin{align}\label{eq:J=J}
J(\mathfrak e, \gamma', (A,B) \mapsto f^H_{p,M'}) = p^*  J(\mathfrak e , \gamma', k).
\end{align}
If this is non-zero, then $SO_{\gamma'} (k(\emptyset,\emptyset)) \neq 0$ by (\ref{eq:Jk M1 M2}), and it easily follows that either $\gamma'$ or $w_1 \gamma'$ is good at $p$. This implies that $\gamma' \in \Sigma(M')_2$ (as $a \geq 1$). Thus 
\begin{align}\label{eq:non-inv term}
	\sum_{\gamma'\in \Sigma(M')_1- \Sigma(M')_2} Q( \mathfrak e,  \gamma') J (\mathfrak e, \gamma', (A,B)\mapsto f^H_{p,M'}) = 0.
\end{align}  

Now suppose $\gamma' \in \Sigma(M')_2$. By (\ref{inv:J w, case M_1}) and (\ref{eq:J=J}) we have $$\sum_{w\in \W} J(\mathfrak e, w(\gamma'), (A,B) \mapsto f^H_{p,M'}) = 2p^* J(\mathfrak e, \gamma', k) = 2p^* J(\mathfrak e, w_1(\gamma'), k).$$ Suppose this is non-zero. Then one of $\gamma'$ and $w_1 (\gamma')$ is good at $p$, by the same argument as before. Also, 	by (\ref{eq:Jk M1 M2}), we have $\Phi^G_M(j_M(\gamma')^{-1}, \Theta_{\mathbb V^*}) \neq 0$. By the vanishing statement in Proposition \ref{Arch comp M2}, the last condition implies that $\gamma'$ (and hence also $w_1(\gamma')$) is good at $\infty$ when $M = M_2$.  Note that at most one of $\gamma'$ and $w_1 (\gamma')$ can be good at $p$. Hence 
\begin{align}\label{eq:inv term}
	\sum_{\gamma'\in \Sigma(M')_2}  &  Q( \mathfrak e,  \gamma')\abs{\W}^{-1} \sum_{w\in \W}  J(\mathfrak e, w(\gamma'), ( A, B) \mapsto f^H_{p,M'}) \\ \nonumber &  = \sum_{\substack{\gamma'\in \Sigma(M')_2 \\ \gamma' \text{ good at }p,\infty}}  Q( \mathfrak e,  \gamma') 2 ^{-1}  2p^*  J(\mathfrak e, \gamma', k) \\ \nonumber &  + \sum_{\substack{\gamma'\in \Sigma(M')_2 \\ w_1(\gamma') \text{ good at }p,\infty}}  Q( \mathfrak e,  \gamma') 2 ^{-1}  2p^*  J(\mathfrak e, w_1(\gamma'), k)  \\ \nonumber &  = \sum_{\substack{\gamma'\in \Sigma(M')_2 \\ \gamma' \text{ good at }p,\infty}}  Q( \mathfrak e,  \gamma') 2 ^{-1}  2p^*  J(\mathfrak e, \gamma', k)  \\  \nonumber & + \sum_{\substack{\gamma'\in \Sigma(M')_2 \\ \gamma' \text{ good at }p,\infty}}  Q( \mathfrak e,  \gamma') 2 ^{-1}  2p^*  J(\mathfrak e, \gamma', k) \\ \nonumber & = \sum_{\substack{\gamma'\in \Sigma(M')_2 \\ \gamma' \text{ good at }p,\infty} }  Q( \mathfrak e,  \gamma')  2p^*  J(\mathfrak e, \gamma', k)    .  
					\end{align} Here for the second equality, we made the substitution $\gamma' \mapsto w_1(\gamma') $ in the second summation and used Lemma \ref{lem:symmetry of Q}. 
 The proposition follows from Proposition \ref{cor:breaking symmetry}, (\ref{eq:non-inv term}), and (\ref{eq:inv term}).  
\end{proof}

   \section{Main computation}\label{pf:main}
  We keep letting $M$ denote one of $M_1, M_2, M_{12}$, and excluding $M_2$  in the even case. \begin{prop}
Let \index{$\mathscr C_M$} $\mathscr C_M = 1$ for $M= M_{12}$ or $M_1$, and let $\mathscr C_M = 2$ for $M = M_2$ (in the odd case).  When $a\in \ZZ_{>0}$ is large enough (for a fixed $f^{p,\infty}$), we have 
\begin{multline}\label{eq:pre-main computation}
\mathscr C_M  \Tr'_M = 4p^* \sum_{\mathfrak e_{\fkp}(M) =(M',\lang M', s_M',\eta_M) \in \dot {\mathscr E} (M)^{c,\ur} }(-1) ^{\dim A_{M'}+ q(G_{\RR}) } \abs{\Out _M(\mathfrak e_{\fkp}(M) ) } ^{-1}    \\  \cdot  \sum_{\gamma'}  Q(\mathfrak e_{\fkp} (M), \gamma') SO_{\gamma'} (k_1(\emptyset)) L_M(j_M(\gamma'))  ,
\end{multline}
where $\gamma'$ runs through the elements of $\Sigma(M') _1$ that are good at $\infty$ and good at $p$. Here we understand that $k_1(\emptyset) : = k(\emptyset,\emptyset)$ when $M = M_1$ or $M_2$; see \S \ref{para:J notations} for $k_1$ and $k$. Moreover,  $(-1)^{\dim A_{M'}}$ depends only on $M$, and is $1$ if $M= M_{12}$ and $-1$ other wise.\footnote{This dichotomy is to be compared with the dichotomy of behaviors of signs in Propositions \ref{Arch comp odd ++} and \ref{Arch comp M12 even} for $M_{12}$ on the one hand, and in Propositions \ref{Arch comp M_1} and \ref{Arch comp M2} for $M_1$ and $M_2$ on the other hand.}
   \end{prop}
   \begin{proof}The claim about $(-1) ^{\dim A_{M'}}$ is straightforward. To prove (\ref{eq:pre-main computation}),  we first treat the odd case with $M= M_{12}$. By Lemma \ref{lem:comp n^G_M}, we have $n^G_M = 8$. Then by Proposition \ref{cor:symmetry broken M12} and (\ref{eq:Jk_1}), we have \begin{align}\label{eq:premain}
 8 \Tr'_M & = \sum_{\mathfrak e =(M',\lang M', s_M',\eta_M) \in \dot {\mathscr E} (M)^{c,\ur} }  \abs{\Out _M(\mathfrak e) } ^{-1} \sum_{\gamma'}  Q(\mathfrak e, \gamma') 4 p^*  J(\mathfrak e, \gamma', k_1) \\ \nonumber  & = \sum_{\mathfrak e}   \abs{\Out _M(\mathfrak e) } ^{-1}   \sum_{\gamma'}  Q(\mathfrak e, \gamma') 8 p^*  SO_{\gamma'} (k_1(\emptyset)) \cdot \\ \nonumber  & \cdot \bigg [  \Phi ^G_M (j_M(\gamma') ^{-1} , \Theta_{\mathbb V^*}) + \epsilon_R(j_M(\gamma'^{-1})) \epsilon _{R_H} (\gamma'^{-1})| _{A= \set{1}}  \Phi ^G_M (j_M(\gamma'^{-1}), \Theta_{\mathbb V^*}) _{\Endos}\bigg ]  ,
   	\end{align}
   where $\gamma'$ runs through the elements in $\Sigma(M') _1$ that are good at $\infty$ and good at $p$. 
   Suppose that $\gamma'$ contributes non-trivially to the above sum. Then $Q(\mathfrak e, \gamma') \neq 0$. From Definition \ref{defn:Q}, we have 
   $$O^{s_M'} _{\gamma_M} (f^{p,\infty} _M) \neq 0, $$ where $\gamma_M$ is as in that definition. 
   Therefore the component of $\gamma_M$ in $M^{\GL}(\adele_f^p)$ lies in a compact subset that depends only on $f^{p,\infty}$ and not on $a$. Because $\gamma'$ is an image of $\gamma_M$, the component of $\gamma'$ in $M^{\GL} (\QQ)$ is equal to the component of $\gamma_M$ in $M^{\GL} (\adele_f^p)$. When $a$ is large enough, this observation together with the assumption that $\gamma'$ is good at $p$ implies that the real absolute value of the component of $\gamma'$ in the first $\GG_m$ is strictly smaller than the $\pm 1$-st power of that of the second. In other words, $j_M(\gamma')$ is in the range $x_1 < -\abs{x_2}$ considered in Propositions \ref{Arch comp odd ++} and  \ref{Arch comp M12 even}. Observe 
   \begin{align*}
\Phi^G_M(j_M(\gamma') ^{-1} ,\Theta_{\mathbb V^*}) &  = \Phi^G_M (j_M (\gamma') , \Theta_{\mathbb V}) ,\\ \Phi^G_M(j_M(\gamma') ^{-1} ,\Theta_{\mathbb V^*})_{\Endos} &  = \Phi^G_M (j_M (\gamma') ,  \Theta_{\mathbb V})_{\Endos},
   \end{align*}
and $$ \epsilon _R (j_M(\gamma'^{-1})) \epsilon_{R_H} (\gamma'^{-1})|_{A=\set{1}} = 1  $$ for $j_M(\gamma'^{-1})$ in the range mentioned above. 
  Therefore by Proposition \ref{Arch comp odd ++}, the sum in the bracket in  (\ref{eq:premain}) is $4 (-1)^{q(G_{\RR})} L_M(j_M(\gamma'))$. Substituting this into (\ref{eq:premain}), dividing both sides by $8$, and inserting the sign $(-1)^{\dim A_{M'}} = 1$ on the right hand side, we obtain the desired (\ref{eq:pre-main computation}). 
   
   The even case $M_{12}$, odd and even case $M_1$, and odd case $M_2$, are proved in a similar way, by applying the corresponding computation in \S \ref{subsec:comp of J} and Propositions \ref{cor:symmetry broken M1 and M2}, \ref{Arch comp M12 even}, \ref{Arch comp M_1}, and \ref{Arch comp M2}. (The number $n_G^M$ can again be computed using Lemma \ref{lem:comp n^G_M}, and is seen to be $8,2,2$ for $M_{12}, M_1, M_2$.) We only add the following details: When $M= M_1$, we only know that the component of $\gamma'$ in $M^{\GL} (\QQ) = \GL_2(\QQ)$ is (stably) conjugate to the component of $\gamma_M$ in $M ^{\GL} (\adele_f^p) = \GL_2(\adele_f^p)$ (as opposed to knowing that they are equal), but this already implies that they have equal determinant. Again from the assumption that $a$ is large and $\gamma'$ is good at $p$, we deduce that $j_M(\gamma')$ is in the range $\det <1$ considered in Proposition \ref{Arch comp M_1}. When $M= M_{2}$ in the odd case, we deduce that $j_M(\gamma')$ is in the range $0< a < 1$ considered in Proposition \ref{Arch comp M2} in the same way as when $M = M_{12}$. Finally, we note that the constant $n^G_M$ appearing in Proposition \ref{cor:symmetry broken M1 and M2} is the same for $M_1$ and $M_2$ (equal to $2$), but in Propositions \ref{Arch comp M_1} there is an extra factor $2$ on the right hand side compared to Proposition \ref{Arch comp M2}. This is why in the current proposition we have $\mathscr C_{M_1} =1$ and $\mathscr C_{M_2} = 2$.  
       \end{proof}
   
   \subsection{}
   We now plug the definition of $Q(\mathfrak e, \gamma')$ (Definition \ref{defn:Q}) into the formula (\ref{eq:pre-main computation}), and obtain: 
 \begin{multline}\label{eq:plugged in}
\mathscr C_M  \Tr'_M = 4p^* \tau(M)   k(M)k(G)^{-1} \sum_{\mathfrak e =(M', \lang M', s_M',\eta_M) \in \dot{\mathscr E} (M)^{c,\ur} } \abs{\Out _M(\mathfrak e) } ^{-1} \\    \cdot  \sum_{\gamma'}  SO_{\gamma'} (k_1(\emptyset)) L_M(j_M(\gamma'))  
   \bar \iota^{M'} (\gamma')^{-1}   (\Delta^{M}_{M'})^{\emptyset,\emptyset}(\gamma',\gamma_M) \\ \cdot  O_{\gamma_M}^{s_M'} (f^{p,\infty} _M)   \bar v ({M'}_{\gamma'} ^0) ^{-1}  \Delta_{j_M, B_M} ^{\emptyset,\emptyset} (\gamma', j_M(\gamma')) .
   \end{multline} (The sign $(-1)^{\dim A_{M'}}$ appears both in the definition of $Q(\fke,\gamma')$ and in (\ref{eq:pre-main computation}), and hence it gets canceled in the above.) 
Observe that when $\gamma' \in \Sigma(M')_1$ is good at $p$, we have
\begin{align}\label{eq:change of test function}
SO_{\gamma'} (k_1(\emptyset)) = SO_{\gamma'} ( k_a \otimes 1_{M ^{\prime,\SO}, p}),
\end{align}
 where, in the notation of Proposition \ref{prop:main computation at p}, $k_a \in \mathcal H^{\ur}(M^{\GL}_{\QQ_p})$ is given by  \index{$k_a$}
 \begin{equation}\label{defnLk_a}
  \begin{cases}
 -\xi_1^{-a}, & M = M_{12} \mbox{ or } M_2 ~(\mbox{so that } M^{\GL} = \GG_m^2 \mbox{ or } \GG_m \mbox{ resp.}), \\
 -\zeta_1^{-a} -\zeta_2^{-a} ,  &  M = M_1 ~(\mbox{so that } M^{\GL} = \GL_2),
 \end{cases}
 \end{equation}  and $1_{M^{\prime,\SO},p}$ \index{$1_{M^{\prime,\SO},p}$} denotes the unit element of $\mathcal H^{\ur} (M^{\prime,\SO}_{\QQ_p})$. (Thus $k_a$ differs from $k(\emptyset,\emptyset)$ in that we throw away the \emph{positive} powers of the variables $\xi_i, \zeta_i$, as well as powers of $\xi_2$ when $M = M_{12}$.) Conversely, if the right hand side of (\ref{eq:change of test function}) is non-zero, then $\gamma'$ is necessarily good at $p$. Thus after making the substitution (\ref{eq:change of test function}) inside (\ref{eq:plugged in}), we no longer need to impose the condition of being good at $p$ in the summation over $\gamma'$. 

Let $\gamma'_{p,\GL}$\index{$\gamma'_{p,\GL}$} (resp.~$\gamma'_{p,\SO}$)\index{$\gamma'_{p,\SO}$} be the component of $\gamma'$ in $M^{\GL} (\QQ_p)$ (resp.~$M^{\prime, \SO}(\QQ_p)$). Then we can rewrite (\ref{eq:change of test function}) as 
\begin{align}\label{eq:prod of SO at p}
SO_{\gamma'}(k_1(\emptyset)) = SO_{\gamma'_{p,\GL}} (k_a) SO_{\gamma'_{p,\SO}} (1_{M^{\prime,\SO},p}).  
\end{align} Since $\gamma'$ is $(M,M')$-regular (being in  $\Sigma(M')_1$), $\gamma'_{p,\SO}$ is $(M^{\SO}, M^{\prime,\SO})$-regular. By the Fundamental Lemma (Theorem \ref{thm:LS} (2)), we know that $$SO_{\gamma'_{p,\SO}} (1_{M^{\prime,\SO},p}) \neq 0$$ only if $\gamma'_{p,\SO}$ is an image of a semi-simple element $\gamma_{p,\SO} \in M^{\SO} (\QQ_p)$, and in this case we have 
\begin{align}\label{eq:FL applied}
SO_{\gamma'_{p,\SO}} (1_{M^{\prime,\SO}}) = \Delta^{M^{\SO}} _{M ^{\prime,\SO}}  (  \gamma'_{p,\SO}, \gamma_{p,\SO}) O^{s^{\SO}}_{\gamma_{p, \SO}}(1 _{M^{\SO},p}), 
\end{align} where $\Delta^{M^{\SO}} _{M ^{\prime,\SO}}$ is the canonical unramified normalization of transfer factors at $p$ associated to the hyperspecial subgroup $M^{\SO} (\QQ_p) \cap \mathcal M(\ZZ_p) \subset M^{\SO}(\QQ_p)$, and $1_{M^{\SO},p}$ denotes the unit element of $\mathcal H (M^{\SO}(\QQ_p) \sslash (M^{\SO} (\QQ_p) \cap \mathcal M(\ZZ_p) ))$. 

When $\gamma'_{p,\SO}$ is an image of $\gamma_{p,\SO} \in M^{\SO}(\QQ_p)$ as above, note that $\gamma' = \gamma'_{p,\GL} \gamma'_{p,\SO}$ is an image of $\gamma
'_{p,\GL} \gamma_{p,\SO} \in M(\QQ_p)$, and for the canonical unramified normalizations of transfer factors we have 
\begin{align}\label{eq:transfer factor at p equal}
\Delta^{M^{\SO}} _{M^{\prime,\SO}} (  \gamma'_{p,\SO}, \gamma_{p,\SO}) = \Delta ^{M} _{M'} (\gamma', \gamma'_{p,\GL} \gamma_{p,\SO}).
\end{align}

From (\ref{eq:plugged in}) (\ref{eq:prod of SO at p}) (\ref{eq:FL applied}) (\ref{eq:transfer factor at p equal}), we obtain
\begin{multline}\label{eq:to fuse transfer factors}
\mathscr C_M \Tr'_M = 4p^*  \tau(M) k(M)k(G)^{-1} \sum_{\mathfrak e =(M',\lang M', s_M',\eta_M) \in \dot{\mathscr E} (M)^{c,\ur} } \abs{\Out _M(\mathfrak e) } ^{-1}  \\  \cdot \sum_{\gamma'} \bar   \iota^{M'} (\gamma')^{-1}  \bar v ({M'}_{\gamma'} ^0) ^{-1}    SO_{\gamma'_{p,\GL}} (k_a) L_M(j_M(\gamma'))  O_{\gamma_M}^{s_M'} (f^{p,\infty} _M) O_{\gamma_{p,\SO}} ^{s^{\SO}} (1_{M^{\SO},p}) \\  \cdot 
  (\Delta^{M}_{M'})^{\emptyset,\emptyset} (\gamma',\gamma_M) \Delta_{j_M, B_M} ^{\emptyset,\emptyset} (\gamma', j_M(\gamma'))  \Delta ^{M} _{M'} (\gamma', \gamma'_{p,\GL} \gamma_{p,\SO})  ,
\end{multline}
where $\gamma'$ runs through the elements of $\Sigma(M') _1$ that are good at $\infty$, and for each $\gamma'$ we choose $\gamma_M \in M(\adele_f^p)$ and $\gamma_{p, \SO} \in M^{\SO} (\QQ_p)$ such that $\gamma'$ is an image of $\gamma_M$ (over $\adele_f^p$) and an image of $\gamma'_{p,\GL} \gamma_{p,\SO}$ (over $\QQ_p$). Here we no longer need the condition that $\gamma'$ is good at $p$, as we have already seen. 

\begin{lem}\label{lem:global image}
	If $\gamma'\in M'(\QQ)_{\mathrm{ss}}$ is $\RR$-elliptic and is an image from $M(\adele_f)_{\mathrm{ss}}$, then it is an image from $M(\QQ)_{\mathrm{ss}}$. Moreover, if $\gamma_{\infty} \in M (\RR)_{\mathrm{ss}}$ is a prescribed elliptic element of which $\gamma'$ is an image, then $\gamma'$ is an image of some $\gamma\in M(\QQ)_{\mathrm{ss}}$ such that $\gamma$ is conjugate to $\gamma_{\infty}$ in $M(\RR)$. 
\end{lem}
\begin{proof}We recall a construction from \cite{labesse1999} in our setting.
Let $\gamma^* \in M^*(\QQ)_{\mathrm{ss}}$ be such that $\gamma'$ is an image of it. By hypothesis $\gamma'$ is an image from $M(\adele_f)_{\mathrm{ss}}$, and note that $\gamma'$ is also an image from $M(\RR)_{\mathrm{ss}}$ since it is $\RR$-elliptic. Thus let $\gamma_{\adele} \in M(\adele)_{\mathrm{ss}}$ be such that $\gamma'$ is an image of it. When $\gamma_{\infty}$ is prescribed as in the statement of the lemma, we take $\gamma_{\adele}$ such that its archimedean component is $\gamma_{\infty}$. From $\gamma^*$ and $\gamma_{\adele}$, Labesse   constructs a non-empty subset $$ \mathrm{obs}_{\gamma^*} (\gamma_{\adele})\subset \E (I^*, M^*; \adele/\QQ):  = \coh^0_{\ab} (\adele/\QQ, I^*\backslash M^*)/ \coh^0_{\ab} (\adele, M^*), $$ generalizing the construction of Kottwitz in \cite{kottwitzelliptic}; see \cite[\S 2.6]{labesse1999}, with $L= M, H= M^*$. By \cite[Thm.~2.6.3]{labesse1999}, the condition that $ 1\in \mathrm{obs}_{\gamma^*} (\gamma_{\adele})$ would imply the existence of an element of  $M(\QQ)_{\mathrm{ss}}$ that is conjugate to $\gamma_{\adele} \in M(\adele)$, and the current lemma would follow. Thus it suffices to prove that $ 1\in \mathrm{obs}_{\gamma^*} (\gamma_{\adele}) $ for a suitable choice of $\gamma_{\adele}.$ 

Note that to prove the lemma we may always modify $\gamma_{\adele}$ by replacing its $v$-adic component with another element stably conjugate to it over $\QQ_v$, for some finite place $v$. We claim that after such a modification we can achieve $1 \in \mathrm{obs}_{\gamma^*} (\gamma_{\adele})  $. In fact, we know that $\E(I^*, M^* ; \adele/ \QQ)$ is isomorphic to the Pontryagin dual group $\mathfrak K (I^*/\QQ) ^D$ of the finite abelian group $\mathfrak K (I^*/\QQ)$ (for $I^*\subset M^*$) considered in \cite[\S 4.6]{kottwitzelliptic}; cf.~\cite[Cor.~1.7.4]{KSZ}. The same argument as the second paragraph of \cite[p.~188]{kottwitzannarbor} implies that the natural map $\mathfrak K (I^*/\QQ_v) ^D \to \mathfrak K (I^*/\QQ) ^D$ is a surjection for some finite place $v$. On the other hand, $$\mathfrak K(I^*/\QQ_v) ^D \cong \E (I^*, M^* ; \QQ_v) \cong \D (I^*, M^* ;\QQ_v) = \ker (\coh^1(\QQ_v, I^*) \to \coh^1(\QQ_v, M^*) ). $$ From the construction of Labesse we know that if we twist $\gamma_{\adele}$ within its stable conjugacy class by a class $c \in \D (I^*, M^* ;\QQ_v)$, then $\mathrm{obs}_{\gamma^*} (\gamma_{\adele})$ gets shifted by the image of $c$ in the abelian group $\E(I^*, M^* , \adele/\QQ)$. The claim follows.
\end{proof}
\subsection{}
By Lemma \ref{lem:global image}, we may assume that each $\gamma'$ in (\ref{eq:to fuse transfer factors}) is an image of some $\gamma \in M(\QQ)_{\mathrm{ss}}$, and that $\gamma$ is conjugate to $j_M(\gamma')$ in $M(\RR)$. Note that we have $L_M(j_M(\gamma')) = L_M(\gamma)$. (In fact $L_M(\cdot)$ depends only on $\CC$-conjugacy classes.) We may and shall take $\gamma_M$ and $\gamma'_{p,\GL}\gamma_{p,\SO}$ to be localizations of $\gamma$ in $M(\adele_f^p)$ and $M(\QQ_p)$ respectively. 
 
 We have seen that $\tasho (\emptyset,\emptyset) = -1$ in Proposition \ref{prop:computing tasho}. Therefore with the above assumptions on  $\gamma_M$ and $\gamma'_{p,\GL}\gamma_{p,\SO}$, the product of the three transfer factors in the third line of (\ref{eq:to fuse transfer factors}) becomes $-1$.  
We summarize the above discussion in the following proposition.
\begin{prop}\label{prop:suitably choose}When $a\in \ZZ_{>0}$ is large enough (for a fixed $f^{p,\infty}$), we have 
\begin{multline}\label{eq:almost done in one direction} \mathscr C_M
\Tr'_M = - 4p^* \tau(M) k(M)k(G)^{-1} \sum_{\mathfrak e =(M',\lang M',s_M',\eta_M) \in \dot {\mathscr E} (M)^{\ur,c} } \abs{\Out _M(\mathfrak e) } ^{-1}  \\   \cdot \sum_{\gamma'}  \bar   \iota^{M} (\gamma)^{-1}  \bar v ({M}_{\gamma} ^0) ^{-1}  SO_{\gamma_{\GL}} (k_a) L_M(\gamma)  O_{\gamma}^{s_M'} (f^{p,\infty} _M)   O_{ \gamma_{\SO}} ^{s^{\SO}} (1_{M^{\SO},p})    \\   \cdot \Delta_{j_M, B_M} ^{\emptyset,\emptyset} (\gamma', j_M(\gamma'))  \Delta_{j_M, B_M} ^{\emptyset,\emptyset} (\gamma', \gamma) ^{-1},
\end{multline}
where $\gamma'$ runs through the elements of $\Sigma(M')_1$ that are good at $\infty$, and such that $\gamma'$ is an image of some $\gamma \in M(\QQ)_{\mathrm{ss}}$. For each $\gamma'$, we fix a corresponding $\gamma$, and use $\gamma_{\GL}$ and $\gamma_{\SO}$ to denote the (localizations over $\QQ_p$ of) the components of $\gamma$ in $M^{\GL}$ and $M^{\SO}$ respectively. \qed
\end{prop}

\begin{defn}\label{defn:scr D}
	For any reductive group $I$ over $\RR$ that contains elliptic maximal tori, let $\mathscr D(I)$\index{$\mathscr D(\cdot)$} be the cardinality of $\D(T,I;\RR) = \ker (\coh^1(\RR, T) \to \coh^1(\RR, I) )$, where $T$ is any elliptic maximal torus in $I$. 
\end{defn}
\begin{lem}\label{lem:scr D}
	Let $I$ and $T$ be as in Definition \ref{defn:scr D}. 
	\begin{enumerate}
		\item We have $\mathscr D(I) = \abs{\Omega_{\CC}(I,T) / \Omega_{\RR} (I,T) }. $ In particular $\mathscr D(I)$ is independent of the choice of $T$. 
		\item If $\mathscr D(I) =1$, then any two elliptic elements of $I(\RR)$ that are stably conjugate to each other are conjugate under $I(\RR)$. 
		\item If $\mathscr D(I) = 1$, then for any elliptic element $x \in I(\RR)$, we have $\mathscr D(I_x^0) =1$.
	\end{enumerate}
\end{lem}
\begin{proof}
  Statement (1) follows from \cite[Prop.~6.4.2]{labessesnowbird}, and the fact that all elliptic maximal tori are conjugate under $I(\RR)$. For (2), it suffices to prove that for any (connected) reductive subgroup $J$ of $I$ containing an elliptic maximal torus $T$ in $I$, we have $ \D(J,I;\RR) = 1. $ But this follows from \cite[Lem.~10.2]{kottwitzelliptic}, which says that $\coh^1(\RR,T)$ surjects onto $\coh^1(\RR, J)$. Finally, (3) follows from the fact that $I_x^0$ contains a maximal torus which is elliptic in both $I_x^0$ and $I$. 
\end{proof}
\begin{lem}\label{lem:comp mathscr D} We have $\mathscr D(M_{\RR}) = 1$. 
\end{lem}
\begin{proof}  
	If $M = M_1$ or $M_{12}$, then $M_{\RR}$ is a product of copies of  $\GL_2$ or $\GG_m$ and an anisotropic group, so $\mathscr D(M) =1$. Now suppose $M = M_2$ in the odd case. Write $n$ for $d-2$, and recall that $n \geq 3$. We have $M_{\RR} \cong  \GG_m \times \SO(n-1, 1), $ so $\mathscr D(M_{\RR}) = \mathscr D(\SO(n-1, 1))$. To compute $\mathscr D(\SO(n-1,1))$, consider an elliptic (anisotropic) maximal torus $T\cong \Uni(1) ^{(n-1)/2}$ in $\SO(n-1,1)$, which is inside the maximal compact subgroup $\mathrm S ( \mathrm O(n-1) \times \mathrm O(1) )$ of $\SO(n-1,1)$. It is well known (see for instance \cite[Prop.~6.16]{adamstaibi}) that we have 
	 $$ \abs{\Omega_{\RR} (\SO(n-1,1), T)} = \abs{\Nor_{\mathrm S (\mathrm O(n-1) \times \mathrm O(1))(\CC)}( T(\CC))/ T(\CC)}.$$ On the other hand one can directly check that as subgroups of $\Aut(T_{\CC})$ we have 
	 $$ \Nor_{\mathrm S (\mathrm O(n-1) \times \mathrm O(1))(\CC)}( T(\CC))/ T(\CC) = \Omega_{\CC}( \SO(n-1,1), T) \cong \set{\pm 1}^{(n-1)/2} \rtimes \mathfrak S_{(n-1)/2}. $$ It then follows from Lemma \ref{lem:scr D} (1) that $\mathscr D(\SO(n-1,1)) =1$.  
\end{proof}
 
\begin{prop}\label{prop:mathscr D} 	Keep the setting and notation of Proposition \ref{prop:suitably choose}.  We have 
	\begin{multline}\label{eq:in cor}
\mathscr C_M \Tr'_M = - 4p^* \tau(M) k(M)k(G)^{-1} \sum_{\gamma_0} \sum_{ \kappa}   \bar   \iota^{M} (\gamma_0)^{-1}  \bar v (I_0) ^{-1}  SO_{\gamma_{0,\GL}} (k_a) L_M(\gamma_0)  \\  \cdot
 O_{\gamma_0}^{\kappa} (f^{p,\infty} _M) O_{ \gamma_{0, \SO}} ^{\kappa^{\SO}} (1_{M^{\SO},p}) ,
	\end{multline}
where \begin{itemize}
	\item $\gamma_0$ runs through a fixed set of representatives of the stable conjugacy classes in $M(\QQ)$ that are elliptic over $\RR$ and good at $\infty$. We let $\gamma_{0,\GL}$ and $\gamma_{0,\SO}$ denote the (localizations over $\QQ_p$ of) the components of $\gamma_0$ in $M^{\GL}$ and $M^{\SO}$ respectively.
	\item $I_0 : = M_{\gamma_0} ^0$.
	\item $\kappa$ runs through $\mathfrak K(I_0/\QQ) = \E(I_0, M ; \adele/\QQ) ^D$.
\end{itemize}
\end{prop}
\begin{proof} By Lemma \ref{lem:scr D} (2) and Lemma \ref{lem:comp mathscr D},  every $\gamma$ in (\ref{eq:almost done in one direction}) is conjugate to $j_M(\gamma')$ over $\RR$. Hence the quotient of the two 
transfer factors at the end of (\ref{eq:almost done in one direction})  is equal to $1$. Thus we have 
	\begin{multline*} 
	\mathscr C_M	\Tr'_M = - 4p^* \tau(M) k(M)k(G)^{-1} \sum_{\mathfrak e =(M',\lang M',s_M',\eta_M) \in \dot {\mathscr E} (M)^{\ur,c} } \abs{\Out _M(\mathfrak e) } ^{-1}  \\   \cdot \sum_{\gamma'}  \bar   \iota^{M} (\gamma)^{-1}  \bar v ({M}_{\gamma} ^0) ^{-1}  SO_{\gamma_{\GL}} (k_a) L_M(\gamma)  O_{\gamma}^{s_M'} (f^{p,\infty} _M)   O_{ \gamma_{\SO}} ^{s^{\SO}} (1_{M^{\SO},p}).
	\end{multline*}This implies (\ref{eq:in cor}) by the usual conversion from summation over $(\mathfrak e, \gamma')$ to summation over $(\gamma_0,\kappa)$ in the theory of stabilization (see \cite[Cor.~IV.3.6]{labesse04} and \cite[\S 8.3]{KSZ}). 
\end{proof}
\subsection{}
 Now by Fourier analysis on the finite abelian groups $\mathfrak K(I_0/\QQ) ^D = \E(I_0, M;\adele/\QQ)$, (cf.~\cite[p.~395]{kottwitzelliptic}, \cite[p.~174]{kottwitzannarbor}, \cite[\S 8.1]{KSZ}), from Proposition \ref{prop:mathscr D} we deduce 
\begin{multline}\label{eq:Fourier applied} 
\mathscr C_M \Tr'_M = - 4p^* \tau(M) k(M)k(G)^{-1}   \sum_{\gamma_0, \gamma_1} \bar   \iota^{M} (\gamma_0)^{-1}  \bar v (I_0) ^{-1} e(I_{0,\RR})  SO_{\gamma_{0,\GL}} (k_a) \\ \cdot  L_M(\gamma_0)  O_{\gamma_1} (f^{p,\infty} _M) O_{ \gamma_{1,\SO}}  (1_{M^{\SO},p})  \bigg [\tau(M)^{-1} \tau(I_0)   \abs{\ker (\ker^1(\QQ, I_{0}) \to \ker ^1 (\QQ, M )) } \bigg],
\end{multline} 
where \begin{itemize}
	\item $\gamma_0$ runs through a fixed set of representatives of the stable conjugacy classes in $M(\QQ)$ that are elliptic over $\RR$ and good at $\infty$. 
	\item $I_0 : = M_{\gamma_0} ^0$.
	\item $\gamma_1$ runs through the subset of $ \D(I_{0}, M; \adele) : = \ker(\coh^1(\adele,I_0) \to \coh^1(\adele, M))$ consisting of elements whose images in $\E (I_{0}, M ; \adele/\QQ)$ are trivial. Each such $\gamma_1$ determines a conjugacy class in $M(\adele)$ which we also denote by $\gamma_1$. We let $\gamma_{1,\SO}$ be the component of $\gamma_1$ in $M^{\SO} (\QQ_p)$. 
	\item The number $\bigg [\tau(M)^{-1} \tau(I_0)   \abs{\ker (\ker^1(\QQ, I_{0}) \to \ker ^1 (\QQ, M )) } \bigg]$ is none other than the cardinality of $\mathfrak K(I_0/\QQ)$. (This can be shown by combining \cite[\S 9]{kottwitzelliptic} and Weil's conjecture on the Tamagawa number proved by Kottwitz \cite{kottTama}, cf.~\cite[\S 4]{kottwitzannarbor}.)
\end{itemize}

\subsection{}
The last major operation to be applied to (\ref{eq:Fourier applied}) is the Base Change Fundamental Lemma\index[n]{Base Change Fundamental Lemma}, which relates $\SO_{\gamma_{0,\GL}} (k_a)$ to the twisted orbital integrals in Kottwitz's point counting formula. We only need this result for $\GG_m$, in which case it is trivial, and for $\GL_2$, in which case it was initially proved by Langlands \cite{langlandsbasechange}. For an account of the theory for $\GL_n$ see \cite{arthurclozel} and for the proof in the general case see \cite{clozelFL, LabDuke}. 

Observe that the function $k_a \in \mathcal H^{\ur} (M^{\GL}_{\QQ_p})$ defined in (\ref{defnLk_a}) is equal to the image under the base change map (see \S \ref{para:BC}) 
$$\mathcal H^{\ur} (M^{\GL}_{\QQ_{p^a}}) \To \mathcal H^{\ur } (M^{\GL}_{\QQ_p})$$ of the element $p^{-a/2} \phi _a^{M_h}$, resp.~$-\phi_a^{M_h}$, resp.~$-\phi_a^{M_h} \otimes 1$, where $\phi_a^{M_h}$ is as in Definition \ref{defn:phi_a}, when $M= M_1$, resp.~$M_2$, resp.~$M_{12}$. Here when $M= M_{12}$ we have $M^{\GL} = M_h \times \GG_m$, and we write $-\phi_a^{M_h} \otimes 1$ corresponding to this decomposition, where $1$ is the unit of $\mathcal H^{\ur} (\GG_{m,\QQ_{p^a}})$. By the Base Change Fundamental Lemma, we have, for any semi-simple conjugacy class (which is the same as stable conjugacy class) $\gamma_{0,\GL}$ in $M^{\GL} (\QQ)$, the following identity:
\begin{align}\label{eq:BFL applied}
SO_{\gamma_{0,\GL}} (k_a) = \begin{cases}
- \sum_{\delta} e(\delta) TO_{\delta}( \phi_a^{M_h}), & \text{if } M =M_2, \\
- p^{-a/2}\sum_{\delta} e(\delta) TO_{\delta}( \phi_a^{M_h}), & \text{if } M =M_1 ,\\
-\sum_{\delta} e(\delta) TO_{\delta} (\phi_a^{M_h} ) 1_{\ZZ_p^{\times}} (y), & \text{if } M = M_{12},  
\end{cases} 
\end{align} 
where $\delta$ runs through the $\sigma$-conjugacy classes in $M_h(\QQ_{p^a})$ such that it has norm the $M_h$-component of $\gamma_{\GL}$,  $e(\delta)$ denotes the Kottwitz sign of the twisted centralizer of $\delta$ (a reductive group over $\QQ_p$), and in the last case we write  $\gamma_{0,\GL} = (x,y) \in M_h \times \GG_m $. (In fact, by \cite{arthurclozel} or direct verification, the above summation over $\delta$ is either empty or over a singleton.)

The next lemma is sometimes called ``pre-stabilization'' in the literature. \index[n]{pre-stabilization}

\begin{lem}\label{lem:pre-stabilize} Let $F(x,y)$ be a $\CC$-valued function on the set of  compatible pairs $(x,y) $ of a stable conjugacy class $x$ in $M(\QQ)$ and a conjugacy class $y$ in $M(\adele)$. Then we have 
$$ \sum_{\gamma} \iota^M (\gamma) ^{-1}  F(\gamma,\gamma) = \sum _{\gamma_0, \gamma_1} \bar \iota ^{M} (\gamma_0) ^{-1}\abs{\ker (\ker^1(\QQ, I_{0}) \to \ker ^1 (\QQ, M )) } F( \gamma_0, \gamma_1), $$
where on the LHS $\gamma$ runs through the conjugacy classes in $M(\QQ)$ which are $\RR$-elliptic, and on the RHS $\gamma_0$ runs through an arbitrary set of representatives of the stable conjugacy classes in $M(\QQ)$ that are $\RR$-elliptic, and $\gamma_1$ runs through the subset of $\D(I_{0}, M; \adele)$ consisting of elements whose images in $ \E (I_{0}, M ; \adele/\QQ)$ are trivial. Here we have denoted $I_0: = M_{\gamma_0} ^0$. Moreover, if we restrict the summation on the LHS to only those $\gamma$ good at $\infty$, and restrict the summation on the RHS to only those $\gamma_0$ good at $\infty$, we still get an equality.
\end{lem}
\begin{proof}
	The multiplicity of a $M(\QQ)$-conjugacy class $\gamma$ appearing in the set $\D(I_0 , M ;\QQ)$ is equal to $\bar \iota^{M} (\gamma_0)  \cdot \iota^M (\gamma) ^{-1}$. The fibers of the map $\D (I_0, M ;\QQ) \to \D(I_0, M ,\adele)$ all have size $$\abs{\ker (\ker^1(\QQ, I_{0}) \to \ker ^1 (\QQ, M )) }. $$ The lemma then easily follows. 
\end{proof}

We are now ready to prove Theorem \ref{thm:maincomp}. 
\begin{proof}[Proof of Theorem \ref{thm:maincomp}]
By (\ref{eq:Fourier applied}) and Lemma \ref{lem:pre-stabilize}, we have \begin{multline}\label{eq:in the proof of main comp} 
\mathscr C _M\Tr'_M = - 4p^*  k(M)k(G)^{-1} \sum_{\gamma}   \iota^{M} (\gamma)^{-1}  \bigg [\bar v (M_{\gamma} ^0) ^{-1} e (M_{\gamma,\RR} ^0)  \tau(M_{\gamma} ^0)\bigg]  \\  \cdot  SO_{\gamma_{\GL}} (k_a) L_M(\gamma)  O_{\gamma} (f^{p,\infty} _M) O_{ \gamma_{\SO}}  (1_{M^{\SO},p}) ,
\end{multline}
where $\gamma$ runs through conjugacy classes in $M(\QQ)$ that are elliptic over $\RR$ and good at $\infty$. By Harder's formula (see \cite[\S 7.10]{GKM}), we have 
\begin{align*}
\chi (M_{\gamma} ^0) = \bar v (M_{\gamma} ^0) ^{-1} e(M_{\gamma,\RR} ^0) \mathscr D(M_{\gamma,\RR} ^0)\tau(M_{\gamma} ^0).
\end{align*}
By Lemma \ref{lem:scr D} (3) and Lemma \ref{lem:comp mathscr D}, $D(M_{\gamma,\RR} ^0) = 1$. Hence the product in the bracket in (\ref{eq:in the proof of main comp}) is equal to $\chi(M_{\gamma}^0)$, and therefore 
\begin{multline}\label{eq:Harder}
	\mathscr C _M\Tr'_M = - 4p^*  k(M)k(G)^{-1} \sum_{\gamma}   \iota^{M} (\gamma)^{-1}  \chi(M_{\gamma}^0)  \\  \cdot  SO_{\gamma_{\GL}} (k_a) L_M(\gamma)  O_{\gamma} (f^{p,\infty} _M) O_{ \gamma_{\SO}}  (1_{M^{\SO},p}) .
\end{multline}

Denote\index{$p^{**}$}
$$ p^{**} : = \begin{cases}
p^*, & \text{if } M= M_2 \mbox{~or~} M_{12}, \\ 
p^{-a/2} p^*, & \text{if } M= M_1.
\end{cases}$$
By (\ref{eq:BFL applied}) and (\ref{eq:Harder}) we have 
\begin{multline}\label{eq:last eq}
\mathscr C_M \Tr'_M =  4p^{** } k(M)k(G)^{-1} \\ \cdot  \sum_{\gamma,\delta}   \iota^{M} (\gamma)^{-1}  \chi(M_{\gamma} ^0)   e(\delta)TO_{\delta}( \phi_a^{M_h}) L_M(\gamma)  O_{\gamma} (f^{p,\infty} _M) O_{ \gamma_{L}}  (1_{M_l(\ZZ_p)}) ,
\end{multline} where $\gamma_L$ denotes the component of $\gamma$ in $M_l$ under the decomposition $M = M_h \times M_l$ (which only differs from the decomposition $M = M^{\GL} \times M^{\SO}$ when $M = M_{12}$), and $1_{M_l (\ZZ_p)}$ is as in Definition \ref{Defn Tr12}.
To finish the proof we divide into different cases.

\textbf{Case $M= M_{12}$.} 

In Definition \ref{Defn Tr12}, $\delta$ runs through those elements of $\QQ_{p^a} ^{\times}$ with norm $\gamma_0$ and such that the Kottwitz invariant of $\delta$ in $\pi_1 (M_h) _{\Gamma_p} = X_*(\GG_m) = \ZZ$ is equal to the image of $-\mu$. The last condition is equivalent to requiring that $v_p(\delta) = -1$, which is a necessary (and also sufficient) condition for $TO_{\delta} (\phi_a^{M_h}) \neq 0$. Hence we may drop this condition in the summation in Definition \ref{Defn Tr12}. Every term $c(\gamma_0,\gamma,\delta)$ is easily computed to be $2^{-1}$ (with $c_1 = \vol(\GG_m(\RR) / \GG_m(\RR)^0 )^{-1} = 2^{-1}, c_2 = 1$). On the other hand, in (\ref{eq:last eq}) every term $e(\delta)$ is $1$. Comparing  Definition \ref{Defn Tr12} and (\ref{eq:last eq}), we see that it suffices to prove that 
\begin{align}\label{eq:desired in M12}
2^{-1}\delta^{1/2}_{P_{12}(\QQ_p)} (\gamma_h) = 4p^* k(M) k(G)^{-1} \chi (M_{h,\gamma_h})
\end{align}
for $\gamma = \gamma_h \gamma_L$ contributing to (\ref{eq:last eq}). (Here $\gamma_h$ and $\gamma_L$ denote the components of $\gamma$ in $M_h(\QQ)$ and $M_l(\QQ)$.)
We have $\chi (M_{h,\gamma_h}) = \chi(\GG_m) = 2^{-1}$ by Harder's formula, and we have $ k(M) = 2^{m-3}, k(G) = 2^{m-1}$ by Propositions \ref{prop:comp k} and \ref{prop:Tamagawa GL}. 
Moreover, if $\gamma= \gamma_h \gamma_L$ contributes then $v_p (\gamma_h) = -a$ (because $\delta $ should exist) , and therefore in the odd case
$$ \delta _{P_{12}(\QQ_p)} (\gamma_h) = \prod_{\alpha \in \Phi^+ - \Phi_M^+} \abs{\alpha (\gamma_h)}_p =\abs{\gamma_h} _p ^{2m -1} = p ^{(d-2) a } = (p^*)^{2} ,$$ where the contributing roots are $\epsilon_1, \epsilon_1 \pm \epsilon_j,  j \geq 2$. Similarly, in the even case, 
$$ \delta _{P_{12}(\QQ_p)} (\gamma_h) = \prod_{\alpha \in \Phi^+ - \Phi_M^+} \abs{\alpha (\gamma_h)}_p =\abs{\gamma_h} _p ^{2m -2} = p ^{(d-2) a } = (p^*)^{2} ,$$ where the contributing roots are $ \epsilon_1 \pm \epsilon_j,  j \geq 2$. The equality (\ref{eq:desired in M12}) follows, and the proof is finished in this case.

\textbf{Case $M = M_1$.}

First we claim that if $\gamma_0 \in \GL_2(\QQ)$ is semi-simple and $\RR$-elliptic, then $$c_2(\gamma_0) = \tau (\GL_{2,\gamma_0}) =1.$$ In particular, we have \footnote{This equality also follows from the formula for $c$ on p.~174 of \cite{kottwitzannarbor}, the fact that $\tau(\GL_2) =1$, and Lemma \ref{simplification 1 M1}}, $$c(\gamma_0) = c_1(\gamma_0 ) c_2(\gamma_0) = \vol (A_{\GL_2} (\RR) ^0 \backslash \overline{ \GL_{2,\gamma_0} } (\RR) )  ^{-1}. $$

We prove the claim. Write $I_0$ for $\GL_{2,\gamma_0}$. If $I_0 = \GL_2$, then $\tau(I_0) =1$ by Proposition \ref{prop:Tamagawa GL}, and $c_2(\gamma_0) = 1$ by definition. Otherwise $I_0 = T$ is a maximal torus in $\GL_2$ that is elliptic over $\RR$. Observe that $T = \Res_{F/\QQ} \GG_m$ for some imaginary quadratic field $F$, so $\coh^1(\QQ, T) = 0$ by Shapiro's lemma and Hilbert 90. Hence $c_2(\gamma_0) =1$. Now by \cite[(5.1.1)]{kottwitzcuspidal} and Weil's conjecture on Tamagawa numbers proved in \cite{kottTama},
	we have $$\tau (T) c_2(\gamma_0) = \tau(T) \abs{\ker^1(\QQ, T)} = \tau (T) \abs{\ker^1(\Gamma, \widehat T)} = \abs{\pi_0 ( \widehat T ^{\Gamma})}.$$ Thus to show $\tau(T) =1$ it suffices to show that $\widehat T^{\Gamma}$ is connected. We have seen in the proof of Lemma \ref{simplification 1 M1} that $\widehat T^{\Gamma_{\infty}} \subset Z(\widehat \GL_2) $. On the other hand $Z(\widehat \GL_2) \subset \widehat T^{\Gamma}$. Hence $\widehat T^{\Gamma} = Z(\widehat \GL_2) = \CC^{\times}$, which is connected as desired. The claim is proved.
	
	 We continue to consider such $\gamma_0 \in \GL_2(\QQ)$ as in the claim, and write $I_0$ for $\GL_{2,\gamma_0}$. By Harder's formula we have $$\chi  (I_0) = e (I_{0,\RR})\bar v ^{-1} (I_0) \tau(I_0) \abs{\mathscr D (I_{0,\RR})} .$$ Since $I_{0,\RR}$ is either $\GL_{2,\RR}$ or an elliptic maximal torus in $\GL_{2,\RR}$, we have $e(I_{0,\RR}) = \abs{\mathscr D(I_{0,\RR})} = 1$. Hence 
	 $$ \chi  (I_0)  =  e(\overline{I_0}) \vol (A_{\GL_2} (\RR) ^0 \backslash \overline{ I_0})^{-1} \tau(I_0) $$ 
	 where $\overline{I_0}$ is the inner form over $\RR$ of $I_{0,\RR}$ that is anisotropic modulo center.

	 If $\delta \in G(\QQ_{p^a}) $ has norm stably conjugate to some $\gamma_0 \in \GL_2(\QQ)$ and $\gamma_0$ is good at $p$ (i.e., its determinant has valuation $-a$), then we have 
	 $ e(\delta)  = e (\overline {I_0})$, where $\overline {I_0}$ is defined in terms of $\gamma_0$ as above. In fact, this follows from the existence of the (global) inner form $I$ of $I_0$ as in \S \ref{subsubsec:Haar}, the product formula for the Kottwitz signs for $I$, and the observation that for any place finite $v \neq p$, $e(I_{0,\QQ_v}) =1$ since $I_{0,\QQ_v}$ is either a torus or $\GL_{2,\QQ_v}$.  
	 
	From the discussion so far we deduce that for $\delta$ and $\gamma_0$ as in the last paragraph we have $$c(\gamma_0) = e(\delta) \chi(I_0).$$ Moreover, if $\delta\in M_h(\QQ_{p^a})$ is such that $TO_{\delta} (\phi_a^{M_h}) \neq 0$, then necessarily $v_p (\det \delta) = -1$, and it follows easily that the Kottwitz invariant of $\delta$ in $\pi_1(M_h) _{\Gamma_p} \cong \ZZ$ is equal to the image of $-\mu$. It remains to show that $$\delta_{P_1(\QQ_p)} (\gamma_h) ^{1/2} = 4p^{**} k(M) k(G) ^{-1},$$ for any $\gamma = \gamma_h \gamma_L$ contributing to (\ref{eq:last eq}). We have $k(M) = 2^{m-3}, k(G) = 2^{m-1}$ by Propositions \ref{prop:comp k} and \ref{prop:Tamagawa GL}. For $\gamma = \gamma_h \gamma_L$ contributing, we have $v_p(\det \gamma_h) =- a$ (because $\delta$ should exist), and therefore in the odd case
	 	$$ \delta _{P_1(\QQ_p)} (\gamma_h) = \prod_{\alpha \in \Phi^+ - \Phi_M^+} \abs{\alpha (\gamma _h ) }_p =\abs{\det(\gamma_h)} _p ^{2m-2} = p ^{(d-3) a } = (p^{**}) ^2,$$ where the contributing roots are 
	 	$ \epsilon_1,\epsilon_2, \epsilon_1 +\epsilon_2, \epsilon_1 \pm \epsilon_j, \epsilon_2 \pm \epsilon_j, j\geq 3$. In the even case, the contributing roots are $ \epsilon_1 +\epsilon_2, \epsilon_1 \pm \epsilon_j, \epsilon_2 \pm \epsilon_j, j\geq 3$, and $\abs{\det(\gamma_h)} _p ^{2m-2}$ is replaced by $\abs{\det(\gamma_h)} _p ^{2m-3}$, which is still equal to $(p^{**})^2$
. 	 The proof is finished in this case.
	 
\textbf{Case $M = M_2$ (odd case).} 

Similarly to the case $M = M_{12}$, we reduce the proof to proving the following equality:
$$2^{-1}\delta^{1/2}_{P_{12}(\QQ_p)} (\gamma) = 2^{-1} 4p^* k(M) k(G)^{-1} \chi (M_{h,\gamma_h}) . $$	 The extra factor $2^{-1}$ on the RHS in comparison to (\ref{eq:desired in M12}) appears due to the fact that $\mathscr C_M =2$ for $M = M_2$. We have $\chi (M_{h,\gamma_h}) = \chi(\GG_m) = 2^{-1}$ by Harder's formula, and  $ k(M) = 2^{m-2}, k(G) = 2^{m-1}$ by Propositions \ref{prop:comp k} and \ref{prop:Tamagawa GL}.
Also as in the $M_{12}$ case, if $\gamma= \gamma_h \gamma_L$ contributes then  
$$ \delta _{P_{2}(\QQ_p)} (\gamma_h) = \prod_{\alpha \in \Phi^+ - \Phi_M^+} \abs{\alpha (\gamma_h)}_p =\abs{\gamma_h} _p ^{2m -1} = p ^{(d-2) a } = (p^*)^{2},$$ where the contributing roots are $\epsilon_1, \epsilon_1 \pm \epsilon_j,  j \geq 2$.	 The proof is finished in this case. 
	 \end{proof}
	 
At this point we have completed the proof of Theorem \ref{thm:maincomp}. In the next two sections we prove vanishing results that are complementary to Theorem \ref{thm:maincomp}.

\section{A vanishing result, odd case} \label{subsec:odd case vanishing}
\subsection{}\label{para:preparation for vanishing}
Assume we are in the odd case. Consider a Levi subgroup $M^*$ of $G^* = \SO(\underline V)$ of the form considered in \S \ref{subsec:G-endosc arch}. Thus we fix $r, t \in \ZZ_{\geq 0}$, a non-degenerate subspace $\underline W$ of $\underline V$ of codimension $2(r+2t)$, a hyperbolic basis $\mathbb B_{\underline W^{\perp}}$ of $\underline W^{\perp}$, an embedding $$\GG_m^r \times \GL_2^t \isom M^{*, \GL} \subset \SO(\underline W^{\perp})$$ as in (\ref{eq:M}), and obtain $M^*$ as $M^* = M^{*, \GL} \times \SO(\underline W) \subset G^*$. We write $M^{*,\SO}$ for $\SO(\underline W)$. As in \S \ref{para:presentation of endoscopic G-data} and Proposition \ref{prop:odd classification of G-endoscopy}, isomorphism classes in $\mathscr E_{G^*}(M^*)$ have explicit representatives $\fke_{A,B,\fkp}$ for parameters $(A,B, \fkp) \in \mathscr P_{r,t} \times' \mathscr P_{\underline W}$. In complete analogy with \S \ref{para:prepare for main}, we fix 
$\dot {\mathscr E}_{G^*}(M^*)$\index{$\dot {\mathscr E}_{G^*}(M^*)$} to be a subset of these $\fke_{A,B,\fkp} = (M', \lang M', s_{M^*}, \eta_{M^*})$ such that the component of $s_{M^*}$ in $\widehat{M^{*,\SO}}$ is not $-1$ and such that each isomorphism class in $\mathscr E_{G^*}(M^*)$ is represented exactly once. For each $\fke_{A,B,\fkp} = \fke_{A,B, d^+,\delta^+,d^-, \delta^-}  =  (M', \lang M', s_{M^*}, \eta_{M^*}) \in \dot {\mathscr E}_{G^*}(M^*)$, we let  $$ (H, \lang H, s, \eta) : = \mathfrak e_{d^+ + 2\abs{A} +4 \abs{B} , \delta^+,  d^- + 2\abs{A^c} + 4 \abs{B^c}, \delta^-}$$ which is the induced elliptic endoscopic datum for $G^*$ as in Proposition \ref{prop:odd classification of G-endoscopy}. We also view $(H, \lang H, s, \eta)$ as an elliptic endoscopic datum for $G$. Since $H^+$ is non-trivial by our assumption on $s_{M^*}$, the function $f^H$ is defined as in \S \ref{subsec:test functions}. Moreover, as in \S \ref{subsec:test functions}, we have the fixed pair $(j: T_H \to T_G, B_{G,H})$, and a normalization for transfer factors between $H$ and $G$ at all finite places. We fix $M' \hookrightarrow H$ as in \S \ref{para:two maps from endoscopic G data} so as to view $M'$ as a Levi subgroup of $H$, and define $ST^H_{M'} (f^H)$ as in Definition \ref{defn:pre geometric side}. In analogy with (\ref{defn:Tr'}), we define \index{$\Tr'_{M^*}$}
\begin{multline}\label{eq:Tr'M*}
\Tr'_{M^*}   : =  (n^{G^*}_{M^*})^{-1}    \sum _{ \substack {\fke  = (M', \lang M', s_{M^*}, \eta_{M^*}) \\ \in \dot {\mathscr E}_{G^*}(M^*)} } \abs{\Out_{G^*} (\fke)} ^{-1} \tau(G ) \tau(H)^{-1} ST_{M'} ^H (f^H).
\end{multline}  
 
\begin{thm}\label{vanishing odd}Assume that $M^*$ does not transfer to $G$. Then $\Tr'_{M^*} = 0 .$
\end{thm}
\begin{proof}
 By hypothesis at least one of the following conditions holds:
	$$ rt>0 \qquad\text{or} \qquad r \geq 3 \qquad \text{or} \qquad t\geq 2.$$

	Let $\mathscr E(M^*)^{c,\ur}$ be the subset of $\mathscr E(M^*)$ consisting of isomorphism classes of endoscopic data whose groups are cuspidal over $\QQ$ (which is automatic in the odd case) and unramified over $\QQ_p$. 
	Define a set $\dot{\mathscr{E}}(M^*)^{c,\ur}$ of representatives of  $\mathscr E(M^*)^{c,\ur}$ in exactly the same way as in \S \ref{para:dot E for M}. Thus  $\dot{\mathscr{E}}(M^*)^{c,\ur}$ consists of $\fke_{\fkp}(M^*)$ for certain $\fkp = (d^+,\delta^+,d^-,\delta^-)\in \mathscr P_{\underline W}$, which all satisfy that $d^+ \geq 2$.  
	Then the same arguments as in \S\S \ref{pf:1}--\ref{pf:2} yield a decomposition of $\Tr'_{M^*}$ into a sum as follows. The indexing set for the sum is the set of pairs $(\fke,\gamma')$, where $\fke =(M',\lang M', s'_{M^*}, \eta_{M^*})$ runs through $\dot {\mathscr E}(M^*)^{c,\ur}$, and for each fixed $\fke$, $\gamma'$ runs through a set of representatives in $M'(\QQ)$ of the semi-simple $\RR$-elliptic $(M^*, M')$-regular stable conjugacy classes. For each $(\fke,\gamma')$, the summand is a complex number times 
	\begin{align}\label{eq:first eq for vanishing}
\sum_{A,B}  SO_{\gamma'}((f^{H,p,\infty})_{M'}) SO_{\gamma'}(f^H_{p,M'}) \sum_{\varphi_H \in \Phi_H(\varphi_{\mathbb V^*})} \det(\omega_*(\varphi_H)) \Phi^H_{M'}(\gamma'^{-1}, \Theta_{\varphi_H})
	\end{align}
	where:
	\begin{itemize}
		\item The first summation is over all subsets  $A$  of $[r]$ (recall that this is our short-hand notation for $\set{1,2,\cdots, r}$) and all subsets $B$ of $[t]$. 
		\item For each $(A,B)$, we define $(H,\lang H, s ,\eta)$ with respect to $\fke$ and $(A,B)$, and view $M'$ as a Levi subgroup of $H$, as explained in \S \ref{para:preparation for vanishing}. 
	\end{itemize}
	
	We now fix $(A,B)$ and analyze the terms  $SO_{\gamma'}((f^{H,p,\infty})_{M'})$ and $ SO_{\gamma'}(f^H_{p,M'})$. 
	If there is one finite place $v \neq p$ such that $M^*_{\QQ_v}$ does not transfer to $G_{\QQ_v}$, then $SO_{\gamma'}((f^{H,p,\infty})_{M'})=0$ by the proof of \cite[Lem.~6.3.5 (ii)]{morel2010book}. In this case (\ref{eq:first eq for vanishing}) is zero for all $(\fke, \gamma')$, and the theorem is already proved. Thus we assume that $M^*_{\QQ_v}$ transfers to a Levi subgroup $M_v$ of $G_{\QQ_v}$ at each finite place $v\neq p$. In this case, the localization at $v$ of $\fke$ can be viewed as an endoscopic datum for $M_v$, and there is a normalization $(\Delta^{M_v}_{M'})^{A,B}_v$ of transfer factors between $M'$ and $M_v$ inherited from the normalization $(\Delta^G_H)_v$ of transfer factors between $H$ and $G$ at $v$ fixed in \S \ref{normalizing the transf factors}. For almost all $v$, $(\Delta^{M_v}_{M'})^{A,B}_v$ is the canonical unramified normalization (associated to the hyperspecial subgroup of $M_v(\QQ_v)$ determined by the hyperspecial subgroup of $G(\QQ_v)$ determined by some reductive model of $G$ over some Zariski open in $\spec \ZZ$), and is hence independent of $(A,B)$. Define  
	$$ \epsilon^{p,\infty}(A,B) : = \prod_{v \neq p,\infty} \frac{(\Delta_{M'}^{M_v}) ^{A,B}_v}{(\Delta_{M'}^{M_v})^{\emptyset,\emptyset}_v} ,$$ which is a finite product. 
	Then as an analogue of Proposition \ref{prop:LS for Levi}, $SO_{\gamma'}((f^{H,p,\infty})_{M'})$ is equal to $\epsilon^{p,\infty}(A,B)$
	times a number independent of $(A,B)$.

By Proposition \ref{prop:main computation at p}, we know that 
$ SO_{\gamma'}(f^H_{p,M'})$ is a linear combination of $\nabla_i(A)$, $\nabla_j(B)$, and $1$ (where $i\in [r]$ and $j\in[t]$) with coefficients independent of $(A,B)$. We conclude that  (\ref{eq:first eq for vanishing}) is a linear combination of the following $r+t+1$ expressions: 
\begin{align*}
	R_i & : =
\sum_{A,B} \nabla_i(A) \epsilon^{p,\infty}(A,B)   \sum_{\varphi_H \in \Phi_H(\varphi_{\mathbb V^*})} \det(\omega_*(\varphi_H)) \Phi^H_{M'}(\gamma'^{-1}, \Theta_{\varphi_H}), \qquad 1\leq i \leq r, \\ T_j & :=
\sum_{A,B} \nabla_j(B) \epsilon^{p,\infty}(A,B) \sum_{\varphi_H \in \Phi_H(\varphi_{\mathbb V^*})} \det(\omega_*(\varphi_H)) \Phi^H_{M'}(\gamma'^{-1}, \Theta_{\varphi_H}), \qquad 1 \leq j \leq t ,\\ S & : = 
\sum_{A,B}  \epsilon^{p,\infty}(A,B)  \sum_{\varphi_H \in \Phi_H(\varphi_{\mathbb V^*})} \det(\omega_*(\varphi_H)) \Phi^H_{M'}(\gamma'^{-1}, \Theta_{\varphi_H}).
\end{align*}
	We shall show that these $r+t+1$ expressions are all zero, which will prove the theorem. 
	
	We first seek to compute the term $\sum_{\varphi_H \in \Phi_H(\varphi_{\mathbb V^*})} \det(\omega_*(\varphi_H)) \Phi^H_{M'}(\gamma'^{-1}, \Theta_{\varphi_H})$ for each fixed $(A,B)$, in a way similar to \S \ref{subsec:comp K}. 
	Fix an elliptic maximal torus $T_{M'}$ of $M'_{\RR}$ such that $\gamma' \in T_{M'}(\RR)$. As usual we have $M' = M^{*,\GL} \times M^{\prime,\SO}$, so necessarily $T_{M'}$ is a direct product of (1) the direct factor $\GG_m^r$ of $M^{*,\GL}$, (2) an elliptic maximal torus in the direct factor $\GL_2^t$ of $M^{*,\GL}$, and (3) an elliptic (anisotropic) maximal torus $T_{M^{\prime,\SO}} = T_{M^{\prime,\SO,+}} \times T_{M^{\prime,\SO,-}}$ in $M^{\prime , \SO} = M^{\prime,\SO,+} \times M^{\prime,\SO,-}$. We denote the product of (1) and (2) by $T_{M^{*,\GL}}$. 
	Note that all of $R_i, T_j, S$ can be viewed as continuous functions in $\gamma'$ varying in $T_{M'}(\RR)$ (cf.~\S \ref{subsubsec:mathbb V}).  Hence we may and shall assume the following condition:
	\begin{enumerate}
		\item[($\dagger$)] The $r$ components of $\gamma'$ in $\GG_m^r \subset M^{*,\GL}$ are distinct from each other and distinct from the inverse of each other. 
	\end{enumerate}
	  Let $r'$ be the number such that exactly $r'$ among the $r$ components of $\gamma'$ in $\GG_{m}^r$ are positive. 
	
	Fix an elliptic maximal torus $T_{M^*}$ in $M^*_{\RR}$ of the form $T_{M^{*,\GL}} \times T_{M^{*,\SO}}$, where $T_{M^{*,\GL}}$ is as above and $T_{M^{*,\SO}}$ is an elliptic (anisotropic) maximal torus in $M^{*,\SO}$. Fix an admissible isomorphism $j_{M^*}: T_{M'} \isom T_{M^*}$ of the form $\id_{T_{M^{*,\GL}}} \times j_{M^{*,\SO}}$, where $j_{M^{*,\SO}}$ is an admissible isomorphism $T_{M^{\prime,\SO}} \isom T_{M^{*,\SO}}$.  
	As in \S \ref{para:B_0}, for any choice of Borel subgroup $B_0$ of $G^*_{\CC}$ containing $T_{M^*,\CC}$, we obtain $m$ cocharacters of $T_{M^*,\CC}$ forming a basis of $X_*(T_{M^*})$. We denote them by \index{$\tau_{0_i}, \tau_i$}
	$$ \tau_{0_1},\cdots,\tau_{0_{r+2t}}, \tau_1,\cdots,\tau_{m-r-2t}. $$ Since we are in the odd case, by making different choices of $B_0$ we can arbitrarily permute the $\tau$'s and replace an arbitrary number of them by their inverses. By similar arguments as in \S \ref{para:B_0}, we can choose $B_0$ such that the following conditions are satisfied. (Here condition \textbf{C} depends on the assumption ($\dagger$) above.) 
	\begin{description}
		\item[\bf A] For each $1\leq i\leq r$, $\tau_{0_i}$ is a cocharacter of the direct factor $\GG_m^r$ of $M^{*,\GL}$. Moreover, there is a permutation $\delta \in \mathfrak S_r$ such that for each $1\leq i\leq r$, $\tau_{0_i}$ is either the identity cocharacter or the inverse of the identity cocharacter of the $\delta(i)$-th copy of $\GG_m$. 
		\item[\bf B] For each $1\leq j \leq t$, $\tau_{0_{r+2j-1}}$ and $\tau_{0_{r+2j}}$ are cocharacters of the $j$-th copy of $\GL_2$ in $M^{*,\GL}$. Moreover, these two are simultaneously $\GL_2$-conjugate to the following cocharacters of $\GL_2$: 
		$$z \longmapsto  \begin{pmatrix}
			z \\ & 1 
		\end{pmatrix}    \qquad \text{and} \qquad z \longmapsto  \begin{pmatrix}
			1\\ & z 
		\end{pmatrix}.$$
	\item[\bf C] Let $\set{\epsilon_1,\cdots,\epsilon_r}$ be the basis of $X^*(\GG_{m}^r)$ dual to the basis $\set{\tau_{0_1},\cdots, \tau_{0_r}}$ of $X_*(\GG_m^r)$. We also view each $\epsilon_i$ as a character on $T_{M^*}$, via the projection from $T_{M^*}$ to the direct factor $\GG_m^r$ of $T_{M^{*,\GL}}$.  For each $1 \leq i \leq r$, we require that 
	\begin{align}\label{eq:cond pos}
	\epsilon_i(\gamma')>0  \text{ if and only if }  i \leq r'. 
	\end{align} 
	For all $1 \leq i < j \leq r'$, or $r'+1 \leq i < j \leq r$, we require that 
	\begin{align}\label{eq:cond increasing}
 \frac{\epsilon_i(\gamma'^{-1}) }{\epsilon_j(\gamma'^{-1})} \in ] 0, 1[,
	\end{align} and 
\begin{align}\label{eq:cond small}
\abs{\epsilon_i(\gamma'^{-1})} < 1. 
\end{align}
	\item[\bf D] Let $n^-$ be the dimension of $T_{M^{\prime, \SO, -} , \CC}$. For each $1 \leq i \leq n^-$, $j_{M^*}^{-1} \circ 
	\tau_i$ is a cocharacter of $T_{M^{\prime, \SO, -} , \CC}$.    
	\end{description}
	
	The pair $(j: T_H \isom T_G, B_{G,H})$ fixed in \S \ref{subsec:test functions} can be transferred to a pair $(\underline j: T_H \isom T_{G^*}, B_{G^*,H})$ as follows. We fix an anisotropic maximal torus $T_{G^*}$ in $G^*_{\RR}$ and an isomorphism $\nu : T_G \isom T_{G^*}$ coming from any inner twisting $G_{\CC} \isom G^*_{\CC}$ in the canonical $G^{*}(\CC)$-conjugacy class of such inner twistings. Then we define $\underline j:  =\nu \circ  j $, and define $B_{G^*,H}$ to be the Borel subgroup  of $G^*_{\CC}$ containing $T_{G^*}$ such that $\nu$ relates all $B_{G,H}$-positive roots on $T_{G,\CC}$ with $B_{G^*,H}$-positive roots on $T_{G^*,\CC}$. From $(\underline j, B_{G^*,H})$, we obtain an ordered $m$-tuple of cocharacters of $T_{G^*,\CC}$ \index{$\underline \rho_1, \underline \rho_2,\cdots$} $$ \underline \rho_1,\cdots,\underline \rho_m $$ similarly as in \S \ref{para:rho's}. 
	Define an isomorphism \index{$i_{G^*}(A,B)$} $i_{G^*}(A,B): T_{M^*,\CC} \isom T_{G^*,\CC}$ by the following rule. Write $m^{\pm}$ for the absolute ranks of $H^{\pm}$, and $n^{\pm}$ for the absolute ranks of $M^{\prime,\SO, \pm}$.  Thus we have 
	\begin{align*}
		m^+ & = n^+ + \abs{A} + 2\abs{B}, \\  m^- & = n^- + \abs{A^c} + 2 \abs{B^c}.
	\end{align*}  Let $\sigma \in \mathfrak S_{m}$ be the unique permutation such that $\sigma^{-1}$ is increasing on $\set{1,2,\cdots, m^-}$ and on $\set{m^-+1,m^- +2, \cdots, m}$, and 
\begin{multline*}
\sigma^{-1}(\set{1,\cdots,m^-})   \\ = A^c \cup \set{r+2j-1, r+2j\mid j\in B^c} \cup \set{r+2t+1,\cdots, r+2t +n^-}.
\end{multline*}  We then require that $i_{G^*}(A,B)$ sends $\tau_{0_1},\cdots, \tau_{0_{r+2t}},\tau_1,\cdots , \tau_{m-r-2t}$ respectively to $\underline \rho_{\sigma(1)}, \cdots, \underline \rho_{\sigma(m)}$. Our $i_{G^*}(A,B)$ is a direct analogue of $i_G(A,B)$ in Definition \ref{defn:i_G(A)}, an it enjoys similar properties as in Lemmas \ref{lem:i_G and i_H} and \ref{lem:B_H'}, with $j$ and $j_M$ replaced by $\underline j$ and $j_{M^*}$. Let $B_{M^*}: = B_0 \cap M^*$, and let 
\begin{align}\label{eq:Delta_{jM*}}
\Delta^{A,B}_{j_{M^*}, B_{M^*}} : = (-1)^{q(G_{\RR}) + q(H_\RR) + q(M^*_{\RR}) + q(M'_{\RR})} \Delta_{j_{M^*}, B_{M^*}}. 
\end{align}
$$ $$
	By \cite[Prop.~3.2.5]{morel2011suite} (cf.~Proposition \ref{prop:morel transfer}) and similar arguments as in \S \ref{para:investigate diff}, and the proof of Lemma \ref{lem:B_H=B_H'}, we have \begin{multline}\label{eq:real transfer in vanishing}
	\sum_{\varphi_H \in \Phi_H(\varphi_{\mathbb V^*})} \det(\omega_*(\varphi_H)) \Phi^H_{M'}(\gamma'^{-1}, \Theta_{\varphi_H}) \\ = \sign(\sigma)  \epsilon_{R}(j_{M^*}(\gamma'^{-1})) \epsilon_{R_H}(\gamma'^{-1}) \Delta^{A,B}_{j_{M^*}, B_{M^*}} (\gamma', j_{M^*}(\gamma')) \Phi^{G^*}_{M^*} (j_{M^*}(\gamma'^{-1}), \Theta^H_{\mathbb V^*}).	  
	\end{multline}
	Here,
	\begin{itemize}
		\item $\sigma$ is the permutation as above, used to define $i_{G^*}(A,B)$. 
		\item $R$ is the set of real roots of $(G^*_{\CC}, T_{M^*,\CC})$, and $\epsilon_R(t)$ is $-1$ to the number of $B_0$-positive roots $\alpha$ in $R$ such that $0<\alpha(t)<1$.
		\item $R_H$ is the set of real roots of $(H_{\CC}, T_{M',\CC})$, and $\epsilon_{R_H}(t')$ is $-1$ to the number of $\alpha \in R_H$ such that $0<\alpha(t')<1$ and such that $\alpha \circ (j_{M^*})^{-1} \circ i_{G_*}(A,B)^{-1} \circ \underline j \in X^*(T_H)$ is a $B_H$-positive root. 
		\item $\Phi^{G^*} _{M^*} (\cdot, \Theta^H_{{\mathbb V^*}} )$ is defined analogously as $\Phi^G_M(\cdot, \Theta)_{\Endos}$ in (\ref{eq:n_eds})  and (\ref{eq:Phi^G_M endos}), with the role of $\mathbb V$ played by $\mathbb V^*$, and the role of $R_{\Endos}$ in (\ref{eq:n_eds}) played by the root system $$R_{H,\gamma'} : = \set{\alpha \in R_H\mid  \alpha (\gamma') > 0 }.$$
		
	\end{itemize}

	We analyze how the terms on the right hand side of (\ref{eq:real transfer in vanishing}) depend on $(A,B)$. 
 We observe that $\epsilon_R(j_{M^*}(\gamma'^{-1}))$ is independent of $(A,B)$, while $R_H $ and $R_{H,\gamma'}$ as above depend only on $A$, not on $B$. To simplify notation we denote \index{$\Phi(\gamma',A)$}
\begin{align}
\label{eq:defn Phi(A)}
\Phi (\gamma', A ) : =    \Phi^{G^*} _{M^*} (j_{M^*} (\gamma'^{-1}) , \Theta_{{\mathbb V^*}} ^H).
\end{align}	 and \index{$R_A$} \index{$R_{A,\gamma'}$}
\begin{align}\label{eq:defn R_A}
R_A: = R_H , \quad R_{A,\gamma'} : = R_{H, \gamma'}.
\end{align}
 	
 	We claim that $\epsilon_{R_H }(\gamma'^{-1})$ is independent of $(A,B)$. Indeed, the roots $\alpha\in R_H$ such that $\alpha \circ (j_{M^*})^{-1} \circ i_{G_*}(A,B)^{-1} \circ \underline j $ are  $B_H$-positive are exactly $\epsilon_i +\epsilon_j$ and  
 	$\epsilon_i - \epsilon_j$ where $i<j$ and $i,j$ simultaneously belong to one of $A$ and $A^c$, together with $\epsilon_i$ for all $i \in [r]$, together with certain characters of the direct factor $T_{M^{^*,\GL}} \cap \GL_2^t$ of $T_{M^*}$ constituting a set independent of $(A,B)$.  Among them, those satisfying $0<\alpha(\gamma'^{-1})<1$ are, by (\ref{eq:cond pos})  (\ref{eq:cond increasing}) (\ref{eq:cond small}), exactly $\epsilon_i +\epsilon_j$ and   
 	$\epsilon_i - \epsilon_j$ where $i<j$ and $i,j$ simultaneously belong to one of the four sets $\set{u\in A \mid u \leq r'}, \set{u\in A^{c}\mid u\leq r'}, \set{ u\in A\mid u >r'} , \set{u \in A^c \mid u > r'}$, together with certain other roots constituting a set  independent of $(A,B)$. The total number of   $\epsilon_i +\epsilon_j$ and   
 	$\epsilon_i - \epsilon_j$ where $i<j$ and $i,j$ simultaneously belong to one of the four sets as above is obviously even. Our claim follows. 
 	
 		In the rest of the proof, we write ``$\cost$''\index{$\cost$} for any quantity that is independent of $(A,B)$. By (\ref{eq:Delta_{jM*}}), (\ref{eq:real transfer in vanishing}), and the above analysis, we have 
\begin{align}\label{eq:first cost}
\sum_{\varphi_H \in \Phi_H(\varphi_{\mathbb V^*})} \det(\omega_*(\varphi_H)) \Phi^H_{M'}(\gamma'^{-1}, \Theta_{\varphi_H})   = \cost \sign(\sigma) (-1)^{q(H_\RR)} \Phi(\gamma',A).
\end{align} 		
 
We now simplify $\sign(\sigma)$ and $(-1)^{q(H_\RR)}$. Define $\omega_0(A)$\index{$\omega_0(A)$} to be the sign of the element $\sigma_A\in\mathfrak S_r$ which sends $\set{1,2,\cdots, \abs{A^c}}$ increasingly to  $A^c$ and sends $\set{\abs{A^c}+1,\cdots, r}$ increasingly to $A$. If we view $\sigma_A$ as an element of $\mathfrak S_{m}$, then $\sigma_A^{-1}\circ\sigma^{-1}$ sends $\set{1,\cdots,m^-}$ increasingly to $$  \set{1,\cdots, \abs{A^c}}\cup \set{r+2j-1, r+2j\mid j\in B^c} \cup \set{r+2t+1,\cdots, r+2t +n^-} ,$$ and sends $\set{m^-+1,\cdots, m }$ increasingly to $$ \set{\abs{A^c}+1,\cdots, r} \cup \set{r+2j-1, r+2j\mid j\in B^c} \cup \set{r+2t  +n^-+1, \cdots, m}.$$ From this, one sees that the sign of $\sigma_A^{-1} 
	\circ \sigma^{-1}$ is $(-1)^{\abs{A} n^-}$ (since all $n^-$ elements of $\set{r+2t+1,\cdots, r+2t +n^-}$ are greater than all $\abs{A}$ elements of $\set{\abs{A^c}+1,\cdots, r}$). Hence we have 
	\begin{align}\label{eq:sign sigma}
\sign(\sigma) = \omega_0(A) (-1) ^{\abs{A} n^-} .
	\end{align} 
	As for $(-1)^{q(H_{\RR})}$, we compute
	\begin{multline*}
		2 q(H_{\RR})  =  m^+(m^+ +1) + m^- (m^- +1) \\
		= (n^+ + \abs{ A} + 2\abs{B}) (n^+ + \abs{ A} + 2\abs{B} +1) + (n^- + \abs{ A^c} + 2\abs{B^c} )(n^- + \abs{ A^c} + 2\abs{B^c} +1) ,
	\end{multline*}
	and so 
	\begin{align}\label{eq:q(H)}
		q(H_{\RR}) \equiv \cost + (m+1)(\abs{A}+ 2\abs{B} ) \equiv \cost + (m+1) \abs{A} \mod 2.
	\end{align}

	Plugging (\ref{eq:sign sigma}) and (\ref{eq:q(H)}) into (\ref{eq:first cost}), we get	\begin{align*}
		\sum_{\varphi_H \in \Phi_H(\varphi_{\mathbb V^*})} \det(\omega_*(\varphi_H)) \Phi^H_{M'}(\gamma'^{-1}, \Theta_{\varphi_H}) = \cost \omega_0(A) (-1)^{\abs{A}(n^- +m +1)} \Phi(\gamma',A). 
	\end{align*} 
	Hence \begin{align} \label{eq:temp1}
		R_i & = \cost  \sum_{A,B} \nabla_i(A) \epsilon^{p,\infty}(A,B) \omega_0(A) (-1)^{\abs{A}(n^- +m +1)} \Phi(\gamma',A), \\ \label{eq:temp2}
		 	T_j & = \cost  \sum_{A,B} \nabla_j(B) \epsilon^{p,\infty}(A,B) \omega_0(A) (-1)^{\abs{A}(n^- +m +1)} \Phi(\gamma',A),\\ \label{eq:temp3}
		 	S & = \cost  \sum_{A,B}  \epsilon^{p,\infty}(A,B) \omega_0(A) (-1)^{\abs{A}(n^- +m +1)} \Phi(\gamma',A).
	\end{align}

	We now compute $\epsilon^{p,\infty}(A,B)$. 
	Let $(H,\lang H, s,\eta)$ be determined by $(A,B)$. For each  place $v$, as explained in Remark \ref{rem:pure inner form}, the choice of $\phi_{V_{\QQ_v}}: V_{\QQ_v} \otimes \overline \QQ_v \isom \underline V_{\QQ_v}\otimes {\overline \QQ_v}$ and the resulting pure inner twist $(\psi_{V_{\QQ_v}}, u_{V_{\QQ_v}})$ allows us to pass between normalizations of transfer factors between $H$ and $G$ and  between $H$ and $G^*$ at $v$. Hence we obtain from $(\Delta^G_H)_v$ a normalization $(\Delta^{G^*}_H)_v$ of transfer factors between $H$ and $G^*$ at $v$, and then inherit from the latter a normalization $(\Delta^{M^*}_{M'})^{A,B}_v$ of transfer factors between $M'$ and $M^*$ at $v$. For each finite $v \neq p$, we have
	$$\frac{(\Delta_{M'}^{M^*}) ^{A,B}_v}{(\Delta_{M'}^{M^*})^{\emptyset,\emptyset}_v} = \frac{(\Delta_{M'}^{M_v}) ^{A,B}_v}{(\Delta_{M'}^{M_v})^{\emptyset,\emptyset}_v}, $$
	and so
	 	$$ \epsilon^{p,\infty}(A,B) = \prod_{v \neq p,\infty} \frac{(\Delta_{M'}^{M^*}) ^{A,B}_v}{(\Delta_{M'}^{M^*})^{\emptyset,\emptyset}_v} .$$
 Recall that the normalizations $(\Delta^G_H)_v$ for all places $v$ satisfy the global product formula. We claim that 
	$ (\Delta^{M^*}_{M'})^{A,B}_{v}$ for all $v$ also satisfy the global product formula, for which we provide an argument that also works in the even case. Recall from Remarks \ref{rem:remember phi_V} and \ref{rem:pure inner form} that for each $v$ we have the freedom of changing $\phi_{V_{\QQ_v}}: V_{\QQ_v} \otimes_{\QQ_v} \overline {\QQ}_v \isom \underline V_{\QQ_v} \otimes_{\QQ_v} \overline {\QQ}_v$ by composing it with an element of $G^*(\overline \QQ_v)$. Also recall the compatibility condition (1) imposed in \S \ref{para:local global comp}. Thus for the sake of proving the claim,  we may replace each $\phi_{V_{\QQ_v}}$ by the isomorphism $V \otimes_{\QQ} \overline \QQ_v \isom \underline V \otimes_{\QQ} \overline \QQ_v$ induced by the global $\phi_V : V\otimes_{\QQ}\overline \QQ \isom \underline V_{\QQ} \otimes _{\QQ} \overline \QQ$. Then one sees that $(\Delta^{G^*}_H)_v$ for all $v$ satisfy the global product formula, since the local cocycles $u_{V_{\QQ_v}} : \rho \mapsto \lix^{\rho} \phi_{V_{\QQ_v}} \phi_{V_{\QQ_v}}^{-1}$ come from the global cocycle $u_V: \rho \mapsto \lix^{\rho} \phi_V \phi_V^{-1}$. Therefore the inherited normalizations $(\Delta^{M^*}_{M'})^{A,B}_{v}$ also satisfy the global product formula.

	By our claim, the product $\prod_v (\Delta^{M^*}_{M'})^{A,B}_{v}$ over all places $v$ is  independent of $(A,B)$. Hence 	$$ \epsilon^{p,\infty}(A,B) = \prod_{v \in \set{p,\infty}} \frac{(\Delta_{M'}^{M^*})^{\emptyset,\emptyset}_v} {(\Delta_{M'}^{M^*}) ^{A,B}_v} .$$
	Now $V_{\QQ_p}$ is quasi-split by our assumption that $G_{\QQ_p}$ is unramified (in particular split) and by Proposition \ref{prop: even TFAE}. Hence there exists $g\in G^*(\overline \QQ_p)$ such that $g\circ \phi_{V_{\QQ_p}}$ is defined over $\QQ_p$. (Clearly we can find $g' \in \mathrm{O}(\underline V)(\overline \QQ_p)$ such that $g' \circ \phi_{V_{\QQ_p}}$ is defined over $\QQ_p$. We can then construct $g$ by left multiplying $g'$ by any element of $\mathrm{O}(\underline V)(\QQ_p)$ of determinant $-1$, which exists.) It then follows that   $(\Delta^{M^*}_{M'})^{A,B}_{p}$ is the canonical unramified normalization associated to a hyperspecial subgroup of $M^*(\QQ_p)$ that is independent of $(A,B)$.\footnote{In the current odd case, all hyperspecial subgroups of $M^*(\QQ_p)$ are conjugate under $M^*(\QQ_p)$, so the canonical unramified normalizations associated to all hyperspecial subgroups are actually equal to each other. This is no longer true in the even case. Nevertheless, the statement in the text remains true in the even case, as long as there exists $g\in G^*(\overline \QQ_p)$ such that $g\circ \phi_{V_{\QQ_p}}$ is defined over $\QQ_p$.} Hence $(\Delta^{M^*}_{M'})^{A,B}_{p}$ is independent of $(A,B)$. We conclude that 	$$ \epsilon^{p,\infty}(A,B) =   \frac{(\Delta_{M'}^{M^*})^{\emptyset,\emptyset}_\infty} {(\Delta_{M'}^{M^*}) ^{A,B}_\infty} .$$
	
	By the same argument as in the proof of Proposition \ref{prop:computing tasho} (see the ``claim'' in that proof),  the Whittaker normalization between $M'$ and $M^*$ at $\infty$ is inherited from the Whittaker normalization between $H$ and $G^*$ at $\infty$. The former is independent of $(A,B)$. Hence $\epsilon(A,B)$ is up to a non-zero multiplicative constant equal to the ratio of the Whittaker normalization between $H$ and $G^*$ at $\infty$ to the normalization   $(\Delta^{G^*}_H)_{\infty}$. This ratio is the same as the ratio of the Whittaker normalization between $H$ and $G$ to $(\Delta^G_H)_{\infty} = \Delta_{j,B_{G,H}}$, which is equal to   
	$$ (-1) ^{\ceil{m^+/2} +1 } = (-1) ^{\ceil{\frac{n^+ + \abs{A} + 2 \abs{B} } {2}}  +1 } $$ as shown in the proof of Proposition \ref{prop:computing tasho}. 
	When $n^+$ is even, the above is equal to $\cost (-1) ^{\ceil{\abs{A} /2} + \abs{B}		}$. When $n^+$ is odd, the above is equal to $\cost (-1) ^{\floor{\abs{A} /2} + \abs{B}		}$. In both cases, taking into account the equality $m = n^+ + n^- + r + 2t$, we obtain:
	\begin{align*}
\epsilon^{p,\infty}(A,B) (-1)^{\abs{A}(n^- +m +1)} = \cost  (-1) ^{r\abs{A } + \floor{\abs{A} /2}  + \abs{B}}.
	\end{align*}Plugging this into (\ref{eq:temp1}), (\ref{eq:temp2}), and (\ref{eq:temp3}), we obtain 
\begin{align} \label{eq:new defn of R}
 R_i & = \cost \sum_{A, B} \nabla_i(A)  \omega_0(A) (-1) ^{r\abs{A } + \floor{\abs{A} /2}  + \abs{B}} \Phi (\gamma', A) , \\  \label{eq:new defn of T}
 T_j  & = \cost \sum_{A,B} \nabla_j(B) \omega_0(A) (-1) ^{r\abs{A } + \floor{\abs{A} /2}  + \abs{B}}\Phi(\gamma', A)  , \\ \label{eq:new defn of S}
S & = \cost \sum_{A, B}  \omega_0(A) (-1) ^{r\abs{A } + \floor{\abs{A} /2}  + \abs{B}} \Phi(\gamma', A ), 
\end{align}where $A$ runs through subsets of $[r]$ and $B$ runs through subsets of $[t]$. 
We need to show that the right hand sides are all zero. This we accomplish in the next proposition. \end{proof}
	
	\begin{prop}\label{prop:step 2 in odd vanishing} Assume $rt> 0$, or $ r \geq 3$, or $t \geq 2$. 
The right hand sides of (\ref{eq:new defn of R}) (\ref{eq:new defn of T})  (\ref{eq:new defn of S}) are all zero.
	\end{prop}
\begin{proof}

	We first treat the case $t\geq 2$, which is the easiest. In this case we have the elementary combinatorial identities 
	\begin{align}\label{eq:easy vanishing}
\sum_{B\subset [t]} (-1) ^{\abs{B}} =0  
	\end{align} and 
	\begin{align}\label{eq:elem id 2}
\sum_{B\subset [t]} \nabla_j (B) (-1) ^{\abs{B}} &  = \sum_{k = 0} ^{t} (-1) ^k \bigg[ \# \set{B\mid  \abs{B}= k , j\in B}  - \# \set{B\mid  \abs{B}= k ,j\notin B}\bigg]  \\ \nonumber & = \sum_{k = 0} ^{t} (-1) ^k \left[ {{t-1}\choose {k-1}} - {{t-1}\choose{k}} \right]\\ \nonumber & = -2 \sum_{k=0}^{t-1} (-1)^k {{t-1}\choose {k}} = 0 .
 	\end{align}
	(Note that for $t=1$, (\ref{eq:easy vanishing}) still holds, but  $\sum_B \nabla_1(B) (-1)^{\abs{B}} =-2$.)
	Hence we have $R_i = T_j = S =0$ in this case, and the proof is finished. 
	
Before treating the other cases, we observe that  
$$\omega_0(A) \omega_0(A^c) = (-1)^{\abs{A} \abs{A^c}},$$ from which  
\begin{align}\label{A and A^c}
\omega_0(A) (-1) ^{r\abs{A } + \floor{\abs{A} /2}  } \omega_0(A^c) (-1) ^{r\abs{A ^c} + \floor{\abs{A^c} /2} } = (-1) ^{\ceil{r/2}}.
\end{align} 

Now suppose $rt >0$ and $r \in \set{1,2}$. Again   (\ref{eq:easy vanishing}) holds, so $R_i = S =0$. To show $T_j =0$, observe that $\Phi(\gamma', A  ) = \Phi(\gamma' , A^c )$, so it suffices to show that (\ref{A and A^c}) is $-1$, which is indeed true for $r=1,2$.

	Finally we treat the case $r \geq 3$, which is the most complicated. We need a computation that is similar to \cite[pp.~1698-1699]{morel2011suite}, applying the result of Herb \cite{herb}. 
In the following we will view $\gamma'$ and $B$ as being fixed, and let $A$ vary. 
 
 We have $$A_{M'} = A_{M^*} = \GG_m^r \times \GG_m^t,$$ where the factor $\GG_m^r$ is the canonical copy of $\GG_m^r$ in $M^{*,\GL} = \GG_m^r \times \GL_2^t$, and the factor $\GG_m^t$ is the product of the centers of the $t$ copies of $\GL_2$ in $M^{*,\GL}$. Let $\epsilon_1,\cdots, \epsilon_r \in X^*(\GG_m^r)$ be as in condition \textbf{C} satisfied by $B_0$ in the proof of Theorem \ref{vanishing odd}. Let $\set{\alpha_1,\cdots, \alpha_t}$ be the standard basis of $X^* (\GG_m^t)$.  Define 
\begin{align*}
I^+ &  : = \set{i\in [r] \mid  \epsilon_i(\gamma') > 0} , &  I^- & : = [r] - I^+ ,\\
A^+ &: = A \cap I^+, &  A^- & : = A\cap I^-,  \\ A^{c,+} & : = A^c \cap I^+,&   A^{c,-} & :  = A^c \cap I^-.\end{align*}By (\ref{eq:cond pos}), we know that $I^+ = [r']$. 
	
	Let $R_{A,\gamma'}  = R_{H,\gamma'}$ be the real root system involved in the definition of $\Phi(\gamma' ,A)$; see (\ref{eq:defn Phi(A)}) and (\ref{eq:defn R_A}). Then $R_{A, \gamma'}$ is of type  
	\begin{align}\label{eq:Dynkin type}
\mathsf B_{\abs{A^+}} \times \mathsf B_{\abs{A^{c,+}}} \times \mathsf D_{\abs{A^-}} \times \mathsf D_{\abs{A^{c,-}} } \times \mathsf  A_1^{\times t}  ,
	\end{align}
	where $\mathsf B_{\abs{A^{+}}}$ consists of the roots $$  \epsilon_i, ~\epsilon_i \pm \epsilon_j, ~ i,j \in A^{+}, i\neq j$$ and $\mathsf D_{\abs{A^{-}}}$ consists of the roots $$\epsilon_i \pm \epsilon_j, ~ i,j \in A^- , i\neq j  ,$$ and similarly for $\mathsf B_{\abs{A^{c,+}}}$ and $\mathsf D_{\abs{A ^{c,-}}}$. The part $\mathsf A_1^{\times t}$ consists of the roots\footnote{This follows from the following argument: Let $\epsilon_1, \epsilon_2$ denote the two standard characters on the diagonal torus in $\GL_2$, and identify them with two characters on an elliptic maximal torus in $\GL_{2,\RR}$. Then with respect to the real structure of the latter, $\pm (\epsilon _1 + \epsilon_2 )$ are the only real characters among $\epsilon_1, \epsilon_2, \epsilon_1 \pm \epsilon_2, -\epsilon _1 \pm \epsilon_2$. } 
	$$\pm 2\alpha_1,\cdots , \pm 2 \alpha_t.$$ 
	
	By (\ref{eq:Dynkin type}), we see that the Weyl group  of $R_{A,\gamma'}$ contains $-1$ if and only if $\abs{A^-}$ and $\abs{A^{c,-}}$ (and \textit{a fortiori} $\abs{I^-}$) are even, if and only if $ \gamma' \in H(\RR) ^0$. These conditions are necessary for $\Phi(\gamma', A)$ to be non-zero. 
	Assume that these conditions are satisfied. Then $$ \Phi(\gamma', A) = \sum_{\omega \in \Omega} C(\gamma',\omega)  n_A (\gamma', \omega B_0),$$
where $\Omega$ is the complex Weyl group of $G^*$, the coefficients $C(\gamma',\omega)$ are independent of $A$, and
$$n_A (\gamma' ,\omega B_0) : = \bar c_{R_{A,\gamma'}} ( x, \wp ( \omega \lambda_{B_0} + \omega  \rho_{B_0})), $$ with notations explained below:
\begin{itemize}
	\item 
	$x \in X_*(A_{M^*}) _{\RR}$ is characterized by the condition 
	\begin{align}\label{eq:characterization of x}
	j_{M^*} (\gamma'^{-1} )\in \exp (x)  T_{M^*} (\RR) _1  \subset T_{M^*} (\RR) ,
	\end{align} where $T_{M^*} (\RR) _1$ is the maximal compact subgroup of $ T_{M^*} (\RR).$
	\item $\wp : X^*(T_{M^*})_{\RR} \to X^*(A_{M^*}) _{\RR}$ is the natural restriction map. 
	\item  $\rho_{B_0}$ is the half sum of the $B_0$-positive (absolute) roots in $X^*(T_{M^*})$, and $\lambda_{B_0} \in X^*(T_{M^*}) $ is the $B_0$-highest weight of $\mathbb V^*$. 
	\item $\bar c _{R_{A,\gamma'}} (\cdot , \cdot )$\index{$\bar c _{R_{A,\gamma'}}$} is the function associated to the root system $R_{A,\gamma'} \subset X^*(A_{M^*})_{\RR}$ as in (\ref{eq:bar c}).
\end{itemize}
We note that  $$\chi:  = \wp ( \omega \lambda_{B_0} + \omega  \rho_{B_0}) \in X^*(A_{M^*}) _{\RR}$$ is independent of $A$. In the following we will use only this property of $\chi$.

	Thus to show that $R_i = T_j = S = 0$, it suffices to show that the following  quantities are zero, where the summations are over $A \subset [r]$ such that  $\abs{A^-}$ and $\abs{A^{c,-}}$ are both even: 
	\begin{align}\label{eq:defn M_i}
M_i & : = \sum_A \nabla_i(A) \omega_0(A) (-1) ^{r\abs{A} + \floor{\abs{A} /2} } \bar c_{R_{A,\gamma'}} (x, \chi), \qquad 1\leq i \leq r  \\\label{eq:defn N}
N  & : = \sum_A  \omega_0(A) (-1) ^{r\abs{A} + \floor{\abs{A} /2} } \bar c_{R_{A,\gamma'}} (x, \chi). 
	\end{align}
	More precisely, the vanishing of $M_i$ implies the vanishing of $R_i$, and the vanishing of $N$ implies the vanishing of $T_j $ and $S$. 
	We show the vanishing of $M_i$ and $N$ (for $r \geq 3$) in the next proposition. \end{proof}

	 \begin{prop} \label{prop:step 3 in odd vanishing} Let $x \in X_*(A_{M^*}) _{\RR}$ be characterized by the condition (\ref{eq:characterization of x}), where $\gamma' \in T_{M'} (\RR)$ satisfies the conditions (\ref{eq:cond pos}), (\ref{eq:cond increasing}), and (\ref{eq:cond small}). Let $\chi \in X^*(A_{M^*}) _{\RR}$ be an element independent of $A$. When $r \geq 3$, the quantities $M_i$ and $N$ in (\ref{eq:defn M_i}) and (\ref{eq:defn N}) are zero.
 	 \end{prop}
	
	\subsection{}
	In the proof of Proposition \ref{prop:step 3 in odd vanishing} we need to apply Herb's formula for $\bar c_{R_{A,\gamma'}}$, which we now recall. We will follow the notation and definitions of \cite[pp.~1698-1699]{morel2011suite}. Note that in \textit{loc.~cit.~}root systems of types $\mathsf C$ and $\mathsf D$ are considered, whereas  we need to consider root systems of types $\mathsf B$ and $\mathsf D$. Nevertheless the formulas for type $\mathsf B$ and type $\mathsf C$ root systems are identical; see \cite{herb}. 
	
	For $a, b \in \RR$, we define \index{$c_1, c_{2,\mathsf B}, c_{2,\mathsf D}$} \begin{align*}
c_1 (a) & : = \begin{cases}
1, & \text{if } a>0, \\ 0, & \text{otherwise}.
\end{cases} \\
 c_{2,\mathsf B} (a,b)  & : = \begin{cases}
1, & \text{if } 0 < a< b \mbox{ or } 0< -b < a , \\ 0, &  \text{otherwise}.
\end{cases} \\
c_{2,\mathsf D} (a,b) & : = \begin{cases}
1, & \text{if } a> \abs{b}  , \\ 0, & \text{otherwise}.
\end{cases}
	\end{align*}
Our $c_{2,\mathsf B}$ is equal $c_{2,\mathsf C}$ in \cite{morel2011suite}. \ignore{We also remark that the specific positions of the non-zero areas of the above functions are designed to suit the conditions (1) (2) above (\ref{eq:defn Phi(A)}).\footnote{To compare with \cite{morel2011suite}, our $\gamma'^{-1}$ is analogous to $\gamma_M$ in loc. cit.} }
	
	Let $I$ be a finite set. We will denote an unordered partition $p$ of $I$ by $p = \set{I_z \mid  z\in Z}$, where $Z$ is the indexing set, and $I = \coprod_{z\in Z} I_z$. Let $\mathcal P^0_{\leq 2} (I)$\index{$\mathcal P^0_{\leq 2} (I)$} be the set of unordered partitions $\set{I_z\mid z\in Z}$ of $I$ such that  all $I_z$ have cardinality $2$ or $1$ and at most one $I_z$ has cardinality $1$. If $I$ is equipped with a total order $\leq$, we can define a sign function \index{$\epsilon: \mathcal P^0_{\leq 2} (I) \to \set{\pm 1}$} 
	\begin{align}\label{eq:epsilon on P^0}
\epsilon: \mathcal P^0_{\leq 2} (I) \To \set{\pm 1}
	\end{align}
as follows. Given $p  \in \mathcal P^0_{\leq 2} (I)$, we  enumerate the elements of $p$ as $I_1,\cdots, I_k$, and let $\sigma$ be the unique bijection $I \isom I$ satisfying the following conditions:  
	\begin{itemize}
		\item For all $i,j \in [k]$ with $i<j$, and for all $s \in \sigma (I_{i})$ and $s' \in \sigma (I_{j})$, we have $s < s'$.
		\item If $i \in [k]$ is such that $\abs{I_i} = 2 $, then $\sigma$ is increasing on $I_i$.  
	\end{itemize}    With respect to the total order on $I$, the permutation $\sigma$ of $I $ has a well-defined sign. We define $\epsilon(p)$ to be that sign. This definition does not depend on the enumeration of the elements of $p$. 
	
	For $\mu \in \RR^r$ and $J$ a subset of $[r]$ of cardinality $1$ or $2$, we make the following definitions. If $J = \set{s}$, define \index{$c_{J,\mathsf B}, c_{J,\mathsf D}$}
	$$ c_{J,\mathsf B} (\mu) : = c_1(\mu_s).$$ If $J = \set{s_1 , s_2}$ with $s_1< s_2$, define 
	\begin{align*}
c_{J, \mathsf B} (\mu) & : = c_{2,\mathsf B} (\mu_{s_1} , \mu_{s_2}) ,\\ 
c_{J, \mathsf D} (\mu) & : = c_{2,\mathsf D} (\mu_{s_1} , \mu_{s_2}).
	\end{align*}
		
	Now for $I \subset [r]$ and $p = \set{I_z\mid  z\in Z} \in \mathcal P^0_{\leq 2} (I)$, define \index{$c_{\mathsf B} (p,\cdot) , c_{\mathsf D} (p,\cdot)$}
	$$c_{\mathsf B} (p,\mu) : =  \prod_{z\in Z} c_{I_z, \mathsf B}  (\mu) .$$
	If in addition $\abs{I}$ is even, define 
	 $$c_{\mathsf D} (p,\mu) : =  \prod_{z\in Z} c_{I_z, \mathsf D}  (\mu). $$
	
	Let $\chi \in X^* (A_{M^*}) _{\RR}$ and let $\mu$ be its projection to $X^*(\GG_m^r) _{\RR}$. We identify $X^*(\GG_m^r) _{\RR}$ with $\RR^r$ using the basis $\set{\epsilon_1,\cdots,\epsilon_r}$ fixed in the proof of Theorem \ref{vanishing odd} (as opposed to the standard basis), and view $\mu$ as an element of $\RR^r$. Let $x$ be as in the statement of the Proposition \ref{prop:step 3 in odd vanishing}. Then Herb's formula states that  
	\begin{multline} \label{eq:Herb}
\bar c_{R_{A , \gamma'}} (x, \chi)    = \cost \sum_{p_1 ^+ \in \mathcal P^0_{\leq 2}(A^+) } \sum _{p_1^- \in \mathcal P^0_{\leq 2}(A^-)}   \sum_{ p_2 ^+ \in \mathcal P^0_{\leq 2} (A^{c,+}) } \sum _{p_2^- \in \mathcal P^0_{\leq 2} (A^{c,-})} \\  \epsilon(p_1^+) \epsilon(p_1^-) \epsilon(p_2^+) \epsilon(p_2^-)   c_{\mathsf B} (p_1^+ ,\mu) c_{\mathsf B} (p_2^+, \mu) c_{\mathsf D} (p_1^- , \mu) c_{\mathsf D} (p_2^- ,\mu),
	\end{multline}
where $\cost$ is independent of $A$.
	 \begin{rem}
To compare (\ref{eq:Herb}) with the formula on p.~1699 of \cite{morel2011suite}, note that the root system considered in  \textit{loc.~cit.~}is of type $\mathsf C_{\abs{A_1^-}} \times \mathsf C_{\abs{A_1^+}} \times 
\mathsf D_{\abs{A_2^-}} \times \mathsf D_{\abs{A_2 ^+}} \times \mathsf A_1^{\times t}$, whereas our root system is $\mathsf B_{\abs{A^+}} \times \mathsf B_{\abs{A^{c,+}}} \times \mathsf D_{\abs{A^-}} \times \mathsf D_{\abs{A^{c,-}} } \times \mathsf A_1^{\times t}.$ Our $\gamma'^{-1}$ plays the same role as $\gamma_M$ in \textit{loc.~cit.}. 
	 \end{rem}
 
 \begin{proof}[Proof of Proposition \ref{prop:step 3 in odd vanishing}]
 	We divide the proof into  two cases according to the parity of $r$ . 
 	
	\textbf{The case where $r \geq 3$ is odd.}
	
	 Since $\abs{A^-}$ and $\abs{A^{c,-}}$ must be even, we know that
	$\abs{A^+}$ and $\abs{A^{c,+}}$ must have different parity. In particular $I^+$ has odd cardinality. Write $\abs{I^+}= 2k-1$ with $k \geq 1$, and write $\abs{I^-} = 2l$ with $l \geq 0$.
	
	 For $p_1^+ \in \mathcal P^0_{\leq 2}(A^+)$ and $p_2^+ \in \mathcal P^0_{\leq 2}(A^{c,+})$, we have $p^+: = p_1^+ \cup p_2^+ \in \mathcal P^0_{\leq 2}(I^+)$. Also for $p_1^- \in \mathcal P^0_{\leq 2}(A^-)$ and $p_2^- \in \mathcal P^0_{\leq 2}(A^{c,-})$, we have $p^-: = p_1^- \cup p_2^- \in \mathcal P^0_{\leq 2}(I^-)$. We also have $$ \omega_0(A) \epsilon(p_1^+) \epsilon(p_1^-) \epsilon(p_2^+) \epsilon(p_2^-)  = \epsilon (p^+) \epsilon(p^-). $$ 
	In this way we have ``encoded'' the quadruple $(p_1^+, p_2^+, p_1^-, p_2^-)$ and the left hand side of the above equality into $(p^+, p^-)$.

	Conversely, we explain how to recover $(p_1^+, p_2^+)$ from $p^+$ with extra data, and recover $(p_1^-, p_2^-)$ from $p^-$ with extra data.  Given $p^+ \in \mathcal P^0_{\leq 2}(I^+)$, write $p^+ = p^+(2) \sqcup p^+(1)$, where $p^+(2)$ consists of the cardinality-$2$ members of $p^+$ and $p^+(1)$ consists of the singleton member of $p^+$. (Note that $\abs{p^+(2)} =k-1$ and $\abs{p^+(1)} = 1$.) To recover $(p_1^+, p_2^+)$  is the same as to recover the subset $A^+ $ of $I^+$. For that it suffices to specify a subset $U$ of $p^+(2)$ and a subset $V$ of $p^+(1)$ such that $A^+ = \bigcup_{ I \in U \cup V} I$. Thus we have established a bijection from the set of $(A^+, p_1^+, p_2^+)$ to the set of $(p^+, U, V)$. Under this bijection, we have $\abs{A^+} = 2\abs{U} + \abs{V}$. For a fixed $i \in I^+$, we can also encode the function $A^+ \mapsto \nabla_i (A^+)$ into a function in the variables $p^+$, $U$, and $V$ as follows. Define 
	$$\nabla_i (p^+, U, V): =\begin{cases}
	1, & \text{if } i\in I \text{ for some } I \in U\cup V, \\ 
	-1, & \text{otherwise}.
	\end{cases}$$
	Then we have $\nabla_i(A^+) = \nabla_i (p^+, U, V)$ if $(A^+,p_1^+, p_2^+)$ corresponds to $(p^+, U, V)$ as above.  
	
	 Similarly, given $p^- \in \mathcal P^0_{\leq 2}(I^-)$, to recover $(p_1^- , p_2^-)$ or equivalently $A^- $, it suffices to specify a subset $W$ of $p^-$ such that $A^-  = \bigcup _{I \in W} I$. This again establishes a bijection from the set of $(A^-, p_1^-, p_2^-)$ to the set of $(p^-, W)$. We have $\abs{A^-} = 2\abs{W}.$ For a fixed $i\in I^-$, define 
	 	$$\nabla_i (p^-, W): =\begin{cases}
	 	1, & \text{if } i\in I \text{ for some } I \in W, \\ 
	 	-1, & \text{otherwise}.
	 \end{cases}$$
  Then we have $\nabla_i (A^-) = \nabla_i (p^- , W)$.

	 In conclusion, we may change the summation index $(p_1^+, p_1 ^-, p_2^+, p_2^-)$ in (\ref{eq:Herb}) into the new summation index $(p^+, p^-, U, V, W)$, and obtain
 	 \begin{align*}
 	N & = \cost \sum _{p^+ \in \mathcal P^0_{\leq 2}(I^+)}     \sum _{p^- \in \mathcal P^0_{\leq 2}(I^-)} \epsilon(p^+) \epsilon(p^-) c_{\mathsf B} (p^+,\mu)  c_{\mathsf D} (p^- , \mu)  \\   & \cdot \sum_{U\subset p^+(2), V \subset p^+(1), W \subset p^- } (-1) ^{r (2\abs{U} + \abs{V} + 2\abs{W} )+ \floor{ (2\abs{U} + \abs{V} + 2\abs{W}) /2 } }   \\ 
 	& = \cost   \sum_{U\subset [k-1], V \subset [1], W \subset [l] } (-1) ^{\abs{U}+ \abs{V} + \abs{W}   }  ,
 \end{align*} and for $i\in [r]$ 
	 \begin{multline*}
M_i = \cost \sum _{p^+ \in \mathcal P^0_{\leq 2}(I^+)}   \sum _{p^- \in \mathcal P^0_{\leq 2}(I^-)} \epsilon(p^+) \epsilon(p^-) c_{\mathsf B} (p^+,\mu)  c_{\mathsf D} (p^- , \mu)   \\   \cdot  \sum_{U\subset p^+(2), V \subset p^+(1), W \subset p^- } (-1)^{\abs{U} + \abs{V} + \abs{W}  }   \nabla_i(p^+, p^- , U, V, W), 
	 \end{multline*} where 
 $$\nabla_i (p^+, p^- , U, V ,W) : = \begin{cases}
 	\nabla_i (p^+, U , V) , & \text{if } i \in I^+ ,\\
 	\nabla_i(p^-, W), & \text{if }  i \in I^-.
 \end{cases} $$
 Note that  
 \begin{align}\label{eq:sum over V} 
 \sum_{V \subset [1]} (-1)^{\abs{V}} = 0 .
 \end{align}
 Hence $ N= 0 $ as desired. 
	To show $M_i =0$, it suffices to prove that for each fixed $p^+ \in \mathcal P^0_{\leq 2}(I^+)$ and $ p^- \in \mathcal P^0_{\leq 2}(I^-)$, the quantity  
	$$ L : = \sum_{U\subset p^+(2), V \subset p^+(1), W \subset p^- } (-1)^{\abs{U} + \abs{V} + \abs{W}}   \nabla_i(p^+, p^- , U, V, W)$$ is zero. By definition, depending on the relative position of $(p^+,p^-,i)$, the term $\nabla_i(p^+, p^- , U, V, W) $ is either independent of $V$, or independent of $(U,W)$. In the first case, we know $L=0$ because of (\ref{eq:sum over V}). In the second case, unless $k=1$ and $l = 0$, we have either $$\sum_{U\subset p^+(2)}(-1)^{\abs{U}} = \sum_{U\subset [k-1]}(-1)^{\abs{U}} = 0$$ or $$\sum_{W\subset p^- }(-1)^{\abs{W}} = \sum_{W\subset [l]}(-1)^{\abs{W}} = 0,$$ and therefore $L=0. $ But if $k=1$ and $l=0$, then $r = \abs{I^+} + \abs{I^-} = 2k -1 +2l =1$, a contradiction. Thus $L = 0 $ as desired. The proof of the proposition for odd $r \geq 3$ is complete.

\textbf{The case where $r \geq 3$ is even.}

Now $\abs{I^+}$ and $\abs{I^-}$ are both even. Write $ \abs{I^+} = 2k$ and $\abs{I^-}= 2l$, with $k,l \geq 0$ and $k+l =r/2 \geq 2$. 

We need some combinatorial preparations. For a finite set $I$ of even cardinality, we define \index{$\mathcal P'(I)$}$\mathcal P' (I)$ to be the set of unordered partitions $p=\set{I_z \mid  z\in Z}$ of $I$ equipped with a marked element of $p$ such that exactly two members of $p$ are singletons, all the other members of $p$ have cardinality $2$, and the marked element of $p$ is one of the two singleton members.  When $I$ is equipped with a  total order $\leq $, we define a map \index{$\epsilon: \mathcal P'(I) \to \set{\pm 1}$} $$\epsilon: \mathcal P'(I) \To \set{\pm 1}$$  as follows. Given $p \in \mathcal P'(I)$,  we can merge the two singletons in $p$ into a cardinality-$2$ set and obtain an element $p_0 \in \mathcal P^0_{\leq 2}(I)$. Then we define $\epsilon(p)$ to be $\epsilon(p_0)$ if the marked singleton in $p$ is greater than the other singleton in $p$, and define $\epsilon(p)$ to be $ - \epsilon(p_0)$ otherwise.  Here $\epsilon(p_0)$ is as in (\ref{eq:epsilon on P^0}). If $I$ is a subset of $[r]$ and $p\in \mathcal P' (I)$, we define
$$ c_{\mathsf B} ( p ,\mu) : = \prod _{z\in Z} c_{I_z} (\mu), $$ 
where $\set{I_z \mid  z\in Z}$ is the partition of $I$ underlying $p$.

We now seek to change the summation index in (\ref{eq:Herb}) in a similar manner as in the previous case with odd $r$. If $\abs{A^+}$ is odd then so is $\abs{A^{c,+}}$. In this case $k \geq 1$, and from each $p_1^+ \in \mathcal P^0_{\leq 2}(A^+)$ and $p_2^+ \in \mathcal P^0_{\leq 2}(A^{c,+})$, we obtain an element $p^+ : = p_1^+ \cup p_2^+ \in \mathcal P' (I^+)$, where the marked singleton in $p^+$ is defined to be the singleton in $p_1^+$. Conversely, suppose $k\geq 1$ and suppose  $p^+ \in \mathcal P' (I^+)$. Write $p^+ = p^+(2) \sqcup \set{I_{p^+} ^u, I_{p^+}^m}$, where $p^+(2)$ consists of the cardinality-$2$ members of $p^+$, and we denote by $I_{p^+}^u$ and $I_{p^+}^m$ the unmarked and marked singleton members of $p^+$ respectively. (Note that $\abs{p^+(2)} = k-1$.) Then we can recover $A^+$ from $p^+$ together with a subset $U$ of $p^+(2)$ such that $A^+ = \bigcup_{I \in U} I  \cup I_{p^+}^m$. We have $\abs{A^+} = 2 \abs{U} +1$. For $i\in I^+$, define $$\nabla_i  (p^+, U) : = \begin{cases}
1 , &\text{if }i \in I \text{ for some }I \in U \text{ or }i\in I_{p^+}^m, \\ 
-1, & \text{otherwise}.
\end{cases} $$    	 
	 Then we have $\nabla_i (A^+) = \nabla_i (p^+ , U) $. 
	 
If $\abs{A^+}$ is even, then so is $\abs{A^{c,+}}$. From each $p_1^+ \in \mathcal P^0_{\leq 2}(A^+)$ and $p_2^+ \in \mathcal P^0_{\leq 2}(A^{c, +})$, we obtain $p^+ : = p_1^+ \cup p_2^+ \in \mathcal P^0_{\leq 2}(I^+)$. Conversely, given $p^+ \in \mathcal P^0_{\leq 2}(I^+)$, to recover $A^+$ it suffices to specify a subset $U$ of $p^+$ such that $A^+ = \bigcup_{ I \in U} I$. We have $\abs{A^+} = 2 \abs{U}$. For $i \in I^+$, define  
$$\nabla_i (p^+, U): =\begin{cases}
	1, & \text{if } i\in I \text{ for some } I \in U, \\ 
	-1, & \text{otherwise}.
\end{cases}$$   Then $\nabla_i(A^+) =\nabla_i(p^+,U)  .$ 	 
	 
	 Similarly, since $\abs{A^-}$ and $\abs{A^{c,-}}$ are always  even, from $p_1^- \in \mathcal P^0_{\leq 2}(A^-)$ and $p_2^- \in \mathcal P^0_{\leq 2}(A^{c, -})$  we obtain an element $p^- : = p_1^-  \cup p_2^- \in \mathcal P ^0_{\leq 2}  (I^-)$, and conversely, given $p^- \in \mathcal P^0_{\leq 2}(I^-)$, to recover $A^-$ it suffices to specify a subset $W$ of $p^-$ such that $A^- = \bigcup_{I \in W} I$. We have $\abs{A^-} = 2 \abs{W}$. For $i \in I^-$, define	$$\nabla_i (p^-, W): =\begin{cases}
	 	1, & \text{if } i\in I \text{ for some } I \in W, \\ 
	 	-1, & \text{otherwise}.
	 \end{cases}$$
	 Then we have $\nabla_i (A^-) = \nabla_i (p^- , W)$. 
	 
	 For both parities of $\abs{A^+}$, we have 
	 $$ \omega _0 (A) \epsilon(p_1^+) \epsilon(p_2^+) \epsilon(p_1^-) \epsilon(p_2^-)  = \epsilon(p^+) \epsilon(p^-).$$

	   We now  split 
	   \begin{align*} 
N = \sum_{A \subset [r], \abs{A^-} \text{ even} } \omega_0(A) (-1) ^{r \abs{A} + \floor{ \abs{A} / 2}} \bar c_{R_{A,\gamma'}} (x, \chi)
	   \end{align*}
 as $N = N_{(1)} +N_{(2)}$, where $N_{(1)}$ (resp.~$N_{(2)}$) is the sum of the terms indexed by $A$ such that $\abs{A^+}$ is odd (resp.~even). Similarly, for $i\in [r]$, we split 
	$$ M_i  = \sum_{A \subset [r], \abs{A^-} \text{ even}} \nabla_i(A) \omega_0(A) (-1) ^{r\abs{A} + \floor{\abs{A} /2} } \bar c_{R_{A,\gamma'}} (x, \chi)  $$ as $M_i = M_{i, (1)} + M_{i, (2)}$. We shall prove that $N_{(1)} = N_{(2)} = M_{i,(1)} = M_{i,(2)} =0$. Note that when dealing with $N_{(1)}$ and $M_{i, (1)}$ we may assume that $ k \geq 1$, since otherwise they are obviously zero.
	
	The above discussion shows that 
\begin{align*}
N_{(1)} & =\cost \sum _{ p^+ \in \mathcal P' (I^+) } \sum _{ p^- \in \mathcal P^0_{\leq 2} (I^-)} \epsilon(p^+) \epsilon(p^-) c_{\mathsf B} (p^+ , \mu) c_{\mathsf D} (p^- ,\mu)   
\\
& \cdot  \sum_{U \subset p^+(2), W\subset p^-} (-1) ^{r (2\abs{U} +1 + 2\abs{W}) + \floor{ (2\abs{U} +1 + 2 \abs{W}) /2} } \\
& = \cost \sum_{U \subset [k-1], W\subset [l]} (-1) ^{\abs{U} + \abs{W}}.
\end{align*}	
	This is zero because by $k+l \geq 2$ we have either $l\geq 1$ or $k-1 \geq 1$. 
	Also, 
	\begin{align*}
	 N_{(2)} &  =\cost \sum _{ p^+ \in \mathcal P^0_{\leq 2} (I^+) } \sum _{ p^- \in \mathcal P^0_{\leq 2} (I^-)} \epsilon(p^+) \epsilon(p^-) c_{\mathsf B} (p^+ , \mu) c_{\mathsf D} (p^- ,\mu) \\ 
	 & \cdot \sum_{U \subset p^+, W\subset p^-} (-1) ^{r(2\abs{U} + 2\abs{W}) + \floor{ (2\abs{U}  + 2 \abs{W}) /2} }  \\
	 & =   \cost \sum_{U \subset [k], W\subset [l]} (-1) ^{\abs{U} + \abs{W}},
	\end{align*}
 which is zero because $kl >0$.
	
	Similarly, we have 
	\begin{multline}\label{eq:M_{s,1}}
M_{i, (1)} = \cost \sum _{ p^+ \in \mathcal P' (I^+) } \sum _{ p^- \in \mathcal P^0_{\leq 2} (I^-)} \epsilon(p^+) \epsilon(p^-) c_{\mathsf B} (p^+ , \mu) c_{\mathsf D} (p^- ,\mu) \\ \cdot \sum_{U \subset p^+(2), W\subset p^-} (-1) ^{\abs{U} + \abs{W} } \nabla_i (p^+, p^-, U , W),
	\end{multline}
	and 
	\begin{multline}\label{eq:M_{s,2}}
	M_{i, (2)} = \cost \sum _{ p^+ \in \mathcal P^0_{\leq 2} (I^+) } \sum _{ p^- \in \mathcal P^0_{\leq 2} (I^-)} \epsilon(p^+) \epsilon(p^-) c_{\mathsf B} (p^+ , \mu) c_{\mathsf D} (p^- ,\mu) \\ \cdot  \sum_{U \subset p^+, W\subset p^-} (-1) ^{\abs{U} + \abs{W} } \nabla_i (p^+, p^-, U , W),
	\end{multline} where  
	$$\nabla_i (p^+, p^- , U , W) : = \begin{cases}
		\nabla_i (p^+, U),  & \text{if } i\in I^+, \\ 
		\nabla_i (p^-, W), & \text{if } i \in I^- . 
	\end{cases} $$ (Here the formula for $M_{i,(1)}$ presupposes that $k\geq 1$; otherwise we already know that $M_{i,(1)} =0$.) In the rest of the proof we show that $M_{i,(1)} = M_{i,(2)} = 0$. We introduce two auxiliary definitions. For  $q^+ \in \mathcal P' (I^+), p ^+ \in \mathcal P^0_{\leq 2} (I^+), p^- \in \mathcal P^0_{\leq 2}(I^-)$, let 
	\begin{align*}
	 L_{i,(1)}(q^+, p^-) & : = \sum_{U \subset q^+(2), W\subset p^-} (-1) ^{\abs{U} + \abs{W} } \nabla_i (q^+, p^-, U , W) , 
	\\ 
	 L_{i,(2)}(p^+, p^-) & : =  \sum_{U \subset p^+, W\subset p^-} (-1) ^{\abs{U} + \abs{W} } \nabla_i (p^+, p^-, U , W).
	\end{align*}

We first show that $M_{i,(1)} = 0$. We may assume that $k \geq 1$. If $i \in I^-$, then the function $\mathcal P'(I^+) \times \mathcal P^0_{\leq 2}(I^-) \ni (p^+, p^-) \mapsto L_{i,(1)}(p^+, p^-)$ is constant with respect to the variable $p^+$ . Hence by (\ref{eq:M_{s,1}}) we have 
$$
	M_{i,(1)} = \cost \sum_{p^+ \in \mathcal P'(I^+)} \epsilon(p^+) c_{\mathsf B} (p^+,\mu). $$
This is zero because on $\mathcal P'(I^+) $ we have a non-trivial involution $p^+ \mapsto \overline{ p^+}$ where $\overline{ p^+}$ has the same underlying partition as $p^+$ but has different marked singleton, and this involution satisfies $\epsilon(p^+) = - \epsilon(\overline{ p^+}), c_{\mathsf B}(p^+,\mu) = c_{\mathsf B}(\overline{p^+},\mu)$. 

It remains to treat the case where $i\in I^+$.  Let $p^+ \in \mathcal P'(I^+)$. If one of the singletons in $p^+$ contains $i$, then for arbitrary $ p^- \in \mathcal P^0_{\leq 2}(I^-)$, $L_{i,(1)}(p^+, p^-)$ is equal to a certain number times $$ \sum_{U\subset [k-1], W \subset [l]} (-1)^{\abs{U} +\abs{W}}, $$
which is zero since either $k-1 \geq 1$ or $l \geq 1$. Thus the contribution of such $p^+$ to (\ref{eq:M_{s,1}}) is zero. If one of the cardinality-$2$ members of $p^+$ contains $i$, then so does one of the cardinality-$2$ members of $\overline{ p^+}$. For such a pair $\set{p^+,\overline {p^+}}$, the contribution of $p^+$ to (\ref{eq:M_{s,1}})  is equal to the negative of the contribution of $\overline{ p^+}$, since for any fixed $p^- \in \mathcal P^0_{\leq 2}(I^-)$ we have $L_{i,(1)} (p^+, p^- ) = L_{i,(1)}(\overline {p^+}, p^-)$, and as before we have $\epsilon(p^+) = - \epsilon(\overline {p^+}), c_{\mathsf B}(p^+,\mu) = c_{\mathsf B}(\overline {p^+},\mu)$. 
 We have completed the proof that $M_{i,(1)} =0$. 
 
 We now show that 
$M_{i,(2)} =0$. By (\ref{eq:M_{s,2}}), it suffices to show that $L_{i,(2)}(p^+,p^-) = 0$ for all $p^+\in \mathcal P^0_{\leq 2}(I^+), p^-\in \mathcal P^0_{\leq 2}(I^-)$. To show this, by symmetry we may assume without loss of generality that $s \in I^-$. Enumerate the elements of $p^-$ as $I_1,\cdots, I_l$ such that $i \in I_1$. Using this enumeration we identify the sets $p^-$ and $[l]$. (Here $l \geq 1$.) Then $\nabla_i (p^+, p^- , U, W) = \nabla_1 (W)$ for all $W \subset p^- = [l]$. Hence $$ L_{i, (2)} (p^+, p^-)  = \sum_{U\subset p^+, W \subset [l]} (-1)^{\abs{U} + \abs{W}} \nabla_1(W)  = \sum_{U\subset [k], W \subset [l]} (-1)^{\abs{U} + \abs{W}} \nabla_1(W).$$  If $k>0$, then $L_{i,(2)}(p^+,p^-) =0$ because $ \sum_{ U \subset [k]} (-1) ^{\abs{U}} = 0.$ If $k = 0$, then $l \geq 2$, and we have  $  \sum_{W \subset [l]} (-1) ^{\abs{W}} \nabla_1 (W)  =0 $ as in (\ref{eq:elem id 2}), from which $L_{i,(2)}(p^+,p^-) =0$. The proof of the proposition for even $r \geq 3$ is complete. 
 \end{proof}

\section{A vanishing result, even case}\label{subsec:even case vanishing}
\subsection{}\label{para:preparation for vanishing even}
Assume we are in the even case. We are to state and prove the analogue of Theorem \ref{vanishing odd}. We only point out some new features in the even case, without repeating most of the identical steps.
 
 As in \S \ref{para:preparation for vanishing}, we consider a  Levi subgroup $M^*$ of $G^*$ of the form $\GG_m^r \times \GL_2^t \times \SO(\underline W)$. Without loss of generality, we may and shall assume that $\SO(\underline W)$ is not the split $\SO_2$ over $\QQ$, since in that case we can ``absorb'' it into the factor $\GG_m^r$ (or more precisely, we can replace $\underline W^{\perp}$ by the whole $\underline V$, and extend the hyperbolic basis $\mathbb B_{\underline W^{\perp}}$ to a hyperbolic basis of $\underline V$, after which we obtain the same Levi subgroup $M^*$ but presented in the form $M^* = M^{*,\GL} = \GG_m^{r+1} \times \GL_2^t$). In the current even case we impose the assumption that $M^*$ is cuspidal. This is equivalent to $\SO(\underline W)_{\RR}$ having anisotropic maximal tori (since $\SO(\underline W)$ is not the split $\SO_2$ over $\QQ$), and equivalent to $r$ being even. 

Define $\dot {\mathscr E}_{G^*}(M^*)$ in the same way as in \S \ref{para:preparation for vanishing}. As in \S \ref{para:preparation for vanishing}, for each $\fke_{A,B,\fkp} = \fke_{A,B, d^+,\delta^+,d^-, \delta^-}  =  (M', \lang M', s_{M^*}, \eta_{M^*}) \in \dot {\mathscr E}_{G^*}(M^*)$, we let  $$ (H, \lang H, s, \eta) : = \mathfrak e_{d^+ + 2\abs{A} +4 \abs{B} ,  \delta^+, d^- + 2\abs{A^c} + 4 \abs{B^c}, \delta^- }$$ (viewed as an elliptic endoscopic datum for $G$),  fix an embedding $M' \hookrightarrow H$ as in \S \ref{para:two maps from endoscopic G data}, and define $ST^H_{M'} (f^H)$ as in Definition \ref{defn:pre geometric side}. Then as in (\ref{eq:Tr'M*}), we define 
\begin{multline}\label{eq:Tr'M* even}
	\Tr'_{M^*}   : =  (n^{G^*}_{M^*})^{-1}    \sum _{ \substack {\fke  = (M', \lang M', s_{M^*}, \eta_{M^*}) \\ \in \dot {\mathscr E}_{G^*}(M^*)} } \abs{\Out_{G^*} (\fke)} ^{-1} \tau(G ) \tau(H)^{-1} ST_{M'} ^H (f^H).
\end{multline}  

In the odd case, since $G_{\QQ_p}$ is unramified, it is split, and this already implies that the quadratic space $(V,q)$ is (quasi)-split over $\QQ_p$ (see Proposition \ref{prop: even TFAE}). In the even case, it no longer follows from the unramifiedness of $G_{\QQ_p}$ that $(V,q)$ is quasi-split over $\QQ_p$. However, we shall impose this as a hypothesis\footnote{This is equivalent to asking that $G_{\QQ_p}$ as a pure inner form of $G^*_{\QQ_p}$ is trivial.} in the following theorem. By Proposition \ref{prop: even TFAE}, given the unramifiedness of $G_{\QQ_p}$, in order for $(V,q)$ to be quasi-split over $\QQ_p$ it is sufficient and necessary that the Hasse invariant of $(V,q)$ at $p$ is trivial. 
\begin{thm}\label{vanishing even}Keep the assumptions on $M^*$ in \S \ref{para:preparation for vanishing even}, and assume that $M^*$ does not transfer to $G$. Assume that the quadratic space $(V,q)$ is quasi-split over $\QQ_p$. Then $\Tr'_{M^*} = 0 . $
\end{thm}
\begin{proof}
	The proof is similar to the proof of Theorem \ref{vanishing odd}. We follow most of the notations introduced in the proofs of Theorem \ref{vanishing odd} and Propositions \ref{prop:step 2 in odd vanishing}, \ref{prop:step 3 in odd vanishing}. 
	
Recall that $r$ is even. By hypothesis at least one of the following conditions holds:
$$ rt>0 \qquad\text{or} \qquad r \geq 4 \qquad \text{or} \qquad t\geq 2.$$
As in the proof of Theorem \ref{vanishing odd} we reduce the current proof to showing the vanishing of 
\begin{align*}
	R_i & : =
	\sum_{A,B} \nabla_i(A) \epsilon^{p,\infty}(A,B)   \sum_{\varphi_H \in \Phi_H(\varphi_{\mathbb V^*})} \det(\omega_*(\varphi_H)) \Phi^H_{M'}(\gamma'^{-1}, \Theta_{\varphi_H}), \qquad 1\leq i \leq r, \\ T_j & :=
	\sum_{A,B} \nabla_j(B) \epsilon^{p,\infty}(A,B) \sum_{\varphi_H \in \Phi_H(\varphi_{\mathbb V^*})} \det(\omega_*(\varphi_H)) \Phi^H_{M'}(\gamma'^{-1}, \Theta_{\varphi_H}), \qquad 1 \leq j \leq t ,\\ S & : = 
	\sum_{A,B}  \epsilon^{p,\infty}(A,B)  \sum_{\varphi_H \in \Phi_H(\varphi_{\mathbb V^*})} \det(\omega_*(\varphi_H)) \Phi^H_{M'}(\gamma'^{-1}, \Theta_{\varphi_H}),
\end{align*} for an arbitrary element $\fke = (M',\lang M', s'_{M^*}, \eta_{M^*}) \in \dot {\mathscr E}(M^*)^{c,\ur}$. Here $\dot {\mathscr E}(M^*)^{c,\ur}$ is defined in the beginning of the proof of Theorem \ref{vanishing odd}, and in its definition we do impose that its elements $(M',\lang M', s'_{M^*}, \eta_{M^*})$ should be such that $M'$ is cuspidal (which was automatic in the odd case). In all the above summations, $B$ runs through all subsets of $[t]$, while $A$ only runs through \emph{even-cardinality} subsets of $[r]$, because otherwise the resulting group $H$ will not be cuspidal. On the other hand, indeed all choices of $(A,B)$ with $A$ having even cardinality  will contribute, in the sense that if we write $\fke = \fke_{d^+,\delta^+,d^-,\delta^-} (M^*)$, then the usual formula $\fke _{d^+ + 2 \abs{A} + 4 \abs{B}, \delta^+, d^- + 2 \abs{A^c} + 4 \abs{B^c}, \delta^-}$ as in \S \ref{para:preparation for vanishing even} defines an elliptic endoscopic datum  $(H,\lang H, s, \eta)$ for $G$. In other words, neither of $(d^+ + 2 \abs{A} + 4 \abs{B},\delta^+) $ and $( d^- + 2 \abs{A^c} + 4 \abs{B^c}, \delta^-)$ is equal to $(2,1) \in \ZZ_{ \geq 0} \times (\QQ^{\times}/\QQ^{\times,2})$. To see this, we recall that $M^{\SO}$ is assumed not to be the split $\SO_2$ over $\QQ$, so neither of $(d^{\pm}, \delta^{\pm})$ is $(2,1)$. Then since $\abs{A}$ and $\abs{A^c}$ are even it is clear that neither of $(d^+ + 2 \abs{A} + 4 \abs{B},\delta^+) $ and $( d^- + 2 \abs{A^c} + 4 \abs{B^c}, \delta^-)$ is $(2,1)$.

Since we are in the even case, when choosing $B_0$ as in the proof of Theorem \ref{vanishing odd}, by making a different choice we can only replace an \emph{even number} of the $\tau$'s by their inverses. This means that in condition \textbf{C}, we may not be able to arrange (\ref{eq:cond small}). Nevertheless, it is easy to see that we can always arrange either of the following two conditions:
\begin{itemize}
	\item The original condition \textbf{C}.
	\item The modification of condition \textbf{C}  where (\ref{eq:cond pos}) and (\ref{eq:cond increasing}) are still in force, and (\ref{eq:cond small}) is replaced by the following condition:
	\begin{align*}\abs{
	\epsilon_i(\gamma'^{-1})} <1 \text{ for all } i<r, \text{ and } 1 < \abs{\epsilon_r(\gamma'^{-1})} < \min_{r' < i < r} \abs{\epsilon_i(\gamma'^{-1})}^{-1}. 
	\end{align*}
\end{itemize}

In either case, it is still true that $\epsilon_{R_H}(\gamma'^{-1})$ is independent of $(A,B)$. Moreover, (\ref{eq:sign sigma}) still holds, and it reads $\sign(\sigma) = \omega_0(A)$ since $\abs{A}$ is even. Instead of (\ref{eq:q(H)}) we have $q(H_{\RR}) \equiv 0 \mod 2$ by the cuspidality of $H$. Hence \begin{align*}  
	R_i & = \cost  \sum_{A,B} \nabla_i(A) \epsilon^{p,\infty}(A,B) \omega_0(A) \Phi(\gamma',A), \\  
	T_j & = \cost  \sum_{A,B} \nabla_j(B) \epsilon^{p,\infty}(A,B) \omega_0(A) \Phi(\gamma',A),\\  
	S & = \cost  \sum_{A,B}  \epsilon^{p,\infty}(A,B) \omega_0(A) \Phi(\gamma',A).
\end{align*}

To compute $\epsilon^{p,\infty}(A,B)$, in the proof of Theorem \ref{vanishing odd} we used the fact that the quadratic space $V_{\QQ_p}$ is quasi-split. This is now an assumption in the current theorem. When we showed that the Whittaker normalization between $M'$ and $M^*$ at $\infty$ is inherited from the Whittaker normalization between $H$ and $G^*$ at $\infty$ in the even case in the proof of Proposition \ref{prop:computing tasho}, we used that $m^- \equiv n^- \mod 2$. This is indeed true here since $m^- = n^- + \abs{A^c} + 2 \abs{B^c}$ and we know that  $\abs{A^c} =r - \abs{A}$ is even. Thus by the same argument as in the proof of Theorem \ref{vanishing odd}, $\epsilon^{p,\infty}(A,B)$ is up to a multiplicative constant equal to the ratio of the Whittaker normalization between $H$ and $G$ at $\infty$ to the normalization $\Delta_{j,B_{G,H}}$. This ratio is equal to 
$$ (-1) ^{\floor{ m^- /2}} = (-1)^{\floor{ \frac{ n^- + \abs{A^c} + 2 \abs{B^c} }{2} } } $$
as shown in the proof of Proposition \ref{prop:computing tasho}. Hence 
$$ \epsilon^{p,\infty}(A,B) = \cost (-1) ^{\abs {B} + \abs{A}/2} ,$$
and we have 
\begin{align*} 
	R_i & = \cost  \sum_{A,B} \nabla_i(A) (-1) ^{\abs {B} + \abs{A}/2} \omega_0(A) \Phi(\gamma',A), \\
	T_j & = \cost  \sum_{A,B} \nabla_j(B) (-1) ^{\abs {B} + \abs{A}/2} \omega_0(A) \Phi(\gamma',A),\\
	S & = \cost  \sum_{A,B}  (-1) ^{\abs {B} + \abs{A}/2} \omega_0(A) \Phi(\gamma',A).
\end{align*}
Since $\abs{A} $ is even, we have $\omega_0 (A) = \omega_0 (A^c) $. 
In particular,
\begin{align}\label{A and A^c even}
	\omega_0(A) (-1)^{\abs{A}/2} \omega_0 (A^c) (-1)^{\abs{A^c}/2} = (-1) ^{r/2}.
\end{align}

We now start to show the vanishing of $R_i, T_j, S$. 
As in the proof of Proposition \ref{prop:step 2 in odd vanishing},  
the case where $t \geq 2$ is the easiest. In this case we have $$ \sum_B (-1) ^{\abs{B}} = \sum_B \nabla_j (B) (-1)^{\abs{B}} =0,$$ so $R_i = T_j = S= 0$. 
Now consider the case where $t= 1$ and $r =2$. Then $R_i = S = 0$ because $\sum_B (-1) ^{\abs {B}} =0$. To show $T_j = 0$, we use the fact that (\ref{A and A^c even}) is equal to $-1$ and $\Phi(\gamma',A) = \Phi(\gamma', A^c)$.

Finally we treat the case where $r \geq 4$. The corresponding discussion in \S \ref{subsec:odd case vanishing} for $r\geq 3$ needs almost no change to be carried over here. The only differences are:
\begin{itemize}
	\item All the sets $I^+, I^-, A^+, A^{c,+} , A^-, A^{c,-}$ have to have even cardinality in the present case.
	\item The root system $R_{A, \gamma'}$ in the present case is of type $\mathsf D_{\abs{A^+}} \times \mathsf D_{\abs{A^{c,+}}} \times \mathsf D _{\abs{A^-}} \times \mathsf D_{\abs{A^{c,-}}}.$ 
\item  Herb's formula reads 
\begin{multline}\label{eq:Herb even}
\bar c_{R_{A , \gamma'}} (x, \chi)  = \cost \sum_{p_1 ^+ \in \mathcal P^0_{\leq 2}(A^+) } \sum _{p_1^- \in \mathcal P ^0_{\leq 2} (A^-)} \sum_{ p_2 ^+ \in \mathcal P^0_{\leq 2} (A^{c,+}) } \sum _{p_2^- \in \mathcal P^0_{\leq 2} (A^{c,-})} \\  \epsilon(p_1^+) \epsilon(p_1^-) \epsilon(p_2^+) \epsilon(p_2^-)   c_{\mathsf D} (p_1^+ ,\mu) c_{\mathsf D} (p_2^+, \mu) c_{\mathsf D} (p_1^- , \mu) c_{\mathsf D} (p_2^- ,\mu).
\end{multline}
\end{itemize}
	As in the proof of Proposition \ref{prop:step 2 in odd vanishing},  define 
	\begin{align*}
 M_i & : = \sum_A \nabla_i(A) \omega_0 (A) (-1) ^{\abs{A} /2} \bar c_{R_{A,\gamma'}} (x, \chi),  \\ N & : = \sum_ A \omega_0(A) (-1) ^{\abs{A} /2} \bar c_{R_{A,\gamma'}} (x, \chi) ,
	\end{align*}
 where $A$ runs through subsets of $[r]$ such that  $\abs{A^{\pm}}$ and $\abs{A^{c, \pm}}$ are all even. Then the desired vanishing of $R_i, T_j, S$ reduces to the vanishing of $M_i$ and $N$, which we now show. 
	
	Write $k = \abs { I^+} /2 , l = \abs{I^-} /2$. (They are both integers.)
	For $i\in I^+$, $p^+ \in \mathcal P^0_{\leq 2} (I^+)$, and $U \subset p^+$, define 
	$$\nabla_i ( p^+ , U) : =   \begin{cases}
		1, & \text{if } i\in I \text{ for some } I \in U, \\ 
		-1, & \text{otherwise}.
	\end{cases}. $$
	Similarly, for $i\in I^-$, $p^-\in \mathcal P^0_{\leq 2  }(I^-)$, and $W \subset p^-$,  we define $\nabla_i (p^- , W)$.   

	Herb's formula (\ref{eq:Herb even}) together with a similar argument as in the proof of Proposition \ref{prop:step 3 in odd vanishing} implies that 
	\begin{align*}
N & = \sum _{ p^+ \in \mathcal P^0_{\leq 2}(I^+) } \sum_{ p^- \in \mathcal P^0_{\leq 2}(I^-)} \epsilon(p^+) \epsilon (p^-) c_{\mathsf D} (p^+ , \mu) c_{\mathsf D} (p^- , \mu )  \sum  _{ U \subset p^+ , W \subset p^-} (-1) ^{\abs{U} + \abs{W}} \\
&  = \cost \sum  _{ U \subset [k] , W \subset [l]} (-1) ^{\abs{U} + \abs{W}},
	\end{align*} and for $i\in [r]$
\begin{multline*}
M_i = \sum _{ p^+ \in \mathcal P^0_{\leq 2}(I^+) } \sum_{ p^- \in \mathcal P^0_{\leq 2}(I^-)} \epsilon(p^+) \epsilon (p^-) c_{\mathsf D} (p^+ , \mu) c_{\mathsf D} (p^- , \mu )  \\ \cdot  \sum  _{ U \subset p^+ , W \subset p^-}  (-1) ^{\abs{U} + \abs{W}} \nabla_i(p ^+, p^- ,U ,W),
\end{multline*} 
	where 
	$$ \nabla_i(p ^+, p^- ,U ,W) : = \begin{cases}
		\nabla_i(p ^+,  U ) , & \text{if } i\in I^+, \\
		\nabla_i(p^- , W) , & \text{if } i\in I^-.
	\end{cases}$$
	Since $ lk \neq 0$, we have $N=0$. We now show $M_i =0 $.
	Fix $  p^+ \in \mathcal P^0_{\leq 2}(I^+)  , p^- \in \mathcal P^0_{\leq 2}(I^-) $. It suffices to show that
	$$ L : = \sum_{ U \subset p^+, W \subset p^-}\nabla_i(p ^+, p^- ,U ,W)  (-1) ^{\abs{U} + \abs{W}}  $$ is zero. 
	By symmetry we may assume that $
	i\in I^-$. After fixing an enumeration of the elements of $p^-$ such that the first element contains $i$, we get 
	$$ L = \sum_{ U \subset [k], W \subset [l]} \nabla_1(W) (-1)^{\abs{U} + \abs{W}}.$$ If $k \geq 1$, then $L= 0$ because $\sum_{ U \subset [k]} (-1) ^{\abs { U}} =0 .$ If $k =0$, then $l = r/2 \geq 2$, and $L = 0 $ because 
	$\sum _{ W \subset [l]} \nabla_1 (W) (-1) ^{\abs{W}} = 0$ as in (\ref{eq:elem id 2}). 	This concludes the proof.
\end{proof}
\section{The main identity}
\subsection{}\label{para:prepare for main stabilization}
Keep the notation and setting in \S \ref{subsubsec:setting for Morel's formula} and Theorem \ref{geometric assertion}. Fix a prime $p\notin \Sigma(\mathbf{O} (V), \mathbb V, \lambda, K,f^{\infty})$. In the even case, assume that the quadratic space $(V,q)$ is quasi-split over $\QQ_p$, or equivalently, that its Hasse invariant at $p$ is trivial (cf.~\S \ref{para:preparation for vanishing even}). 
Let $ f^{p,\infty}$ and $ dg^{p,\infty}$ be as in \S \ref{subsubsec:setting for Morel's formula}. Fix a set $\dot {\mathscr E}(G)$\index{$\dot {\mathscr E}(G)$} of representatives of the isomorphism classes in $\mathscr E(G)$ such that each element of $\dot {\mathscr E}(G)$ is of the form $\fke_{\fkp}$ for some $\fkp = (d^+,\delta^+,d^-,\delta^-) \in \mathscr P_V$ with $d^+ \geq 2$ (cf.~\S \ref{para:test function intro}). As in \S \ref{para:test function intro}, assume that $\mathbb V$ is absolutely irreducible. Then for each $\fke_{\fkp} = (H,\lang H, s, \eta) \in \dot{\mathscr E}(G)$, we have a test function $f^H \in C^{\infty}_c(H(\adele))$ fixed in \S\ref{subsec:test functions}.

\begin{cor}\label{main identity}

For $a \in \NN$ large enough, we have 
\begin{multline*}
 \Tr _{M_1} (f^ {p, \infty} dg^{p,\infty} ,K , a ) + \Tr _{M_2} (f^ {p, \infty} dg^{p,\infty} ,K , a ) + \Tr _{M_{12}} (f^ {p, \infty} dg^{p,\infty} ,K , a ) = \\  \sum _{(H, \lang H, s ,\eta) \in \dot{\mathscr E}(G)}  \iota (G, H) [ST^H (f^H) - ST^H_e (f^H)] .
\end{multline*}
 Here $\iota(G,H) : = \tau(G) \tau(H) ^{-1} \abs{\Out (H,\lang H, s,\eta)} ^{-1}$, and $ST^H_e  (f^H): = ST^H_H (f^H)$ as defined in \S \ref{simplified geometric side}.
\end{cor}
\begin{proof}The right hand side of the desired identity is by definition  $$ \sum_{ (H,\lang H,s ,\eta) \in \dot{\mathscr E}(G)} \abs{\Out(H,\lang H, s,\eta)} ^{-1 } \sum_{ L} (n^H_{L})^{-1} \tau(G)\tau(H)^{-1} ST^H_L(f^H) ,$$ where $L$ runs through a set of representatives of the $H(\QQ)$-conjugacy classes of proper Levi subgroups of $H$ (cf.~\S \ref{simplified geometric side}). By an observation of Kottwitz which can be verified directly in our case (see also \cite[Lem.~2.4.2]{morel2010book}), the above is equal to 
	$$ 
 \sum_{M\in \set{M_1,M_2,M_{12}}} \Tr_M'  + \sum_{M^*}\Tr_{M^*}' , 
	$$ where \begin{itemize}
		\item For each $M \in \set{M_1,M_2,M_{12}}$, the term $\Tr_M'$ is as in \S \ref{para:prepare for main}. 
		\item The second sum is over cuspidal Levi subgroups $M^*$ of $G^*$ of the form considered in \S \ref{para:preparation for vanishing} and \S \ref{para:preparation for vanishing even} in such a way that each conjugacy class of cuspidal Levi subgroups of $G^*$ that does not transfer to $G$ is represented exactly once, and that no other conjugacy classes show up.\footnote{Note that in general $G^*$ has Levi subgroups which have direct factors $\GL_j$ with $j \geq 3$. These Levi subgroups are not conjugate to the ones considered in \S \ref{para:preparation for vanishing} and \S \ref{para:preparation for vanishing even}, but none of them are cuspidal. On the other hand, every cuspidal Levi subgroup of $G^*$ is conjugate to the ones considered in \S \ref{para:preparation for vanishing} and \S \ref{para:preparation for vanishing even}.}
		\item For each $M^*$, the term $\Tr_{M^*}'$ is as in (\ref{eq:Tr'M*}) and (\ref{eq:Tr'M* even}).  
	\end{itemize}
The corollary then follows from Theorems \ref{thm:maincomp}, \ref{vanishing odd}, \ref{vanishing even}.  
\end{proof}

\begin{rem} In Corollary \ref{main identity} we \emph{defined} $ST^H_e(f^H)$ to be $ST^H_H(f^H)$, where $ST^H_H$ is defined only when the test function at the archimedean place is stable cuspidal (see \S \ref{simplified geometric side}). On the other hand, $ST^H_e$ has a more general definition, namely it is the elliptic part of the stable trace formula for $H$ as in \cite{kottwitzelliptic}. Of course it is expected (and proved in Kottwitz's unpublished notes) that these two definitions agree when the test function at the archimedean place is stable cuspidal. For our particular $f^H_{\infty}$, this compatibility is essentially proved in \cite[\S 7]{kottwitzannarbor}. In fact, if we substitute the archimedean stable orbital integrals in the general definition of $ST^H_e (f^H)$ by the formula \cite[(7.4)]{kottwitzannarbor}, then we obtain precisely $ST^H_H (f^H)$. 
 \end{rem}
The following is a special case of the main result of \cite{KSZ}.
\begin{thm}\label{from KSZ}Keep the setting of \S \ref{para:prepare for main stabilization}. For $a\in \NN$ large enough, we have
$$	\Tr(\Frob_p^a \times f^{\infty} dg ^{\infty} \mid  \coh_c^* ( \Sh _K, \mathbb V))= \sum _{(H, \lang H, s ,\eta) \in \dot{\mathscr E}(G)}  \iota (G, H) ST^H_e (f^H).  $$ \qed 
\end{thm}

\begin{cor}\label{Main Main result} For $a \in \NN$ large enough, we have  
	\begin{align}\label{eq:main}
\Tr(\Frob_p^a \times f^{ \infty} dg^{\infty} \mid  \icoh^* (\overline{ \Sh _K},\mathbb V) )   = \sum _{(H, \lang H, s ,\eta) \in \dot {\mathscr E}(G)}  \iota (G, H) ST^H (f^H) .
	\end{align}

\end{cor}
\begin{proof}
	This follows from Theorem \ref{geometric assertion}, Corollary \ref{main identity}, and Theorem \ref{from KSZ}.
\end{proof}

\begin{rem}\label{rem:in E}
	The right hand side of (\ref{eq:main}) is \textit{a priori} a number in $\CC$. However, as we have seen in Theorem \ref{geometric assertion}, the left hand side is in fact a number in $\mathbb E$, the number field over which $\mathbb V$ is defined. 
\end{rem}

\chapter[Spectral expansion and Hasse--Weil zeta functions]{Application: spectral expansion and Hasse--Weil zeta functions}\label{sec:app}
\section{Introductory remarks}\label{subsec:intro}
\subsection{}
In \cite[Part II]{kottwitzannarbor}, Kottwitz explained how the formula in Corollary \ref{Main Main result} would imply a description of $\sum_i (-1)^i \icoh^i (\overline {\Sh_K} , \mathbb V)$ in the Grothendieck group of $\mathcal H(G(\adele_f ) \sslash K)_{\QQ} \times \Gamma_{\QQ}$-modules over $\overline {\QQ}_{\ell}$. More precisely, the Grothendieck group is taken with respect to the category of $\mathcal H(G(\adele_f ) \sslash K)_{\QQ}
\otimes_{\QQ} \overline \QQ_{\ell}$-modules which are finite-dimensional over $\overline \QQ_{\ell}$ and are equipped with a continuous (with respect to the $\ell$-adic topology) $\Gamma_{\QQ}$-action that commutes with $\mathcal H(G(\adele_f ) \sslash K)_{\QQ}
\otimes_{\QQ} \overline \QQ_{\ell}$. This description is in terms of the conjectural parametrization of automorphic representations by Arthur parameters. The main hypotheses assumed by Kottwitz are the following (see \cite[\S 8]{kottwitzannarbor}):

\begin{enumerate}
	\item  Arthur's conjectural parametrization and multiplicity formula for automorphic representations.
	\item The closely related conjectural spectral expansion of the stable trace formula in terms of Arthur parameters.
\end{enumerate}

Recent developments have seen the proof of variations of these hypotheses in specific instances. For the groups that are relevant to this paper, Arthur \cite{arthurbook} has established (1) and (2) for quasi-split special orthogonal groups over number fields, and Ta\"ibi \cite{taibi} has generalized (1) to some inner forms of these groups (and under a regular algebraic assumption). Among the inputs to Ta\"ibi's work are the theory of rigid inner forms established by Kaletha \cite{kalrigid, kalethaglobal} and results of Arancibia--Moeglin--Renard \cite{AMR} on archimedean Arthur packets. (For the special orthogonal groups of interest to us, only the special case of Kaletha's theory, namely that of pure inner forms, is needed.) We mention that Arthur's work \cite{arthurbook} depends on the stabilization of the twisted trace formula as a hypothesis, and the latter has been established by Moeglin--Waldspurger \cite{MWbook}.\footnote{However, see footnote \ref{foot:Arthur} on p.~\pageref{foot:Arthur}.} It is thus possible to combine Corollary \ref{Main Main result} with the results from \cite{arthurbook} and \cite{taibi} to obtain an unconditional description of $\icoh^*(\overline{\Sh_K}, \mathbb V) $ in certain special cases.  
In the following we carry this out, for the special cases described in Lemma \ref{lem:existence of V}. 

In the sequel, we shall assume the following hypothesis.

\begin{hypo}\label{hypo}
Let $H$ be a quasi-split reductive group over $\QQ$. For test functions $f$ on $H(\adele)$  which are stable cuspidal at infinity, we have $ST^H(f) = S^H(f)$. Here $ST^H(f)$ denotes Kottwitz's simplified geometric side of the stable trace formula (see \S \ref{simplified geometric side}), and $S^H(f)$\index{$S^H$} denotes Arthur's stable trace formula\index[n]{Arthur's stable trace formula} \cite{arthursta1,arthursta2,arthursta3}.
\end{hypo}

This hypothesis essentially follows from Kottwitz's stabilization of the trace formula with stable cuspidal test functions at infinity in his unpublished notes. Recently an alternative proof has been given by Z.~Peng \cite{peng}. Let us make some comments on the former. Firstly we state and prove two lemmas that are well known and independent of Hypothesis \ref{hypo}. 

\begin{lem}\label{lem:transfer of stable cuspidal}
Let $H$ be a semi-simple (for simplicity) reductive group over $\RR$. Assume that $H$ is cuspidal (Definition \ref{defn:cuspidal}). Let $f : H(\RR) \to \CC$ be a stable cuspidal function (see \cite[\S 4]{arthurlef}, \cite[5.4]{morel2010book}).  The following statements hold.
\begin{enumerate}
	\item The function $f$ is equal to a finite linear combination $\sum_{ \varphi} c_{\varphi} f_{\varphi}, c_{\varphi} \in \CC$, where $\varphi$ runs through the discrete Langlands parameters for $H$ and each $f_{\varphi}$ is a stable pseudo-coefficient for the L-packet of $\varphi$ as in (\ref{eq:defn of stable pseudo coeff}). 
	\item Let $(H', \mathcal H', s, \eta: \mathcal H' \to \lang H)$ be an elliptic endoscopic datum for $H$. For simplicity assume $\mathcal H' = \lang H'$. Then a Langlands--Shelstad transfer of $f_{\varphi}$ as in (1) to $H' $ can be taken to be a stable cuspidal function on $H'(\RR)$ that is supported on those discrete Langlands parameters $\varphi'$ for $H'$ such that $\eta \circ \varphi'$ is equivalent to $\varphi$.
\end{enumerate}
\end{lem}
\begin{proof}
	(1) is a formal consequence of the definitions. In fact, by the definition of being stable cuspidal, we know there exists a function $f'$ of the desired form $\sum_{i=1}^k c_{i} f_{\varphi_i} , c_i \in \CC^{\times}$ such that $\delta: = f-f'$ has zero trace on all tempered representations of $H(\RR)$. By definition we have $f_{\varphi_1} = \sum _{ \pi} f_{\pi}$, where $\pi$ runs through the L-packet of $\varphi_1$ and each $f_{\pi}$ is a pseudo-coefficient of $\pi$. Then for one such $\pi$ we may replace $f_{\pi}$ by $  f_{\pi} + \delta/c_1 $, which is still a pseudo-coefficient of $\pi$. After making this replacement $f$ is precisely equal to $\sum_{i=1}^k c_{i} f_{\varphi_i}$, with the new definition of $f_{\varphi_1}$. 
	
	(2) follows from the fact, due to Shelstad (see for instance \cite{she2,she3}), that the spectral transfer factor between a tempered Langlands parameter $\varphi'$ for $H'$ and a tempered representation $\pi$ for $H$ vanishes unless $\pi$ lies in the L-packet of $\eta \circ \varphi'$. For a summary of Shelstad's theory of spectral transfer factors see \cite[p.~621]{kalrigid}.
\end{proof}
\begin{lem}\label{lem:discrete part}
	Let $H$ be a semi-simple (for simplicity) reductive group over $\QQ$. Assume that $H$ is cuspidal (Definition \ref{defn:cuspidal}). Let $f_{\infty} \in C^{\infty}_c(H(\RR))$ be a stable cuspidal function, and let $f^{\infty} \in C^{\infty} _c (H(\adele_f))$. Let $I_H$ denote the invariant trace formula\index[n]{invariant trace formula} for $H$ and let $I_{H,\disc} = \sum_{t\geq 0} I_{\disc, t} $ denote its discrete part; see \cite{arthurinvII} and \cite[\S 3]{arthurlef}. Then $$ I_H(f_{\infty} f^{\infty} ) = I_{H, \disc } (f_{\infty} f^{\infty}), $$ and they are also equal to 
	$$ \Tr ( f_{\infty} f^{\infty} \mid  L^2_{\disc} (H(\QQ) \backslash H(\adele))). $$ 
\end{lem}
\begin{proof}
	By Lemma \ref{lem:transfer of stable cuspidal} we may assume that $f_{\infty}  = f_{\varphi}$ for a discrete Langlands parameter ${\varphi}$. Then the lemma follows from \cite[\S 3]{arthurlef} (where our $f_{\varphi}$ is equal to the function denoted by $f_{\mu}$ up to a multiplicative constant).  
\end{proof}
\subsection{} We now explain how Kottwitz's stabilization in the aforementioned unpublished notes is related to Hypothesis \ref{hypo}.  
For $f_{\infty} = f_{\varphi}$ as in the above proof, Arthur \cite{arthurlef} shows that the value $I_H(f_{\varphi} f^{\infty})$ has the interpretation as the $L^2$ Lefschetz number\index[n]{$L^2$ Lefschetz number} of a Hecke operator on a locally symmetric space, with coefficients in a sheaf determined by $\varphi$. This Lefschetz number is evaluated by Arthur \cite{arthurlef} and independently by Goresky--Kottwitz--MacPherson \cite{GKM}. Hence the general $I_H( f_{\infty} f^{\infty})$ with stable cuspidal $f_{\infty}$ as in the above lemma is just a linear combination of these Lefschetz number formulas. Based on this, Kottwitz proves in his unpublished notes a stabilization
\begin{align}\label{eq:Kottwitz's stabilization}
I_H(f_{\infty} f^{\infty}) = \sum_{H' \in \mathscr E(H)} \iota(H, H') ST^{H'}(f^{H'}), 
\end{align} 
where the terms are explained below: 
\begin{itemize}
	\item The left hand side is as in Lemma \ref{lem:discrete part}.
	\item In the sum $H'$ runs through the elliptic endoscopic data for $H$ up to isomorphism.
	\item For each $H'\in \mathscr E(H)$, the function $f^{H'} $ is of the form $f^{H'}_{\infty} f^{H',\infty}$, where $f^{H'}_{\infty}$ (resp.~$ f^{H',\infty}$) is a Langlands--Shelstad transfer of $f_{\infty}$ (resp.~of $ f^{\infty}$). Here by Lemma \ref{lem:transfer of stable cuspidal} we may and do take $f^{H'}_{\infty}$ to be stable cuspidal.
	\item For each $H'\in \mathscr E(H)$, the term $ST^{H'} (f^{H'})$ is the simplified geometric side of the stable trace formula, as  recalled in \S \ref{simplified geometric side}. 
	\item For each $H'\in \mathscr E(H)$, the term $\iota(H,H') \in \QQ$ is the usual constant in the stabilization of trace formulas; cf.~Corollary \ref{main identity}. 
\end{itemize}

On the other hand, according to Arthur's stabilization \cite{arthursta1} \cite{arthursta2} \cite{arthursta3}, we have 
\begin{align}\label{eq:Arthur's stabilization}
I_H(f_{\infty} f^{\infty}) & = \sum_{H' \in \mathscr E(H)} \iota(H, H') S^{H'}(f^{H'}),\\ \label{eq:Arthur's discrete stablilization}
I_{H,\disc}(f_{\infty} f^{\infty}) &  = \sum_{H' \in \mathscr E(H)} \iota(H, H') S^{H'}_{\disc}(f^{H'})
\end{align} 
where $S^{H'}$ (resp.~$S^{H'} _{\disc}$\index{$S^H_{\disc}$}) is Arthur's stable trace formula for $H'$ (resp.~the discrete part\index[n]{discrete part of the stable trace formula} thereof\footnote{More precisely, each of $I_{H,\disc}$ and $S^{H'} _{\disc}$ is formally a sum over a parameter $t \in \RR_{\geq 0}$ of respective contributions $I_{H,\disc,t}$ and $S^{H'} _{\disc,t},$ and (\ref{eq:Arthur's discrete stablilization}) could be stated parameter-wise for each $t$.}; see \cite[\S\S 3.1, 3.2]{arthurbook}), and the rest of the notations are the same as in (\ref{eq:Kottwitz's stabilization}). 
Comparing (\ref{eq:Kottwitz's stabilization}) and (\ref{eq:Arthur's stabilization}) for $H$ quasi-split (so that $H \in \mathscr E(H)$) and by induction on the dimension of the group in Hypothesis \ref{hypo}, we conclude that $$  ST ^H (f_{\infty} f^{\infty}) = S^H (f_{\infty } f^{\infty}).$$ Thus  Hypothesis \ref{hypo} is proved. Moreover, comparing Lemma \ref{lem:discrete part} and (\ref{eq:Arthur's stabilization}), (\ref{eq:Arthur's discrete stablilization}) for $H$ quasi-split and by induction, we also draw the following conclusion independently of Hypothesis \ref{hypo}: 
\begin{prop}\label{prop:S^H}
	Keep the setting of Lemma \ref{lem:discrete part} and assume in addition that $H$ is quasi-split. Then
$$
S^H (f_{\infty} f^{\infty})  = S^H_{\disc} (f_{\infty} f^{\infty}). $$ \qed
\end{prop}
\
\begin{cor}\label{cor:simp}
We may replace each $ST^H$ in Corollary \ref{Main Main result} by $S^H_{\disc}$. 
\end{cor}
\begin{proof}
	This follows from Hypothesis \ref{hypo} and Proposition \ref{prop:S^H}. 
\end{proof}

\section{Review of Arthur's results} \label{subsec:review of Arthur} We loosely follow \cite[\S 2]{taibi} to recall some of the main constructions and results in \cite{arthurbook}. We fix a quasi-split quadratic space $(\underline{V}, \underline{q})$ over $\QQ$, of dimension $d$ and discriminant $\delta \in \QQ^{\times} /\QQ^{\times ,2}$. (See \S \ref{subsec:generalities on quad sp} for what we mean by a quasi-split quadratic space.) Let $G^* : = \SO(\underline{ V} , \underline q)$. 
As usual we explicitly fix the $L$-group $ \lang G^*$, and fix explicit representatives $(H, \mathcal H = \lang H , s, \eta: \lang H \to \lang G^*)$ for the isomorphism classes of elliptic endoscopic data for $G^*$, as discussed in \S \ref{sec:endoscopic}. 

\section*{Self-dual cuspidal automorphic representations of \texorpdfstring{$\GL_N$}{GLN}}
\subsection{}\label{subsubsec:self-dual}
Let $N \in \ZZ_{\geq 1}$. Let $\pi$ be a self-dual cuspidal automorphic representation of $\GL_N$ over $\QQ$. Arthur \cite[Thm.~1.4.1]{arthurbook} associates to $\pi$ a quasi-split orthogonal or symplectic group $G_{\pi}$\index{$G_{\pi}$} over $\QQ$, such that $\widehat{ G_{\pi}}$ is isomorphic to $\Sp_{N}(\CC)$ or $\SO_N(\CC) $. We view $\Sp_{N}(\CC)$ and $\SO_N(\CC) $ as standard subgroups of $\GL_N(\CC)$ as in  \S \ref{subsec:matrix groups}. There is a standard representation 
\index{$\mathrm{Std}_{\pi}$}$$\stan_{\pi}: \lang G_{\pi} \To \lang \GL_N = \GL_N(\CC) $$ extending the inclusion $\widehat{ G_{\pi}} \hookrightarrow \GL_N(\CC)$ determined as follows. The central character $\omega_{\pi}$ of $\pi$ determines a character $\eta_{\pi}: \Gamma_{\QQ} \to \set{\pm 1}$. Let $E/\QQ$ be the degree one or two extension given by $\eta_{\pi}$. When $E = \QQ$, the group $G_{\pi}$ is split. In this case we may take $\lang G_{\pi} = \widehat{ G_{\pi}}$ and there is nothing to do. When $E \neq \QQ$, the group $G_{\pi}$ is either symplectic, or the non-split quasi-split even special orthogonal group over $\QQ$ which is split over $E$. Thus when $E\neq \QQ$ we have  $\widehat{ G_\pi} = \SO_N(\CC)$, and we may take $\lang G_{\pi} $ to be $\widehat{G_{\pi}} \rtimes \Gal(E/\QQ)$ (which is a direct product when $G_{\pi}$ is symplectic). When $G_{\pi}$ is symplectic, we define $\stan_{\pi}$ to send the non-trivial element of  $\Gal(E/\QQ)$ to $-1 \in \GL_N(\CC)$. When $G_{\pi}$ is the non-split quasi-split even special orthogonal group, we define $\stan_{\pi}$ to send the non-trivial element of  $\Gal(E/\QQ)$ to the permutation matrix switching $\hat e_{N/2}$ and $\hat e_{1+N/2}$ in the notation of \S \ref{subsec:matrix groups}. Thus in the last case $\stan_{\pi}$ maps $\lang G_{\pi}$ isomorphically onto the subgroup $\mathrm{O}_N(\CC)$ of $\GL_N(\CC)$ as in \S \ref{subsec:matrix groups}. 

Let $v$ be a place of $\QQ$. Under the local Langlands correspondence for $\GL_N$,\index[n]{local Langlands correspondence for $\GL_N$} established by Langlands \cite{langlandsLLC} in the archimedean case and by Harris--Taylor \cite{harristaylor}, Henniart \cite{henniartLLC}, and Scholze \cite{scholzeLLC} in the non-archimedean case, the local component $\pi_v $ of $\pi$ corresponds to a Langlands parameter $\varphi_{\pi_v}: \WD_v  \to \GL_N(\CC)$. Here $\WD_v$\index{$\mathrm{WD}_v$} denotes the Weil--Deligne group\index[n]{Weil--Deligne group} of $\QQ_v$ (denoted by $L_{\QQ_v}$ in \cite{arthurbook}), which is by definition the Weil group when $\QQ_v = \RR$, and the direct product of the Weil group with $\SU_2(\RR)$ when $\QQ_v$ is non-archimedean. Arthur shows \cite[Thm.~1.4.1, Thm.~1.4.2]{arthurbook} that $\varphi_{\pi_v}$ is conjugate to $\stan_{\pi } \circ  \varphi_v $ for some Langlands parameter 
\begin{align}\label{eq:existence of localization}
 \varphi_v : \WD_v \To \lang G _{\pi}.
\end{align} The $\Aut(\lang G_{\pi})$-orbit of $\varphi_v$ is uniquely determined by $\varphi_{\pi_v}$. (See \cite[\S 2.1]{taibi} for $\Aut(\lang G_{\pi})$, also cf.~Remark \ref{rem:taut} below.)
Define \index{$\mathrm{sgn}(\pi)$}
$$\sign (\pi) : = \begin{cases}
1, & \text{if }\widehat{ G_{\pi} } \mbox{ is orthogonal},\\
-1, & \text{if } \widehat{ G_{\pi} } \mbox{ is symplectic}.
\end{cases}$$
\section*{Substitutes for global Arthur parameters}
\subsection{}\label{subsubsec:substitutes}
Similar to the definition of $\stan_{\pi }$ above, we have a standard representation \index{$\mathrm{Std}_{G^*}$}
\begin{align}\label{eq:standard rep}
\stan_{G^*}: \lang G^* \To \GL_N(\CC)
\end{align}
where $N = d-1$ (resp.~$N=d$) when $d$ is odd (resp.~even).

 Let $\Psi(N)$\index{$\Psi(N)$} denote the set of formal unordered sums $$ \psi =   \underset{k \in K_{\psi}} {\boxplus}  \pi_k [d_k], $$
where $K_{\psi}$ is a finite indexing set, each $\pi_k$ is a unitary cuspidal automorphic representation of $\GL_{N_k}$ over $\QQ$ for some $N_k \in \ZZ_{\geq 1}$, and each $d_k  $ is a positive integer, satisfying $\sum_k N_k d_k = N$. Let $\widetilde \Psi (N)$\index{$\widetilde \Psi (N)$} denote the set of $$\psi =  \underset{k \in K_{\psi} } {\boxplus}  \pi_k [d_k] \in \Psi (N)$$ satisfying the condition that there is an involution $k \mapsto k^{\vee}$ on the indexing set $K_{\psi}$ such that $  ( \pi_k) ^{\vee} \cong \pi_ {k^{\vee} } $ and $ d_k = d_{k^{\vee}}$ for all $ k \in K_{\psi}$. 
Let $\widetilde{\Psi}_{\Ell} (N) $\index{$\widetilde{\Psi}_{\Ell} (N) $} be the subset of $\widetilde{\Psi} (N) $ defined by the conditions that  each $\pi_k$ should be self-dual and that the pairs $(\pi_k , d_k)$ should be distinct (i.e., for $k\neq k'$, either $\pi_k$ is not isomorphic to $\pi_{k'}$ or $d_k \neq d_{k'}$).

For any $\psi  \in \widetilde{\Psi}  (N)$, we write $$\psi = \underset{ i \in I }{\boxplus} \pi_i[d_i] \underset{j\in J}{\boxplus} (\pi_j [d_j] \boxplus \pi_j^{\vee} [d_j] ),$$ where $\pi_i$ is self-dual for each $i \in I$ and $\pi_j$ is not self-dual for each $j\in J$. Let $\mathcal L _{\psi} $\index{$\mathcal L_{\psi}$} be the fiber product over $\Gamma_{\QQ}$ of $\lang G_{\pi_i}$ and $\GL_{N _j} (\CC)$ for all $i\in I, j \in J$. For $j \in J$, we define  \begin{align*}
\stan_{N_j} \oplus \stan_{N_j}^{\vee} : \GL_{N_j } (\CC ) &\To \GL_{2N_j} (\CC)\\  g &\longmapsto g \oplus (g^{\mathsf T})^{-1} .
\end{align*}
 Define 
\begin{align*} 
\tilde \psi:  = (\bigoplus_{i\in I} \stan_{ \pi_i } \otimes \nu_{d_i}) \oplus \bigoplus _{ j \in J} (\stan_{N_j} \oplus \stan_{N_j}^{\vee} ) \otimes \nu_{d _j}  : \mathcal L_{\psi} \times \SL_2(\CC) \to \GL_N(\CC),
\end{align*}
 where $\nu_{k}$ denotes the irreducible representation of $\SL_2(\CC)$ of dimension $k$ for any positive integer $k$. Let $\widetilde{\Psi} (G^*)$ \index{$\widetilde{\Psi} (G^*)$}  be the set of $\psi \in \widetilde{\Psi} (N)$ for which there exists $$\dot \psi: \mathcal L_{\psi} \times \SL_2(\CC) \To \lang G^*$$ such that 
 $\stan _ {G^*} \circ \dot \psi$ is conjugate under $\GL_N(\CC)$ to $\tilde \psi$.  Let $\Psi (G^*)$ be the set of pairs $(\psi, \dot{\psi})$ where $\psi \in \widetilde{\Psi} (G^*)$ and $\dot \psi$ is a choice as above.  For $\psi \in \widetilde \Psi(G^*)$, we define \index{$m_{\psi}$}
 \begin{align}
 	\label{eq:defn of m psi}
 	m_{\psi}: = \mbox{the number of  $\dot \psi$ modulo $\widehat{ G^*}$-conjugation such that } (\psi,\dot \psi) \in \Psi (G^*).
 \end{align} 
 
 We define\footnote{In \cite{taibi}, our $\widetilde{\Psi}_2(G^*)$ and $\Psi_2(G^*)$ are denoted by $\widetilde{\Psi}_{\disc}(G^*)$ and $\Psi_{\disc}(G^*)$ respectively. However, in \cite{arthurbook}, the usage of the subscript ``$\disc$'' is different; see p.~172. We follow \cite{arthurbook} to use the subscript ``$2$'' here.} \index{$\widetilde{\Psi}_{2} (G^*)$} $$\widetilde{\Psi}_{2} (G^*): =  \widetilde{\Psi} _{\Ell} (N) \cap \widetilde 
\Psi (G^*), $$
and define $\Psi_2(G^*)$ \index{$\Psi_2(G^*)$} to be the preimage of $\widetilde{\Psi} _2 (G^*)$ in $\Psi (G^*)$ along the forgetful map $\Psi (G^*) \to \widetilde{\Psi} (G^*)$. Recall that $d$ and $\delta$ denote the dimension and discriminant of the quadratic space $\underline V$. For $\psi = \boxplus _k \pi_k [d_k] \in \widetilde{\Psi} _{\Ell} (N)$, the following condition is equivalent to the condition that $\psi \in \widetilde{\Psi} _{2} (G^*)$:
\begin{itemize}
	\item The character $\Gamma_{\QQ} \to \set{\pm 1}$ given by $\prod_k  \eta_{\pi_k} ^{d_k} $ is trivial if $G^*$ is split, and corresponds to the quadratic extension $\QQ(\sqrt{\delta}) /\QQ$ if $G^*$ is non-split, i.e., if $d$ is even and $\delta\notin \QQ^{\times ,2}$. Moreover \begin{align}\label{eq:sign pi}
 \sign (\pi_k) (-1) ^{d_k - 1} = (-1)^d
	\end{align}  for all $k$.
\end{itemize}

For $\psi \in \widetilde \Psi_2(G^*)$, we know that $m_{\psi} \leq 2$, and $m_{\psi} =2$ if and only if $d$ and all $N_kd_k$ are even; see \cite[p.~47]{arthurbook}. In the latter case the two  $\widehat{ G^*}$-conjugacy classes of $\dot \psi$ are interchanged by the non-trivial outer automorphism of $\widehat{G^*} = \SO_d(\CC)$. 
 
For $(\psi, \dot \psi )\in \Psi (G^*)$, we define \index{$S_{\dot \psi}$} \index{$\mathcal S_{\dot \psi}$}
\begin{align*}
 S_{\dot \psi} & : = \mathrm{Cent} (\dot \psi, \widehat{ G^*} ), \\ \mathcal S _{\dot  \psi } & :=  S_{\dot \psi }  /  S_{\dot \psi} ^0 Z(\widehat{ G^*})^{\Gamma_{\QQ}}.
\end{align*}
In fact $\mathcal S _{\dot \psi}$ is isomorphic to a finite power of $\ZZ/2\ZZ$. Moreover, $S_{\dot \psi}$ is finite if and only if $(\psi, \dot \psi) \in \Psi _2(G^*)$, in which case $S_{\dot \psi}$ is a finite power of $\ZZ/2\ZZ$. These statements follow easily from the description \cite[(1.4.8)]{arthurbook} of $S_{\dot \psi}$. By abuse of notation we shall write $S_{\psi}$ and $\mathcal S_{\psi}$ for $S_{\dot \psi}$ and $\mathcal S_{\dot \psi}$ respectively.\index{$S_{\psi}$, global case}\index{$\mathcal S_{\psi}$, global case}\footnote{Here we follow the notation of \cite{arthurbook}, which differs slightly from that in \cite{kottwitzcuspidal} and \cite{taibi}. In the latter two papers the notation $S_{ \psi}$ refers to a larger group, which in the present case is equal to $S_{\psi} Z(\widehat{ G^*})$ in our notation. More specifically, in our notation we have $S_{\psi} \supset Z(\widehat {G^*}) = Z(\widehat {G^*})^{\Gamma_{\QQ}}$ unless $G^*$ is a non-split $\SO_2$, in which case $S_{\psi} =Z(\widehat {G^*})^{\Gamma_{\QQ}}$ and $Z(\widehat{G^*}) = \widehat{G^*}$. In particular, we see that the formula $S_{\psi} / S_{\psi} ^0 Z(\widehat{ G^*}) ^{\Gamma_{\QQ}}$ defines the same group $\mathcal S_{\psi}$ with both interpretations of the notation $S_{\psi}$.} In the case where $(\psi,\dot \psi) \in  \Psi_2(G^*)$
(which is the only case relevant to us in practice), our abuse of notation is essentially harmless for the following reason. Since $S_{\dot \psi}$ is abelian, it depends on $\dot \psi$ only via its $\widehat{G^*}$-conjugacy class, up to canonical isomorphism. Moreover, in the even case with $m_{\psi} =2$, it follows from the description \cite[(1.4.9)]{arthurbook} of $S_{\dot \psi}$ that there is an element of $\mathrm{O}_N (\CC)-\SO_N(\CC) = \mathrm{O}_N (\CC) - \widehat{G^*}$ centralizing  $S_{\dot \psi}$. Hence in both the odd and even cases, for $(\psi, \dot \psi) \in \Psi_2(G^*)$, the group $S_{\dot \psi}$ depends only on $\psi$ up to canonical isomorphism. The similar remark applies to $\mathcal S_{\dot \psi}$. Moreover, it also follows from the above discussion that the $\widehat {G^*}$-conjugacy class of the subgroup $S_{\dot \psi} \subset \widehat {G^*}$ depends only on $\psi$.

For $\psi \in \widetilde\Psi (G^*)$, we define $s_{\psi} \in S_{\psi}$ by \index{$s_{\psi}$}
\begin{align}\label{eq:defn of s psi}
s_{\psi} : = \dot \psi (-1), \text{ where } -1 \in \SL_2(\CC).
\end{align} (Here we implicitly fix a lift $(\psi,\dot \psi) \in \Psi(G^*)$.) We will also need the canonical character\index{$\epsilon_{\psi}$}
\begin{align}\label{eq:symp root number}
\epsilon_{\psi} : \mathcal S_{\psi} \To \set{\pm 1}
\end{align}
defined on p.~48 of \cite{arthurbook} using symplectic root numbers. We do not recall its definition here.

Let $(H, \lang H, s, \eta: \lang H \to \lang G^*) $ be an elliptic endoscopic datum for $G^*$, presented in the explicit form as in \S \ref{subsec:elliptic end data}. Recall that $H$ is a direct product $H^+ \times H^-$ of two quasi-split special orthogonal groups over $\QQ$. The above discussion for $G^*$ applies equally to $H^+ $ and $H^-$. We define \index{$\widetilde{\Psi}(H)$} \index{$\Psi(H)$}
\begin{align*}
 \widetilde{\Psi} (H) & : = \widetilde{\Psi}  (H^+) \times \widetilde{\Psi} (H^-), \\  {\Psi} (H)& : = {\Psi}(H^+) \times {\Psi} (H^-).
\end{align*}
Similarly we define $\widetilde{\Psi} _2 (H)$ and $\Psi_2(H)$.\index{$\widetilde{\Psi}_2 (H)$}\index{$\Psi_2(H)$} For $\psi' = (\psi^+, \psi^-) \in \widetilde{\Psi} (H)$, we define \index{$S_{\psi'}$} \index{$\mathcal S_{\psi'}$} \index{$s_{\psi'}$} \index{$m_{\psi'}$} \index{$\epsilon_{\psi'}$} \begin{align*}
S_{\psi'} & :=  S_{\psi^+} \times S_{\psi^-} , \\ \mathcal S_{\psi'} & := \mathcal  S_{\psi^+} \times \mathcal S_{\psi^-} , \\  s_{\psi'}  & : = (s_{\psi ^+} , s_{\psi ^-}  ) \in S_{\psi '} ,\\ m_{\psi'} &: = m_{\psi^+ } m_{\psi ^-} ,\\ \epsilon_{\psi'}& : = \epsilon_{\psi^+} \otimes \epsilon_{\psi^-} : \mathcal S_{\psi'} \To \set{\pm 1}.
\end{align*}

We have a natural map  
\begin{align*}
\widetilde \Psi (H) & \To \widetilde{\Psi} (G^*)\\  (\psi ^+, \psi ^-) & \longmapsto \psi ^+ \boxplus \psi ^-,
\end{align*} which we shall denote by  
 $$\psi ' \longmapsto \eta\circ \psi'. $$
 \ignore{ 
This map upgrades to a map $\eta: \Psi (H) \to \Psi (G^*)$\index{$\eta: \Psi (H) \to \Psi (G^*)$} which we now describe. Let $( \psi^+, \dot \psi^+,  \psi ^- , \dot \psi^-) \in \Psi(H)$ and let $\psi = \psi ^+ \boxplus \psi ^ - \in \widetilde{  \Psi} (G^*)$. The group $\mathcal L_{\psi}$ is isomorphic to the fiber product of $\mathcal L_{\psi^+}$ and $\mathcal L_{\psi ^-}$ over $\Gamma_{\QQ}$, and the maps $\dot \psi^+, \dot \psi^- $ naturally induce a map $\mathcal L_{\psi} \times \SL_2(\CC) \to \lang H$. Let $\dot \psi$ be the composition of the last map with $\eta : \lang H \to \lang G^*$. Then we obtain an element $(\psi, \dot \psi ) \in \Psi (G^*)$. We define $\eta ( \psi^+, \dot \psi^+,  \psi ^- , \dot \psi^-) $ to be $(\psi, \dot \psi )$.
}
\section*{Local Arthur packets}
\subsection{}\label{subsubsec:local packets}
Let $v$ be a place of $\QQ$. We abbreviate $G^*_v:  = G^* _{\QQ_v}$. Let $\Psi^+(G^*_v)$\index{$\Psi^+(G^*_v)$} be the set of all \emph{Arthur--Langlands parameters over $\QQ_v$}\index[n]{Arthur--Langlands parameters}
$$ \psi: \WD_v \times \SL_2(\CC) \To \lang G^*_v$$ satisfying the usual axioms (without the requirement that $\psi(\WD_v)$ is bounded); see \cite[\S 2.5]{taibi}. Let $\Psi(G^*_v)$\index{$\Psi(G^*_v)$} be the set of $\psi \in \Psi^+(G^*_v)$ such that $\psi (\WD_v)$ is bounded. 

Following \cite[\S 1.5]{arthurbook} we define a subset $\Psi^+_{\uni} (G^*_v)$\index{$\Psi^+_{\uni} (G^*_v)$} of $\psi \in \Psi^+(G^*_v)$ as follows. For any $\psi \in \Psi^+(G^*_v)$, the parameter $$\stan_{ G^*} \circ \psi : \WD_v \times \SL_2(\CC) \To \GL_N(\CC)$$ gives rise to an irreducible representation $\pi_1 \boxtimes \cdots \boxtimes \pi_r$ of a standard Levi subgroup $\prod_{i=1}^r \GL_{N_i } (\QQ_v)$ of $\GL_N(\QQ_v)$; see \cite[p.~45]{arthurbook} and \cite[\S 1.2.2]{KMSW} for this construction (using the local Langlands correspondence for general linear groups). By definition, $\psi$ is an element of $\Psi^+_{\uni}(G^*_v)$ if and only if the normalized parabolic induction $\pi_1 \times \cdots \pi_r$ of  $\pi_1 \boxtimes \cdots \boxtimes \pi_r$ to $\GL_N(\QQ_v)$ is irreducible and unitary. 
As on p.~45 of \cite{arthurbook}, we have a chain of subsets 
$$ \Psi (G^*_v) \subset \Psi^+_{\uni} (G^*_v) \subset \Psi ^+(G^*_v).$$
 
For $\psi \in \Psi^+_{\uni} ({G^*_v})$, we define\index{$S_{\psi}$, local case} \index{$\mathcal S_{\psi}$, local case}
\begin{align*}
S_{\psi} & : = \mathrm{Cent} (\psi, \widehat{ G^*} ), \\ \mathcal S _{\psi } & :=  S_{ \psi }  /  S_{\psi} ^0 Z(\widehat{ G^*})^{\Gamma_{v}}.
\end{align*}
As in the global case, the group $\mathcal S_{\psi}$ is a finite abelian $2$-group. We write $\mathcal S_{\psi} ^D$ for its Pontryagin dual group. Denote by $s_{\psi} \in  S_{\psi}$\index{$s_{\psi}$} the image of $-1 \in \SL_2(\CC)$ under $\psi$. 

 We fix a $\QQ_v$-splitting $\spl_v$ for $G^*_v$. When $d$ is even, let $\theta_v$ be the unique non-trivial automorphism of ${G^*_v}$ fixing $\spl_v$ (which is of order $2$). When $d$ is odd we take $\theta _v=\id_{G^*_v}$. For both parities of $d$, we fix a Whittaker datum $\mathfrak w_v$ for $G^*_v$ that is fixed by $\theta_v$. (For instance, in the even case we can construct $\mathfrak w_v$ from $\spl_v$ and the choice of a non-trivial character $\QQ_v \to \CC^{\times}$ in the usual manner.) 
 
 In the even case, if we let $\spl_v$ vary over all $\QQ_v$-splittings of $G^*_v$, then the   resulting $\theta_v$'s are all of the form $\Int(g)|_{G^*_v}$ for certain $g \in \mathrm{O}(\underline V) (\QQ_v) - G^*(\QQ_v)$. In fact, by explicit construction it is easy to see that there is one choice of $\theta_v$ that is of the asserted form. To see that \emph{all} choices of $\theta_v$ are of the asserted form, use that all $\QQ_v$-splittings of $G^*_v$ are conjugate under $G^{*, \ad}(\QQ_v)$, and that $G^{*, \ad}(\QQ_v)$ naturally acts on $\mathrm{O}(\underline V) (\QQ_v)$ by conjugation since the center of $G^*$ is central in $ \mathrm{O}(\underline V)$.
 As a consequence of this observation, if we have two choices $\theta_v$ and $\theta_v'$, then $\theta_v = \theta_v' \circ \Int(g_0)$ for some $g _0 \in G^*_v(\QQ_v)$. In particular, the way in which $\theta_v$ permutes isomorphism classes of representations of $G^*_v(\QQ_v)$ (resp.~conjugacy classes in $G^*_v(\QQ_v)$) is the same as the way in which $\theta_v'$ permutes these objects.

  Let $\psi \in \Psi^+_{\uni} ({G^*_v})$.  Then Arthur \cite[\S 1.5]{arthurbook} associates to $\psi$ a finite multi-set\footnote{In \cite{taibi}, this set is simply denoted by $\Pi_{ \psi}$.} $\widetilde \Pi_{\psi} ({G^*_v})$.\index{$\widetilde \Pi_{\psi}(G^*_v)$} Here each element of $\widetilde \Pi_{\psi}({G^*_v})$ is a $\set{1,\theta_v}$-orbit of isomorphism classes of finite-length smooth representations\footnote{By construction these representations are obtained as parabolic inductions of irreducible representations, and are hence finite-length smooth representations. \ignore{(See for instance \cite[Lem.~VI.6.2]{renardbook} for the last claim).}} of $G^*(\QQ_v)$, and such an element is allowed to repeat itself for finitely many times in $\widetilde{\Pi}_{\psi} ({G^*_v})$ (thus ``multi-set''). If $\psi \in \Psi({G^*_v})$ then these representations are all irreducible and unitary.  Moreover, for general $\psi \in \Psi^+_{\uni} ({G^*_v})$, there is a canonical map (depending on the choice of $\mathfrak w_v$) \index{$\langle \cdot,\pi \rangle$} \begin{align}\label{eq:Arthur's pairing}
\widetilde \Pi _{\psi}({G^*_v})  & \To \mathcal S_{\psi}^D  \\ \nonumber \pi & \longmapsto \lprod{\cdot ,\pi}. 
\end{align} 
\begin{defn}\label{defn:local Hecke alg}
We define the \emph{Hecke algebra}\index[n]{Hecke algebra (local)} $\mathcal H(G^*_v)$\index{$\mathcal H(G^*_v)$} as follows. When $v$ is finite, we define $\mathcal H({G^*_v})$ to be $ C^{\infty} _c({G^*}(\QQ_v))$. When $v = \infty$, we fix a maximal compact subgroup $K_{\infty} \subset G^*(\RR)$, and define $\mathcal H({G^*_v})$ to consist of smooth compactly supported functions on ${G^*}(\RR)$ that are bi-finite under $K_{\infty}$. Moreover for each place $v$ we define \index{$\widetilde{\mathcal H} (G^*_v)$}
$$\widetilde{\mathcal H} ({G^*_v}): = \mathcal H({G^*_v}) ^{\theta_v =1}, $$ and define \index{$\widetilde{\mathcal H}^{\st}(G^*_v)$} $$\widetilde{\mathcal H}^{\st} (G^*_v) \subset \mathcal H(G^*_v)$$ to be the subspace consisting of $f \in  \mathcal H(G^*_v)$ such that $f - \theta_v^* f $ has all stable orbital integrals equal to $0$.  
\end{defn}  
\subsection{}\label{para:stable dist} Let $\psi \in \Psi^+_{\uni}({G^*_v})$. 
In \cite[Thm.~2.2.1]{arthurbook}, Arthur gives a characterization of $\widetilde \Pi_{\psi} ({G^*_v})$ and the map $\pi \mapsto \lprod{\cdot, \pi}$, and proves that the linear form \index{$\Lambda_{\psi}$, local case}
\begin{align}\label{eq:Lambda}
\Lambda_{\psi}: \widetilde{\mathcal H} ({G^*_v}) & \To \CC \\ \nonumber  f& \longmapsto \sum_{ \pi \in \widetilde \Pi_{\psi}({G^*_v}) } \lprod{s_{\psi} , \pi } \Tr (\pi ( f dg ) ) 
\end{align}is \emph{stable}, in the sense that $\Lambda_{\psi}(f) = 0$ if all stable orbital integrals of $f$ vanish. (In \textit{loc.~cit.~}these  results are explicitly stated only for $\psi \in \Psi(G^*_v)$, but see Remark \ref{rem:analytic continuation} below.) We explain the notations. Here $dg$ is a fixed Haar measure on ${G^*}(\QQ_v)$. The summation takes into account the multiplicities of the elements $\pi$ in the multi-set $\widetilde \Pi_{ \psi} ({G^*_v})$. For each such element $\pi$, which is a $\set{1,\theta_v}$-orbit of representations of $G^*(\QQ_v)$, we let $\dot \pi$ be any element of this orbit, and define $\Tr(\pi (fdg)): = \Tr (\dot \pi (fdg))$. Since $f \in \widetilde{\mathcal H} (G^*_v)$ is by definition fixed by $\theta_v$ and since $dg$ is obviously fixed by $\theta_v$ (as $\theta_v$ has order at most $2$), this definition is independent of the choice of $\dot \pi$.

 It is clear from the characterization in \cite[Thm.~2.2.1]{arthurbook} that $\Lambda _{\psi}$ is independent of the choice of $\mathfrak w_v$, although the definition of the map $\pi \mapsto \lprod{\cdot ,\pi}$ depends on $\mathfrak w_v$. Moreover, since $\Lambda_{\psi}$ is stable, we can naturally extend its domain of definition to $\widetilde{\mathcal H}^{\st}(G_v^*)$, and still obtain a stable distribution
 \begin{align*}
\Lambda_{\psi}: \widetilde{\mathcal H}^{\st}(G_v^*) &\To \CC \\
f & \longmapsto \Lambda_{\psi} ( \frac{f + \theta_v^* f}{2}) .
 \end{align*}    In \S \ref{subsubsec:local packets}, we observed that  different choices of $\theta_v$ permute conjugacy classes in $G^*_v(\QQ_v)$ in the same way. In particular, $\widetilde{\mathcal H}^{\st}(G_v^*)$ is independent of the choice of $\theta_v$. If we view $\Lambda_{\psi}$ as being defined over $\widetilde{\mathcal H}^{\st}(G_v^*)$, then it is also independent of the choice of $\theta_v$, as follows from the characterization in \cite[Thm.~2.2.1]{arthurbook}.    

\begin{rem}\label{rem:taut}
We note that $\widetilde \Pi_{\psi}({G^*_v})$
depends on $\psi\in \Psi^+_{\uni}({G^*_v})$ only  via its orbit under $\Aut(\lang G^*_v)$. In the odd case such an orbit is the same as a $\widehat {G^*}$-conjugacy class, since $\Aut(\lang G^*_v) = (\widehat {G^*})^{\ad}$.  In the even case, by contrast, such an orbit could contain up to two $\widehat{ G^*}$-conjugacy classes. This is because $\Aut(\lang G^*_v)$ is identified with $\mathrm{O}_N(\CC)^{\ad}$, whose action on $\lang G^*_v$ is determined by the following two conditions: 
\begin{enumerate}
	\item The projection map from $\lang G^*_v$ to the Galois factor is preserved.
	\item The map $\lang G^*_v \to \lang G^* \xrightarrow{\stan_{ G^*}} \mathrm{O}_N(\CC) \subset \GL_N(\CC)$ is $\mathrm{O}_N(\CC)^{\ad}$-equivariant, where $\mathrm{O}_N(\CC)^{\ad}$ acts on $\mathrm{O}_N(\CC)$ by conjugation.
\end{enumerate} In particular, $(\widehat {G^*})^{\ad}$ is of index $2$ in $\Aut(\lang G^*_v)$. When the $\Aut(\lang G^*_v)$-orbit of $\psi$ contains two $\widehat{ G^*}$-conjugacy classes,  one should regard $\widetilde \Pi_{\psi}(G_v^*)$ as the concoction of  two conjectural Arthur packets.
\end{rem}
\begin{rem}\label{rem:mult free}
As remarked in \cite[\S 1.5]{arthurbook}, it follows from the work of Moeglin \cite{moeglinmult1} that the multi-set $\widetilde \Pi _{\psi} ({G^*_v})$ for $\psi \in \Psi({G^*_v})$ (and therefore also for $\psi \in  \Psi^+_{\uni}({G^*_v})$ by construction) is in fact multiplicity free in the non-archimedean case.
\end{rem}
 
 \subsection{}\label{para:variant for prod of two groups}
 Let $(H, \lang H, s, \eta)$ be an endoscopic datum for $G^*_v$, and assume that it is the localization of an elliptic endoscopic datum for $G^*$ over $\QQ$. Thus $H = H^+ \times H^-$ is the direct product of two quasi-split special orthogonal groups over $\QQ_v$. (Under our assumption, the endoscopic datum $(H, \lang H, s,\eta)$ over $\QQ_v$ itself may still be non-elliptic. More precisely, in the odd case it is always elliptic, while in the even case it is elliptic if and only if either $G^*_v$ is the split $\SO_2$ over $\QQ_v$ or neither of $H^{\pm}$ is the split $\SO_2$ over $\QQ_v$; cf.~the discussion at the beginning of \S \ref{subsubsec:explicit tb}.)

 As in \S \ref{subsubsec:local packets}, let $\Psi^+(H), \Psi^+(H^+), \Psi^+(H^-)$ be the sets of all Arthur--Langlands parameters for $H, H^+, H^-$ over $\QQ_v$ respectively. 
 We have a natural identification $\Psi^+ (H ) \cong \Psi^+(H^+) \times \Psi^+(H^-)$, to be viewed as the identity. 
 We define $\Psi^+_{\uni} (H)$\index{$\Psi^+_{\uni}(H)$} to be the preimage 
 of $\Psi^{+}_{\uni} (G^*_v)$, defined in \S \ref{subsubsec:local packets}, under the map $\Psi^+(H) \to \Psi^+(G^*_v),\psi \mapsto \eta \circ \psi.$  Also, we define $\Psi^+_{\uni}(H^{\pm})$\index{$\Psi^+_{\uni}(H^{\pm})$} in a similar way as in \S \ref{subsubsec:local packets}, with $G^*_v$ replaced by the quasi-split special orthogonal group $H^{\pm}$. We have
 $$\Psi^+_{\uni}(H^+) \times \Psi^+_{\uni}(H^-) \subset \Psi^+_{\uni} (H) .$$
Indeed, this containment boils down to the fact that every representation of $\GL_N(\QQ_v)$ that is the normalized parabolic induction of an irreducible unitary representation of a Levi subgroup is irreducible unitary. In the non-archimedean case this fact is  Bernstein's theorem \cite{Ber84}. In the archimedean case this fact is implicit in the work of Vogan \cite{voganunit} and also follows from Kirillov's conjecture proved by Baruch \cite{Baruch} plus the work of Sahi \cite{Sahi}. We note, however, that in general 
 $$\Psi^+_{\uni}(H^+) \times \Psi^+_{\uni}(H^-) \subsetneqq \Psi^+_{\uni}(H) .$$
 
 Now let $\psi \in \Psi^+_{\uni} (H) $, and write $\psi^{\pm}$ for the components of $\psi$ in $\Psi^+(H^\pm)$. Similarly as in \S \ref{para:stable dist}, we have stable distributions
 \begin{align*}
\Lambda_{\psi^{+}}: \widetilde{\mathcal H}^{\st} (H^{+}) & \To \CC , \\\Lambda_{\psi^{-}}: \widetilde{\mathcal H}^{\st} (H^{-}) &\To \CC,
 \end{align*}  (after fixing Haar measures). We define \index{$\widetilde{\mathcal H}^{\st}(H)$}
 $$ \widetilde{\mathcal H}^{\st}(H) : = \widetilde{\mathcal H}^{\st} (H^+)\otimes_{\CC} \widetilde{\mathcal H}^{\st} (H^-) .$$ Taking the product of $\Lambda_{\psi^+}$ and $\Lambda_{\psi^-}$, we obtain a stable distribution \index{$\Lambda_{\psi}$, local case}
 $$ \Lambda_{\psi} :  \widetilde{\mathcal H}^{\st}(H)  \To \CC. $$
 
 We have an expansion of $\Lambda_{\psi}$ similar to (\ref{eq:Lambda}). To make this precise, 
 similarly as in \S \ref{subsubsec:local packets}, we fix a $\QQ_v$-splitting $\spl_{H^\pm}$ of $H^{\pm}$, and let $\theta_{H^\pm}$ be the unique non-trivial automorphism of $H^{\pm}$ fixing $\spl_{H^{\pm}}$ in the even case, and the identity on $H^{\pm}$ in the odd case. Fix a Whittaker datum $\mathfrak w_{H^{\pm}}$ for $H^{\pm}$ that is fixed by $\theta_{H^{\pm}}$. Then similarly as in \S \ref{subsubsec:local packets}, we have the local packet $\widetilde{\Pi} _{\psi^{+}} (H^{+})$, which is a multi-set whose elements are $\lprod{\theta_{H^+}}$-orbits of isomorphism classes of representations of $H^+(\QQ_v)$. Similarly we have $\widetilde{\Pi} _{\psi^{-}} (H^{-})$. Define the packet $\widetilde {\Pi}_{\psi}(H)$ as the product of $\widetilde {\Pi}_{\psi^{\pm}}(H^{\pm})$, and we regard its elements as 
 $\lprod{\theta_{H^+}} \times \lprod{\theta_{H^-}}$-orbits of isomorphism classes of representations of $H(\QQ_v) = H^+ (\QQ_v) \times H^- (\QQ_v)$. We have maps $\widetilde \Pi_{\psi^{\pm}} (H^{\pm}) \to \mathcal S_{\psi^{\pm}}^D$ as in (\ref{eq:Arthur's pairing}), and taking the product we obtain a map $\widetilde \Pi_{\psi} (H) \to \mathcal S_{\psi}^D$, which we still denote by $\pi \mapsto \lprod{\cdot,\pi}$.\index{$\langle \cdot,\pi \rangle$} 
 Define $$ \widetilde{\mathcal H} (H^{\pm}) : = \mathcal H(H^{\pm})^{\theta_{H^{\pm} }= 1} , $$ and $$ \widetilde{\mathcal H} (H) : = \widetilde{\mathcal H} (H^+) \otimes \widetilde{\mathcal H} (H^-) .$$ We then have the expansion 
 \begin{align}\label{eq:stab dist for H}
\Lambda _{\psi} (h)=  \sum _{\pi \in \widetilde \Pi _{\psi} (H)} \lprod{s_{\psi} , \pi } \Tr (\pi (h)) , \qquad \forall h \in \widetilde{\mathcal H} (H).
 \end{align}
 Here, as in (\ref{eq:Lambda}), the summation takes into account the multiplicities, and for each $\pi$ we define $\Tr(\pi(h))$ to be $\Tr(\dot \pi (h))$ for any $\dot \pi \in \pi$, the Haar measure on $H(\QQ_v)$ being implicit. 
 
 We comment that the constructions of the packets $\widetilde \Pi_{\psi^{\pm}} (H^{\pm})$, the maps from them to $\mathcal S_{\psi^{\pm}}^D$, and the stable distributions $\Lambda_{\psi^{\pm}}$, are of a slightly more general nature than the previous constructions for $G^*_v$ in \S\S \ref{subsubsec:local packets} and \ref{para:stable dist}, since $\psi^{\pm}$ may not lie in $\Psi^+_{\uni}(H^{\pm})$. Nevertheless, the assumption that $\psi  =(\psi^+, \psi^-)$ lies in $\Psi^+_{\uni}(H)$ implies that $\psi^{\pm}$ can be constructed from a Levi subgroup $M \subset H^{\pm}$, a parameter in $\Psi(M)$, and a point $\lambda \in \mathfrak a_M^*$ as on p.~45 of \cite{arthurbook}, in exactly the same way as any element of $\Psi^+_{\uni}(H^{\pm})$ can be constructed from such data. The proof of this  fact, which is implicitly used in \cite{arthurbook}, is an elementary exercise using \cite[Thm.~D]{tadic} in the non-archimedean case and \cite{TadicGLnC} in the archimedean case. Thus the
  construction using 
 parabolic induction on the representation side and analytic continuation on the character side as indicated on pp.~45--46 of \cite{arthurbook} works for the current $\psi^{\pm}$ in the same way as it works for elements of $\Psi^+_{\uni}(H^{\pm})$.

\subsection{}\label{rem:canonicity of stable dist}
Fix $\psi \in \Psi^+_{\uni}({G^*_v})$ and fix a semi-simple element $s \in S_{\psi}$. Then there is an induced endoscopic datum $(H, \mathcal H, s, \eta: \mathcal H \to \lang {G^*_v})$ over $\QQ_v$. Arthur has proved an endoscopic character relation for such $\psi$ and $s$. For our applications, we only need the case where the endoscopic datum $(H, \mathcal H, s, \eta)$ is the localization over $\QQ_v$ of an elliptic endoscopic datum for $G^*$ over $
\QQ$, so we assume this for simplicity. Thus as in \S \ref{para:variant for prod of two groups}, $H = H^+ \times H^-$ is the direct product of two quasi-split special orthogonal groups over $\QQ_v$, and as usual we choose an identification $\mathcal H \cong \lang H$.  We have $\psi = \eta \circ \psi'$ for a unique $\psi' \in  \Psi^+_{\uni} (H)$. As in \S \ref{para:variant for prod of two groups}, we have the stable distribution 
 $ \Lambda_{\psi'} :  \widetilde{\mathcal H}^{\st}(H)  \To \CC $
 after fixing a   Haar measure $dh$ on $H(\QQ_v)$.
 
The Whittaker datum $\mathfrak w_v$ for ${G^*_v}$ determines a normalization of the transfer factors between $H$ and ${G^*_v}$; cf.~\S \ref{subsubsec:setting for transf factor}.  For any $f \in \widetilde{\mathcal H} ({G^*_v})$, let $f'$ be a Langlands--Shelstad transfer in $\mathcal H(H)$, with respect to the normalization of transfer factors just mentioned and the Haar measures $dg$ on $G^*_v(\QQ_v)$, $dh$ on $H(\QQ_v)$. Then $f' \in \widetilde{\mathcal H} ^{\st} (H)$; see \cite[\S 2.1]{arthurbook} or \cite[Prop.~3.3.1]{taibi}. We have the following \emph{endoscopic character relation}\index[n]{endoscopic character relation} (\cite[Thm.~2.2.1 (b)]{arthurbook}):
\begin{align}\label{eq:endoscopic character relation}
 \sum _{\pi \in \widetilde \Pi _{\psi} ({G^*_v})} \lprod{ s_{\psi}  s, \pi} \Tr (\pi ( f dg ) ) = \Lambda_{\psi ' } (f') .
\end{align}  
 
\begin{rem}\label{rem:analytic continuation}
In  \cite[Thm.~2.2.1]{arthurbook}, the stability of $\Lambda _{\psi'}$ and the relation (\ref{eq:endoscopic character relation}) are    explicitly stated only in the case where $\psi \in \Psi({G^*_v})$ and $ \psi' \in \Psi (H)$. The generalization to the case where $\psi \in \Psi^+_{\uni} ({G^*_v})$ and $ \psi ' \in  \Psi^+_{\uni}(H) $ can be easily obtained by analytic continuation, as explained on p.~46 of \cite{arthurbook}.
\end{rem}

\section*{Unramified parameters and representations}

\subsection{}\label{para:set for unr}
We complement our exposition   with a discussion on how unramified representations appear in local Arthur packets. Keep the setting and notation of \S \ref{subsubsec:local packets}, and assume that the place $v$ is finite. 
We say that a parameter $\psi : \WD_v \times \SL_2(\CC) \to \lang G^*_v$ in $ \Psi^+ (G^*_v) $ is \emph{unramified},\index[n]{unramified (parameter)} if the reductive group $G^*_v$ over $\QQ_v$ is unramified, and the restriction of $\psi$ to $\WD_v = W_{\QQ_v} \times \SU_2(\RR)$ is trivial on $\SU_2(\RR)$ and sends every element $\tau$ of the inertia subgroup of $W_{\QQ_v}$ to $1 \rtimes \tau \in \lang G^*_v$.

The existence of an unramified $\psi \in \Psi^+(G^*_v)$ by definition presupposes that $G^*_v$ is unramified. We  assume that this is the case. Then inside $G^*(\QQ_v)$, there is a unique $G^*(\QQ_v)$-conjugacy class of hyperspecial subgroups which are compatible with the fixed Whittaker datum $\mathfrak w_v$, in the sense of \cite{CS80}. Let $K^*_v$ be such a hyperspecial subgroup. Since $\theta_v$ fixes $\mathfrak w_v$, we know that  $\theta_v$ stabilizes the $G^*(\QQ_v)$-conjugacy class of $K^*_v$. In particular, $\theta_v$ permutes isomorphism classes of $K^*_v$-unramified representations of $G^*(\QQ_v)$. 

\begin{lem}\label{lem:parameter is unramified} Assume that $G^*_v$ is unramified, and let $K^*_v$ be a hyperspecial subgroup of $G^*(\QQ_v)$ as in \S \ref{para:set for unr}. Let $\psi \in  \Psi^+_{\uni} (G^*_v)$. The following statements hold. 
	\begin{enumerate}
		\item  The packet $\widetilde \Pi _{\psi} (G^*_v)$ contains at most one element that is a $\set{1,\theta_v}$-orbit of $K^*_v$-unramified representations of $G^*(\QQ_v)$. It contains one if and only if $\psi$ is unramified. 
		\item Assume that $\psi $ is unramified, and let $\pi \in \widetilde{\Pi} _{\psi} (G^*_v)$ be the unique element that is a  $\set{1,\theta_v}$-orbit of $K^*_v$-unramified representations, as in (1). Then for any $\dot \pi \in \pi$, we have $\dim (\dot \pi ^{K^*_v}) =1$, or equivalently, $\dot \pi$ has a unique $K^*_v$-unramified Jordan--H\"older constituent. Moreover, the unramified Langlands parameter $\WD_v \to \lang G^*_v$ of that Jordan--H\"older constituent (with respect to the unramified local Langlands correspondence) is in the same $\Aut(\lang G^*_v)$-orbit (see Remark \ref{rem:taut}) as the Langlands parameter $\varphi _{\psi}$ associated to $\psi$. Here $\varphi _{\psi} (w): = \psi ( w, \diag(\norm{w}  ^{1/2}, \norm{w} ^{-1/2}) ) $ for $ w\in \WD_v.$
		\item Let $\psi$ and $\pi$ be as in (2). We have $\lprod{\cdot ,\pi} =1 \in \mathcal S_{\psi}^D$. 
	\end{enumerate}  	
\end{lem}
\begin{proof} If $\psi \in \Psi (G_v^*)$, then parts (1) and (3) are proved in \cite[Lem.~4.1.1]{taibidim}, and part (2) follows from the characterization in \cite[Thm.~2.2.1]{arthurbook}. (In this case, all elements of $\widetilde \Pi_{\psi} (G^*_v)$ are $\set{1,\theta_v}$-orbits of smooth \emph{irreducible} representations of $G^*(\QQ_v)$.) For general $\psi \in   \Psi^+_{\uni}(G^*_v)$, we know that $\psi$ arises from a standard Levi subgroup $M \subset G^*$, an element $\psi_M \in \Psi(M)$ (i.e., a local Arthur--Langlands parameter for $M$ which is bounded on $\WD_v$), and an element $\lambda \in \mathfrak a_M^*$, as on p.~45 of \cite{arthurbook}. The packet $\widetilde \Pi_{\psi}(G^*_v)$ is constructed from the packet $\widetilde \Pi_{\psi_M}(M)$ of $M(\QQ_v)$-representations associated to $\psi_M$  via a certain parabolic induction process which involves $\lambda$; see \textit{loc.~cit.~}for more details. It is easy to see that $\psi$ is unramified if and only if $\psi_M$ is unramified. Moreover, the obvious analogue of the current lemma holds for $(M,\psi_M)$ in place of $(G^*_v, \psi)$. (More precisely, $M$ is a direct product of several general linear groups and one unramified special orthogonal group. The special case of the lemma for parameters bounded on $\WD_v$, which we have already proved, takes care of the special orthogonal factor of $M$. The general linear factors are taken care of by the local Langlands correspondence.) The lemma for $(G^*_v, \psi)$ then follows from the lemma for $(M, \psi_M)$, by basic properties of the parabolic induction process used in the definition of $\widetilde \Pi_{\psi}(G^*_v)$. (More specifically, we may assume that the standard parabolic subgroup $P\subset G^*_v$ containing $M$ as the Levi component is compatible with $K_v^*$ in the sense that $G^*(\QQ_v) = P(\QQ_v) K_v^*$. Let $K_M$ be the hyperspecial subgroup of $M(\QQ_v)$ given by the image of $P(\QQ_v) \cap K_v^*$ under the projection $P(\QQ_v) \to M(\QQ_v)$. Then for any irreducible smooth representation $\tau $ of $M(\QQ_v)$, the parabolic induction $\mathcal I_P(\tau)$ of $\tau$ to $G^*(\QQ_v)$ satisfies $\dim \mathcal I_P(\tau) ^{K^*_v} = \dim \tau^{K_M} \in \set{0,1}$. Moreover, when this number is $1$, we have compatibility between the unramified Langlands parameter of the unique $K^*_v$-unramified constituent of $\mathcal I_P(\tau) $ and that of $\tau$.)
\end{proof}

\subsection{}\label{para:setting for unramified lemma}
We have an obvious analogue of Lemma \ref{lem:parameter is unramified} with $G^*_v$ replaced by the group $H = H^+ \times H^-$ over $\QQ_v$ as in \S \ref{para:variant for prod of two groups}. To set up the notation, we assume that $H$ is unramified, and let $K_{H^{\pm}}$ be a hyperspecial subgroup of $H^{\pm}(\QQ_v)$ that is compatible with the Whittaker datum $\mathfrak w_{H^{\pm}}$ for $H^{\pm}$ (so $K_{H^{\pm}}$ is unique up to $H^{\pm}(\QQ_v)$-conjugacy). Let $K_{H} : = K_{H^+} \times K_{H^-} \subset H(\QQ_v)$. Since $\mathfrak w_{H^{\pm}}$ is fixed by  $\theta_{H^{\pm}}$, we know that  elements of the group $\lprod{\theta_{H^+}} \times \lprod{\theta_{H^-}} \subset \Aut(H)$ stabilize the $H(\QQ_v)$-conjugacy class of $K_{H}$. In particular,  $\lprod{\theta_{H^+}} \times \lprod{\theta_{H^-}} $ permutes isomorphism classes of $K_{H}$-unramified representations of $H(\QQ_v)$. 
  
\begin{lem}\label{lem:parameter is unramified for H} Keep the setting of \S \ref{para:setting for unramified lemma}. Let $\psi \in \Psi^+_{\uni} (H)$. The following statements hold. 
	\begin{enumerate}
		\item  The packet $\widetilde \Pi _{\psi} (H)$ contains at most one element that is a $\lprod{\theta_{H^+}} \times \lprod{\theta_{H^-} }$-orbit of $K_{H}$-unramified representations of $H(\QQ_v)$. It contains one if and only if $\psi$ is unramified. 
		\item Assume that $\psi$ is unramified. Let $\pi \in \widetilde{\Pi} _{\psi} (H)$ be the unique element that is a  $\lprod{\theta_{H ^+}} \times \lprod{\theta_{H^-} }$-orbit of $K_{H}$-unramified representations, as in (1). Then for any $\dot \pi \in \pi$, $\dot \pi$ has a unique $K_H$-unramified Jordan--H\"older constituent. Moreover, the unramified Langlands parameter $\WD_v \to \lang H$ of that Jordan--H\"older constituent is in the same $\Aut(\lang H)$-orbit as the Langlands parameter $\varphi _{\psi}$ associated to $\psi$.  
	\item Let $\psi$ and $\pi$ be as in (2).  We have $\lprod{\cdot ,\pi} =1 \in \mathcal S_{\psi}^D$. 
	\end{enumerate}  	
\end{lem}
\begin{proof} This follows from Lemma \ref{lem:parameter is unramified} applied to $H^+$ and $H^-$ separately. More precisely, write $\psi = (\psi^+, \psi^-)$ with $\psi ^{\pm}\in \Psi^+(H^{\pm})$. Although $\psi^{\pm}$ may not lie in $\Psi^+_{\uni}(H^{\pm})$, the proof of Lemma \ref{lem:parameter is unramified} still applies to $(H^{\pm}, \psi^{\pm})$ in place of $(G^*_v, \psi)$, in view of the comment at the end of \S \ref{para:variant for prod of two groups}. 
\end{proof}

\section*[Spectral expansion of discrete part of stable trace formula]{The spectral expansion of the discrete part of the stable trace formula}
 
 \subsection{}\label{para:spec exp}
Consider an elliptic endoscopic datum $(H = H^+ \times H^-, \lang H,s ,\eta) $ for $G^*$ over $\QQ$, presented in the explicit form as in \S \ref{subsec:elliptic end data}. Let $\psi \in \widetilde{\Psi} (H)$.  For each place $v$ of $\QQ$, there is a natural \emph{localization}\index[n]{localization (of a global Arthur parameter)} $$\psi_v = (\psi_v^+, \psi_v^-) \in  \Psi^+_{\uni}(H^+_{\QQ_v}) \times \Psi^+_{\uni}(H^-_{\QQ_v}) \subset \Psi^+_{\uni}(H_{\QQ_v})$$ of $\psi$ that is well defined up to the action of $\Aut(\lang H_{\QQ_v}) = \Aut(\lang H^+_{\QQ_v}) \times \Aut(\lang H^-_{\QQ_v})$, and there are natural homomorphisms $S_{\psi} \to S_{\psi_v}$ and $\mathcal S_{\psi} \to \mathcal S_{\psi_v}$; see \cite[\S 1.4 and pp.~46--47]{arthurbook}.  Note that the image of $s_{\psi} \in S_{\psi}$ under $S_{\psi} \to S_{\psi_v}$ is precisely $s_{\psi_v}$. 

  Let $\widetilde {\mathcal H}^{\st}(H)$\index{$\widetilde{\mathcal H}^{\st}(H)$ (global)} be the restricted tensor product of $\widetilde{\mathcal H} ^{\st} (H_{\QQ_v})$ over all places $v$. More precisely, consider a large enough finite set of prime numbers $\Sigma$ such that $H$ extends to a reductive group scheme $H'$ over $\ZZ[1/\Sigma]$, and such that the image of a fixed admissible splitting $\Out(H) \to \Aut (H)$ is contained in $\Aut(H') \subset \Aut (H)$. Then for all primes $p \notin \Sigma$, the function $1_{H'(\ZZ_p)}$ is in $\widetilde{\mathcal H} ^{\st} (H_{\QQ_p})$. We form the restricted tensor product with respect to these distinguished elements for almost all $p$. As usual, the result is independent of the choices of $\Sigma $ and $H'$. 
 
The discrete part of Arthur's stable trace formula for $H$ is a formal sum $$S^H_{\disc} = \sum_{t\geq 0} S^H_{\disc, t} $$ of stable distributions over all real numbers $t\geq 0$; see \cite[\S\S 3.1, 3.2]{arthurbook}, and cf.~\S \ref{subsec:intro}. For each $t \geq 0$ and any $f \in \widetilde{\mathcal H}^{\st} (H)$, we have the following spectral expansion by \cite[Lem.~3.3.1, Prop.~3.4.1, Thm.~4.1.2]{arthurbook}:
\begin{align}\label{eq:one t}
S^H_{\disc , t } (f) = \sum_{ \psi \in \widetilde{\Psi} (H), t(\psi) = t } m_{\psi }  \abs{\mathcal S_{\psi}} ^{-1} \sigma(\bar S_{\psi} ^0) \epsilon_{\psi} (s_{\psi}) \Lambda_{ \psi} (f),\end{align}
where $\Lambda_{\psi}$\index{$\Lambda_{\psi}$, global case} is the product\footnote{Here it is implicit that if we fix a finite set $\Sigma$ of primes and fix a reductive model $H'$ of $H$ over $\ZZ[1/\Sigma]$, then for almost all primes $p \notin \Sigma$ we have $\Lambda_{\psi_p}(1_{H'(\ZZ_p)}) =1$. It follows that $\Lambda_{\psi}$ is well defined on $\widetilde {\mathcal H}^{\st} (H)$, i.e., there is no issue with infinite products.}  of the local stable distributions $\Lambda_{\psi_v}  : \widetilde{\mathcal H}^{\st}(H_{\QQ_v}) \to \CC$ as in \S \ref{para:variant for prod of two groups}, and $\sigma(\bar S_{\psi} ^0)$\index{$\sigma(\bar S_{\psi}^0)$} is an invariant associated to the following connected complex reductive group (see \cite[Prop.~4.1.1]{arthurbook}):\index{$\bar S_{\psi}^0$}
$$\bar S_{\psi} ^0 : = (S_{\psi} / Z(\widehat{ H}) ^{\Gamma_{\QQ}}) ^0. $$
Thus formally we have 
\begin{align}\label{eq:spectral expansion}
 S^H_{\disc } (f) = \sum_{ \psi \in \widetilde{\Psi} (H)} m_{\psi }  \abs{\mathcal S_{\psi}} ^{-1} \sigma(\bar S_{\psi} ^0) \epsilon_{\psi} (s_{\psi}) \Lambda_{ \psi} (f).
\end{align}

\section[Ta\"ibi's parametrization of local packets]{Ta\"ibi's parametrization of local packets for certain pure inner forms} \label{subsec:Taibi}
\subsection{}\label{para:equiv of qsplit}
 We keep the setting of \S \ref{subsec:review of Arthur}. In particular, we fix $G^* = \SO(\underline V, \underline q)$. For each place $v$ of $\QQ$, we shall consider a \emph{pure inner form} $(G_v, \Xi_v, z_v)$ of $G^*_v = G^*_{\QQ_v}$, by which we mean the following data: \begin{itemize}
	\item a reductive group $G_v$ over $\QQ_v$;
	\item an isomorphism $\Xi_v :  G^* _{ \overline \QQ_v } \isom (G_v) _{ \overline \QQ_v }$ defined over $\overline \QQ_v$;
	\item a (continuous) cocycle $z_v\in Z^1(\Gamma_v, G^*_v)$ such that $\leftidx^{\rho} \Xi_v ^{-1}\Xi_v = \Int( z_v(\rho))^{-1}$ for all $\rho \in \Gamma_v$. 
\end{itemize} 
 We recall Ta\"ibi's parametrization in \cite{taibi} of the Arthur packets for $G_v$ under special hypotheses. For each place $v$, note the equivalence of the following conditions:
 \begin{enumerate}
 	\item The image of $z_v$ in $\coh^1(\QQ_v, G^{*,\ad})$ is trivial.
 	\item The reductive group $G_v$ over $\QQ_v$ is quasi-split.
 \end{enumerate}
 Indeed, that (1) implies (2) is clear, and  the converse amounts to the assertion that only the trivial element of $\coh^1(\QQ_v, G^{*,\ad})$ goes to the trivial element of $\coh^1(\Gamma_v, \Aut(G^*_{\overline \QQ_v}))$. This is clear in the odd case since all automorphisms of $G^*_{\overline \QQ_v}$ are inner. In the even case, this is true because the inner automorphisms form an index $2$ subgroup of $\Aut(G^*_{\overline \QQ_v})$, and in the complement there is an element invariant under $\Gamma_v$, for instance the conjugation action on $G^*_v$ by any element of $\mathrm{O}(\underline V)(\QQ_v)$ of determinant $-1$.  
 \section*{Finite places}
\subsection{}\label{subsubsec:finite places} Let $v$ be a finite place of $\QQ$. 
We assume that the image of $z_v$ in $\coh^1 (\QQ_v, G^{*,\ad})$ is trivial, or equivalently (see \S \ref{para:equiv of qsplit}), that $G_v$ is quasi-split as an abstract reductive group over $\QQ_v$. We caution the reader that under our assumption it could still happen that $z_v$ has non-trivial image in $\coh^1(\QQ_v, G^*)$ (when $d$ is even).

In the odd case, let $\theta_{G_v}$ be the identity automorphism of $G_v$. In the even case, fix a  $\QQ_v$-splitting of $G_v$ and let $\theta_{G_v}$\index{$\theta_{G_v}$}
 be the unique non-trivial automorphism of $G_v$ fixing that splitting (which is of order $2$). As we have observed in \S \ref{subsubsec:local packets}, the way in which $\theta_{G_v}$ permutes isomorphism classes of  representations of $G_v(\QQ_v)$ or conjugacy classes in $G_v(\QQ_v)$ is canonical. 

Fix a Whittaker datum $\mathfrak w_v$ for $G^*_v$.
As explained in \cite[\S 2.2]{kaldepth0} (cf.~Remark \ref{rem:pure inner form}), the datum $(\mathfrak w_v, \Xi_v ,z_v )$ determines a  normalization of transfer factors between any endoscopic datum $H$ for $G_v$ and $G_v$. We denote this normalization by $\Delta ^{G_v}_H (\mathfrak w_v, \Xi_v , z_v)$.\index{$\Delta^{G_v}_H (\mathfrak w_v, \Xi_v, z_v)$} We summarize in the next proposition the construction in \cite[\S 3.3]{taibi}.

\begin{prop}\label{prop:Taibi finite} 
For each $\psi \in  {\Psi}^+_{\uni} (G^*_v)$, there is a finite multi-set\footnote{This is denoted by $\Pi_{\psi}(G_v)$ in \cite[\S 3.3]{taibi}. By its construction and by Remark \ref{rem:mult free}, this multi-set is actually multiplicity free.}\index{$\widetilde{\Pi}_{\psi} (G_v)$} $\widetilde{\Pi}_{\psi} (G_v)$ of $\set{1,\theta_{G_v}}$-orbits of isomorphism classes of finite-length smooth representations of $G_v(\QQ_v)$, and a canonical map (depending on $(\mathfrak w_v, \Xi_v, z_v)$) \index{$\langle \cdot,\pi \rangle$} \begin{align*}
\widetilde \Pi _{\psi} (G_v) & \To \pi_0 (S_{\psi})^D \\  \pi & \longmapsto \lprod{\cdot ,\pi}.
\end{align*}
 Moreover, if all the representations in $\widetilde{  \Pi}_{\psi} (G_v^*)$ are irreducible, then so are those in $\widetilde {\Pi} _{\psi} (G_v)$. For each semi-simple $s \in S_{\psi}$ inducing an endoscopic datum $(H, \mathcal H, s, \eta)$ over $\QQ_v$, we have an endoscopic character relation\index[n]{endoscopic character relation}. For simplicity, we describe it only under the same assumption on $(H, \mathcal H, s, \eta)$ as in \S \ref{rem:canonicity of stable dist}. As usual fix an identification $\lang  H \cong \mathcal H$. Let $\psi' \in \Psi^+_{\uni}(H)$ be such  that $ \psi = \eta\circ \psi'$. Fix Haar measures on $G_v(\QQ_v)$ and $H(\QQ_v)$. Let $f\in\mathcal H(G_v)$, and assume that the orbital integrals of $f$ are invariant under $\theta_{G_v}$. Let $f' \in \mathcal H(H)$ be a Langlands--Shelstad transfer of $f$ with respect to the normalization $\Delta^{G_v}_H (\mathfrak w_v, \Xi_v, z_v)$ of transfer factors. Then we have $
f'\in \widetilde {\mathcal H} ^{\st} (H),$ and $$  \sum _{ \pi \in \widetilde{\Pi} _{\psi}(G_v) } \lprod{s_{\psi} s , \pi } \Tr (\pi (f)) = \Lambda_{ \psi'} (f'). $$
Here we understand that $s, s_{\psi} \in S_{\psi}$ are naturally mapped into $\pi_0 (S_{\psi})$ in writing $\lprod{s_{\psi} s , \pi }$. 
\end{prop}

\begin{proof} 
	In \cite[\S 3.3]{taibi}, it is assumed that $\psi \in \Psi  (G^*_v)$, and $\lprod{\cdot ,\pi}$ is constructed as a character on $\mathcal S_{\psi} ^+$ rather than a character on $\pi_0(S_{\psi})$. Here $\mathcal S_{\psi} ^+$ is a certain finite extension of $\mathcal S_{\psi}$ sitting in a chain of surjective group homomorphisms 
	$$ \mathcal S_{\psi} ^+ \To \pi_0 (S_{\psi}) \To \mathcal S_{\psi}. $$
	We indicate why the reformulation as in the present proposition is valid. 
	
	We first note that the construction in \cite[\S 3.3]{taibi} generalizes verbatim from $\psi \in \Psi (G^*_v)$ to $\psi \in  \Psi^+_{\uni} (G^*_v)$, based on the ``$\Psi^+_{\uni}$-version'' of   Arthur's results recalled in \S\S \ref{subsubsec:local packets}--\ref{rem:canonicity of stable dist} and Remark \ref{rem:analytic continuation}. 
Moreover the finite-length and irreducible properties stated in the proposition follow from the corresponding properties of  $\widetilde{\Pi}_{\psi} (G_v^*)$, since by construction $\widetilde{\Pi}_{\psi} (G_v)$ contains the same representations as $\widetilde{\Pi}_{\psi} (G_v^*)$, with respect to a certain $\QQ_v$-isomorphism $G_v^* \isom G_v$ which we do not explain.

It remains to explain why it is valid to replace $\mathcal S_{\psi} ^+$ by  $\pi_0 (S_{\psi})$ (which is denoted by $\pi_0(C_{\psi})$ in \cite{taibi}). 	The reason that one needs to consider $\mathcal S_{\psi} ^+$ in general is due to the fact that when $G_v$ is fixed as a rigid inner form of $G_v^*$, in order to normalize transfer factors between an endoscopic datum and $G_v$ one needs to upgrade the former to a ``refined endoscopic datum'', which roughly means picking a lift in $\mathcal S_{\psi}^+$ of the image of $s\in S_{\psi}$ in $\pi_0(S_{\psi})$.  In our present case,  this is not necessary thanks to the fact that $(G_v,\Xi_v, z_v)$ is a pure inner form of $G^*$: Each semi-simple element  $s\in S_{\psi}$ determines an endoscopic datum $(H, \mathcal H, s, \eta)$, and the datum $(\mathfrak w_v, \Xi_v, z_v)$ already determines canonically a  normalization of transfer factors between $H$ and $G_v$. Moreover, as noted in \cite[Rmk.~3.3.2]{taibi}, the pairing $\lprod{\cdot, \pi}$ for $\pi \in \widetilde{\Pi } (G_v)$ descends to a character on $\pi_0 (S_{\psi})$ in our case. In conclusion it is valid to replace the group $\mathcal S^+_{\psi}$ in \cite[\S 3.3]{taibi} by $\pi_0 (S_{\psi})$ in our case.
 \end{proof}
\section*{The archimedean place}
\subsection{}\label{subsubsec:arch para}
Let $v = \infty$. Assume that $G^*_v$ contains anisotropic maximal tori. Let $(G_v, \Xi_v, z_v)$ be an arbitrary pure inner form of $G^*_v$ as in \S \ref{para:equiv of qsplit}. Thus $G_v$ also contains anisotropic maximal tori. As in the non-archimedean case, we fix a Whittaker datum $\mathfrak w_v$ for $G^*_v$, and then the datum $(\mathfrak w_v, \Xi_v, z_v)$ determines a normalization of transfer factors between any endoscopic datum $H$ for $G_v$ and $G_v$, 
 which we denote by $\Delta^{G_v}_H(\mathfrak w_v, \Xi_v, z_v)$. 

Recall that any Arthur--Langlands parameter $\psi \in \Psi^+(G^*_{v})$ (through its associated Langlands parameter $\varphi_{\psi}$) has a well-defined \emph{infinitesimal character},\index[n]{infinitesimal character} which is an $\Omega_{\CC}(G,T)$-orbit in $X^*(T) \otimes_{\ZZ} \CC$. Here $T$ is any maximal torus in $G^*_{\CC}$, and $\Omega_{\CC}(G,T)$ is the complex Weyl group. For an account see for instance \cite[\S 4.1.2]{taibidim}, where the infinitesimal character is denoted by $\mu_1$. Following the terminology of Buzzard--Gee \cite{buzzardgee}, we say that the infinitesimal character is \emph{C-algebraic} (resp.~\emph{regular C-algebraic})\index[n]{regular C-algebraic}\index[n]{C-algebraic} if it is the $\Omega_{\CC}(G,T)$-orbit of an element of $\rho + X^*(T)$ (resp.~a regular element of $\rho + X^*(T)$), where $\rho \in \frac{1}{2} X^*(T)$ is the half sum of a system of positive roots. 

 For $\psi \in \Psi^+(G^*_v)$, we say that it is \emph{Adams--Johnson}\index[n]{Adams--Johnson parameter} if it is bounded on $W_{\RR}$ (i.e., $\psi \in \Psi(G_{v}^*)$) and has regular C-algebraic infinitesimal character. For more details see \cite[\S 4.2.2]{taibidim} and  \cite[\S 8.1]{AMR}. We denote by $\Psi^{\AJ} (G^*_v)$\index{$\Psi^{\AJ} (G^*_v)$} the set of Adams--Johnson parameters for $G^*_v$. We know that all $\psi \in \Psi^{\AJ}(G^*_v)$ are \emph{discrete}, in the sense that $S_{\psi} = \pi_0 (S_{\psi})$. 

 For each $\psi \in \Psi^{\AJ}(G^*_v)$, Adams--Johnson \cite{adamsjohnson} have explicitly constructed a packet $\Pi_{\psi} ^{\AJ} (G_v)$\index{$\Pi_{\psi}^{\AJ}(G_v)$} of representations of $G_v(\RR)$. Using the rigidifying datum $(\mathfrak w_v, \Xi_v ,z_v)$, Ta\"ibi \cite[\S\S 3.2.2--3.2.3]{taibi} associates to each $\pi \in \Pi_{ \psi} ^{\AJ} (G_v)$ a character $\lprod{\cdot, \pi}$ of $\mathcal S_{\psi}^+$. Here as in the proof of Proposition \ref{prop:Taibi finite} the finite group $\mathcal S_{\psi}^+$ sits in a chain of surjective group homomorphisms 
$$ \mathcal S_{\psi} ^+ \To \pi_0 (S_{\psi}) \To \mathcal S_{\psi}, $$
and its introduction is in fact unnecessary in our situation thanks to the fact that we have fixed $G_{v}$ as a pure inner form of $G^*_{v}$ (as opposed to a more general rigid inner form). Namely, for each $\psi \in \Psi^{\AJ} (G^*_v)$ and $\pi \in \Pi^{\AJ} _{\psi} (G_{v}) $, the pairing $\lprod{\cdot , \pi}$ descends to a character on $\pi_0(S_{\psi})= S_{\psi}$. This assertion could either be directly checked by going through Ta\"ibi's construction, or be proved as follows: By the well-definedness of the normalization $\Delta^{G_v}_H (\mathfrak w_v, \Xi_v ,z_v)$ of transfer factors between an endoscopic datum $H$ and the pure inner form $G_{v}$, we know that the right hand side of the endoscopic character relation in \cite[Prop.~3.2.5]{taibi} depends on $\dot s \in \mathcal S^+_{\psi }$ only via its image in $\pi_0 (S_{\psi})  = S_{\psi}$. It follows that so does the left hand side, which means that $\lprod{\cdot, \pi}$ descends to $S_{\psi}$ as desired. 

With the above modification, we summarize the results in \cite[\S\S 3.2.2--3.2.3]{taibi} together with a comparison result in \cite{AMR} as follows. \begin{prop} \label{prop:summary in arch}
	For any $\psi \in {\Psi}^{\AJ} (G^*_v) $, let $\Pi^{\AJ} _{\psi} (G_{v})$ be the associated (finite) Adams--Johnson packet. There is a canonical map (depending on $(\mathfrak w_v, \Xi_v, z_v)$)\footnote{Here the subscript ``$\AJT$'' stands for Adams--Johnson--Ta\"ibi.} \index{$\langle \cdot,\pi\rangle_{\AJT}$}
	\begin{align*}
\Pi ^{\AJ } _{\psi} (G_{v}) & \To \pi_0 (S_{\psi})^D = S_{\psi}^D \\  \pi & \longmapsto \lprod{\cdot ,\pi} _{\AJT}.
	\end{align*}
	Fix $s \in S_{\psi}$, and let $(H, \mathcal H, s, \eta)$ be the induced endoscopic datum over $\RR$, which is necessarily an elliptic endoscopic datum because $S_{\psi}$ is discrete. We have an endoscopic character relation\index[n]{endoscopic character relation} described as follows. As usual fix an identification $\lang  H \cong \mathcal H$, and let $\psi' \in \Psi^+(H)$ be such that $\psi = \eta \circ \psi'$. Then $\psi' \in \Psi^{\AJ} (H)$.   Fix Haar measures on $G_v(\RR)$ and $H(\RR)$.\ignore{When $d$ is even fix a non-inner automorphism $\theta_{G_v}$ of $G_v$ defined over $\RR$ (which exists; see \cite[Rmk.~3.1.1, \S 3.2.3]{taibi}). When $d$ is odd let $\theta_{G_v} : = 1$. 
	} The following statements hold.
\begin{enumerate}
	\item Fix a Haar measure $dh$ on $H(\RR)$. The distribution \index{$\Lambda_{\psi'}^{\AJ}$}
	\begin{align*}
\Lambda_{\psi'} ^{\AJ} : C^{\infty}_c(H(\RR)) & \To \CC \\  f' & \longmapsto   \sum _{\pi' \in \Pi^{\AJ} _{\psi'} (H) } \lprod{s_{\psi'} , \pi'}_{\AJT} ~ \Tr(\pi'(f' dh))
	\end{align*} is stable.
	\item For $f' \in \widetilde {\mathcal H} ^{\st}(H)$, we have 
	$$ \Lambda_{\psi'} ^{\AJ} (f') = \Lambda_{\psi'} (f'). $$ Here $\Lambda_{\psi'}$ is as in \S \ref{para:variant for prod of two groups}. 
		\item Fix a Haar measure $dg$ on $G_v(\RR)$. Let $f\in  C_{c} ^{\infty}(G_{v} (\RR))$, and let $f'$ be a Langlands--Shelstad transfer of $f$ in $C^{\infty}_c(H(\RR))$, with respect to the normalization $\Delta^{G_v}_H (\mathfrak w_v, \Xi_v, z_v)$ of transfer factors and the Haar measures $dg, dh$. 
		We have 
		\begin{align}\label{eq:AJ character relation}
 e(G_v) \sum _{\pi \in \Pi^{\AJ} _{\psi} (G_v) } \lprod{s_{\psi} s, \pi} _{\AJT}~ \Tr(\pi(f))  =  \Lambda_{\psi'} ^{\AJ} (f').
		\end{align} Here $e(G_v)$ is the Kottwitz sign of $G_v$. 
		\ignore{\item Assume that $f\in\mathcal H(G_v)$ has $\theta_{G_v}$-invariant orbital integrals. Let $f'$ be a Langlands--Shelstad transfer of $f$ in $\mathcal H(H)$. Then $f' \in \widetilde{\mathcal H} ^{\st} (H)$, and in particular (by (2) and (3)) we have  
		\begin{align}\label{eq:AJ Arthur character relation}
e(G_{v}) \sum _{\pi \in \Pi^{\AJ} _{\psi} (G_v) } \lprod{s_{\psi} s, \pi}_{\AJT} ~ \Tr(\pi(f))  =  \Lambda_{\psi'} (f'). 
		\end{align}}
	\qed
	\end{enumerate}
	
\end{prop}

\begin{rem}\label{rem:comparison with AJ}
	By the formula $\lprod{s_{\psi} , \pi_{ \psi, \mathbf Q, \mathbf L}} = e(\mathbf L)$ in the proof of \cite[Prop.~3.2.5]{taibi}, the distribution $\Lambda_{\psi'} ^{\AJ}$ in part (1) of Proposition \ref{prop:summary in arch} is none other than the distribution that appears in \cite[Thm.~2.13]{adamsjohnson}. With this understanding, part (1) is the same as \cite[Thm.~2.13]{adamsjohnson}, and part (2) is proved in \cite{AMR}. \end{rem}
\section{The global group $G$}\label{subsec:setting in application}
\subsection{}\label{para:global G^*}
	Fix $d  = 2m+1$ or $2m$ where $m \in \ZZ_{\geq 1},$ and fix $\delta\in \QQ^{\times}/ \QQ^{\times ,2}$. Assume that $(-1)^m \delta > 0$. Let $(\underline V, \underline q)$ be the quasi-split quadratic space (in the sense of Definition \ref{defn:qsplit quad sp}) over $\QQ$ of dimension $d$ and discriminant $\delta$, which is unique up to isomorphism. Let $G^* = \SO(\underline V, \underline q)$. We note that by our assumption on $\delta$, there exist inner twistings between the $\RR$-groups $\SO(d-2,2)$ and $G^*_{\RR}$. We fix a $G^*(\CC)$-conjugacy class\footnote{In the even case there are two such conjugacy classes to choose from. Nevertheless, the resulting two  ways of viewing $\SO(d-2,2)$ as an inner form of $G^*_{\RR}$ give rise to isomorphic inner forms of $G^*_{\RR}$. This is because the two $G^*(\CC)$-conjugacy classes of inner twistings are interchanged under any non-inner automorphism of $\SO(d-2,2)_{\CC}$, and there exists one such automorphism defined over $\RR$.} of such inner twistings, and thereby view $\SO(d-2,2)$ as an inner form of $G^*_{\RR}$.
\begin{lem}\label{lem:existence of V}
 The following statements hold. 
	\begin{enumerate}
		\item There exists at most one isomorphism class $G$ of inner forms of $G^*$ such that $G $ is isomorphic to $\SO(d-2,2)$ as inner forms of $G^*$ over $\RR$ and $G$ is quasi-split over $\QQ_v$ as a reductive group (or equivalently, $G_{\QQ_v}$ is isomorphic to $G^*_{\QQ_v}$ as inner forms of $G^*_{\QQ_v}$; see \S \ref{para:equiv of qsplit}) for all finite places $v$.
		\item 
		Assume either of the following two conditions:
		\begin{itemize}
			\item $d \equiv 2,3,4,5, 6 \mod 8$.
			\item $d \equiv 0 \mod 8$ and $ \delta \neq 1 \in \QQ^{\times}/ \QQ^{\times ,2}$. 
		\end{itemize} Then there is a quadratic space $(V,q)$ over $\QQ$, of dimension $d$, discriminant $\delta$, and signature $(d-2, 2)$ at $\infty$, such that $G: = \SO(V, q)$ is quasi-slit at all finite places. 
	\end{enumerate} 
\end{lem}
\begin{proof}
	Let $F $ be $\QQ$ or $\QQ_v$. The set $\coh^1(F, G^*)$ classifies isomorphism classes of pure inner forms of $G^*$ over $F$, and it also classifies isomorphism classes of quadratic spaces $(V,q)$ over $F$ whose dimension is $d$ and discriminant is $\delta$. Thus the lemma is just a reformulation of parts 1,2 of \cite[Prop.~3.1.2]{taibi}, in the special case where the base number field is $\QQ$. In fact, the condition in part 1 of that proposition reads $d \equiv 3,5 \mod 8$. The condition in part 2 (a) reads $d \equiv 2,6 \mod 8$. The condition in part 3 reads $d\equiv 4 \mod 8$, or $(d \equiv 0 \mod 4 \mbox{~and~} \delta \neq 1)$. 
\end{proof}
\begin{rem}In part (2) of the above lemma, the isomorphism class of $(V, q)$ may not be unique in the even case. The quadratic space $(V,q) \otimes_{\QQ} \QQ_v$ may not be quasi-split (in the sense of Definition \ref{defn:qsplit quad sp}) for all finite places $v$. 
\end{rem}
\subsection{}\label{para:trivial pure}
In the rest of the paper we fix $d\geq 5 , \delta, (\underline{V }, \underline{q}),  G^*$ as in \S \ref{para:global G^*}, and fix  $(V , q), G$ as in part (2) of Lemma \ref{lem:existence of V}. We shall apply the preceding parts of this paper, in particular Corollary \ref{Main Main result}, to $(V,q)$ and $G$. As in \S \ref{Fixing inner twist}, we fix an isometry $\phi_V: (V,q)\otimes \overline \QQ \isom (\underline{V},\underline{q})\otimes \overline \QQ ,$ and use it to define the inner twisting $\psi_V: G_{\overline \QQ} \isom G^*_{\overline \QQ}, g \mapsto \phi_V g \phi_V^{-1}$ as well as the function  $u_V : \Gamma_{\QQ} \to G^*(\overline \QQ), \rho \mapsto \leftidx^\rho\phi_V \phi_V^{-1}$. To conform with the convention of \cite{taibi}, we let $\Xi$ be $\psi_V^{-1}$ and let $z$ be the function $\Gamma_{\QQ} \to G^*(\overline \QQ), \rho \mapsto u_V(\rho)^{-1}$. Then according to that convention it is $z$ rather than $u_V$ that is a cocycle, and $(G,\Xi, z)$ is a global pure inner form of $G^*$ over $\QQ$. 

At each place $v$ of $\QQ$, by localization we obtain a pure inner form $(G_v , \Xi_v ,z_v)$ of $G^*_v$, where $G_v : = G_{\QQ_v}$. By construction this pure inner form  satisfies the hypothesis in \S \ref{subsubsec:finite places} when $v$ is finite.  

We fix once and for all a global Whittaker datum $\mathfrak w$ for $G^*$.

We also fix an automorphism $\theta_G$\index{$\theta_G$} of $G$ once and for all, as follows. In the odd case let $\theta_G = \id_G$. In the even case, we fix an element $\mathbf r$\index{$\mathbf r$} of $\mathrm{O}(V) (\QQ) - G(\QQ)$ of order $2$ (for instance, the reflection on $V$ associated to an anisotropic vector), and let $\theta_G = \Int(\mathbf r)|_G$. Thus in this case $\theta_G$ is of order $2$.

We know that there exists a large enough finite set $\Sigma$ of prime numbers such that $G^*$ (resp.~$G$) admits a reductive model $\mathcal G$ (resp.~$\mathcal G^*$) over $\ZZ [1/\Sigma]$. In particular, for any prime $p \notin \Sigma$, the group $G^*$ (resp.~$G$) is unramified over $\QQ_p$, and $\mathcal G^* (\ZZ_p)$ (resp.~$\mathcal G(\ZZ_p)$) is a hyperspecial subgroup of $G^*(\QQ_p)$ (resp.~$G(\QQ_p)$). Moreover, we may and shall assume that $\theta_{G}$  stabilizes $\mathcal G(\ZZ_p)$ for all $ p \notin \Sigma$, up to enlarging $\Sigma$. In fact, the $\QQ$-automorphism $\theta_{G}$ of $G$ extends to a $\ZZ[1/\Sigma]$-automorphism of the model $\mathcal G$ after suitably enlarging $\Sigma$. 

As argued in \cite[\S 3.4]{taibi}, we may further enlarge $\Sigma$ to a finite set of prime numbers, denoted by $\Sigma(\mathcal G^*, \mathcal G, \Xi, z, \mathfrak w , \theta_G)$,\index{$\Sigma(\mathcal G^*, \mathcal G, \Xi, z, \mathfrak w , \theta_G)$} such that the following conditions hold for all primes $p$ outside the set:
\begin{enumerate}
	\item As we have already assumed,  $\theta_G$ stabilizes $\mathcal G(\ZZ_p)$.
	\item The localization $\mathfrak w_p$ of $\mathfrak w$, which is a Whittaker datum for $G^*_p$, is compatible with the hyperspecial subgroup $\mathcal G^* (\ZZ_p) \subset G^*(\QQ_p)$ in the sense of \cite{CS80}. 
	\item The pure inner form $(G_p, \Xi_p, z_p)$ of $G^*_p$ over $\QQ_p$ is trivial. Equivalently, the quadratic spaces $(V,q)\otimes {\QQ_p}$ and $(\underline V, \underline q) \otimes \QQ_p$ are abstractly isomorphic over $\QQ_p$ (but $\phi_V$ itself may not be defined over $\QQ_p$). In particular, we have a canonical $G(\QQ_p)$-conjugacy class of $\QQ_p$-isomorphisms $G^*_p \isom G_p$, consisting of isomorphisms induced by isometries $V \otimes \QQ_p \isom \underline V \otimes \QQ_p$ that differ from $\phi_V$ by elements of $G^*(\overline \QQ_v)$ (as opposed to $\mathrm{O}(\underline V) (\overline \QQ_v)$). 
	\item Inside the canonical $G(\QQ_p)$-conjugacy class of $\QQ_p$-isomorphisms $G^*_p \isom G_p$ as in (3), there is one that extends to a $\ZZ_p$-isomorphism $\mathcal G^* _{\ZZ_p} \isom \mathcal G _{\ZZ_p}$. 
\end{enumerate}

\begin{defn}\label{defn:vartheta}
	Let $S$ be a finite set of places of $\QQ$. Let $\vartheta^S$ be the infinite direct product group $\prod_{v \notin S} \ZZ/2\ZZ$, where the product is over all places of $\QQ$ outside $S$. Let $\vartheta^S$\index{$\vartheta^S$} act on $G(\adele^S)$ by 
	$$ (\epsilon_v)_v \cdot (g_v)_v : = (\theta_G^{\epsilon_v} (g_v)) _v , \qquad \forall (\epsilon_v)_v \in \vartheta^S, (g_v)_v \in G(\adele^S).$$ Since $\theta_G$ fixes $\mathcal G(\ZZ_p)$ for almost all primes $p$, this action is well defined, and each element of $\vartheta^S$ acts via a topological group automorphism of $G(\adele^S)$. Similarly, we define $\vartheta_S : = \prod_{v \in S} \ZZ/2\ZZ$ and let $\vartheta_S$\index{$\vartheta_S$} act on $\prod_{v\in S} G(\QQ_v)$ by the same formula.
\end{defn}
 
\subsection{}\label{para:global packets}
Let $v$ be a finite place of $\QQ$ and let $\psi_v \in \Psi^+_{\uni} (G^*_v)$.  As in Proposition \ref{prop:Taibi finite} the local packet $\widetilde {\Pi}_{\psi_v} (G_v)$ is a set of $\set{1,\theta_{G_v}}$-orbits of isomorphism classes of representations of $G(\QQ_v)$, where $\theta_{G_v} \in \Aut(G_v)$ is chosen as in  \S \ref{subsubsec:finite places}. Since $\theta_{G_v}$ is of the form $\Int(g_v) |_{G_v}$ for some $g_v \in \mathrm{O}(V)(\QQ_v) - G(\QQ_v)$, we have $\theta _G = \theta_{G_v} \circ \Int (h_v)$ for some $h_v \in G(\QQ_v)$. Therefore we can view each element of $\widetilde {\Pi}_{\psi_v} (G_v)$ as a $\set{1,\theta_G}$-orbit, or equivalently, a $\vartheta_v$-orbit, of isomorphism classes of representations of $G(\QQ_v)$. We normalize the map $\widetilde \Pi_{\psi_v}(G_v) \to \pi_0 ( S_{\psi_v})^D, \pi_v \mapsto \lprod{\cdot, \pi_v}$  as in Proposition \ref{prop:Taibi finite}  with respect to the localization $(\mathfrak w_{v}, \Xi_{v}, z_{v})$ of $(\mathfrak w, \Xi , z )$ at $v$, where $(\mathfrak w, \Xi , z )$ is fixed in \S \ref{para:trivial pure}. Similarly, for any $\psi_{\infty} \in \Psi^{\AJ}(G^*_{\infty})$, we have the local packet $\Pi_{\psi}^{\AJ}(G_{\infty})$ as in Proposition \ref{prop:summary in arch}, and we normalize the map $\pi \mapsto \lprod{\cdot,\pi}_{\AJT}$ in that proposition with respect to the localization $(\mathfrak w_{\infty}, \Xi_{\infty}, z_{\infty})$ of $(\mathfrak w, \Xi , z)$ at $\infty$. In the sequel we always keep these normalizations, without explicitly mentioning them.

Now let $\psi \in \widetilde \Psi(G^*)$. For each place $v$ of $
\QQ$, we fix a localization $\psi_v \in \Psi^+_{\uni}(G^*_v)$ of $\psi$; see \S \ref{para:spec exp}. Let $S$ be a finite set of places of $\QQ$ containing $\infty$. We define the global (away from $S$) Arthur packet \index{$\widetilde \Pi_{\psi}^S(G)$}$\widetilde \Pi_{\psi}^S(G)$ to be the set of $(\pi_v)_{v\notin S} \in \prod_{v \notin S} \widetilde \Pi_{\psi_v}(G_v)$ such that $\pi_v$ is a $\vartheta_v$-orbit of isomorphism classes of $\mathcal G(\ZZ_v)$-unramified representations for almost all $v$. (Here note that for almost all $v$, $\vartheta_v$ permutes isomorphism classes of $\mathcal G(\ZZ_v)$-unramified representations.)   Now for all primes $v$ not in $\Sigma(\mathcal G^*, \mathcal G, \Xi, z, \mathfrak w , \theta_G)$, the packet $\widetilde \Pi_{\psi_v}(G_v)$  together with the map from it to $\pi_0 ( S_{\psi_v})^D$ is constructed from  $\widetilde \Pi_{\psi_v}(G^*_v)$  via an isomorphism $G_v \isom G^*_v$ as in \S \ref{para:trivial pure} (4); see \cite[\S 3.3]{taibi}. Moreover, $\psi_v$ is unramified for almost all $v$. Thus for almost all $v$, by Lemma \ref{lem:parameter is unramified} applied to $(G_v^*, \psi_v, \mathcal G^*(\ZZ_v))$, there is a unique $\pi_v \in \widetilde \Pi_{\psi_v} (G_v)$ which is a $\vartheta_v$-orbit of $\mathcal G(\ZZ_v)$-representations, and moreover for this $\pi_v$ we have $\lprod{\cdot, \pi_v} = 1 \in  \pi_0 ( S_{\psi_v})^D$ and $\dim (\dot \pi_v) ^{\mathcal G(\ZZ_v)} =1$ for any $\dot \pi_v \in \pi_v$. We conclude that for $\pi^S = (\pi_v)_{v\notin S} \in \widetilde \Pi^S_{\psi}(G)$, we have $\lprod{\cdot, \pi_v} = 1 \in  \pi_0 ( S_{\psi_v})^D$ for almost all $v$. 

For $\pi^S = (\pi_v)_{v\notin S} \in \widetilde \Pi^S_{\psi}(G)$, we choose a member $\dot \pi_v \in \pi_v$ for each $v$, and form the restricted tensor product $\dot \pi ^S : = \bigotimes'_{v \notin S} \dot \pi_v$, which makes sense as a smooth admissible representation of $G(\adele^S)$ since almost all $\dot \pi_v$ satisfy $\dim (\dot \pi_v) ^{\mathcal G(\ZZ_v)} =1$. The isomorphism class of the $G(\adele^S)$-representation $\dot \pi^S$ is well defined up to the $\vartheta^S$-action.

\ignore{
\begin{defn}\label{defn:symm Hecke for G}
 For each finite place $v$, define \index{$\mathcal H (G_v)$} \index{$\widetilde{\mathcal H}(G_v)$} \begin{align*}
 {  \mathcal H} (G_v) & : = C^{\infty} _c (G_v(\QQ_v)), \\ \widetilde {\mathcal H} (G_v) & : =   {\mathcal H} (G_v)^{\theta_{G}}. 
	\end{align*}
	Let $S$ be a finite set of places containing $\infty$. Define $\widetilde {\mathcal H}^{S} (G)$\index{$\widetilde {\mathcal H}^{S}(G)$} to be the restricted tensor product of $\widetilde {\mathcal H} (G_v)$ over all places $v\notin S$, with respect to the elements $1_{\mathcal G(\ZZ_v) } \in \widetilde{\mathcal H} (G_v)$ for almost all $v \notin S$. For any compact open subgroup $K^S \subset G(\adele^S)$, we define \index{$\widetilde{\mathcal H}(G(\adele^S)\sslash K^S)$}
	$$ \widetilde{  \mathcal H} (G(\adele^S) \sslash K^S) : = {  \mathcal H} (G(\adele^S) \sslash K^S) \cap \widetilde{  \mathcal H} ^S (G),$$ where the intersection is inside the space of all $C^{\infty}_c$-functions $G(\adele^S ) \to \CC$.  
\end{defn}

}

\ignore{ 
\begin{rem}\label{rem:other defn of symm Hecke}
	In the above definition, it can be easily checked that $\widetilde{ \mathcal H}^S (G)$ is also equal to the set of $f$ in the usual Hecke algebra $ \mathcal H^S(G)$ away from $S$ such that $f$ (not necessarily a pure tensor) is invariant under $\theta_{G_v}$ for all places $v\notin S$.
\end{rem}
}

\section{Spectral evaluation}\label{subsec:spectral eval}
\subsection{}\label{para:two V}
In the following we keep the setting and notation of \S \ref{subsubsec:setting for Morel's formula}, Theorem \ref{geometric assertion}, and Corollary \ref{Main Main result}, for the quadratic space $(V,q)$ fixed in \S \ref{para:trivial pure}. In particular we fix a neat compact open subgroup $K \subset G(\adele_f)$, and fix $f^{\infty} dg^{\infty} \in \mathcal H(G(\adele_f)\sslash K)_{\QQ}$.
 
We need a modified version of Corollary \ref{Main Main result} as follows. In \S \ref{para:test function intro}, we assumed that $\mathbb V$ is absolutely irreducible. In the odd case we keep that assumption, but in the even case we assume either one of the following two conditions:
\begin{enumerate}
	\item The algebraic $G_{\mathbb E}$-representation $\mathbb V$ is absolutely irreducible, and the isomorphism class of the $G_{\overline \QQ}$-representation $\mathbb V \otimes_{\mathbb E} \overline \QQ$ is preserved by outer automorphisms of $G_{\overline \QQ}$.
	\item We have $\mathbb V \cong \mathbb V_0 \oplus \mathbb V_1$, where $\mathbb V_0$ and $\mathbb V_1$ are absolutely irreducible algebraic $G_{\mathbb E}$-representations such that the isomorphism classes of the $G_{\overline \QQ}$-representations $\mathbb V_{0} \otimes_{\mathbb E} \overline \QQ$ and $\mathbb V_{1} \otimes_{\mathbb E} \overline \QQ$ are unequal and interchanged with each other under an outer automorphism of $G_{\overline \QQ}$.
\end{enumerate}
We shall call case (1) the \emph{even symmetric case},\index[n]{the even symmetric case} and case (2) the \emph{even composite case}.\index[n]{the even composite case} In the odd case and the even symmetric case, Corollary \ref{Main Main result} directly applies. In the even composite case, as in Theorem \ref{geometric assertion}, for each fixed $f^{\infty} dg^{\infty}$ we obtain two finite sets of prime numbers $\Sigma(\mathbf{O} (V), \mathbb V_i, \lambda, K, f^{\infty})$ for $i = 0,1$. We define $\Sigma(\mathbf{O} (V), \mathbb V, \lambda, K, f^{\infty})$ to be the union of these two sets. Clearly (\ref{eq:main}) still holds in this case, for any prime $p$ outside $ \Sigma(\mathbf{O} (V), \mathbb V, \lambda, K, f^{\infty})$ and satisfying the assumption in \S \ref{para:prepare for main stabilization}, if on the right hand side we define $f^H_{\infty}$ to be the \emph{sum} of the two test functions corresponding to $\mathbb V_0$ and $\mathbb V_1$. Indeed one obtains this by simply summing the two cases of (\ref{eq:main}) corresponding 
to $\mathbb V_0$ and $\mathbb V_1$.

In all of the odd case, the even symmetric case, and the even composite case, we define the finite sets of primes
$$ \Sigma _{\bad} ' (K, f^{\infty})  :  =   \Sigma(\mathbf{O} (V), \mathbb V, \lambda, K, f^{\infty}) \cup  \Sigma(\mathcal G^*, \mathcal G, \Xi, z, \mathfrak w , \theta_G),  $$ and \begin{align}
\label{eq:Sigma bad} 
\Sigma_{\bad}(K, f^{\infty}) & : = \Sigma
_{\bad}'(f^{\infty}) \cup \set{p \notin \Sigma
	_{\bad}'(K, f^{\infty}) \mid  K_p \neq  \mathcal G(\ZZ_p)}. 
\end{align}
We now fix a prime $p \notin \Sigma_{\bad}(K, f^{\infty})$, and apply (the modified) (\ref{eq:main}) to $p$. Note that the extra assumption on $p$ in the even case in \S \ref{para:prepare for main stabilization} is satisfied here by condition (2) in \S \ref{para:trivial pure}.  We thus obtain 
	\begin{align}\label{eq:main modified}
	\Tr(\Frob_p^a \times f^{ \infty} dg^{\infty} \mid  \icoh^* (\overline{ \Sh _K},\mathbb V) )   = \sum _{(H, \lang H, s ,\eta) \in \dot {\mathscr E}(G)}  \iota (G, H) ST^H (f^H) ,
\end{align}
for every sufficiently large $a$. On the right hand side, as we have already indicated,  the archimedean test function $f^H_{\infty}$ is defined to be the sum of the test functions constructed in \S \ref{subsec:test functions} corresponding to $\mathbb V_0$ and $\mathbb V_1$ in the even composite case. Here we view the two sides of (\ref{eq:main modified}) as numbers in $\mathbb C$, but recall from  Theorem \ref{geometric assertion} and Remark \ref{rem:in E} that the left hand side is actually in $
\mathbb E$. 
  
 \begin{rem} In the even composite case,
 	 $\icoh^* (\overline{ \Sh _K},\mathbb V)$ is the direct sum of $\icoh^* (\overline{\Sh _K},\mathbb V_0)$ and $\icoh^* (\overline{\Sh _K},\mathbb V_ 1)$ as $\mathcal H(G(\adele_f ) \sslash K)_{\QQ} \times \Gamma_{\QQ}$-modules. We explain how the latter two are related to each other. Let $K' = K \cap \theta_G(K)$. Then $K'$ is a compact open subgroup of $K$, and the $\mathcal H(G(\adele_f ) \sslash K)_{\QQ} \times \Gamma_{\QQ}$-module $\icoh^* (\overline{\Sh _K},\mathbb V_i)$ is obtained from the $\mathcal H(G(\adele_f ) \sslash K')_{\QQ} \times \Gamma_{\QQ}$-module $\icoh^* (\overline{\Sh _{K'}},\mathbb V_i)$ by taking $K$-invariants. It is easier to describe the relation between $\icoh^* (\overline{\Sh _{K'}},\mathbb V_i)$ and $\icoh^* (\overline{\Sh _{K'}},\mathbb V_i)$, so we replace $K$ by $K'$. Write $\mathcal H$ for $\mathcal H(G(\adele_f ) \sslash K)_{\QQ}$. Then $\theta_G$ induces a ring automorphism of $\mathcal H$. Now observe\footnote{We thank the anonymous referee for bringing this observation to our attention.} that the automorphism $\theta_G = \Int(\mathbf r)|_G$ of $G$ induces an automorphism of the Shimura datum $\mathbf O(V) = (G,\mathcal X, h)$, since $\mathbf r\in \mathrm{O}(V)(\QQ)$ induces an automorphism of the space $\mathcal X$ of oriented negative definite planes in $V_{\RR}$, and $h$ intertwines this automorphism with the automorphism $f\mapsto \theta_G \circ f$ of $\Hom(\mathbb S, G_{\RR})$. Moreover $\theta_G $ interchanges the isomorphism classes of the $G_{\overline \QQ}$-representations $\mathbb V_{0,\overline \QQ}$ and $\mathbb V_{1,\overline \QQ}$. Therefore by transport of structure we have an  $\mathcal H \times \Gamma_{\QQ}$-module isomorphism 
 	$$  \icoh^* (\overline{\Sh _K},\mathbb V_1) \cong \icoh^* (\overline{\Sh _K},\mathbb V_0) \otimes _ { \mathcal H, \theta_G} \mathcal H $$
 \end{rem}
 
\begin{lem}\label{lem:invariance} Suppose that $f^{\infty}$, as a function on $G(\adele_f)$, is fixed by the group $\vartheta^{\infty}$ (see Definition \ref{defn:vartheta}).  Then for each $(H,\lang H, s, \eta) \in \dot{\mathscr E}(G) $  we have $f^H \in \widetilde{\mathcal H} ^{\st} (H)$, where $\widetilde{\mathcal H} ^{\st} (H)$ is defined in \S \ref{para:spec exp}.
 \end{lem}
\begin{proof} If $(H,\lang H, s, \eta)$ does not satisfy the conditions ($\dagger$) and ($\ddagger$) in \S \ref{para:test function intro}, then by definition $f^H =0$. In the following we assume that these conditions are satisfied. We can factorize $f^{\infty,p}$ as $f_S f^{\infty,p, S}$, where $S$ is a finite set of primes not containing $p$, $f_S \in C^{\infty}_c(\prod_{v\in S} G(\QQ_v))$, and $f^{\infty, p, S} = 1_{\mathcal G (\widehat \ZZ^{p, S})}\in C^{\infty}_c(G(\adele^{\infty,p, S}))$. Moreover, up to enlarging $S$, we may assume that $1_{\mathcal G (\widehat \ZZ^{p, S})}$ is fixed by $\vartheta^{\infty,p,S}$. Since $p$ is not in $\Sigma_{\bad}(f^{\infty})$, we also know that $1_{K_p} = 1_{\mathcal G(\ZZ_p)}$ is fixed by $\vartheta_p$. Hence our assumption that $f^{\infty}$ is invariant under $\vartheta^{\infty}$ implies that $f_S$ is invariant under $\vartheta_S$. By induction on $\abs{S}$, it is an elementary exercise to show that $f_S$ can be written as a sum of functions in $C^{\infty}_c(\prod_{v\in S} G(\QQ_v))$ each of which is completely factorizable (i.e., a product over $v\in S$ of functions in $C^{\infty}_c(G(\QQ_v))$) and invariant under $\vartheta_S$. 
Hence $f^{\infty,p}$ is a sum of functions in $C^{\infty}_c(G(\adele_f^p))$ each of which is completely factorizable and invariant under $\vartheta_S$. We have thus reduced to the case where $f^{\infty,p} = \prod_{v \neq \infty,p} f_v$, with each $f_v \in C^{\infty}_c(G(\QQ_v))$  invariant under $\theta_G$, and $f_v = 1_{\mathcal G(\ZZ_v)}$ for almost all $v$.

	 For each finite place $v \neq p$, we can choose an automorphism $\theta_{G_v}$ of $G_v$ as in \S \ref{subsubsec:finite places}. As we have observed in \S \ref{para:global packets}, $\theta _G = \theta_{G_v} \circ \Int (h_v)$ for some $h_v \in G(\QQ_v)$. Therefore the fact that $f_v$ is invariant under $\theta_G$ implies that $f_v$ has $\theta_{G_v}$-invariant orbital integrals. By Proposition \ref{prop:Taibi finite} we know that $f^H_v$, which is a Langlands--Shelstad transfer of $f_v$, lies in $\widetilde {\mathcal H}^{\st}(H_{\QQ_v})$. It remains to check that $f^H_{v} \in \widetilde{\mathcal H} ^{\st} (H_{\QQ_v})$
	for $v = \infty, p$. 
	
	The fact that $f^H_{\infty} \in \widetilde{\mathcal H} ^{\st} (H_{\RR})$ follows from the following ingredients: 
	\begin{itemize}
		\item The implicit fact that we may (and do) take $f^H_{\infty}$ inside $\mathcal H (H_{\RR}) \subset C^{\infty}_c(H(\RR))$. (By the construction in \S \ref{subsec:test functions}, this reduces to the fact \cite[Lem.~3.1]{arthurlef} that for any discrete series representation of $H(\RR)$, a pseudo-coefficient of it may be taken to be bi-finite under a prescribed maximal compact subgroup of $H(\RR)$.)
		\item The formula \cite[(7.4)]{kottwitzannarbor} for the stable orbital integrals of $f^H_{\infty}$.
		\item The invariance properties of the transfer factors shown in the proof of \cite[Prop.~3.2.6]{taibi}.
		\item The fact that for any semi-simple elliptic element $\gamma_0\in G(\RR)$ the term $e(I) \vol ^{-1} \Tr \xi_{\CC} (\gamma_ 0)$ (where $\xi_{\CC} = \mathbb V_{\CC}$) in \cite[(7.4)]{kottwitzannarbor} is invariant under replacing $\gamma_0$ by its image under any automorphism of $G_{\RR}$. (Note that in the even case this is false if we take $\xi_{\CC}$ to be a general irreducible representation of $G_{\CC}$.)
	\end{itemize}

 We now prove that $f^H_p \in \widetilde{\mathcal H} ^{\st} (H_{\QQ_p})$. By the discussion in \S \ref{para:transition maps for Hecke}, we have canonical actions of $\Aut(H_{\QQ_p})$ on $\mathcal H^{\ur}(H_{\QQ_p})$ and on $\mathscr A_{H_{\QQ_p}}$, under which the subgroup $H^{\ad}(\QQ_p) \subset \Aut(H_{\QQ_p})$ (consisting of inner automorphisms) acts trivially. Thus the outer automorphism group $\Out(H_{\QQ_p}) = \Aut(H_{\QQ_p})/ H^{\ad}(\QQ_p)$ acts on $\mathcal H^{\ur}(H_{\QQ_p})$ and $\mathscr A_{H_{\QQ_p}}$. Moreover the canonical Satake isomorphism $\mathcal H^{\ur}(H_{\QQ_p})\isom\mathscr A_{H_{\QQ_p}}$ is $\Out(H_{\QQ_p})$-equivariant. We need only show that the Satake transform of $f^H_p$ in $\mathscr A_{H_{\QQ_p}} = \mathscr A_{H_{\QQ_p}^+}  \otimes \mathscr A_{H_{\QQ_p}^-} $, which is computed in (\ref{eq:five cases}), is invariant under  $\Out(H_{\QQ_p}) = \Out(H_{\QQ_p}^+) \times \Out(H_{\QQ_p}^+)$.   In all the five cases in (\ref{eq:five cases}), the image of $\Out(H_{\QQ_p})$ in $\Aut(\mathscr A_{H_{\QQ_p}})$ is generated by the automorphism $Z_1 \mapsto - Z_1$ of $\mathscr A_{H_{\QQ_p}^+}$ (non-trivial in the second and fourth cases) and the automorphism $Y_1 \mapsto - Y_1$ of $\mathscr A_{H_{\QQ_p}^-}$ (non-trivial in the second and third cases). By (\ref{eq:five cases}), the Satake transform of $f^H_p $ is indeed invariant under $\Out(H_{\QQ_p})$. \end{proof}

\subsection{} We keep the assumption in Lemma \ref{lem:invariance} that $f^{\infty}$ is fixed by $\vartheta^{\infty}$. 
We assume Hypothesis \ref{hypo}. By Corollary \ref{cor:simp}, the expansion (\ref{eq:spectral expansion}), and Lemma \ref{lem:invariance}, we can rewrite (\ref{eq:main modified}) as 
\begin{multline}\label{eq:first expansion} 	\Tr(\Frob_p^a \times f^{ \infty} dg^{\infty} \mid  \icoh^* (\overline{ \Sh _K},\mathbb V) )  \\  = \sum _{ (H,\lang H, s,\eta) \in \dot{\mathscr E}(G)}\iota(G,H) \sum_{ \psi' \in \widetilde{\Psi} (H)} m_{\psi' }  \abs{\mathcal S_{\psi'}} ^{-1} \sigma(\bar S_{\psi'} ^0) \epsilon_{\psi'} (s_{\psi'}) \Lambda_{ \psi' } (f^H).
\end{multline}

\begin{lem}\label{lem:parameter is AJ}
	Assume that $\psi' \in \widetilde{\Psi} (H)$ contributes non-trivially to the RHS of (\ref{eq:first expansion}). Then $H$ is cuspidal, and $ \eta\circ \psi'_{\infty}   \in \Psi ^{\AJ} (G^*_{\infty})$. (In particular, $\psi'_{\infty} \in \Psi ^{\AJ} (H_{\RR})$.) Moreover, $\eta\circ \psi' _\infty$ has the same infinitesimal character as that of $\mathbb V_{\CC}^*$ (resp.~that of $\mathbb V_{0,\CC}^*$ or $\mathbb V_{1,\CC}^*$) in the odd case and the even symmetric case (resp.~the even composite case). 
\end{lem}

\begin{proof} We only treat the even composite case, the other two cases being similar. Recall that $f^H_{\infty} = f^H_{\infty,0}+ f^H_{\infty,1}$ where $f^H_{\infty, i}$ is the analogue of $f^H_{\infty}$ constructed in \S \ref{subsec:test functions} with $\mathbb V$ replaced by $\mathbb V_i$. Thus $f^H_{\infty} = 0$ unless $H$ is cuspidal; see \S \ref{para:test function intro}. Assume that $H$ is cuspidal. 
By \cite[Lem.~4.1.3]{taibidim} we know that for any $\psi''_{\infty}  \in \Psi (H_{\RR})$, all the representations in $\widetilde{\Pi } _{\psi''_{\infty}} (H_{\RR})$ have the same infinitesimal character as that of $\psi''_{\infty}$. By analytic continuation (see \cite[p.~46]{arthurbook} and cf.~Remark \ref{rem:analytic continuation}), the same conclusion holds for all $\psi''_{\infty} \in \Psi ^+_{\uni } (H_{\RR})$.  Hence in order that $\Lambda_{\psi'_{\infty}} (f^H_{\infty}) \neq 0$, the infinitesimal character of $\eta \circ \psi'_{\infty}$ must be the same as that of $\mathbb V_{0,\CC}^*$ or $\mathbb V_{1,\CC}^*$, which are regular C-algebraic. It remains to check that $\eta \circ \psi' _{\infty}$ is bounded on $W_{\RR}$. But this follows from the fact that $\eta\circ \psi'_{\infty}$ is the localization of the global parameter $\eta\circ \psi'$, the fact that it has C-algebraic infinitesimal character, and Clozel's purity lemma \cite[Lem.~4.9]{clozelannarbor}. (For a similar argument cf.~\cite[p.~309]{taibidim}.)  
\end{proof}
 
\subsection{}\label{para:preparation for Kottwitz}
Let $(H,\lang H, s,\eta) \in \dot {\mathscr E}(G)$. For each place $v$ of $\QQ$, let  $(\mathfrak w_{v}, \Xi_{v}, z_{v})$ be the localization at $v$ of $(\mathfrak w, \Xi , z )$ fixed in \S \ref{para:trivial pure}.  In \S \ref{subsubsec:finite places} and \S \ref{subsubsec:arch para}, we introduced a normalization $\Delta_{H_{\QQ_v}} ^{G_v}  (\mathfrak w_v, \Xi_v ,z_v)$ of  transfer factors between $H_{\QQ_v}$ and $G_v$ for each place $v$. In \S \ref{normalizing the transf factors} we also introduced a normalization $(\Delta^G_H)_v$. Thus we have 
$$a^G_{H,v} \Delta_{H_{\QQ_v}}^{G_v}  (\mathfrak w_v, \Xi_v ,z_v) = (\Delta^G_H)_v,  $$
for a constant $a^G_{H,v} \in \CC^{\times}$.\index{$a^G_{H,v}$} By construction, the normalizations $ (\Delta^G_H)_v$ are the canonical unramified normalizations at almost all places $v$ (associated to hyperspecial subgroups determined by a reductive model of $G$ over some Zariski open of $\spec \ZZ$), and satisfy the global product formula. The same holds for the normalizations $\Delta_{H_{\QQ_v}}^{G_v} (\mathfrak w_v, \Xi_v ,z_v)$, since $\mathfrak w_v$ for various $v$ are localizations of the global Whittaker datum $\mathfrak w$, and $(\Xi_v, z_v)$ for various $v$ are localizations of a global pure inner twist $(\Xi, z)$; see \cite[p.~137]{arthurbook} or \cite[Prop.~4.4.1]{kalethaglobal}. It follows that \begin{align}\label{eq:prod formula for a}
 \prod_ v  a^G_{H,v} = 1,
\end{align}where almost all terms in the product are $1$.  

Let $\psi' \in \widetilde{\Psi} (H)$.  In the following we compute the contribution of $\psi'$ to the RHS of (\ref{eq:first expansion}), based on Kottwitz's results in \cite[\S 9]{kottwitzannarbor}. For each place $v$, let $$\psi'_v \in \Psi^+_{\uni} (H_{\QQ_v}) $$ be a localization of $\psi'$ as in \S \ref{para:spec exp}. Let  $$\psi_v:  = \eta \circ \psi'_v \in {\Psi}^+_{\uni} (G_v^*), $$ and let 
$$ \psi : = \eta \circ \psi \in \widetilde \Psi (G^*)$$ as in \S \ref{subsubsec:substitutes}. For each place $v$, $\psi_v$ is indeed a localization of $\psi$, so our notation is consistent. 
	In Lemma \ref{lem:parameter is AJ}, we have already seen that a necessary condition for $\psi'$ to contribute non-trivially to the RHS of (\ref{eq:first expansion}) is that $\psi_{\infty}$ is Adams--Johnson with infinitesimal character determined by $\mathbb V$. In the following we assume this condition (but we do not assume that $\psi'$ has a non-zero contribution \textit{a priori}). In particular, $\psi'_{\infty}$ is discrete, and so  $\bar S_{\psi'} ^0 = \set{1}$. Thus by the definition of $\sigma(\bar S_{\psi'} ^0) $ in \cite[Prop.~4.1.1]{arthurbook}, we have 
\begin{align}\label{eq:sigma=1}
\sigma(\bar S_{\psi'} ^0) = 1.
\end{align}

We make several observations and definitions which will be understood in the statement of the next lemma. Recall from \S \ref{para:global packets} that 
for each finite place $v$ we view elements of $\widetilde {\Pi}_{\psi_v} (G_v)$ as $\vartheta_v$-orbits of isomorphism classes of representations of $G(\QQ_v)$. Since $p \notin \Sigma_{\bad}(K, f^{\infty})$, we know that $K_p = \mathcal G(\ZZ_p)$ and that $\vartheta_p$ stabilizes $\mathcal G(\ZZ_p)$. Hence $\vartheta_p$ permutes the isomorphism classes of $K_p$-unramified representations of $G(\QQ_p)$. Thus we can speak of whether an element of $\widetilde {\Pi}_{\psi_p}(G_p)$ is a $\vartheta_p$-orbit of $K_p$-unramified representations.  We write $\Lambda_{\psi'}^{p,\infty} $ for the product of the local stable distributions $\Lambda_{\psi'_v}$ over all places $v \notin \set{\infty, p}$, so we have $$\Lambda_{\psi'}(f^H) = \Lambda_{\psi'_{\infty}}(f^H_{\infty}) \Lambda_{\psi'_p} (f^H_p) \Lambda_{\psi'}^{p,\infty} (f^{H,p,\infty}). $$ As in \S \ref{para:global packets}, we define the global packet $\widetilde \Pi_{\psi}^{p,\infty}(G)$, and for each $\pi^{p,\infty} \in \widetilde \Pi_{\psi}^{p,\infty}(G)$ we define the $G(\adele_f^p)$-representation $\dot \pi ^{p,\infty}$ (which depends on arbitrary choices).  \begin{lem}\label{lem:Kott destab} Let $(H, \lang H, s,\eta), \psi', \psi, \psi'_v, \psi_v$ be as in \S \ref{para:preparation for Kottwitz}, and keep assuming that $\psi_{\infty}$ is Adams--Johnson with infinitesimal character determined by $\mathbb V$ as in Lemma \ref{lem:parameter is AJ}. The following statements hold. 
 	\begin{enumerate}
 		\item 

 	We have
\begin{align*} 
\Lambda_{\psi' _{\infty}}^{\AJ} (f^H_{\infty})  = (-1) ^{q(G_{\infty})} \lprod{s_{\psi} s , \lambda_{\pi_{ \infty}}} \lprod{s_{\psi} s , \pi_{ \infty}}_{\AJT} ~ a^G_{H,\infty}. 
\end{align*}
	Here \begin{itemize}
	\item $\pi_{\infty} $ is any element of $\Pi ^{\AJ} _{\psi_{\infty} } (G_{\RR})$. 
	\item $\lprod{\cdot, \lambda_{ \pi _{\infty}}}$ is a character on $S_{\psi_{\infty}} = \pi_0 ( S_{\psi_{\infty}}) $ defined on p.~195 of \cite{kottwitzannarbor}.
	\item  The pairing $\lprod{s_{\psi} s , \pi_{ \infty}}_{\AJT}$ is as in Proposition \ref{prop:summary in arch}, defined with respect to  $(\mathfrak w_{\infty}, \Xi_{\infty}, z_{\infty})$.
	\item  The product $\lprod{s_{\psi} s , \lambda_{\pi_{ \infty}}} \lprod{s_{\psi} s , \pi_{ \infty}}_{\AJT} $ is independent of the choice of $\pi_{ \infty}$. 
\end{itemize}

\item  We have 
\begin{align*} 
	\Lambda_{\psi' _{\infty}} (f^H_{\infty})  = (-1) ^{q(G_{\infty})} \lprod{s_{\psi} s , \lambda_{\pi_{ \infty}}} \lprod{s_{\psi} s , \pi_{ \infty}}_{\AJT}~ a^G_{H,\infty} 
\end{align*}

\item For each finite place $v$, $\psi_v$ is unramified if and only if $G^*_v$ is unramified and $\psi_v'$ is unramified. 

\item If $\Lambda_{\psi_p'} (f_p^H) \neq 0$, then   $\psi_p$ is unramified. Conversely, assume that   $\psi_p$ is unramified. Then inside $\widetilde \Pi_{\psi_p}(G_p)$, there is a unique element  $\pi_p$ that is a $\vartheta_p$-orbit of $K_p$-unramified representations of $G(\QQ_p)$. Each $\dot \pi_p \in \pi_p$ satisfies $\dim (\dot \pi_p)^{K_p} = 1$. We have 
\begin{align}\label{eq:spectral eval at p}
	\Lambda_{\psi' _p} (f^H_p)=  
		\lprod{s_{\psi} s, \pi_p} p ^{a n/2} \Tr ( s \varphi _{\psi_p} (\Frob_p ^a)  \mid  \stan_{G})~ a ^{G} _{H,p}.
\end{align} 
Here 
\begin{itemize}
	\item $n =d-2$ is the dimension of the Shimura variety; see \S \ref{subsec:datum}. 
	\item $\stan_G = \stan_{G^*} $ is the standard representation (\ref{eq:standard rep}) of  
	$\lang G = \lang G^*$. 
	\item $\Frob_p$ denotes any choice of a lift of the geometric Frobenius in $W_{\QQ_p}$. 
\end{itemize}
\item We have  \begin{align*} 
	\Lambda _{\psi'}^{ p,\infty} (f^{H,p,\infty}) =  \sum_{ \pi ^{p,\infty} = (\pi_v)_v \in \widetilde \Pi_{\psi}^{p,\infty}(G)  }   \Tr \big(\dot \pi^{p,\infty} (f^{p,\infty} dg^{p,\infty}) \big  )  \prod _{v \neq p,\infty} \lprod{s_{\psi} s , \pi _v} a^{G}_{H,v},
\end{align*}where $f^{p,\infty} dg^{p,\infty}$ is determined by $f^{\infty} dg^{\infty}$ in the same manner as in \S \ref{subsubsec:setting for Morel's formula}.  
 	\end{enumerate}

 \end{lem}
\begin{proof}
\textbf{(1)} 	
This follows from \cite[Lem.~9.2]{kottwitzannarbor}.	More precisely, we know (Remark \ref{rem:comparison with AJ}) that $\Lambda ^{\AJ} _{\psi'_{\infty}}$ is the stable distribution considered by Adams--Johnson \cite{adamsjohnson} and Kottwitz \cite[\S 9]{kottwitzannarbor}, and the latter is Kottwitz's \emph{definition} of \cite[(9.4)]{kottwitzannarbor}. We know from Proposition \ref{prop:summary in arch} (3) that $\lprod{s_{\psi} s, \pi_{ \infty}}_{\AJT}$ serves 
	as the spectral transfer factor that is denoted by $\Delta _{\infty} (\psi_H , \pi)$ (for $\psi_H = \psi'$, $\pi = \pi_{\infty}$) in \cite[Lem.~9.2]{kottwitzannarbor}, up to the correction factor  $a _{H,\infty} ^G$. Here  $a _{H,\infty} ^G$ arises because the spectral transfer factors used in \textit{loc.~cit.~}are assumed to be compatible with the normalization $(\Delta^G_H)_{\infty} = a^{G} _{H,\infty} \Delta_{H_\RR} ^{G_\infty}  (\mathfrak w_\infty, \Xi_\infty ,z_\infty)$ of the geometric transfer factors, whereas the endoscopic character relation (\ref{eq:AJ character relation}) is with respect to the normalization $ \Delta_{H_\RR} ^{G_\infty}  (\mathfrak w_{\infty}, \Xi_{\infty} ,z_{\infty}).$ Note that in the even symmetric case, the fact that $f^H_{\infty}$ is a sum $f^H_{\infty,0}+ f^H_{\infty,1}$, where $f^H_{\infty, i}$ corresponds to $\mathbb V_i$, does not affect the validity of \cite[Lem.~9.2]{kottwitzannarbor}. This is because the infinitesimal characters of $\mathbb V_{0,\CC}$ and $\mathbb V_{1,
	\CC}$ are unequal, and in the evaluation $\Lambda_{\psi'_{\infty}}^{\AJ}(f^H_{\infty})$ only one of $f^H_{\infty, i}$ will contribute, according to whether the infinitesimal character of $\eta \circ \psi'_\infty$ is equal to that of $\mathbb V_{0,\CC}^*$ or $\mathbb V_{1,
	\CC}^*$. 
	
	\textbf{(2)} This follows from part (1) together with Proposition \ref{prop:summary in arch} (2) and the fact that $f^H_{\infty} \in \widetilde { \mathcal H}^{\st}(H_{\RR})$ shown in the proof of Lemma \ref{lem:invariance}. 
	
	\textbf{(3)}  Write $I_v$ for the inertia subgroup of $W_{\QQ_v}$. For each $\tau \in I_v$, write $\psi'_v(\tau) = a_\tau \rtimes \tau$, with $a_{\tau} \in \widehat H$. 
	
	Assume that $\psi_v$ is unramified. Then by definition $G^*_v$ is unramified. It also immediately follows that $\psi_v'$ is trivial on $\SU_2(\RR)$, and $\eta(\tau) = \eta(a_{\tau}^{-1}) \rtimes \tau$ for all $\tau \in I_v$.  Since $\tau$ acts trivially on $\widehat {G^*}$, for all $x \in \widehat H$ we have $\eta(\lix^{\tau} x) = \eta(\tau) \eta(x) \eta(\tau)^{-1} = \eta(a_{\tau}^{-1} x a_{\tau})$. Therefore $I_v$ acts on $\widehat H$ via inner automorphisms, which implies that $H_{\QQ_v}$ is unramified. Then by our explicit presentation we know that the endoscopic datum $(H,\lang H, s, \eta)$ is unramified over $\QQ_v$ (cf.~\S \ref{para:test function intro}), that is, $\eta(\tau) = 1 \rtimes \tau$ for all $\tau \in I_v$. This implies that $a_{\tau} =1$ for all $\tau \in I_v$. Since $H_{\QQ_v}$ is unramified and we have already seen that $\psi'_v$ is trivial on $\SU_2(\RR)$, we know that $\psi'_v$ is unramified.
	
	Conversely, assume that $\psi'_v$ is unramified and $G^*_v$ is unramified. Then $H_{\QQ_v}$ is unramified, and as before the endoscopic datum $(H,\lang H, s, \eta)$ is unramified over $\QQ_v$. Since $\psi_v = \eta \circ \psi'_v$, we know that $\psi_v$ is trivial on $\SU_2(\RR)$ and sends every $\tau \in I_v$ to $1 \rtimes \tau$. Thus $\psi_v$ is unramified since $G^*_v$ is unramified. 
	
	\textbf{(4)} Suppose $\Lambda_{\psi'_p}(f^H_p) \neq 0$. Then $f^H_p\neq 0$, so by the definition of $f^H_p$ we know that $H$ is unramified over $\QQ_p$. Fix a Whittaker datum $\mathfrak w_{H,p} = (\mathfrak w_{H^+_{\QQ_p}} , \mathfrak w_{H^-_{\QQ_p}})$ for $H_{\QQ_p}$, and fix a hyperspecial subgroup $K_{H,p}$ of $H(\QQ_p)$ that is compatible with $\mathfrak w_{H,p}$, as in \S\S \ref{para:variant for prod of two groups} and \ref{para:setting for unramified lemma}. Recall from Definition \ref{defn:f^H_p} and  Remark \ref{rem:realization of Hecke} that $f^H_p$ is well defined as an element of the canonical unramified Hecke algebra $\mathcal H^{\ur}(H_{\QQ_p})$, and its stable orbital integrals are independent of how we realize $f^H_p$ in $C^{\infty}_c(H(\QQ_p))$. Thus we may assume that $f^H_p \in \mathcal H(H(\QQ_p) \sslash K_{H,p})$ without loss of generality. Then by (\ref{eq:stab dist for H}) and Lemma \ref{lem:parameter is unramified for H} (1), we know that $\psi_p'$ is unramified. By part (3) above, this implies that  $\psi_p$ is unramified.
	
	Conversely, assume that $\psi_p$ is unramified. By part (3) above, $\psi_p'$ is unramified (since we know that $G^*_{\QQ_p}$ is unramified), so in particular $H_{\QQ_p}$ is unramified. Fix $\mathfrak w_{H,p}$ and $K_{H,p}$ as in the preceding paragraph. Inside $G^*(\QQ_p)$, we have the hyperspecial subgroup $\mathcal G^*(\ZZ_p)$, and it is compatible with the Whittaker datum $\mathfrak w_p$ for $G^*_p$ since $p \notin \Sigma(\mathcal G^*, \mathcal G, \Xi, z, \mathfrak w , \theta_G)$; see \S \ref{para:trivial pure}.  We normalize the Haar measures on $G^*(\QQ_p)$ and $H(\QQ_p)$ once and for all such that hyperspecial subgroups have volume $1$. By Lemmas \ref{lem:parameter is unramified} and \ref{lem:parameter is unramified for H}, we know that inside $\widetilde {\Pi}_{\psi_p}(G^*_p)$ (resp.~$\widetilde {\Pi}_{\psi_p'} (H_{\QQ_p})$) there is a unique element $\pi_{p,G^*}$ (resp.~$\pi_{p,H}$) whose members are $\mathcal G^*(\ZZ_p)$-unramified (resp.~$K_{H,p}$-unramified), and moreover the members of $\pi_{p,G^*}$ (resp.~$\pi_{p,H}$) have $1$-dimensional fixed spaces under $\mathcal G^*(\ZZ_p)$ (resp.~$K_{H,p}$). As in the preceding paragraph we may assume that $f^H_p \in \mathcal H(H(\QQ_p) \sslash K_{H,p})$. Then by (\ref{eq:stab dist for H}), we have 
	$$ \Lambda_{\psi'_p} (f_p^H) = \lprod{s_{\psi'_p} , \pi_{p,H}} \Tr(\pi_{p,H} (f^H_p)) .$$ Here the pairing $\lprod{s_{\psi'_p} , \pi_{p,H}}$ is defined with respect to $\mathfrak w_{H,p}$. 
	In view of the compatibility between local unramified Arthur parameters and unramified Langlands parameters in Lemma \ref{lem:parameter is unramified for H} (2), the same argument as Kottwitz's proof that \cite[(9.3)]{kottwitzannarbor} is equal to \cite[(9.7)]{kottwitzannarbor} gives 
	\begin{align*}
	 \Tr(\pi_{p,H} (f^H_p))  =  
	  p ^{a n/2} \Tr ( s \varphi _{\psi_p} (\Frob_p ^a)  \mid  \stan_{G} ) .
	\end{align*}  Indeed, one easily checks that the irreducible representation of $\lang G$ determined by the Shimura datum appearing in \cite[(9.7)]{kottwitzannarbor} is $\stan_G$, and that the ambiguity in $\varphi_{\psi_p}$ up to the $\Aut(\lang G^*_p)$-action disappears when we consider the $\GL_N(\CC)$-conjugacy class of the composition of $\varphi_{\psi_p}$ with $\stan_G: \lang G = \lang G^* \to \GL_N(\CC)$. In conclusion we have 
$$ \Lambda_{\psi'_p} (f_p^H) = \lprod{s_{\psi'_p} , \pi_{p,H}}    p ^{a n/2} \Tr ( s \varphi _{\psi_p} (\Frob_p ^a)  \mid  \stan_{G} ) . $$

  As we have already mentioned in \S \ref{para:global packets}, since $p \notin \Sigma(\mathcal G^*, \mathcal G, \Xi, z, \mathfrak w , \theta_G)$, the packet $\widetilde \Pi_{\psi_p}(G_p)$ together with the map to $\mathcal S_{\psi_p}^D$ is constructed from $\widetilde \Pi_{\psi_p}(G_p^*)$ by identifying $G_p$ with $G^*_p$ via an isomorphism $G^*_p \isom G_p$ as in condition (4) in \S \ref{para:trivial pure}. Hence the existence and uniqueness of $\pi_p$ and the fact that members of $\pi_p$ have $1$-dimensional fixed spaces under $K_p =\mathcal G(\ZZ_p)$ follow from the existence and uniqueness of $\pi_{p,G^*}$ and the fact that members of $\pi_{p,G^*}$ have $1$-dimensional fixed spaces under $\mathcal G^*(\ZZ_p)$. Also we have $\lprod{\cdot, \pi_p} = \lprod{\cdot,\pi_{p,G^*}} \in \mathcal S_{\psi_p}^D$.  To finish the proof it suffices to show that\footnote{By Lemmas \ref{lem:parameter is unramified} and \ref{lem:parameter is unramified for H} we know that $\lprod{\cdot, \pi_{p,H}}$ and $\lprod{\cdot, \pi_{p, G^*}}$ are trivial, but in the current proof we do not need this.}
	\begin{align}\label{eq:a_p}
\lprod{s_{\psi'_p}, \pi_{p,H} }  =  \lprod{s_{\psi} s, \pi_{p,G^*}}  a^G_{H,p}.
	\end{align}

Comparing the Fundamental Lemma (Theorem \ref{thm:LS} (2)) with the endoscopic character relation (\ref{eq:endoscopic character relation}) and the expansion (\ref{eq:stab dist for H}), we get 
\begin{align}\label{eq:plug in 1_K}
\sum _{\xi \in \widetilde \Pi _{\psi_p} ({G^*_p})} \lprod{ s_{\psi}  s, \xi} \Tr (\xi ( 1_{\mathcal G^*(\ZZ_p)}) ) = (a_{H,p}^G )^{-1}\sum _{\xi' \in \widetilde \Pi_{\psi_p'} (H_{\QQ_p})} \lprod{s_{\psi'_p}, \xi'} \Tr(\xi' ( 1_{K_{H,p}}))  . 
\end{align} 
 Here, the pairing $\lprod{s_{\psi'_p}, \xi'}$ is defined with respect to $\mathfrak w_{H,p}$, and the factor $(a_{H,p}^G )^{-1}$ appears because it is $(a_{H,p}^G )^{-1} 1_{K_{H,p}}$, rather than $1_{K_{H,p}}$, that is a Langlands--Shelstad transfer of $1_{\mathcal G^*(\ZZ_p)}$ with respect to the Whittaker normalization of transfer factors between $H_{\QQ_p}$ and $G^*_p$ associated to $\mathfrak w_p$. 
By Lemmas \ref{lem:parameter is unramified} and \ref{lem:parameter is unramified for H}, the two sides of (\ref{eq:plug in 1_K}) are equal to $\lprod{s_{\psi} s, \pi_{p,G^*}}$ and $(a_{H,p}^G )^{-1} \lprod{s_{\psi'_p}, \pi_{p,H}}$ respectively. This proves (\ref{eq:a_p}). 

	\textbf{(5)} First observe that for each $\pi^{p,\infty} \in \widetilde \Pi^{p,\infty}_{\psi}(G)$, the ambiguity in the $G(\adele_f^p)$-representation $\dot \pi^{p,\infty}$ up to the $\vartheta^{p,\infty}$-action does not affect the value of $ \Tr \big( \dot \pi^{p,\infty} (f^{p,\infty} dg^{p,\infty}) \big )$. Indeed, since $\theta_G$ is an automorphism of $G$ of order at most $2$, it is clear that $dg^{p,\infty}$ is fixed by $\vartheta^{p,\infty}$. In the proof of Lemma \ref{lem:invariance} we observed that $f^{p,\infty}$ is fixed by $\vartheta^{p,\infty}$ (under the overall assumption that $f^{\infty}$ is fixed by $\vartheta^{\infty}$). Hence the trace of $f^{p,\infty} dg^{p,\infty}$ on a $G(\adele_f^p)$-representation depends only on the $\vartheta^{p,\infty}$-orbit of the isomorphism class of that representation.

	Now as in the proof of Lemma \ref{lem:invariance}, we may assume that $f^{p,\infty} = \prod_{v \neq p,\infty} f_v$ with each $f_v \in \mathcal H(G_v)$ being fixed by $\vartheta_v$. The desired statement then follows from the endoscopic character relation in Proposition \ref{prop:Taibi finite} applied to each $f_v$. Here the term $a^G_{H,v}$ appears because it is $(a^G_{H,v})^{-1} f^{H}_v$, rather than $f^H_v$, that is a Langlands--Shelstad transfer of $f_v$ with respect to the normalization $\Delta^{G_v}_{H_{\QQ_v}} (\mathfrak w_v,  \Xi_v, z_v)$ of transfer factors.  
	\end{proof} 
\ignore{
\begin{rem}
	Let $\pi_p$ be as in the above lemma.  We point out that by \cite[Lem.~4.1.1]{taibidim}, we have \begin{align}\label{eq:unram prod = 1}
	\lprod{\cdot, \pi_p} = 1.
	\end{align}
	In fact, this pairing, after identifying $\widetilde{\Pi} _{\psi_p} (G_p)$ with $\widetilde{\Pi} _{\psi_p} (G_p^*)$, is none other than the pairing (\ref{eq:Arthur's pairing}) defined with respect to $\mathfrak w$. Under this identification $\pi_p$ corresponds to the unique $\mathcal G^*(\ZZ_p)$-unramified element of $\widetilde{\Pi} _{\psi_p} (G_p^*)$. To apply \cite[Lem.~4.1.1]{taibidim}, we need to check that $\mathfrak w$ is compatible with $\mathcal G^* (\ZZ_p)$. But this is indeed the case by property (1) above Definition \ref{defn:symm Hecke for G}, since $p \notin \Sigma_{\bad}$.
\end{rem}}
\ignore{
\begin{rem}
	Since the representation $\stan_G$ is self-dual and $s^2 =1$, there is no need to distinguish between $\Frob_p$ and $\Frob_p ^{-1}$ in (\ref{eq:spectral eval at p}). 
\end{rem}}

We summarize the results we have obtained so far in the following proposition.  
\begin{prop}\label{prop:contribution}
Let $(H,\lang H, s,\eta) \in \dot {\mathscr E}(G)$ and  $\psi' \in \widetilde{\Psi} (H)$.  For each place $v$ of $\QQ$, let $\psi'_v \in \Psi^+_{\uni}(H_{\QQ_v})$ be a localization of $\psi'$, and let $\psi_v:= \eta\circ \psi'_v \in \Psi^+_{\uni}(G^*_v)$. Let $\psi = \eta \circ \psi' \in \widetilde \Psi(G^*)$. The following statements hold:
\begin{enumerate}
	\item  For $\psi'$ to contribute non-trivially to the RHS of (\ref{eq:first expansion}), it is necessary that $H$ is cuspidal,  and that $\psi_{\infty}$ is Adams--Johnson with  infinitesimal character determined by $\mathbb V$ as in Lemma \ref{lem:parameter is AJ}. 
	\item Assume that the necessary conditions in (1) are satisfied. Then the contribution of $\psi'$ to the RHS of (\ref{eq:first expansion}), without the factor $\iota(G,H)$, is equal to 
	\begin{multline} 
		m_{\psi' }  \abs{\mathcal S_{\psi'}} ^{-1}\epsilon_{\psi'} (s_{\psi'})  \lprod{s_{\psi} s , \lambda_{\pi_{ \infty}}} \lprod{s_{\psi} s , \pi_{ \infty}} _{\AJT} ~ A(\psi, s, p ,a ) \\  \sum_{\pi^{\infty} = (\pi_v)_v \in \widetilde \Pi^{\infty}_{\psi}(G) }   \Tr \big( \dot \pi^{\infty} (f^{\infty} dg^{\infty}) \big)  \prod _{v \neq \infty} \lprod{s_{\psi} s , \pi _v} . \end{multline}  
with notations explained below: \begin{itemize}
		\item The product   $\lprod{s_{\psi} s , \lambda_{\pi_{ \infty}}} \lprod{s_{\psi} s , \pi_{ \infty}}_{\AJT} $
		is as in Lemma \ref{lem:Kott destab} (1).
		\item We define \index{$A(\psi,s,p,a)$} $$ A(\psi, s, p,a): =  (-1) ^{q(G_{\infty})}  p ^{a n/2} \Tr ( s \varphi _{\psi_p} (\Frob_p ^a)  \mid  \stan_{G}) . $$ The notations $n$, $\Frob_p$, and $\stan_G$ are as in Lemma \ref{lem:Kott destab} (3),  and we have $q(G_{\infty}) = n$. 
	\end{itemize}
\end{enumerate} 
\end{prop}
\begin{proof}  This follows from Lemma \ref{lem:parameter is AJ}, Lemma \ref{lem:Kott destab}, (\ref{eq:prod formula for a}),  (\ref{eq:sigma=1}), and the following simple observations: 
	\begin{enumerate}
		\item 	For any finite-length smooth representation $\tau_p$ of $G(\QQ_p)$, we have  
		$$ \Tr (\tau_p  (1_{K_p} dg_p)) = \dim \tau_p^{K_p} .$$ 
	Here, as in \S \ref{subsubsec:setting for Morel's formula}, $dg_p$ is the Haar measure on $G(\QQ_p)$ giving volume $1$ to hyperspecial subgroups. 
	\item If $\psi_p$ is ramified, then no element of $\widetilde {\Pi}_{\psi_p} (G_p)$ is a $\vartheta_p$-orbit of $K_p$-unramified representations. Indeed,  as we have mentioned in the proof of Lemma \ref{lem:Kott destab} (4), the packet $\widetilde {\Pi}_{\psi_p} (G_p)$ is constructed from the packet $\widetilde {\Pi}_{\psi_p} (G^*_p)$ via an isomorphism $  G^*_p \isom G_p$ as in condition (4) in \S \ref{para:trivial pure}. Hence the current assertion follows from Lemma \ref{lem:parameter is unramified} applied to  $G^*_p$ and $\psi_p$. 
	\end{enumerate}
	\end{proof}

	\section{Spectral expansion of the intersection cohomology} We keep the same setting and notation as in \S \ref{subsec:spectral eval}. In particular,  $\mathbb V$ is as in \S \ref{para:two V}, and we speak of the odd case, the even symmetric case, and the even composite case. 
	\begin{defn}\label{defn:Psi_V}
	We denote by $\widetilde{  \Psi} (G^*)_{\mathbb V}$ the set of $\psi \in \widetilde{\Psi} (G^*)$ such that the localization $\psi _{\infty} $ of $\psi$ at $\infty$ lies in $\Psi ^{\AJ} (G_{\infty}^*)$ and has the same infinitesimal character as that of $\mathbb V^*_{\CC}$ in the odd case and the even symmetric case, and the same infinitesimal character as that of $\mathbb V_{0,\CC} ^*$  or $\mathbb V_{1,\CC}^*$ in the even composite case. (This condition is insensitive to the ambiguity in $\psi_{\infty}$ up to the $\Aut(\lang G^*_{\infty})$-action.) 
In particular, for any $\psi  \in \widetilde{  \Psi} (G^*)_{\mathbb V}$, we have $\psi \in \widetilde{\Psi} _2 (G^*)$, and $S_{\psi} = \pi_0 (S_{\psi})$ is a finite abelian group.
	\end{defn}
		\begin{defn}\label{defn:compleltely symm}
			We say that a compact open subgroup $K \subset G(\adele_f)$ is \emph{completely symmetric},\index[n]{completely symmetric compact open subgroup} if $K = \prod_{v} K_v$ where $v$ runs through all primes, with each $K_v$ a compact open subgroup of $G(\QQ_v)$ that is stable under $\theta_G$. 
		\end{defn}
	\begin{rem}\label{rem:completely sym}
		Completely symmetric compact open subgroups of $G(\adele_f)$ form a cofinal system of compact open subgroups. Indeed, given any compact open subgroup $W$ of $G(\adele_f)$, we know that $W$ contains a compact open subgroup of the form $\prod_{v\in S} U_v \times \prod_{v\notin S} \mathcal G(\ZZ_v)$, where $S$ is a sufficiently large finite set of primes, $\mathcal G$ is as in \S \ref{para:trivial pure}, and $U_v$ is a compact open subgroup of $G(\QQ_v)$ for each $v \in S$. If $S$ is sufficiently large, we know that $\mathcal G(\ZZ_v)$ is $\theta_G$-stable for all $v \notin S$; see \S \ref{para:trivial pure}. Note that $U'_v : = U_v \cap \theta_G(U_v)$ is a $\theta_G$-stable compact open subgroup of $G(\QQ_v)$ for each $v \in S$. Hence $W$ contains the completely symmetric compact open subgroup $\prod_{v\in S} U'_v  \times \prod_{v \notin S}\mathcal G(\ZZ_v)$.    
	\end{rem}
\begin{thm}\label{thm:final} Assume Hypothesis \ref{hypo}. Fix a neat compact open subgroup $K$ of $G(\adele_f)$, and fix $f^{\infty} dg^{\infty} \in \mathcal H(G(\adele_f)\sslash K)_{\QQ}$. Assume that $K$ is completely symmetric, and that $f^{\infty}$ is fixed by $\vartheta^{\infty}$. Let $p$ be a prime not in the set  $\Sigma_{\bad}(K,f^{\infty})$ as in (\ref{eq:Sigma bad}). Let $a\in \ZZ$ be arbitrary. We have 
	\begin{multline}\label{eq:semifinal}
\Tr(\Frob_p^a \times f^{\infty} dg^{\infty} \mid  \icoh^* (\overline{ \Sh _K},\mathbb V) )  \\  =  \sum_{\psi\in  \widetilde{\Psi} (G^*)_{\mathbb V} }  m_{\psi}  \sum _{\pi^{\infty}\in \widetilde \Pi_{\psi}^{\infty}(G) } \Tr \big( \dot \pi^{\infty} (f^{\infty} dg^{\infty}) \big)  \abs{S_{\psi} } ^{-1} \sum _{ s \in S_{\psi}} B(  \psi , s , \pi^{\infty} , p, a),
	\end{multline} with notations explained below:
\begin{itemize}
	\item For each $\psi$, $m_{\psi}\in \set{1,2}$ is as in (\ref{eq:defn of m psi}).
	\item For each $\psi \in \widetilde \Psi(G^*)_{\mathbb V}$, $\pi^{\infty} = (\pi_v)_v \in \widetilde \Pi_{\psi}^{\infty}(G)$, and $s \in S_{\psi}$, we define \index{$B(\psi,s,\pi^{\infty},p,a)$} $$B( \psi , s, \pi^{\infty}, p , a) : = \epsilon_{\psi } ( s)  \lprod{ s, \lambda_{ \pi _{\infty}}} \lprod {s, \pi _{\infty}}  _{\AJT} ~A(\psi, s_{\psi} s , p,a) \prod_{v\neq \infty} \lprod{s, \pi_v},   $$ where the terms $\lprod{ s, \lambda_{ \pi _{\infty}}} \lprod {s, \pi _{\infty}}  _{\AJT}$ and $A(\psi, s_{\psi} s , p,a)$ are as in Proposition \ref{prop:contribution}, but with a change of variable from $s$ to $s_{\psi} s$. (Recall that $s_{\psi} \in S_{\psi}$ and $s_{\psi}^2 =1$.)
\end{itemize} 
Moreover, the summands on the right hand side of (\ref{eq:semifinal}) vanish outside a finite set of summation indices $(\psi, \pi ^{\infty})$ which depends only on $K,\mathbb V, \mathfrak w, \Xi, z$ and not on $f^{\infty}dg^{\infty}, p, a$. 
\end{thm}

\begin{proof}  Throughout the proof, it will always be understood that the data $K , \mathbb V, \mathfrak w, \Xi, z$ are fixed. Also, since varying $f^{\infty} dg^{\infty}$ is equivalent to varying $f^{\infty}$ while keeping $dg^{\infty}$ fixed, we will omit $dg^{\infty}$ in the notations throughout. We first prove that when $ f^{\infty}$ and $p$ are fixed, (\ref{eq:semifinal}) holds for all sufficiently large $a$ (in a way depending on $f^{\infty}$ and $p$). By (\ref{eq:first expansion}) and Proposition \ref{prop:contribution}, we know that when $a$ is sufficiently large, the LHS of (\ref{eq:semifinal}) is equal to 	$$\sum_{\psi \in \widetilde {\Psi}(G^*)_{\mathbb V}} \sum_{s \in S_{\psi}}  C(\psi, s)   \sum _{\pi^{\infty}\in \widetilde \Pi_{\psi}^{\infty}(G) } \Tr \big( \dot \pi^{\infty} (f^{\infty}) \big)   B(\psi, s_{\psi} s  ,   \pi^{\infty}, p , a) ,$$
if we define $$C(\psi, s) :  = \sum_{\substack{\mathfrak e= (H,\lang H, s_1, \eta) \in \dot{\mathscr E}(G) \\ H \text{ is cuspidal}} } \sum_{\substack{ \psi'\in \widetilde {\Psi}(H) \\  (\mathfrak e, \psi') \mapsto (\psi, s)}} \iota(G,H) m_{\psi'} \abs{\mathcal S_{\psi'}}^{-1} \epsilon_{\psi'}(s_{\psi'}) \epsilon_{\psi}(s_{\psi} s )^{-1} , $$ where the second summation is over $ \psi'\in \widetilde {\Psi}(H)$ such that $(H,\lang H, s_1, \eta, \psi')$ gives rise to $(\psi,s)$ as on p.~36 of \cite{arthurbook}. Now in the definition of $C(\psi, s)$  we can   drop the condition that $H$ is cuspidal in the first summation, for the following reason. If there exists $\psi' \in \widetilde {\Psi}(H)$ such that $(H,\lang H, s_1, \eta, \psi')$ gives rise to $(\psi,s)$, then by the argument in the last paragraph of \cite[p.~196]{kottwitzannarbor},  elliptic maximal tori in $G^*_{\RR}$ (which are anisotropic) must come from $H_{\RR}$ since $\psi_{\infty}$ is Adams--Johnson. It follows that $H_{\RR}$ contains anisotropic maximal tori, and hence $H$ is cuspidal.  
 
	Thus the proof of (\ref{eq:semifinal}) reduces to the proof of the identity $$m_{\psi}\abs{S_{\psi}} ^{-1}  = \sum_{ \mathfrak e= (H,\lang H, s_1, \eta) \in \dot{\mathscr E}(G) } \sum_{\substack{ \psi'\in \widetilde {\Psi}(H) \\  (\mathfrak e, \psi') \mapsto (\psi, s)}} \iota(G,H) m_{\psi'} \abs{\mathcal S_{\psi'}}^{-1} \epsilon_{\psi'}(s_{\psi'}) \epsilon_{\psi}(s_{\psi} s )^{-1} $$ for all $\psi \in \widetilde {\Psi}(G^*)_{\mathbb V}$ and $s \in S_{\psi}$. This step is identical to the corresponding step in the proof of \cite[Thm.~4.0.1]{taibi}. Without the extra complication in the even case (i.e., the integers $m_{\psi'}, m_{\psi}$ being possibly larger than $1$), this step is also given in \cite[\S 10]{kottwitzannarbor}. Both references rely on Arthur's identity $\epsilon _{\psi'} (s_{\psi'}) = \epsilon_{\psi} (s_{\psi}s)$, which is known in our case by \cite[Lem.~4.4.1]{arthurbook}.

Before showing that (\ref{eq:semifinal}) holds for all $a\in \ZZ$, we show that the summands on the right hand side of it vanish outside a finite set of summation indices $(\psi, \pi^{\infty})$ independently of $ f^{\infty} , p,a$. Here $f^{\infty} $ is allowed to range over all $\vartheta^{\infty}$-fixed elements of $\mathcal H(G(\adele_f)\sslash K)_{\QQ}$, $p$ is allowed to range over all primes that are hyperspecial for $K$ and unramified for $f^{\infty}$, and $a$ is allowed to range over all positive integers, not necessarily ``sufficiently large'' with respect to $f^{\infty}$ and $p$ in the previous sense. (Afterwards we will show the stronger finiteness result when $a$ is allowed to range over all integers.)

Since $K$ is completely symmetric, we have $K = \prod_v K_v$ with each $K_v$ a $\theta_G$-stable compact open subgroup of $G(\QQ_v)$. Let $\Sigma_0$ be a finite set of primes containing the set  $\Sigma(\mathcal G^*, \mathcal G, \Xi, z, \mathfrak w , \theta_G)$ from \S \ref{para:trivial pure} such that $ K_v = \mathcal G(\ZZ_v)$ for all $v \notin \Sigma_0$. Now since $f^{\infty} $ is bi-invariant under $K$, any $\pi^\infty =(\pi_v)_v$ appearing in (\ref{eq:semifinal}) such that $\Tr \big( \dot \pi^{\infty} (f^{\infty} dg^{\infty}) \big) \neq 0$ must satisfy the condition that $\pi_v$ is a $\vartheta_v$-orbit of $\mathcal G(\ZZ_v)$-unramified representations for all primes $v \notin \Sigma_0$. By the discussion in \S \ref{para:global packets}, we know that for each $\psi \in \widetilde \Psi(G^*)_{\mathbb V}$, there are only finitely many elements $\pi^{\infty} \in \widetilde \Pi_{\psi}^{\infty}(G)$ satisfying the aforementioned condition, and these elements exist only when the localizations $\psi_v$ of $\psi$ are unramified for all primes $v \notin \Sigma_0$. By our above proof of (\ref{eq:semifinal}), the desired finiteness of the summation range independently of $f^{\infty}, p, a $ follows from the following statements:
\begin{enumerate}
	\item If $(H,\lang H, s,\eta) \in \dot{\mathscr E}(G)$ and $\psi' \in \widetilde{\Psi} (H)$ are such that $\eta \circ \psi'_v$ is unramified for a prime $v  $, then $\psi'_v$ is unramified, and in particular $H_{\QQ_v}$ is unramified.
\item There are only finitely many elements $(H,\lang H, s,\eta) \in \dot{\mathscr E}(G)$ such that  $H_{\QQ_v}$ is unramified for all primes $v \notin \Sigma _{0}$. 
\item Fix $(H,\lang H, s,\eta) \in \dot{\mathscr E}(G)$. For each choice of $(f^{\infty}, p,a)$ (with 
$a\in \ZZ_{\geq 1}$), define $f^H = f^H_{(f^{\infty}, p,a)} \in C^{\infty}_c(H(\adele))$ as in \S \ref{subsec:test functions} (cf.~Lemma \ref{lem:invariance}). Then in the expansion (\ref{eq:spectral expansion}) with respect to the test function $f^H$, the summands vanish outside a finite subset of $\widetilde \Psi(H)$ which is independent of  $(f^{\infty}, p,a)$.  
\end{enumerate}
Now statement (1) follows from Lemma \ref{lem:Kott destab} (3). By the explicit presentation of $(H,\lang H, s, \eta)$ and by Proposition \ref{prop: even TFAE},  statement (2) reduces to the fact that there are only finitely many elements in $\QQ^{\times}/ \QQ^{\times ,2}$ that have even valuations at all primes not in $\Sigma_{0}$. 
For (3), we may assume that $H$ is cuspidal, as otherwise $f^H=0$. Now note that for our given test function $f^H_{\infty}$ on $H(\RR)$, there are only finitely many values of $t\geq 0$ (depending only on $\mathbb V$) that contribute non-trivially to the expansion $S^H_{\disc} (f^H)= \sum_{t \geq 0} S^H_{\disc, t} (f^H)$ (see \S \ref{para:spec exp}) by Lemma \ref{lem:parameter is AJ}. Thus it suffices to show that for a fixed $t$ the summands in (\ref{eq:one t}) with respect to $f^H$ vanish outside a finite subset of $\widetilde \Psi(H)$ independently of $(f^{\infty}, p , a )$. By \cite[Thm.~1.3.2, Lem.~3.3.1]{arthurbook}, we need only check that $f^H$ has a \emph{Hecke type}\index[n]{Hecke type} (see \cite[p.~129]{arthurbook}) that is independent of $(f^{\infty}, p,a)$. Since $f^H_{\infty}$ is independent of $(f^{\infty}, p,a)$, this amounts to the existence of a compact open subgroup $K_H \subset H(\adele_f)$ such that $f^{H,\infty} = f^{H,p,\infty} f^H_p$ can be chosen to be bi-invariant under $K_H$ independently of $(f^{\infty}, p ,a )$. 

We now construct $K_H$. Let $S$ be the set of primes $v$ such that either $G_{\QQ_v}$ is ramified or $H_{\QQ_v}$ is ramified. For each prime $v \notin S$, we pick a hyperspecial subgroup $U_v \subset H(\QQ_v)$, in such a way that $\prod_{v\notin S} U_v$ is a compact open subgroup of $H(\adele_f^S)$. By the two main theorems of \cite[\S 6]{arthuronlocal} (cf.~the proof of \cite[Lem.~3.3.1]{arthurbook}), we know that for every prime $v$ there is a compact open subgroup $V_v \subset H(\QQ_v)$ with the property that every $K_v$-bi-invariant function in $C^{\infty}_c(G(\QQ_v))$ has a Langlands--Shelstad transfer in $C^{\infty}_c(H(\QQ_v))$ that is bi-invariant under $V_v$. By the Fundamental Lemma for the full unramified Hecke algebra proved by Hales \cite{halesFL} (which is conditional on the Fundamental Lemma for the unit as recalled in Theorem \ref{thm:LS}),  for every prime $v\notin S$ we may and shall take $V_v$ to be $U_v$. We take $K_H$ to be the product of $V_v$ over all primes, which is a compact open subgroup of $H(\adele_f)$. Now for every choice of $(f^{\infty}, p , a)$, the corresponding function $f^{H,\infty}$ is non-zero only when $p \notin S$, and in the latter case we can choose $f^{H,p,\infty}$ to be bi-invariant under $\prod_{v \neq p} V_v$, and choose $f^H_{p}$ to be bi-invariant under $U_p = V_p$, as is clear from the construction in \S \ref{subsec:test functions}. It follows that $f^{H,\infty}$ is bi-invariant under $K_H$ as desired. 

We have proved that the summands on the RHS of (\ref{eq:semifinal}) vanish outside a finite set of summation indices $(\psi, \pi^{\infty})$ independently of $f^{\infty}, p , a\in \ZZ_{\geq 1}$. Note that the same holds even if $a$ is allowed to range over all integers. This is because each summand, as a function in $a \in \ZZ$, is of the form $  \sum _{i=1}^k  c_i z_i^a$, where $c_i, z_i\in \CC$ are independent of $a$. Thus if a summand is zero for all $a \in \ZZ_{\geq 1}$, then it is zero for all $a \in \ZZ$. 

To finish the proof it remains to show that (\ref{eq:semifinal}) holds for all $a \in \ZZ$. By what we have already shown, for each fixed $f^{\infty}$ and $p$, the right hand side of  (\ref{eq:semifinal}) is of the form $  \sum _{i=1}^k  c_i z_i^a$, where $c_i, z_i\in \CC$ are independent of $a$. It is easy to see that the left hand side is also of a similar form as a function in $a$. Hence the identity (\ref{eq:semifinal}) holding for all sufficiently large $a $ implies that it holds for all $a \in \ZZ$. 
\end{proof}	
	\begin{rem}
		A form of Theorem \ref{thm:final} is conjectured in \cite[(10.2)]{kottwitzannarbor}. 
	\end{rem}

\section{The Hasse--Weil zeta function} \label{subsec:final}
We deduce an immediate consequence of Theorem \ref{thm:final} concerning the Hasse--Weil zeta function associated to $  \icoh^* (\overline{ \Sh _K},\mathbb V) ) $. 
\begin{defn}\label{defn:WD rep} Let $p$ be a prime number. 
	Let $\mathcal M$ be a finite-dimensional representation over $\CC$ of $\WD_p$. \begin{enumerate}
		\item We view $\mathcal M$ as a Weil--Deligne representation of $W_{\QQ_p}$, and define its local $L$-factor\index[n]{local $L$-factor} at $p$ in the usual way as in \cite{tatebackground}, denoted by 
		$L_p (\mathcal M ,s)$.\index{$L_p(\mathcal M,s)$} In particular, when the representation is unramified (i.e.~trivial on $\SU_2(\RR)$ and on the inertia subgroup), we have $$  L_p ( \mathcal M ,s) : = \big (\exp ( \sum_{ a\geq 1} \Tr (\Frob_p^a  \mid \mathcal M) p^{-as}/a ) \big ) ^{-1} = \det (1 - \Frob _p p^{-s} \mid  \mathcal M) ^{-1}$$  where $\Frob_p$ is any lift of geometric Frobenius in $W_{\QQ_p}$.
		\item For any  real number $\alpha$, we define $\norm{\cdot}^{\alpha} \mathcal M$\index{$\norm{\cdot}^{\alpha} \mathcal M$} to be the twist of $\mathcal M$ by the quasi-character $\norm{\cdot}^{\alpha}$ on $W_{\QQ_p}$. Here the normalization is such that $\norm{\Frob_p} = p^{-1}$. 
		\item For any positive integer $n$, we define  $\mathcal M ^{(n)}$\index{$\mathcal M ^{(n)}$} to be $$ \norm{\cdot}^{(n-1)/2} \mathcal M  \oplus \norm{\cdot}^{(n-3)/2} \mathcal M \oplus \cdots \oplus \norm{\cdot}^{(1-n)/2} \mathcal M .$$
	\end{enumerate} 
\end{defn}
\begin{rem}
	We have
	 $$ L_p (\norm{\cdot}^{\alpha}\mathcal M, s) = L_p(\mathcal M, \alpha + s), $$ and  $$L_p(\mathcal M^{(n)}, s) = L_p(\mathcal M, s + \frac{n-1}{2}) L_p(\mathcal M, s + \frac{n-3}{2}) \cdots L_p(\mathcal M, s + \frac{1-n}{2}). $$ 
\end{rem}
\subsection{}
 \label{defn:mathcal M}
Let $\psi \in \widetilde{\Psi}_2 (G^*)$. Recall from \S \ref{subsubsec:substitutes} that $S_{\psi}$ is a finite power of $\ZZ/2\ZZ$. Let $\nu: S_{\psi} \to \CC^{\times}$ be a character. Let $\mathcal V = \CC^N$ be the vector space used to define $\GL_N(\CC)$. The group $S_{\psi}$ acts on $\mathcal V$ via 
$$ S_{\psi} \subset \lang G^* \xrightarrow{\stan_{G^*}} \GL_N(\CC) .$$ Let $\mathcal V_{\nu} \subset \mathcal V$ be the $\nu$-eigenspace for this action. For each prime number $p$, consider the action of $\WD_p$ on $\mathcal V$ defined by
$$ \WD_p \xrightarrow{\varphi _{\psi_p}}  \lang G ^* \xrightarrow{\stan_{G^*}} \GL_N(\CC). $$
(Note that in the even case, although ${\psi _p}$ is not always well defined up to $\widehat{ {G^*}}$-conjugacy, the above composite map is always well defined up to $\GL_N(\CC)$-conjugacy.)
   This action commutes with the action of $S_{\psi}$, so we have an action of $\WD_p$ on $\mathcal V_{\nu}$. We denote this $\WD_p$-representation on $\mathcal V_{\nu}$ by $\mathcal V_p (\psi , \nu)$\index{$\mathcal V_p(\psi,\nu)$}. Define \index{$\mathcal M_p(\psi,\nu)$} $$\mathcal M_p(\psi,\nu)  : = \norm{\cdot } ^{-n/2}  \mathcal V_p(\psi ,\nu),$$ where $n = d-2$ is the dimension of the Shimura variety $\Sh_K$. The motivation for this twist is to account for the factor $p^{an/2}$ in the definition of $A(\psi,s,p,a)$ in Proposition \ref{prop:contribution}.
   
   We can  classify the $\WD_p$-representations $\mathcal V_p(\psi,\nu)$ and $\mathcal M_p(\psi,\nu)$ more explicitly in terms of the local Langlands correspondence for general linear groups, as follows. 
   Since $\psi \in \widetilde \Pi_2(G^*)$, it is of the form $$\psi = \boxplus_{i\in I} \pi_i[d_i],$$ where each $\pi_i$ is a self-dual cuspidal automorphic representation of $\GL_{N_i}$, $d_i$ are positive integers such that $\sum N_i d_i = N$, and the pairs $(\pi_i,d_i)$ are distinct.  For any irreducible admissible representation $\pi_p$ of a general linear group over $\QQ_p$, we write $\mathcal V(\pi_p)$\index{$\mathcal V(\pi_p)$} for the representation of $\WD_p$ corresponding to $\pi_p$ under the local Langlands correspondence.  By the explicit description of $S_{\psi}$ in  \cite[(1.4.9)]{arthurbook} (the notation $N_i$ in \textit{loc.~cit.~}corresponding to our $N_id_i$), we have the following classification of $\mathcal V_p(\psi, \nu)$.
  
   \begin{enumerate}
   	\item The odd case. We have $S_{\psi} \cong \set{\pm 1}^I$. Set $I_{\nu} = \set{i}$ if $\nu$ is given by the $i$-th projection $\set{\pm 1}^I \to \set{\pm 1}$ for some $i\in I$. Otherwise, set $I_{\nu} = \emptyset$. 
   	\item The even case. Let $I_{\mathrm{odd}}$ be the set of $i\in I$ such that $\widehat {G_{\pi_i}}$ is odd orthogonal (or equivalently, $N_i d_i$ is odd), and let $I_{\mathrm{even}} = I - I_{\mathrm{odd}}$. We have $S_\psi \cong \set{\pm 1}^{I_{\mathrm{even}}} \times \set{\pm 1}^{I_{\mathrm{odd}}, \prime}$, where as usual we write $\set{\pm 1}^{J,\prime}$ for the kernel of the  map $\set{\pm 1}^{J} \to \set{\pm 1}, (z_j)_{j} \mapsto \prod_{j} z_j$ for any finite set $J$.  Suppose $\nu$ is the restriction to $\set{\pm 1}^{I_{\mathrm{even}}} \times \set{\pm 1}^{I_{\mathrm{odd}}, \prime} $ of the $i$-th projection $\set{\pm 1}^I \to \set{\pm 1}$ for some $i\in I$. Then we set $I_{\nu} = \set{i}$ unless $i \in I_{\mathrm
   {odd}}$ and $
   	\abs{I_{\mathrm
   			{odd}}} =2$, in which case we set $I_{\nu} = I_{\mathrm
   		{odd}}$. In all the other cases, set $I_{\nu} = \emptyset$. 
   \end{enumerate}
    Then in both the odd and even cases we have $$\mathcal V_{p}(\psi,\nu) = \bigoplus_{i\in I_{\nu}} \mathcal V(\pi_{i,p})^{(d_i)} $$ for all $p$.

   For any finite set $S$ of prime numbers, we define \index{$L^S(\mathcal M(\psi , \nu) ,s) $}
\begin{align*}
   L^S(\mathcal M(\psi , \nu) ,s)  : = \prod _{p\notin S} L_p (\mathcal M_p (\psi,\nu)  ,s) ,
\end{align*}
where $\mathcal M(\psi ,\nu)$ is just a formal symbol, and the product is over all prime numbers $p \notin S$. By the previous classification, $  L^S(\mathcal M(\psi , \nu) ,s)$ is nothing but a finite product of the $S$-partial standard $L$-functions associated to automorphic representations of general linear groups with some shifting in the variable $s$. Therefore the infinite product defining $L^S(\mathcal M(\psi , \nu) ,s)$ converges absolutely in some right half plane and continues to a meromorphic function in $s$ over the whole $\CC$. 
 Specifically, letting $I_{\nu} \subset I$ be as above, we have 
$$ L^S(\mathcal M(\psi , \nu) ,s)  = \prod_{i\in I_\nu} \prod_{j=0}^{d_i -1} L^S(\pi_i, s- \frac{n}{2} + \frac{d_i - 1}{2} -j).$$  
\subsection{}
Let $\mathbb V$ be as in \S \ref{para:two V}, and fix a neat compact open subgroup $K$ of $G(\adele_f)$ assumed to be  completely symmetric (see Definition \ref{defn:compleltely symm}). In the following we fix an isomorphism $\overline \QQ_{\ell} \cong \CC$. For each prime $p$ unequal to $\ell$ and  unramified for the $\Gamma_{\QQ}$-module $\icoh^* (\overline{ \Sh _K},\mathbb V) $ over $\overline \QQ_{\ell}$ (that is, unramified for each degree $*$), we define \index{$\zeta_p(\icoh^* (\overline{ \Sh _K},\mathbb V) ,s )$}
$$\zeta_p(\icoh^* (\overline{ \Sh _K},\mathbb V) ,s ) : =\prod_{j} \det  (1 - \Frob _p p^{-s} \mid  \icoh^j (\overline{ \Sh _K},\mathbb V)  ) ^{(-1)^{j+1}}, $$ where on the right hand side $\icoh^j (\overline{ \Sh _K},\mathbb V) $ is viewed as a vector space over $\CC$. (The product is finite, since $\icoh^j(\overline {\Sh_K},\mathbb V)$ is non-zero only for $0\leq j \leq 2\dim \Sh_K$.) This is the Euler factor at $p$ of the Hasse--Weil zeta function of $\icoh^* (\overline{ \Sh _K},\mathbb V)$, and it is a rational function in $p^s$. If $S$ is a finite set of primes  containing $\ell$ such that every prime $p$ outside $S$ is unramified for $\icoh^* (\overline{ \Sh _K},\mathbb V) $, then we define the formal Dirichlet series \index{$\zeta^S(\icoh^* (\overline{ \Sh _K},\mathbb V) ,s )$}
$$ \zeta^S(\icoh^* (\overline{ \Sh _K},\mathbb V) ,s )  : = \prod_{p\notin S} \zeta_p(\icoh^* (\overline{ \Sh _K},\mathbb V) ,s ) .$$ This is the $S$-partial  Hasse--Weil zeta function of $\icoh^* (\overline{ \Sh _K},\mathbb V)$. 
\begin{thm}\label{thm:whole zeta}Assume Hypothesis \ref{hypo}.  Let $S$ be the set  $\Sigma_{\bad}(K,1_K)$ as in (\ref{eq:Sigma bad}), applied to $f^{\infty} = 1_K$. For all primes $p \notin S$ we have
	\begin{multline*}
		\log \zeta_p (\icoh^* (\overline{ \Sh _K},\mathbb V)  ,s)  = \sum _{\psi \in \widetilde{\Psi} (G^*)_{\mathbb V}} \sum _{ \pi^{\infty} \in \widetilde{\Pi} _{\psi} ^{\infty} (G)} \dim  (\dot \pi^{\infty}) ^K  \\ \cdot  \sum _ {\nu \in S_{\psi} ^D}  m (\pi^{\infty}, \psi, \nu) (-1) ^{n} \nu (s_{\psi}) \log L_p (\mathcal M _p (\psi ,\nu) , s)
	\end{multline*}
	with notations explained below.
	\begin{itemize}
		\item  The set $\widetilde{\Psi} (G^*) _{\mathbb V}$ is as in Definition \ref{defn:Psi_V}.
		\item The number $m_{\psi} \in \set{1,2}$ is defined in (\ref{eq:defn of m psi}). In the odd case it is always $1$.   
		\item For each $\psi\in \widetilde \Psi(G^*)_{\mathbb V}$, $\pi^{\infty} = (\pi_v)_v \in \widetilde \Pi_{\psi}^{\infty}(G)$, and $\nu \in S^D_{\psi}$, the number $m(\pi^{\infty} , \psi, \nu) \in \set{0,1}$ is defined as follows. Fix an arbitrary $\pi_{\infty} \in \Pi^ {\AJ} _{\psi_{\infty} } (G_{\infty})$. On $S_{\psi} $ we have the character: $$ s\longmapsto  \nu( s)^{-1} \epsilon_{\psi} ( s)  \lprod{ s, \lambda_{\pi _{\infty}}}   \lprod{s,\pi_{\infty}}_{\AJT}\prod_{v \neq \infty}\lprod{ s,\pi_v} , $$  where $\lprod{s, \lambda_{\pi _\infty}}$ is defined on p.~195 of \cite{kottwitzannarbor}, and $\epsilon_{\psi}$ is as in (\ref{eq:symp root number}). We define $m(\pi^{\infty} ,\psi, \nu)$ to be $1$ if this character is trivial and $0$ otherwise. 
		\item The number $\nu(s_{\psi})$ is $1$ or $-1$ since $s_{\psi}^2 =1$; see (\ref{eq:defn of s psi}) for the definition of $s_{\psi} \in S_{\psi}$. 
	\end{itemize}
	In particular, we have
		\begin{multline*}
		\log \zeta^S (\icoh^* (\overline{ \Sh _K},\mathbb V) )  ,s)  = \sum _{\psi \in \widetilde{\Psi} (G^*)_{\mathbb V}} \sum _{ \pi^{\infty} \in \widetilde{\Pi} _{\psi} ^{\infty} (G)} \dim  (\dot \pi^{\infty}) ^K  \\ \cdot  \sum _ {\nu \in S_{\psi} ^D}  m (\pi^{\infty}, \psi, \nu) (-1) ^{n} \nu (s_{\psi}) \log L^S (\mathcal M(\psi ,\nu) , s),
	\end{multline*} for $s$ in some right half plane. This expresses 
$  \zeta^S (\icoh^* (\overline{ \Sh _K},\mathbb V) )  ,s)$ as a finite product of integral powers of $L^S (\mathcal M(\psi ,\nu) , s) $ for various $\psi$ and $\nu$, and gives a meromorphic continuation of $\zeta^S (\icoh^* (\overline{ \Sh _K},\mathbb V) )  ,s)$ to the whole $\CC$. 
\end{thm}
\begin{proof}
	This immediately follows from Theorem \ref{thm:final} applied to $f^{\infty} dg^{\infty} = \vol (K) ^{-1} 1_{K} dg^{\infty}$. 
\end{proof}
\begin{rem}\label{rem:more general K}
	Theorem \ref{thm:whole zeta} can be slightly generalized as follows. We can replace the completely symmetric $K$ by a more general neat compact open subgroup $K'$ of $G(\adele_f)$ stable under $\vartheta^{\infty}$, and replace $S$ by a sufficiently large, finite set of primes depending on $K'$. For the proof of this generalization, we can take a completely symmetric $K$ contained in $K'$ (see Remark \ref{rem:completely sym}), and apply Theorem \ref{thm:final} to $K$ and the element $\vol (K') ^{-1} 1_{K'} dg^{\infty}$ of $\mathcal H(G(\adele_f)\sslash K)_{\QQ}$.
\end{rem}
\section{More refined decompositions}\label{subsec:refined decomp}
\subsection{}\label{para:Rama}
Throughout we assume the setting of Theorem \ref{thm:final}. In particular we fix $\mathbb V$ as in \S \ref{para:two V} and assume that $K$ is completely symmetric. By Remark \ref{rem:completely sym}, this  assumption on $K$ is harmless for the understanding of $ \icoh^* (\overline{\Sh_K} , \mathbb V)$ for general $K$.
We also keep assuming Hypothesis \ref{hypo} without further mentioning.

In the sequel, we write $\icoh^j$ for $\icoh^j(\overline{\Sh_K}, \mathbb V)$. This is non-zero only for $0 \leq j \leq 2 \dim \Sh_K = 2n$.  We fix an isomorphism $\CC \cong \overline \QQ_{\ell}$, and do not distinguish between representations over $\CC$ and over $\overline \QQ_{\ell}$, nor between $\CC$-valued functions and $\overline \QQ_{\ell}$-valued functions. Nevertheless, we remember that $\Gamma_{\QQ}$-representations on vector spaces over $\CC \cong \overline \QQ_{\ell}$ are always continuous with respect to the $\ell$-adic topology. Let $\mathcal H_K : = \mathcal H(G(\adele_f)\sslash K)_{\QQ}\otimes_{\QQ} \overline \QQ_{\ell}$.\index{$\mathcal H_K$}

We shall apply Theorem \ref{thm:final} to obtain information about more refined decompositions of $ \icoh^j$ as a $\mathcal H_K \times \Gamma_{\QQ}$-module. Ideally, one would like to decompose $ \icoh^j$ into $\pi^{\infty}$-isotypic components $\icoh^j[\pi^{\infty}]$, for $\pi^{\infty}$ running through all irreducible admissible representations of $G(\adele_f)$, and to describe the Galois module structure of each $\icoh^j[\pi^\infty]$. However there are the following two technical obstructions (which can be overcome in the odd case, as we shall eventually see):
\begin{enumerate}
	\item In the even case, each element of a global packet $\widetilde \Pi _{\psi}^{\infty} (G)$ as in \S \ref{para:global packets} does not give rise to a well-defined isomorphism class of $G(\adele
	_f)$-representations, but rather it gives rise to an $\vartheta^{\infty}$-orbit of such isomorphism classes. This obstruction is intrinsic in the endoscopic classification in \cite{arthurbook} and \cite{taibi}. As a result, we are only able to describe the Galois module structure for the direct sum of  $\icoh^j [\pi^{\infty}]$ over all $\pi^{\infty}$ in the same $\vartheta^{\infty}$-orbit, as opposed to each individual $\icoh^j[\pi^{\infty}]$. We mention that in the even case the need to assume that $\mathbb V$ is of the special form as in \S \ref{para:two V} also stems from the same obstruction in the endoscopic classification.
	\item In both the odd and even cases, for a general $\psi \in \widetilde{\Psi}_2(G^*)$ it is not known (although expected, as would follow from the Ramanujan--Petersson conjecture for general linear groups) that the  localization $\psi_v$ is bounded on $\WD_v$ for all finite places $v$. As a result of this drawback, the $G(\QQ_v)$-representations in the local packet $\widetilde \Pi _{\psi_v} (G_v)$ are not known to be irreducible. 
\end{enumerate}

We make several comments on (2). Recall that for any $\psi \in \widetilde \Psi_2(G^*)$, the localization $\psi_v $ of $\psi$ lies in $\Psi^+_{\uni}(G^*_v)$. For arbitrary $\psi_v \in \Psi^+_{\uni}(G^*_v)$ (which may not arise as the localization of a global parameter), Arthur has conjectured that the $G^*(\QQ_v)$-representations in the local packet $\widetilde \Pi _{\psi_v} (G_v^*)$ are irreducible. See \cite[\S\S 1.3--1.5, Conjecture 8.3.1]{arthurbook} for more details. This conjecture would imply that the $G(\QQ_v)$-representations in $\widetilde \Pi _{\psi_v} (G_v)$ are irreducible. In the even case, if $\psi$ is ``trivial on $\SL_2$'' in the sense that $\psi = \boxplus _i \pi_i [d_i]$ with all $d_i$ equal to $1$, then this conjecture has been proved\footnote{We thank the referee for pointing this out to us.} by B.~Xu \cite[Appendix]{Xu}. 

We can sometimes circumvent Arthur's conjecture by using known cases of the Ramanujan--Petersson conjecture. To wit, assume that $\psi = \boxplus _i \pi_i [d_i] \in \widetilde \Psi_2(G^*)$ satisfies the following condition: 
\begin{enumerate}
	\item [($\dagger$)]  The constituents $\pi_i$, which we recall are self-dual unitary cuspidal automorphic representations of $\GL_{N_i}$ over $\QQ$, are all regular C-algebraic or regular L-algebraic.\footnote{The meaning of ``regular C-algebraic'' here is that the infinitesimal character of $\pi_{i,\infty}$ should be regular C-algebraic as in \S \ref{subsubsec:arch para}. In the more classical literature this condition is usually referred to as ``regular algebraic''. The meaning of ``regular L-algebraic'' is that the infinitesimal character of $\pi_{i,\infty}$ should be the Weyl orbit of a regular \emph{integral} character of a maximal torus. The two notions are the same for $\GL_{N_i}$ precisely when $N_i$ is odd.} 
\end{enumerate} Then we  know that $\psi_v$ is bounded on $\WD_v$ for all finite places $v$, since the Ramanujan--Petersson conjecture for $\pi_i$ is known at all $v$. Indeed, let $\pi_i'$ be the twist of $\pi_i$ by $\GL_{N_i}(\adele) \to \RR^\times, g \mapsto \abs{\det(g)}^{1/2}$ if $\pi_i$ is $L$-algebraic and $N_i$ is even, and let $\pi_i' =\pi_i$ in all the other cases. Then $\pi_i'$ regular C-algebraic, cuspidal, and essentially self-dual, and by the work of a long list of authors culminating in Caraiani's  work \cite[Thm.~1.2]{Caraiani}, $(\pi_i')_v$ is essentially tempered for all finite places $v$ (cf.~\cite[Thm.~2.1.1]{BLGGT14} for the essentially self-dual case, as well as a list of references). It then follows that $\pi_{i,v}$ is tempered for all finite places $v$,  for instance by the unitarity of the central character. In conclusion, if $\psi$ satisfies ($\dagger$), we know that all the $G(\QQ_v)$-representations  in $\widetilde \Pi _{\psi_v} (G_v)$ are irreducible for all finite places $v$. 

\subsection{}\label{para:algebraicity}Fix $\psi = \boxplus_i \pi_i [d_i] \in \widetilde \Psi(G^*)_{\mathbb V}$. We investigate when $\psi $ satisfies ($\dagger$) in \S \ref{para:Rama}. Let $N_i$ be the integer such that $\pi_i$ is a self-dual cuspidal automorphic representation of $\GL_{N_i}$. 

For any positive integer $r$, let $T_r$ denote the diagonal matrix in $\GL_{r}$, and identify $X^*(T_r)$ with $\ZZ^r$
 as usual. The half sum of the standard system of positive roots is $ (\frac{r-1}{2}, \frac{r-3}{2},\cdots, \frac{1-r}{2})$. Hence an infinitesimal character $\mu \in (X^*(T_r)\otimes \CC)/\mathfrak S_r = \CC^r / \mathfrak S_r$ for $\GL_r$ is C-algebraic if and only if it lies in $\ZZ^r/ \mathfrak S_r$ when $r$ is odd, and lies in $((\frac{1}{2} ,\cdots, \frac{1}{2} ) + \ZZ^r)/\mathfrak S_r$ when $r$ is even. 
 
 When $G$ is an odd special orthogonal group, we choose a Borel pair in $G_{\CC}$, and identify the based root datum with $\mathrm{BRD}(\mathsf B_m)$ (see \S \ref{subsubsec:standard RD}). Let $\rho$ be the half sum of the positive roots, and let $\lambda$ be the highest weight of $\mathbb V^*$. Thus $$\lambda  = x_1 \epsilon_1 + \cdots + x_m \epsilon_m$$ with $x_i \in \ZZ$ satisfying $x_1 \geq x_2 \geq  \cdots \geq x_m \geq 0$, and 
 $$ \rho  = (m-1) \epsilon_1 + (m-2) \epsilon_2 + \cdots + \epsilon_{m-1}+ \frac{1}{2}(\epsilon_1 +\cdots + \epsilon_m). $$ Under $\stan_G : \widehat G \to \widehat{\GL_{2m}}$, the infinitesimal character $\lambda + \rho$ of $\psi_\infty$ gives rise to the infinitesimal character 
 $$ (m-1 + \frac{1}{2 } + x_1, m-2 + \frac{1}{2} + x_2,\cdots, \frac{1}{2}+ x_m, - (\frac{1}{2}+ x_m), \cdots,   -  (m-1 + \frac{1}{2 } + x_1) )  $$ for $\GL_{2m}$. 
 We see that it is always regular. It immediately follows that the infinitesimal character of each $\pi_{i,\infty}$ must be regular. Moreover, if $d_i$ is odd, then the infinitesimal character of $\pi_{i,\infty}$ must lie in $((\frac{1}{2},\cdots,\frac{1}{2}) + \ZZ^{N_i})/\mathfrak S_{N_i}$, so $\pi_i$ is $C$-algebraic if and only if  $N_i$ is even. In fact this is  automatic, since for $d_i$ odd $\widehat {G_{\pi_i}}$ must be  symplectic; see (\ref{eq:sign pi}). If $d_i$ is even, then the infinitesimal character of $\pi_{i,\infty}$ must lie in $  \ZZ^{N_i}/\mathfrak S_{N_i}$, so $\pi_i$ is $L$-algebraic. We conclude that ($\dagger$) holds automatically. 
 
 When $G$ is an even special orthogonal group,   we choose a Borel pair in $G_{\CC}$, and identify the based root datum with $\mathrm{BRD}(\mathsf D_m)$. Let $\rho$ be the half sum of the positive roots, and let $\lambda_0$ be the highest weight of $\mathbb V^*_{\CC}$ (resp.~$\mathbb V^*_{0,\CC}$) in the even symmetric case (resp.~the even composite case). (See \S \ref{para:two V} for this dichotomy.) Thus $$\lambda_0  = x_1 \epsilon_1 + \cdots + x_m \epsilon_m$$ with $x_i \in \ZZ$ satisfying $x_1 \geq x_2 \geq \cdots \geq x_{m-1} \geq \abs{x_m}$, and 
$$ \rho  = (m-1) \epsilon_1 + (m-2) \epsilon_2 + \cdots + \epsilon_{m-1}. $$ We have  $x_m =0$ in the even symmetric case, and $x_m \neq 0$ in the  even composite case. Under $\stan_G : \widehat G \to \widehat{\GL_{2m}}$, the infinitesimal character $\lambda_0 + \rho$ gives rise to the infinitesimal character 
$$ (m-1  + x_1, m-2 +   x_2,\cdots, x_m, - x_m, \cdots,   -  (m-1 + x_1) )  $$ of $\GL_{2m}$.  We see that in the even symmetric case, we cannot guarantee that the infinitesimal characters of $\pi_{i,\infty}$ are regular, whereas in the even composite case this is guaranteed.
Moreover, by a similar analysis as in the odd case,  $\pi_i$ is $C$-algebraic if $d_i$ and $N_i$ are even, and $\pi_i$ is $L$-algebraic if $d_i$ is odd. Now when $d_i$ is even, $\widehat {G_{\pi_i}}$ must be symplectic, so $N_i$ is automatically even. We conclude that ($\dagger$) automatically holds in the even composite case.

We summarize the above discussion in \S \ref{para:Rama} and \S \ref{para:algebraicity} in the following lemma.
\begin{lem}\label{lem:summary}
Let $\psi = \boxplus_{i} \pi_i[d_i] \in \widetilde \Psi(G^*)_{\mathbb V}$.	In the odd case and the even composite case, the $G(\QQ_v)$-representations in $\widetilde \Pi_{\psi_v} (G_v)$ are irreducible for all finite places $v$. If we are in the even symmetric case and all $d_i$ are equal to $1$, then the same conclusion also holds.  
	\qed
\end{lem} 
\subsection{}\label{para:Hecke syst}
As in \cite[\S 3.4]{arthurbook} we define a  set $\widetilde {\mathcal C}_{\adele} (G^*)$\index{$\widetilde {\mathcal C}_{\adele} (G^*)$} of   \emph{Hecke systems} for $G^*$ modulo a certain equivalence relation. Here a Hecke system\index[n]{Hecke system} for $G^*$ is a family $(c_v)_v$, where $v$ runs through all primes outside an unspecified finite set of primes containing all the ramified primes for $G^*$, and each $c_v$ is a semi-simple conjugacy class in $\lang G^*_v$ (where we take $\lang G^*_v$ to be $\widehat{ {G^*}} \rtimes \Gal(\QQ_v^{\ur}/\QQ_v)$ here for convenience) whose projection to $\Gal(\QQ_v^{\ur}/\QQ_v)$ is the Frobenius. Two such families $(c_v)_v$ and $(d_v)_v$ are said to be \emph{equivalent} and thus define the same element of $\widetilde {\mathcal C}_{\adele} (G^*)$, if for almost all $v$, the conjugacy classes $c_v$ and $d_v$ are in the same $\Aut(\lang G^*_v)$-orbit, or equivalently, the images of $c_v$ and $d_v$ under $\lang G^*_v \to \lang G^* \xrightarrow{\stan_{ G^*}} \GL_N(\CC)$ are conjugate. (The equivalence of the two conditions follows easily from the description of $\Aut(\lang G^*_v)$ in Remark \ref{rem:taut} and the fact that two elements of $\mathrm{O}_N(\CC)$ are conjugate if and only if they are conjugate in $\GL_N(\CC)$.)

Recall that as a fundamental construction in \cite{arthurbook}, we have a canonical injection 
\begin{align}\label{eq:inj to Hecke syst}
	\widetilde{  \Psi} (G^*)  &  \To  \widetilde {\mathcal C} _{\adele} (G^*)  \\ \nonumber \psi & \longmapsto c(\psi) 
\end{align}
whose well-definedness is guaranteed by \cite[Thm.~1.3.2, Thm.~1.4.1]{arthurbook}. This map has the following characterization: Let $\psi \in 	\widetilde{  \Psi} (G^*) $, and for almost all primes $v$ for which $\psi_v$ is unramified, denote by $\pi_v$ the unique element of $\widetilde \Pi_{\psi_v}(G_v)$ that is a $\vartheta_v$-orbit of $\mathcal G(\ZZ_v)$-unramified representations (see \S \ref{para:global packets}). Then for almost all $v$, for every $\dot \pi_v \in \pi_v$ the Satake parameter of the $\mathcal H(G(\QQ_v) \sslash \mathcal G(\ZZ_v))$-module $\dot \pi_v ^{\mathcal G(\ZZ_v)}$ (which is $1$-dimensional over $\CC$, cf.~Lemma \ref{lem:parameter is unramified})  belongs to the $\Aut(\lang G^*_v)$-orbit of the component of $c(\psi)$ at $v$.

As in Definition \ref{defn:compleltely symm}, we have $K = \prod_v K_v$. We have a canonical (finite) direct sum decomposition of $\mathcal H_K \times \Gamma_{\QQ}$-modules:
$$ \icoh^j  = \bigoplus_{c \in \widetilde {\mathcal C} _{\adele} (G^*)} \icoh^j_c ,$$ where $\icoh^j_c$ is characterized by the property that for almost all primes $v$ for which $K_v$ is hyperspecial, the action of the unramified Hecke algebra $\mathcal H(G(\QQ_v) \sslash K_v)$ on $\icoh^j_c$ is via  characters which correspond under the Satake isomorphism   to elements of the $\Aut (\lang G^*)$-orbit of the component of $c$ at $v$. 

We denote by  $\Irr$  \index{$\mathrm{Irr}(G(\adele_f))$} the set of isomorphism classes of irreducible admissible representations of $G(\adele_f)$ (over $\CC \cong \overline \QQ_{\ell}$). 
For each $c \in   \widetilde{\mathcal C} _{\adele} (G^*)$ and $\tau \in \Irr$,  let $$ W^j_c(\tau) : = \Hom_{ \mathcal H_K} (\tau^K, \icoh^j_c). $$
We then have direct sum decompositions of  $ \mathcal H_K \times \Gamma_{\QQ}$-modules
\begin{align} \nonumber
\icoh^j_c & = \bigoplus_{\tau \in \Irr} \tau^K \otimes_{\overline \QQ_{\ell}} W^j_c(\tau), \\ \label{eq:complete decomp}
\icoh^j  & = \bigoplus_{c \in \widetilde {\mathcal C} _{\adele} (G^*)}  \bigoplus_{\tau \in \Irr} \tau^K \otimes_{\overline \QQ_{\ell}} W^j_c(\tau),
\end{align} 
		where on the right hand sides $\mathcal H_K $ acts on $\tau^K$ and $\Gamma_{\QQ}$ acts on $W^j_c(\tau)$. Here we have used the fact that $\icoh^j$ is a semi-simple $\mathcal H_K$-module, which follows from the ``Matsushima formula'' for $L^2$-cohomology \cite[Thm.~4.5]{BorelCasselman} and Zucker's conjecture comparing $L^2$-cohomology with intersection cohomology, proved by Looijenga \cite{Loo}, Saper--Stern \cite{SaperStern}, and Looijenga--Rapoport \cite{LooRap}.

 	\begin{thm}\label{thm:decomp c}Assume Hypothesis \ref{hypo}. Let $c \in \widetilde {\mathcal C} _{\adele} (G^*)$. The following statements hold.
		\begin{enumerate} 
			\item If $c$ is not in the image of  $\widetilde{  \Psi} (G^*)_{\mathbb V}$ under the map (\ref{eq:inj to Hecke syst}), then $$\icoh^j_c  = 0 $$ for all $j$.  
			\item Assume that $c = c(\psi)$ for $\psi \in \widetilde{  \Psi} (G^*)_{\mathbb V}$. The for almost all primes $p$ and all integers $a$  we have
			\begin{multline}\label{eq:final 1}
				\sum_j (-1) ^j \Tr  (\Frob_p^a \mid  \icoh^j_c )\\ = m_{\psi}   \sum _{ \pi^{\infty} \in \widetilde{\Pi} _{\psi} ^{\infty} (G)}  \dim (\dot \pi^{\infty})^K  \sum _ {\nu \in S_{\psi} ^D}  m (\pi^{\infty}, \psi,  \nu) (-1) ^{n} \nu (s_{\psi}) \Tr (\Frob_p^a \mid  \mathcal M_p (\psi ,\nu)),
			\end{multline}
			where the terms on the right hand side are defined in the same way as in Theorem \ref{thm:whole zeta}, with $\mathcal M_p(\psi,\nu)$ defined in \S \ref{defn:mathcal M}. 
			\item  Keep the assumption in (2), and assume that $\icoh^j_c \neq 0$ for some $j$. Write $\psi = \boxplus_{i\in I} \pi_i [d_i]$. Then for each $i\in I$ and for almost all primes $p$, $\pi_{i,p}$ is tempered.\footnote{We thank the anonymous referee for suggesting this result to us.}
\item  Keep the assumption in (2), and assume that the conclusion of Lemma \ref{lem:summary} holds for $\psi$. Thus each $\pi^{\infty} \in \widetilde \Pi_{\psi}^{\infty}(G)$ determines a $\vartheta^{\infty}$-orbit $[\pi^{\infty}]$ in $\Irr$, as in \S \ref{para:global packets}. Let $\tau_0 \in \Irr$ be such that $\tau_0^K \neq 0$ and $\tau_0 \notin [\pi^{\infty}], \forall \pi^{\infty} \in \widetilde \Pi^{\infty}_{\psi}(G)$. Then $$W^j_c(\tau_0) = 0$$ for all $j$. Moreover, for each $\pi^{\infty} \in \widetilde \Pi^{\infty}_{\psi}(G)$, we have  	\begin{multline}\label{eq:final 2}
	\sum_j (-1) ^j \Tr  (\Frob_p^a \mid   \bigoplus_{\tau \in [\pi^\infty]} \dim (\tau^K)  \cdot  W^j_c(\tau) )\\ = m_{\psi}     \dim (\dot \pi^{\infty})^K   \sum _ {\nu \in S_{\psi} ^D}  m (\pi^{\infty}, \psi,  \nu) (-1) ^{n} \nu (s_{\psi}) \Tr (\Frob_p^a \mid  \mathcal M_p (\psi ,\nu)),
\end{multline} for almost all primes $p$ and all integers $a$. 
		\end{enumerate}
	\end{thm}
\begin{proof} Throughout the proof we use the following notations: We fix the Haar measure $dg^{\infty}$ on $G(\adele_f)$ giving volume $1$ to $K$. Let $\tH = (\mathcal H_K)^{\vartheta^{\infty}}$. Then $\tH$ is a  $\CC$-subalgebra of $\mathcal H_K$ with unit $1_K dg^{\infty}$. By the same argument as in the proof of Lemma \ref{lem:invariance}, we know that as a $\CC$-vector space $\tH$ is generated by elements of the form $(\prod_v f_v)dg^{\infty}$, where the product is over all primes $v$, $f_v\in C^{\infty}_c(G(\QQ_v)\sslash K_v)^{\theta_G}$ for all $v$, and $f_v = 1_{K_v}$ for almost all $v$. For any finite set of primes $S$ such that $K_v$ is hyperspecial for all $v\notin S$, we let $\tH^S$ be the $\CC$-vector subspace of $\tH$ spanned by elements of the form $(\prod_v f_v) dg^{\infty}$, where $f_v\in C^{\infty}_c(G(\QQ_v) \sslash K_v)^{\theta_G}$ for all $v$, and the set of $v$ such that $f_v \neq 1_{K_v}$ is finite and disjoint from $S$. Then $\tH^S$ is a commutative unital subring of $\tH$, identified with the restricted tensor product of the $\vartheta_v$-fixed subrings of the unramified Hecke algebras  $C^{\infty}_c(G(\QQ_v)\sslash K_v)$ over all $v \notin S$.

		 For each $\psi \in \widetilde \Psi(G^*)_{\mathbb V}$ and each $\pi^{\infty} \in \widetilde \Pi_{\psi}^{\infty}(G)$, 
	recall from \S \ref{para:global packets} that the $G(\adele_f)$-representation $\dot \pi^{\infty}$ is the restricted tensor product of $\dot \pi_v$ over all primes $v$, where for each $v$ we choose a member $\dot \pi_v \in \pi_v$. The $\tH$-module $(\dot \pi^{\infty})^K$ depends only on $\pi^{\infty}$, not on the extra choices. We henceforth denote it $(\pi^{\infty})^K$.

	\textbf{(1)} 
		By the finiteness statement in Theorem \ref{thm:final}, on the RHS of (\ref{eq:semifinal}) only a finite subset $\Psi_0 \subset \widetilde \Psi(G^*)_{\mathbb V}$  would potentially contribute non-trivially, and for each $\psi \in \Psi_0$ only a finite subset $\mathcal U_{\psi} \subset \widetilde \Pi^{\infty}_{\psi} (G)$ would potentially contribute non-trivially. Moreover $\Psi_0$ and $(\mathcal U_{\psi})_{\psi \in \Psi_0}$ are independent of $f^{\infty} dg^{\infty}, p, a$. We may and shall take each $\mathcal U_{\psi}$ such that its members $\pi^{\infty}$ satisfy $(\pi^\infty)^K \neq 0$.    
		
	Suppose $c$ is as in (1) and $\icoh^j_c \neq 0$ for some $j$. Let $S$ be a finite set of primes such that $K_v$ is hyperspecial for all $v\notin S$. Then  the set of characters through which $\tH^S$ acts on $\icoh^j_c$ (i.e., the isomorphism classes of the simple $\tH^S$-submodules of $\icoh^j_c$) is disjoint from the set of characters through which $\tH^S$ acts on $\icoh^{j'}_{c'}$ for all $j'$ and all $c' \neq c$, and on  $(\pi^{\infty})^K$ for all $\pi^{\infty} \in \coprod_{\psi \in \Psi_0} \mathcal U_{\psi}$. Indeed, this follows from the observation that for any  two characters  $\chi_1, \chi_2: C^{\infty}_c(G(\QQ_v)\sslash K_v) \to \CC$ having the same restriction to the $\vartheta_v$-fixed subring, $\chi_1$ and $\chi_2$ must be related by $\vartheta_v$, and hence the Satake parameters of $\chi_1$ and $\chi_2$ must be related by $\Aut(\lang G^*_v)$. The observation itself follows from the identity $\chi_1 + \theta_G(\chi_1) = \chi_2 + \theta_G(\chi_2)$ (which holds since for all $F \in C^{\infty}_c(G(\QQ_v)\sslash K_v)$, $F + \theta_G ^* F$ lies in the $\vartheta_v$-fixed subring) and the linear independence of characters.  Since all these $\tH^S$-modules are  finite-dimensional over $\overline \QQ_{\ell} \cong \CC$ and there are only finitely many of them which are non-zero, there exists $f^{\infty} dg^{\infty} \in \tH^S\subset \tH$  
		that acts as the identity on $\icoh^j_c$ for all $j$, as zero on $\icoh^{j'}_{c'}$ for all $j'$ and all $c' \neq c$, and as zero on  $(\pi^{\infty})^K$ for  all $\pi^{\infty} \in  \coprod_{\psi \in \Psi_0} \mathcal U_{\psi}$. We then apply Theorem \ref{thm:final} (generalized in the obvious manner from $\vartheta^{\infty}$-fixed elements of $\mathcal H(G(\adele_f)\sslash K)_{\QQ}$ to elements of $\tH$) to $f^{\infty} dg^{\infty}$ and obtain 
		$$ \sum_j (-1)^j \Tr(\Frob_p^a \mid \icoh^j_c) =  0  $$ for all sufficiently large primes $p$ and all integers $a$. 	By Chebotarev's density theorem and the Brauer--Nesbitt theorem, this implies that in the Grothendieck group of $\Gamma_{\QQ}$-representations over $\overline \QQ_{\ell}$ we have 
	$$ \sum_j (-1)^j [\icoh^j_c] = 0.$$
	By a purity result of Pink \cite[Prop.~5.6.2]{pink1992ladic} applied to our Shimura datum $\mathbf O(V)$ of abelian type, and by the purity of intersection cohomology, we know that for almost all primes $p$ the action of $\Frob_p$ on $\icoh^j$ has weight $j$. (Note that the weight cocharacter of the Shimura datum which appears in \cite[
	\S 5.4]{pink1992ladic} is trivial in our case.)  It then follows that there is no cancellation between $[\icoh^j_c]$ for different $j$ in the Grothendieck group. Hence $\icoh^j_c = 0$ for all $j$, which proves (1).
	
\textbf{(2)}  Similarly as in the proof of (1), the set of characters  through which $\tH^S$ acts on $\icoh^j_c$ for all $j$ and on $(\pi^\infty)^K$ for all $\pi^\infty \in \mathcal U_\psi$, is disjoint from the set of characters through which $\tH^S$ acts on $\icoh^{j'}_{c'}$ for all $j'$ and all $c' \neq c$  and on $(\pi^{\infty})^K$ for all  $\pi^{\infty} \in \coprod_{\psi' \in \Psi_0-\set{\psi}}\mathcal U_{\psi'}$.  Thus we can find $f^{\infty} dg^{\infty} \in \tH^S \subset \tH$ which acts as the identity on $\icoh^j_c$ for all $j$, as the identity on $(\pi^{\infty})^K$ for all $\pi^{\infty} \in \mathcal U_{\psi}$, as zero on $\icoh^{j'}_{c'}$ for all $j'$ and all $c' \neq c$, and as zero on   $(\pi^{\infty})^K$ for all $\pi^{\infty} \in \coprod_{\psi' \in \Psi_0-\set{\psi}}\mathcal U_{\psi'}$. Applying Theorem \ref{thm:final} to $f^{\infty} dg^{\infty}$ then gives the desired result. 
	
\textbf{(3)} For almost all $p$, part (2) gives a multiplicative relation 
$$ \det(T - \Frob_p \mid \mathcal M_p (\psi,\nu))^{k_{\nu}} = \prod_j  \det(T - \Frob_p \mid \icoh^j_c)^{(-1)^j}  $$ for each $\nu \in S_{\psi}^D$, where $k_{\nu}$ is an integer (independent of $p$). By the purity results used in the proof of (1) and by our assumption that $\icoh^j_c \neq 0$ for some $j$, we conclude that $k_{\nu}\neq 0$ (since the right hand side of the above relation cannot be $1$). It then follows from the above relation that  $\Frob_p$ acts on $\mathcal M_p(\psi, \nu)$ with integer weights for each $\nu$.  On the other hand we have an isomorphism of $\WD_p$-representations
$$ \bigoplus_{\nu\in S_{\psi}^D} \mathcal V_p(\psi,\nu) \cong \bigoplus_{i\in I} \mathcal V(\pi_{i,p})^{(d_i)}, $$ where $\mathcal V_p(\psi,\nu)$ is as in \S \ref{defn:mathcal M}, and $\mathcal V(\pi_{i,p})$ is the $\WD_p$-representation corresponding to $\pi_{i,p}$ under the local Langlands correspondence. Clearly $\pi_{i,p}$ is unitary since $\pi_i$ is. By \cite{Shalika}, $\pi_{i,p}$ is generic. Hence by \cite[Cor.~2.5]{JS}, all eigenvalues $\lambda$ of $\Frob_p $ on $\mathcal V(\pi_{i,p})$ satisfy  $p^{-1/2} < \abs{\lambda} <p^{1/2}$. Therefore if $\pi_{i,p}$ is not tempered for one $i$, then there is at least one eigenvalue of $\Frob_p$ on $\bigoplus_{\nu} \mathcal V_p(\psi,\nu)$ whose absolute value is not an integer power of $p^{1/2}$. This contradicts with the fact that $\Frob_p$ acts on $\mathcal M_p(\psi,\nu) = \norm{\cdot}^{-n/2} \mathcal V_p(\psi,\nu)$ with integer weights for each $\nu$. We have proved (3). 

\textbf{(4)}  Pick $S$ large enough such that $\tH^S$  acts on $(\pi^{\infty})^K$ for all $\pi^{\infty} \in \mathcal U_{\psi}$ through a common character $\chi^S: \tH^S \to \CC$. Since $\tau_0$ is an irreducible admissible $G(\adele_f)$-representation, we know that $\tH^S$ must act on $\tau_0^K$ via a character $\chi^S_0$ (as opposed to several different characters). Assume for the sake of contradiction that  $W^j_c(\tau_0) \neq 0$. Then up to enlarging $S$ we must have $\chi^S =\chi^S_0$, by the definition of $\icoh_c^j$. In the following we assume that this is the case. We have   $\tau_0 = \bigotimes'_v \tau_{0,v}$, where each $\tau_{0,v}$ is an irreducible admissible representation of $G(\QQ_v)$.   Write $G_S$ for $\prod_{v\in S} G(\QQ_v)$, and write $K_S$ for $\prod_{v\in S} K_v$.  By a similar argument as in the proof of (1), our assumption that $\chi^S = \chi^S_0$ implies that for each $v \notin S$, the $\vartheta_v$-orbit of  the isomorphism class of the irreducible admissible $G(\QQ_v)$-representation  $\tau_{0,v}$ agrees with the $\vartheta_v$-orbit arising from every $\pi^{\infty} \in \mathcal U_{\psi}$. Therefore our assumption on $\tau_0$ implies that the $\vartheta_S$-orbit of the isomorphism class of the irreducible admissible $G_S$-representation $\bigotimes_{v\in S} \tau_{0,v}$ is disjoint from the $\vartheta_S$-orbit arising from any $\pi^{\infty} \in \mathcal U_{\psi}$. We can therefore find $f_S \in C^{\infty}_c(G_S \sslash K_S) = \bigotimes_{v\in S} C^{\infty}_c(G(\QQ_v) \sslash K_v)$ such that (for a certain normalization of Haar measure) it acts as the identity on every $\vartheta_S$-translate of $(\bigotimes_{v\in S} \tau_{0,v})^{K_S}$ and as zero on $(\bigotimes_{v\in S} \dot \pi_v)^{K_S}$ for all $\pi^\infty = (\pi_v)_v \in \mathcal U_{\psi}$ and all choices $(\dot \pi_v \in \pi_v)_{v\in S}$.  Note that the defining property of $f_S$ is invariant under the action of $\vartheta_S$ on $C^{\infty}_c(G_S\sslash K_S)$. Hence we can replace $f_S$ by its average under the finite group $\vartheta_S$, and assume that $f_S$ is fixed by $\vartheta_S$. 

After suitable scaling, the element $(f_S \cdot \prod_{v \notin S} 1_{K_v}) dg^{\infty} \in \tH$ acts as the identity on every $\vartheta^{\infty}$-translate of $\tau_0^K$ and as zero on 
$(\pi^{\infty})^K$ for all $(\pi^{\infty}) \in \mathcal U_{\psi}$. By a similar argument, we can also construct an element of $\tH$ which acts as the identity on every $\vartheta^{\infty}$-translate of $\tau_0^K$ and as zero on $\tau^K$ for every $\tau \in \Irr$ such that $\tau$ is not isomorphic to a $\vartheta^{\infty}$-translate of $\tau_0$ and $\tau^K \neq 0, W^j_c(\tau) \neq 0$. We have a third element of $\tH$, as constructed in the proof of (2), which acts as the identity on $\icoh^j_c$ for all $j$,  as zero on $\icoh^{j'}_{c'}$ for all $j'$ and all $c' \neq c$, and as zero on $(\pi^{\infty})^K$ for all    $\pi^{\infty} \in \coprod_{\psi' \in \Psi_0 - \set{\psi}} \mathcal U_{\psi'}$. Multiplying these  three elements together, we obtain an element of $\tH$ which acts on $\icoh^j$ for each $j$ as the projection to $\bigoplus_{\tau \in [\tau_0]} \tau^K \otimes W^j_c(\tau)$ with respect to (\ref{eq:complete decomp}), and acts as zero on $(\pi^{\infty})^K$ for  all $\pi^\infty \in  \coprod_{\psi \in \Psi_0} \mathcal U_{\psi}$. Here $[\tau_0]$ denotes the $\vartheta^{\infty}$-orbit of $\tau_0$ in $\Irr$. Applying Theorem \ref{thm:final} to this element we obtain 
	$$ \sum_j (-1)^j \Tr(\Frob_p^a \mid \bigoplus_{\tau \in [\tau_0]} \tau^K \otimes W^j_c(\tau)) =  0  $$ for almost all primes $p$ and all integers $a$. By a similar argument as in (1), this implies that $\tau^K \otimes W^j_c(\tau) = 0$ for all $\tau \in [\tau_0]$, and in particular $W^j_c(\tau_0) =0$, as desired. 
	
	Finally we prove (\ref{eq:final 2}). Since different elements $\pi^\infty \in \mathcal U_{\psi}$ give rise to disjoint $\vartheta^{\infty}$-orbits $[\pi^\infty]$, essentially the same argument as before gives us an element of $\tH$ which acts on $\icoh^j$ for each $j$ as the projection to $\bigoplus_{\tau \in [\pi^\infty]} \tau^K \otimes W^j_c(\tau)$ with respect to (\ref{eq:complete decomp}), acts as zero on $(\pi^{\infty,\prime})^K$ for all $\pi^{\infty, \prime} \in (\coprod_{\psi' \in \Psi_0 - \set{\psi}} \mathcal U_{\psi'}) \sqcup (\mathcal U_{\psi} - \set{\pi^\infty})$, and acts as the identity on $(\pi^\infty)^K$. Applying Theorem \ref{thm:final} to this element we obtain (\ref{eq:final 2}).   
\end{proof}
\begin{rem}\label{rem:Rama} Part (3) of Theorem \ref{thm:decomp c} proves the Ramanujan--Petersson conjecture for $\pi_i$ for almost all primes. 
	As we have discussed in \S  \ref{para:Rama} and \S \ref{para:algebraicity}, this is known in the odd case and in the even composite case (where the conjecture is known for all primes). In the even symmetric case, however, the infinitesimal character of $\pi_{i,\infty}$ can be non-regular, and thus  $\pi_i \otimes \abs{\det}^{\alpha} $ is not cohomological for any $\alpha \in \CC$. For such $\pi_i$ our result proves new instances of the conjecture. We postpone a more systematic treatment to future work. 
\end{rem} 
\subsection{}\label{para:final setting}By utilizing Theorem \ref{thm:decomp c} (3), we can separate the contributions of different degrees $j$ to the right hand sides of (\ref{eq:final 1}) and (\ref{eq:final 2}) as follows. 

Let $\psi = \boxplus_{i\in I} \pi_i [d_i] \in \widetilde \Psi(G^*)_{\mathbb V}$, and keep the notation in \S \ref{defn:mathcal M} for $\psi$. Let $\nu \in S_{\psi}^D$. Recall from \S \ref{defn:mathcal M} that there is a subset $I_\nu$ of $I$ of cardinality at most $2$ such that $\mathcal V_p(\psi,\nu) = \bigoplus_{i\in I_\nu} \mathcal V(\pi_{i,p})^{(d_i)}$ for all primes $p$. Recall that $n = \dim \Sh_K$. For each integer $j$, define \index{$\mathcal M_p(\psi,\nu,j)$}
$$ \mathcal M_p(\psi,\nu,j) : = \bigoplus_{\substack{i\in I_\nu \\ d_i -1 \geq \abs{n-j}\\  d_i - 1 \equiv n-j \mod 2 }} \norm{\cdot}^{-j/2 } \mathcal V(\pi_{i,p}). $$ Thus \begin{align}\label{eq:decomp into j}
\mathcal M_p(\psi, \nu) = \bigoplus_{j\in \ZZ} \mathcal M_p(\psi, \nu , j). 
\end{align}$ $ 

\begin{cor} \label{cor:single degree}
	Let $c= c (\psi)$. For each integer $j$, we have 
		\begin{multline} \label{eq:single deg 1}
		 (-1) ^j \Tr  (\Frob_p^a \mid  \icoh^j_c )\\ = m_{\psi}   \sum _{ \pi^{\infty} \in \widetilde{\Pi} _{\psi} ^{\infty} (G)}  \dim (\dot \pi^{\infty})^K  \sum _ {\nu \in S_{\psi} ^D}  m (\pi^{\infty}, \psi,  \nu) (-1) ^{n} \nu (s_{\psi}) \Tr (\Frob_p^a \mid  \mathcal M_p (\psi ,\nu,j))
	\end{multline} for almost all primes $p$ and all integers $a$. If all $d_i$ are $1$, then $$\icoh^j_{c } = 0$$ for all $j\neq n$.
If we assume that the conclusion of Lemma \ref{lem:summary} holds for $\psi$, then for each integer $j$ and each $\pi^\infty_0 \in \widetilde \Pi_{\psi}^\infty (G)$, we have 
\begin{multline*} 
	  (-1) ^j \Tr  (\Frob_p^a \mid   \bigoplus_{\tau \in [\pi^\infty_0]} \dim (\tau^K)  \cdot  W^j_c(\tau) )\\ = m_{\psi}     \dim (\dot \pi^{\infty}_0)^K   \sum _ {\nu \in S_{\psi} ^D}  m (\pi^{\infty}_0, \psi,  \nu) (-1) ^{n} \nu (s_{\psi}) \Tr (\Frob_p^a \mid  \mathcal M_p (\psi ,\nu,j)),
\end{multline*} 
for almost all primes $p$ and all integers $a$. 

\end{cor}
\begin{proof}
	By Theorem \ref{thm:decomp c} (3), we know that for almost all primes $p$, $\Frob_p$ acts on $\mathcal M_p(\psi,\nu,j)$ with weight $j$. By the purity results used in the proof of Theorem \ref{thm:decomp c} (1), $\Frob_p$ acts on $\icoh^j$ with weight $j$, for almost all $p$. The first and third statements in the corollary follow from these two facts, the decomposition (\ref{eq:decomp into j}), and the two formulas (\ref{eq:final 1}) and  (\ref{eq:final 2}). For the second statement, for $j\neq n$ we have  $\mathcal M_p(\psi,\nu, j) = 0$ for all $\nu \in S_{\psi}^D$. Applying (\ref{eq:single deg 1}) to $a=0$ gives the result.
\end{proof} 
\subsection{} 	Keep the notation of \S \ref{para:final setting}, and assume that we are in the odd case or the even composite case. From the discussion in \S \ref{para:algebraicity}, one easily sees that for each $i\in I$ and $j\in \ZZ$ such that $d_i -1 \equiv n-j \mod 2$, the cuspidal automorphic representation $\pi_i \otimes \abs{\det}^{-j/2}$ of $\GL_{N_i}$ is essentially self-dual and regular L-algebraic. Thus the semi-simple $\ell$-adic $\Gamma_{\QQ}$-representation associated to  $\pi_i \otimes \abs{\det}^{-j/2}$ is known to exist and satisfies local-global compatibility; see for instance \cite[Thm.~2.1.1]{BLGGT14}.\footnote{In that reference, $\pi$ is assumed to be a  regular C-algebraic cuspidal essentially self-dual representation of $\GL_n$, but the Galois representation is associated to $\pi \otimes \abs{\det}^{(1-n)/2}$, which is regular L-algebraic.} It follows that for each $j\in \ZZ$ and $\nu \in S_{\psi}^D$, there is a semi-simple $\ell$-adic $\Gamma_{\QQ}$-representation $\mathcal M(\psi,\nu,j)$,\index{$\mathcal M(\psi,\nu,j)$} obtained by taking a direct sum of the ones just mentioned over $i\in I_{\nu}$ such that $d_i -1 \geq \abs{n-j}$ and $d_i-1 \equiv n-j \mod 2$, such that for every prime $p\neq \ell$ the localization of $\mathcal M(\psi,\nu,j)$ gives the $\WD_p$-representation $\mathcal M_p(\psi,\nu,j)$ up to semi-simplification. 
\begin{cor} \label{cor:known Galois}	Let $c= c (\psi)$, and let $\pi^\infty_0 \in \widetilde \Pi_{\psi}^\infty (G)$. Assume that we are in the odd case or the even composite case.
	Up to semi-simplification, the $\Gamma_{\QQ}$-representations $\icoh^j_c$ and $  \bigoplus_{\tau \in [\pi^\infty_0]} \dim (\tau^K)  \cdot  W^j_c(\tau)$ are isomorphic to the virtual representations
	$$
		  m_{\psi}   \bigoplus _{ \pi^{\infty} \in \widetilde{\Pi} _{\psi} ^{\infty} (G)}  \dim (\dot \pi^{\infty})^K  \bigoplus _ {\nu \in S_{\psi} ^D} (-1)^{j+n} \nu(s_\psi) m (\pi^{\infty}, \psi,  \nu)   \mathcal M (\psi ,\nu,j)
$$
	and 
$$	  m_{\psi}     \dim (\dot \pi^{\infty}_0)^K   \bigoplus _ {\nu \in S_{\psi} ^D}  (-1)^{j+n} \nu(s_\psi)  m (\pi^{\infty}_0, \psi,  \nu)      \mathcal M (\psi ,\nu,j)
$$ respectively. In the odd case, the semi-simplification of $ W^j_c(\pi^{\infty}_0) $ is isomorphic to $$\bigoplus _ {\nu \in S_{\psi} ^D} (-1)^{j+1} \nu(s_\psi)  m (\pi^{\infty}_0, \psi,  \nu)      \mathcal M (\psi ,\nu,j). $$  
\end{cor}
\begin{proof}
	This follows from Lemma \ref{lem:summary}, Corollary \ref{cor:single degree}, Chebotarev's density theorem, and the Brauer--Nesbitt theorem.
\end{proof}

\begin{rem}
In the even symmetric case, for $\psi = \boxplus_{i\in I} \pi_i [d_i] \in \widetilde \Psi(G^*)_{\mathbb V}$, the infinitesimal character of $\pi_{i,\infty}$ can be non-regular. Thus the conjectural $\ell$-adic $\Gamma_{\QQ}$-representation associated to (an L-algebraic twist of) $\pi_i$ has not been constructed. To this end our Corollary \ref{cor:single degree} can be utilized for the construction of such a Galois representation. We will investigate this on another occasion. 
\end{rem}

\backmatter

\printindex
\printindex[n]

	\bibliographystyle{smfalpha}
	
\bibliography{myref}

 \end{document}